\documentclass{amsart}
\usepackage{amssymb, latexsym}
\usepackage{amsthm}
\usepackage{amstext}
\usepackage{amsxtra}
\usepackage{graphicx}
\usepackage{amsfonts}
\usepackage{amsmath}
\usepackage{amscd}
\usepackage{MnSymbol}
\usepackage{enumerate}
\usepackage{tikz}
\usepackage{ytableau}
\usepackage{mathdots}
\usepackage{OrbitCombinatoricsStyle}
\usepackage{hyperref}

\begin{document}
\title{Orbits and Equivariant Local Systems Combinatorics in Graded Lie Algebras}
\author{Robert B\'edard}
\address{D\'epartement de math\'ematiques \\ Universit\'e du Qu\'ebec \`a Montr\'eal \\   C.P. 8888, Succursale Centre-Ville, Montréal, Qu\'ebec H3C 3P8, Canada }
\email{bedard.robert@uqam.ca}

\begin{abstract}
In this paper, we will describe a combinatorial object to list the orbits in the ${\mathbb Z}$-graded Lie algebra, their Jordan bloc decomposition, their dimension, the partial order and the equivariant local system (up to isomorphism) for four infinite families: two are for the symplectic groups and two are for the special orthogonal groups. These orbits and equivariant local systems appear in the study of perverse sheaves arising from graded Lie algebras.
\end{abstract}

\setcounter{tocdepth}{1}

\maketitle

\tableofcontents

\setcounter{section}{-1}
\section{Introduction}

\subsection{}\label{SS:CharacteristicK}
Let ${\mathbf k}$ be an algebraically closed field of characteristic $p \geq 0$. In the case where $p > 0$, we will assume that $p$ is a large prime number so that our operations with Lie algebras are as in the characteristic 0 case.  In the following,  $G$ will denote a classical connected reductive algebraic group over ${\mathbf k}$, $\mathfrak g$ will denote the Lie algebra of $G$ and $\iota:{\mathbf k}^{\times} \longrightarrow G$ will be a homomorphism of algebraic groups. 

\subsection{} 
It is known that the centralizer 
\[
G^{\iota} = \{g \in G \mid g \iota(t) = \iota(t) g \text{ for all $t \in {\mathbf k}^{\times}$}\}
\]
of $\iota$ is a connected reductive subgroup of $G$ and it acts  on the eigenspace 
\[
{\mathfrak g}_m = \{X \in {\mathfrak g} \mid Ad(\iota(t)) X = t^m X \text{  for all $t \in {\mathbf k}^{\times}$}\},
\]
 where $m \in {\mathbb Z}$, by restriction of the adjoint action of $G$.  For $m \ne 0$, the elements in ${\mathfrak g}_m$ are nilpotent elements.

\subsection{} 
We will study the $G^{\iota}$-orbits on the eigenspace ${\frak g}_2$ in the case of symplectic and special orthogonal groups.
The orbits of $G^{\iota}$ on these eigenspaces have already been studied by Vinberg in \cite{V1979}, by Kawanaka in \cite{K1987}, as well as by Derksen and Weyman in \cite{DW2002}. Lusztig has also given such a parametrization in \cite{L2010} and with Yun in  \cite{LY2018}.  We will use the approach of Derksen and Weyman  in \cite{DW2002}  to parametrize these orbits. We will recall their approach in the next section.  In paragraph 3.2 in \cite{C2008}, Ciubotaru gave a different approach for enumerating these $G^{\iota}$-orbits. It is a general algorithm for which the computations are proportional to the rank of the Lie algebra ${\frak g}$.  

\subsection{}
The example 5.5 in Lusztig's paper \cite{ L2010} was our starting point for our study.  To be more precise,  we will  have a finite dimensional vector space $V$ over ${\mathbf k}$ with a non-degenerate  bilinear form $\langle\ , \ \rangle:V \times V \rightarrow {\mathbf k}$ which is either symmetric or skew-symmetric and an ${\mathcal I}$-grading $V = \oplus_{i \in {\mathcal I}} V_i$ such that $\langle u, v \rangle = 0$ whenever $u \in V_i$, $v \in V_j$ and $i + j \ne 0$.  Here ${\mathcal I}$ will be a subset of integers with  the same parity (either they are all even or all odd) such that if $i \in {\mathcal I}$, then $-i \in {\mathcal I}$. ${\mathfrak g}$ will be the Lie algebra of $V$ relative to this bilinear form. In other words, ${\mathfrak g}$  will be the set of endomorphisms $X:V \rightarrow V$ of $V$ such that $\langle X(u), v\rangle + \langle v, X(v)\rangle = 0$ whenever $u, v \in V$,   $G$ will be the connected reductive  subgroup of $GL(V)$ such that  $\langle gu, gv\rangle = \langle u, v \rangle$ for all $u, v \in V$ and $g \in G$ and  whose Lie algebra is ${\mathfrak g}$ and $\iota:{\mathbf k}^{\times} \rightarrow G$ is the unique algebraic homomorphism such that $\iota(t) u = t^i u$ for all $t \in {\mathbf k}^{\times}$ whenever $u \in V_i$ and $i \in {\mathcal I}$.  

As above we have that $G^{\iota}$ acts by restriction of the adjoint action of $G$ on the eigenspace ${\mathfrak g}_m$. We will give a combinatorial object: ${\mathcal I}$-tableau  such that we have a bijection between a subset of such tableaux and the $G^{\iota}$-orbits in ${\mathfrak g}_2$. For each $G^{\iota}$-orbit in ${\mathfrak g}_2$, we will give using the corresponding ${\mathcal I}$-tableau:
\begin{itemize}
\item a Jacobson-Morozov triple $(E, H, F)$;
\item an equation of the dimension of the orbit;
\item the partition corresponding of the Jordan blocs for any element in this orbit;
\item when two $G^{\iota}$-orbits are in the same $G$-orbit;
\item the parabolic subalgebra of $G$ and the Levi subalgebra  as defined in section 5 of \cite{L1995} corresponding to this $G^{\iota}$-orbit;
\item the number of irreducible $G^{\iota}$-equivariant local systems (up to isomorphism) on this orbit;
\item the set of symbols as in sections 11 to 13 of \cite{L1984}  corresponding to the irreducible $G^{\iota}$-equivariant local systems (up to isomorphism) on this orbit.
\end{itemize}

In \cite{BI2021},  M. Boos and G. Cerulli Irelli, were able to describe the partial order on the $G^{\iota}$-orbits in ${\frak g}_2$, that is when a $G^{\iota}$-orbit ${\mathcal O}$ is in the Zariski closure $\overline{\mathcal O'}$ of the $G^{\iota}$-orbit ${\mathcal O'}$. We will  state their result to be complete using the ${\mathcal I}$-tableaux. 

\subsection{}
In section 1, we will recall the results of Derksen and Weyman on orthogonal and symplectic representations of symmetric quivers.  In section 2, we will construct for each indecomposable such representation: a Jacobson-Morozov triple. In section 3 (respectively section 4), we will study the $G^{\iota}$-orbits in ${\mathfrak g}_2$ when $\vert {\mathcal I} \vert$ is even (respectively odd).  

In the case where $\vert {\mathcal I} \vert$ is odd and the bilinear form $\langle \  , \  \rangle$ is symmetric, we have to be careful. Derksen and Weyman restrict their study to orthogonal isomorphism, but we can also ask what happens if we  consider special orthogonal isomorphism. We look at this in section 5.  In section 6, we study the parabolic subalgebra  (respectively subgroup) of $G$ and the Levi subalgebra   (respectively subgroup) as defined in section 5 of \cite{L1995}. Finally in section 7, we express the number  of irreducible $G^{\iota}$-equivariant local systems (up to isomorphism) on each orbit. We also present examples where we use Lusztig's symbols as in sections 11 to 13 of \cite{L1984} to enumerate these $G^{\iota}$-equivariant irreducible local systems.

\section{Symmetric quivers}

In this section, we will recall some of the results on symmetric quivers and their  orthogonal (resp. symplectic) representations as presented in the article \cite{DW2002} of Derksen and Weyman

\begin{definition}
A {\it quiver} $\Omega = ({\mathcal I}, {\mathcal A}, h: {\mathcal A} \rightarrow {\mathcal I},  t: {\mathcal A} \rightarrow {\mathcal I})$ consists of a finite set ${\mathcal I}$ of vertices, a finite set  ${\mathcal A}$ of arrows and two functions such that, for an arrow $(i \rightarrow j)$ in ${\mathcal A}$, then  $h(i \rightarrow j) = j$ is the head of the arrow and  $t(i \rightarrow j) = i$ is the tail of the arrow.
\end{definition}

\begin{definition} A {\it representation} $(V, \phi)$ of the quiver $\Omega = ({\mathcal I}, {\mathcal A}, h: {\mathcal A} \rightarrow {\mathcal I},  t: {\mathcal A} \rightarrow {\mathcal I})$ consists of a collection of finite-dimensional ${\mathbf k}$-vector spaces $(V_i)_{i \in {\mathcal I}}$ and a collection of linear transformations $\phi_{a}: V_{t(a)} \rightarrow V_{h(a)}$, one for each arrow $a \in  {\mathcal A}$.  
\end{definition}

\begin{definition} \label{Def1.3} A {\it morphism} between two quiver representations 
\[
\psi: (V, \phi) \rightarrow (V', \phi')
\]
 of  $\Omega = ({\mathcal I}, {\mathcal A}, h: {\mathcal A} \rightarrow {\mathcal I},  t: {\mathcal A} \rightarrow {\mathcal I})$
is a collection of linear transformations $\psi_i: V_i \rightarrow V'_i$, one for each $i \in {\mathcal I}$
such that $\phi'_a \circ \psi_{t(a)} = \psi_{h(a)} \circ \phi_a$ for each arrow $a \in {\mathcal A}$. 
We will denote by $V_{\Sigma}$: the {\it direct sum} $\oplus_{i \in {\mathcal I}} V_i$ and by $\psi_{\Sigma}: V_{\Sigma} \rightarrow V'_{\Sigma}$: the unique linear map such that $\psi_{\Sigma}(v) = \psi_i(v)$ for all $i \in {\mathcal I}$ and all $v \in V_i$.

The {\it identity morphism} $Id_{(V, \phi)}: (V, \phi) \rightarrow (V, \phi)$  and the composition of morphisms are defined as expected. 

Two quiver representations $(V, \phi)$ and $(V', \phi')$ are said to be {\it isomorphic} if there exist morphisms $\psi:(V, \phi) \rightarrow (V', \phi')$ and $\psi':(V', \phi') \rightarrow (V, \phi)$ such that $\psi' \circ \psi = Id_{(V, \phi)}$ and $\psi \circ \psi' = Id_{(V', \phi')}$.
\end{definition}

\begin{definition}
The representations of a quiver $\Omega$ form an abelian category with the morphisms as defined above. If $(V, \phi)$ and $(V', \phi')$ are representations of $\Omega$, then their {\it direct sum} $(V, \phi) \oplus (V', \phi')$ is the representation:  $(V \oplus V')_{i \in {\mathcal I}} = (V_i \oplus V'_i)_{i \in {\mathcal I}}$ and 
\[
(\phi \oplus \phi')_a = \phi_a \oplus \phi'_a: V_{t(a)} \oplus V'_{t(a)} \rightarrow V_{h(a)} \oplus V'_{h(a)}
\]
defined by
\[
(v_{t(a)}, v'_{t(a)}) \mapsto  (\phi_a(v_{t(a)}), \phi'_a(v'_{t(a)}))
\]
for each arrow $a \in {\mathcal A}$. A representation $(V, \phi)$ is said to be {\it indecomposable} if it is non-trivial and it is not the direct sum of two non-trivial representations. 
\end{definition}

\begin{definition}
A quiver $\Omega = ({\mathcal I}, {\mathcal A}, h: {\mathcal A} \rightarrow {\mathcal I},  t: {\mathcal A} \rightarrow {\mathcal I})$ with an involution $\sigma$ of the set ${\mathcal I}$ of vertices and of the set ${\mathcal A}$ of arrows such that the orientation of  all the arrows is reversed by $\sigma$ is said to be a {\it symmetric quiver} $(\Omega, \sigma)$.
\end{definition}

\begin{definition}\label{D:IsoDual}
An {\it orthogonal} (respectively {\it symplectic}) {\it representation} $(V, \phi, \langle\  , \  \rangle)$ of a symmetric quiver $(\Omega, \sigma)$ is a representation $(V, \phi)$  of the quiver $\Omega$ with a non-degenerate symmetric (respectively skew-symmetric) bilinear form  $\langle \  ,  \   \rangle$ on $V_{\Sigma}$ such that the restriction of $\langle\  , \  \rangle$ to $V_i \times V_j$ is $0$ if $j \ne \sigma(i)$ and $\langle\phi_a(v), v'\rangle + \langle v, \phi_{\sigma(a)}(v')\rangle = 0$ for all arrows $a \in {\mathcal A}$ and for all  vectors $v \in V_{t(a)}$,  $v' \in V_{\sigma(h(a))}$. 

Using the non-degenerate bilinear form $\langle\ , \  \rangle$ and the map $V_{\Sigma} \rightarrow V^*_{\Sigma}$ given by $v \mapsto \langle v, \  \rangle$, we can identify $V_{\Sigma}$ with  its dual $V^*_{\Sigma}$. 
\end{definition}

\begin{definition}
Given a symmetric quiver $(\Omega, \sigma)$ and a representation $(V, \phi)$ of the quiver $\Omega$, then the {\it dual representation} $(V, \phi)^* =  (V^*, \phi^*)$ is defined as the representation of $\Omega$ where $V^*$ is the collection of finite-dimensional vector spaces $(V^*)_{i} = (V_{\sigma(i)})^*$ for all vertices $i \in {\mathcal I}$ and $(\phi^*)_a = -(\phi_{\sigma(a)})^*$ for all arrows $a \in {\mathcal A}$. 

If $\psi:(V, \phi) \rightarrow (V', \phi')$ is a morphism of representations of the quiver $\Omega$, then its dual is $\psi^*:(V', \phi')^* \rightarrow (V, \phi)^*$ is defined by 
\[
(\psi^*)_i = (\psi_{\sigma(i)})^* :((V')^*)_{i} = (V'_{\sigma(i)})^* \rightarrow (V^*)_i = (V_{\sigma(i)})^* 
\]
for all vertices $i \in {\mathcal I}$.
\end{definition}

\begin{remark}
If $(V, \phi, \langle\  , \  \rangle)$ is an orthogonal (respectively symplectic) representation of the symmetric quiver $(\Omega, \sigma)$, then $(V, \phi)$ and $(V, \phi)^*$ can be identified using the non-degenerate bilinear form $\langle\ , \  \rangle$. Under this identification of $V_{\Sigma}$ with $V_{\Sigma}^*$ given in \ref{D:IsoDual} above, we get $(V^*)_i = (V_{\sigma(i)})^* =V_i$ and also that $(\phi^*)_a = - (\phi_{\sigma(a)})^* = \phi_a$.  This last equation follows from  $\langle\phi_a(v), v'\rangle = -  \langle v, \phi_{\sigma(a)}(v')\rangle$ for all arrows $a \in {\mathcal A}$ and for all vectors $v \in V_{t(a)}$,  $v' \in V_{\sigma(h(a))}$. In other words, an orthogonal (respectively symplectic) representation of a symmetric quiver is always {\it self-dual}.
\end{remark}

\begin{definition}
 Two orthogonal (resp.  symplectic) representations $(V, \phi, \langle\  , \   \rangle)$ and $(V', \phi', \langle\  , \   \rangle')$ of the symmetric quiver $(\Omega, \sigma)$ are {\it isomorphic} if there exists a morphism $\psi: (V, \phi) \rightarrow (V', \phi')$ of representations of the quiver $\Omega$ such that $\psi^* \circ \psi = Id_{(V, \phi)}$ and $\psi \circ \psi^* = Id_{(V', \phi')}$
\end{definition}

\begin{definition}
 If $(V, \phi, \langle\  , \   \rangle)$ and $(V', \phi', \langle\  , \   \rangle')$ are two orthogonal (respectively  symplectic) representations  of the symmetric quiver $(\Omega, \sigma)$, then their {\it direct sum} $(V, \phi, \langle\  , \   \rangle) \oplus (V', \phi', \langle\  , \   \rangle')$ is the orthogonal (resp. symplectic) representation $ = (V'', \phi'', \langle\  , \   \rangle'')$  of the symmetric quiver $(\Omega, \sigma)$, where  $(V'', \phi'')$ is the direct sum $(V, \phi) \oplus (V', \phi')$ as defined in \ref{Def1.3} and  the bilinear form $\langle\  , \   \rangle''$ on $V''_{\Sigma} = V_{\Sigma} \oplus V'_{\Sigma}$ is the direct sum of  $\langle\  , \   \rangle$ and $\langle\  , \   \rangle'$. 
\end{definition}

\begin{definition}
A orthogonal (respectively  symplectic) representation $(V, \phi, \langle\  , \   \rangle)$ is said to be {\it indecomposable} if it is non-trivial and it is not the direct sum of two non-trivial orthogonal (resp. symplectic) representations. 

Note that an indecomposable orthogonal (respectively symplectic) representation of the symmetric quiver $(\Omega, \sigma)$ need not be indecomposable as a representation of the quiver $\Omega$.
\end{definition}

\begin{theorem}\label{T:CaracterisationRepresentation}
Let   $(\Omega, \sigma)$ be a symmetric quiver  and  two orthogonal (resp. symplectic) representations  $(V, \phi, \langle\  , \   \rangle)$, $(V', \phi', \langle\  , \   \rangle')$ of $(\Omega, \sigma)$. Then $(V, \phi, \langle\  , \   \rangle)$ and $(V', \phi', \langle\  , \   \rangle')$ are isomorphic as  orthogonal (resp. symplectic) representations of the symmetric quiver $(\Omega, \sigma)$ if and only if $(V, \phi)$ and $(V', \phi')$ are isomorphic as   representations of the quiver $\Omega$.
\end{theorem}
\begin{proof}
This is theorem 2.6 in \cite{DW2002}. See also theorem 2.6 in \cite{BI2021} for another proof.
\end{proof}

Now to get a hold on the indecomposable orthogonal (resp. symplectic) representations of the symmetric quiver $(\Omega, \sigma)$,  the following proposition also proved in \cite{DW2002} is useful.

\begin{proposition}\label{PropDecomposition}
If $(V, \phi, \langle\  , \   \rangle)$ is an indecomposable orthogonal (respectively  symplectic) representation of the symmetric quiver $(\Omega, \sigma)$, then, as a representation of the quiver $\Omega$, $(V, \phi)$ is either indecomposable or is isomorphic to the direct sum of $(V', \phi')$ and $(V', \phi')^*$, where $(V', \phi')$ is an indecomposable representation of $\Omega$.
\end{proposition}
\begin{proof}
This is proposition 2.7 in \cite{DW2002}.
\end{proof}

\begin{definition}
If $(\Omega, \sigma)$ and $(\Omega', \sigma')$ are two symmetric quivers, where $\Omega = ({\mathcal I}, {\mathcal A}, h: {\mathcal A} \rightarrow {\mathcal I},  t: {\mathcal A} \rightarrow {\mathcal I})$ and $\Omega' = ({\mathcal I}', {\mathcal A}', h': {\mathcal A}' \rightarrow {\mathcal I}',  t': {\mathcal A}' \rightarrow {\mathcal I}')$, then we define the {\it disjoint union} $(\Omega, \sigma) \coprod (\Omega', \sigma')$ of $(\Omega, \sigma)$ and $(\Omega', \sigma')$ as the symmetric quiver $(\Omega'',\sigma'')$ where 
$
\Omega'' = ({\mathcal I} \coprod {\mathcal I}', {\mathcal A} \coprod {\mathcal A}', h \coprod h': {\mathcal A} \coprod {\mathcal A}' \rightarrow {\mathcal I} \coprod {\mathcal I}', t \coprod t': {\mathcal A} \coprod {\mathcal A}' \rightarrow {\mathcal I} \coprod {\mathcal I}')
$
and  the restrictions  of  $\sigma''$ to ${\mathcal I} \coprod {\mathcal I}'$ and ${\mathcal A} \coprod {\mathcal A}'$ are $\sigma$ and $\sigma'$  respectively.

A symmetric quiver is said to be {\it irreducible} if it is not the disjoint union of two non-trivial symmetric quivers.
\end{definition} 

\begin{definition}
A symmetric quiver is said to be of {\it finite type} if it has only finitely many indecomposable orthogonal (respectively symplectic) representations up to isomorphism.
\end{definition}

\begin{proposition}
\begin{enumerate}[\upshape (a)]
\item A symmetric quiver $(\Omega, \sigma)$ is of finite type if and only if the quiver $\Omega$ is of finite type.
\item If $(\Omega, \sigma)$ is an irreducible symmetric quiver, then there are two possibilities: either the quiver $\Omega$ is connected, or $\Omega$ has two connected components which are interchanged by $\sigma$
\item If $(\Omega, \sigma)$ is an irreducible symmetric quiver such that the quiver $\Omega$ has two connected components, then $\Omega$ is isomorphic to the disjoint union of two quivers $\Omega' = ({\mathcal I}', {\mathcal A}', h': {\mathcal A}' \rightarrow {\mathcal I}',  t': {\mathcal A}' \rightarrow {\mathcal I}')$ and $\Omega''= ({\mathcal I}'', {\mathcal A}'', t'': {\mathcal A}'' \rightarrow {\mathcal I}'',  h'': {\mathcal A}'' \rightarrow {\mathcal I}'')$ which  is the dual of $\Omega'$ obtained by reversing all the arrows of $\Omega'$. $\sigma$ is obtained by interchanging the vertices and arrows of $\Omega'$ with the vertices and arrows of $\Omega''$.  Moreover there is a correspondence between quiver representations of $\Omega'$ and orthogonal (respectively  symplectic) representations of the symmetric quiver $(\Omega, \sigma)$.
\item If $(\Omega, \sigma)$ is a symmetric quiver of finite type  such that the quiver $\Omega$ is connected, then $\Omega$ is a quiver of Dynkin type $A_n$.
\end{enumerate}
\end{proposition}
\begin{proof}
(a) is proved in theorem 3.1 in  \cite{DW2002}. (b) follows easily from the action of $\sigma$ of the sets of vertices and arrows.  (c) is proved in proposition 3.2 in  \cite{DW2002}. (d) is proved in proposition 3.3 in  \cite{DW2002}.
\end{proof}

\begin{notation}\label{SymmQuiver}
Let $(\Omega, \sigma)$ be a symmetric quiver of finite type  such that the quiver $\Omega$ is a quiver of Dynkin type $A_n$.  We will be considering the following two possibilities:

\medskip

\underline{\it Even case:}

\medskip

We have $n = 2m$ ,  the set of vertices is 
\[
{\mathcal I} = \{i \in {\mathbb Z} \mid i \equiv 1 \pmod 2, -2m < i < 2m\},
\]
the set of arrows is 
\[
{\mathcal A} = \{ ((i + 2) \leftarrow i) \mid  i \in {\mathcal I},  -(2m - 1) \leq i < (2m - 1)\}
\]
and the symmetric quiver is $(\Omega, \sigma)$ as illustrated below

\bigskip

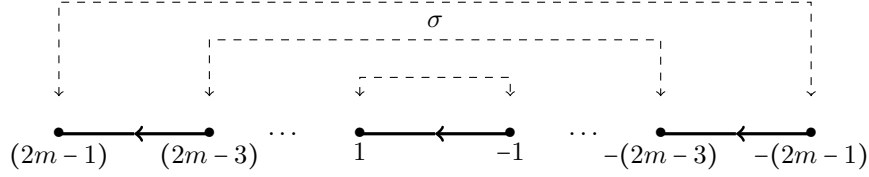
\begin{center}
\begin{tikzpicture}
\draw (0, 0) node {$\bullet$} node[anchor=north]{$(2m - 1)$};
\draw (2, 0) node {$\bullet$} node[anchor=north]{$(2m - 3)$};
\draw [very thick](0,0) -- (1,0);
\draw [<-,  very thick](1,0) -- (2,0);
\draw (3,0) node {$\dots$};
\draw (4, 0) node {$\bullet$} node[anchor=north]{$1$};
\draw [very thick](4,0) -- (5,0);
\draw [<-, very thick](5,0) -- (6,0);
\draw (6,0) node {$\bullet$} node[anchor=north]{$-1$};
\draw (7,0) node {$\dots$};
\draw (8, 0) node {$\bullet$} node[anchor=north]{$-(2m - 3)$};
\draw [very thick](8,0) -- (9,0);
\draw [<-, very thick](9,0) -- (10,0);
\draw (10,0) node {$\bullet$} node[anchor=north]{$-(2m - 1)$};
\draw[<->, dashed] (0, 0.5) -- (0, 1.75) -- (10, 1.75) -- (10, 0.5);
\draw[<->, dashed] (2, 0.5) -- (2, 1.25) -- (8, 1.25) -- (8, 0.5);
\draw[<->, dashed] (4, 0.5) -- (4, 0.75) -- (6, 0.75) -- (6, 0.5);
\draw (5, 1.5) node {$\sigma$};
\end{tikzpicture}
\end{center}
where $\sigma$ is the unique  involution of the graph $A_{2m}$ such that $\sigma(i) = -i$ for all $i \in {\mathcal I}$ and reversing  the orientation of the arrows.  We will  denote this  symmetric quiver by $A_{m}^{even}$ as in \cite{DW2002}. 

\medskip

\underline{\it Odd case:}

\medskip

We have $n = (2m - 1)$ with $m \in {\mathbb N}$ and $m \geq 1$,  the set of vertices is
\[
{\mathcal I}  = \{i \in {\mathbb Z} \mid i \equiv 0 \pmod 2, -2m < i < 2m\},
\]
the set of arrows is 
 \[
 {\mathcal A} = \{ ((i + 2) \leftarrow i) \mid  i \in {\mathcal I},  -(2m - 2) \leq i < (2m - 2)\}
 \]
and the symmetric quiver is $(\Omega, \sigma)$ as illustrated below

\bigskip

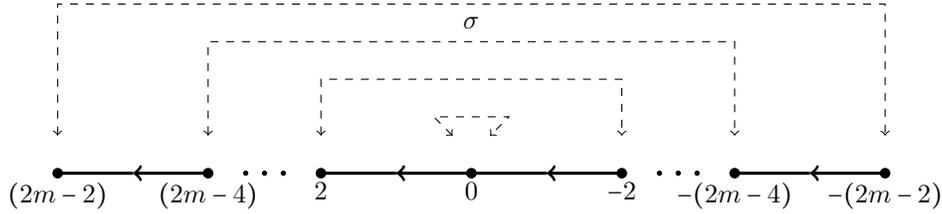
\begin{center}
\begin{tikzpicture}
\fill[black] (0, 0) circle (2pt) node[anchor=north]{$(2m - 2)$};
\fill[black] (2, 0) circle (2pt)node[anchor=north]{$(2m - 4)$};
\draw [very thick](0,0) -- (1,0);
\draw [<-, very thick](1,0) -- (2,0);
\fill[black] (2.5, 0) circle (1pt);
\fill[black] (2.75, 0) circle (1pt);
\fill[black] (3, 0) circle (1pt);
\draw [very thick](3.5, 0) -- (4.5, 0);
\draw[<-, very thick](4.5, 0) -- (5.5, 0);
\fill[black] (3.5, 0) circle (2pt) node[anchor=north]{$2$};
\fill[black] (5.5, 0) circle (2pt) node[anchor=north]{$0$};
\draw [very thick](5.5, 0) -- (6.5, 0);
\draw[<-, very thick](6.5, 0) -- (7.5, 0);
\fill[black] (7.5, 0) circle (2pt) node[anchor=north]{$-2$};
\fill[black] (8, 0) circle (1pt);
\fill[black] (8.25, 0) circle (1pt);
\fill[black] (8.5, 0) circle (1pt);
\draw [very thick](9, 0) -- (10, 0);
\draw[<-, very thick](10, 0) -- (11, 0);
\fill[black] (9, 0) circle (2pt) node[anchor=north]{$-(2m - 4)$};
\fill[black] (11, 0) circle (2pt)node[anchor=north]{$-(2m - 2)$};
\draw[<->, dashed] (0, 0.5) -- (0, 2.25) -- (11, 2.25) -- (11, 0.5);
\draw[<->, dashed] (2, 0.5) -- (2, 1.75) -- (9, 1.75) -- (9, 0.5);
\draw[<->, dashed] (3.5, 0.5) -- (3.5, 1.25) -- (7.5, 1.25) -- (7.5, 0.5);
\draw[<->, dashed] (5.25, 0.5) -- (5, 0.75) -- (6, 0.75) -- (5.75, 0.5);
\draw (5.5, 2) node {$\sigma$};
\end{tikzpicture}
\end{center}
where $\sigma$ is the unique  involution of the graph $A_{2m - 1}$ such that $\sigma(i) = -i$ for all $i \in {\mathcal I}$ and reversing  the orientation of the arrows.  We will  denote this symmetric quiver by $A_{m}^{odd}$ as in \cite{DW2002}. 

For both cases above, we will denote $\{i \in {\mathcal I} \mid i \geq 0\}$ by ${\mathcal I}_+$. So in both cases,  $\vert {\mathcal I}_+ \vert = m$ and 
\[
{\mathcal I}_+ = \begin{cases} \{1, 3, \dots, (2m - 1)\}, &\text{ if we are in case $A_m^{even}$;}\\ \\ \{0, 2, 4, \dots, (2m -2)\}, &\text{  if we are in case $A_m^{odd}$.}\end{cases}
\]

\end{notation}

\begin{remark}
There are also other symmetric quivers, for example alternating quivers. See the introduction in \cite{BI2021} for these alternating examples. For us, we will only  study  homomorphisms $\iota: {\mathbf k}^{\times} \rightarrow G$ related to the symmetric quiver  defined in \ref{SymmQuiver}.

\end{remark}

Derksen and Weyman also gave in \cite{DW2002} the number of indecomposable orthogonal (respectively symplectic) representations of $A_{m}^{even}$  and $A_{m}^{odd}$ up to isomorphism and also their dimension vectors. We will now described their results.

\begin{definition}\label{DefDim}
Given the symmetric quiver $A_{m}^{even}$,  an orthogonal (respectively symplectic) representation $(V, \phi, \langle\  , \   \rangle)$ of $A_{m}^{even}$, where 
\[
{\mathcal I} = \{i \in {\mathbb Z} \mid i \equiv 1 \pmod 2, -2m <  i < 2m\}
\]
and $V$ is the collection of the vector spaces $(V_i)_{i \in {\mathcal I}}$, then we define the {\it dimension vector} of $(V, \phi, \langle\  , \   \rangle)$ as
\[
(\dim(V_i))_{i \in {\mathcal I}_+} = ( \dim(V_{2m - 1}), \dim(V_{2m - 3}), \dots, \dim(V_3), \dim(V_1)) \in {\mathbb N}^m.
\]  

Given the symmetric quiver $A_{m}^{odd}$,  an orthogonal (respectively symplectic) representation $(V, \phi, \langle\  , \   \rangle)$ of $A_{m}^{odd}$, where 
\[ 
{\mathcal I} = \{i \in {\mathbb Z} \mid i \equiv 0 \pmod 2, -2m < i < 2m\}
\]
and $V$ is the collection of the vector spaces $(V_i)_{i \in {\mathcal I}}$, we define the {\it dimension vector} of $(V, \phi, \langle\  , \   \rangle)$ as
\[
(\dim(V_i))_{i \in {\mathcal I}_+} = ( \dim(V_{2m - 2}), \dim(V_{2m - 4}), \dots, \dim(V_2), \dim(V_0)) \in {\mathbb N}^{m}.
\]  
\end{definition}

\begin{proposition}\label{Prop_Indecomposable_Ortho_Aeven}
Let $(\Omega, \sigma)$  be the symmetric quiver $A_{m}^{even}$.
\begin{enumerate}[\upshape (a)]
\item There are $(m + 1)m$ indecomposable orthogonal representations (up to isomorphism) for $(\Omega, \sigma)$. Each occurs in a different dimension. The list of dimension vectors consists of all $m$-tuples that correspond to the dimension of indecomposable representations for the quiver of type $A_m$ and also all nondecreasing $m$-tuples whose last term is 2. 
\item There are $(m + 1)m$ indecomposable symplectic representations (up to isomorphism) for  $(\Omega, \sigma)$. The list of dimension vectors consists of all $m$-tuples that correspond to the dimension of indecomposable representations for the quiver of type $A_m$ and all nondecreasing $m$-tuples whose last term is 2, and which contain at least one term equal to 1. For the dimension $(0,0, \dots, 0, 1, \dots 1)$, there are   two indecomposable symplectic representations (up to isomorphism) and for all other dimensions, there is only  one indecomposable  symplectic representation (up to isomorphism).
\end{enumerate}
\end{proposition}
\begin{proof}
This is proposition 3.8 in \cite{DW2002}.
\end{proof}

\begin{remark}
In proposition~\ref{Prop_Indecomposable_Ortho_Aeven} (a), the indecomposable representations are isomorphic relative to the orthogonal group. We will show later in this section that they are in fact isomorphic relative to the special orthogonal group.
\end{remark}

\begin{example}
The dimension vector $(\dim(V_5), \dim(V_3), \dim(V_1))$ for the indecomposable  orthogonal representations of $A_{3}^{even}$ are:
\[
\begin{matrix}
(1, 0, 0) & (0, 1, 0) &(0, 0, 1) &(1,1,0)\\ (0, 1, 1) & (0, 0, 2) & (1, 1, 1) & (0, 1, 2)\\  (1, 1, 2) & (0, 2, 2) & (1, 2, 2) & (2, 2, 2)
\end{matrix}
\]
with one indecomposable  orthogonal representation of $A_{3}^{even}$ (up to isomorphism) for each dimension vector.
\end{example}

\begin{example}
The dimension vector $(\dim(V_5), \dim(V_3), \dim(V_1))$ for the indecomposable  symplectic representations of $A_{3}^{even}$ are:
\[
\begin{matrix}
(1, 0, 0) & (0, 1, 0) &(0, 0, 1) \\ (1,1,0) & (0, 1, 1) & (1, 1, 1) \\ (0, 1, 2) &  (1, 1, 2) &  (1, 2, 2) 
\end{matrix}
\]
where, for the dimension vectors $(0, 0, 1)$, $(0, 1, 1)$ and $(1, 1, 1)$, there are two indecomposable symplectic representations (up to isomorphism) and for all other dimension vectors, there is only one indecomposable symplectic representation (up to isomorphism).
\end{example}

\begin{proposition}\label{Prop_Indecomposable_Ortho_Aodd}
Let $(\Omega, \sigma)$  be the symmetric quiver $A_{m}^{odd}$.
\begin{enumerate}[\upshape (a)]
\item There are $m^2$ indecomposable orthogonal representations (up to isomorphism) for $(\Omega, \sigma)$. The dimension vector of the indecomposable orthogonal representations of  $A_{m}^{odd}$ correspond naturally to the positive roots of the root system of type $B_m$ with one indecomposable orthogonal representation up to isomorphism for each dimension. If we use the same notation as in Bourbaki \cite{B1971-1975} for the root system $B_m$,  where $\alpha_1, \alpha_2, \dots, \alpha_m$ are the simple roots with $\alpha_m$ being the short root, then the positive root $(c_1 \alpha_1 + c_2 \alpha_2 + \dots + c_m \alpha_m)$ of the root system $B_m$  corresponds to the dimension vector $(c_1, c_2, \dots, c_m)$, more precisely  the corresponding orthogonal representation $V$ is such that $\dim(V_{2(m - i)}) = c_{i}$ for $i = 1, 2, \dots, m$. 
\item There are $m^2$ indecomposable symplectic representations (up to isomorphism) for $(\Omega, \sigma)$. The dimension vector of the indecomposable symplectic representations of  $A_{m}^{odd}$ correspond naturally to the positive roots of the root system of type $C_m$ with one indecomposable symplectic representation up to isomorphism for each dimension. If we use the same notation as in Bourbaki  \cite{B1971-1975}  for the root system $C_m$,  where $\alpha_1, \alpha_2, \dots, \alpha_m$ are the simple roots with $\alpha_m$ being the long root, then the positive root $(c_1 \alpha_1 + c_2 \alpha_2 + \dots + c_m \alpha_m)$ of the root system $C_m$  corresponds to the dimension vector $(c_1, c_2, \dots, 2c_m)$, more precisely  the corresponding symplectic representation $V$ is such that $\dim(V_{2(m - i)}) = c_{i}$ for $i = 1, 2, \dots, (m - 1)$ and $\dim(V_0) = 2c_m$. 
\end{enumerate}
\end{proposition}
\begin{proof}
This is proposition 3.6 in \cite{DW2002}.
\end{proof}

\begin{remark}
In proposition~\ref{Prop_Indecomposable_Ortho_Aodd} (a), the indecomposable representations are isomorphic relative to the orthogonal group. We will  later describe  how an indecomposable  orthogonal representation  splits  into distinct  special orthogonal isomorphism classes and we will give a representative for each of these special orthogonal isomorphism classes.
\end{remark}

\begin{example}
The dimension vector $(\dim(V_4), \dim(V_2), \dim(V_0))$ for the indecomposable  orthogonal representations of $A_{3}^{odd}$ are:
\[
\begin{matrix}
(1, 0, 0) & (0, 1, 0) &(0, 0, 1) \\ (1,1,0) & (0, 1, 2) & (1, 1, 2) \\ (0, 1, 1) & (1, 2, 2) &  (1, 1, 1)
\end{matrix}
\]
with  one indecomposable  orthogonal representations of $A_{3}^{odd}$ (up to isomorphism) for each dimension.
\end{example}

\begin{example}
The dimension vector $(\dim(V_4), \dim(V_2), \dim(V_0))$ for the indecomposable  symplectic representations of $A_{3}^{odd}$ are:
\[
\begin{matrix}
(1, 0, 0) & (0, 1, 0) &(0, 0, 2) \\ (1,1,0) & (0, 1, 2) & (1, 1, 2) \\ (0, 2, 2) &  (1, 2, 2) &  (2, 2, 2) 
\end{matrix}
\]
with  one indecomposable  symplectic representations of $A_{3}^{odd}$ (up to isomorphism) for each dimension.
\end{example}

\subsection{}
Later we will need to know how an indecomposable orthogonal or symplectic representation $(V, \phi, \langle\ , \ \rangle)$ of the symmetric quiver $A_m^{even}$ (respectively  $A_m^{odd}$) decompose  as a representation of the quiver $A_{2m}$ (respectively $A_{2m - 1}$). By proposition \ref{PropDecomposition}, we know that $(V, \phi)$ will either be indecomposable or isomorphic to the direct sum of $(V', \phi')$ and $(V', \phi')^*$, where $(V', \phi')$ is an indecomposable representation of $A_{2m}$ (respectively $A_{2m - 1}$). The indecomposable representations of $A_{2m}$ (respectively $A_{2m - 1}$)  are well-know and their dimension vectors are  the positive roots of $A_{2m}$ (respectively $A_{2m - 1}$) up to isomorphism. It is easy in most cases to know what this decomposition of $(V, \phi)$ will be just from its dimension vector and the fact that,  in the case where $(V, \phi)$ as a representation of $A_{2m}$ (respectively $A_{2m - 1}$) is the direct sum of $(V', \phi')$ and $(V', \phi')^*$, then the dimension of $(V', \phi')^*$ is obtained from $(V', \phi')$ by the use of the involution $\sigma$. We will now enumerate these decompositions.

\begin{proposition}\label{Decomposition_even}
Using the same notation as in \ref{DefDim}, let $(V, \phi, \langle\ , \ \rangle)$ be an indecomposable orthogonal (respectively symplectic) representation of the symmetric quiver $A_m^{even}$ and denote its dimension  vector  by 
\[
\alpha = (\alpha_i)_{i \in {\mathcal I}_+}  = (\alpha_{(2m - 1)}, \alpha_{(2m - 3)}, \dots, \alpha_3, \alpha_1) \in {\mathbb N}^m,
\]
where $\dim(V_i) = \alpha_i$ for all $i \in {\mathcal I}_+$. Recall that  
\[
{\mathcal I} = \{i \in {\mathbb Z} \mid i \equiv 1 \pmod 2,  -2m < i < 2m\}  \text{ and }  {\mathcal I}_+ = \{i \in {\mathcal I} \mid i \geq 0\}.
\]  
\begin{enumerate}[\upshape (a)] 
\item In the orthogonal case, as a representation of $A_{2m}$, we have that $(V, \phi)$ decomposes as follows.

\begin{enumerate} [\upshape (i)]
\item If the dimension vector $\alpha= (\alpha_i)_{i \in {\mathcal I}_+} \in {\mathbb N}^m$  is such that 
\[
\alpha_i = \begin{cases} 1, &\text{if $i \in {\mathcal I}$ and $a \leq i \leq b$;} \\ 0, &\text{otherwise;} \end{cases}
\]
where $a, b \in {\mathcal I}$ and $1 < a \leq b$, then $(V, \phi)$  is isomorphic to the  direct sum of two  indecomposable representations of $A_{2m}$, whose dimension vectors are $\beta = (\beta_i)_{i \in {\mathcal I}}  \in {\mathbb N}^{2m}$ and $\beta' = (\beta'_i)_{i \in {\mathcal I}} \in {\mathbb N}^{2m}$,
where 
\[
\beta_i = \begin{cases} 1, &\text{if $i \in {\mathcal I}$ and $-b \leq i \leq -a$;}\\ 0, &\text{otherwise;} \end{cases}
\]
and
\[
\beta'_i = \begin{cases} 1, &\text{if $i \in {\mathcal I}$ and $a \leq i \leq b$;}\\ 0, &\text{otherwise.} \end{cases}
\]

\item If the dimension  vector $\alpha= (\alpha_i)_{i \in {\mathcal I}_+} \in {\mathbb N}^m$  is such that 
\[
\alpha_i = \begin{cases} 1, &\text{if $i \in {\mathcal I}$ and $0 < i \leq a$;} \\ 0, &\text{otherwise;} \end{cases}
\]
where $a \in {\mathcal I}$ and $0 < a$, then $(V, \phi)$  is isomorphic to the  direct sum of two indecomposable representations of $A_{2m}$, whose dimension vectors are
 $\beta = (\beta_i)_{i \in {\mathcal I}} \in {\mathbb N}^{2m}$ and $\beta' = (\beta'_i)_{i \in {\mathcal I}} \in {\mathbb N}^{2m}$,
where 
\[
\beta_i = \begin{cases} 1, &\text{if $i \in {\mathcal I}$ and $-a \leq i < 0$;}\\ 0, &\text{otherwise;} \end{cases}
\]
and
\[
\beta'_i = \begin{cases} 1, &\text{if $i \in {\mathcal I}$ and $0 < i \leq a$;}\\ 0, &\text{otherwise.} \end{cases}
\]

\item If the dimension vector $\alpha= (\alpha_i)_{i \in {\mathcal I}_+} \in {\mathbb N}^m$  is such that 
\[
\alpha_i = \begin{cases} 2, &\text{if $i \in {\mathcal I}$ and $0 < i \leq a$;} \\ 1, &\text{if $i \in {\mathcal I}$ and $a < i \leq b$;}\\ 0, &\text{otherwise;} \end{cases}
\]
where $a, b \in {\mathcal I}$ and $0 <  a < b$, then $(V, \phi)$  is isomorphic to the  direct sum of two  indecomposable representations of $A_{2m}$, whose dimension vectors are $\beta = (\beta_i)_{i \in {\mathcal I}} \in {\mathbb N}^{2m}$ and $\beta' = (\beta'_i)_{i \in {\mathcal I}} \in {\mathbb N}^{2m}$,
where 
\[
\beta_i = \begin{cases} 1, &\text{if $i \in {\mathcal I}$ and $-a \leq i \leq b$;}\\ 0, &\text{otherwise;} \end{cases}
\]
and
\[
\beta'_i = \begin{cases} 1, &\text{if $i \in {\mathcal I}$ and $-b \leq i \leq  a$;}\\ 0, &\text{otherwise.} \end{cases}
\]

\item If the dimension  vector $\alpha= (\alpha_i)_{i \in {\mathcal I}_+} \in {\mathbb N}^m$  is such that 
\[
\alpha_i = \begin{cases} 2, &\text{if $i \in {\mathcal I}$ and $0 < i \leq a$;} \\ 0, &\text{otherwise;} \end{cases}
\]
where $a \in {\mathcal I}$ and $0 < a$, then $(V, \phi)$  is isomorphic to the  direct sum of two indecomposable representations of $A_{2m}$, whose dimension vectors are
 $\beta = (\beta_i)_{i \in {\mathcal I}} \in {\mathbb N}^{2m}$ and $\beta' = (\beta'_i)_{i \in {\mathcal I}} \in {\mathbb N}^{2m}$
where 
\[
\beta_i = \begin{cases} 1, &\text{if $i \in {\mathcal I}$ and $-a \leq i \leq a$;}\\ 0, &\text{otherwise;} \end{cases}
\]
and
\[
\beta'_i = \begin{cases} 1, &\text{if $i \in {\mathcal I}$ and $-a \leq i \leq a$;}\\ 0, &\text{otherwise.} \end{cases}
\]
These two indecomposable representations of $A_{2m}$ are isomorphic.

\end{enumerate}

\item In the symplectic case, as a representation of $A_{2m}$, we have that $(V, \phi)$ decomposes as follows.

\begin{enumerate} [\upshape (i)]
\item If the dimension vector $\alpha= (\alpha_i)_{i \in {\mathcal I}_+} \in {\mathbb N}^m$  is such that 
\[
\alpha_i = \begin{cases} 1, &\text{if $i \in {\mathcal I}$ and $a \leq i \leq b$;} \\ 0, &\text{otherwise;} \end{cases}
\]
where $a, b \in {\mathcal I}$ and $1 < a \leq b$, then $(V, \phi)$  is isomorphic to the  direct sum of two  indecomposable representations of $A_{2m}$, whose dimension vectors are $\beta = (\beta_i)_{i \in {\mathcal I}}$ and $\beta' = (\beta'_i)_{i \in {\mathcal I}} \in {\mathbb N}^{2m}$
where 
\[
\beta_i = \begin{cases} 1, &\text{if $i \in {\mathcal I}$ and $-b \leq i \leq -a$;}\\ 0, &\text{otherwise;} \end{cases}
\]
and
\[
\beta'_i = \begin{cases} 1, &\text{if $i \in {\mathcal I}$ and $a \leq i \leq b$;}\\ 0, &\text{otherwise.} \end{cases}
\]

\item If the dimension vector   $\alpha= (\alpha_i)_{i \in {\mathcal I}_+} \in {\mathbb N}^m$  is such that 
\[
\alpha_i = \begin{cases} 1, &\text{if $i \in {\mathcal I}$ and $0 < i \leq a$;} \\ 0, &\text{otherwise;} \end{cases}
\]
where $a \in {\mathcal I}$ and $0 < a$, then there are two isomorphism classes of symplectic representations with this dimension vector, they are distinguished by the fact that in one class, the map $\phi_{(-1 \rightarrow 1)}: V_{-1} \rightarrow V_1$ is $0$ and in the other, it is $\ne 0$. We will distinguished these two by writing for their dimension vectors: $\alpha^0$ for the one where $\phi_{(-1 \rightarrow 1)}: V_{-1} \rightarrow V_1$ is $0$  and $\alpha^1$ for the one where $\phi_{(-1 \rightarrow 1)}: V_{-1} \rightarrow V_1$ is $\ne 0$.  

In the case of the isomorphism class of the symplectic representations corresponding to $\alpha^0$, then $(V, \phi)$  is isomorphic to the  direct sum of two indecomposable representations of $A_{2m}$, whose dimension vectors are
 $\beta = (\beta_i)_{i \in {\mathcal I}} \in {\mathbb N}^{2m}$ and $\beta' = (\beta'_i)_{i \in {\mathcal I}} \in {\mathbb N}^{2m}$
where 
\[
\beta_i = \begin{cases} 1, &\text{if $i \in {\mathcal I}$ and $-a \leq i < 0$;}\\ 0, &\text{otherwise;} \end{cases}
\]
and
\[
\beta'_i = \begin{cases} 1, &\text{if $i \in {\mathcal I}$ and $0 < i \leq a$;}\\ 0, &\text{otherwise.} \end{cases}
\]

In the case of the isomorphism class of the symplectic representations corresponding to $\alpha^1$, then $(V, \phi)$  is isomorphic to the  indecomposable representation of $A_{2m}$, whose dimension vector is
 $\beta = (\beta_i)_{i \in {\mathcal I}} \in {\mathbb N}^{2m}$
where 
\[
\beta_i = \begin{cases} 1, &\text{if $i \in {\mathcal I}$ and $-a  \leq i \leq a$;}\\ 0, &\text{otherwise.} \end{cases}
\]

\item If the dimension vector $\alpha= (\alpha_i)_{i \in {\mathcal I}_+} \in {\mathbb N}^m$  is such that 
\[
\alpha_i = \begin{cases} 2, &\text{if $i \in {\mathcal I}$ and $0 < i \leq a$;} \\ 1, &\text{if $i \in {\mathcal I}$ and $a < i \leq b$;}\\ 0, &\text{otherwise;} \end{cases}
\]
where $a, b \in {\mathcal I}$ and $0 <  a < b$, then $(V, \phi)$  is isomorphic to the  direct sum of two  indecomposable representations of $A_{2m}$, whose dimension vectors are $\beta = (\beta_i)_{i \in {\mathcal I}} \in {\mathbb N}^{2m}$ and $\beta' = (\beta'_i)_{i \in {\mathcal I}} \in {\mathbb N}^{2m}$
where 
\[
\beta_i = \begin{cases} 1, &\text{if $i \in {\mathcal I}$ and $-a \leq i \leq b$;}\\ 0, &\text{otherwise;} \end{cases}
\]
and
\[
\beta'_i = \begin{cases} 1, &\text{if $i \in {\mathcal I}$ and $-b \leq i \leq  a$;}\\ 0, &\text{otherwise.} \end{cases}
\]
\end{enumerate}
\end{enumerate}
\end{proposition}

\begin{proof}
(a) In the orthogonal case, (i)  follows easily by considering just the dimension vector. 

For (ii), we must consider the map $\phi_{(-1 \rightarrow 1)}:V_{-1} \rightarrow V_1$. We can pick bases $\{v_1\}$ and $\{v_{-1}\}$ of $V_1$ and $V_{-1}$ respectively and we can assume that our choice is such that $\langle v_1, v_{-1}\rangle = \langle v_{-1}, v_{1}\rangle = 1$. Write $\phi_{(-1\rightarrow 1)}(v_{-1}) = xv_1$, where $x \in {\mathbf k}$. Because $(V, \phi, \langle\  , \  \rangle)$ is an orthogonal representation, we have
\[
\langle \phi_{-1\rightarrow 1}(v_{-1}), v_{-1}\rangle + \langle v_{-1}, \phi_{-1\rightarrow 1}(v_{-1})\rangle = 0 \Rightarrow 2x = 0.
\]
So $\phi_{(-1 \rightarrow 1)}$ is $0$ and (ii) follows from this observation.

For (iii) and (iv), we again consider the map $\phi_{(-1 \rightarrow 1)}:V_{-1} \rightarrow V_1$. We can pick bases $\{u_1, v_1\}$ and $\{u_{-1}, v_{-1}\}$ of $V_1$ and $V_{-1}$ respectively and we can assume that our choice is such that 
\[
\begin{aligned}
\langle u_1, u_{-1}\rangle &= \langle u_{-1}, u_{1}\rangle &= 1, \quad \quad \langle v_1, v_{-1}\rangle &= \langle v_{-1}, v_{1}\rangle &= 1,\\ \langle u_1, v_{-1}\rangle &= \langle v_{-1}, u_{1}\rangle &= 0, \quad \quad \langle v_1, u_{-1}\rangle &= \langle u_{-1}, v_{1}\rangle &= 0.
\end{aligned}
\]
 Write $\phi_{(-1\rightarrow 1)}(u_{-1}) = xu_1 + yv_1$ and $\phi_{(-1\rightarrow 1)}(v_{-1}) = zu_1 + wv_1$ where $x, y, z, w \in {\mathbf k}$. Because $(V, \phi, \langle\  , \  \rangle)$ is an orthogonal representation, we have
\[
\begin{aligned}
\langle \phi_{-1\rightarrow 1}(u_{-1}), u_{-1}\rangle + \langle u_{-1}, \phi_{-1\rightarrow 1}(u_{-1})\rangle &= 0 \Rightarrow 2x = 0\\
\langle \phi_{-1\rightarrow 1}(u_{-1}), v_{-1}\rangle + \langle u_{-1}, \phi_{-1\rightarrow 1}(v_{-1})\rangle &= 0 \Rightarrow y + z = 0\\
\langle \phi_{-1\rightarrow 1}(v_{-1}), u_{-1}\rangle + \langle v_{-1}, \phi_{-1\rightarrow 1}(u_{-1})\rangle &= 0 \Rightarrow z + y = 0\\
\langle \phi_{-1\rightarrow 1}(v_{-1}), v_{-1}\rangle + \langle v_{-1}, \phi_{-1\rightarrow 1}(v_{-1})\rangle &= 0 \Rightarrow 2w = 0\\
\end{aligned}
\]
If $z = 0$, then $\phi_{(-1 \rightarrow 1)}$ is $0$ and consequently $(V, \phi, \langle\ , \ \rangle)$ would not be indecomposable. Thus $z \ne 0$ and $y \ne 0$ and (iii) and (iv) follow from this observation.

(b) In the symplectic case,  (i) follows easily by considering just the dimension vector. 

In case (ii), we will consider the map $\phi_{(-1 \rightarrow 1)}:V_{-1} \rightarrow V_1$. We can pick bases $\{v_1\}$ and $\{v_{-1}\}$ of $V_1$ and $V_{-1}$ respectively and we can assume that our choice is such that $\langle v_1, v_{-1}\rangle = 1$ and $\langle v_{-1}, v_{1}\rangle = -1$. Write $\phi_{(-1\rightarrow 1)}(v_{-1}) = xv_1$, where $x \in {\mathbf k}$. Because $(V, \phi, \langle\  , \  \rangle)$ is an symplectic representation, we have
\[
\langle \phi_{-1\rightarrow 1}(v_{-1}), v_{-1}\rangle + \langle v_{-1}, \phi_{-1\rightarrow 1}(v_{-1})\rangle = \langle xv_1, v_{-1}\rangle + \langle v_{-1}, xv_1\rangle = x - x = 0.
\]
So $\phi_{(-1 \rightarrow 1)}$ can be $0$ or $\ne 0$. If $\phi_{(-1 \rightarrow 1)} \ne 0$, then  we get easily that $(V, \phi)$  is isomorphic to the  indecomposable representation of $A_{2m}$ whose dimension vector is
 $\beta = (\beta_i)_{i \in {\mathcal I}} \in {\mathbb N}^{2m}$
where 
\[
\beta_i = \begin{cases} 1, &\text{if $i \in {\mathcal I}$ and $-a  \leq i \leq a$;}\\ 0, &\text{otherwise.} \end{cases}
\]

If $\phi_{(-1 \rightarrow 1)} = 0$ then  we get easily that $(V, \phi)$  is isomorphic to the  direct sum of two indecomposable representations of $A_{2m}$, whose dimension vectors are
 $\beta = (\beta_i)_{i \in {\mathcal I}} \in {\mathbb N}^{2m}$ and $\beta' = (\beta'_i)_{i \in {\mathcal I}} \in {\mathbb N}^{2m}$
where 
\[
\beta_i = \begin{cases} 1, &\text{if $i \in {\mathcal I}$ and $-a \leq i < 0$;}\\ 0, &\text{otherwise;} \end{cases}
\]
and
\[
\beta'_i = \begin{cases} 1, &\text{if $i \in {\mathcal I}$ and $0 < i \leq a$;}\\ 0, &\text{otherwise.} \end{cases}
\]
So (ii) follows from this observation.

As for (iii), there are two possibilities: either
\[
\beta_i = \begin{cases} 1, &\text{if $i \in {\mathcal I}$ and $-a \leq i \leq a$;}\\ 0, &\text{otherwise;} \end{cases}
\quad \text{ and } \quad
\beta'_i = \begin{cases} 1, &\text{if $i \in {\mathcal I}$ and $-b \leq i \leq  b$;}\\ 0, &\text{otherwise;} \end{cases}
\]
or 
\[
\beta_i = \begin{cases} 1, &\text{if $i \in {\mathcal I}$ and $-a \leq i \leq b$;}\\ 0, &\text{otherwise;} \end{cases}
\quad \text{ and } \quad
\beta'_i = \begin{cases} 1, &\text{if $i \in {\mathcal I}$ and $-b \leq i \leq  a$;}\\ 0, &\text{otherwise.} \end{cases}
\]
The first of these possibilities has to be excluded. Otherwise $(V, \phi, \langle\ , \ \rangle)$ would be isomorphic by theorem~\ref{T:CaracterisationRepresentation} as a symplectic representation to the direct sum of an indecomposable symplectic representation with dimension $\beta = (\beta_i)_{i \in {\mathcal I}_+}$ and an indecomposable symplectic representation with dimension $\beta' = (\beta'_i)_{i \in {\mathcal I}_+}$. But this means that $(V, \phi, \langle\ , \ \rangle)$  is not an indecomposable representation of the symmetric quiver $A_m^{even}$. This proves (iii).

\end{proof}

\begin{proposition}\label{Decomposition_odd}
Using the same notation as in \ref{DefDim},  let $(V, \phi, \langle\ , \ \rangle)$ be an indecomposable orthogonal (respectively symplectic) representation of the symmetric quiver $A_m^{odd}$ and denote its dimension vector by 
\[
\alpha= (\alpha_i)_{i \in {\mathcal I}_+} = (\alpha_{2(m - 1)}, \alpha_{2(m - 2)} \dots, \alpha_2, \alpha_0) \in {\mathbb N}^m, 
\]
where $\dim(V_i) = \alpha_i$ for all $i \in {\mathcal I}_+$.  Recall that  
\[
{\mathcal I} = \{i \in {\mathbb Z} \mid i \equiv 0 \pmod 2,  -2m <  i < 2m\}  \text{ and }  {\mathcal I}_+ = \{i \in {\mathcal I} \mid i \geq 0\}.
\]  
\begin{enumerate}[\upshape (a)] 
\item In the orthogonal case, as a representation of $A_{2m - 1}$, we have that $(V, \phi)$ decomposes as follows.

\begin{enumerate} [\upshape (i)]
\item If the dimension vector $\alpha= (\alpha_i)_{i \in {\mathcal I}_+} \in {\mathbb N}^m$  is 
\[
\alpha_i = \begin{cases} 1, &\text{if $i \in {\mathcal I}$ and $a \leq i \leq b$;} \\ 0, &\text{otherwise;} \end{cases}
\]
where $a, b \in {\mathcal I}$ and $0 < a \leq b$, then $(V, \phi)$  is isomorphic to the  direct sum of two  indecomposable representations of $A_{2m - 1}$ whose dimension vectors are $\beta = (\beta_i)_{i \in {\mathcal I}}  \in {\mathbb N}^{2m - 1}$ and $\beta' = (\beta'_i)_{i \in {\mathcal I}} \in {\mathbb N}^{2m - 1}$
where 
\[
\beta_i = \begin{cases} 1, &\text{if $i \in {\mathcal I}$ and $-b \leq i \leq -a$;}\\ 0, &\text{otherwise.} \end{cases}
\]
and
\[
\beta'_i = \begin{cases} 1, &\text{if $i \in {\mathcal I}$ and $a \leq i \leq b$.}\\ 0, &\text{otherwise;} \end{cases}
\]

\item If the dimension  vector $\alpha= (\alpha_i)_{i \in {\mathcal I}_+} \in {\mathbb N}^m$  is
\[
\alpha_i = \begin{cases} 1, &\text{if $i \in {\mathcal I}$ and $0 \leq i \leq a$;} \\ 0, &\text{otherwise;} \end{cases}
\]
where $a \in {\mathcal I}$ and $0 \leq a$, then $(V, \phi)$  is isomorphic to the  indecomposable representation of $A_{2m - 1}$ whose dimension vector is
 $\beta = (\beta_i)_{i \in {\mathcal I}} \in {\mathbb N}^{2m - 1}$
where 
\[
\beta_i = \begin{cases} 1, &\text{if $i \in {\mathcal I}$ and $-a  \leq i \leq a$;}\\ 0, &\text{otherwise.} \end{cases}
\]

\item If the dimension vector $\alpha= (\alpha_i)_{i \in {\mathcal I}_+} \in {\mathbb N}^m$  is 
\[
\alpha_i = \begin{cases} 2, &\text{if $i \in {\mathcal I}$ and $0 \leq i \leq a$;} \\ 1, &\text{if $i \in {\mathcal I}$ and $a < i \leq b$;}\\ 0, &\text{otherwise;} \end{cases}
\]
where $a, b \in {\mathcal I}$ and $0 \leq  a < b$, then $(V, \phi)$  is isomorphic to the  direct sum of two  indecomposable representations of $A_{2m - 1}$ whose dimension vectors are $\beta = (\beta_i)_{i \in {\mathcal I}} \in {\mathbb N}^{2m - 1}$ and $\beta' = (\beta'_i)_{i \in {\mathcal I}} \in {\mathbb N}^{2m - 1}$
where 
\[
\beta_i = \begin{cases} 1, &\text{if $i \in {\mathcal I}$ and $-a \leq i \leq b$;}\\ 0, &\text{otherwise;} \end{cases}
\]
and
\[
\beta'_i = \begin{cases} 1, &\text{if $i \in {\mathcal I}$ and $-b \leq i \leq  a$;}\\ 0, &\text{otherwise.} \end{cases}
\]

\end{enumerate}

\item In the symplectic case, as a representation of $A_{2m - 1}$, we have that $(V, \phi)$ decomposes as follows.

\begin{enumerate} [\upshape (i)]
\item If the dimension vector $\alpha= (\alpha_i)_{i \in {\mathcal I}_+} \in {\mathbb N}^m$  is 
\[
\alpha_i = \begin{cases} 1, &\text{if $i \in {\mathcal I}$ and $a \leq i \leq b$;} \\ 0, &\text{otherwise;} \end{cases}
\]
where $a, b \in {\mathcal I}$ and $0 < a \leq b$, then $(V, \phi)$  is isomorphic to the  direct sum of two  indecomposable representations of $A_{2m - 1}$ whose dimension vectors are $\beta = (\beta_i)_{i \in {\mathcal I}} \in {\mathbb N}^{2m - 1}$ and $\beta' = (\beta'_i)_{i \in {\mathcal I}} \in {\mathbb N}^{2m - 1}$
where 
\[
\beta_i = \begin{cases} 1, &\text{if $i \in {\mathcal I}$ and $-b \leq i \leq -a$;}\\ 0, &\text{otherwise;} \end{cases}
\]
and
\[
\beta'_i = \begin{cases} 1, &\text{if $i \in {\mathcal I}$ and $a \leq i \leq b$;}\\ 0, &\text{otherwise.} \end{cases}
\]

\item If the dimension  vector $\alpha= (\alpha_i)_{i \in {\mathcal I}_+} \in {\mathbb N}^m$  is 
\[
\alpha_i = \begin{cases} 2, &\text{if $i \in {\mathcal I}$ and $0 \leq i \leq a$;} \\ 0, &\text{otherwise;} \end{cases}
\]
where $a \in {\mathcal I}$ and $0 \leq a$, then $(V, \phi)$  is isomorphic to the  direct sum of two  indecomposable representations of $A_{2m - 1}$ whose dimension vectors are $\beta = (\beta_i)_{i \in {\mathcal I}}  \in {\mathbb N}^{2m - 1}$ and $\beta' = (\beta'_i)_{i \in {\mathcal I}} \in {\mathbb N}^{2m - 1}$
where 
\[
\beta_i = \begin{cases} 1, &\text{if $i \in {\mathcal I}$ and $-a \leq i \leq a$;}\\ 0, &\text{otherwise;} \end{cases}
\]
and
\[
\beta'_i = \begin{cases} 1, &\text{if $i \in {\mathcal I}$ and $-a \leq i \leq  a$;}\\ 0, &\text{otherwise.} \end{cases}
\]
These two indecomposable representations of $A_{2m - 1}$ are isomorphic.

\item If the dimension vector $\alpha= (\alpha_i)_{i \in {\mathcal I}_+} \in {\mathbb N}^m$  with 
\[
\alpha_i = \begin{cases} 2, &\text{if $i \in {\mathcal I}$ and $0 \leq i \leq a$;} \\ 1, &\text{if $i \in {\mathcal I}$ and $a < i \leq b$;}\\ 0, &\text{otherwise;} \end{cases}
\]
where $a, b \in {\mathcal I}$ and $0 \leq a < b$, then $(V, \phi)$  is isomorphic to the  direct sum of two  indecomposable representations of $A_{2m - 1}$ whose dimension vectors are $\beta = (\beta_i)_{i \in {\mathcal I}}  \in {\mathbb N}^{2m - 1}$ and $\beta' = (\beta'_i)_{i \in {\mathcal I}} \in {\mathbb N}^{2m - 1}$
where 
\[
\beta_i = \begin{cases} 1, &\text{if $i \in {\mathcal I}$ and $-a \leq i \leq b$;}\\ 0, &\text{otherwise;} \end{cases}
\]
and
\[
\beta'_i = \begin{cases} 1, &\text{if $i \in {\mathcal I}$ and $-b \leq i \leq  a$;}\\ 0, &\text{otherwise.} \end{cases}
\]
\end{enumerate}
\end{enumerate}
\end{proposition}

\begin{proof}
(a) In the orthogonal case, (i) follows easily by considering just the dimension vector. 

For (ii), we must consider the maps $\phi_{(-2 \rightarrow 0)}:V_{-2} \rightarrow V_0$ and $\phi_{(0 \rightarrow 2)}:V_{0} \rightarrow V_{2}$. We can pick bases $\{v_2\}$, $\{v_0\}$ and $\{v_{-2}\}$ of $V_2$,  $V_0$ and $V_{-2}$ respectively and we can assume that our choice is such that $\langle v_2, v_{-2}\rangle = \langle v_{-2}, v_{2}\rangle = 1$ and $\langle v_0, v_{0}\rangle = 1$ and $\langle v_2, v_{0}\rangle = \langle v_{0}, v_{2}\rangle = \langle v_{-2}, v_{0}\rangle = \langle v_{0}, v_{-2}\rangle  =   \langle v_{2}, v_{2}\rangle  =  \langle v_{-2}, v_{-2}\rangle = 0$. Write $\phi_{(-2\rightarrow 0)}(v_{-2}) = xv_0$ and $\phi_{(0\rightarrow 2)}(v_{0}) = yv_2$, where $x, y \in {\mathbf k}$. Because $(V, \phi, \langle\  , \  \rangle)$ is an orthogonal representation, we have
\[
\langle \phi_{-2\rightarrow 0}(v_{-2}), v_{0}\rangle + \langle v_{-2}, \phi_{0\rightarrow 2}(v_{0})\rangle = 0 \Rightarrow x + y = 0.
\]
If $x = 0$, then $y = 0$ and both $\phi_{(-2 \rightarrow 0)}:V_{-2} \rightarrow V_0$ and $\phi_{(0 \rightarrow 2)}:V_{0} \rightarrow V_{2}$ are $0$. But this is impossible because it contradicts the fact that 
$(V, \phi, \langle\ , \ \rangle)$ is indecomposable. 
So $x \ne 0$, $y \ne 0$ and $\phi_{(-2 \rightarrow 0)}:V_{-2} \rightarrow V_0$ and $\phi_{(0 \rightarrow 2)}:V_{0} \rightarrow V_{2}$ are $\ne 0$. From this observation, we get that  (ii) follows.

As for (iii), there are two possibilities: either
\[
\beta_i = \begin{cases} 1, &\text{if $i \in {\mathcal I}$ and $-a \leq i \leq a$;}\\ 0, &\text{otherwise;} \end{cases}
\quad \text{ and } \quad
\beta'_i = \begin{cases} 1, &\text{if $i \in {\mathcal I}$ and $-b \leq i \leq  b$;}\\ 0, &\text{otherwise;} \end{cases}
\]
or 
\[
\beta_i = \begin{cases} 1, &\text{if $i \in {\mathcal I}$ and $-a \leq i \leq b$;}\\ 0, &\text{otherwise;} \end{cases}
\quad \text{ and } \quad
\beta'_i = \begin{cases} 1, &\text{if $i \in {\mathcal I}$ and $-b \leq i \leq  a$;}\\ 0, &\text{otherwise.} \end{cases}
\]
The first of these possibilities has to be excluded. Otherwise $(V, \phi, \langle\ , \ \rangle)$ would be isomorphic by theorem~\ref{T:CaracterisationRepresentation} as an orthogonal representation to the direct sum of an indecomposable orthogonal representation with dimension $\beta = (\beta_i)_{i \in {\mathcal I}_+}$ and an indecomposable orthogonal representation with dimension $\beta' = (\beta'_i)_{i \in {\mathcal I}_+}$. But this means that $(V, \phi, \langle\ , \ \rangle)$  is not an indecomposable representation of the symmetric quiver $A_m^{odd}$. This prove (iii).

(b) In the symplectic case,  (i) and (ii) follow easily by considering just the dimension vectors. 

As for (iii), because $a < b$, there are two distinct possibilities: either
\[
\beta_i = \begin{cases} 1, &\text{if $i \in {\mathcal I}$ and $-a \leq i \leq a$;}\\ 0, &\text{otherwise;} \end{cases}
\quad \text{ and } \quad
\beta'_i = \begin{cases} 1, &\text{if $i \in {\mathcal I}$ and $-b \leq i \leq  b$;}\\ 0, &\text{otherwise;} \end{cases}
\]
or 
\[
\beta_i = \begin{cases} 1, &\text{if $i \in {\mathcal I}$ and $-a \leq i \leq b$;}\\ 0, &\text{otherwise;} \end{cases}
\quad \text{ and } \quad
\beta'_i = \begin{cases} 1, &\text{if $i \in {\mathcal I}$ and $-b \leq i \leq  a$;}\\ 0, &\text{otherwise.} \end{cases}
\]
The first possibility has to be excluded because $a < b$, the proposition~\ref{PropDecomposition} and the fact that $\sigma(\beta_i)_{i \in {\mathcal I}} = (\beta_i)_{i \in {\mathcal I}}$ and $\sigma(\beta'_i)_{i \in {\mathcal I}} = (\beta'_i)_{i \in {\mathcal I}}$. 
This proves the claim.
\end{proof}

\begin{lemma}\label{L:DetThetaAOdd} 
Let $(V, \phi, \langle\ , \ \rangle)$ and $(V', \phi', \langle\ , \ \rangle')$ be two   isomorphic  orthogonal  representations of the symmetric quiver $A_m^{odd}$, where $V = V'$ as ${\mathcal I}$-graded vector space and let $\theta: V \rightarrow V'$ be an ${\mathcal I}$-graded isomorphism of orthogonal representations of the symmetric quiver $A_m^{odd}$.  
If $i \in {\mathcal I}$ is such that $V_i \ne \{0\}$ and consider  the induced isomorphism $\theta_i: V_i \rightarrow V'_i$, then $\det(\theta_{-i}) = \det(\theta_i)^{-1}$. If $V_0 \ne \{0\}$, $\theta_0$ is an orthogonal linear transformation of $V_0$ and $\det(\theta_0) = \pm1$. 
\end{lemma}
\begin{proof}
Because  we  have $\langle \theta_i(u), \theta_{\sigma(i)}(v)\rangle' = \langle u, v \rangle$ for all $u \in V_i$, $v \in V_{\sigma(i)}$ and $i \in {\mathcal I}$. A consequence of this is that $\det(\theta_{\sigma(i)}) = \det(\theta_{-i}) = \det(\theta_i)^{-1}$ for all $i \in {\mathcal I}$. Another consequence is that  $\langle \theta_0(u), \theta_0(v)\rangle' = \langle u, v \rangle$ for all $u, v \in V_0$ From this, we get that $\theta_0$ is an orthogonal linear transformation and $\det(\theta_0) = \pm 1$. 
\end{proof}

\begin{definition}
Let $(V, \phi, \langle\ , \ \rangle)$ and $(V', \phi', \langle\ , \ \rangle')$ be two   isomorphic  orthogonal  representations of the symmetric quiver $A_m^{odd}$, where $V = V'$ as ${\mathcal I}$-graded vector space. We say they are {\it isomorphic relative to the  special orthogonal group} if and only if $V_0 \ne \{0\}$ and there exists an ${\mathcal I}$-graded isomorphism $\theta: V \rightarrow V'$  as orthogonal representations of the symmetric quiver $A_m^{odd}$ such that $\det(\theta_0) = 1$. 
\end{definition}

\begin{notation}\label{N:EpsilonValue}
Denote by
\[
\epsilon = \begin{cases} \phantom{ - }1, &\text{ if we are in the orthogonal case;} \\ -1, &\text{ if we are in the symplectic case;} \end{cases}
\]
and, for an orthogonal (respectively symplectic) representation $(V, \phi, \langle\ , \ \rangle)$  of the symmetric quiver $(\Omega, \sigma)$, we will say that $(V, \phi, \langle\ , \ \rangle)$ is an {\it $\epsilon$-representation} when $\epsilon = 1$ (respectively $\epsilon = -1$). Note that with this notation,  we have for the bilinear form $\langle v, u \rangle = \epsilon \langle u, v \rangle$ for all $u, v \in V$.

\end{notation}

\section{Jacobson-Morozov triple for  indecomposable representations}

\subsection{}\label{SS:Iset}

Fix $m \in {\mathbb N}$ such that  $m \geq 1$,  $\epsilon \in \{1, -1\}$ as in notation~\ref{N:EpsilonValue}  and the symmetric quiver $(\Omega, \sigma)$ which  is either $A_m^{even}$ or $A_m^{odd}$. Thus  ${\mathcal I}$, as in \ref{SymmQuiver}, is 
\[
{\mathcal I} = \begin{cases} \{i \in {\mathbb Z} \mid i \equiv 1 \pmod 2, -2m < i < 2m\}, &\text{ in the case $A_m^{even}$;}\\ \\
\{i \in {\mathbb Z} \mid i \equiv 0 \pmod 2, -2m < i < 2m\}, &\text{ in the case $A_m^{odd}$;} \end{cases}
\]
and $\sigma:{\mathcal I} \rightarrow {\mathcal I}$ is $\sigma(i) = -i$ for all $i \in {\mathcal I}$.

The case where $\vert {\mathcal I} \vert = 1$ is easy to deal in our study. Thus from now on, we will always assume that $\vert {\mathcal I} \vert > 1$. 

Fix also the integer 
\[
\mu_{\max} = \begin{cases} 1, &\text{if  $\vert {\mathcal I}\vert \equiv 0 \pmod 2$ and $\epsilon = 1$;}\\ 0, &\text{if  $\vert {\mathcal I}\vert \equiv 0 \pmod 2$ and $\epsilon = -1$;}\\ 0, &\text{if $\vert {\mathcal I}\vert \equiv 1 \pmod 2$ and $\epsilon = 1$;}\\ 1, &\text{if  $\vert {\mathcal I}\vert \equiv 1 \pmod 2$ and $\epsilon = -1$.}\\
\end{cases}
\]

\begin{definition}\label{D:IBoxSet}
An {\it ${\mathcal I}$-diagram } is the Young diagram corresponding to the partition
\[
\vert{\mathcal I}\vert > (\vert{\mathcal I}\vert - 1) > \dots > 3 > 2 > 1
\]
where the $r^{th}$ part is $(\vert{\mathcal I}\vert - r + 1)$ for $r = 1, 2, \dots, \vert{\mathcal I}\vert$.

The rows are indexed by the elements of ${\mathcal I}$ in decreasing order from the first row to the last row and the columns are indexed by the elements of ${\mathcal I}$ in increasing order from the first column to the last column. Consequently we have a box at the intersection of the $i^{th}$ row and the $j^{th}$ column with $i, j \in {\mathcal I}$ if and only if $i \geq j$. 

The {\it principal diagonal} of the ${\mathcal I}$-diagram will correspond to the set of boxes on the row $i$ and the column $-i$ for $i \in \{i \in {\mathcal I} \mid i \geq 0\}$ and the {\it dimension diagonal} of the ${\mathcal I}$-diagram will correspond to the set of boxes on the row $i$ and the column $i$ for $i \in  {\mathcal I}$.

The ${\mathcal I}$-diagram come with a {\it box multiplicity function}
\[
\mu:\{(i, j) \in {\mathcal I} \times {\mathcal I} \mid i \geq j\} \rightarrow \{0, 1\}
\]
where 
\[
\mu(i, j) = \begin{cases} 0, &\text{if $i + j \ne 0$;}\\ \mu_{\max}, &\text{if $i + j = 0$.}\\
\end{cases}
\]

We define the set of {\it ${\mathcal I}$-boxes} of the ${\mathcal I}$-diagram to be the set
\[
{\mathbf B} = \{b(i, j, k) \mid i, j \in {\mathcal I}, i \geq j, 0 \leq k \leq \mu(i, j)\}.
\]
\end{definition}

\begin{example} 
For the symmetric quiver $A_3^{even}$, then   ${\mathcal I} = \{5, 3, 1, -1, -3, -5\}$  and the  ${\mathcal I}$-diagram is 

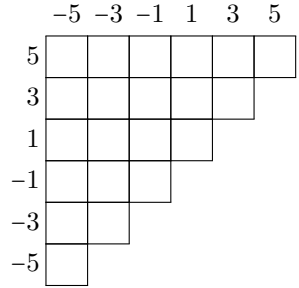
\begin{center}
\ytableausetup{centertableaux}
\begin{ytableau}
\none & \none [-5] & \none [-3] & \none [-1] & \none [1] & \none [3] & \none [5]\\ \none [{\phantom -}5] & & &  &  &  &  \\ \none [{\phantom -}3] &  &  &  &  & \\  \none [{\phantom -}1] &  &  & &  \\ \none [-1] & & & \\ \none [-3] & & \\ \none [-5] & \\
\end{ytableau}
\end{center}
where the indices for the rows and columns are inscribed.  

For the ${\mathcal I}$-boxes, the boxes in the picture of the ${\mathcal I}$-diagram are counted once if $i + j  \ne 0$ (i.e. not on the principal diagonal)  and counted $\mu_{max} + 1$ times if $i + j = 0$ (i.e. on the principal diagonal), that is once in the symplectic case and twice in the orthogonal case.
\end{example}

\begin{example} 
For the symmetric quiver $A_3^{odd}$, then  ${\mathcal I} = \{4, 2, 0, -2, -4\}$  and the  ${\mathcal I}$-diagram is 

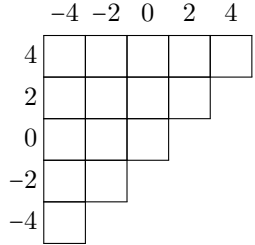
\begin{center}
\ytableausetup{centertableaux}
\begin{ytableau}
\none & \none [-4] & \none [-2] & \none [0] & \none [2] & \none [4]  \\ \none [{\phantom -}4] & & &  &  &   \\ \none [{\phantom -}2] &  &  &  &  \\  \none [{\phantom -}0] &  &  &  \\ \none [-2] & & \\ \none [-4] &  \\ 
\end{ytableau}
\end{center}
where the indices for the rows and columns are inscribed.

For the ${\mathcal I}$-boxes, the boxes in the picture of the ${\mathcal I}$-diagram are counted once if $i + j  \ne 0$ (i.e. not on the principal diagonal)  and counted $\mu_{max} + 1$ times if $i + j = 0$ (i.e. on the principal diagonal), that is once in the orthogonal case and twice in the symplectic case.

\end{example}

\begin{notation}\label{N:TauDefinition}
Let $\tau$  denotes  the involution $\tau:{\mathbf B} \rightarrow {\mathbf B}$ on the set ${\mathbf B}$ of ${\mathcal I}$-boxes given by 
\[
\tau(b(i, j, k)) = \begin{cases} b(-j, -i, k), &\text{ if $i + j \ne 0$;}\\ b(i, j, k), &\text{ if $i + j = 0$ and $\mu_{max} = 0$;} \\ b(i, j, 1 - k), &\text{ if $i + j = 0$ and $\mu_{max} = 1$} \end{cases}
\]
for $i, j \in {\mathcal I}$, $i \geq j$ and $0 \leq k \leq \mu(i, j)$. 

Because $\vert{\mathcal I}\vert > 1$, then $\tau \ne Id_{\mathcal B}$.  Denote by $\langle \tau \rangle$: the group generated by $\tau$ and,  by our hypothesis, $\langle \tau \rangle$ is isomorphic to ${\mathbb Z}/2{\mathbb Z}$. 
\end{notation}

\begin{remark}\label{R:CardinalityIndecomposable}
We can count the number of $\langle \tau \rangle$-orbits in ${\mathbf B}$. If we are in the case of $A_m^{even}$, then $\vert {\mathcal I}\vert = 2m$ is even,  it is easy to see that this number  of $\langle \tau \rangle$-orbits is 
\[
2m + 2(m - 1) + \dots + 4 + 2 = 2\left( \frac{m(m + 1)}{2}\right) = m(m + 1); 
\]
while if we are in the case of $A_m^{odd}$, then $\vert {\mathcal I} \vert = 2m - 1$ is odd, it is also easy to see that this number is
\[
(2m - 1) + (2m - 3) + \dots + 3 + 1 = 2\left(\frac{m(m + 1)}{2}\right) - m = m^2.
\]
We see that this number of $\langle \tau \rangle$-orbits in ${\mathbf B}$ is the same as the number of isomorphism classes  of indecomposable orthogonal or symplectic representations. 
\end{remark}

\subsection{} \label{SS:JMTriple}
In this section,  we will construct for each  $\langle \tau \rangle$-orbit ${\mathcal O}$ in ${\mathbf B}$:
\begin{itemize}
\item a ${\mathcal I}$-graded vector space $V({\mathcal O}) = \oplus_{i \in {\mathcal I}} V_i({\mathcal O})$ over ${\mathbf k}$ with an ${\mathcal I}$-graded basis $B_{\mathcal O} = \coprod_{i \in {\mathcal I}} B_i$;
\item a non-degenerate   $\epsilon$-bilinear form $\langle\  , \  \rangle_{\mathcal O}$ on  $V({\mathcal O})$  such that the restriction of $\langle\  , \  \rangle_{\mathcal O}$ to $V_i({\mathcal O}) \times V_{j}({\mathcal O})$ is $0$ if $j \ne \sigma(i)$ for all $i, j \in {\mathcal I}$;
\item  three ${\mathcal I}$-graded  linear transformations  $E_{\mathcal O}:V({\mathcal O}) \rightarrow V({\mathcal O})$,  $H_{\mathcal O}:V({\mathcal O}) \rightarrow V({\mathcal O})$ and $F_{\mathcal O}:V({\mathcal O}) \rightarrow V({\mathcal O})$ such that 
\begin{itemize}
\item $E_{\mathcal O}$ is of degree 2, $H_{\mathcal O}$ is of degree 0 and $F_{\mathcal O}$ is of degree -2;
\item $\langle Xu, v\rangle_{\mathcal O} + \langle u, Xv\rangle_{\mathcal O} = 0$ for all $X \in \{E_{\mathcal O}, H_{\mathcal O}, F_{\mathcal O}\}$ and all $u, v \in V({\mathcal O})$;
\item $[H_{\mathcal O}, E_{\mathcal O}] = 2E_{\mathcal O}$, $[H_{\mathcal O}, F_{\mathcal O}] = -2F_{\mathcal O}$ and $[E_{\mathcal O}, F_{\mathcal O}] = H_{\mathcal O}$. In other words,  $(E_{\mathcal O}, H_{\mathcal O}, F_{\mathcal O})$ is an ${\mathcal I}$-graded Jacobson-Morozov triple for the Lie algebra corresponding to $(V({\mathcal O}), \langle\   ,  \   \rangle_{\mathcal O})$;
\end{itemize}
\item the ${\mathcal I}$-graded linear transformation $\phi_{\mathcal O}: V({\mathcal O}) \rightarrow V({\mathcal O})$ given by 
\[
\phi_{{\mathcal O}, i} = {E_{\mathcal O}}\vert_{V_i({\mathcal O})}: V_{i}({\mathcal O}) \rightarrow V_{i + 2}({\mathcal O}) \quad \text{for all $i \in {\mathcal I}$  such that $(i + 2) \in {\mathcal I}$},
\]
is  such that $(V({\mathcal O}), \phi_{\mathcal O}, \langle\ , \ \rangle)$ is an $\epsilon$-representation.
\end{itemize}
Moreover $\{(V({\mathcal O}), \phi_{\mathcal O}, \langle\ , \ \rangle) \mid \text{ ${\mathcal O}$  belongs to the set of $\langle \tau \rangle$-orbits  in ${\mathbf B}$}\}$   is a set of representatives of the isomorphism classes of indecomposable $\epsilon$-representations of $(\Omega, \sigma)$.

\begin{definition}
An ${\mathcal I}$-box $b = b(i, j, k) \in {\mathbf B}$ with $i, j \in {\mathcal I}$, $i \geq j$ and $0 \leq k \leq \mu(i, j)$,  is said to be {\it strictly above the principal diagonal} (respectively {\it on the principal diagonal}, {\it strictly below the principal diagonal} ) if $i + j > 0$ (respectively $i + j = 0$, $i + j < 0$). 
\end{definition}

\begin{definition}
Given an ${\mathcal I}$-box $b = b(i, j, k) \in {\mathbf B}$, the {\it support of $b$} denoted $Supp(b)$ is defined to be the subset $Supp(b) = \{i' \in {\mathcal I} \mid j \leq i' \leq i\}$ of ${\mathcal I}$. 
\end{definition}

\begin{lemma}\label{L:TauBox}
Let $b = b(i, j, k) \in {\mathbf B}$ be an ${\mathcal I}$-box with $i, j \in {\mathcal I}$, $i \geq j$ and $0 \leq k \leq \mu(i, j)$.
\begin{enumerate}[\upshape (a)] 
\item $Supp(\tau(b)) = \sigma(Supp(b))$.
\item If $b$ is strictly above (respectively below) the principal diagonal, then $\tau(b)$ is strictly below (respectively above) the principal diagonal.
\item If $b$ is not on the principal diagonal, then $\tau(b) \ne b$.
\item If $b$ is on the principal diagonal and $\tau(b) \ne b$, then $\mu_{\max} = 1$ and $\tau(b) = b(i, j, 1 - k)$.
\item If  $b$ is on the principal diagonal , $\tau(b) = b$ and $\vert {\mathcal I}\vert$ is even, then  $\mu_{\max} = 0$,  $0 \not\in {\mathcal I}$ and $\epsilon = -1$.
\item If  $b$ is on the principal diagonal , $\tau(b) = b$ and $\vert {\mathcal I}\vert$ is odd, then $\mu_{\max} = 0$, $0 \in {\mathcal I}$ and $\epsilon = 1$.
\end{enumerate}
\end{lemma}
\begin{proof}
This lemma follows easily from the definitions.
\end{proof}
 
\begin{notation}\label{N:BasisVOrbit}
For the ${\mathcal I}$-box $b \in {\mathbf B}$, $V(b)$ denotes the ${\mathcal I}$-graded  finite dimensional ${\bf k}$-vector space 
\[
V(b) = \bigoplus_{i \in Supp(b)} V_i(b)
\]
with the ${\mathcal I}$-graded basis ${\mathcal B}_b = \{v_i(b) \mid i \in Supp(b)\}$ such that  $v_i(b) \in V_i(b)$.  For the $\langle \tau \rangle$-orbit ${\mathcal O}$, $V({\mathcal O})$ denotes  the ${\mathcal I}$-graded  finite dimensional ${\bf k}$-vector space 
\[
V({\mathcal O}) = \bigoplus_{b \in {\mathcal O}} V(b) 
\]
with  the ${\mathcal I}$-graded basis ${\mathcal B}_{\mathcal O} =\coprod_{b \in {\mathcal O}} {\mathcal B}_b$.  

We will denote the $i$-component of $V({\mathcal O})$ by 
\[
V_i({\mathcal O}) =  \bigoplus_{\substack{b \in {\mathcal O}\\ i \in Supp(b)}} V_i(b). 
\]

Consider the unique bilinear form $\langle\ , \  \rangle_{\mathcal O}: V({\mathcal O}) \times V({\mathcal O})  \rightarrow {\mathbf k}$ such that, for $b = b(i_1, j_1, k_1)  \in {\mathcal O}, b' = b(i_2,  j_2, k_2) \in {\mathcal O}$, $i \in Supp(b)$ and $i' \in Supp(b')$, we have

\[
\langle v_i(b), v_{i'}(b')\rangle_{\mathcal O} 
= \begin{cases} 1, &\text{if $b' = \tau(b)$, $i > 0$ and $i + i' = 0$; } \\
\epsilon, &\text{if $b' = \tau(b)$, $i < 0$ and $i + i' = 0$; }\\ 
1, &\text{if $b$ is above  the principal diagonal, $b' = \tau(b)$}\\  &\text{and $i = i' = 0$; } \\ 
\epsilon, &\text{if $b$ is below  the principal diagonal, $b' = \tau(b)$}\\  &\text{and $i = i' = 0$; } \\ 
\epsilon^{k_1}, &\text{if $b$ is on the principal diagonal, $b \ne \tau(b)$, }\\  &\text{$b' = \tau(b)$ and $i = i' = 0$; }\\ 
1, &\text{if $b$ is on the principal diagonal, $b = \tau(b) = b'$}\\  &\text{and $i = i' = 0$; }\\ 
0, &\text{otherwise.}
\end{cases}
\]
\end{notation}

\begin{lemma}
$\langle \  , \  \rangle_{\mathcal O}$ is a non-degenerate $\epsilon$-bilinear form on $V({\mathcal O})$  such that the restriction of $\langle\  , \  \rangle_{\mathcal O}$ to $V_i({\mathcal O}) \times V_j({\mathcal O})$ is $0$ if $j \ne \sigma(i)$ for all $i, j \in {\mathcal I}$. 
\end{lemma}
\begin{proof}
To show that the bilinear form $\langle\ , \  \rangle_{\mathcal O}$ is an $\epsilon$-bilinear form, it is enough to show it on the basis ${\mathcal B}_{\mathcal O}$ of $V({\mathcal O})$. We will consider separately the cases when  $\vert {\mathcal I}\vert$ is even and when it is odd.  

First assume $\vert {\mathcal I} \vert$ is even then $0 \not \in {\mathcal I}$. Assume that the orbit ${\mathcal O}$ has two elements: $\{b, \tau(b)\}$ with $b \ne \tau(b)$. If $b$ is not on the principal diagonal, then $\tau(b)$ is also not on the principal diagonal. Note that $\langle v_i(b), v_j(\tau(b))\rangle_{\mathcal O}$ and $\langle v_i(\tau(b)), v_j(b)\rangle_{\mathcal O}$ are $\ne 0$ if and only if $i + j = 0$. Thus it is enough to consider
\[
\langle v_i(b), v_{-i}(\tau(b))\rangle_{\mathcal O} = \begin{cases} 1, &\text{if $i > 0$;} \\ \epsilon, &\text{if $i < 0$;} \end{cases} \quad \text{ and } \quad \langle v_{-i}(\tau(b)), v_{i}(b)\rangle_{\mathcal O} = \begin{cases} \epsilon, &\text{if $i > 0$;} \\ 1, &\text{if $i < 0$.} \end{cases}
\] 
A similar argument works if $b$ is on the principal diagonal  and  $\tau(b) \ne b$. This is left to the reader.  Assume now that the orbit ${\mathcal O}$ has exactly one element: $\{b\}$ with $b = \tau(b)$. This happens only when $b$ is on the principal diagonal. Note that in this case $\langle v_i(b), v_j(b)\rangle_{\mathcal O}$ is  $\ne 0$ if and only if $i + j = 0$. Thus it is enough to consider
\[
\langle v_i(b), v_{-i}(b)\rangle_{\mathcal O} = \begin{cases} 1, &\text{if $i > 0$;} \\ \epsilon, &\text{if $i < 0$;} \end{cases} \quad \text{ and } \quad \langle v_{-i}(b), v_{i}(b)\rangle_{\mathcal O} = \begin{cases} \epsilon, &\text{if $i > 0$;} \\ 1, &\text{if $i < 0$.} \end{cases}
\] 
So if $\vert {\mathcal I} \vert$ is even, then we have $\langle u, v \rangle_{\mathcal O} = \epsilon \langle v, u \rangle_{\mathcal O}$ for all $u, v \in {\mathcal B}_{\mathcal O}$. 

Secondly assume $\vert {\mathcal I} \vert$ is odd,  then $0 \in {\mathcal I}$. Assume that the orbit ${\mathcal O}$ has two elements: $\{b, \tau(b)\}$ with $b \ne \tau(b)$. If $b$ is not on the principal diagonal, then $\tau(b)$ is also not on the principal diagonal. More precisely, if $b$ is above (respectively below) the principal diagonal, then $\tau(b)$ is below (respectively above) the principal diagonal. Note that $\langle v_i(b), v_j(\tau(b))\rangle_{\mathcal O}$ and $\langle v_i(\tau(b)), v_j(b)\rangle_{\mathcal O}$ are $\ne 0$ if and only if $i + j = 0$.  Assume that $b$ is above the principal diagonal, then 
\[
\langle v_i(b), v_{-i}(\tau(b))\rangle_{\mathcal O} = \begin{cases} 1, &\text{if $i \geq 0$;} \\ \epsilon, &\text{if $i < 0$;} \end{cases} \quad \text{ and } \quad \langle v_{-i}(\tau(b)), v_{i}(b)\rangle_{\mathcal O} = \begin{cases} \epsilon, &\text{if $i \geq 0$;} \\ 1, &\text{if $i < 0$.} \end{cases}
\] 
A similar argument works if $b$ is below the principal diagonal.  Assume now that $b$ is on the principal diagonal and $\tau(b) \ne b$, then,  as seen in lemma~\ref{L:TauBox} (d), $\mu_{\max} = 1$ and if $b = b(i_1, j_1, k_1)$ with $i_1, j_1 \in {\mathcal I}$, $i_1 \geq j_1$ and $0 \leq k_1 \leq \mu(i_1, j_1)$, then $\tau(b) = b(i_1, j_1, 1 - k_1)$. Note that $\langle v_i(b), v_j(\tau(b))\rangle_{\mathcal O}$ and $\langle v_i(\tau(b)), v_j(b)\rangle_{\mathcal O}$ are $\ne 0$ if and only if $i + j = 0$. Thus it is enough to consider
\[
\langle v_i(b), v_{-i}(\tau(b))\rangle_{\mathcal O} = \begin{cases} 1, &\text{if $i > 0$;} \\ \epsilon^{k_1}, &\text{if $i = 0$;} \\ \epsilon, &\text{if $i < 0$;} \end{cases} \  \text{ and } \  \langle v_{-i}(\tau(b)), v_{i}(b)\rangle_{\mathcal O} = \begin{cases} \epsilon, &\text{if $i > 0$;} \\ \epsilon^{1 - k_1}, &\text{if $i = 0$;}\\ 1, &\text{if $i < 0$.} \end{cases}
\] 
Assume now that the orbit ${\mathcal O}$ has exactly one element: $\{b\}$ with $b = \tau(b)$. This  happens only when $b$ is on the principal diagonal. By lemma~\ref{L:TauBox} (f), $\epsilon = 1$. Note that in this case $\langle v_i(b), v_j(b)\rangle_{\mathcal O}$ is  $\ne 0$ if and only if $i + j = 0$.  Thus we consider 
\[
\langle v_i(b), v_{-i}(b)\rangle_{\mathcal O} = \begin{cases} 1, &\text{if $i \geq 0$;} \\ 1, &\text{if $i < 0$;} \end{cases} \quad \text{ and } \quad \langle v_{-i}(b), v_{i}(b)\rangle_{\mathcal O} = \begin{cases} 1, &\text{if $i \geq 0$;} \\ 1, &\text{if $i < 0$.} \end{cases}
\] 
So if $\vert {\mathcal I} \vert$ is odd, then we have $\langle u, v \rangle_{\mathcal O} = \epsilon \langle v, u \rangle_{\mathcal O}$ for all $u, v \in {\mathcal B}_{\mathcal O}$.

To show that $\langle\  ,\   \rangle_{\mathcal O}$ is non-degenerate, consider $v \in V({\mathcal O})$ such that $\langle v, v' \rangle_{\mathcal O} = 0$ for all $v' \in V({\mathcal O})$. Write $v = \sum_{b \in {\mathcal O}} \sum_{i \in Supp(b)} c_{i, b} v_i(b)$ with $c_{i, b} \in {\mathbf k}$.  For any $b' \in {\mathcal O}$ and $i' \in Supp(b')$, we have $\tau(b') \in {\mathcal O}$, $-i' = \sigma(i') \in Supp(\tau(b'))$ by lemma~\ref{L:TauBox} (a) and $v_{-i'}(\tau(b')) \in V({\mathcal O})$.  We get that 
\[
\begin{aligned}
0 &= \langle v, v_{-i'}(\tau(b'))\rangle_{\mathcal O} =   \sum_{b \in {\mathcal O}} \sum_{i \in Supp(b)} c_{i, b} \langle v_i(b), v_{-i'}(\tau(b'))\rangle_{\mathcal O}\\  &= c_{i', b'} \langle v_{i'}(b'), v_{-i'}(\tau(b'))\rangle_{\mathcal O}
\end{aligned}
\]
and $ \langle v_{i'}(b'), v_{-i'}(\tau(b'))\rangle_{\mathcal O} \ne 0$, from our definition of $\langle\  , \  \rangle_{\mathcal O}$. Thus  we get that $c_{i', b'} = 0$ for all $b' \in {\mathcal O}$ and $i' \in Supp(b')$ and $v = 0$. So $\langle\  , \  \rangle_{\mathcal O}$ is non-degenerate. 

Let $v \in V_i({\mathcal O})$ and $v' \in V_{i'}({\mathcal O})$ with $i' \ne \sigma(i) = -i$.  Write 
\[
v = \sum_{\substack{b \in {\mathcal O}  \\ i \in Supp(b)}} c_{i, b} v_i(b) \quad \text{ and } \quad v' = \sum_{\substack{b' \in {\mathcal O}, \\ i' \in Supp(b')}} c_{i', b'} v_{i'}(b').
\]
Then 
\[
\langle v, v' \rangle_{\mathcal O} = \sum_{\substack{b, b' \in {\mathcal O}\\ i \in Supp(b)\\ i' \in Supp(b')}} c_{i, b} c_{i', b'} \langle v_i(b), v_{i'}(b')\rangle_{\mathcal O} = 0 
\]
from our definition of $\langle\  , \  \rangle_{\mathcal O}$ and because $i' \ne -i$. 
\end{proof}

\begin{definition}
For the ${\mathcal I}$-box $b = b(i, j, k) \in  {\mathbf B}$ with $i, j, \in {\mathcal I}$, $i \geq j$ and $0 \leq k \leq \mu(i, j)$  and for $i' \in Supp(b)$, then we define
\[
h_{i'}(b) = \#\{j'' \in {\mathcal I} \mid j \leq j'' \leq i'\}  - \#\{i'' \in {\mathcal I} \mid i' \leq i'' \leq i\}
\]
and 
\[
f_{i'}(b) = \#\{ (i'', j'') \in {\mathcal I} \times {\mathcal I} \mid i'' \geq j'',\   i' \leq i'' \leq i,\    j < j'' \leq i'\}.
\]
\end{definition}

\begin{lemma}\label{L:RelationCoefficientsHandF}
Let $b \in  {\mathbf B}$ be the ${\mathcal I}$-box $b = b(i, j, k)$  with $i, j, \in {\mathcal I}$, $i \geq j$ and $0 \leq k \leq \mu(i, j)$  and  let $i' \in Supp(b)$. Then we have
\begin{enumerate}[\upshape (a)] 
\item $h_{-i'}(\tau(b)) = - h_{i'}(b)$;
\item If $i' \ne \min(Supp(b))$, then $(-i' + 2) \in Supp(\tau(b))$, $(-i' + 2) \ne \min(Supp(\tau(b))$ and $f_{-i' + 2}(\tau(b)) = f_{i'}(b)$;
\item If $i' \ne \max(Supp(b))$, then $(i' + 2) \in Supp(b)$ and $(h_{i' + 2}(b) - h_{i'}(b)) = 2$. 
\item If $i' \ne \min(Supp(b))$, then $(i' - 2) \in Supp(b)$ and $(h_{i' - 2}(b) - h_{i'}(b)) = -2$. 
\item If $i' = \min(Supp(b)) < \max(Supp(b))$, then $(i' - 2) \not\in Supp(b)$, $(i' + 2) \in Supp(b)$,  $f_{i'}(b) = 0$ and $f_{i' + 2}(b) = -h_{i'}(b)$;
\item If $\min(Supp(b)) < i' =  \max(Supp(b))$, then $(i' - 2) \in Supp(b)$, $(i' + 2) \not\in Supp(b)$ and $f_{i'}(b) = h_{i'}(b)$;
\item If $\min(Supp(b)) < i' <  \max(Supp(b))$, then $(i' - 2), (i' + 2)  \in Supp(b)$ and $(f_{i'}(b) - f_{i' + 2}(b)) = h_{i'}(b)$.
\end{enumerate}
\end{lemma}
\begin{proof}
(a) We have that $\tau(b) = b(-j, -i, k')$ where $k' \in \{k, (1 - k)\}$. Note also that $-i' \in Supp(\tau(b))$.  We have 
\[
h_{i'}(b) = \#\{j'' \in {\mathcal I} \mid j \leq j'' \leq i'\}  - \#\{i'' \in {\mathcal I} \mid i' \leq i'' \leq i\}
\]
and 
\[
\begin{aligned}
h_{-i'}(\tau(b)) &= \#\{j'' \in {\mathcal I} \mid -i \leq j'' \leq -i'\}  - \#\{i'' \in {\mathcal I} \mid -i' \leq i'' \leq -j\}\\
&=  \#\{j'' \in {\mathcal I} \mid i \geq -j'' \geq i'\}  - \#\{i'' \in {\mathcal I} \mid i' \geq -i'' \geq j\} = -h_{i'}(b)
\end{aligned}
\]

(b) So we have $j < i' \leq i$, because $i' \ne \min(Supp(b))$. Thus $(-j + 2) > (-i' + 2) \geq (-i + 2)$. Note that $Supp(\tau(b)) = \{i'' \in {\mathcal I} \mid -j \geq i'' \geq -i\}$. Clearly $(-i' + 2) \in Supp(\tau(b))$, $(-i' + 2) \ne \min(Supp(\tau(b))$.  Now we have 
that 
\[
f_{i'}(b) = \#\{ (i'', j'') \in {\mathcal I} \times {\mathcal I} \mid i'' \geq j'',\   i' \leq i'' \leq i,\    j < j'' \leq i'\}
\]
and $f_{-i' + 2}(\tau(b))$ is equal to 
\[
\begin{aligned}
 &\#\{ (i'', j'') \in {\mathcal I} \times {\mathcal I} \mid i'' \geq j'',\  (- i' + 2) \leq i'' \leq -j,\    -i < j'' \leq (-i' + 2)\}\\
 &=\#\left\{ (i'', j'') \in {\mathcal I} \times {\mathcal I} \left \vert  \begin{aligned} &(- j'' + 2) \leq (-i'' + 2),\   i'  \geq (-i'' + 2) > j,\\   & \quad i \geq (- j'' + 2)  \geq i' \end{aligned}\right. \right\}\\
 &= f_{i'}(b).
 \end{aligned}
\]

(c) Clearly $(i' + 2) \in Supp(b)$. Now we have
\[
\begin{aligned}
h_{i' + 2}(b) - h_{i'}(b) &= (\#\{j'' \in {\mathcal I} \mid j \leq j'' \leq (i' + 2)\}  - \#\{i'' \in {\mathcal I} \mid (i' + 2) \leq i'' \leq i\})\\ &\quad - (\#\{j'' \in {\mathcal I} \mid j \leq j'' \leq i'\}  - \#\{i'' \in {\mathcal I} \mid i' \leq i'' \leq i\})\\ &= 2.
\end{aligned}
\]

(d) The proof is similar to the proof of $(c)$ and is left to the reader.

(e) If $i' = \min(Supp(b)) < \max(Supp(b))$, then clearly $i' = j$, $(i' - 2) \not \in Supp(b)$, $(i' + 2) \in Supp(b)$  and 
\[
 \{ (i'', j'') \in {\mathcal I} \times {\mathcal I} \mid i'' \geq j'',\   j \leq i'' \leq i,\    j < j'' \leq j\} = \emptyset. 
\]
So 
\[
f_{i'}(b) = \#\{ (i'', j'') \in {\mathcal I} \times {\mathcal I} \mid i'' \geq j'',\   j \leq i'' \leq i,\    j < j'' \leq j\} = 0.
\]

If we now consider
\[
\begin{aligned}
f_{i' + 2}(b) &= \{ (i'', j'') \in {\mathcal I} \times {\mathcal I} \mid i'' \geq j'',\   (i' + 2) \leq i'' \leq i,\    j < j'' \leq (i' + 2)\}\\
&=  \{ (i'', (i' + 2)) \in {\mathcal I} \times {\mathcal I} \mid i'' \geq j'',\   (i' + 2) \leq i'' \leq i\} = - h_{i'}(b),
\end{aligned}
\]
because 
\[
\begin{aligned}
h_{i'}(b) &= \#\{j'' \in {\mathcal I} \mid j \leq j'' \leq  i'\}  - \#\{i'' \in {\mathcal I} \mid i' \leq i'' \leq i\}\\ 
&=  - \#\{i'' \in {\mathcal I} \mid i'< i'' \leq i\}.
\end{aligned}
\]

(f)  If $\min(Supp(b)) < i' =  \max(Supp(b))$, then clearly $i' = i$, $(i' - 2) \in Supp(b)$, $(i' + 2)\not  \in Supp(b)$  and 
\[
\begin{aligned}
f_{i'}(b) &= \#\{ (i'', j'') \in {\mathcal I} \times {\mathcal I} \mid i'' \geq j'',\   i' \leq i'' \leq i,\    j < j'' \leq i'\}\\
&= \#\{ (i, j'') \in {\mathcal I} \times {\mathcal I} \mid   j < j'' \leq i'\} = h_{i'}(b).
\end{aligned}
\]

(g) If $\min(Supp(b)) < i' <  \max(Supp(b))$, then clearly  $(i' - 2), (i' + 2) \in Supp(b)$  and 
\[
\begin{aligned}
f_{i'}(b) - f_{i' + 2}(b) &= \#\{(i'', j'') \in {\mathcal I} \times {\mathcal I} \mid i'' \geq j'', i' \leq i'' \leq i,\   j < j'' \leq i'\}\\  & - \#\{(i'', j'') \in {\mathcal I} \times {\mathcal I} \mid i'' \geq j'', (i' + 2) \leq i'' \leq i, j < j'' \leq (i' + 2)\} \\ &= \#\{(i', j'') \in {\mathcal I} \times {\mathcal I} \mid \   j < j'' \leq i'\} \\ &\quad +  \#\{(i'', j'') \in {\mathcal I} \times {\mathcal I} \mid i'' \geq j'', (i' + 2) \leq i'' \leq i, j < j'' \leq i'\}\\ &\quad - \#\{(i'', j'') \in {\mathcal I} \times {\mathcal I} \mid i'' \geq j'', (i' + 2) \leq i'' \leq i, j < j'' \leq i'\}\\ &\quad - \#\{(i'', (i' + 2)) \in {\mathcal I} \times {\mathcal I} \mid (i' + 2) \leq i'' \leq i\}\\ &= \#\{(i', j'') \in {\mathcal I} \times {\mathcal I} \mid  j < j'' \leq i'\} \\ & \quad -\#\{(i'', (i' + 2)) \in {\mathcal I} \times {\mathcal I} \mid  (i' + 2) \leq i'' \leq i\} \\&= \#\{(i', j'') \in {\mathcal I} \times {\mathcal I} \mid  j \leq j'' \leq i'\} \\ & \quad -\#\{(i'', (i' + 2)) \in {\mathcal I} \times {\mathcal I} \mid i' \leq i'' \leq i\}
\\& = h_{i'}(b).
\end{aligned}
\]

\end{proof}

\subsection{}\label{SS:EHFDefinition}
To construct the Jacobson-Morozov triple $(E_{\mathcal O}, H_{\mathcal O}, F_{\mathcal O})$ as in \ref{SS:JMTriple}, we will need to consider separately the two cases:  when $\vert {\mathcal I} \vert$ is even and when  $\vert {\mathcal I} \vert$ is odd.

{\bf Even case:} $\vert {\mathcal I} \vert$ is even; more precisely the symmetric quiver $(\Omega, \sigma)$ is $A_m^{even}$ and ${\mathcal I} =  \{i \in {\mathbb Z} \mid i \equiv 1 \pmod 2, -2m < i < 2m\}$.   We will also have to distinguish two situations for our  $\langle \tau\rangle$-orbit ${\mathcal O}$. If  $b$ is an ${\mathcal I}$-box in ${\mathcal O}$ then either $Supp(b) \cap Supp(\tau(b)) = \emptyset$ or $Supp(b) \cap Supp(\tau(b)) \ne \emptyset$. This condition of being empty or not empty is independant of the choice of the ${\mathcal I}$-box $b$ in ${\mathcal O}$. 

If $Supp(b) \cap Supp(\tau(b)) = \emptyset$, then we will say that the orbit ${\mathcal O}$ has {\it no overlapping ${\mathcal I}$-box supports}; while if $Supp(b) \cap Supp(\tau(b)) \ne \emptyset$, then we will say that the orbit ${\mathcal O}$ has {\it overlapping ${\mathcal I}$-box supports} 

Note that   the orbit ${\mathcal O}$ has overlapping ${\mathcal I}$-box supports  if and only if $\{1, -1\} \subset Supp(b) \cap Supp(\tau(b))$. One direction is obvious. If $b = b(i, j, k)$ with $i, j, \in {\mathcal I}$, $i \geq j$ and $0 \leq k \leq \mu(i, j)$, then $Supp(b) = \{i' \in {\mathcal I} \mid i \geq i' \geq j\}$ and $Supp(\tau(b)) = \{i' \in {\mathcal I} \mid -j \geq i' \geq -i\}$.   If $Supp(b) \cap Supp(\tau(b)) \ne \emptyset$, say $i' \in Supp(b) \cap Supp(\tau(b))$, then $ i \geq i' \geq j$ and $ i \geq -i' \geq j$. From which,  we get that $\{1, -1\} \subset Supp(b)$.  Similarly we get that $\{1, -1\} \subset Supp(\tau(b))$. From this, we get also that $\max(Supp(b)) \geq 1$ and $\min(Supp(b)) \leq -1$ when  the orbit ${\mathcal O}$ has overlapping ${\mathcal I}$-box supports. 

We can now define the ${\mathcal I}$-graded linear transformations 
\[
E_{\mathcal O}:V({\mathcal O}) \rightarrow V({\mathcal O}), \quad   H_{\mathcal O}:V({\mathcal O}) \rightarrow V({\mathcal O}) \quad  \text{ and } \quad F_{\mathcal O}:V({\mathcal O}) \rightarrow V({\mathcal O}).
\] 

If the orbit ${\mathcal O}$ has no overlapping ${\mathcal I}$-box supports, then for $b \in {\mathcal O}$ and $i \in Supp(b)$, we define
\[
E_{\mathcal O}(v_i(b)) = \begin{cases} {\phantom -} 0, &\text{ if $i = \max(Supp(b))$;}\\ {\phantom -} v_{i + 2}(b), &\text{ if $i \ne \max(Supp(b))$ and $i > 0$;}\\  -v_{i + 2}(b), &\text{ if $i \ne \max(Supp(b))$ and $i < 0$;} \end{cases}
\]

\[
H_{\mathcal O}(v_i(b)) = h_i(b) v_i(b)
\]

and 

\[
F_{\mathcal O}(v_i(b)) = \begin{cases} {\phantom -} 0, &\text{ if $i = \min(Supp(b))$;}\\ {\phantom -} f_i(b)v_{i - 2}(b), &\text{ if $i \ne \min(Supp(b))$ and $i > 0$;}\\  -f_i(b)v_{i - 2}(b), &\text{ if $i \ne \min(Supp(b))$ and $i < 0$.} \end{cases}
\]

If the orbit ${\mathcal O}$ has overlapping ${\mathcal I}$-box supports, then for $b = b(i_1, j_1, k_1) \in {\mathcal O}$ and $i \in Supp(b)$, we define
\[
E_{\mathcal O}(v_i(b)) = \begin{cases} {\phantom -} 0, &\text{ if $i = \max(Supp(b))$;}\\ {\phantom -} v_{i + 2}(b), &\text{ if $i \ne \max(Supp(b))$ and $i \geq 1$;}\\  -v_{i + 2}(b), &\text{ if  $i < -1$;}\\ {\phantom -} v_1(b), &\text{ if  $i = -1$ and $b$ is below  the principal diagonal;}\\ - \epsilon v_1(b), &\text{ if $i = -1$ and $b$ is above  the principal diagonal;}\\  (- \epsilon)^{k_1} v_1(b), &\text{ if $i = -1$ and $b$ is on  the principal diagonal;}
\end{cases}
\]

\[
H_{\mathcal O}(v_i(b)) = h_i(b) v_i(b)
\]

and

\[
F_{\mathcal O}(v_i(b)) = \begin{cases} {\phantom -} 0, &\text{ if $i = \min(Supp(b))$;}\\ {\phantom -} f_i(b)v_{i - 2}(b), &\text{ if $i > 1$;}\\  -f_i(b)v_{i - 2}(b), &\text{ if $i \ne \min(Supp(b))$ and $i \leq -1$.}\\ {\phantom -} f_1(b)v_{-1}(b), &\text{ if  $i = 1$ and $b$ is below  the principal diagonal;} \\ -\epsilon f_1(b)v_{-1}(b), &\text{ if  $i = 1$ and $b$ is above  the principal diagonal;}\\  (- \epsilon)^{k_1} f_1(b)v_{-1}(b), &\text{ if  $i = 1$ and $b$ is on the principal diagonal;}  \end{cases}
\]

{\bf Odd case:} $\vert {\mathcal I} \vert$ is odd, more precisely the symmetric quiver $(\Omega, \sigma)$ is $A_m^{odd}$ and ${\mathcal I} =  \{i \in {\mathbb Z} \mid i \equiv 0 \pmod 2, -2m  <  i < 2m\}$. We will also have to distinguish two situations for our  $\langle \tau\rangle$-orbit ${\mathcal O}$. If  $b$ is an ${\mathcal I}$-box in ${\mathcal O}$ then either $Supp(b) \cap Supp(\tau(b)) = \emptyset$ or $Supp(b) \cap Supp(\tau(b)) \ne \emptyset$. This condition of being empty or not empty is independant of the choice of the ${\mathcal I}$-box $b$ in ${\mathcal O}$.  

If $Supp(b) \cap Supp(\tau(b)) = \emptyset$, then we will say that the orbit ${\mathcal O}$ has {\it no overlapping ${\mathcal I}$-box supports} and if $Supp(b) \cap Supp(\tau(b)) \ne \emptyset$, then we will say that the orbit ${\mathcal O}$ has {\it overlapping ${\mathcal I}$-box supports} 

Note that   the orbit ${\mathcal O}$ has overlapping ${\mathcal I}$-box supports  if and only if $0 \in Supp(b)$. If $0 \in Supp(b)$, then $\sigma(0) = 0 \in Supp(\tau(b))$ and  one direction of the equivalence follows. If $b = b(i, j, k)$ with $i, j, \in {\mathcal I}$, $i \geq j$ and $0 \leq k \leq \mu(i, j)$, then $Supp(b) = \{i' \in {\mathcal I} \mid i \geq i' \geq j\}$ and $Supp(\tau(b)) = \{i' \in {\mathcal I} \mid -j \geq i' \geq -i\}$.   If  $i' \in Supp(b) \cap Supp(\tau(b))$, then $ i \geq i' \geq j$ and $ i \geq -i' \geq j$. From which,  we get that $0 \in Supp(b)$, because the elements of ${\mathcal I}$ consists of even integers.  From this, we also get  that $\max(Supp(b)) \geq 0$ and $\min(Supp(b)) \leq 0$, when  the orbit ${\mathcal O}$ has overlapping ${\mathcal I}$-box supports. 

If the orbit ${\mathcal O}$ has no overlapping ${\mathcal I}$-box supports, then for $b \in {\mathcal O}$ and $i \in Supp(b)$, we define
\[
E_{\mathcal O}(v_i(b)) = \begin{cases} {\phantom -} 0, &\text{ if $i = \max(Supp(b))$;}\\ {\phantom -} v_{i + 2}(b), &\text{ if $i \ne \max(Supp(b))$ and $i > 0$;}\\  -v_{i + 2}(b), &\text{ if $i \ne \max(Supp(b))$ and $i < 0$;} \end{cases}
\]

\[
H_{\mathcal O}(v_i(b)) = h_i(b) v_i(b)
\]

and 

\[
F_{\mathcal O}(v_i(b)) = \begin{cases} {\phantom -} 0, &\text{ if $i = \min(Supp(b))$;}\\ {\phantom -} f_i(b)v_{i - 2}(b), &\text{ if $i \ne \min(Supp(b))$ and $i > 0$;}\\  -f_i(b)v_{i - 2}(b), &\text{ if $i \ne \min(Supp(b))$ and $i < 0$.} \end{cases}
\]

If the orbit ${\mathcal O}$ has overlapping ${\mathcal I}$-box supports, then for $b = b(i_1, j_1, k_1) \in {\mathcal O}$ and $i \in Supp(b)$, we define
\[
E_{\mathcal O}(v_i(b)) = \begin{cases} {\phantom -} 0, &\text{if $i = \max(Supp(b))$;}\\ {\phantom -} v_{i + 2}(b), &\text{if $i \ne \max(Supp(b))$ and $i > 0$;}\\  -v_{i + 2}(b), &\text{if  $i < 0$;}\\   {\phantom -} \epsilon v_2(b), &\text{if $i = 0$, $ \max(Supp(b)) > 0$ and $b$ is below the}\\ &\text{principal diagonal;}\\ {\phantom -}  v_2(b), &\text{if $i = 0$, $ \max(Supp(b)) > 0$ and $b$ is above the}\\ &\text{principal diagonal;} \\ {\phantom -} \epsilon^{k_1}v_{2}(b), &\text{if $i = 0$, $ \max(Supp(b)) > 0$ and $b$ is on the principal}\\&\text{diagonal;}  \end{cases}
\]

\[
H_{\mathcal O}(v_i(b)) = h_i(b) v_i(b)
\]

and 

\[
F_{\mathcal O}(v_i(b)) = \begin{cases} {\phantom -} 0, &\text{ if $i = \min(Supp(b))$;}\\ {\phantom -} f_i(b)v_{i - 2}(b), &\text{ if  $i > 2$;}\\  -f_i(b)v_{i - 2}(b), &\text{ if $i \ne \min(Supp(b))$ and $i \leq 0$.}\\  {\phantom -}\epsilon  f_2(b)v_{0}(b), &\text{ if $i = 2$ and and $b$ is below the principal diagonal;}\\ {\phantom -} f_2(b)v_{0}(b), &\text{ if $i = 2$ and and $b$ is above the principal diagonal;}\\ {\phantom -} \epsilon^{k_1} f_2(b)v_{0}(b), &\text{ if $i = 2$ and and $b$ is on the principal diagonal;} \end{cases}
\]

\begin{proposition}\label{P:JMTripleProof}
 With the above definition, we have that the three ${\mathcal I}$-graded  linear transformations  
 \[
 E_{\mathcal O}:V({\mathcal O}) \rightarrow V({\mathcal O}), \quad  H_{\mathcal O}:V({\mathcal O}) \rightarrow V({\mathcal O}) \quad \text{ and } \quad F_{\mathcal O}:V({\mathcal O}) \rightarrow V({\mathcal O})
 \]
are such that 
\begin{itemize}
\item $E_{\mathcal O}$ is of degree $2$, $H_{\mathcal O}$ is of degree $0$ and $F_{\mathcal O}$ is of degree $-2$;
\item $\langle Xu, v\rangle_{\mathcal O} + \langle u, Xv\rangle_{\mathcal O} = 0$ for all $X \in \{E_{\mathcal O}, H_{\mathcal O}, F_{\mathcal O}\}$ and all $u, v \in V({\mathcal O})$;
\item $[H_{\mathcal O}, E_{\mathcal O}] = 2E_{\mathcal O}$, $[H_{\mathcal O}, F_{\mathcal O}] = -2F_{\mathcal O}$ and $[E_{\mathcal O}, F_{\mathcal O}] = H_{\mathcal O}$, in other words,  $(E_{\mathcal O}, H_{\mathcal O}, F_{\mathcal O})$ is an ${\mathcal I}$-graded  Jacobson-Morozov triple for the Lie algebra corresponding to $(V({\mathcal O}), \langle \  , \   \rangle_{\mathcal O})$;
\item  the ${\mathcal I}$-graded linear transformation $\phi_{\mathcal O}: V({\mathcal O}) \rightarrow V({\mathcal O})$ given by 
\[
\phi_{{\mathcal O}, i} = {E_{\mathcal O}}\vert_{V_i({\mathcal O})}: V_{i}({\mathcal O}) \rightarrow V_{i + 2}({\mathcal O})  \quad \text{for all $i \in {\mathcal I}$  such that $(i + 2) \in {\mathcal I}$},
\]
is  such that $(V({\mathcal O}), \phi_{\mathcal O}, \langle\ , \ \rangle_{\mathcal O})$ is an $\epsilon$-representation.
\end{itemize}
Moreover $\{(V({\mathcal O}), \phi_{\mathcal O}, \langle\ , \ \rangle_{\mathcal O}) \mid \text{ ${\mathcal O}$  belongs to the set of $\langle \tau \rangle$-orbits  in ${\mathbf B}$}\}$   is a set of representatives of the isomorphism classes of indecomposable $\epsilon$-representations of $(\Omega, \sigma)$.
\end{proposition}
\begin{proof}
This proof is not difficult, but there are many cases to consider. To be complete all the details are provided. 

The first statement about the degrees of $E_{\mathcal O}$, $H_{\mathcal O}$ and $F_{\mathcal O}$ is obvious. To check that $\langle Xu, v\rangle_{\mathcal O} + \langle u, Xv\rangle_{\mathcal O} = 0$ for all $X \in \{E_{\mathcal O}, H_{\mathcal O}, F_{\mathcal O}\}$ and all $u, v \in V({\mathcal O})$, it is enough to verify this on the ${\mathcal I}$-graded basis  ${\mathcal B}_{\mathcal O}$ of $V({\mathcal O})$. Thus we can assume that $u = v_{i_1}(b_1)$, $v = v_{i_2}(b_2)$ with $b_1, b_2 \in {\mathcal O}$, $i_1 \in Supp(b_1)$ and $i_2 \in Supp(b_2)$,

{\bf Even case:}

Consider first the proof of the proposition when $\vert {\mathcal I} \vert$  is even. In this case,  $0 \not\in {\mathcal I}$ and we will start with $X = E_{\mathcal O}$.

If  the orbit ${\mathcal O}$ has no overlapping ${\mathcal I}$-box supports, we have 
\[
\langle E_{\mathcal O}(v_{i_1}(b_1)), v_{i_2}(b_2)\rangle_{\mathcal O} = \begin{cases} {\phantom -} 0, &\text{ if $i_1 = \max(Supp(b_1))$;}\\ {\phantom -} \langle v_{i_1 + 2}(b_1), v_{i_2}(b_2)\rangle_{\mathcal O} , &\text{ if $i_1 \ne \max(Supp(b_1))$}\\ &\text{ and $i_1 > 0$;}\\  -\langle v_{i_1+ 2}(b_1), v_{i_2}(b_2)\rangle_{\mathcal O} ,  &\text{ if $i_1 \ne \max(Supp(b_1))$}\\ &\text{and $i_1 < 0$;} \end{cases}
\] 
and 
\[
\langle v_{i_1}(b_1), E_{\mathcal O}(v_{i_2}(b_2))\rangle_{\mathcal O}  = \begin{cases} {\phantom -} 0, &\text{ if $i_2 = \max(Supp(b_2))$;}\\ {\phantom -} \langle v_{i_1}(b_1), v_{i_2 + 2}(b_2)\rangle_{\mathcal O} , &\text{ if $i_2 \ne \max(Supp(b_2))$}\\ &\text{and $i_2 > 0$;}\\  -\langle v_{i_1}(b_1), v_{i_2 + 2}(b_2)\rangle_{\mathcal O} ,  &\text{ if $i_2 \ne \max(Supp(b_2))$}\\ &\text{and $i_2 < 0$.} \end{cases}
\]

If $i_1 = \max(Supp(b_1))$, then $\langle E_{\mathcal O}(v_{i_1}(b_1)), v_{i_2}(b_2)\rangle_{\mathcal O}  = 0$,  by the definition of $E_{\mathcal O}$ and we have $\langle v_{i_1}(b_1), E_{\mathcal O}(v_{i_2}(b_2))\rangle_{\mathcal O}  = 0$,  otherwise if  $\langle v_{i_1}(b_1), E_{\mathcal O}(v_{i_2}(b_2))\rangle_{\mathcal O}  \ne 0$, then  $i_2 \ne \max(Supp(b_2))$, $b_2 = \tau(b_1)$ and  $i_1 + 2 + i_2 = 0$. This is impossible,  because $Supp(b_2) = Supp(\tau(b_1)) = \sigma(Supp(b_1))$ and $\min(Supp(b_2)) = -i_1$, but from $i_1 + 2 + i_2 = 0$, we get the contradiction because  $i_2 = -(i_1 + 2) \in Supp(b_2)$. Thus 
\[
\langle E_{\mathcal O}(v_{i_1}(b_1)), v_{i_2}(b_2)\rangle_{\mathcal O}  + \langle v_{i_1}(b_1), E_{\mathcal O}(v_{i_2}(b_2))\rangle_{\mathcal O}  = 0
\]
if $i_1 = \max(Supp(b_1)$. By symmetry, we also have this equality if $i_2 = \max(Supp(b_2))$. Thus we can now assume that $i_1 \ne \max(Supp(b_1))$ and  $i_2 \ne \max(Supp(b_2))$.

So we can now assume that $i_1 \ne \max(Supp(b_1))$ and $i_2 \ne \max(Supp(b_2))$. Moreover if either $\langle E_{\mathcal O}(v_{i_1}(b_1)), v_{i_2}(b_2)\rangle_{\mathcal O} \ne 0$ or $\langle v_{i_1}(b_1), E_{\mathcal O}(v_{i_2}(b_2))\rangle_{\mathcal O} \ne 0$, then $b_2 = \tau(b_1)$ and $i_1 + i_2 + 2 = 0$  from our definition of the bilinear form $\langle\  , \  \rangle_{\mathcal O}$. From our hypothesis on the orbit ${\mathcal O}$, we have that $\tau(b_1) \ne b_1$ and $\tau(b_2) \ne b_2$.  If $b_1$ is not on the principal diagonal, then $b_2 = \tau(b_1)$ is also not on the principal diagonal. If $i_1 > 0$, then $i_2 = -i_1 - 2 < 0$ because $i_1 + i_2 + 2 = 0$. Thus we get that 
\begin{multline*}
\langle E_{\mathcal O}(v_{i_1}(b_1)), v_{i_2}(b_2)\rangle_{\mathcal O} + \langle v_{i_1}(b_1), E_{\mathcal O}(v_{i_2}(b_2))\rangle_{\mathcal O} = \\ \langle v_{i_1 + 2}(b_1), v_{i_2}(b_2)\rangle_{\mathcal O} + \langle v_{i_1}(b_1), - v_{i_2 + 2}(b_2)\rangle_{\mathcal O} = 1 - 1 = 0.
\end{multline*}
If $i_1 < 0$ and $i_1 \ne -1$, then $i_1 + 2 < 0$ and $i_2 > 0$. Thus we get that
\begin{multline*}
\langle E_{\mathcal O}(v_{i_1}(b_1)), v_{i_2}(b_2)\rangle_{\mathcal O} + \langle v_{i_1}(b_1), E_{\mathcal O}(v_{i_2}(b_2))\rangle_{\mathcal O} = \\ \langle -v_{i_1 + 2}(b_1), v_{i_2}(b_2)\rangle_{\mathcal O} + \langle v_{i_1}(b_1),  v_{i_2 + 2}(b_2)\rangle_{\mathcal O} = -\epsilon   + \epsilon = 0.
\end{multline*}
To conclude this case, we can observe that $i_1 \ne -1$. Otherwise $i_2 = -i_1 - 2 = -1$ implies that $1 \in Supp(b_1)$. So $\{1, -1\} \subset Supp(b_1)$ and similarly $\{1, -1\} \subset Supp(\tau(b_1))$. But by  \ref{SS:EHFDefinition}, this is equivalent to the orbit ${\mathcal O}$ has overlapping  ${\mathcal I}$-box supports contradicting our hypothesis. 

Note that $b_1$ cannot be on the principal diagonal, otherwise  $b_2 = \tau(b_1)$ is also on the principal diagonal and  $Supp(b_1) = Supp(b_2)$. So we are in the overlapping ${\mathcal I}$-box supports case contradicting our hypothesis on the orbit ${\mathcal O}$. 

We now consider the case where  the orbit ${\mathcal O}$ has overlapping ${\mathcal I}$-box supports.  If $i_1 = \max(Supp(b_1))$, the proof is the same as when the orbit ${\mathcal O}$ has no overlapping ${\mathcal I}$-box supports and is left to the reader. As above we only have to consider the case where $i_1 \ne \max(Supp(b_1))$ and $i_2 \ne \max(Supp(b_2))$. Moreover if either $\langle E_{\mathcal O}(v_{i_1}(b_1)), v_{i_2}(b_2)\rangle_{\mathcal O} \ne 0$ or $\langle v_{i_1}(b_1), E_{\mathcal O}(v_{i_2}(b_2))\rangle_{\mathcal O} \ne 0$, then $b_2 = \tau(b_1)$ and $i_1 + i_2 + 2 = 0$  from our definition of the bilinear form $\langle\  , \  \rangle_{\mathcal O}$.  So we will assume that $b_2 = \tau(b_1)$ and $i_1 + i_2 + 2 = 0$.  Clearly if $b_1$ is below (respectively above, on)  the principal diagonal, then $b_2$ is above (respectively below, on) the principal diagonal. In the case where the ${\mathcal I}$-boxes $b_1$ and $b_2$ are on the principal diagonal, then  we will write $b_1 = b(i'_1, j'_1, k'_1)$ and $b_2 = \tau(b_1) = b(i'_1, j'_1, k''_1)$ where $i'_1, j'_1 \in {\mathcal I}, i'_1 \geq j'_1$ and $k''_1 \in \{k'_1, (1 - k'_1)\}$. More precisely,  if $\mu_{\max} = 0$, then $k'_1 = k''_1 = 0$ and if $\mu_{\max} = 1$, then $k''_1 = (1 - k'_1)$.

We have that $\langle E_{\mathcal O}(v_{i_1}(b_1)), v_{i_2}(b_2)\rangle_{\mathcal O}$ is equal to
\[
 \begin{cases} {\phantom -} 0, &\text{ if $i_1 = \max(Supp(b_1))$;}\\ {\phantom -} \langle v_{i_1 + 2}(b_1), v_{i_2}(b_2)\rangle_{\mathcal O}, &\text{ if $i_1 \ne \max(Supp(b_1))$ and $i_1 > 0$;}\\  -\langle v_{i_1+ 2}(b_1), v_{i_2}(b_2)\rangle_{\mathcal O},  &\text{ if $i_1 < -1$;}\\ {\phantom -} \langle v_{1}(b_1), v_{i_2}(b_2)\rangle_{\mathcal O}, &\text{ if $i_1 = -1$ and $b_1$ is below the principal diagonal;}\\ -\epsilon \langle v_{1}(b_1), v_{i_2}(b_2)\rangle_{\mathcal O}, &\text{ if $i_1 = -1$ and $b_1$ is above the principal diagonal;}\\ (-\epsilon)^{k'_1} \langle v_{1}(b_1), v_{i_2}(b_2)\rangle_{\mathcal O}, &\text{ if $i_1 = -1$ and $b_1$ is on the principal diagonal;}\end{cases}
\] 
and $\langle v_{i_1}(b_1), E_{\mathcal O}(v_{i_2}(b_2))\rangle_{\mathcal O}$ is equal to
\[
 \begin{cases} {\phantom -} 0, &\text{ if $i_2 = \max(Supp(b_2))$;}\\ {\phantom -} \langle v_{i_1}(b_1), v_{i_2 + 2}(b_2)\rangle_{\mathcal O}, &\text{ if $i_2 \ne \max(Supp(b_2))$ and $i_2 > 0$;}\\  -\langle v_{i_1}(b_1), v_{i_2 + 2}(b_2)\rangle_{\mathcal O},  &\text{ if $i_2 < -1$;}\\ {\phantom - } \langle v_{i_1}(b_1), v_{1}(b_2)\rangle_{\mathcal O},  &\text{ if $i_2 = -1$ and $b_2$ is below the principal diagonal;}\\  - \epsilon \langle v_{i_1}(b_1), v_{1}(b_2)\rangle_{\mathcal O},  &\text{ if $i_2 = -1$ and $b_2$ is above the principal diagonal;} \\ (- \epsilon)^{ k''_1} \langle v_{i_1}(b_1), v_{1}(b_2)\rangle_{\mathcal O},  &\text{ if $i_2 = -1$ and $b_2$ is on  the principal diagonal.} \end{cases}
\]

The proofs when $i_1 > 0$ and also when $i_1 < 0$, $i_1 \ne -1$ are similar to the proofs when the orbit ${\mathcal O}$ has no overlapping ${\mathcal I}$-box supports and are left to the reader. We will just consider the case where $i_1 = -1$ and consequently $i_2 = -1$. 

We have that  $\langle E_{\mathcal O}(v_{-1}(b_1)), v_{-1}(b_2)\rangle_{\mathcal O}$ is equal to 
\[
 \begin{cases} \langle v_1(b_1), v_{-1}(b_2) \rangle_{\mathcal O} = 1, &\text{ if $b_1$ is below the principal diagonal;}\\ -\epsilon \langle  v_1(b_1), v_{-1}(b_2)\rangle_{\mathcal O} = -\epsilon , &\text{ if $b_1$ is above the principal diagonal;}\\  (-\epsilon)^{k'_1} \langle v_1(b_1), v_{-1}(b_2)\rangle_{\mathcal O} = (-\epsilon)^{k'_1}, &\text{ if $b_1$ is on the principal diagonal;}\end{cases}
\]
and   $\langle v_{-1}(b_1), E_{\mathcal O}(v_{-1}(b_2))\rangle_{\mathcal O}$ is equal to 
\[
\begin{cases} -\epsilon \langle v_{-1}(b_1),  v_{1}(b_2) \rangle_{\mathcal O} = -\epsilon^2 , &\text{ if $b_1$ is below the principal diagonal;}\\  \langle  v_{-1}(b_1), v_{1}(b_2)\rangle_{\mathcal O} = \epsilon, &\text{ if $b_1$ is above the principal diagonal;}\\ (-\epsilon)^{k''_1} \langle v_{-1}(b_1), v_{1}(b_2)\rangle_{\mathcal O} = (-\epsilon)^{k''_1} \epsilon&\text{ if $b_1$ is on the principal diagonal.} \end{cases}
\]
From this, we get easily if $b_1$ is below or above the principal diagonal that 
\[
\langle E_{\mathcal O}(v_{i_1}(b_1)), v_{i_2}(b_2)\rangle_{\mathcal O} + \langle v_{i_1}(b_1), E_{\mathcal O}(v_{i_2}(b_2))\rangle_{\mathcal O} = 0.
\]

In the case that the ${\mathcal I}$-box $b_1$ is on the principal diagonal, we have to consider two cases: either $\tau(b_1) \ne b_1$  or $\tau(b_1) = b_1$. By lemma~\ref{L:TauBox} (d), if $\tau(b_1) \ne b_1$, then  $\mu_{\max} = 1$ and $k''_1 = 1 - k'_1$. Moreover because of \ref{SS:Iset} and ${\mathcal I}$ is even, then $\epsilon = 1$. So in this case, 
\[
\langle E_{\mathcal O}(v_{i_1}(b_1)), v_{i_2}(b_2)\rangle_{\mathcal O} + \langle v_{i_1}(b_1), E_{\mathcal O}(v_{i_2}(b_2))\rangle_{\mathcal O} = (-\epsilon)^{k'_1} + (-\epsilon)^{1 - k'_1} \epsilon = 0.
\]
By lemma~\ref{L:TauBox} (e), if $\tau(b_1) = b_1$, then $\mu_{\max} = 0$, $k''_1 = k'_1$ and $\epsilon = -1$. 
So in this case, 
\[
\langle E_{\mathcal O}(v_{i_1}(b_1)), v_{i_2}(b_2)\rangle_{\mathcal O} + \langle v_{i_1}(b_1), E_{\mathcal O}(v_{i_2}(b_2))\rangle_{\mathcal O} = (-\epsilon)^{k'_1} + (-\epsilon)^{k'_1} \epsilon = 0.
\]

If $X = H_{\mathcal O}$, then for either $\langle H_{\mathcal O}(v_{i_1}(b_1)), v_{i_2}(b_2)\rangle_{\mathcal O} =  h_{i_1}(b_1) \langle v_{i_1}(b_1), v_{i_2}(b_2)\rangle_{\mathcal O} \ne 0$ or  $\langle v_{i_1}(b_1), H_{\mathcal O}(v_{i_2}(b_2))\rangle_{\mathcal O}  = h_{i_2}(b_2) \langle v_{i_1}(b_1), v_{i_2}(b_2)\rangle_{\mathcal O} \ne 0$, we must have that $b_2 = \tau(b_1)$ and $i_1 + i_2 = 0$. Note that in that case, both are $\ne 0$. So we can assume that $b_2 = \tau(b_1)$ and $i_2 = -i_1$. By lemma \ref{L:RelationCoefficientsHandF} (a),  we get that
\begin{multline*}
\langle H_{\mathcal O}(v_{i_1}(b_1)), v_{i_2}(b_2)\rangle_{\mathcal O} + \langle v_{i_1}(b_1), H_{\mathcal O}(v_{i_2}(b_2))\rangle_{\mathcal O}\\  = (h_{i_1}(b_1) + h_{i_2}(b_2)) \langle v_{i_1}(b_1), v_{i_2}(b_2)\rangle_{\mathcal O} = 0
\end{multline*}

Now consider  $X = F_{\mathcal O}$. The proof is similar with small modification with the one of $E_{\mathcal O}$. If  the orbit ${\mathcal O}$ has no overlapping ${\mathcal I}$-box supports, then we have that $\langle F_{\mathcal O}(v_{i_1}(b_1)), v_{i_2}(b_2)\rangle_{\mathcal O}$ is equal to 
\[
\begin{cases} {\phantom -} 0, &\text{ if $i_1 = \min(Supp(b_1))$;}\\ {\phantom -}  f_{i_1}(b_1) \langle v_{i_1 - 2}(b_1), v_{i_2}(b_2)\rangle_{\mathcal O}, &\text{ if $i_1 \ne \min(Supp(b_1))$ and $i_1 > 0$;}\\  - f_{i_1}(b_1) \langle v_{i_1 - 2}(b_1), v_{i_2}(b_2)\rangle_{\mathcal O},  &\text{ if $i_1 \ne \min(Supp(b_1))$ and $i_1 < 0$;} \end{cases}
\] 
and $\langle v_{i_1}(b_1), F_{\mathcal O}(v_{i_2}(b_2))\rangle_{\mathcal O}$ is equal to 
\[
= \begin{cases} {\phantom -} 0, &\text{ if $i_2 = \min(Supp(b_2))$;}\\ {\phantom -} f_{i_2}(b_2) \langle v_{i_1}(b_1), v_{i_2 - 2}(b_2)\rangle_{\mathcal O}, &\text{ if $i_2 \ne \min(Supp(b_2))$ and $i_2 > 0$;}\\  - f_{i_2}(b_2) \langle v_{i_1}(b_1), v_{i_2 - 2}(b_2)\rangle_{\mathcal O},  &\text{ if $i_2 \ne \min(Supp(b_2))$ and $i_2 < 0$.} \end{cases}
\]

If $i_1 = \min(Supp(b_1))$, then $\langle F_{\mathcal O}(v_{i_1}(b_1)), v_{i_2}(b_2)\rangle_{\mathcal O} = 0$,  by the definition of $F_{\mathcal O}$ and we have $\langle v_{i_1}(b_1), F_{\mathcal O}(v_{i_2}(b_2))\rangle_{\mathcal O} = 0$,  otherwise if  $\langle v_{i_1}(b_1), F_{\mathcal O}(v_{i_2}(b_2))\rangle_{\mathcal O} \ne 0$, then  $i_2 \ne \min(Supp(b_2))$, $b_2 = \tau(b_1)$ and  $i_1 - 2 + i_2 = 0$. This is impossible,  because $Supp(b_2) = Supp(\tau(b_1)) = \sigma(Supp(b_1))$ and $\max(Supp(b_2)) = -i_1$, but from $i_1 - 2 + i_2 = 0$, we get the contradiction because  $i_2 = -i_1 + 2 \in Supp(b_2)$. Thus 
\[
\langle F_{\mathcal O}(v_{i_1}(b_1)), v_{i_2}(b_2)\rangle_{\mathcal O} + \langle v_{i_1}(b_1), F_{\mathcal O}(v_{i_2}(b_2))\rangle_{\mathcal O} = 0
\]
if $i_1 = \min(Supp(b_1)$. By symmetry, we also have this equality if $i_2 = \min(Supp(b_2))$. 

So we can now assume that $i_1 \ne \min(Supp(b_1))$ and $i_2 \ne \min(Supp(b_2))$. Moreover if either $\langle F_{\mathcal O}(v_{i_1}(b_1)), v_{i_2}(b_2)\rangle_{\mathcal O} \ne 0$ or $\langle v_{i_1}(b_1), F_{\mathcal O}(v_{i_2}(b_2))\rangle_{\mathcal O} \ne 0$, then $b_2 = \tau(b_1)$ and $i_1 + i_2 - 2 = 0$  from our definition of the bilinear form $\langle\  , \  \rangle_{\mathcal O}$. From our hypothesis on the orbit ${\mathcal O}$, we have that $\tau(b_1) \ne b_1$ and $\tau(b_2) \ne b_2$.  If $b_1$ is not on the principal diagonal, then $b_2 = \tau(b_1)$ is also not on the principal diagonal. If $i_1 > 0$ and $i_1 \ne 1$, then $i_2 = -i_1 + 2 < 0$ because $i_1 + i_2 - 2 = 0$. Thus we get that 
\begin{multline*}
\langle F_{\mathcal O}(v_{i_1}(b_1)), v_{i_2}(b_2)\rangle_{\mathcal O} + \langle v_{i_1}(b_1), F_{\mathcal O}(v_{i_2}(b_2))\rangle_{\mathcal O} = \\ f_{i_1}(b_1) \langle v_{i_1 - 2}(b_1), v_{i_2}(b_2)\rangle_{\mathcal O} - f_{i_2}(b_2) \langle v_{i_1}(b_1),  v_{i_2 - 2}(b_2)\rangle_{\mathcal O} =\\ (f_{i_1}(b_1) - f_{-i_1 + 2}(\tau(b_1)) = 0
\end{multline*}
because of  lemma~\ref{L:RelationCoefficientsHandF} (b) and our definition of the bilinear form.

If $i_1 < 0$, then $i_2 = -i_1 + 2 > 0$. Thus we get that
\begin{multline*}
\langle F_{\mathcal O}(v_{i_1}(b_1)), v_{i_2}(b_2)\rangle_{\mathcal O} + \langle v_{i_1}(b_1), F_{\mathcal O}(v_{i_2}(b_2))\rangle_{\mathcal O} = \\ -f_{i_1}(b_1) \langle v_{i_1 - 2}(b_1), v_{i_2}(b_2)\rangle_{\mathcal O} + f_{i_2}(b_2) \langle v_{i_1}(b_1),  v_{i_2 - 2}(b_2)\rangle_{\mathcal O} = \\ ( -\epsilon f_{i_1}(b_1)  + \epsilon f_{i_2}(b_2)) = 0
\end{multline*}
because of  lemma~\ref{L:RelationCoefficientsHandF} (b) and our definition of the bilinear form.
To conclude this case, we can observe that $i_1 \ne 1$. Otherwise $i_2 = -i_1 + 2 = 1$ implies that $-1 \in Supp(b_1)$. So $\{1, -1\} \subset Supp(b_1)$ and similarly $\{1, -1\} \subset Supp(\tau(b_1))$. But by  \ref{SS:EHFDefinition}, this is equivalent to the orbit ${\mathcal O}$ has overlapping  ${\mathcal I}$-box supports contradicting our hypothesis. 

Note that $b_1$ cannot be on the principal diagonal, otherwise  $b_2 = \tau(b_1)$ is also on the principal diagonal and  $Supp(b_1) = Supp(b_2)$. So we are in the overlapping ${\mathcal I}$-box supports case contradicting our hypothesis on the orbit ${\mathcal O}$.

We now consider the case where  the orbit ${\mathcal O}$ has overlapping ${\mathcal I}$-box supports.  If $i_1 = \min(Supp(b_1))$, the proof is the same as when the orbit ${\mathcal O}$ has no overlapping ${\mathcal I}$-box supports and is left to the reader. As above we only have to consider the case where $i_1 \ne \min(Supp(b_1))$ and $i_2 \ne \min(Supp(b_2))$. Moreover if either $\langle F_{\mathcal O}(v_{i_1}(b_1)), v_{i_2}(b_2)\rangle_{\mathcal O} \ne 0$ or $\langle v_{i_1}(b_1), F_{\mathcal O}(v_{i_2}(b_2))\rangle_{\mathcal O} \ne 0$, then $b_2 = \tau(b_1)$ and $i_1 + i_2 - 2 = 0$  from our definition of the bilinear form $\langle\  , \  \rangle_{\mathcal O}$.  So we will assume that $b_2 = \tau(b_1)$ and $i_1 + i_2 - 2 = 0$.  Clearly if $b_1$ is below (respectively above, on)  the principal diagonal, then $b_2$ is above (respectively below, on) the principal diagonal. In the case where the ${\mathcal I}$-boxes $b_1$ and $b_2$ are on the principal diagonal, then  we will write $b_1 = b(i'_1, j'_1, k'_1)$ and $b_2 = \tau(b_1) = b(i'_1, j'_1,  k''_1)$ where $i'_1, j'_1 \in {\mathcal I}, i'_1 \geq j'_1$ and $k''_1 \in \{k'_1, (1 - k'_1)\}$.

We have that $\langle F_{\mathcal O}(v_{i_1}(b_1)), v_{i_2}(b_2)\rangle_{\mathcal O}$ is equal to
\[
 \begin{cases} {\phantom -} 0, &\text{ if $i_1 = \min(Supp(b_1))$;}\\ {\phantom -} f_{i_1}(b_1)\langle v_{i_1 - 2}(b_1), v_{i_2}(b_2)\rangle_{\mathcal O}, &\text{ if $i_1 > 1$;}\\  -f_{i_1}(b_1)\langle v_{i_1- 2}(b_1), v_{i_2}(b_2)\rangle_{\mathcal O},  &\text{ if $i_1 \ne \min(Supp(b_1))$ and $i_1 \leq -1$;}\\ {\phantom -} f_1(b_1)\langle v_{-1}(b_1), v_{i_2}(b_2)\rangle_{\mathcal O}, &\text{ if $i_1 = 1$ and $b_1$ is below the principal diagonal;}\\ -\epsilon f_1(b_1) \langle v_{-1}(b_1), v_{i_2}(b_2)\rangle_{\mathcal O}, &\text{ if $i_1 = 1$ and $b_1$ is above the principal diagonal;}\\ (-\epsilon)^{k'_1}f_1(b_1) \langle v_{-1}(b_1), v_{i_2}(b_2)\rangle_{\mathcal O}, &\text{ if $i_1 = 1$ and $b_1$ is on the principal diagonal;}\end{cases}
\] 
and $\langle v_{i_1}(b_1), F_{\mathcal O}(v_{i_2}(b_2))\rangle_{\mathcal O}$ is equal to
\[
 \begin{cases} {\phantom -} 0, &\text{ if $i_2 = \min(Supp(b_2))$;}\\ {\phantom -} f_{i_2}(b_2)\langle v_{i_1}(b_1), v_{i_2 - 2}(b_2)\rangle_{\mathcal O}, &\text{ if $i_2 > 1$;}\\  -f_{i_2}(b_2)\langle v_{i_1}(b_1), v_{i_2 - 2}(b_2)\rangle_{\mathcal O},  &\text{ if $i_2 \ne \min(Supp(b_2)$ and $i_2 \leq -1$;}\\ {\phantom - } f_1(b_2)\langle v_{i_1}(b_1), v_{-1}(b_2)\rangle_{\mathcal O},  &\text{ if $i_2 = 1$ and $b_2$ is below the principal diagonal;}\\  - \epsilon f_1(b_2) \langle v_{i_1}(b_1), v_{-1}(b_2)\rangle_{\mathcal O},  &\text{ if $i_2 = 1$ and $b_2$ is above the principal diagonal;} \\ (- \epsilon)^{ k''_1}f_1(b_2) \langle v_{i_1}(b_1), v_{-1}(b_2)\rangle_{\mathcal O},  &\text{ if $i_2 = 1$ and $b_2$ is on  the principal diagonal.} \end{cases}
\]

The proofs when $i_1 > 0$  and $i_1 \ne 1$ or when $i_1 < 0$  are similar to the proofs when the orbit ${\mathcal O}$ has no overlapping ${\mathcal I}$-box supports and are left to the reader. We will just consider the case where $i_1 = 1$ and consequently $i_2 = 1$. 
We have that $\langle F_{\mathcal O}(v_{i_1}(b_1)), v_{i_2}(b_2)\rangle_{\mathcal O}$ is equal to
\[
 \begin{cases}  {\phantom -} f_1(b_1)\langle v_{-1}(b_1), v_{1}(b_2)\rangle_{\mathcal O}, &\text{ if  $b_1$ is below the principal diagonal;}\\ -\epsilon f_1(b_1) \langle v_{-1}(b_1), v_{1}(b_2)\rangle_{\mathcal O}, &\text{ if  $b_1$ is above the principal diagonal;}\\ (-\epsilon)^{k'_1}f_1(b_1) \langle v_{-1}(b_1), v_{1}(b_2)\rangle_{\mathcal O}, &\text{ if  $b_1$ is on the principal diagonal;}\end{cases}
\] 
and $\langle v_{i_1}(b_1), F_{\mathcal O}(v_{i_2}(b_2))\rangle_{\mathcal O}$ is equal to
\[
 \begin{cases}  - \epsilon f_1(b_2) \langle v_{1}(b_1), v_{-1}(b_2)\rangle_{\mathcal O},  &\text{ if  $b_1$ is below the principal diagonal;} \\  {\phantom - } f_1(b_2)\langle v_{1}(b_1), v_{-1}(b_2)\rangle_{\mathcal O},  &\text{ if $b_1$ is above the principal diagonal;}\\  (- \epsilon)^{ k''_1}f_1(b_2) \langle v_{1}(b_1), v_{-1}(b_2)\rangle_{\mathcal O},  &\text{ if  $b_1$ is on  the principal diagonal.} \end{cases}
\]

 From this and lemma~\ref{L:RelationCoefficientsHandF} (b), we get easily if $b_1$ is below or above the principal diagonal that 
\[
\langle F_{\mathcal O}(v_{1}(b_1)), v_{1}(b_2)\rangle_{\mathcal O} + \langle v_{1}(b_1), F_{\mathcal O}(v_{1}(b_2))\rangle_{\mathcal O} = 0.
\]

In the case that the ${\mathcal I}$-box $b_1$ is on the principal diagonal, we have to consider two cases: either $\tau(b_1) \ne b_1$  or $\tau(b_1) = b_1$. By lemma~\ref{L:TauBox} (d), if $\tau(b_1) \ne b_1$, then  $\mu_{\max} = 1$ and $k''_1 = 1 - k'_1$. Moreover because of \ref{SS:Iset} and ${\mathcal I}$ is even, then $\epsilon = 1$. So in this case, 
\[
\langle F_{\mathcal O}(v_{1}(b_1)), v_{1}(b_2)\rangle_{\mathcal O} + \langle v_{-1}(b_1), F_{\mathcal O}(v_{1}(b_2))\rangle_{\mathcal O} = (-\epsilon)^{k'_1}f_1(b_1)\epsilon + (-\epsilon)^{1 - k'_1}f_1(b_2) \epsilon = 0.
\]
by lemma~\ref{L:RelationCoefficientsHandF} (b).
By lemma~\ref{L:TauBox} (e), if $\tau(b_1) = b_1$, then $\mu_{\max} = 0$, $k''_1 = k'_1$ and $\epsilon = -1$. 
So in this case, 
\[
\langle F_{\mathcal O}(v_{1}(b_1)), v_{1}(b_2)\rangle_{\mathcal O} + \langle v_{1}(b_1), F_{\mathcal O}(v_{1}(b_2))\rangle_{\mathcal O} = (-\epsilon)^{k'_1} f_1(b_1) \epsilon+ (-\epsilon)^{k'_1} f_1(b_2) = 0.
\]
again by lemma~\ref{L:RelationCoefficientsHandF} (b).

We have now proved that $\langle X(u), v\rangle_{\mathcal O} + \langle u, X(v)\rangle_{\mathcal O} = 0$ for all $X \in \{E_{\mathcal O}, H_{\mathcal O}, F_{\mathcal O}\}$ and all $u, v \in V({\mathcal O})$ when $\vert {\mathcal I} \vert$ is even.

We want to prove now the Jacobson-Morozov triple relations in the case where $\vert {\mathcal I} \vert$ is even. In other words, $[H_{\mathcal O}, E_{\mathcal O}] = 2E_{\mathcal O}$, $[H_{\mathcal O}, F_{\mathcal O}] = -2F_{\mathcal O}$ and $[E_{\mathcal O}, F_{\mathcal O}] = H_{\mathcal O}$. It is enough to verify these only on the ${\mathcal I}$-graded basis ${\mathcal B}_{\mathcal O}$ of $V({\mathcal O})$. 

If the orbit ${\mathcal O}$ has no overlapping ${\mathcal I}$-box supports, then we have already computed  above our expressions  for $E_{\mathcal O}$, $H_{\mathcal O}$ and $F_{\mathcal O}$ on ${\mathcal B}_{\mathcal O}$. Let $b$ be the ${\mathcal I}$-box in ${\mathcal O}$ and $i \in Supp(b)$. Then $[H_{\mathcal O}, E_{\mathcal O}](v_i(b)) =  (H_{\mathcal O} E_{\mathcal O} - E_{\mathcal O} H_{\mathcal O}) (v_i(b))$ is equal to
\[
\begin{aligned}
&=\begin{cases} 0 - 0, &\text{ if $i = \max(Supp(b))$;}\\ h_{i + 2}(b) v_{i + 2}(b) - h_i(b) v_{i + 2}(b), &\text{ if $i \ne \max(Supp(b))$ and $i > 0$;}\\  -h_{i + 2}(b) v_{i + 2}(b) + h_i(b) v_{i + 2}(b), &\text{ if $i \ne \max(Supp(b))$ and $i <  0$;}
\end{cases}\\
& =  \quad (h_{i + 2}(b) - h_i(b))E_{\mathcal O}(v_i(b)) = 2 E_{\mathcal O}(v_i(b))
\end{aligned}
\]
by lemma~\ref{L:RelationCoefficientsHandF} (c).

Secondly we have  $[H_{\mathcal O}, F_{\mathcal O}](v_i(b)) =  (H_{\mathcal O} F_{\mathcal O} - F_{\mathcal O} H_{\mathcal O}) (v_i(b))$ is equal to
\[
\begin{aligned}
&=\begin{cases} 0 - 0, &\text{ if $i = \min(Supp(b))$;}\\ h_{i - 2}(b) f_i(b) v_{i - 2}(b) - h_i(b) f_i(b) v_{i - 2}(b), &\text{ if $i \ne \min(Supp(b))$ and $i > 0$;}\\  -h_{i - 2}(b) f_i(b) v_{i - 2}(b) + h_i(b) f_i(b) v_{i - 2}(b), &\text{ if $i \ne \min(Supp(b))$ and $i <  0$;}
\end{cases}\\
& =  \quad (h_{i - 2}(b) - h_i(b))F_{\mathcal O}(v_i(b)) = -2 F_{\mathcal O}(v_i(b))
\end{aligned}
\]
by lemma~\ref{L:RelationCoefficientsHandF} (d).

Thirdly we have to compute $[E_{\mathcal O}, F_{\mathcal O}](v_i(b)) =  (E_{\mathcal O} F_{\mathcal O} - F_{\mathcal O} E_{\mathcal O}) (v_i(b))$. Write $b = b(i_1, j_1, k_1)$ with $i_1, j_1 \in {\mathcal I}$, $i_1 \geq j_1$ and $0 \leq k \leq \mu(i_1, j_1)$  and for $i \in Supp(b)$. Thus $i_1 \geq i \geq j_1$. We have different cases to consider. Because ${\mathcal O}$ has no overlapping ${\mathcal I}$-box supports, then either $j_1 \geq 1$ or $i_1 \leq -1$. If $i_1 = j_1 = i$, then easily we have $h_i(b) = 0$ and  $(E_{\mathcal O} F_{\mathcal O} - F_{\mathcal O} E_{\mathcal O}) (v_i(b)) = 0 = H_{\mathcal O} (v_i(b))$. If $i_1 > j_1$ and $i = i_1 = \max(Supp(b))$, then 
\[
\begin{aligned}
f_i(b) &= \#\{(i'', j'') \in {\mathcal I} \times {\mathcal I}  \mid i'' \geq j'', i \leq i'' \leq i_1, j_1 < j'' \leq i\}\\
&= \#\{(i'', j'') \in {\mathcal I} \times {\mathcal I}  \mid i'' \geq j'', i = i'' = i_1, j_1 < j'' \leq i\} = h_i(b),
\end{aligned}
\]
because of lemma~\ref{L:RelationCoefficientsHandF} (f).

So
\[
(E_{\mathcal O} F_{\mathcal O} - F_{\mathcal O} E_{\mathcal O}) (v_i(b)) = f_i(b) v_i(b) = H_{\mathcal O}(v_i(b))
\]
If $i_1 > j_1$ and $i = j_1 = \min(Supp(b))$, then 
\[
\begin{aligned}
f_{i + 2}(b) &= \#\{(i'', j'') \in {\mathcal I} \times {\mathcal I}  \mid i'' \geq j'', (i + 2) \leq i'' \leq i_1, j_1 < j'' \leq (i + 2)\}\\ 
&=   \#\{(i'', j'') \in {\mathcal I} \times {\mathcal I}  \mid i'' \geq j'', (i + 2) \leq i'' \leq i_1,  j'' = (i + 2)\} = -h_i(b),
\end{aligned}
\]
because of lemma~\ref{L:RelationCoefficientsHandF} (e).

So
\[
(E_{\mathcal O} F_{\mathcal O} - F_{\mathcal O} E_{\mathcal O}) (v_i(b)) = 0 - f_{i + 2}(b) v_i(b) = H_{\mathcal O}(v_i(b))
\]
If $i_1 > i > j_1$, then both $(i - 2)$, $(i + 2)$ belongs to $Supp(b)$. Consequently they are both either simultanously $> 0$ or simultanously $< 0$.  So
\[
(E_{\mathcal O} F_{\mathcal O} - F_{\mathcal O} E_{\mathcal O}) (v_i(b)) = (f_i(b) - f_{i + 2}(b)) v_i(b) = H_{\mathcal O}(v_i(b))
\]
because of lemma~\ref{L:RelationCoefficientsHandF} (g).

If the orbit ${\mathcal O}$ has  overlapping ${\mathcal I}$-box supports, then we have already computed  above our expressions  for $E_{\mathcal O}$, $H_{\mathcal O}$ and $F_{\mathcal O}$ on ${\mathcal B}_{\mathcal O}$. Let $b = b(i_1, j_1, k_1)$ be the ${\mathcal I}$-box in ${\mathcal O}$ and $i \in Supp(b)$. We have  $[H_{\mathcal O}, E_{\mathcal O}](v_i(b)) =  (H_{\mathcal O} E_{\mathcal O} - E_{\mathcal O} H_{\mathcal O}) (v_i(b))$ is equal to
\[
\begin{cases}
0, &\text{ if $i = \max(Supp(b))$;}\\  (h_{i + 2}(b) - h_i(b))v_{i + 2}(b) = 2v_{i + 2}(b), &\text{ if $i \ne \max(Supp(b))$ and $i > 0$;}\\  -(h_{i + 2}(b) - h_i(b)) v_{i + 2}(b)=  - 2v_{i + 2}(b), &\text{ if  $i < -1$;}\\ (h_{1}(b) - h_{-1}(b)) v_{1}(b)=  2v_{1}(b), &\text{ if  $i = -1$ and $b$ is below the }\\ &\text{principal  diagonal;}\\ -\epsilon(h_{1}(b) - h_{-1}(b)) v_{1}(b)=  -\epsilon 2v_{1}(b), &\text{ if  $i = -1$ and $b$ is above the }\\ &\text{principal  diagonal;}\\ (-\epsilon)^{k_1} (h_{1}(b) - h_{-1}(b)) v_{1}(b)=  (-\epsilon)^{k_1} 2v_{1}(b), &\text{ if  $i = -1$ and $b$ is on the }\\ &\text{principal  diagonal;}\\ 
\end{cases}
\]
by lemma~\ref{L:RelationCoefficientsHandF} (c). Clearly we get that $[H_{\mathcal O}, E_{\mathcal O}](v_i(b)) = 2E_{\mathcal O}(v_i(b))$.  

Similarly we have 
$[H_{\mathcal O}, F_{\mathcal O}](v_i(b)) =  (H_{\mathcal O} F_{\mathcal O} - F_{\mathcal O} H_{\mathcal O}) (v_i(b))$ is equal to
\[
\begin{cases}
0, &\text{ if $i = \min(Supp(b))$;}\\  (h_{i - 2}(b) - h_i(b))f_i(b)v_{i - 2}(b) = -2f_i(b)v_{i - 2}(b), &\text{ if $i > 1$;}\\  -(h_{i - 2}(b) - h_i(b)) f_i(b)v_{i - 2}(b)=  2f_i(b)v_{i - 2}(b), &\text{ if  $i \ne \min(Supp(b))$}\\ &\text{ and $i \leq -1$;}\\ (h_{-1}(b) - h_{1}(b))f_1(b) v_{-1}(b)=  -2f_1(b)v_{-1}(b), &\text{ if  $i = 1$ and $b$ is below the }\\ &\text{principal  diagonal;}\\ -\epsilon(h_{-1}(b) - h_{1}(b)) f_1(b) v_{-1}(b)=   -2(-\epsilon)f_1(b)v_{-1}(b), &\text{ if  $i = 1$ and $b$ is above the }\\ &\text{principal  diagonal;}\\ (-\epsilon)^{k_1} (h_{-1}(b) - h_{1}(b)) f_1(b)v_{-1}(b) &\text{ if  $i = 1$ and $b$ is on the }\\ \quad\quad =  (-2) (-\epsilon)^{k_1} f_1(b) v_{-1}(b) &\text{principal  diagonal;}\\ 
\end{cases}
\]
by lemma~\ref{L:RelationCoefficientsHandF} (d). Clearly we get that $[H_{\mathcal O}, F_{\mathcal O}](v_i(b)) = -2F_{\mathcal O}(v_i(b))$. 

Finally we have to prove $[E_{\mathcal O}, F_{\mathcal O}] = H_{\mathcal O}$. Assume $b = b(i_1, j_1, k_1)$ with $i_1, j_1 \in {\mathcal I}$, $i_1 \geq j_1$ and $0 \leq k_1 \leq \mu(i_1, j_1)$ and let $i \in Supp(b)$. By our hypothesis, we get that $i_1 \geq 1 > -1 \geq j_1$. Note that $i_1 \geq i \geq j_1$.  

If $i = i_1 = \max(Supp(b))$, then $[E_{\mathcal O}, F_{\mathcal O}](v_i(b)) = (E_{\mathcal O} F_{\mathcal O} - F_{\mathcal O} E_{\mathcal O})(v_i(b))$ is equal to 
\[
\begin{aligned}
&\begin{cases} f_i(b) v_i(b), &\text{if $i \ne 1$;}\\ f_1(b) v_1(b), &\text{if $i = 1$ and $b$ is below the principal diagonal;}\\  (-\epsilon)^2 f_1(b) v_1(b), &\text{if $i = 1$ and $b$ is above the principal diagonal;}\\ (-\epsilon)^{2k_1} f_1(b) v_1(b), &\text{if $i = 1$ and $b$ is on the principal diagonal;}
\end{cases}\\
&= h_i(b) v_i(b) = H_{\mathcal O}(v_i(b))
\end{aligned}
\]
by lemma~\ref{L:RelationCoefficientsHandF} (f).

If $i = j_1 = \min(Supp(b))$, then $[E_{\mathcal O}, F_{\mathcal O}](v_i(b)) = (E_{\mathcal O} F_{\mathcal O} - F_{\mathcal O} E_{\mathcal O})(v_i(b))$ is equal to 
\[
\begin{aligned}
&\begin{cases} - f_{i + 2}(b) v_i(b), &\text{if $i < -1$;}\\ -f_1(b) v_{-1}(b), &\text{if $i = -1$ and $b$ is below the principal diagonal;}\\  -(-\epsilon)^2 f_1(b) v_{-1}(b), &\text{if $i = -1$ and $b$ is above the principal diagonal;}\\ -(-\epsilon)^{2k_1} f_1(b) v_{-1}(b), &\text{if $i = -1$ and $b$ is on the principal diagonal;}
\end{cases}\\
&= h_i(b) v_i(b) = H_{\mathcal O}(v_i(b))
\end{aligned}
\]
by lemma~\ref{L:RelationCoefficientsHandF} (e).

If $i_1 > i > j_1$, then $[E_{\mathcal O}, F_{\mathcal O}](v_i(b)) = (E_{\mathcal O} F_{\mathcal O} - F_{\mathcal O} E_{\mathcal O})(v_i(b))$ is equal to 
\[
\begin{aligned}
&\begin{cases} (f_i(b) - f_{i + 2}(b)) v_i(b), &\text{if $i >1$;}\\ (f_i(b) - f_{i + 2}(b)) v_{i}(b), &\text{if $i < -1$;}\\  (f_1(b)  - f_3(b))v_{1}(b), &\text{if $i = 1$ and $b$ is below the principal diagonal;}\\ ((-\epsilon)^2f_1(b)  - f_3(b))v_{1}(b), &\text{if $i = 1$ and $b$ is above the principal diagonal;}\\ ((-\epsilon)^{2k_1} f_1(b) - f_3(b)) v_{1}(b), &\text{if $i = 1$ and $b$ is on the principal diagonal;}\\ (f_{-1}(b)  - f_1(b))v_{-1}(b), &\text{if $i = -1$ and $b$ is below the principal diagonal;}\\ (f_{-1}(b)  - (-\epsilon)^2 f_1(b))v_{-1}(b), &\text{if $i = -1$ and $b$ is above the principal diagonal;}\\ ( f_{-1}(b) - (-\epsilon)^{2k_1}f_1(b)) v_{-1}(b), &\text{if $i = -1$ and $b$ is on the principal diagonal;}
\end{cases}\\
&= (f_i(b) - f_{i + 2}(b))v_i(b) = h_i(b) v_i(b) = H_{\mathcal O}(v_i(b))
\end{aligned}
\]
by lemma~\ref{L:RelationCoefficientsHandF} (g).

Thus we have proved that the triple $E_{\mathcal O}$, $H_{\mathcal O}$ and $F_{\mathcal O}$ is an ${\mathcal I}$-graded Jacobson-Morozov triple when $\vert {\mathcal I} \vert$ is even. 

We now want to prove that 
\[
\{(V({\mathcal O}), \phi_{\mathcal O}, \langle\ , \ \rangle) \mid \text{ ${\mathcal O}$
belongs to the set of $\langle \tau \rangle$-orbits  in ${\mathbf B}$}\}
\]
is a set of representatives of the isomorphism classes of indecomposable $\epsilon$-representa~- tions of $(\Omega, \sigma)$ when $\vert {\mathcal I} \vert$ is even.

Let the ${\mathcal I}$-box $b = b(i_1, j_1, k_1) \in {\mathbf B}$ be an element of the $\langle \tau \rangle$-orbit ${\mathcal O}$ with $i_1, j_1 \in {\mathcal I}$, $i_1 \geq j_1$ and $0 \leq k_1 \leq \mu(i_1, j_1)$. If we consider the restriction of $E_{\mathcal O}$ to $V_i(b)$: $\phi_{b, i} = E_{\mathcal O}\vert_{V_i(b)}:V_i(b) \rightarrow V_{i + 2}(b)$ for each $i \in Supp(b)$, $i \ne \max(Supp(b))$, we get a representation $\phi_b:V(b) \rightarrow V(b)$ of the quiver $\Omega$ of type $A_{2m}$. Here $\vert {\mathcal I} \vert = 2m$. From our definition of $E_{\mathcal O}$, we get easily that 
\begin{itemize}
\item $\phi_b$ is an indecomposable representation of $A_{2m}$ of dimension $\dim(\phi_b) = (\beta_i)_{i \in {\mathcal I}}$ where 
\[
\beta_i = \begin{cases} 1, &\text{if $i_1 \geq i \geq j_1$;}\\ 0, &\text{otherwise.}\end{cases}
\]
\item $\phi_{\mathcal O}:V({\mathcal O}) \rightarrow V({\mathcal O})$ is the direct sum $\oplus_{b \in {\mathcal O}} \phi_b$ as a representation of $A_{2m}$.
\end{itemize}
Because $\langle E_{\mathcal O}(u), v\rangle_{\mathcal O} + \langle u, E_{\mathcal O}(v)\rangle_{\mathcal O} = 0$ for all $u, v \in V({\mathcal O})$, we get that the representation $(V({\mathcal O}), \phi_{\mathcal O}, \langle\ , \ \rangle_{\mathcal O})$ is an $\epsilon$-representation. To show that $(V({\mathcal O}), \phi_{\mathcal O}, \langle \  , \  \rangle_{\mathcal O})$ is  indecomposable, we can use theorem~\ref{T:CaracterisationRepresentation}.

Each $\langle \tau \rangle$-orbit ${\mathcal O}$ contains at least one ${\mathcal I}$-box $b = b(i_1, j_1, k_1) \in {\mathbf B}$ such that $(i_1 + j_1) \geq 0$. If $(i_1 + j_1) > 0$, this ${\mathcal I}$-box is unique in the orbit, while if $(i_1 + j_1) = 0$, then there are $\mu_{\max}$ ${\mathcal I}$-boxes in the orbit. If $(i_1 + j_1) > 0$, then $\tau(b) \ne b$ and ${\mathcal O} = \{b, \tau(b)\}$. We have $i_1 \geq j_1$ and $i_1 \geq \vert j_1 \vert$ because  $(i_1 + j_1) > 0$.  As in the definition~\ref{DefDim}, we will write the dimension vector of $(V({\mathcal O}), \phi_{\mathcal O}, \langle\  , \  \rangle_{\mathcal O})$ as representation of the symmetric quiver $({\Omega}, \sigma)$ by
\[
\dim(V({\mathcal O}))  = (\dim(V_i({\mathcal O}))_{i \in {\mathcal I}_+} = (\alpha_i)_{i \in {\mathcal I}_+} = \alpha.
\]
We will write also the dimensions of $\dim(\phi_b)$ and $\dim(\phi_{\tau(b)})$ as representation of the quiver $A_{2m}$ by 
\[
\dim(\phi_b) = \beta' = (\beta'_i)_{i \in {\mathcal I}} \quad \text{ and } \quad  \dim(\phi_{\tau(b)}) = \beta = (\beta_i)_{i \in {\mathcal I}}.
\]  

If $j_1 > 0$, then 
\[
\beta'_i = \begin{cases}1, &\text{if $-j_1 \geq i \geq -i_1$;}\\ 0, &\text{otherwise; } \end{cases} \quad  \quad 
\beta_i = \begin{cases}1, &\text{if $i_1 \geq i \geq j_1$;}\\ 0, &\text{otherwise;} \end{cases}
\]
for $i \in {\mathcal I}$ and 
\[
\alpha_i =\begin{cases}  1,&\text{if $i_1 \geq i \geq j_1$;}\\ 0, &\text{otherwise;}\end{cases}
\]
for $i \in {\mathcal I}_+$.  Using proposition~\ref{Decomposition_even} (a) (i) and (b)(i) when $j_1 > 1$, proposition~\ref{Decomposition_even} (a) (ii) and (b) (ii) for $\alpha^0$ when $j_1 = 1$ and theorem~\ref{T:CaracterisationRepresentation}, this shows that $\phi_{\mathcal O}$ is indecomposable and isomorphic to the indecomposable representation of dimension $\alpha$.

If $0 > j_1$, then 
\[
\beta'_i = \begin{cases}1, &\text{if $-j_1 \geq i \geq -i_1$;}\\ 0, &\text{otherwise; } \end{cases} \quad  \quad 
\beta_i = \begin{cases}1, &\text{if $i_1 \geq i \geq j_1$;}\\ 0, &\text{otherwise;} \end{cases}
\]
for $i \in {\mathcal I}$ and 
\[
\alpha_i =\begin{cases}  2,&\text{if $i_1 \geq i > 0$;}\\ 1,&\text{if $i_1 \geq i > \vert j_1 \vert$;}\\ 0, &\text{otherwise;}\end{cases}
\]
for $i \in {\mathcal I}_+$.  Using proposition~\ref{Decomposition_even} (a) (iii) and (b)(iii)  and theorem~\ref{T:CaracterisationRepresentation}, this shows that $\phi_{\mathcal O}$ is indecomposable and isomorphic to the indecomposable representation of dimension $\alpha$.

If $(i_1 + j_1) = 0$, then the ${\mathcal I}$-box $b(i_1, j_1, k_1)$ is on the principal diagonal and $i_1 > 0$, because $i_1 \geq j_1$. There are two cases to consider: either $\tau(b) \ne b$ or $\tau(b) = b$. When $\tau(b) \ne b$, then ${\mathcal O} = \{b, \tau(b)\}$ and by lemma~\ref{L:TauBox} (d), $\mu_{\max} = 1$, $\epsilon = 1$, $\tau(b) = b(i_1, j_1, (1 - k_1))$. So 
\[
\beta'_i = \beta_i = \begin{cases}1, &\text{if $i_1 \geq i \geq -i_1$;}\\ 0, &\text{otherwise; } \end{cases}
\]
for $i \in {\mathcal I}$ and 
\[
\alpha_i =\begin{cases}  2,&\text{if $i_1 \geq i > 0$;}\\  0, &\text{otherwise;}\end{cases}
\]
for $i \in {\mathcal I}_+$. Using proposition~\ref{Decomposition_even} (a) (iv) and theorem~\ref{T:CaracterisationRepresentation}, this shows that $\phi_{\mathcal O}$ is indecomposable and isomorphic to the indecomposable representation of dimension $\alpha$. 

When $\tau(b) = b$, then ${\mathcal O} = \{b \}$ and by lemma~\ref{L:TauBox} (e), $\mu_{\max} = 0$, $\epsilon = -1$, $V({\mathcal O}) = V_b$. We will write $\dim(\phi_b) = \beta = (\beta_i)_{i \in {\mathcal I}}$. So 
\[
\beta_i = \begin{cases}1, &\text{if $i_1 \geq i \geq -i_1$;}\\ 0, &\text{otherwise; } \end{cases}
\]
for $i \in {\mathcal I}$ and 
\[
\alpha_i =\begin{cases}  1,&\text{if $i_1 \geq i > 0$;}\\  0, &\text{otherwise.}\end{cases}
\]
for $i \in {\mathcal I}_+$. Using proposition~\ref{Decomposition_even} (b) (ii) for $\alpha^1$ and theorem~\ref{T:CaracterisationRepresentation}, this shows that $\phi_{\mathcal O}$ is indecomposable and isomorphic to the indecomposable representation of dimension $\alpha^1$. 

To finish the proof when $\vert {\mathcal I} \vert$ is even, it suffices to notice by inspection of the dimension vectors of $(V({\mathcal O}), \phi_{\mathcal O}, \langle\  , \  \rangle_{\mathcal O})$,  the direct sum decomposition $\phi_{\mathcal O} = \oplus_{b \in {\mathcal O}} \phi_b$ and theorem~\ref{T:CaracterisationRepresentation}, that $(V({\mathcal O}), \phi_{\mathcal O}, \langle\  , \  \rangle_{\mathcal O})$ and $(V({\mathcal O'}), \phi_{\mathcal O'}, \langle\  , \  \rangle_{{\mathcal O}'})$ are not isomorphic for two distinct $\langle \tau \rangle$-orbits ${\mathcal O}$ and ${\mathcal O}'$. By remark~\ref{R:CardinalityIndecomposable}, we get the result that $\{(V({\mathcal O}), \phi_{\mathcal O}, \langle\ , \ \rangle_{\mathcal O}) \mid \text{ ${\mathcal O}$  belongs to the set of $\langle \tau \rangle$-orbits  in ${\mathbf B}$}\}$   is a set of representatives of the isomorphism classes of indecomposable $\epsilon$-representations of $(\Omega, \sigma)$.


{\bf Odd case:}

Consider secondly the proof of the proposition when $\vert {\mathcal I} \vert$ is odd. In this case, $0 \in {\mathcal I}$. We want to prove that $\langle X(u), v\rangle_{\mathcal O} + \langle u, X(v)\rangle_{\mathcal O} = 0$ for all basis vectors $u, v \in {\mathcal B}_{\mathcal O}$ and $X \in \{E_{\mathcal O}, H_{\mathcal O}, F_{\mathcal O}\}$. We will start with $X = E_{\mathcal O}$.

If the orbit ${\mathcal O}$ has no overlapping ${\mathcal I}$-box supports, then we saw in \ref{SS:EHFDefinition} that this is equivalent to $0 \not\in Supp(b)$ or equivalently $0 \not \in Supp(\tau(b))$,  when $b \in {\mathcal O}$. We have that $\langle E_{\mathcal O}(v_{i_1}(b_1)), v_{i_2}(b_2)\rangle_{\mathcal O}$  is equal to 
\[
\begin{cases} {\phantom -}0, &\text{if $i_1 = \max(Supp(b_1))$;} \\  {\phantom -}\langle v_{i_1 + 2}(b_1), v_{i_2}(b_2)\rangle_{\mathcal O}, &\text{if $i_1 \ne \max(Supp(b_1))$ and $i_1 > 0$;}\\  -\langle v_{i_1 + 2}(b_1), v_{i_2}(b_2)\rangle_{\mathcal O}, &\text{if $i_1 \ne \max(Supp(b_1))$ and $i_1 < 0$;}\end{cases}
\]
and that $\langle v_{i_1}(b_1), E_{\mathcal O}(v_{i_2}(b_2))\rangle_{\mathcal O}$  is equal to 
\[
\begin{cases}  {\phantom -}0, &\text{if $i_2 = \max(Supp(b_2))$;} \\  {\phantom -}\langle v_{i_1}(b_1), v_{i_2 + 2}(b_2)\rangle_{\mathcal O}, &\text{if $i_2 \ne \max(Supp(b_1))$ and $i_2 > 0$;}\\  -\langle v_{i_1}(b_1), v_{i_2 + 2}(b_2)\rangle_{\mathcal O}, &\text{if $i_2 \ne \max(Supp(b_2))$ and $i_2 < 0$.}\end{cases}
\]

If $i_1 = \max(Supp(b_1))$, then $\langle E_{\mathcal O}(v_{i_1}(b_1)), v_{i_2}(b_2)\rangle_{\mathcal O} = 0$ by the definition of $E_{\mathcal O}$ and we have $\langle v_{i_1}(b_1), E_{\mathcal O}(v_{i_2}(b_2))\rangle_{\mathcal O} = 0$, otherwise if $\langle v_{i_1}(b_1), E_{\mathcal O}(v_{i_2}(b_2))\rangle_{\mathcal O} \ne 0$, then $i_2 \ne \max(Supp(b_2))$, $b_2 = \tau(b_1)$ and $i_1 + i_2 +2 = 0$. This is impossible, because $Supp(b_2) = \sigma(Supp(b_1))$ and $\min(Supp(b_2))= -i_1$, but from $i_1 + i_2 + 2 = 0$, we get a contradiction because $i_2 = (-i_1 - 2) \in Supp(b_2)$.  Thus
\[
\langle E_{\mathcal O}(v_{i_1}(b_1)), v_{i_2}(b_2)\rangle_{\mathcal O} + \langle v_{i_1}(b_1), E_{\mathcal O}(v_{i_2}(b_2))\rangle_{\mathcal O} = 0
\]
if $i_1 = \max(Supp(b_1))$. By symmetry, we also have this equality if $i_2 = \max(Supp(b_2))$.

So we can now assume that $i_1 \ne \max(Supp(b_1))$ and $i_2 \ne \max(Supp(b_2))$. Moreover if either $\langle E_{\mathcal O}(v_{i_1}(b_1)), v_{i_2}(b_2)\rangle_{\mathcal O} \ne 0$ or $\langle v_{i_1}(b_1), E_{\mathcal O}(v_{i_2}(b_2))\rangle_{\mathcal O} \ne 0$, then $b_2 = \tau(b_1)$ and $i_1 + i_2 + 2 = 0$ from our definition of the bilinear form $\langle\  , \  \rangle$. From our hypothesis on the orbit ${\mathcal O}$, we have that $\tau(b_1) \ne b_1$ and $\tau(b_2) \ne b_2$. If $b_1$ is not on the principal diagonal, then $b_2 = \tau(b_1)$ is also not on the principal diagonal. If $i_1 > 0$, then $i_2 = -i_1 -2 < 0$, because $i_1 + i_2 + 2 = 0$. Thus we get 
\begin{multline*}
\langle E_{\mathcal O}(v_{i_1}(b_1)), v_{i_2}(b_2)\rangle_{\mathcal O} + \langle v_{i_1}(b_1), E_{\mathcal O}(v_{i_2}(b_2))\rangle_{\mathcal O}\\ =\langle v_{i_1 + 2}(b_1), v_{i_2}(b_2)\rangle + \langle v_{i_1}(b_1), -v_{i_2 + 2}(b_2)\rangle = 1 - 1 = 0.
\end{multline*}
If $i_1 < 0$ and $i_1 \ne -2$, then $i_2 = -(i_1 + 2) > 0$. Thus we get that 
\begin{multline*}
\langle E_{\mathcal O}(v_{i_1}(b_1)), v_{i_2}(b_2)\rangle_{\mathcal O} + \langle v_{i_1}(b_1), E_{\mathcal O}(v_{i_2}(b_2))\rangle_{\mathcal O}\\ =\langle -v_{i_1 + 2}(b_1), v_{i_2}(b_2)\rangle_{\mathcal O} + \langle v_{i_1}(b_1), v_{i_2 + 2}(b_2)\rangle_{\mathcal O} = -\epsilon + \epsilon = 0.
\end{multline*}
Note that  we cannot have $i_1 = -2$, otherwise we get that $i_2 = 0 \in Supp(\tau(b_2))$, but this contradicts our hypothesis on the orbit. 

Note that $b_1$ cannot be on the principal diagonal, otherwise $b_2 = \tau(b_1)$ is also on the principal diagonal  and $Supp(b_2) = Supp(b_1)$. This contradicts the fact that ${\mathcal O}$ has no overlapping ${\mathcal I}$-box supports.

We will now consider the case where the orbit has overlapping  ${\mathcal I}$-box supports. If $i_1 = \max(Supp(b_1))$, the proof is the same as when the orbit ${\mathcal O}$ has no overlapping ${\mathcal I}$-box supports  and is left to the reader. As above we only have to consider the case where $i_1 \ne \max(Supp(b_1))$ and $i_2 \ne \max(Supp(b_2))$. Moreover if either $\langle E_{\mathcal O}(v_{i_1}(b_1)), v_{i_2}(b_2)\rangle_{\mathcal O} \ne 0$ or $\langle v_{i_1}(b_1), E_{\mathcal O}(v_{i_2}(b_2))\rangle_{\mathcal O} \ne 0$, then $b_2 = \tau(b_1)$ and $i_1 + i_2 + 2 = 0$ from our definition of $\langle\  , \  \rangle_{\mathcal O}$. Clearly if $b_1$ is below (respectively above, on) the principal diagonal, then $b_2$ is above (respectively below, on) the principal diagonal. In the case where the ${\mathcal I}$-boxes $b_1$ and $b_2$ are on the principal diagonal, then we will write $b_1 = b_1(i'_1, j'_1, k'_1)$ and $b_2 = \tau(b_1) = b(i'_1, j'_1, k''_1)$ where $i'_1, j'_1 \in {\mathcal I}$, $i'_1 \geq j'_1$ and $k''_1 \in \{k'_1, (1 - k'_1)\}$. More precisely, if $\mu_{\max} = 0$, then $k''_1 = k'_1$  and if $\mu_{\max} = 1$, then $k''_1 = (1 - k'_1)$.

We have that $\langle E_{\mathcal O}(v_{i_1}(b_1)), v_{i_2}(b_2)\rangle_{\mathcal O}$  is equal to 
\[
\begin{cases} 0, &\text{if $i_1 = \max(Supp(b_1))$;}\\ \langle v_{i_1 + 2}(b_1), v_{i_2}(b_2)\rangle_{\mathcal O}, &\text{if $i_1 \ne \max(Supp(b_1))$ and $i_1 > 0$;}\\ \langle -v_{i_1 + 2}(b_1), v_{i_2}(b_2)\rangle_{\mathcal O}, &\text{if $i_1 < 0$;}\\ \langle \epsilon v_{2}(b_1), v_{-2}(b_2)\rangle_{\mathcal O}, &\text{if $i_1 \ne \max(Supp(b_1))$,  $i_1 = 0$ and $b_1$ is below the}\\ &\text{principal diagonal;}\\ \langle  v_{2}(b_1), v_{-2}(b_2)\rangle_{\mathcal O}, &\text{if $i_1 \ne \max(Supp(b_1))$,  $i_1 = 0$ and $b_1$ is above the}\\ &\text{principal diagonal;}\\ \langle \epsilon^{k'_1} v_{2}(b_1), v_{-2}(b_2)\rangle_{\mathcal O}, &\text{if $i_1 \ne \max(Supp(b_1))$,  $i_1 = 0$ and $b_1$ is on the}\\ &\text{principal diagonal;}
\end{cases}
\]
and $\langle v_{i_1}(b_1), E_{\mathcal O}(v_{i_2}(b_2))\rangle_{\mathcal O}$  is equal to 
\[
\begin{cases} 0, &\text{if $i_2 = \max(Supp(b_2))$;}\\ \langle v_{i_1}(b_1), v_{i_2 + 2}(b_2)\rangle_{\mathcal O}, &\text{if $i_2 \ne \max(Supp(b_2))$ and $i_2 > 0$;}\\ \langle v_{i_1}(b_1), -v_{i_2 + 2}(b_2)\rangle_{\mathcal O}, &\text{if $i_2 < 0$;}\\ \langle v_{-2}(b_1),  \epsilon v_{2}(b_2)\rangle_{\mathcal O}, &\text{if $i_2 \ne \max(Supp(b_2))$,  $i_2 = 0$ and $b_2$ is below the}\\ &\text{principal diagonal;}\\ \langle  v_{-2}(b_1), v_{2}(b_2)\rangle_{\mathcal O}, &\text{if $i_2 \ne \max(Supp(b_2))$,  $i_2 = 0$ and $b_2$ is above the}\\ &\text{principal diagonal;}\\ \langle  v_{-2}(b_1), \epsilon^{k''_1} v_{2}(b_2)\rangle_{\mathcal O}, &\text{if $i_2 \ne \max(Supp(b_2))$,  $i_2 = 0$ and $b_2$ is on the}\\ &\text{principal diagonal.}
\end{cases}
\]
The proofs when $i_1 > 0$ and  when $i_1 < -2$ are similar to the proofs when the orbit ${\mathcal O}$ has no overlapping ${\mathcal I}$-box supports and are left to the reader. We will only consider the cases where $i_1 = 0$ and $i_1 = -2$.

If $i_1 = 0$, then $i_2 = -2$ and consequently $2 \in Supp(b_1)$, because $-2 \in Supp(\tau(b_1))$. So $i_1 \ne \max(Supp(b_1))$.  We have that $\langle E_{\mathcal O}(v_{0}(b_1)), v_{-2}(b_2)\rangle_{\mathcal O}$  is equal to 
\[
\begin{cases}  \langle \epsilon v_{2}(b_1), v_{-2}(b_2)\rangle_{\mathcal O} = \epsilon, &\text{if  $b_1$ is below the principal diagonal;}\\ \langle  v_{2}(b_1), v_{-2}(b_2)\rangle_{\mathcal O} = 1, &\text{if $b_1$ is above the principal diagonal;}\\ \langle \epsilon^{k'_1} v_{2}(b_1), v_{-2}(b_2)\rangle_{\mathcal O} = \epsilon^{k'_1}, &\text{if  $b_1$ is on the principal diagonal;}
\end{cases}
\]
and $\langle v_{0}(b_1), E_{\mathcal O}(v_{-2}(b_2))\rangle_{\mathcal O} = \langle v_0(b_1), - v_0(b_2)\rangle_{\mathcal O}$  is equal to 
\[
\begin{cases} -\epsilon, &\text{if $b_1$ is below the principal diagonal;}\\ -1, &\text{if  $b_1$ is above the principal diagonal;}\\ -\epsilon^{k'_1}, &\text{if $b_1$ is on the principal diagonal and $\tau(b_1) \ne b_1$;} \\ -1, &\text{if $b_1$ is on the principal diagonal and $\tau(b_1) = b_1$;}
\end{cases}
\]

From this, we get easily if $b_1$ is below or above the principal diagonal that
\[
\langle E_{\mathcal O}(v_{0}(b_1)), v_{-2}(b_2)\rangle_{\mathcal O} + \langle v_{0}(b_1), E_{\mathcal O}(v_{-2}(b_2))\rangle_{\mathcal O} = 0. \quad\quad (*)
\]
In the case that the ${\mathcal I}$-box $b_1$ is on the principal diagonal, we have to consider two cases: either $\tau(b_1) \ne b_1$ or $\tau(b_1) = b_1$. If $\tau(b_1) \ne b_1$, we easily get the equation (*). If $\tau(b_1) = b_1$, then, by lemma~\ref{L:TauBox} (f), $\epsilon = 1$ and $(*)$ is also verified.

If $i_1 = -2$, then $i_2 = 0$ and by symmetry we also get $(*)$. 

If $X = H_{\mathcal O}$, then for either  $\langle H_{\mathcal O}(v_{i_1}(b_1)), v_{i_2}(b_2)\rangle_{\mathcal O} =  h_{i_1}(b_1) \langle v_{i_1}(b_1), v_{i_2}(b_2)\rangle_{\mathcal O} \ne 0$ or for $\langle v_{i_1}(b_1), H_{\mathcal O}(v_{i_2}(b_2))\rangle_{\mathcal O} =  h_{i_2}(b_2) \langle v_{i_1}(b_1), v_{i_2}(b_2)\rangle_{\mathcal O} \ne 0$, we must have $b_2 = \tau(b_1)$ and $i_1 + i_2 = 0$. Note that in both cases, $\langle v_{i_1}(b_1), v_{i_2}(b_2)\rangle_{\mathcal O} \ne 0$. So we can assume that $b_2 = \tau(b_1)$ and $i_2 = -i_1$. By lemma~\ref{L:RelationCoefficientsHandF} (a), we get that 
\begin{multline*}
\langle H_{\mathcal O}(v_{0}(b_1)), v_{-2}(b_2)\rangle_{\mathcal O} + \langle v_{0}(b_1), H_{\mathcal O}(v_{-2}(b_2))\rangle_{\mathcal O} \\ =(h_{i_1}(b_1) + h_{i_2}(b_2))\langle v_{i_1}(b_1), v_{i_2}(b_2)\rangle_{\mathcal O} = 0.
\end{multline*}

Now we consider $X = F_{\mathcal O}$. If the orbit has no overlapping ${\mathcal I}$-box supports, then we have that $\langle F_{\mathcal O}(v_{i_1}(b_1)), v_{i_2}(b_2)\rangle_{\mathcal O}$ is equal to 
\[
\begin{cases} {\phantom -}0, &\text{if $i_1 = \min(Supp(b_1))$;}\\  {\phantom -} \langle f_{i_1}(b_1) v_{i_1 - 2}(b_1), v_{i_2}(b_2)\rangle_{\mathcal O}, &\text{if $i_1 \ne\min(Supp(b_1))$ and $i_1 > 0$;}\\ - \langle f_{i_1}(b_1) v_{i_1 - 2}(b_1), v_{i_2}(b_2)\rangle_{\mathcal O}, &\text{if $i_1 \ne \min(Supp(b_1))$ and $i_1 < 0$;} \end{cases}
\]
and 
$\langle v_{i_1}(b_1), F_{\mathcal O}(v_{i_2}(b_2))\rangle_{\mathcal O}$ is equal to 
\[
\begin{cases} {\phantom -}0, &\text{if $i_2 = \min(Supp(b_2))$;}\\  {\phantom -} \langle  v_{i_1}(b_1), f_{i_2}(b_2)(v_{i_2 - 2}(b_2)\rangle_{\mathcal O}, &\text{if $i_2 \ne\min(Supp(b_2))$ and $i_2 > 0$;}\\ - \langle  v_{i_1}(b_1), f_{i_2}(b_2) v_{i_2 - 2}(b_2)\rangle_{\mathcal O}, &\text{if $i_2 \ne \min(Supp(b_2))$ and $i_2 < 0$;} \end{cases}
\]

If $i_1 = \min(Supp(b_1))$, then $\langle F_{\mathcal O}(v_{i_1}(b_1)), v_{i_2}(b_2)\rangle_{\mathcal O} = 0$ by definition of $F_{\mathcal O}$ and we have $\langle v_{i_1}(b_1), F_{\mathcal O}(v_{i_2}(b_2))\rangle_{\mathcal O} = 0$, otherwise if  $\langle v_{i_1}(b_1), F_{\mathcal O}(v_{i_2}(b_2))\rangle \ne 0$, then $i_2 \ne \min(Supp(b_2))$, $b_2 = \tau(b_1)$ and $i_1 + i_2 - 2 = 0$. This is impossible, because $Supp(b_2) = \sigma(Supp(b_1))$ and $\max(Supp(b_2)) = -i_1$, but from $i_2 = (-i_1 + 2) \in Supp(b_2)$ we get the contradiction. Thus   
\[
\langle F_{\mathcal O}(v_{i_1}(b_1)), v_{i_2}(b_2)\rangle_{\mathcal O} + \langle v_{i_1}(b_1), F_{\mathcal O}(v_{i_2}(b_2))\rangle_{\mathcal O} = 0
\]
if $i_1 = \min(Supp(b_1))$. By symmetry, we also have this last equality if $i_2 = \min(Supp(b_2))$. From now on, we will assume that $i_1 \ne \min(Supp(b_1))$ and $i_2 \ne \min(Supp(b_2))$.  

If either $\langle F_{\mathcal O}(v_{i_1}(b_1)), v_{i_2}(b_2)\rangle_{\mathcal O} \ne 0$ or  $\langle v_{i_1}(b_1), F_{\mathcal O}(v_{i_2}(b_2))\rangle_{\mathcal O} \ne  0$, then $b_2 = \tau(b_1)$ and $i_1 + i_2 - 2 = 0$ from our definition of the bilinear form $\langle\  , \  \rangle_{\mathcal O}$. From our hypothesis on the orbit ${\mathcal O}$, we have that $\tau(b_1) \ne b_1$ and $\tau(b_2) \ne b_2$. If $b_1$ is not on the principal diagonal, then $b_2 = \tau(b_1)$ is also not on the principal diagonal. If $i_1 > 0$ and $i_1 \ne 2$, then $i_2 = -i_1 + 2 < 0$. Thus we get that 
\begin{multline*}
\langle F_{\mathcal O}(v_{i_1}(b_1)), v_{i_2}(b_2)\rangle_{\mathcal O} + \langle v_{i_1}(b_1), F_{\mathcal O}(v_{i_2}(b_2))\rangle_{\mathcal O} \\ =\langle f_{i_1}(b_1) v_{i_1 - 2}(b_1), v_{i_2}(b_2)\rangle_{\mathcal O} - f_{i_2}(b_2) \langle v_{i_1}(b_1), v_{i_2 - 2}(b_2)\rangle_{\mathcal O} \\ (f_{i_1}(b_1) - f_{i_2}(b_2)) = 0
\end{multline*}
by lemma~\ref{L:RelationCoefficientsHandF} (b) and our definition of $\langle\  , \  \rangle_{\mathcal O}$.

If $i_1 < 0$, then $i_2 = - i_1 + 2 > 2$ and we get that 
\begin{multline*}
\langle F_{\mathcal O}(v_{i_1}(b_1)), v_{i_2}(b_2)\rangle_{\mathcal O} + \langle v_{i_1}(b_1), F_{\mathcal O}(v_{i_2}(b_2))\rangle_{\mathcal O} \\ = -\langle f_{i_1}(b_1) v_{i_1 - 2}(b_1), v_{i_2}(b_2)\rangle_{\mathcal O} + f_{i_2}(b_2) \langle v_{i_1}(b_1), v_{i_2 - 2}(b_2)\rangle_{\mathcal O} \\ (-f_{i_1}(b_1) + f_{i_2}(b_2))\epsilon = 0
\end{multline*}
by lemma~\ref{L:RelationCoefficientsHandF} (b) and our definition of $\langle\  , \  \rangle_{\mathcal O}$.

To conclude this case, we can observe that $i_1 \ne 2$. Otherwise $i_2 = 0$ and, by \ref{SS:EHFDefinition}, this means that $0 \in Supp(b_1) \cap Supp(b_2)$ and ${\mathcal O}$ has overlapping ${\mathcal I}$-box supports contradicting our hypothesis on ${\mathcal O}$. 

Note that $b_1$ cannot be on the principal diagonal, otherwise $b_2 = \tau(b_1)$ is also on the principal diagonal and $Supp(b_2) = Supp(b_1)$. So ${\mathcal O}$ has overlapping ${\mathcal I}$-box supports contradicting our hypothesis on ${\mathcal O}$. We also have $i_1 \ne 0$, otherwise ${\mathcal O}$ has overlapping ${\mathcal I}$-box supports contradicting our hypothesis.

Consider now the case where the orbit ${\mathcal O}$ has overlapping ${\mathcal I}$-box supports. If $i_1 = \min(Supp(b_1))$, then the proof is the same as when the orbit ${\mathcal O}$ has no overlapping ${\mathcal I}$-box supports and is left to the reader. As above we only have to consider the case where $i_1 \ne \min(Supp(b_1))$ and $i_2 \ne \min(Supp(b_2))$. Moreover if either $\langle F_{\mathcal O}(v_{i_1}(b_1)), v_{i_2}(b_2)\rangle_{\mathcal O} \ne 0$ or $\langle v_{i_1}(b_1), F_{\mathcal O}(v_{i_2}(b_2))\rangle_{\mathcal O} \ne 0$, then $b_2 = \tau(b_1)$ and $i_1 + i_2 - 2 = 0$.  Clearly if $b_1$ is below (respectively above, on) the principal diagonal, then $b_2$ is above (respectively below, on) the principal diagonal. In the case where the ${\mathcal I}$-boxes $b_1$ and $b_2$ are on the principal diagonal, then we will write $b_1 = b(i'_1, j'_1, k'_1)$ and $b_2 = \tau(b_1) = b(i'_1, j'_1, k''_1)$ where $i'_1, j'_1 \in {\mathcal I}$, $i'_1 \geq j'_1$ and $k''_1 \in \{k'_1, (1 - k'_1)\}$. More precisely, if $\mu_{\max} = 0$, then $k''_1 = k'_1$  and if $\mu_{\max} = 1$, then $k''_1 = (1 - k'_1)$.

We have that $\langle F_{\mathcal O}(v_{i_1}(b_1)), v_{i_2}(b_2)\rangle_{\mathcal O}$ is equal to 
\[
\begin{cases}0, &\text{if $i_1 = \min(Supp(b_1))$;} \\ \langle f_{i_1}(b_1) v_{i_1 - 2}(b_1), v_{i_2}(b_2)\rangle_{\mathcal O}, &\text{if $i_1 > 2$;}\\ \langle -f_{i_1}(b_1) v_{i_1 - 2}(b_1), v_{i_2}(b_2)\rangle_{\mathcal O}, &\text{if $i_1 \ne \min(Supp(b_1))$ and $i_1 \leq 0$;}\\ \langle \epsilon f_{2}(b_1) v_{0}(b_1), v_{0}(b_2)\rangle_{\mathcal O}, &\text{if $i_1 = 2$ and $b_1$ is below the principal diagonal;}\\ \langle f_{2}(b_1) v_{0}(b_1), v_{0}(b_2)\rangle_{\mathcal O}, &\text{if $i_1 = 2$ and $b_1$ is above the principal diagonal;}\\ \langle \epsilon^{k'_1} f_{2}(b_1) v_{0}(b_1), v_{0}(b_2)\rangle_{\mathcal O}, &\text{if $i_1 = 2$ and $b_1$ is on the principal diagonal;}\\
\end{cases}
\]
and $\langle v_{i_1}(b_1), F_{\mathcal O}(v_{i_2}(b_2))\rangle_{\mathcal O}$ is equal to 
\[
\begin{cases}0, &\text{if $i_2 = \min(Supp(b_2))$;} \\ \langle v_{i_1}(b_1),  f_{i_2}(b_2) v_{i_2 - 2}(b_2)\rangle_{\mathcal O}, &\text{if $i_2 > 2$;}\\ \langle  v_{i_1}(b_1), -f_{i_2}(b_2)v_{i_2 - 2}(b_2)\rangle_{\mathcal O}, &\text{if $i_2 \ne \min(Supp(b_2))$ and $i_2 \leq 0$;}\\ \langle v_{0}(b_1),  \epsilon f_{2}(b_2) v_{0}(b_2)\rangle_{\mathcal O}, &\text{if $i_2 = 2$ and $b_2$ is below the principal diagonal;}\\ \langle v_{0}(b_1),  f_{2}(b_2) v_{0}(b_2)\rangle_{\mathcal O}, &\text{if $i_2 = 2$ and $b_2$ is above the principal diagonal;}\\ \langle  v_{0}(b_1), \epsilon^{k''_1} f_{2}(b_2) v_{0}(b_2)\rangle_{\mathcal O}, &\text{if $i_2 = 2$ and $b_2$ is on the principal diagonal.}\\
\end{cases}
\]

The proofs when $i_1 > 2$ and when $i_1 < 0$ are similar  to the proofs when the orbit ${\mathcal O}$ has no overlapping ${\mathcal I}$-box supports and are left to the reader. We will just consider the cases where $i_1 = 0$ and $i_1 = 2$.  

If $i_1 = 0$, then $i_2 = 2$. We have that  $\langle F_{\mathcal O}(v_{i_1}(b_1)), v_{i_2}(b_2)\rangle_{\mathcal O}$ 
is equal to 
\[
-\langle f_0(b_1) v_{-2}(b_1), v_2(b_2)\rangle = -\epsilon f_0(b_1)
\]
when $b_1$ is either not on the principal diagonal or on the principal diagonal and  $\langle v_0(b_1), F_{\mathcal O}(v_2(b_2))\rangle_{\mathcal O}$ is equal to
\[
\begin{cases}
\langle v_0(b_1), \epsilon f_2(b_2) v_0(b_2)\rangle_{\mathcal O}, &\text{if $b_2$ is below the principal diagonal;}\\ \langle v_0(b_1),  f_2(b_2) v_0(b_2)\rangle_{\mathcal O}, &\text{if $b_2$ is above the principal diagonal;}\\ \langle v_0(b_1), \epsilon^{k''_1} f_2(b_2) v_0(b_2)\rangle_{\mathcal O}, &\text{if $b_2$ is on the principal diagonal.}\\
\end{cases}
\]

Thus $\langle F_{\mathcal O}(v_{i_1}(b_1)), v_{i_2}(b_2)\rangle_{\mathcal O} + \langle v_0(b_1), F_{\mathcal O}(v_2(b_2))\rangle_{\mathcal O}$ is equal to 
\[
\begin{cases} (-\epsilon f_0(b_1) + \epsilon f_2(b_2))&\text{if $b_1$ is above the principal diagonal;}\\ (-\epsilon f_0(b_1) + \epsilon f_2(b_2))&\text{if $b_1$ is below the principal diagonal;}\\ (-\epsilon f_0(b_1) + \epsilon^{k''_1 + k'_1} f_2(b_2))&\text{if $b_1$ is on the principal diagonal and $\tau(b_1) \ne b_1$;}\\  (-\epsilon f_0(b_1) + \epsilon^{k''_1} f_2(b_2))&\text{if $b_1$ is on the principal diagonal and $\tau(b_1) =  b_1$.}\end{cases}
\]
If $b_1$ is not on the principal diagonal, then, by lemma~\ref{L:RelationCoefficientsHandF} (b), we get that 
\[
\langle F_{\mathcal O}(v_{i_1}(b_1)), v_{i_2}(b_2)\rangle_{\mathcal O} + \langle v_0(b_1), F_{\mathcal O}(v_2(b_2))\rangle_{\mathcal O} = 0. \quad\quad (**)
\]
If $b_1$ is on the principal diagonal and $\tau(b_1) = b_1$, then $\epsilon = 1$  by lemma~\ref{L:TauBox} (f). If $b_1$ is on the principal diagonal and $\tau(b_1) \ne b_1$, then $k''_1 = (1 - k'_1)$ and $\epsilon^{k''_1 + k'_1} = \epsilon$ by lemma~\ref{L:TauBox} (d).  In all cases, the equation $(**)$ is verified.

If $i_1 = 2$, then $i_2 = 0$ and we also get that $(**)$ is verified by symmetry. We have now prove that $\langle X(u), v\rangle_{\mathcal O} + \langle u, X(v)\rangle_{\mathcal O} = 0$ for all $X \in \{E_{\mathcal O}, H_{\mathcal O}, F_{\mathcal O}\}$ and all $u, v \in V({\mathcal O})$ when $\vert {\mathcal I} \vert$ is odd.

We want to prove the Jacobson-Morozov triple relations in the case where $\vert {\mathcal I}\vert$ is odd. In other words, $[H_{\mathcal O}, E_{\mathcal O}] = 2E_{\mathcal O}$, $[H_{\mathcal O}, F_{\mathcal O}] = -2F_{\mathcal O}$ and $[E_{\mathcal O}, F_{\mathcal O}] = H_{\mathcal O}$. It is enough to verify these only on the ${\mathcal I}$-graded basis ${\mathcal B}_{\mathcal O}$ of $V({\mathcal O})$.

If the orbit ${\mathcal O}$ has no overlapping ${\mathcal I}$-box supports, then we have already computed above our expressions for $E_{\mathcal O}$, $H_{\mathcal O}$ and $F_{\mathcal O}$. Let $b \in {\mathbf B}$, an ${\mathcal I}$-box in the orbit ${\mathcal O}$ and $i \in Supp(b)$. Then $[H_{\mathcal O}, E_{\mathcal O}](v_i(b)) = (H_{\mathcal O} E_{\mathcal O} - E_{\mathcal O} H_{\mathcal O})(v_i(b))$ is equal to 
\[
\begin{aligned}&
\begin{cases}
0, &\text{if $i = \max(Supp(b))$;}\\ (h_{i + 2}(b) - h_i(b))v_{i + 2}(b), &\text{if $i \ne \max(Supp(b))$ and $i > 0$;}\\  (-h_{i + 2}(b) + h_i(b))v_{i + 2}(b), &\text{if $i \ne \max(Supp(b))$ and $i  < 0$;}\\
\end{cases}\\ &=(h_{i + 2}(b) - h_i(b)) E_{\mathcal O}(v_i(b)) = 2 E_{\mathcal O}(v_i(b))
\end{aligned}
\]
by lemma~\ref{L:RelationCoefficientsHandF} (c). We have also $[H_{\mathcal O}, F_{\mathcal O}](v_i(b)) = (H_{\mathcal O} F_{\mathcal O} - F_{\mathcal O} H_{\mathcal O})(v_i(b))$ is equal to 
\[
\begin{aligned}&
\begin{cases}
0, &\text{if $i = \min(Supp(b))$;}\\ (h_{i - 2}(b) - h_i(b))f_i(b)v_{i - 2}(b), &\text{if $i \ne \min(Supp(b))$ and $i > 0$;}\\  (-h_{i - 2}(b) + h_i(b))f_i(b)v_{i - 2}(b), &\text{if $i \ne \min(Supp(b))$ and $i  < 0$;}\\
\end{cases}\\ &=(h_{i - 2}(b) - h_i(b)) F_{\mathcal O}(v_i(b)) = -2 F_{\mathcal O}(v_i(b))
\end{aligned}
\]
by lemma~\ref{L:RelationCoefficientsHandF} (d).

We now need to compute $[E_{\mathcal O}, F_{\mathcal O}](v_i(b)) = (E_{\mathcal O} F_{\mathcal O} - F_{\mathcal O} E_{\mathcal O})(v_i(b))$. Write $b = b(i_1, j_1, k_1)$ with $i_1, j_1 \in {\mathcal I}$, $i_1 \geq j_1$ and $0 \leq k_1 \leq \mu(i_1, j_1)$ and $i \in Supp(b)$.  Thus $i_1 \geq i \geq j_1$. We have different cases to consider. Because ${\mathcal O}$ has no overlapping ${\mathcal I}$-box supports, then either $j_1 > 0$ or $0 > i_1$.

If $i_1 = j_1 = i$, then we get easily that $i = \max(Supp(b)) = \min(Supp(b))$, $h_i(b) = 0$ and $(E_{\mathcal O}F_{\mathcal O} - F_{\mathcal O}E_{\mathcal O})(v_i(b)) = 0 = H_{\mathcal O}(v_i(b))$. 

If $i_1 > j_1$ and $i = i_1 = \max(Supp(b))$, then $f_i(b) = h_i(b)$ by lemma~\ref{L:RelationCoefficientsHandF} (f) and 
\[
\begin{aligned}
(E_{\mathcal O} F_{\mathcal O} - F_{\mathcal O}E_{\mathcal O})(v_i(b)) &= \begin{cases} f_i(b) v_i(b),&\text{if $j_1 > 0$;}\\ f_i(b) v_i(b),&\text{if $i_1 < 0$;}\end{cases}\\
&= h_i(b) v_i(b) = H_{\mathcal O}(v_i(b)).
\end{aligned}
\]

If $i_1 > j_1$ and $i = j_1 = \min(Supp(b))$, then $f_{i + 2}(b) = -h_i(b)$ by lemma~\ref{L:RelationCoefficientsHandF} (e) and 
\[
\begin{aligned}
(E_{\mathcal O} F_{\mathcal O} - F_{\mathcal O}E_{\mathcal O})(v_i(b)) &= \begin{cases} -f_{i + 2}(b) v_i(b),&\text{if $j_1 > 0$;}\\ -f_{i + 2}(b) v_i(b),&\text{if $i_1 < 0$;}\end{cases}\\
&= h_i(b) v_i(b) = H_{\mathcal O}(v_i(b)).
\end{aligned}
\]

If $i_1 > i > j_1$, then both $(i - 2), (i + 2) \in Supp(b)$. Thus they are both either simultanously $> 0$ or simultanously $< 0$. So $f_i(b) - f_{i + 2}(b) = h_i(b)$ by lemma~\ref{L:RelationCoefficientsHandF} (g) and 
\[
\begin{aligned}
(E_{\mathcal O} F_{\mathcal O} - F_{\mathcal O}E_{\mathcal O})(v_i(b)) &= \begin{cases} (f_i(b) -f_{i + 2}(b)) v_i(b),&\text{if $j_1 > 0$;}\\ (f_i(b) - f_{i + 2}(b)) v_i(b),&\text{if $i_1 < 0$;}\end{cases}\\
&= h_i(b) v_i(b) = H_{\mathcal O}(v_i(b)).
\end{aligned}
\]

If the orbit ${\mathcal O}$ has overlapping ${\mathcal I}$-box supports, then we have already expressed above our expressions for $E_{\mathcal O}$, $H_{\mathcal O}$ and $F_{\mathcal O}$ on ${\mathcal B}_{\mathcal O}$. Let $b = b(i_1, j_1, k_1)$ be an ${\mathcal I}$-box in ${\mathcal O}$ and $i \in Supp(b)$. We have that $[H_{\mathcal O}, E_{\mathcal O}](v_i(b)) = (H_{\mathcal O} E_{\mathcal O} - E_{\mathcal O} H_{\mathcal O})(v_i(b))$ is equal to
\[
\begin{aligned}
&\begin{cases}
0, &\text{if $i = \max(Supp(b))$;}\\ (h_{i + 2}(b) - h_i(b)) v_{i + 2}(b), &\text{if $i \ne \max(Supp(b))$ and $i > 0$;}\\ -(h_{i + 2}(b) - h_i(b)) v_{i + 2}(b), &\text{if $i  < 0$;}\\ \epsilon(h_{2}(b) - h_0(b)) v_{2}(b), &\text{if $i = 0$ and $b$ is below the principal diagonal;}\\(h_{2}(b) - h_0(b)) v_{2}(b), &\text{if $i = 0$ and $b$ is above the principal diagonal;}\\ \epsilon^{k_1}(h_{2}(b) - h_0(b)) v_{2}(b), &\text{if $i = 0$ and $b$ is on  the principal diagonal;;}
\end{cases}\\
&= 2 E_{\mathcal O}(v_i(b))
\end{aligned}
\]
by lemma~\ref{L:RelationCoefficientsHandF} (c). So $[H_{\mathcal O}, E_{\mathcal O}] = 2E_{\mathcal O}$.

Similarly we have $[H_{\mathcal O}, F_{\mathcal O}](v_i(b)) = (H_{\mathcal O} F_{\mathcal O} - F_{\mathcal O}H_{\mathcal O})(v_i(b))$ is equal to  
\[
\begin{aligned}
&\begin{cases}
0, &\text{if $i = \min(Supp(b))$;}\\ (h_{i - 2}(b) - h_i(b)) f_i(b)v_{i - 2}(b), &\text{if $i > 2$;}\\ -(h_{i - 2}(b) - h_i(b))f_i(b) v_{i - 2}(b), &\text{if $i  \ne \min(Supp(b))$ and $i \leq 0$;}\\ \epsilon(h_{0}(b) - h_2(b)) f_2(b) v_{0}(b), &\text{if $i = 2$ and $b$ is below the principal diagonal;}\\(h_{0}(b) - h_2(b)) v_{0}(b), &\text{if $i = 2$ and $b$ is above the principal diagonal;}\\ \epsilon^{k_1}(h_{0}(b) - h_2(b)) f_2(b)v_{0}(b), &\text{if $i = 2$ and $b$ is on  the principal diagonal;;}
\end{cases}\\
&= -2 F_{\mathcal O}(v_i(b))
\end{aligned}
\]
by lemma~\ref{L:RelationCoefficientsHandF} (d). So $[H_{\mathcal O}, F_{\mathcal O}] = -2F_{\mathcal O}$.

Finally we want to compute $[E_{\mathcal O}, F_{\mathcal O}](v_i(b)) = (E_{\mathcal O} F_{\mathcal O} - F_{\mathcal O}E_{\mathcal O})(v_i(b))$. We have $i_1 \geq i \geq j_1$. Because ${\mathcal O}$ has overlapping ${\mathcal I}$-support, then we have $i_1 \geq 0 \geq j_1$. We have different cases to consider.

If $i_1 = j_1 = i$, we have $i = \max(Supp(b)) = \min(Supp(b)) = 0$, $h_i(b) = 0$ and 
\[
(E_{\mathcal O} F_{\mathcal O} - F_{\mathcal O} E_{\mathcal O})(v_i(b)) = 0 = H_{\mathcal O}(v_i(b)).
\]

We now consider $i_1 > j_1$.  There are three cases to study: $i_1 = i$, $j_1 = i$ and $i_1 > i > j_1$.  If $i_1 = i > j_1$, then $i = \max(Supp(b))$ and $(E_{\mathcal O} F_{\mathcal O} - F_{\mathcal O} E_{\mathcal O})(v_i(b)) = E_{\mathcal O} F_{\mathcal O}(v_i(b))$ is equal to 
\[
\begin{aligned}
&\begin{cases} f_i(b) v_i(b), &\text{if $i > 2$;}\\ \epsilon^2 f_2(b)v_2(b), &\text{if $i = 2$ and $b$ is below the principal diagonal;}\\ f_2(b)v_2(b), &\text{if $i = 2$ and $b$ is above the principal diagonal;}\\ (\epsilon)^{2k_1} f_2(b)v_2(b), &\text{if $i = 2$ and $b$ is on the principal diagonal;}\\ f_0(b) v_0(b), &\text{if $i = 0$;}
\end{cases}\\
&= h_i(b) v_i(b) = H_{\mathcal O}(v_i(b))
\end{aligned}
\]
by lemma~\ref{L:RelationCoefficientsHandF} (f).

If $i_1 > i = j_1$, then $i = \min(Supp(b))$ and $(E_{\mathcal O} F_{\mathcal O} - F_{\mathcal O} E_{\mathcal O})(v_i(b)) = -F_{\mathcal O} E_{\mathcal O}(v_i(b))$ is equal to 
\[
\begin{aligned}
&\begin{cases} - \epsilon^2 f_2(b) v_0(b), &\text{if $i = 0$ and $b$ is below the principal diagonal;}\\  - f_2(b) v_0(b), &\text{if $i = 0$ and $b$ is above the principal diagonal;}\\ - (\epsilon)^{2k_1} f_2(b) v_0(b), &\text{if $i = 0$ and $b$ is on the principal diagonal;}\\ - f_{i + 2}(b)v_i(b), &\text{if $i < 0$;}
\end{cases}\\
&= h_i(b) v_i(b) = H_{\mathcal O}(v_i(b))
\end{aligned}
\]
by lemma~\ref{L:RelationCoefficientsHandF} (e).

Finally if $i_1 > i > j_1$, then $(i + 2), (i - 2) \in Supp(b)$ and $(E_{\mathcal O} F_{\mathcal O} - F_{\mathcal O} E_{\mathcal O})(v_i(b))$ is equal to 
\[
\begin{aligned}
&\begin{cases}
(f_i(b) - f_{i + 2}(b))v_i(b), &\text{if $i > 2$;}\\
(\epsilon^2 f_2(b) - f_4(b)) v_2(b), &\text{$i = 2$ and $b$ is below the principal diagonal;}\\
(f_2(b) - f_4(b)) v_2(b), &\text{$i = 2$ and $b$ is above the principal diagonal;}\\
((\epsilon)^{2k_1} f_2(b) - f_4(b)) v_2(b), &\text{$i = 2$ and $b$ is on the principal diagonal;}\\
(f_0(b) - \epsilon^2 f_2(b)) v_0(b), &\text{$i = 0$ and $b$ is below the principal diagonal;}\\
(f_0(b) - f_2(b)) v_0(b), &\text{$i = 0$ and $b$ is above the principal diagonal;}\\
(f_0(b) - (\epsilon)^{2k_1}f_2(b)) v_0(b), &\text{$i = 0$ and $b$ is on the principal diagonal;}\\
(f_i(b) - f_{i + 2}(b))v_i(b), &\text{if $i \leq -2$;}
\end{cases}\\
&= h_i(b) v_i(b) = H_{\mathcal O}(v_i(b))
\end{aligned}
\]
by lemma~\ref{L:RelationCoefficientsHandF} (g).

Thus we have proved that the triple $E_{\mathcal O}$, $H_{\mathcal O}$ and $F_{\mathcal O}$ is an ${\mathcal I}$-graded  Jacobson-Morozov triple when $\vert {\mathcal I} \vert$ is odd.

We now want to prove that 
\[
\{(V({\mathcal O}), \phi_{\mathcal O}, \langle\ , \ \rangle_{\mathcal O}) \mid \text{ ${\mathcal O}$
belongs to the set of $\langle \tau \rangle$-orbits  in ${\mathbf B}$}\}
\]
is a set of representatives of the isomorphism classes of indecomposable $\epsilon$-represen~- tations of $(\Omega, \sigma)$ when $\vert {\mathcal I} \vert$ is odd.

Let the ${\mathcal I}$-box $b = b(i_1, j_1, k_1) \in {\mathbf B}$ be an element of the $\langle \tau \rangle$-orbit ${\mathcal O}$ with $i_1, j_1 \in {\mathcal I}$, $i_1 \geq j_1$ and $0 \leq k_1 \leq \mu(i_1, j_1)$. If we consider the restriction of $E_{\mathcal O}$ to $V_i(b)$: $\phi_{b, i} = E_{\mathcal O}\vert_{V_i(b)}:V_i(b) \rightarrow V_{i + 2}(b)$ for each $i \in Supp(b)$, $i \ne \max(Supp(b))$, we get a representation $\phi_b:V(b) \rightarrow V(b)$ of the quiver $\Omega$ of type $A_{2m-1}$. Here $\vert {\mathcal I} \vert = (2m - 1)$. From our definition of $E_{\mathcal O}$, we get easily that 
\begin{itemize}
\item $\phi_b$ is an indecomposable representation of $A_{2m-1}$ of dimension $\dim(\phi_b) = (\beta_i)_{i \in {\mathcal I}}$ where 
\[
\beta_i = \begin{cases} 1, &\text{if $i_1 \geq i \geq j_1$;}\\ 0, &\text{otherwise.}\end{cases}
\]
\item $\phi_{\mathcal O}:V({\mathcal O}) \rightarrow V({\mathcal O})$ is the direct sum $\oplus_{b \in {\mathcal O}} \phi_b$ as a representation of $A_{2m-1}$.
\end{itemize}
Because $\langle E_{\mathcal O}(u), v\rangle_{\mathcal O} + \langle u, E_{\mathcal O}(v)\rangle_{\mathcal O} = 0$ for all $u, v \in V({\mathcal O})$, then $(V({\mathcal O}), \phi_{\mathcal O}, \langle\ , \ \rangle_{\mathcal O})$ is an $\epsilon$-representation. To show that $(V({\mathcal O}), \phi_{\mathcal O}, \langle \  , \  \rangle_{\mathcal O})$ is  indecomposable, we can use theorem~\ref{T:CaracterisationRepresentation}.

Each $\langle \tau \rangle$-orbit ${\mathcal O}$ contains at least one ${\mathcal I}$-box $b = b(i_1, j_1, k_1) \in {\mathbf B}$ such that $(i_1 + j_1) \geq 0$. If $(i_1 + j_1) > 0$, this ${\mathcal I}$-box is unique in the orbit, while if $(i_1 + j_1) = 0$, then there are $\mu_{\max}$ ${\mathcal I}$-boxes in the orbit. If $(i_1 + j_1) > 0$, then $\tau(b) \ne b$ and ${\mathcal O} = \{b, \tau(b)\}$. We have $i_1 \geq j_1$ and $i_1 \geq \vert j_1 \vert$ because  $(i_1 + j_1) > 0$.  As in the definition~\ref{DefDim}, we will write the dimension vector of $(V_{\mathcal O}, \phi_{\mathcal O}, \langle\  , \  \rangle)$ as representation of the symmetric quiver $(\Omega, \sigma)$ by
\[
\dim(V_{\mathcal O})  = (\dim(V_{{\mathcal O}, i}))_{i \in {\mathcal I}_+} = (\alpha_i)_{i \in {\mathcal I}_+} = \alpha.
\]
We will write also the dimension of $\dim(\phi_b)$ and $\dim(\phi_{\tau(b)})$ as
\[
\dim(\phi_b) = \beta' = (\beta'_i)_{i \in {\mathcal I}} \quad \text{ and } \quad  \dim(\phi_{\tau(b)}) = \beta = (\beta_i)_{i \in {\mathcal I}}.
\]  

If $j_1 > 0$, then 
\[
\beta_i = \begin{cases}1, &\text{if $-j_1 \geq i \geq -i_1$;}\\ 0, &\text{otherwise; } \end{cases} \quad  \quad 
\beta'_i = \begin{cases}1, &\text{if $i_1 \geq i \geq j_1$;}\\ 0, &\text{otherwise;} \end{cases}
\]
for $i \in {\mathcal I}$ and 
\[
\alpha_i =\begin{cases}  1,&\text{if $i_1 \geq i \geq j_1$;}\\ 0, &\text{otherwise;}\end{cases}
\]
for $i \in {\mathcal I}_+$.  Using proposition~\ref{Decomposition_odd} (a) (i) and (b)(i)  and theorem~\ref{T:CaracterisationRepresentation}, this shows that $\phi_{\mathcal O}$ is indecomposable and isomorphic to the indecomposable representation of dimension $\alpha$.

If $0 > j_1$, then 
\[
\beta_i = \begin{cases}1, &\text{if $-j_1 \geq i \geq -i_1$;}\\ 0, &\text{otherwise; } \end{cases} \quad  \quad 
\beta'_i = \begin{cases}1, &\text{if $i_1 \geq i \geq j_1$;}\\ 0, &\text{otherwise;} \end{cases}
\]
for $i \in {\mathcal I}$ and 
\[
\alpha_i =\begin{cases}  2,&\text{if $\vert j_1\vert  \geq i \geq 0$;}\\ 1,&\text{if $i_1 \geq i > \vert j_1 \vert$;}\\ 0, &\text{otherwise;}\end{cases}
\]
for $i \in {\mathcal I}_+$.  Using proposition~\ref{Decomposition_odd} (a) (iii) and (b)(iii)  and theorem~\ref{T:CaracterisationRepresentation}, this shows that $\phi_{\mathcal O}$ is indecomposable and isomorphic to the indecomposable representation of dimension $\alpha$.

If $(i_1 + j_1) = 0$, then the ${\mathcal I}$-box $b(i_1, j_1, k_1)$ is on the principal diagonal and $i_1 \geq 0$, because $i_1 \geq j_1$. There are two cases to consider: either $\tau(b) \ne b$ or $\tau(b) = b$. When $\tau(b) \ne b$, then ${\mathcal O} = \{b, \tau(b)\}$ and by lemma~\ref{L:TauBox} (d) and \ref{SS:Iset}, we have that $\mu_{\max} = 1$, $\epsilon = -1$, $\tau(b) = b(i_1, j_1, (1 - k_1))$. So 
\[
\beta'_i = \beta_i = \begin{cases}1, &\text{if $i_1 \geq i \geq -i_1$;}\\ 0, &\text{otherwise; } \end{cases}
\]
for $i \in {\mathcal I}$ and 
\[
\alpha_i =\begin{cases}  2,&\text{if $i_1 \geq i \geq 0$;}\\  0, &\text{otherwise;}\end{cases}
\]
for $i \in {\mathcal I}_+$. Using proposition~\ref{Decomposition_odd} (b) (ii) and theorem~\ref{T:CaracterisationRepresentation}, this shows that $\phi_{\mathcal O}$ is indecomposable and isomorphic to the indecomposable representation of dimension $\alpha$. 

When $\tau(b) = b$, then ${\mathcal O} = \{b \}$ and by lemma~\ref{L:TauBox} (e), $\mu_{\max} = 0$, $\epsilon = 1$, $V({\mathcal O}) = V_b$. We will write $\dim(\phi_b) = \beta = (\beta_i)_{i \in {\mathcal I}}$. So 
\[
\beta_i = \begin{cases}1, &\text{if $i_1 \geq i \geq -i_1$;}\\ 0, &\text{otherwise; } \end{cases}
\]
for $i \in {\mathcal I}$ and 
\[
\alpha_i =\begin{cases}  1,&\text{if $i_1 \geq i \geq 0$;}\\  0, &\text{otherwise.}\end{cases}
\]
for $i \in {\mathcal I}_+$. Using proposition~\ref{Decomposition_odd} (a) (ii)  and theorem~\ref{T:CaracterisationRepresentation}, this shows that $\phi_{\mathcal O}$ is indecomposable and isomorphic to the indecomposable representation of dimension $\alpha^1$. 

To finish the proof when $\vert {\mathcal I} \vert$ is odd, it suffices to notice by inspection of the dimension vectors of $(V({\mathcal O}), \phi_{\mathcal O}, \langle\  , \  \rangle_{\mathcal O})$,  the direct sum decomposition $\phi_{\mathcal O} = \oplus_{b \in {\mathcal O}} \phi_b$ and theorem~\ref{T:CaracterisationRepresentation}, that $(V({\mathcal O}), \phi_{\mathcal O}, \langle\  , \  \rangle_{\mathcal O})$ and $(V({\mathcal O'}), \phi_{\mathcal O'}, \langle\  , \  \rangle_{\mathcal O'})$ are not isomorphic for two distinct $\langle \tau \rangle$-orbits ${\mathcal O}$ and ${\mathcal O}'$. By remark~\ref{R:CardinalityIndecomposable}, we get the result that $\{(V({\mathcal O}), \phi_{\mathcal O}, \langle\ , \ \rangle_{\mathcal O}) \mid \text{ ${\mathcal O}$  belongs to the set of $\langle \tau \rangle$-orbits  in ${\mathbf B}$}\}$   is a set of representatives of the isomorphism classes of indecomposable $\epsilon$-representations of $(\Omega, \sigma)$.
\end{proof}

We will need the following lemma and notation in statements, formulas and propositions for the next sections.

\begin{lemma}
Let $b = b(i, j, k) \in{\mathbf B}$, where $i, j \in {\mathcal I}$, $i \geq j$ and $0 \leq k \leq \mu(i, j)$,  and $i_1, i_2 \in Supp(b) \subseteq {\mathcal I}$. Then $i_1 - h_{i_1}(b) = i_2 - h_{i_2}(b) = (i + j)/2$.
\end{lemma} 
\begin{proof}
We can assume that $i_1 > i_2$. Note that $i_1$ and $i_2$ have the same parity by definition of ${\mathcal I}$. So $i_1 = i_2 + 2s$ where $s \in {\mathbb N}$. We have
\[
\begin{aligned}
i_1 - h_{i_1}(b) &= (i_1) - \#\{j'' \in {\mathcal I} \mid j \leq j'' \leq  i_1\}  + \#\{i'' \in {\mathcal I}  \mid i_1 \leq i'' \leq i \}\\
&= (i_2 + 2s) - ( \#\{j'' \in {\mathcal I} \mid j \leq j'' \leq  i_2 \}  + s) + (\#\{ i'' \in {\mathcal I}  \mid i_2 \leq i'' \leq i\} - s)\\ &= i_2 - h_{i_2}(b).
\end{aligned}
\]
If we now consider $i_1 = i$, then 
\[
i - h_i(b) = i - \#\{j'' \in {\mathcal I} \mid j \leq j'' \leq  i\}  + \#\{i'' \in {\mathcal I}  \mid i \leq i'' \leq i \} = \frac{(i + j)}{2}
\]
\end{proof}

\begin{notation}\label{N:LambdaValue}
We will denote the above constant value by $\lambda(b)$.  So for $b = b(i, j, k)$ as above, we have $\lambda(b) = (i + j)/2$. 
\end{notation}

\subsection{} In the case where $\vert {\mathcal I} \vert$ is odd and $\epsilon = 1$, more precisely in the orthogonal case for the symmetric quiver $A_m^{odd}$, we have associated to each $\langle \tau \rangle$-orbit ${\mathcal O}$: an orthogonal isomorphism class of indecomposable orthogonal representations $( V({\mathcal O}), \phi_{\mathcal O}, \langle\  , \  \rangle_{\mathcal O})$. Two such indecomposable orthogonal representations in the same orthogonal isomorphism class of representations are not necessarily isomorphic relative to the special orthogonal group.  With the notation of lemma~\ref{L:DetThetaAOdd}, we have $\det(\theta_0) = \pm 1$, but not necessarily $\det(\theta_0) = 1$.  

To conclude this section, we will now discuss  how an orthogonal isomorphism class of indecomposable orthogonal representations splits into isomorphism classes relative to the special orthogonal group.  We will also give an ${\mathcal I}$-graded  Jacobson-Morozov triple for each isomorphism class relative to the special orthogonal group.

\subsection {}\label{SS:SpecialOrthoRepresentationsInSameOrthoClass}
Let $(\Omega, \sigma)$ be the symmetric quiver $A_m^{odd}$ and $\epsilon = 1$. Recall that 
\[
{\mathcal I} = \{i \in {\mathbb Z} \mid i \equiv 0 \pmod 2 \text{ and } -2m <  i  < 2m\}.
\]
Let $b = b(i_1, j_1, 0)$ be an ${\mathcal I}$-box such that $i_1, j_1 \in {\mathcal I}$, $i_1 \geq 0 \geq j_1$ and $i_1 + j_1 \ne 0$. Then we have easily  that  
\begin{itemize}
\item $\tau(b) = b(-j_1, -i_1, 0) \ne b$;
\item the orbit ${\mathcal O}$ of $b$ is $\{b, \tau(b)\}$ and has overlapping ${\mathcal I}$-box supports;
\item $Supp(b) = \{ i \in {\mathcal I} \mid i_1 \geq i \geq j_1\}$ and $Supp(\tau(b)) = \{ i \in {\mathcal I} \mid -j_1 \geq i \geq -i_1\}$;
\end{itemize} 
In the following, we will assume that $b$ is such that $i_1 + j_1 > 0$. Otherwise we just need to interchange $b$ and $\tau(b)$. By this hypothesis, we get that $i_1 > \vert j_1 \vert \geq j_1$. 

In this case, 
\[
V({\mathcal O}) = \oplus_{i \in {\mathcal I}} V_i({\mathcal O}), \quad \text{where} \quad \dim(V_i({\mathcal O})) = \begin{cases}2, &\text{if $j_1 \leq \vert i \vert  \leq \vert j_1 \vert$;}\\ 1, &\text{if $\vert j_1\vert  < \vert i \vert \leq i_1$;}\\ 0, &\text{otherwise.}
\end{cases}
\]
We have constructed a basis ${\mathcal B}_{\mathcal O}$ for $V({\mathcal O})$ in \ref{N:BasisVOrbit}.  From this, we get that $V_i({\mathcal O})$ has basis
\[
\begin{cases} \{v_i(b), v_i(\tau(b))\}, &\text{if $j_1 \leq i \leq \vert j_1\vert$;} \\ \{v_i(b) \} &\text{if $\vert j_1\vert  < i \leq i_1$;}\\  \{v_i(\tau(b)) \} &\text{if $-i _1  \leq i < j_1$;} \end{cases}
\]
and we have also expressed the values of the bilinear form $\langle \ , \ \rangle_{\mathcal O}$  and  the action of $E_{\mathcal O}$ on these basis vectors.

We have constructed the ${\mathcal I}$-graded Jacobson-Morozov triple $(E_{\mathcal O}, H_{\mathcal O}, F_{\mathcal O})$ and an orthogonal representation $(V({\mathcal O}), \phi_{\mathcal O}, \langle\  , \  \rangle_{\mathcal O})$ in \ref{SS:EHFDefinition} and in the proposition~\ref{P:JMTripleProof}.

Consider now the  linear transformation $S:V({\mathcal O}) \rightarrow V({\mathcal O})$ defined on the basis ${\mathcal B}_{\mathcal O}$ as follows: 
\[
S(u) = \begin{cases} u, &\text{if $u \in V_i({\mathcal O})$ with $i \ne 0$;}\\ v_0(\tau(b)), &\text{if $i = 0$ and $u = v_0(b)$;}\\ v_0(b), &\text{if $i = 0$ and $u = v_0(\tau(b))$.} \end{cases}
\]
It is obviously a ${\mathcal I}$-graded  linear transformation of degree $0$ and it is an involution.

From our definition of $\langle\  , \  \rangle_{\mathcal O}$, we get easily  that $\langle S(u), S(v) \rangle_{\mathcal O} = \langle u, v \rangle_{\mathcal O}$ for all $u, v \in V({\mathcal O})$. Set $E'_{\mathcal O} = S E_{\mathcal O} S$, $H'_{\mathcal O} = S H_{\mathcal O} S$ and $F'_{\mathcal O} = S F_{\mathcal O} S$. We also get easily that 
\begin{itemize}
\item $E'_{\mathcal O}$ is of degree 2, $H'_{\mathcal O}$ is of degree 0 and $F'_{\mathcal O}$ is of degree -2;
\item $\langle Xu, v\rangle + \langle u, Xv\rangle = 0$ for all $X \in \{E'_{\mathcal O}, H'_{\mathcal O}, F'_{\mathcal O}\}$ and all $u, v \in V({\mathcal O})$;
\item $[H'_{\mathcal O}, E'_{\mathcal O}] = 2E'_{\mathcal O}$, $[H'_{\mathcal O}, F'_{\mathcal O}] = -2F'_{\mathcal O}$ and $[E'_{\mathcal O}, F'_{\mathcal O}] = H'_{\mathcal O}$, in other words,  $(E'_{\mathcal O}, H'_{\mathcal O}, F'_{\mathcal O})$ is an ${\mathcal I}$-graded  Jacobson-Morozov triple;
\item  the ${\mathcal I}$-graded linear transformation $\phi'_{\mathcal O}: V({\mathcal O}) \rightarrow V({\mathcal O})$ given by 
\[
\phi'_{{\mathcal O}, i} = {E'_{\mathcal O}}\vert_{V_i({\mathcal O})}: V_{i}({\mathcal O}) \rightarrow V_{i + 2}({\mathcal O}) \quad \text{for all $i \in {\mathcal I}$  such that $(i + 2) \in {\mathcal I}$},
\]
is  such that $(V({\mathcal O}), \phi'_{\mathcal O}, \langle\ , \ \rangle_{\mathcal O})$ is an  orthogonal representation.
\end{itemize}

From our construction, we get that $(V({\mathcal O}), \phi_{\mathcal O}, \langle\  , \  \rangle_{\mathcal O})$ and $(V({\mathcal O}), \phi'_{\mathcal O}, \langle\  , \  \rangle_{\mathcal O})$ are isomorphic relative to the orthogonal group with the isomorphism  given by 
\[
S:(V({\mathcal O}), \phi_{\mathcal O}, \langle\  , \  \rangle_{\mathcal O}) \rightarrow (V({\mathcal O}), \phi'_{\mathcal O}, \langle\  , \  \rangle_{\mathcal O}).
\]

We will now show that these two orthogonal representations  $(V_{\mathcal O}, \phi_{\mathcal O}, \langle\  , \  \rangle_{\mathcal O})$ and $(V_{\mathcal O}, \phi'_{\mathcal O}, \langle\  , \  \rangle_{\mathcal O})$ are not isomorphic relatively to the special orthogonal group.

Assume that $\theta:(V({\mathcal O}), \phi_{\mathcal O}, \langle\ , \ \rangle_{\mathcal O}) \rightarrow (V({\mathcal O}), \phi_{\mathcal O}, \langle\ , \ \rangle_{\mathcal O})$, $\theta':(V({\mathcal O}), \phi'_{\mathcal O}, \langle\ , \ \rangle_{\mathcal O}) \rightarrow (V({\mathcal O}), \phi'_{\mathcal O}, \langle\ , \ \rangle_{\mathcal O})$ and $\theta'':(V({\mathcal O}), \phi_{\mathcal O}, \langle\ , \ \rangle_{\mathcal O}) \rightarrow (V({\mathcal O}), \phi'_{\mathcal O}, \langle\ , \ \rangle_{\mathcal O})$ are orthogonal isomorphisms for these orthogonal representations. In other words, there are isomorphisms $\theta_i:V_i({\mathcal O}) \rightarrow V_i({\mathcal O})$ (respectively  $\theta'_i:V_i({\mathcal O}) \rightarrow V_i({\mathcal O})$,  $\theta''_i:V_i({\mathcal O}) \rightarrow V_i({\mathcal O})$) for each $i \in {\mathcal I}$, $\vert i \vert \leq i_1$ such that 
\begin{itemize}
\item $\langle \theta_i(u), \theta_{-i}(v)\rangle_{\mathcal O} = \langle u, v\rangle_{\mathcal O}$  (resp.  $\langle \theta'_i(u), \theta'_{-i}(v)\rangle_{\mathcal O} = \langle u, v\rangle_{\mathcal O}$ , $\langle \theta''_i(u), \theta''_{-i}(v)\rangle_{\mathcal O} = \langle u, v\rangle_{\mathcal O}$ ) for all $u \in V_i({\mathcal O})$ and $v \in V_{-i}({\mathcal O})$;
\item $\theta_{i + 2}\circ  \phi_{{\mathcal O}, i} = \phi_{{\mathcal O}, i} \circ \theta_i$ (respectively $\theta'_{i + 2}\circ  \phi'_{{\mathcal O}, i} = \phi'_{{\mathcal O}, i} \circ \theta'_i$ , $\theta''_{i + 2}\circ  \phi_{{\mathcal O}, i} = \phi'_{{\mathcal O}, i} \circ \theta''_i$) 
for all $i \in {\mathcal I}$ and $\vert i \vert < i_1$;
\item $\det(\theta_i) \ne 0$ and $\det(\theta_{-i}) = (\det(\theta_i))^{-1}$  (respectively  $\det(\theta'_i) \ne 0$ and $\det(\theta'_{-i}) = (\det(\theta'_i))^{-1}$ ,  $\det(\theta''_i) \ne 0$ and $\det(\theta''_{-i}) = (\det(\theta''_i))^{-1}$ ) for all $i \in {\mathcal I}$, $\vert i \vert \leq i_1$.
\end{itemize}

\begin{lemma}\label{L:DeterminantTheta0} With the above notation, we have that
\begin{enumerate}[\upshape (a)]
\item $\det(\theta_0) = 1$;
\item $\det(\theta'_0) = 1$;
\item $\det(\theta''_0) = -1$.
\end{enumerate}
\end{lemma}
\begin{proof}
From our hypothesis, we have $i_1 > \vert j_1 \vert$.  $V_0({\mathcal O})$ has $\{v_0(b), v_0(\tau(b))\}$ and we know how the elements $E_{\mathcal O}$ and $E'_{\mathcal O}$ act on these vectors.  Write 
\[
\left\{
\begin{aligned}
\theta_0(v_0(b)) &= \alpha v_0(b) + \gamma v_0(\tau(b))\\ \theta_0(v_0(\tau(b))) &= \beta v_0(b) + \xi v_0(\tau(b))
\end{aligned}
\right.
\]
\[
\left\{
\begin{aligned}
\theta'_0(v_0(b)) &= \alpha' v_0(b) + \gamma' v_0(\tau(b))\\ \theta'_0(v_0(\tau(b))) &= \beta' v_0(b) + \xi' v_0(\tau(b))
\end{aligned}
\right. 
\] 
and 
\[
\left\{
\begin{aligned}
\theta''_0(v_0(b)) &= \alpha'' v_0(b) + \gamma'' v_0(\tau(b))\\ \theta''_0(v_0(\tau(b))) &= \beta'' v_0(b) + \xi'' v_0(\tau(b))
\end{aligned}
\right. 
\] 
where $\alpha, \beta, \gamma, \xi, \alpha', \beta', \gamma', \xi', \alpha'', \beta'', \gamma'', \xi'' \in {\mathbf k}$.  
Because 
\[
\langle \theta_0(u), \theta_{0}(v)\rangle_{\mathcal O} = \langle u, v\rangle_{\mathcal O}, \     \langle \theta'_0(u), \theta'_{0}(v)\rangle_{\mathcal O}  = \langle u, v\rangle_{\mathcal O} \text{ and  }  \langle \theta''_0(u), \theta''_{0}(v)\rangle_{\mathcal O}  = \langle u, v\rangle_{\mathcal O} 
\]
for all $u, v \in V_0({\mathcal O})$, we get by considering the following  pairs: $(u, v) = (v_0(b), v_0(b))$, $(v_0(\tau(b)), v_0(\tau(b)))$ and $(v_0(b), v_0(\tau(b)))$, we will get easily that 
\begin{itemize}
\item $2\alpha \gamma = 0$, \quad  $2\alpha' \gamma' = 0$, \quad  $2\alpha'' \gamma'' = 0$; 
\item $2\beta \xi = 0$,\quad  $2\beta' \xi' = 0$, \quad $2\beta'' \xi'' = 0$; 
\item $\alpha\xi + \gamma \beta = 1$,\quad  $\alpha'\xi' + \gamma' \beta' = 1$ \quad and \quad $\alpha'' \xi'' + \gamma'' \beta'' = 1$.
\end{itemize}

Note that $V_{(\vert j_1 \vert +2)}({\mathcal O})$ has dimension $1$ and $\{v_{(\vert j_1 \vert + 2)}(b)\}$ is a basis of  $V_{(\vert j_1 \vert +2)}({\mathcal O})$. Also we have that  $\theta_{(\vert j_1 \vert + 2)}:V_{(\vert j_1 \vert +2)}({\mathcal O}) \rightarrow V_{(\vert j_1 \vert +2)}({\mathcal O})$,  $\theta'_{(\vert j_1 \vert + 2)}:V_{(\vert j_1 \vert +2)}({\mathcal O}) \rightarrow V_{(\vert j_1 \vert +2)}({\mathcal O})$ and $\theta''_{(\vert j_1 \vert + 2)}:V_{(\vert j_1 \vert +2)}({\mathcal O}) \rightarrow V_{(\vert j_1 \vert +2)}({\mathcal O})$ are isomorphisms. Thus $\theta(v_{(\vert j_1 \vert + 2)}(b)) = \chi v_{(\vert j_1 \vert + 2)}(b)$, $\theta'(v_{(\vert j_1 \vert + 2)}) = \chi' v_{(\vert j_1 \vert + 2)}(b)$ and $\theta''(v_{(\vert j_1 \vert + 2)}(b)) = \chi'' v_{(\vert j_1 \vert + 2)}(b)$ with $\chi, \chi', \chi'' \in {\mathbf k}$, $\chi \ne 0$, $\chi' \ne 0$ and $\chi'' \ne 0$.

(a) We have 
\[
\phi_{{\mathcal O}, \vert j_1 \vert} \circ \dots \circ \phi_{{\mathcal O}, 2} \circ \phi_{{\mathcal O}, 0} \circ \theta_0 = \theta_{(\vert j_1 \vert + 2)} \circ \phi_{{\mathcal O}, \vert j_1 \vert} \circ \dots \circ \phi_{{\mathcal O}, 2} \circ \phi_{{\mathcal O}, 0}
\]
If we evaluate this expression on each of the basis vectors of $V_0({\mathcal O})$, we get
\[
\begin{cases} v_0(b) \mapsto \alpha v_{(\vert j_1\vert + 2)}(b)\\ v_0(\tau(b)) \mapsto \beta v_{(\vert j_1 \vert + 2)}(b)\end{cases} \quad \text{ for the left side map }
\]
and
\[
\begin{cases} v_0(b) \mapsto \chi v_{(\vert j_1\vert + 2)}(b)\\ v_0(\tau(b)) \mapsto 0\end{cases} \quad  \text{ for the right side map.}
\]
Thus $\beta = 0$ and we get that $\alpha \xi = 1$ and $\det(\theta_0) = 1$.

(b) We have 
\[
\phi'_{{\mathcal O}, \vert j_1 \vert} \circ \dots \circ \phi'_{{\mathcal O}, 2} \circ \phi'_{{\mathcal O}, 0} \circ \theta'_0 = \theta'_{(\vert j_1 \vert + 2)} \circ \phi'_{{\mathcal O}, \vert j_1 \vert} \circ \dots \circ \phi'_{{\mathcal O}, 2} \circ \phi'_{{\mathcal O}, 0}
\]
If we evaluate this expression on each of the basis vectors of $V_0({\mathcal O})$, we get
\[
\begin{cases} v_0(b) \mapsto \gamma' v_{(\vert j_1\vert + 2)}(b)\\ v_0(\tau(b)) \mapsto \xi' v_{(\vert j_1 \vert + 2)}(b)\end{cases} \quad \text{ for the left side map }
\]
and
\[
\begin{cases} v_0(b) \mapsto 0\\ v_0(\tau(b)) \mapsto \chi' v_{(\vert j_1\vert + 2)}(b)\end{cases} \quad  \text{ for the right side map.}
\]
Thus $\gamma' = 0$ and we get that $\alpha' \xi' = 1$ and $\det(\theta'_0) = 1$.

(c) We have 
\[
\phi'_{{\mathcal O}, \vert j_1 \vert} \circ \dots \circ \phi'_{{\mathcal O}, 2} \circ \phi'_{{\mathcal O}, 0} \circ \theta''_0 = \theta''_{(\vert j_1 \vert + 2)} \circ \phi_{{\mathcal O}, \vert j_1 \vert} \circ \dots \circ \phi_{{\mathcal O}, 2} \circ \phi_{{\mathcal O}, 0}
\]
If we evaluate this expression on each of the basis vectors of $V_0({\mathcal O})$, we get
\[
\begin{cases} v_0(b) \mapsto \gamma'' v_{(\vert j_1\vert + 2)}(b)\\ v_0(\tau(b)) \mapsto \xi'' v_{(\vert j_1 \vert + 2)}(b)\end{cases} \quad \text{ for the left side map }
\]
and
\[
\begin{cases} v_0(b) \mapsto \chi'' v_{(\vert j_1\vert + 2)}(b)\\ v_0(\tau(b)) \mapsto 0\end{cases} \quad  \text{ for the right side map.}
\]
Thus $\xi'' = 0$ and we get that $\gamma'' \beta'' = 1$ and $\det(\theta''_0) = -1$.
\end{proof}

\begin{remark}\label{SS:ExampleIsomorphismSpecialOrtho}
So $(V({\mathcal O}), \phi_{\mathcal O}, \langle\  , \  \rangle_{\mathcal O})$ and $(V({\mathcal O}), \phi'_{\mathcal O}, \langle\  , \  \rangle_{\mathcal O})$ are isomorphic relative to the orthogonal group, but not isomorphic relative to the special orthogonal group. Because the index of the special orthogonal subgroup in the orthogonal group is 2, we get that any representation $(V_{\mathcal O}, \phi''_{\mathcal O}, \langle\  ,  \  \rangle_{\mathcal O})$ isomorphic relative to the orthogonal group to either isomorphic to $(V({\mathcal O}), \phi_{\mathcal O}, \langle\  , \  \rangle_{\mathcal O})$ or $(V({\mathcal O}), \phi'_{\mathcal O}, \langle\  , \  \rangle_{\mathcal O})$ relative to the special orthogonal group. In other words, the orthogonal isomorphism class of $(V({\mathcal O}), \phi_{\mathcal O}, \langle\  , \  \rangle_{\mathcal O})$ splits into two isomorphism classes relatively to the special orthogonal group.

Note it is easy to see that if ${\mathcal O}$ is the $\langle \tau \rangle$-orbit of the ${\mathcal I}$-box $b = b(i_1, j_1, 0) \in {\mathbf B}$ such that either $i_1 + j_1 = 0$ or $i_1 \geq j_1 > 0$ or $0 > i_1 \geq j_1$, then the orthogonal isomorphism class of $(V({\mathcal O}), \phi_{\mathcal O}, \langle\  , \  \rangle_{\mathcal O})$ does not split and is a unique special orthogonal group isomorphism class.  
\end{remark}

\section{$G^{\iota}$-orbits in ${\mathfrak g}_2$ when $\vert {\mathcal I}\vert$ is even.}

\subsection{}\label{S:SetUpAeven}
In this section,  
\begin{itemize}
\item $m$ is an integer $> 0$;
\item $G$ is the connected component of the identity element of the group of automorphisms of a finite dimensional vector space $V$ over ${\mathbf k}$ preserving a fixed non-degenerate symmetric (respectively skew-symmetric) bilinear form $\langle\ , \  \rangle: V \times V \rightarrow {\mathbf k}$, in other words  $G = SO(V)$ (respectively $G = Sp(V)$); 
\item $\mathfrak g$ is the Lie algebra of $G$, more precisely $X \in {\mathfrak g}$ if and only if $X:V \rightarrow V$ is an endomorphism of $V$ such that $\langle X(v_1), v_2\rangle + \langle v_1 , X(v_2) \rangle = 0$ for all $v_1, v_2 \in V$;
\item $V$ is identified to its dual $V^*$ by the isomorphism $\phi: V \rightarrow V^*$ defined by $v \mapsto \phi_v$, where $\phi_v: V \rightarrow {\mathbf k}$ is $\phi_v(x) = \langle v, x\rangle$ for all $x \in V$;  
\item ${\mathcal I} = \{n \in {\mathbb Z} \mid n \equiv 1 \pmod 2, -2m <  n < 2m \}$;
\item $\oplus_{i \in {\mathcal I}} V_i$ is a direct sum decomposition of $V$ by  vector subspaces $V_i \ne 0$ for all $i \in {\mathcal I}$ such that 
\begin{itemize} \item $\langle v, v'\rangle = 0$ whenever $v \in V_i$, $v' \in V_j$ and $i + j \ne 0$;
			\item $V_{-i} = V_i^*$ for all $i \in {\mathcal I}$ under the identification given by $\phi$ above;
\end{itemize} 
\item ${\mathcal B} = \coprod_{i\in {\mathcal I}} {\mathcal B}_i$ is a basis of $V$ such that each ${\mathcal B}_i = \{u_{i, j} \mid 1 \leq j \leq \delta_i \}$  is a basis of $V_i$ for each $i \in {\mathcal I}$, where $\delta_i$ is the dimension of $V_i$. Moreover these bases are such that, for $i, j \in {\mathcal I}$, $1 \leq r \leq \delta_i$, $1 \leq s \leq \delta_j$, we have
\[
\langle u_{i, r}, u_{j, s}\rangle = \begin{cases} 1, &\text{if $i + j = 0$, $r = s$ and $i > 0$;}\\ \epsilon, &\text{if $i + j = 0$, $r = s$ and $i < 0$;}\\ 0, &\text{otherwise.} \end{cases}
\]
\item $\iota: {\mathbf k}^{\times} \rightarrow G$ is the homomorphism defined by $\iota(t) v = t^i v$ for all $i \in {\mathcal I}$, $v \in V_i$ and $t \in {\mathbf k}^{\times}$;
\item ${\mathbf B}$ is the set of ${\mathcal I}$-boxes;
\item ${\mathbf B}/\tau$ denote the set of $\langle \tau \rangle$-orbits ${\mathcal O}$ in ${\mathbf B}$; 
\item $\nu = \vert {\mathbf B}/\tau \vert = m(m + 1)$;
\end{itemize}

\begin{lemma}\label{G2AsRepresEven}
\begin{enumerate}[\upshape (a)]
\item $g \in G^{\iota}$ if and only if all the following conditions are verified:
\begin{itemize}
\item $g(V_i) \subseteq V_i$   for all $i  \in {\mathcal I}$;
\item the restriction $g_{\vert_{V_i}} = g_i$ of $g$ to $V_i$  belongs to $GL(V_i)$ for all $i \in {\mathcal I}$;
\item $g_{-i}^* g_{i} = Id_{V_i}$ for all $i \in {\mathcal I}$.
\end{itemize}
Here $g_{-i}^*: V_{i} \rightarrow V_{i}$ is defined using the identification of $V_{i}$ with $V_{-i}^*$ given by $\phi$.  So $\langle g_{-i}^* (u), v \rangle = \langle u, g_{-i}(v) \rangle$ for all $u \in V_{i}$ and $v \in V_{-i}$.

\item $X \in {\mathfrak g}_2$  if and only if all the following conditions are verified:
\begin{itemize}
\item $X(V_i) \subseteq V_{i + 2}$ for all $i \in {\mathcal I}$, $i \ne (2m - 1)$;
\item $X_{-i} = -X_{i - 2}^T$ for all $i \in {\mathcal I}$  and $i \geq 3$;
\item $X_{-1} = -\epsilon X_{-1}^T$;
\end{itemize}
where $X_i$ is the matrix of the linear transformation $X\vert_{V_i}: V_i \rightarrow V_{i + 2}$ of the restriction of $X$ to $V_i$ relative to the bases ${\mathcal B}_i$ and ${\mathcal B}_{i + 2}$ and $X_i^T$ is its transpose for all $i \in {\mathcal I}$ with $i \ne (2m - 1)$. Moreover the subset of elements 
\[
(X_i)_{\begin{subarray}{l} i \in {\mathcal I} \\ i \ne (2m - 1)\end{subarray}} \in \left(\bigoplus_{\begin{subarray}{c} i \in {\mathcal I}\\ i \ne (2m - 1)\end{subarray}} Hom(V_i, V_{i + 2})\right),
\]
  such that the above conditions for the $X_i$ are satisfied,  is a subspace of 
  \[
  \bigoplus_{\begin{subarray}{c} i \in {\mathcal I} \\ i \ne (2m - 1)\end{subarray}} Hom(V_i, V_{i + 2})
  \]
   isomorphic to ${\mathfrak g}_2$.

 \item $Z \in {\mathfrak g}_0$ if and only if all the following conditions are verified:
\begin{itemize}
\item $Z(V_i) \subseteq V_{i}$ for all $i \in {\mathcal I}$;
\item $Z_{-i} = -Z_{i}^T$ for all $i \in {\mathcal I}$;
\end{itemize}
where we denote  $Z_i$ is the matrix of the linear transformation $Z\vert_{V_i}: V_{i} \rightarrow V_{i}$ the restriction of $Z$ to $V_i$ relative to the bases ${\mathcal B}_i$ and ${\mathcal B}_{i}$  and $Z_i^T$ is its transpose for all $i \in {\mathcal I}$. Moreover the subset of elements 
\[
(Z_i)_{i \in {\mathcal I}} \in \left(\bigoplus_{ i \in {\mathcal I}} Hom(V_i, V_{i})\right),
\]
 such that the above conditions for the $Z_i$ are satisfied,  is a subspace of 
 \[
 \bigoplus_{ i \in {\mathcal I}} Hom(V_i, V_{i})
 \]
  isomorphic to ${\mathfrak g}_{0}$.

\item $\{(g_i)_{i \in {\mathcal I}} \in \prod_{i \in {\mathcal I}} GL(V_i) \mid g_{-i} = (g_i^*)^{-1} \text{ for all $i > 0$}\}$ is a closed subgroup of  $\prod_{i \in {\mathcal I}} GL(V_i)$ and is isomorphic to $\prod_{i \in {\mathcal I}, i > 0} GL(V_i)$. Moreover the function 
\[
\Phi: G^{\iota} \rightarrow \left. \left\{(g_i)_{i \in {\mathcal I}} \in \prod_{i \in {\mathcal I}} GL(V_i)\   \right\vert g_{-i} = (g_i^*)^{-1} \text{ for all $i > 0$} \right\}
\]
defined by $\Phi(g) = (g_{i})_{i \in {\mathcal I}}$  is a well-defined  isomorphism of algebraic groups. 

\item The restriction of the adjoint action of $G$ to $G^{\iota}$ acts on ${\mathfrak g}_2$ and,  under the description of ${\mathfrak g}_2$ as a subspace of $\oplus_{i \in {\mathcal I}, i \ne (2m - 1)} Hom(V_i, V_{i + 2})$ in (b), this action  is given by 
\[
Ad(g)(X) \mapsto (g_{i + 2}X_{i}g_{i}^{-1})_{i \in {\mathcal I}, i \ne (2m - 1)}
\]
for all $g \in G^{\iota}$ and where $\Phi(g) = (g_{i})_{i \in {\mathcal I}}$. 
\end{enumerate}
\end{lemma}
\begin{proof}

(a) $\Rightarrow$) For  $i \in {\mathcal I}$, $V_i$ is the eigenspace of $\iota(t)$ with eigenvalue $t^i$ for all $t \in {\mathbf k}^{\times}$. Moreover  when $i, j \in {\mathcal I}$ and $i \ne j$,  these eigenvalues $t^i$ and $t^j$ are distinct.  Because $g \iota(t) = \iota(t) g$ for all $t \in {\mathbf k}^{\times}$ when $g \in G^{\iota}$, then $g$ preserves these eigenspaces and we get $g(V_i) \subseteq V_i$ for all $i \in I$. $g_i \in GL(V_i)$ follows easily  because $g \in GL(V)$ and $g(V_i) \subseteq V_i$.

Because we have $\langle g(v), g(v') \rangle = \langle v, v' \rangle$ for all $v, v' \in V$ and our conditions on the direct sum decomposition, we get for all $v \in V_i$ and $v' \in V_{-i}$ that 
\[
\langle g_i(v), g_{-i}(v')\rangle = \langle g_{-i}^* g_i(v), v'\rangle = \langle v, v'\rangle \quad \Rightarrow \quad (g_{-i}^*) g_i = Id_{V_i} 
 \]
 for all $i \in I$. 
 
 $\Leftarrow$) If $g(V_i) \subseteq V_i$, $g_i \in GL(V_i)$ and $(g_{-i})^* g_i = Id_{V_i}$ for all $i \in {\mathcal I}$, then $g$ is an automorphism and, because of the condition on our direct sum decomposition of $V$ and the fact that $(g_{-i})^* g_i = Id_{V_i}$ for all $i \in {\mathcal I}$, we get that $g$ preserves the non-degenerate form $\langle\ , \ \rangle$, in other words $g \in G$. Because $g(V_i) \subseteq V_i$, we get that $g$ commutes with $\iota(t)$ for all $t \in {\mathbf k}^{\times}$ and $g \in G^{\iota}$.  (a) is proved.

(b) $\Rightarrow$)  Because $Ad(\iota(t)) X = t^2 X$ for all $t \in {\mathbf k}^{\times}$ when $X \in {\mathfrak g}_2$, we have $\iota(t) X = t^2 X \iota(t)$ for all $t \in {\mathbf k}^{\times}$.  So if $v \in V_i$, then $\iota(t) X(v) =  t^2 X \iota(t) (v) = t^{i + 2} X(v)$ and consequently $X(V_i) \in V_{i + 2}$ for all $i \in {\mathcal I}$, $i \ne (2m - 1)$. Note that this last equation means also that $X(V_{2m - 1}) = 0$

Note that if $i \geq 1$, then $-i = i - 2$ if and only if $i = 1$.  Because $\langle X(v), v' \rangle + \langle v, X(v')\rangle = 0$ for all $v, v' \in V$ and our conditions on the direct sum decomposition, we get for $i \in {\mathcal I}$, $i \geq 1$ and for all $v \in V_{-i}$,  $v' \in V_{i - 2}$ that  
\[
\langle X_{-i}(v), v'\rangle + \langle v, X_{i - 2} (v')\rangle = 0.
\]
In this last equation, if we substitute 
\[
X_{-i}(u_{-i, s}) = \sum_{q' = 1}^{{\delta}_{i - 2}} \xi_{q',s}^{(-i)} u_{(-i + 2), q'} \quad \text{ and } \quad X_{(i - 2)}(u_{(i - 2), r}) = \sum_{q = 1}^{{\delta}_{i}} \xi_{q, r}^{(i - 2)} u_{i,q},
\]
we get for $i \geq 3$ that $\epsilon \xi_{r,s}^{(-i)} + \epsilon \xi_{s,r}^{(i - 2)} = 0$ for all $1 \leq s \leq {\delta}_i$ and $1 \leq r \leq {\delta}_{i - 2}$ and thus $X_{-i} = -X_{i - 2}^T $. While if $i = 1$, then we get that  $\xi_{r,s}^{(-1)} + \epsilon \xi_{s,r}^{(-1)} = 0$ for all $1 \leq r, s \leq {\delta}_1$ and thus $X_{-1} = -\epsilon X_{-1}^T$.

$\Leftarrow$)  If $X_{-i} = -X_{i  - 2}^T$ for all $i \in {\mathcal I}$ and $i \geq 3$, and $X_{-1} = -\epsilon X_{-1}^T$, then we get that $\langle X(v), v' \rangle + \langle v, X(v')\rangle = 0$ for all $v, v' \in V$ by only checking it on the basis ${\mathcal B}$  and this means that $X \in {\mathfrak g}$. If we add the condition  $X(V_i) \subseteq V_{i + 2}$ for all $i \in {\mathcal I}$, $i \ne (2m - 1)$, we have  that $X \in {\mathfrak g}_2$. The rest of the statements in (b) are easily proved. 

(c) $\Rightarrow$)  Because $Ad(\iota(t)) Z = Z$ for all $t \in {\mathbf k}^{\times}$ when $Z \in {\mathfrak g}_0$, we have $\iota(t) Z =  Z \iota(t)$ for all $t \in {\mathbf k}^{\times}$.  So if $v \in V_i$, then $\iota(t) Z(v) = Z \iota(t) (v) = t^{i} Z(v)$ and consequently $Z(V_i) \in V_{i}$ for all $i \in {\mathcal I}$. 

Because $\langle Z(v), v' \rangle + \langle v, Z(v')\rangle = 0$ for all $v, v' \in V$ and our conditions on the direct sum decomposition, we get for $i \in {\mathcal I}$ and for all $v \in V_{i}$,  $v' \in V_{-i}$ that  
\[
\langle Z_{i}(v), v'\rangle + \langle v, Z_{-i} (v')\rangle = 0.
\]
In this last equation, if we substitute 
\[
Z_{i}(u_{i, s}) = \sum_{q' = 1}^{{\delta}_{i}} \zeta_{q', s}^{(i)} u_{i, q'} \quad \text{ and } \quad Z_{-i}(u_{-i, r}) = \sum_{q = 1}^{{\delta}_{i}} \zeta_{q, r}^{(-i)} u_{-i, q},
\]
we get for $i \geq 1$ that $ \zeta_{r, s}^{(i)} +  \zeta_{s, r}^{(-i)} = 0$ for all $1 \leq s, r \leq {\delta}_i$ and for $i \leq -1$ that $ \epsilon\zeta_{r, s}^{(i)} +  \epsilon\zeta_{s, r}^{(-i)} = 0$ for all $1 \leq s, r \leq {\delta}_i$.
 Thus $Z_{-i} = -Z_{i}^T $ for all $i \in {\mathcal I}$.

$\Leftarrow$)  If $Z_{-i} = -Z_{i}^T$ for all $i \in {\mathcal I}$, then we get that $\langle Z(v), v' \rangle + \langle v, Z(v')\rangle = 0$ for all $v, v' \in V$ by only checking it on the basis ${\mathcal B}$  and this means that $Z \in {\mathfrak g}$. If we add the condition  $Z(V_i) \subseteq V_i$ for all $i \in {\mathcal I}$, we have  that $Z \in {\mathfrak g}_0$. The rest of the statements in (d) are easily proved. 

(d) follows easily from (a).

(e) follows easily from (b) and the fact that $Ad(g)(X) = gXg^{-1}$ for all $g \in G^{\iota}$. 
\end{proof}

\begin{corollary}\label{C:DimensionG2EvenCase}
 The dimension $dim({\mathfrak g}_2)$  of ${\mathfrak g}_2$  is equal to
\[
 \left(\frac{(\delta_1 - \epsilon) \delta_1}{2}\right) + \sum_{\substack{ i \in {\mathcal I}\\ 1 \leq i < (2m - 1)}} \delta_i \delta_{i + 2}
\]
\end{corollary}
\begin{proof}
This follows easily from (b) of the previous proposition.
\end{proof}

We will now present the combinatorics needed to enumerate the $G^{\iota}$-orbits in  ${\mathfrak g}_2$. 

\begin{definition}\label{D:CoefficientFunction}
A function $c:{\mathbf B} \rightarrow {\mathbb N}$ is said to be {\it symmetric} if and only if $c(\tau(b)) = c(b)$ for all $b \in {\mathbf B}$.  Given such a symmetric function $c:{\mathbf B} \rightarrow {\mathbb N}$, we can associated to it a unique well-defined {\it coefficient function} ${\mathbf c}: {\mathbf B}/\tau \rightarrow {\mathbb N}$ as follows: ${\mathbf c}({\mathcal O}) = c(b)$ where $b$ is any ${\mathcal I}$-box $b \in {\mathcal O}$.  Reciprocally given a  coefficient function ${\mathbf c}: {\mathbf B}/\tau \rightarrow {\mathbb N}$, we can associated to it a unique well-defined symmetric function $c:{\mathbf B} \rightarrow {\mathbf N}$ as follows: $c(b) = {\mathbf c}({\mathcal O})$, where ${\mathcal O}$ is the unique $\langle \tau \rangle$-orbit containing $b$.

Clearly the set of symmetric functions $c:{\mathbf B} \rightarrow {\mathbb N}$ is in bijection with the set of coefficient function ${\mathbf c}: {\mathbf B}/\tau \rightarrow {\mathbb N}$. 

Given a symmetric function $c:{\mathbf B} \rightarrow {\mathbb N}$ and its associated coefficient function ${\mathbf c}: {\mathbf B}/\tau \rightarrow {\mathbb N}$, we can associated to them their {\it dimension vector} $\delta({\mathbf c}) = \delta(c) = (\delta_i({\mathbf c}))_{i \in {\mathcal I}} = (\delta_i(c))_{i \in {\mathcal I}} \in {\mathbb N}^{\mathcal I}$ as follows: 
\[
\delta_i(c) = \delta_i({\mathbf c}) = \sum_{\substack{b \in {\mathbf B}\\ i \in Supp(b)}} c(b).
\]
It is easy to verify that $\delta_{-i}(c) = \delta_i(c)$ for all $i \in {\mathcal I}$. 
\end{definition}

\begin{notation}\label{N:CoefficientFunctionSet}
Denote by ${\mathfrak C}$: the set of coefficient functions ${\mathbf c}: {\mathbf B}/\tau \rightarrow {\mathbb N}$. If $\delta = (\delta_i)_{i \in {\mathcal I}} \in {\mathbb N}^{\mathcal I}$ is an ${\mathcal I}$-tuple  such that $\delta_{-i} = \delta_i$ for all $i \in {\mathcal I}$. We will denote by  ${\mathfrak C}_{\delta}$: the set of coefficient functions ${\mathbf c}: {\mathbf B}/\tau \rightarrow {\mathbb N}$ such that $\delta({\mathbf c}) = \delta$.
\end{notation}

\begin{notation}\label{N:CCorrespondance}
Let $\delta = (\delta_i)_{i \in {\mathcal I}} \in {\mathbb N}^{\mathcal I}$ be an ${\mathcal I}$-tuple such that $\delta_{-i} = \delta_i$ for all $i \in {\mathcal I}$ and let  ${\mathbf c} \in  {\mathfrak C}_{\delta}$ be a coefficient function of dimension $\delta$.  Denote by $V({\mathbf c})$: the direct sum as ${\mathcal I}$-graded vector space and  as $\epsilon$-representations of ${\mathbf c}({\mathcal O})$ copies of $V({\mathcal O})$ defined in \ref{SS:EHFDefinition} with a ${\mathcal I}$-basis 
\[
{\mathcal B}_{\mathbf c} = \coprod_{{\mathcal O} \in {\mathbf B}/\tau} \coprod_{b \in {\mathcal O}} \{v_i^j(b) \mid i \in Supp(b), 1 \leq j \leq {\mathbf c}({\mathcal O})\}
\]
such that, for $i, i' \in {\mathcal I}$, $1 \leq j \leq c(b)$ and $1 \leq j' \leq c(b')$, we have 
\[
\langle v_i^j(b), v_{i'}^{j'}(b')\rangle_{\mathbf c} 
= \begin{cases} 1, &\text{if $j' = j$, $b' = \tau(b)$, $i > 0$ and $i + i' = 0$; } \\
\epsilon, &\text{if $j' = j$,  $b' = \tau(b)$, $i < 0$ and $i + i' = 0$; }\\ 
0, &\text{otherwise.}
\end{cases}
\]
This is just from the formulae of \ref{N:BasisVOrbit} adapted to our situation. Here $0 \not\in {\mathcal I}$

We will denote by ${\mathfrak g}(V({\mathbf c}))$:  the Lie algebra corresponding to $(V({\mathbf c}),  \langle\  , \   \rangle_{\mathbf c})$. In other words, ${\mathfrak g}(V({\mathbf c}))$ is the vector space of linear transformations $X:V({\mathbf c}) \rightarrow V({\mathbf c})$ such that $\langle X(u), v\rangle_{\mathbf c} + \langle u, X(v)\rangle_{\mathbf c} = 0$ for all  $u, v \in V({\mathbf c})$.

From our construction, we have that $V({\mathbf c}) = \oplus_{i \in {\mathcal I}} V_i({\mathbf c})$, where 
$V_i({\mathbf c})$ is ${\mathbf c}({\mathcal O})$ copies of $V_i({\mathcal O})$. We get that  $\dim(V_i({\mathbf c})) = \delta_i$ for all $i \in {\mathcal I}$.  We get this last equation from
\[
\begin{aligned}
\dim(V_i({\mathbf c})) &= \sum_{{\mathcal O} \in {\mathbf B}/\tau} {\mathbf c}({\mathcal O}) \dim(V_i({\mathcal O})) =  \sum_{{\mathcal O} \in {\mathbf B}/\tau} \sum_{b \in {\mathcal O}} c(b) \dim(V_{i}(b))\\ &= \sum_{\substack{b \in {\mathbf B}\\ i \in Supp(b)}} c(b) = \delta_i({\mathbf c}), 
\end{aligned}
\]
because $\dim(V_i(b)) = 1$ when $i \in Supp(b)$. 

Let the three ${\mathcal I}$-graded  linear transformations  
\[
E_{\mathbf c}:V({\mathbf c}) \rightarrow V({\mathbf c}), \quad   H_{\mathbf c}:V({\mathbf c}) \rightarrow V({\mathbf c}) \quad \text{ and } \quad F_{\mathbf c}:V({\mathbf c}) \rightarrow V({\mathbf c})
\]
denote the direct sum as ${\mathcal I}$-graded linear transformation of ${\mathbf c}({\mathcal O})$ copies of the linear transformations 
\[
E_{\mathcal O}:V({\mathcal O}) \rightarrow V({\mathcal O}), \quad  H_{\mathcal O}:V({\mathcal O}) \rightarrow V({\mathcal O}) \quad \text{ and } \quad F_{\mathcal O}:V({\mathcal O}) \rightarrow V({\mathcal O}).
\]
By applying proposition~\ref{P:JMTripleProof} on each summand, we get that these three linear transformations are such that 
\begin{itemize}
\item $E_{\mathbf c}$ is of degree $2$, $H_{\mathbf c}$ is of degree $0$ and $F_{\mathbf c}$ is of degree $-2$;
\item $E_{\mathbf c}, H_{\mathbf c}, F_{\mathbf c}$ belong to the Lie algebra  ${\mathfrak g}(V({\mathbf c}))$;
\item $[H_{\mathbf c}, E_{\mathbf c}] = 2E_{\mathbf c}$, $[H_{\mathbf c}, F_{\mathbf c}] = -2F_{\mathbf c}$ and $[E_{\mathbf c}, F_{\mathbf c}] = H_{\mathbf c}$; in other words,  the triple $(E_{\mathbf c}, H_{\mathbf c}, F_{\mathbf c})$ is an ${\mathcal I}$-graded Jacobson-Morozov triple for the Lie algebra of $(V({\mathbf c}),  \langle\  , \   \rangle_{\mathbf c})$;
\item  the ${\mathcal I}$-graded linear transformation $\phi_{\mathbf c}: V({\mathbf c}) \rightarrow V({\mathbf c})$ given by 
\[
\phi_{{\mathbf c}, i} = {E_{\mathbf c}}\vert_{V_i({\mathbf c})}: V_i({\mathbf c}) \rightarrow V_{i + 2}({\mathbf c}) \quad \text{for all $i \in {\mathcal I}$  such that $(i + 2) \in {\mathcal I}$},
\]
is  such that $(V({\mathbf c}), \phi_{\mathbf c}, \langle\ , \ \rangle_{\mathbf c})$ is an $\epsilon$-representation.
\end{itemize}

As such we can express these linear transformations on the elements of the basis ${\mathcal B}_{\mathbf c}$ using the expressions in \ref{SS:EHFDefinition}. For example, let ${\mathcal O}$ denote the $\langle \tau \rangle$-orbit of $b$, then,  if the orbit ${\mathcal O}$ has no overlapping ${\mathcal I}$-box supports,  $1 \leq j \leq {\mathbf c}({\mathcal O})$ and $i \in Supp(b)$, we have
\[
E_{\mathbf c}(v_i^j(b)) = \begin{cases} {\phantom -} 0, &\text{ if $i = \max(Supp(b))$;}\\ {\phantom -} v_{i + 2}^j(b), &\text{ if $i \ne \max(Supp(b))$ and $i > 0$;}\\  -v_{i + 2}^j(b), &\text{ if $i \ne \max(Supp(b))$ and $i < 0$;} \end{cases}
\]
while if the orbit ${\mathcal O}$ has overlapping ${\mathcal I}$-box supports,  $b = b(i_1, j_1, k_1) \in {\mathcal O}$, $1 \leq j \leq {\mathbf c}({\mathcal O})$  and $i \in Supp(b)$, then we have 
\[
E_{\mathbf c}(v_i^j(b)) = \begin{cases} {\phantom -} 0, &\text{ if $i = \max(Supp(b))$;}\\ {\phantom -} v_{i + 2}^j(b), &\text{ if $i \ne \max(Supp(b))$ and $i \geq 1$;}\\  -v_{i + 2}^j(b), &\text{ if  $i < -1$;}\\ {\phantom -} v_1^j(b), &\text{ if  $i = -1$ and $b$ is below  the principal diagonal;}\\ - \epsilon v_1^j(b), &\text{ if $i = -1$ and $b$ is above  the principal diagonal;}\\  (- \epsilon)^{k_1} v_1^j(b), &\text{ if $i = -1$ and $b$ is on  the principal diagonal;}
\end{cases}
\]
\end{notation}

\subsection{}\label{SS:coefficientOrbit}
Given an element $X \in {\mathfrak g}_2$, we will now associate to it a coefficient function ${\mathbf c}_X:{\mathbf B}/\tau \rightarrow {\mathbb N}$ as follows. By lemma~\ref{G2AsRepresEven} (b) and with the notation of this lemma, $X \in {\mathfrak g}_2$ is equivalent to giving the $\epsilon$-representation  $(V, \phi_V, \langle\  ,\  \rangle)$,where 
\[
\phi_{V, i} = X\vert_{V_i} : V_i \rightarrow  V_{i + 2} \quad v \mapsto X(v) \quad  \text{ for all $v \in V_i$}
\]
is the restriction of $X$ on $V_i$ for all $i \in {\mathcal I}$ such that  $(i + 2) \in {\mathcal I}$. By the theorem of Krull-Remak-Schmidt, this $\epsilon$-representation can be written in essentially unique way as a direct sum of indecomposable $\epsilon$-representations (up to isomorphism). We saw in proposition~\ref{P:JMTripleProof} that the set of indecomposable $\epsilon$-representations is in bijection with the set ${\mathbf B}/\tau$. More precisely, for each orbit ${\mathcal O} \in {\mathbf B}/\tau$, we have constructed a representative $(V({\mathcal O}), \phi_{\mathcal O}, \langle \  , \  \rangle_{\mathcal O})$ of the isomorphism class of the indecomposable $\epsilon$-representation. Denote by ${\mathbf c}_X({\mathcal O})$: the multiplicity  of the indecomposable $\epsilon$-representation $(V({\mathcal O}), \phi_{\mathcal O}, \langle \  , \  \rangle_{\mathcal O})$ (up to isomorphism). So the coefficient function ${\mathbf c}_X:{\mathbf B}/\tau \rightarrow {\mathbb N}$ is defined by ${\mathcal O} \mapsto {\mathbf c}_X({\mathcal O})$.

\subsection{}\label{SS:IsomorphismV(c)andV}
We will now constructed an ${\mathcal I}$-graded isomorphism $T_{\mathbf c}:V({\mathbf c}) \rightarrow V$ such that $\langle T_{\mathbf c}(u), T_{\mathbf c}(v)\rangle = \langle u, v \rangle_{\mathbf c}$ for all $u, v \in V({\mathbf c})$. We will define this isomorphism on the bases ${\mathcal B}_{\mathbf c}$ of $V({\mathbf c})$ and ${\mathcal B}$ of $V$. 

Let $i \in {\mathcal I}$, $i > 0$. We get easily that 
\[
\{b \in {\mathbf B} \mid i \in Supp(b)\} \rightarrow \{b' \in {\mathbf B} \mid -i = \sigma(i) \in Supp(b')\} \text{ defined by } b \mapsto \tau(b)
\]
is a bijection.  Let $c:{\mathbf B} \rightarrow {\mathbb N}$ be the unique symmetric function corresponding to the coefficient function ${\mathbf c}: {\mathbf B}/\tau \rightarrow {\mathbb N}$.  Fix a total order $b_1 > b_2 > \dots > b_n$ on the set $\{b \in {\mathbf B} \mid i \in Supp(b)\}$ such that $c(b_r) \geq c(b_s)$ whenever $r < s$ for $1 \leq r, s \leq n$ and let $b'_1 > b'_2 > \dots > b'_n$ be the total order on the set $\{b' \in {\mathbf B} \mid -i = \sigma(i) \in Supp(b')\}$ such that     $b'_r = \tau(b_r)$ for $1 \leq r \leq n$. 
Consider the Young diagram corresponding to the partition $c(b_1) \geq c(b_2) \geq \dots \geq c(b_n)$. So the row indexed by $b_r$ will have $c(b_r)$ columns.  We would get the same Young diagram $c(b'_1) \geq c(b'_2) \geq \dots \geq c(b'_n)$ if we had started with the total order  $b'_1 > b'_2 > \dots > b'_n$  on the set $\{b' \in {\mathbf B} \mid -i = \sigma(i) \in Supp(b')\}$. This follows from the fact that $c$ is symmetric.  This partition is a partition of the integer $\delta_i$ from~\ref{N:CCorrespondance}.  In this Young diagram, we fill the boxes with the integers from 1 to $\delta_i$ in strictly increasing order from left to right and top to bottom.   We define $T_{{\mathbf c}, i}: V_i({\mathbf c}) \rightarrow V_i$ (respectively $T_{{\mathbf c}, -i}: V_{-i}({\mathbf c}) \rightarrow V_{-i}$)  to be the unique linear transformation such that if the entry in the box at row $r$ and column $s$ is $j$, then 
$T_{{\mathbf c}, i}(v_i^s(b_r)) = u_{i, j}$ (respectively $T_{{\mathbf c}, -i}(v_{-i}^s(b'_r)) = T_{{\mathbf c}, -i}(v_{-i}^s(\tau(b_r))) = u_{-i, j}$).  

Finally $T_{\mathbf c}:V_{\mathbf c} \rightarrow V$ is the direct sum of the linear transformations $T_{{\mathbf c}, i}: V_i({\mathbf c}) \rightarrow V_i$ and $T_{{\mathbf c}, -i}: V_{-i}({\mathbf c}) \rightarrow V_{-i}$ for all $i \in {\mathcal I}$, $i > 0$.  Note that  the dimension vector  $\delta({\mathbf c}) $ of ${\mathbf c}$ is  $\delta({\mathbf c}) = (\delta_i)_{i \in {\mathcal I}}$, where $\delta_i = \dim(V_i)$ for all $i \in {\mathcal I}$.

\begin{lemma}\label{L:TcIsomorphism}
With the above definition, we have that $T_{\mathbf c}:V({\mathbf c}) \rightarrow V$ is an ${\mathcal I}$-graded isomorphism such that $\langle T_{\mathbf c}(u), T_{\mathbf c}(v)\rangle = \langle u, v \rangle_{\mathbf c}$ for all $u, v \in V({\mathbf c})$.
\end{lemma}
\begin{proof}
From our construction, it is clear that $T_{\mathbf c}$ is an ${\mathcal I}$-graded isomorphism.  We use the fact that $T_{\mathbf c}$ restricted to ${\mathcal B}_{\mathbf c}$ is a bijection with the basis ${\mathcal B}$ of $V$. To check that $\langle T_{\mathbf c}(u), T_{\mathbf c}(v)\rangle = \langle u, v \rangle_{\mathbf c}$ for all $u, v \in V({\mathbf c})$, it is enough to check it on the basis ${\mathcal B}_{\mathbf c}$. We want to consider $\langle T_{\mathbf c}(v_i^s(b)), T_{\mathbf c}(v_{i'}^{s'}(b'))\rangle$ where $i, i' \in {\mathcal I}$, $i \in Supp(b)$, $1 \leq s \leq c(b)$, $i' \in Supp(b')$ and $1 \leq s' \leq c(b')$. Because $T_{\mathbf c}$ is ${\mathcal I}$-graded, our hypothesis in~\ref{S:SetUpAeven} on the basis ${\mathcal B}$ of $V$ and our construction in~\ref{N:CCorrespondance}, we get if $i + i' \ne 0$, that 
\[
\langle T_{\mathbf c}(v_i^s(b)), T_{\mathbf c}(v_{i'}^{s'}(b'))\rangle = 0 = \langle v_i^s(b), v_{i'}^{s'}(b')\rangle_{\mathbf c}.
\]

Now let $i' = -i$ and  assume first that $i > 0$. Consider the same total orders on the sets $\{b \in {\mathbf B} \mid i \in Supp(b)\}$  and $\{b' \in {\mathbf B} \mid -i = \sigma(i) \in Supp(b')\}$ as in \ref{SS:IsomorphismV(c)andV}. Proceeding as above with the Young diagram, assume that   $b = b_r$ in the total order $b_1 > b_2 > \dots > b_n$ on the set $\{b \in {\mathbf B} \mid i \in Supp(b)\}$ and   $b' = b'_{r'} = \tau(b_{r'})$ in the total order $b'_1 > b'_2 > \dots > b'_n$ on the set $\{b' \in {\mathbf B} \mid -i = \sigma(b') \in Supp(b')\}$ , the entry in the box at row $r$ and column $s$ is $j$ and the entry in the box at row $r'$ and column $s'$ is $j'$, then 
\[
\begin{aligned}
\langle T_{\mathbf c}(v_i^s(b)), T_{\mathbf c}(v_{-i}^{s'}(b'))\rangle &= \langle T_{\mathbf c}(v_i^s(b_r)), T_{\mathbf c}(v_{-i}^{s'}(b'_{r'}))\rangle = \langle u_{i, j}, u_{-i, j'}\rangle\\ &= \begin{cases} 1, &\text{if $j = j'$;} \\ 0, &\text{otherwise;}\end{cases} = \begin{cases} 1, &\text{if $r= r'$ and $s = s'$;} \\ 0, &\text{otherwise.}\end{cases} 
\end{aligned}
\]
If we now consider the bilinear form $\langle  \  ,  \  \rangle_{\mathbf c}$  in $V({\mathbf c})$, we have
\[
\begin{aligned}
\langle v_i^s(b), v_{-i}^{s'}(b')\rangle_{\mathbf c} &= \langle v_i^s(b_r), v_{-i}^{s'}(b'_{r'})\rangle_{\mathbf c} = \langle v_i^s(b_r), v_{-i}^{s'}(\tau(b_{r'}))\rangle_{\mathbf c}\\ &= \begin{cases} 1, &\text{if $r= r'$ and $s = s'$;} \\ 0, &\text{otherwise.}\end{cases} 
\end{aligned}
\]
So if $i > 0$, then $\langle T_{\mathbf c}(v_i^s(b)), T_{\mathbf c}(v_{-i}^{s'}(b'))\rangle = \langle v_i^s(b), v_{-i}^{s'}(b')\rangle_{\mathbf c}$.

If $i \in {\mathcal I}$, $i < 0$, then 
\[
\begin{aligned}
\langle T_{\mathbf c}(v_i^s(b)), T_{\mathbf c}(v_{-i}^{s'}(b'))\rangle &= \epsilon \langle T_{\mathbf c}(v_{-i}^{s'}(b')), T_{\mathbf c}(v_{i}^{s}(b))\rangle\\ &=  \epsilon  \langle v_{-i}^{s'}(b'), v_{i}^{s}(b)\rangle_{\mathbf c} =  \langle v_{i}^{s}(b), v_{-i}^{s'}(b')\rangle_{\mathbf c}.
\end{aligned}
\]
We have proved that $\langle T_{\mathbf c}(u), T_{\mathbf c}(v)\rangle = \langle u, v \rangle$ for all $u, v \in V$.

\end{proof}

\begin{proposition} \label{Prop_Indexation_Orbits_Even}
Let $\delta = (\delta_i)_{i \in {\mathcal I}}$ be as in \ref{S:SetUpAeven}  and ${\mathfrak g}_2/G^{\iota}$ denotes the set of  $G^{\iota}$-orbits  in ${\mathfrak g}_2$. Then the map 
\[
\Upsilon:{\mathfrak g}_2/G^{\iota} \rightarrow  {\mathfrak C}_{\delta} \quad \text{ defined by } \quad G^{\iota} \cdot X \mapsto {\mathbf c}_X,
\]
is a well-defined bijection.
\end{proposition}
\begin{proof}

By our last observation in \ref{SS:coefficientOrbit}, if $G^{\iota} \cdot X = G^{\iota} \cdot X'$, then  the corresponding $\epsilon$-representations for $X$ and $X'$ are isomorphic, we get  that ${\mathbf c}_X = {\mathbf c}_{X'}$. So the  function $\Upsilon$ is well-defined. 

To show that $\Upsilon$ is an injective function,  assume that ${\mathbf c}_X = {\mathbf c}_{X'}$, then this means that, by the Krull-Remak-Schmidt theorem, the $\epsilon$-representations for $X$ and $X'$ are isomorphic. Note that in the case of the orthogonal representation (i.e. $\epsilon = 1$), the group $G$ is not $SO(V)$, but rather the group ${\widehat G} = O(V)$, but it is easy to see that  $G^{\iota} = {\widehat G}^{\iota}$ as seen in lemma~\ref{G2AsRepresEven} and the isomorphism is the same for the case of $SO(V)$ and $O(V)$. So $X$ and $X'$ are in the same $G^{\iota}$-orbit and $G^{\iota} \cdot X = G^{\iota} \cdot X'$ 

To prove that $\Upsilon$ is surjective, we need to show that there is an element $X \in {\mathfrak g}_2$ such that ${\mathbf c}_X = {\mathbf c} \in {\mathfrak C}_{\delta}$.  Consider $X = T_{\mathbf c} E_{\mathbf c} T_{\mathbf c}^{-1}:V \rightarrow V$, where $T_{\mathbf c}:V({\mathbf c}) \rightarrow V$ is the ${\mathcal I}$-graded isomorphism defined in \ref{SS:IsomorphismV(c)andV}.  Because $E_{\mathbf c} \in {\mathfrak g}_2(V({\mathbf c}))$ and lemma~\ref{L:TcIsomorphism}, $X \in {\mathfrak g}_2$. From our construction of $E_{\mathbf c}: V({\mathbf c}) \rightarrow V({\mathbf c})$, it is isomorphic as ${\mathcal I}$-graded and $\epsilon$-representation  to the direct sum of ${\mathbf c}({\mathcal O})$ copies of the linear transformation $E_{\mathcal O}:V({\mathcal O}) \rightarrow V({\mathcal O})$. By carrying this direct sum using $T_{\mathbf c}$ to $V$, we do get that ${\mathbf c}_X = {\mathbf c}$ and the proof is complete. 
\end{proof}

\begin{notation}
Let $\delta = (\delta_i)_{i \in {\mathcal I}} \in {\mathbb N}^{\mathcal I}$ be such that $\delta_i = \dim(V_i)$ for all $i \in {\mathcal I}$ as in \ref{S:SetUpAeven}. If ${\mathbf c} \in {\mathfrak C}_{\delta}$, we will denote by ${\mathcal O}_{\mathbf c}$: the unique $G^{\iota}$-orbit in ${\mathfrak g}_2$ such that $\Upsilon({\mathcal O}_{\mathbf c}) = {\mathbf c}$.
\end{notation}

\begin{remark}\label{R:IsomoVVc}
Using the isomorphism $T_{\mathbf c}:V({\mathbf c}) \rightarrow V$, we can transfer results about the orbit ${\mathcal O}_{\mathbf c}$ in the Lie algebra ${\mathfrak g}$ to the orbit of $E_{\mathbf c}$ in the Lie algebra 
${\mathfrak g}(V({\mathbf c}))$.
\end{remark}

\subsection{}
 To compute the dimension of the orbit ${\mathcal O}_{\mathbf c}$ corresponding to the coefficient function ${\mathbf c} \in  {\mathfrak C}_{\delta}$, we need to recall Lusztig's construction of a parabolic subalgebra ${\mathfrak p}_X$ in ${\mathfrak g}$ associated to an element $X \in {\mathfrak g}_2$ and the formula  for the dimension of ${\mathcal O}_{\mathbf c}$ involving the Lie algebra ${\mathfrak p}_X$,  where $X \in {\mathcal O}_{\mathbf c}$.

\begin{definition}\label{D:LieHomo}
Let $X \in {\mathfrak g}_2$. It is known that we can find a Lie algebra homomorphism $\phi: {\mathfrak sl}_2({\mathbf k}) \rightarrow {\mathfrak g}$ such that 
\[
\phi\begin{pmatrix}0 & 1\\ 0 & 0\end{pmatrix} = X \in {\mathfrak g}_2, \qquad \phi\begin{pmatrix}1 & \phantom{-}0\\ 0 & -1\end{pmatrix} \in {\mathfrak g}_0, \qquad \phi\begin{pmatrix}0 & 0\\ 1 & 0\end{pmatrix} \in {\mathfrak g}_{-2}.
\]
Denote by $\widetilde{\phi}: SL_2({\mathbf k}) \rightarrow G$: the homomorphism of algebraic groups whose differential is $\phi$ and by $\iota':{\mathbf k}^{\times} \rightarrow G$: the homomorphism defined by 
\[
\iota'(t) = \widetilde{\phi}\begin{pmatrix}t & 0\phantom{-} \\ 0 & t^{-1}\end{pmatrix} \qquad  \text{ for all $t \in {\mathbf k}^{\times}$}.
\]
This gives us  a second grading for the Lie algebra ${\mathfrak g} = \oplus_r  ({ }_r {\mathfrak g})$ where 
\[
({ }_r {\mathfrak g}) = \{ X \in {\mathfrak g} \mid Ad(\iota'(t))(X) = t^r X \quad \text{ for all $t \in {\mathbf k}^{\times}$}\}
\]
and,  because $\iota'({\mathbf k}^{\times}) \subset G^{\iota}$, we have the direct sum decomposition 
\[
{\mathfrak g} = \bigoplus_{r, r'}  ({ }_r{\mathfrak g}_{r'}).
\]
In  5.2 of \cite{L1995}, Lusztig defines the {\it parabolic subalgebra ${\mathfrak p}_X$} as 
\[
{\mathfrak p}_X = \bigoplus_{\begin{subarray}{c} r, r' \\ r' \leq r \end{subarray}}  ({ }_r{\mathfrak g}_{r'})
\]
\end{definition}

\begin{proposition}[Lusztig]
If $X \in {\mathcal O}_{\mathbf c}$ where ${\mathbf c} \in  {\mathfrak C}_{\delta}$,  then 
\[
\dim({\mathcal O}_{\mathbf c}) = \dim({\mathfrak g}_0) - \dim({\mathfrak p}_0) + \dim({\mathfrak p}_2),
\]
where ${\mathfrak p}$ is the parabolic subalgebra ${\mathfrak p}_X$  associated above to $X$. 
\end{proposition}
\begin{proof}
See 4.1 in \cite{L2010}.
\end{proof}

We now want to use this formula to compute the dimension $\dim({\mathcal O}_{\mathbf c})$ of the orbit ${\mathcal O}_{\mathbf c} \in {\mathfrak g}_2/G^{\iota}$ in term of the coefficient function ${\mathbf c} \in {\mathfrak C}_{\delta} $. This is done in the next proposition.

\begin{proposition}\label{DimensionFormulaEven}
Let ${\mathbf c}  \in  {\mathfrak C}_{\delta}$  and consider the symmetric function $c:{\mathbf B} \rightarrow {\mathbb N}$ corresponding to ${\mathbf c}$.  Then the dimension $\dim({\mathcal O}_{\mathbf c})$ of the orbit ${\mathcal O}_{\mathbf c}$ is equal to
\[
\left[
\begin{aligned}
&\sum_{\begin{subarray}{c} i \in {\mathcal I}\\ i > 0\end{subarray}} \delta_i^2 - \sum_{\begin{subarray}{c}i \in {\mathcal I}\\ i > 0\\ (b, b') \in{\mathbf B} \times {\mathbf B}\\ i \in Supp(b) \cap Supp(b')\\ \lambda(b) \geq\lambda(b') \end{subarray}} c(b) c(b') + \sum_{\begin{subarray}{c} i \in {\mathcal I}\\ 0 < i < (2m - 1)\\ (b, b')\in {\mathbf B} \times {\mathbf B}\\ i \in Supp(b)\\ (i + 2) \in Supp(b')\\\lambda(b) \geq \lambda(b') \end{subarray}} c(b) c(b') \\ \\ & + \frac {1}{2}\sum_{\begin{subarray}{c}(b, b')\in {\mathbf B} \times {\mathbf B} \\ \tau(b) \ne b' \\ 1 \in Supp(b')\\ -1 \in   Supp(b) \\ \lambda(b') \leq \lambda(b)\end{subarray}} c(b) c(b') +  \frac {1}{2}\sum_{\begin{subarray}{c} b \in {\mathbf B}\\ -1 \in Supp(b)\\ \lambda(b) \geq 0 \end{subarray}} c(b) (c(b) - 1) \\ \\  & + \frac{(1 - \epsilon)}{2} \sum_{\begin{subarray}{c} b \in {\mathbf B}\\ -1 \in Supp(b)\\\lambda(b) \geq 0 \end{subarray}} c(b)
\end{aligned}
\right]
\]
Recall that the value $\lambda(b)$ for $b \in {\mathbf B}$ has been defined in notation~\ref{N:LambdaValue}. 
\end{proposition}
\begin{proof}Taking into account our remark~\ref{R:IsomoVVc}, we will  assume that $V = V({\mathbf c})$.  We have that $\dim(V_i) = \delta_i$ for all $i \in {\mathcal I}$ from the fact that $\delta({\mathbf c}) = \delta$.
We have constructed in~\ref{N:CCorrespondance}  a Jacobson-Morozov triple  $(E_{\mathbf c}, H_{\mathbf c}, F_{\mathbf c})$ with $E_{\mathbf c} \in {\mathcal O}_{\mathbf c}$; in other words,  we have 
\[
[H_{\mathbf c}, E_{\mathbf c}] = 2 E_{\mathbf c}, \qquad [H_{\mathbf c}, F_{\mathbf c}] = -2 F_{\mathbf c}, \qquad [E_{\mathbf c}, F_{\mathbf c}] = H_{\mathbf c},
\]
where $E_{\mathbf c} \in {\mathfrak g}_2$, $H_{\mathbf c} \in {\mathfrak g}_0$ and $F_{\mathbf c} \in {\mathfrak g}_{-2}$, and such that $E_{\mathbf c} \in {\mathcal O}_{\mathbf c}$.   Thus we have a Lie algebra homomorphism $\phi: {\mathfrak sl}_2({\mathbf k}) \rightarrow {\mathfrak g}$ such that 
\[
\phi\begin{pmatrix}0 & 1\\ 0 & 0\end{pmatrix} = E_{\mathbf c}, \qquad \phi\begin{pmatrix}1 & \phantom{-}0\\ 0 & -1\end{pmatrix} = H_{\mathbf c}, \qquad \phi\begin{pmatrix}0 & 0\\ 1 & 0\end{pmatrix} = F_{\mathbf c}. 
\]
As in definition~\ref{D:LieHomo}, let $\widetilde{\phi}: SL_2({\mathbf k}) \rightarrow G$ be the homomorphism of algebraic groups whose differential is $\phi$ and by $\iota':{\mathbf k}^{\times} \rightarrow G$: the homomorphism defined by 
\[
\iota'(t) = \widetilde{\phi}\begin{pmatrix}t & 0\phantom{-} \\ 0 & t^{-1}\end{pmatrix} \qquad  \text{ for all $t \in {\mathbf k}^{\times}$.}
\]
We have the grading for the Lie algebra ${\mathfrak g} = \oplus_r  ({ }_r {\mathfrak g})$ where 
\[
({ }_r {\mathfrak g}) = \{ X \in {\mathfrak g} \mid Ad(\iota'(t))(X) = t^r X \quad \text{ for all $t \in {\mathbf k}^{\times}$}\}
\]
By taking the derivative, we get that 
\[
({ }_r {\mathfrak g}) = \{ X \in {\mathfrak g} \mid ad(H_{\mathbf c})(X) =r X \}.
\]
By lemma~\ref{G2AsRepresEven} (c), we have that 
\[
\dim({\mathfrak g}_0) = \sum_{\begin{subarray}{c} i \in {\mathcal I}\\ i > 0\end{subarray}} \delta_i^2.
\]
Let ${\mathfrak p}$ be the parabolic subalgebra associated to the nilpotent element $E_{\mathbf c}$ as in the definition~\ref{D:LieHomo} . From the definition of ${\mathfrak p}$, we get that 
\[
{\mathfrak p}_0 = \bigoplus_{r \geq 0}  ({ }_r{\mathfrak g}_0) \qquad \text{ and } \qquad {\mathfrak p}_2 = \bigoplus_{r \geq 2}  ({ }_r{\mathfrak g}_2).
\]
By our description of ${\mathfrak g}_0$ and ${\mathfrak g}_2$ in lemma~\ref{G2AsRepresEven}, we get that 
\[
 {\mathfrak p}_0 = \bigoplus_{r \geq 0}\bigoplus_{\begin{subarray}{c} i \in {\mathcal I}\\ 0 < i  \end{subarray}} \{(Z_i : V_i \rightarrow V_i) \in Hom(V_i, V_i) \mid ad(H_{\mathbf c}\vert_{V_i}) Z_i = r Z_i \}
\]
and ${\mathfrak p}_2 = {\mathfrak p}'_2 \oplus {\mathfrak p}''_2$ where 
\[
{\mathfrak p}'_2 = 
\bigoplus_{r \geq 2}\bigoplus_{\begin{subarray}{c} i \in {\mathcal I}\\ 0 < i < (2m - 1)\end{subarray}} \{(X_i : V_i \rightarrow V_{i + 2}) \in Hom(V_i, V_{i + 2}) \mid ad(H_{\mathbf c}\vert_{V_i}) X_i = r X_i \}
\]
and 
\[
{\mathfrak p}''_2 = 
\bigoplus_{r \geq 2}\left\{(X_{-1}:V_{-1} \rightarrow V_1) \in Hom(V_{-1}, V_1) \left\vert \begin{aligned} &X_{-1} = -\epsilon X_{-1}^*  \text{ and } \\  &ad(H_{\mathbf c}\vert_{V_{-1}}) X_{-1}  = r X_{-1} \end{aligned}\right.\right\}
\]
We can give  bases for the spaces $Hom(V_i, V_i)$ and $Hom(V_i, V_{i + 2})$ that are eigenvectors for $ad(H_{\mathbf c}\vert_{V_i})$.  From our construction of the Jacobson-Morozov triple $\{E_{\mathbf c}, H_{\mathbf c}, F_{\mathbf c}\}$ and notation~\ref{N:CCorrespondance},  we have for each $i \in {\mathcal I}$ a basis  
\[
{\mathcal B}_i = \{v_{i}^{j}(b) \in V_i  \mid b \in {\mathbf B}, i \in Supp(b), 1 \leq j \leq c(b)\}
\]
of $V_i$ such that $H_{\mathbf c}(v_{i}^{j}(b)) = h_{i}(b) v_{i}^{j}(b)$ for all $b \in {\mathbf B}$, $i \in Supp(b)$ and $1 \leq j \leq c(b)$. For each $b, b' \in {\mathbf B}$ with $i \in Supp(b) \cap Supp(b')$, $1 \leq j \leq c(b)$ and $1 \leq j' \leq c(b')$, let $Z_{(b, j)}^{(b', j')} \in Hom(V_i, V_i)$ be defined on the basis ${\mathcal B}_i $ by 
\[
Z_{(b, j)}^{(b', j')}(v_i^{j''}(b'')) = \begin{cases} v_i^{j'}(b'), &\text{if $b'' = b$ and $j'' = j$;}\\ 0, &\text{otherwise.}\end{cases}
\]

Clearly these linear transformations $Z_{(b, j)}^{(b', j')}$ with  $b, b' \in {\mathbf B}$,  $i \in Supp(b) \cap Supp(b')$, $1 \leq j \leq c(b)$ and $1 \leq j' \leq c(b')$ form a basis of $Hom(V_i, V_i)$ and $Z_{(b, j)}^{(b', j')}$ is an eigenvector of $ad(H_{\mathbf c}\vert_{V_i})$ with
\[
ad(H_{\mathbf c}) (Z_{(b, j)}^{(b', j')}) = (h_{i}(b') - h_{i}(b)) Z_{(b, j)}^{(b', j')}
\]
by computing both sides  on the basis vector $v_{i}^{j''}(b'')$. Consequently
\[
\dim({\mathfrak p}_0) = 
\sum_{\begin{subarray}{c}i \in {\mathcal I}\\ i > 0\\ (b, b') \in {\mathbf B} \times {\mathbf B}\\ i \in Supp(b) \cap Supp(b')\\ (h_{i}(b') - h_{i}(b)) \geq 0\end{subarray}} c(b) c(b') = \sum_{\begin{subarray}{c}i \in {\mathcal I}\\ i > 0\\ (b, b') \in {\mathbf B} \times {\mathbf B}\\ i \in Supp(b) \cap Supp(b')\\ \lambda(b) \geq  \lambda(b') \end{subarray}} c(b) c(b')
\]
because $(h_i(b') - h_i(b)) \geq 0$ is equivalent to $(i - h_i(b)) - (i - h_i(b')) \geq 0$.  This is simply $\lambda(b) \geq \lambda(b')$.

Let now $i \in {\mathcal I}$ and $0 < i  < (2m - 1)$.  For each $b, b' \in {\mathbf B}$ with $i \in Supp(b)$, $(i + 2) \in Supp(b')$, $1 \leq j \leq c(b)$ and $1 \leq j' \leq c(b')$, let $X_{(b, j)}^{(b', j')} \in Hom(V_i, V_{i + 2})$ be defined on the basis ${\mathcal B}_i $ by 
\[
X_{(b, j)}^{(b', j')}(v_i^{j''}(b'')) = \begin{cases} v_{(i + 2)}^{j'}(b'), &\text{if $b'' = b$ and $j'' = j$;}\\ 0, &\text{otherwise.}\end{cases}
\]
Clearly these linear transformations $X_{(b, j)}^{(b', j')}$ with  $b, b' \in {\mathbf B}$ with $i \in Supp(b)$,  $(i + 2) \in Supp(b')$, $1 \leq j \leq c(b)$ and $1 \leq j' \leq c(b')$ form a basis of $Hom(V_i, V_{i + 2})$ and $X_{(b, j)}^{(b', j')}$ is an eigenvector of $ad(H_{\mathbf c}\vert_{V_i})$ with 
\[
ad(H_{\mathbf c}\vert_{V_i}) (X_{(b, j)}^{(b', j')}) = (h_{(i + 2)}(b') - h_{i}(b)) X_{(b, j)}^{(b', j')}
\]
by computing both sides  on the basis vector $v_{i}^{j''}(b'')$. Consequently
\[
\dim({\mathfrak p}'_2) = 
\sum_{\begin{subarray}{c}i \in {\mathcal I}\\0 < i < (2m - 1)\\ (b, b') \in {\mathbf B} \times {\mathbf B}\\ i \in Supp(b)\\ (i + 2) \in  Supp(b')\\ (h_{(i + 2)}(b') - h_i(b)) \geq 2\end{subarray}} c(b) c(b') = \sum_{\begin{subarray}{c}i \in {\mathcal I}\\0 < i < (2m - 1)\\ (b, b') \in {\mathbf B} \times {\mathbf B}\\ i \in Supp(b)\\ (i + 2) \in  Supp(b')\\ \lambda(b) \geq  \lambda(b') \end{subarray}} c(b) c(b')
\]
because $(h_{(i + 2)}(b') - h_i(b)) \geq 2$ is equivalent to $(i - h_i(b)) - ((i + 2) - h_{(i + 2)}(b')) \geq 0$.  This is simply $\lambda(b) \geq \lambda(b')$.

Finally we need to compute the dimension of  ${\mathfrak p}''_2$.  If we  use the basis as in notation~\ref{N:CCorrespondance} appearing in $V_{-1}$, then we get the basis   
\[
\{v_{-1}^j(b) \mid b \in {\mathbf B}, -1 \in Supp(b), 1 \leq j \leq c(b)\}.
\]
of $V_{-1}$.

The involution $\tau$ restricted to $\{b \in {\mathbf B}\mid -1 \in Supp(b)\}$ gives a bijection between  $\{b \in {\mathbf B}\mid -1 \in Supp(b)\}$ and  $\{b \in {\mathbf B}\mid 1 \in Supp(b)\}$. We can use this bijection to index the basis used  in notation~\ref{N:CCorrespondance} for $V_{1}$ and also the fact that the function $c$ is symmetric. Thus we get the basis   
 \[
 \{v_1^j(\tau(b)) \mid b \in{\mathbf B}, -1 \in Supp(b), 1 \leq j \leq c(b)\}.
 \]
 of $V_1$.

For these bases, we have 
\[
\langle v_1^{j'}(\tau(b')), v_{-1}^j(b)\rangle_{\mathbf c} = \epsilon \langle v_{-1}^j(b), v_1^{j'}(\tau(b'))\rangle_{\mathbf c} = \begin{cases} 1, &\text{$b = b', j' = j$;}\\ 0, &\text{ otherwise;} \end{cases}
\]
where $b, b' \in{\mathbf B}$, $-1 \in Supp(b) \cap Supp(b')$, $1 \leq j \leq c(b)$ and $1 \leq j' \leq c(b')$ . 

For each homomorphism $X_{-1}:V_{-1} \rightarrow V_1$, we can write
\[
X_{-1}(v_{-1}^j(b)) = \sum_{(b'', j'')} \xi_{(b, j)}^{(b'', j'')} v_1^{j''}(\tau(b'')) 
\]
and 
\[
X_{-1}(v_{-1}^{j'}(b')) = \sum_{(b''', j''')} \xi_{(b', j')}^{(b''', j''')} v_1^{j'''}(\tau(b'''))
\]
where the sums run over  the set of pairs $(b'', j'')$, $(b''', j''')$ such that $b'', b''' \in {\mathbf B}$,  $-1 \in Supp(b'') \cap Supp(b''')$,   $1 \leq j'' \leq c(b'')$ and $1 \leq j''' \leq c(b''')$.

If we add the condition of being in the $Hom(V_{-1}, V_1)$-component  of ${\mathfrak g}_2$,  that is
\[
\langle X_{-1}(v), v'\rangle + \langle v, X_{-1}(v')\rangle = 0 \quad \text{ for all $v, v' \in V_{-1}$,}
\]
this means that when $v = v_{-1}^j(b)$ and $v' = v_{-1}^{j'}(b')$, we get 
\begin{equation}
\xi_{(b, j)}^{(b', j')} + \epsilon \xi_{(b', j')}^{(b, j)} = 0
\end{equation}
for all $b, b' \in {\mathbf B}$, $-1 \in Supp(b) \cap Supp(b')$, $1 \leq j \leq c(b)$ and $1 \leq j' \leq c(b')$. The converse is also true. More precisely,  if we impose the conditions above for all these coefficients $\xi$, then the corresponding linear transformation $X_{-1}$ belongs to the $Hom(V_{-1}, V_1)$-component  of  ${\mathfrak g}_2$.

We will use this to give a basis of the $Hom(V_{-1}, V_1)$-component  of  ${\mathfrak g}_2$.  For $b, b' \in {\mathbf B}$, $-1 \in Supp(b) \cap Supp(b')$,  $1 \leq j \leq c(b)$ and $1 \leq j' \leq c(b')$.  First we assume that $(b', j') \ne (b, j)$, then we define the linear transformation $X_{(b, j)}^{(b', j')}:V_{-1} \rightarrow V_1$ as follows:
\[
X_{(b, j)}^{(b', j')}(v_{-1}^{j''}(b'')) = \begin{cases} {\phantom -}v_1^{j'}(\tau(b')), &\text{ if $(b'', j'') = (b, j)$}; \\ \\ -\epsilon v_1^j(\tau(b)), &\text{ if $(b'', j'') = (b', j')$;}\\ \\ 0, &\text{ otherwise.}
\end{cases}
\]
In other words, we have 
\[
\xi_{(b, j)}^{(b'', j'')} = \begin{cases}1, &\text{if $(b'', j'') = (b', j')$;}\\ 0, &\text{otherwise;}\end{cases}
\]
\[
\xi_{(b', j')}^{(b''', j''')} = \begin{cases}-\epsilon, &\text{if $(b''', j''') = (b, j)$;}\\ {\phantom -}0, &\text{otherwise}\end{cases}
\]
and equation (1) is satisfied. Thus $X_{(b, j)}^{(b', j')}$ belongs to the $Hom(V_{-1}, V_1)$-component  of  ${\mathfrak g}_2$. By computing on all of the basis vector $v_{-1}^{j''}(b'')$, we get that 
\[
X_{(b, j)}^{(b', j')} = -\epsilon X_{(b', j')}^{(b, j)}.
\]

$X_{(b, j)}^{(b', j')}$ is an eigenvector of $ad(H_{\mathbf c}\vert_{V_{-1}})$ with
\[
ad(H_{\mathbf c}\vert_{V_{-1}})X_{(b, j)}^{(b', j')} = (h_1(\tau(b)) - h_{-1}(b')) X_{(b, j)}^{(b', j')}
\]
by computing both sides  on the basis vector $v_{-1}^{j''}(b'')$. Note that we have used here  lemma~\ref{L:RelationCoefficientsHandF} (a). 

Note that $(b', j') \ne (b, j)$ is equivalent to either $b' \ne b$ or ($b' = b$ and $j' \ne j$).

If we are in the orthogonal case, in other words when $\epsilon = 1$, then we get from equation (1) that $\xi_{(b, j)}^{(b, j)} = 0$ for all $b \in{\mathbf B}$, $-1 \in Supp(b)$ and $1 \leq j \leq c(b)$. We can totally order the finite set $\{(b, j)\mid b \in {\mathbf B}, -1 \in Supp(b)\}$. We will write $<$ for this total order.  With this,  we get easily that the linear transformations  $X_{(b, j)}^{(b', j')}$ where $(b, j) < (b', j')$ is a basis of the $Hom(V_{-1}, V_1)$-component  of  ${\mathfrak g}_2$ and, because of our computation of $ad(H_{\mathbf c}\vert_{V_{-1}})$ on these basis elements, we get that
\[
\dim({\mathfrak p}''_2) = \frac {1}{2}\sum_{\begin{subarray}{c}(b, b')\in {\mathbf B} \times {\mathbf B}\\ b' \ne b\\ -1 \in Supp(b) \cap Supp(b')\\ (h_1(\tau(b)) -  h_{-1}(b') \geq 2\end{subarray}} c(b) c(b')  + \frac {1}{2}\sum_{\begin{subarray}{c} b \in {\mathbf B} \\ -1 \in Supp(b)\\ -h_{-1}(b) \geq 1\end{subarray}} c(b) (c(b) - 1)
\]
Note that $h_1(\tau(b)) - h_{-1}(b') \geq 2$ is equivalent to $-1 - h_{-1}(b) \geq 1 - h_1(\tau(b'))$ by lemma~\ref{L:RelationCoefficientsHandF} (a).  This is simply $\lambda(b) \geq \lambda(\tau(b'))$. Similarly $-h_{-1}(b) \geq 1$ is equivalent to $-1 - h_{-1}(b) \geq 0$. This is $\lambda(b) \geq 0$. In the first sum, we can use $\tau(b')$ rather than $b'$. Finally we get that 
\[
\dim({\mathfrak p}''_2) = \frac {1}{2}\sum_{\begin{subarray}{c}(b, b')\in {\mathbf B} \times {\mathbf B}\\ \tau(b') \ne b\\ -1 \in Supp(b)\\ 1 \in Supp(b')\\ \lambda(b) \geq \lambda(b')\end{subarray}} c(b) c(b')  + \frac {1}{2}\sum_{\begin{subarray}{c} b \in {\mathbf B} \\ -1 \in Supp(b)\\ \lambda(b) \geq 0\end{subarray}} c(b) (c(b) - 1)
\]

If we are in the symplectic case, in other words when $\epsilon = -1$, then we don't necessarily have that $\xi_{(b, j)}^{(b, j)} = 0$. For $b \in {\mathbf B}$, $-1 \in Supp(b)$ and $1 \leq j \leq c(b)$.  Then we can define the linear transformation $X_{(b, j)}^{(b, j)}:V_{-1} \rightarrow V_1$ as follows:
\[
X_{(b, j)}^{(b,  j)}(v_{-1}^{j''}(b'')) = \begin{cases} v_1^j(\tau(b)), &\text{ if $(b'', j'') = (b, j)$}; \\ \\ 0, &\text{ otherwise.}
\end{cases}
\]
In other words, we have 
\[
\xi_{(b, j)}^{(b'', j'')} = \begin{cases}1, &\text{if $(b'', j'') = (b, j)$;}\\ 0, &\text{otherwise;}\end{cases}
\]
and equation (1) is satisfied. Consequently $X_{(b, j)}^{(b, j)}$ belongs to the $Hom(V_{-1}, V_1)$-component  of  ${\mathfrak g}_2$. 

$X_{(b, j)}^{(b, j)}$ is an eigenvector of $ad(H_{\mathbf c}\vert_{V_{-1}})$ with 
\[
ad(H_{\mathbf c}\vert_{V_{-1}})X_{(b, j)}^{(b, j)} = (h_1(\tau(b)) - h_{-1}(b)) X_{(b, j)}^{(b, j)}. 
\]
As above we can totally order the finite set $\{(b, j)\mid b \in {\mathbf B}, -1 \in Supp(b)\}$. We will write $<$ for this total order.  With this,  we get easily that the linear transformations  $X_{(b, j)}^{(b', j')}$ where $(b, j) < (b', j')$ with the linear transformations $X_{(b, j)}^{(b, j)}$ form a basis of the $Hom(V_{-1}, V_1)$-component  of  ${\mathfrak g}_2$ and, because of our computation of $ad(H_{\mathbf c}\vert_{V_{-1}})$ on these basis elements, we get that

\[
\dim({\mathfrak p}''_2) = \frac {1}{2}\sum_{\begin{subarray}{c}(b, b')\in {\mathbf B} \times {\mathbf B}\\ \tau(b') \ne b\\ -1 \in Supp(b)\\ 1 \in Supp(b')\\ \lambda(b) \geq \lambda(b')\end{subarray}} c(b) c(b')  + \frac {1}{2}\sum_{\begin{subarray}{c} b \in {\mathbf B} \\ -1 \in Supp(b)\\ \lambda(b) \geq 0\end{subarray}} c(b) (c(b) - 1) +  \sum_{\begin{subarray}{c} b \in {\mathbf B}\\ -1 \in Supp(b)\\ h_{\beta}(1) \geq 1\end{subarray}} c(b)
\]
For the first two sums above, it is similar to the orthogonal  case. For the third sum above, we have to noted that $h_1(\tau(b)) - h_{-1}(b) \geq 2$ is equivalent to $-1 - h_{-1}(b) - 1 - h_{-1}(b) \geq 0$. This is simply $\lambda(b) \geq 0$. 

From all the previous computation, the proposition follows.
\end{proof}

We will now describe the  partition  of $\dim(V)$ corresponding to the Jordan type for any element in an orbit ${\mathcal O}_{\mathbf c}$. 

\begin{proposition}\label{P:Height_Jordan_Type}
Let $X$ be an element of the $G^{\iota}$-orbit ${\mathcal O}_{\mathbf c}$ where ${\mathbf c}  \in  {\mathfrak C}_{\delta}$ and let $c:{\mathbf B} \rightarrow {\mathbb N}$ the symmetric function corresponding to ${\mathbf c}$.  Then the partition corresponding to the Jordan type of $X$ is 
\[
\prod_{b \in {\mathbf B}} \vert Supp(b) \vert^{c(b)}.
\]
Here we have described the partition in multiplicative form.
\end{proposition}
\begin{proof}
Taking into account our remark~\ref{R:IsomoVVc}, we can assume that $V = V({\mathbf c})$. If $(E_{\bf c}, H_{\bf c}, F_{\bf c})$ is the Jacobson-Morozov triple corresponding to the orbit ${\mathcal O}_{\mathbf c}$ as in~\ref{N:CCorrespondance}, then we can take $E_{\bf c}$ as a representative for the orbit ${\mathcal O}_{\mathbf c}$. As a representation of the quiver $A_{2m}$, we know that  $E_{\bf c}$ is  a direct sum over $b \in {\mathbf B}$  of $c(b)$ copies of an  indecomposable representation $V(b)$ as in notation~\ref{N:BasisVOrbit} and with the formulae in \ref{SS:EHFDefinition}. Thus from these, we get that there is a unique Jordan bloc for the restriction of $E_{\mathbf c}$ on $V(b)$. In other words, $V(b)$ is a cyclic subspace and the partition of the restriction of $E_{\mathbf c}$ on $V(b)$ is $\vert Supp(b) \vert$. From this, we get that the Jordan type of $E_{\mathbf c}$ is as stated in the proposition. 
\end{proof}

\subsection{}
We want to give now a compact way to describe the symmetric function $c: {\mathbf B} \rightarrow {\mathbb N}$ corresponding to the coefficient function ${\mathbf c} \in  {\mathfrak C}_{\delta}$. 

\begin{definition}
A   {\it  symplectic symmetric ${\mathcal I}$-tableau} is the ${\mathcal I}$-diagram with its boxes filled with elements of ${\mathbb N}$ such that if $c_{i,j}$ denoted the entry in the row indexed by $i$ and in the column indexed by $j$,   then we have $c_{-j,-i} = c_{i, j}$ for all $i, j \in {\mathcal I}$ and $i \geq j$. An  {\it  orthogonal symmetric ${\mathcal I}$-tableau} is  a symplectic symmetric ${\mathcal I}$-Young tableau with the added condition that $c_{i, -i}$ is an even integer for all $i \in {\mathcal I}$, $i > 0$ (i.e. the entries on the principal diagonal are even). 
\end{definition}

\begin{example} \label{ExampleYoungDiagramEvenCase}
For $m = 3$ and ${\mathcal I} = \{5, 3, 1, -1, -3, -5\}$, then 

\begin{center}
${\mathcal T}$ = 
\ytableausetup{centertableaux}
\begin{ytableau}
\none & \none [-5] & \none [-3] & \none [-1] & \none [1] & \none [3] & \none [5]\\ \none [{\phantom -}5] & 0 & 0 & 2 & 0 & 0  & 1 \\ \none [{\phantom -}3] & 0 & 0 & 3 & 4 & 1\\  \none [{\phantom -}1] & 2 & 3 & 3 & 0 \\ \none [-1] & 0 & 4 & 0\\ \none [-3] & 0  & 1\\ \none [-5] & 1 \\
\end{ytableau}
\end{center}
is a symplectic symmetric ${\mathcal I}$-tableau, but not an orthogonal symmetric ${\mathcal I}$-tableau. 

\begin{center}
${\mathcal T}'$ = 
\ytableausetup{centertableaux}
\begin{ytableau}
\none & \none [-5] & \none [-3] & \none [-1] & \none [1] & \none [3] & \none [5]\\ \none [{\phantom -}5] & 2 & 0 & 1 & 2 & 0  & 2 \\ \none [{\phantom -}3] & 0 & 4 & 0 & 3 & 1\\  \none [{\phantom -}1] & 1 & 0 & 0 & 1 \\ \none [-1] & 2 & 3 & 1\\ \none [-3] & 0  & 1\\ \none [-5] & 2 \\
\end{ytableau}
\end{center}
is an  orthogonal symmetric ${\mathcal I}$-tableau. We have kept the indices for the rows and columns  inscribed.
\end{example}

\begin{definition}
Given the symplectic (respectively  orthogonal ) symmetric ${\mathcal I}$-tableau ${\mathcal T} = (c_{i, j})_{i, j \in {\mathcal I}, i \geq j}$, we define its {\it dimension vector} $\delta({\mathcal T})$ as the ${\mathcal I}$-tuple $\delta({\mathcal T}) = (\delta_i)_{i \in {\mathcal I}} \in {\mathbb N}^{\mathcal I}$, where
\[
\delta_i = \sum_{\begin{subarray}{c} u, v \in {\mathcal I}\\ u \geq i \geq v \end{subarray}} c_{u, v} \quad \text{for all $i \in {\mathcal I}$.}
\]
\end{definition}

\begin{example} 
If ${\mathcal T}$ and ${\mathcal T}'$  are the symmetric ${\mathcal I}$-tableaux of example~\ref{ExampleYoungDiagramEvenCase}, then their dimension vectors are  respectively  
\[
\begin{aligned}
\delta({\mathcal T}) &= (\delta_{5}, \delta_3, \delta_1, \delta_{-1}, \delta_{-3},  \delta_{-5}) =  (3, 10, 17, 17, 10, 3) \quad \text{ and }\\
\delta({\mathcal T}') &= (\delta_{5}, \delta_3, \delta_1, \delta_{-1}, \delta_{-3},  \delta_{-5}) =  (7, 13, 14, 14, 13, 7).
\end{aligned}
\]
\end{example}

\begin{lemma}\label{L:RankComponentTableau}
Let $\delta = (\delta_i)_{i \in {\mathcal I}} \in {\mathbb N}^{\mathcal I}$ with $\delta_{-i} = \delta_i$ for all $i \in {\mathcal I}$. Denote by ${\mathfrak T}_{\delta}$: the set of symplectic (respectively  orthogonal ) symmetric ${\mathcal I}$-tableaux ${\mathcal T}$ of dimension $\delta({\mathcal T}) = \delta = (\delta_i)_{i \in {\mathcal I}}$ and by ${\mathfrak R}_{\delta}$: the set of symplectic (respectively orthogonal )  symmetric ${\mathcal I}$-tableaux ${\mathcal R} = (r_{i, j})_{i, j \in I, i \geq j}$ such that 
\[
r_{i, i} = \delta_i \quad \text{ and } \quad (r_{i, j} - r_{i + 2, j} - r_{i , j - 2} + r_{i + 2, j - 2}) \geq 0 \quad \text{ for all $i, j \in {\mathcal I}$ and $i \geq j$}
\]
where we set $r_{i, j} = 0$ if either $i \not\in {\mathcal I}$ or $j \not\in {\mathcal I}$ or both. Then the map $\varTheta:{\mathfrak T}_{\delta} \longrightarrow {\mathfrak R}_{\delta}$
\[
\varTheta({\mathcal T}) = \varTheta((c_{i, j})_{i, j \in {\mathcal I}, i \geq j}) = (r_{i, j})_{i, j \in I, i \geq j} \quad \text{where $r_{i, j} = \sum_{\begin{subarray}{c} u, v \in {\mathcal I}\\ u \geq i \geq j \geq v\end{subarray}} c_{u, v}$}  
\]
for all $i , j \in {\mathcal I}$, $i \geq j$. is a well-defined bijection.
\end{lemma}
\begin{proof}
This is a known result. See for example section 2 in \cite{AD1982}. In fact in \cite{AD1982}, the  ${\mathcal I}$-tableaux are not necessarily symmetric, but it is easy to check that $\varTheta$ sends symmetric  ${\mathcal I}$-tableaux to symmetric ${\mathcal I}$-tableaux. 

In the orthogonal case, we also have to see that the entries $r_{i, -i}$ are even for all $i \in {\mathcal I}$, $i > 0$. But this follows from the fact that 
\[
\begin{aligned}
r_{i, -i} &= \sum_{\begin{subarray}{c}u, v \in {\mathcal I}\\ u \geq i  \geq -i \geq v \end{subarray}} c_{u, v} = \sum_{\begin{subarray}{c} u \in {\mathcal I}\\ u \geq i \end{subarray}} c_{u, -u} + \sum_{\begin{subarray}{c} u, v \in {\mathcal I}\\ u \geq i \geq -i \geq v\\ (u + v) > 0\end{subarray}} c_{u, v} + \sum_{\begin{subarray}{c} u, v \in {\mathcal I}\\ u \geq i \geq -i \geq v\\ 0 > (u + v)\end{subarray}} c_{u, v}\\ &= \sum_{\begin{subarray}{c} u \in {\mathcal I}\\ u \geq i \end{subarray}} c_{u, -u} + \sum_{\begin{subarray}{c} u, v \in {\mathcal I}\\ u \geq i \geq -i \geq v\\ (u + v) > 0 \end{subarray}} c_{u, v} + \sum_{\begin{subarray}{c} u, v \in {\mathcal I}\\ u \geq i \geq -i \geq v\\ 0 > (u + v)\end{subarray}} c_{-v, -u}\\ &=  \sum_{\begin{subarray}{c} u \in {\mathcal I}\\ u \geq i \end{subarray}} c_{u, -u} + 2 \sum_{\begin{subarray}{c} u, v \in {\mathcal I}\\ u \geq i \geq -i \geq v\\ (u + v) > 0\end{subarray}} c_{u, v}  
\end{aligned}
\]
and the fact that the $c_{u, -u}$ are even.  

The inverse map of $\varTheta$ is  $\varTheta^{-1}:{\mathfrak R}_{\delta} \longrightarrow {\mathfrak T}_{\delta}$ is 
\[
 \varTheta^{-1}((r_{i, j})_{i, j \in {\mathcal I}, i \geq j}) = (c_{i, j})_{i, j \in I, i \geq j} \quad \text{where $c_{i, j} =  (r_{i, j} - r_{i + 2, j} - r_{i , j - 2} + r_{i + 2, j - 2})$}
\]
We have also to prove that the entries $c_{i, -i}$ are even for all $i \in {\mathcal I}$, $i > 0$ in the orthogonal case.  This follows from
\[
c_{i, -i} = r_{i, -i} - r_{i+2, -i} - r_{i, -i - 2} + r_{i + 2, -i - 2} = r_{i, -i} - 2r_{i + 2, -i}  + r_{i + 2, -i - 2}
\]
and the fact that both $r_{i, -i}$  and $r_{(i + 2), -(i + 2)}$ are even.
\end{proof}

\begin{example}
If ${\mathcal T}$ and ${\mathcal T}'$  are the symmetric ${\mathcal I}$-tableaux of example~\ref{ExampleYoungDiagramEvenCase}, then 
\begin{center}
$\varTheta({\mathcal T})$ = 
\ytableausetup{centertableaux}
\begin{ytableau}
\none & \none [-5] & \none [-3] & \none [-1] & \none [1] & \none [3] & \none [5]\\ \none [{\phantom -}5] & 0 & 0 & 2 & 2 & 2  & 3 \\ \none [{\phantom -}3] & 0 & 0 & 5 & 9 & 10\\  \none [{\phantom -}1] & 2 & 5 & 13 & 17 \\ \none [-1] & 2 & 9 & 17\\ \none [-3] & 2  & 10\\ \none [-5] & 3 \\
\end{ytableau}
\quad
\text{and}
\quad
$\varTheta({\mathcal T}')$ = 
\ytableausetup{centertableaux}
\begin{ytableau}
\none & \none [-5] & \none [-3] & \none [-1] & \none [1] & \none [3] & \none [5]\\ \none [{\phantom -}5] & 2 & 2 & 3 & 5 & 5  & 7 \\ \none [{\phantom -}3] & 2 & 6 & 7 & 12 & 13\\  \none [{\phantom -}1] & 3 & 7 & 8 & 14 \\ \none [-1] & 5 & 12 & 14\\ \none [-3] & 5  & 13\\ \none [-5] & 7 \\
\end{ytableau}.
\end{center}
We have kept the indices for the rows and columns  inscribed.
\end{example}

\begin{definition}\label{D:SymmetricTableauIevenOrtho}
Let $\delta = (\delta_i)_{i \in {\mathcal I}} \in {\mathbb N}^{\mathcal I}$ with $\delta_{-i} = \delta_i$ for all $i \in {\mathcal I}$ and  the coefficient function ${\mathbf c} \in {\mathfrak C}_{\delta}$.  To ${\mathbf c}$, denote the corresponding symmetric function by $c: {\mathbf B} \rightarrow {\mathbb N}$. We define the symmetric ${\mathcal I}$-tableau ${\mathcal T}({\mathbf c})$ as follows: if $c_{i, j}$ is the entry at the row $i$ and column $j$ with $i, j \in {\mathcal I}$ and $i \geq j$, then 
\[
c_{i, j} = \begin{cases} c(b(i, j, 0)), &\text{ if $i + j \ne 0$;}\\ c(b(i, j, 0)), &\text{ if $i + j = 0$ in the symplectic case;} \\  c(b(i, j, 0)) + c(b(i, j, 1)), &\text{ if $i + j = 0$ in the orthogonal case.} \end{cases}
\]

Because $c: {\mathbf B} \rightarrow {\mathbb N}$ is symmetric,  ${\mathcal T}({\mathbf c})$ is clearly a symmetric ${\mathcal I}$-tableau. In the orthogonal case,   we have that $c(b(i, j, 0)) = c(b(i, j, 1))$ when $i + j = 0$ and consequently $c_{i, -i}$ is even. Thus ${\mathcal T}({\mathbf c})$ is an  orthogonal symmetric ${\mathcal I}$-tableau.  

We have 
\[
\begin{aligned}
 \sum_{\begin{subarray}{c} u, v \in {\mathcal I}\\ u \geq i \geq v \end{subarray}} c_{u, v} &=  \sum_{\begin{subarray}{c} u, v \in {\mathcal I}\\ u \geq i \geq v \\ u + v \ne 0 \end{subarray}} c(b(u, v, 0)) \\ & \quad + \begin{cases} \displaystyle \sum_{\begin{subarray}{c} u \in {\mathcal I}\\ u \geq i \geq -u \end{subarray}} c(b(u, -u, 0)), &\text{in the symplectic case;}\\ \\ \displaystyle \sum_{\begin{subarray}{c} u \in {\mathcal I}\\ u \geq i \geq -u \end{subarray}} c(b(u, -u, 0)) + c(b(u, -u, 1)), &\text{in the orthogonal case;} \end{cases} \\
 &= \sum_{\begin{subarray}{c} b \in {\mathbf B}\\ i \in Supp(b) \end{subarray}} c(b) = \delta_i.
 \end{aligned}
 \] 
Consequently $\delta({\mathcal T}({\mathbf c})) = \delta$ and ${\mathcal T}({\mathbf c}) \in {\mathfrak T}_{\delta}$. 
\end{definition}

\begin{lemma}\label{L:TableauOrbitEven}
Let $\delta = (\delta_i)_{i \in {\mathcal I}} \in {\mathbb N}^{\mathcal I}$ with $\delta_{-i} = \delta_i$ for all $i \in {\mathcal I}$ and let ${\mathfrak g}_2/G^{\iota}$ denotes the set of  $G^{\iota}$-orbits  in ${\mathfrak g}_2$. Then
\begin{enumerate}[\upshape (a)]
\item the map ${\mathcal T}: {\mathfrak C}_{\delta} \rightarrow {\mathfrak T}_{\delta}$, where   ${\mathbf c} \mapsto {\mathcal T}({\mathbf c})$ is a bijection;
\item the map ${\mathfrak g}_2/G^{\iota} \rightarrow  {\mathfrak T}_{\delta} \quad \text{ defined by } \quad G^{\iota} \cdot X \mapsto {\mathcal T}({\mathbf c}_X)$ is a well-defined bijection. 
\end{enumerate}
\end{lemma}
\begin{proof}
(a) is obvious from the definition of ${\mathcal T}$ above and (b) is a corollary of proposition~\ref{Prop_Indexation_Orbits_Even}.
\end{proof}

\begin{notation}
We will denote the $G^{\iota}$-orbit  in ${\mathfrak g}_2$ corresponding to the ${\mathcal I}$-tableau ${\mathcal T}  \in {\mathfrak T}_{\delta}$ as defined in the bijection of lemma~\ref{L:TableauOrbitEven} (b)  by ${\mathcal O}_{\mathcal T}$
\end{notation}

In \cite{BI2021},  M. Boos and G. Cerulli Irelli, were able to describe the partial order on the $G^{\iota}$-orbits in ${\frak g}_2$ for symmetric quivers in the orthogonal and symplectic cases for  symmetric quivers of type $A_m^{even}$ and $A_m^{odd}$. We will now describe their result in our case (i.e. $\vert {\mathcal I} \vert$ even).

\begin{proposition}[Boos and Cerulli Irelli]\label{P:PartialOrderIEven}
(In the situation of \ref{S:SetUpAeven}) Let  ${\mathcal I}$-tableaux ${\mathcal T}$, ${\mathcal T}' \in  {\mathfrak T}_{\delta}$ and  the corresponding $G^{\iota}$-orbits ${\mathcal O}_{\mathcal T}$ , ${\mathcal O}_{{\mathcal T}'}$ in ${\frak g}_2$. Then the  $G^{\iota}$-orbit  ${\mathcal O}_{\mathcal T}$ is contained in the Zariski closure of the   $G^{\iota}$-orbit ${\mathcal O}_{{\mathcal T}'}$ if and only if
\[
(r_{i, j})_{i, j \in {\mathcal I}, i \geq j} = {\mathcal R} = \varTheta({\mathcal T}) \leq \varTheta({\mathcal T}') = {\mathcal R}' = (r'_{i, j})_{i, j \in {\mathcal I}, i \geq j}
\]
where the inequality here between the tableaux ${\mathcal R}$ and ${\mathcal R}'$ means $r_{i, j} \leq r'_{i, j}$ for all $i, j \in {\mathcal I}$, $i \geq j$. 
\end{proposition}
\begin{proof}
First we need to observe that if we consider an element $X$ in the $G^{\iota}$-orbit ${\mathcal O}_{{\mathcal T}''}$ in ${\mathfrak g}_2$, where ${\mathcal T}'' \in {\mathfrak T}_{\delta}$, and, as in part (b) of lemma~\ref{G2AsRepresEven}, we write the linear transformations $X_i:V_i \rightarrow V_{i + 2}$ for $i \in {\mathcal I}$, $i \ne (2m - 1)$ associated to $X$ and that gives us a representation of the quiver $A_{2m}$, then the rank of the composition
\[
X_{i - 2} \circ X_{i- 4} \circ \dots \circ X_{j + 2} \circ X_j:V_j \longrightarrow  V_i \quad \text{for $i, j \in {\mathcal I}$, $i >  j$}
\]
is ${r}_{i, j}''$, where $(r_{i, j}'')_{i, j \in {\mathcal I}, i \geq j} = {\mathcal R}'' = \varTheta({\mathcal T}'')$. This is proved in section 2   in \cite{AD1982} for example.  Note also that ${r}_{i, i}'' = \dim(V_i)$ for all $i \in {\mathcal I}$. 

If we apply this observation for the two orbits  ${\mathcal O}_{\mathcal T}$ and ${\mathcal O}_{{\mathcal T}'}$,  where ${\mathcal O}_{\mathcal T}$ is contained in the Zariski closure $ \overline{{\mathcal O}_{{\mathcal T}'}}$  of the orbit ${\mathcal O}_{{\mathcal T}'}$, then by proposition 2.8 in \cite{AD1982}, it follows that 
\[
(r_{i, j})_{i, j \in {\mathcal I}, i \geq j} = {\mathcal R} = \varTheta({\mathcal T}) \leq \varTheta({\mathcal T}') = {\mathcal R}' = (r'_{i, j})_{i, j \in {\mathcal I}, i \geq j}. 
\]
Proposition 2.8 in \cite{AD1982} states that these ranks cannot increase in a degeneration. 

As for the reciprocal, we must use theorem 7.1 in  \cite{BI2021}.  In our case, this theorem means that the partial order  of symplectic (respectively orthogonal) representations is the same as the partial order induced by forgetting  the symplectic (respectively special orthogonal)  condition and just looking at the representation of the quiver $A_{2m}$. But for the equioriented quiver of type $A_{2m}$, from theorem 5.2 in  \cite{AD1982},  we get that  
\[
{\mathcal R} = \varTheta({\mathcal T}) \leq \varTheta({\mathcal T}') = {\mathcal R}' \quad \Rightarrow \quad {\mathcal O}_{\mathcal T} \subseteq \overline{{\mathcal O}_{{\mathcal T}'}}.
\]
Note that in the article  \cite{BI2021} of M. Boos and G. Cerulli Irelli for orthogonal representations, the isomorphism is the orthogonal isomorphism that is used in the article, but the orthogonal isomorphism is equivalent to special orthogonal isomorphism in the case $A_m^{even}$ as we saw in the proof of proposition~ \ref{Prop_Indexation_Orbits_Even}. For this reason, we can use theorem 7.1 of  \cite{BI2021}.
\end{proof}

\begin{example}\label{E:4.23}
(In the situation of \ref{S:SetUpAeven}) Let $m = 2$, $G$ be the group of automorphisms of the finite dimensional vector space $V$ over ${\mathbf k}$ preserving  a fixed non-degenerate skew-symmetric bilinear  form $\langle\ , \  \rangle:V \times V \rightarrow {\mathbf k}$,   ${\mathcal I} = \{3, 1, -1, -3\}$, $\dim(V_1) = \dim(V_{-1}) = 2$, $\dim(V_3) = \dim(V_{-3}) = 1$. In other words, $\delta_1 = \delta_{-1} = 2$, $\delta_3 = \delta_{-3} = 1$,  $G = Sp_6({\mathbf k})$ and $\iota(t)$ is given on $V_i$ for $i \in {\mathcal I}$ by $\iota(t) v = t^i v$ for all $v \in V_i$ and $t \in {\mathbf k}^{\times}$. In this case,  $\dim({\mathfrak g}_2) = 5$ and there are 8 orbits. We will write for each of these orbits:  the symplectic  symmetric  ${\mathcal I}$-tableau ${\mathcal T} \in {\mathfrak T}_{\delta}$, the ${\mathcal I}$-tableau ${\mathcal R} = \varTheta({\mathcal T})   \in  {\mathfrak R}_{\delta}$ and its dimension. 

\begin{table}[h]
	\begin{center}\renewcommand{\arraystretch}{1.25}
		\begin{tabular} {| l || r  | r  | r  | r  | r |  r | c |}
			\hline
			Orbit &  $10$  & $01_0$ & $01_1$ & $11_0$ & $11_1$ & $12$ & Jordan Decomposition \\ \hline
			${\mathcal O}_5^1$ & $0$ & $0$ & $1$ & $0$ & $1$ & $0$ & $2^1 4^1$\\ \hline
			${\mathcal O}_4^1$ & $0$ & $0$ & $0$ & $0$ & $0$ & $1$ & $3^2$\\ \hline
			${\mathcal O}_4^2$ & $0$ & $1$ & $0$ & $0$ & $1$ & $0$ & $1^2 4$\\ \hline
			${\mathcal O}_3^1$ & $0$ & $0$ & $1$ & $1$ & $0$ & $0$ & $2^3$\\ \hline
			${\mathcal O}_3^2$ & $1$ & $0$ & $2$ & $0$ & $0$ & $0$ & $1^2 2^2$\\ \hline
			${\mathcal O}_2^1$ & $0$ & $1$ & $0$ & $1$ & $0$ & $0$ & $1^2 2^2$\\ \hline
			${\mathcal O}_2^2$ & $1$ & $1$ & $1$ & $0$ & $0$ & $0$ & $1^4 2^1$\\ \hline
			${\mathcal O}_0^1$ & $1$ & $2$ & $0$ & $0$ & $0$ & $0$  & $1^6$\\ \hline
		\end{tabular}
	\end{center}
	\caption{List of the values of coefficient function ${\mathbf c}$.}\label{T:Table1}
\end{table}
The columns of the  table~\ref{T:Table1}  are indexed by the  dimension of the indecomposable symplectic representation of $A_2^{even}$ and Jordan type for each orbit.	
The subscript in the notation for each orbit is its dimension and the upperscript is there to distinguish between the different orbits of the same dimension.
Now  the tableau ${\mathcal T} \in {\mathfrak T}_{\delta}$, the tableau ${\mathcal R} \in  {\mathfrak R}_{\delta}$ for each orbit are
\begin{itemize}
\item For the orbit ${\mathcal O}_5^1$ of dimension 5

\begin{center}
${\mathcal T}_5^1$ = 
\ytableausetup{centertableaux}
\begin{ytableau}
\none & \none [-3] & \none [-1] & \none [1] & \none [3] \\ \none [{\phantom -}3] & 1 & 0 & 0 & 0  \\ \none [{\phantom -}1] & 0 & 1 & 0\\  \none [-1] & 0 & 0 \\ \none [-3] & 0\\ \end{ytableau}
\quad
\text{and}
\quad
${\mathcal R}_5^1$ = 
\ytableausetup{centertableaux}
\begin{ytableau}
\none & \none [-3] & \none [-1] & \none [1] & \none [3] \\ \none [{\phantom -}3] & 1 & 1 & 1 & 1  \\ \none [{\phantom -}1] & 1 & 2 & 2\\  \none [-1] & 1 & 2 \\ \none [-3] & 1\\ 
\end{ytableau}.
\end{center}

\item For the orbit ${\mathcal O}_4^1$ of dimension 4

\begin{center}
${\mathcal T}_4^1$ = 
\ytableausetup{centertableaux}
\begin{ytableau}
\none & \none [-3] & \none [-1] & \none [1] & \none [3] \\ \none [{\phantom -}3] & 0 & 1 & 0 & 0  \\ \none [{\phantom -}1] & 1 & 0 & 0\\  \none [-1] & 0 & 0 \\ \none [-3] & 0\\ \end{ytableau}
\quad
\text{and}
\quad
${\mathcal R}_4^1$ = 
\ytableausetup{centertableaux}
\begin{ytableau}
\none & \none [-3] & \none [-1] & \none [1] & \none [3] \\ \none [{\phantom -}3] & 0 & 1 & 1 & 1  \\ \none [{\phantom -}1] & 1 & 2 & 2\\  \none [-1] & 1 & 2 \\ \none [-3] & 1\\ 
\end{ytableau}.
\end{center}

\item For the orbit ${\mathcal O}_4^2$ of dimension 4

\begin{center}
${\mathcal T}_4^2$ = 
\ytableausetup{centertableaux}
\begin{ytableau}
\none & \none [-3] & \none [-1] & \none [1] & \none [3] \\ \none [{\phantom -}3] & 1 & 0 & 0 & 0  \\ \none [{\phantom -}1] & 0 & 0 & 1\\  \none [-1] & 0 & 1 \\ \none [-3] & 0\\ \end{ytableau}
\quad
\text{and}
\quad
${\mathcal R}_4^2$ = 
\ytableausetup{centertableaux}
\begin{ytableau}
\none & \none [-3] & \none [-1] & \none [1] & \none [3] \\ \none [{\phantom -}3] & 1 & 1 & 1 & 1  \\ \none [{\phantom -}1] & 1 & 1 & 2\\  \none [-1] & 1 & 2 \\ \none [-3] & 1\\ 
\end{ytableau}.
\end{center}

\item For the orbit ${\mathcal O}_3^1$ of dimension 3

\begin{center}
${\mathcal T}_3^1$ = 
\ytableausetup{centertableaux}
\begin{ytableau}
\none & \none [-3] & \none [-1] & \none [1] & \none [3] \\ \none [{\phantom -}3] & 0 & 0 & 1 & 0  \\ \none [{\phantom -}1] & 0 & 1 & 0\\  \none [-1] & 1 & 0 \\ \none [-3] & 0\\ \end{ytableau}
\quad
\text{and}
\quad
${\mathcal R}_3^1$ = 
\ytableausetup{centertableaux}
\begin{ytableau}
\none & \none [-3] & \none [-1] & \none [1] & \none [3] \\ \none [{\phantom -}3] & 0 & 0 & 1 & 1  \\ \none [{\phantom -}1] & 0 & 1 & 2\\  \none [-1] & 1 & 2 \\ \none [-3] & 1\\ 
\end{ytableau}.
\end{center}

\item For the orbit ${\mathcal O}_3^2$ of dimension 3

\begin{center}
${\mathcal T}_3^2$ = 
\ytableausetup{centertableaux}
\begin{ytableau}
\none & \none [-3] & \none [-1] & \none [1] & \none [3] \\ \none [{\phantom -}3] & 0 & 0 & 0 & 1  \\ \none [{\phantom -}1] & 0 & 2 & 0\\  \none [-1] & 0 & 0 \\ \none [-3] & 1\\ \end{ytableau}
\quad
\text{and}
\quad
${\mathcal R}_3^2$ = 
\ytableausetup{centertableaux}
\begin{ytableau}
\none & \none [-3] & \none [-1] & \none [1] & \none [3] \\ \none [{\phantom -}3] & 0 & 0 & 0 & 1  \\ \none [{\phantom -}1] & 0 & 2 & 2\\  \none [-1] & 0 & 2 \\ \none [-3] & 1\\ 
\end{ytableau}.
\end{center}

\item For the orbit ${\mathcal O}_2^1$ of dimension 2

\begin{center}
${\mathcal T}_2^1$ = 
\ytableausetup{centertableaux}
\begin{ytableau}
\none & \none [-3] & \none [-1] & \none [1] & \none [3] \\ \none [{\phantom -}3] & 0 & 0 & 1 & 0  \\ \none [{\phantom -}1] & 0 & 0 & 1\\  \none [-1] & 1 & 1 \\ \none [-3] & 0\\ \end{ytableau}
\quad
\text{and}
\quad
${\mathcal R}_2^1$ = 
\ytableausetup{centertableaux}
\begin{ytableau}
\none & \none [-3] & \none [-1] & \none [1] & \none [3] \\ \none [{\phantom -}3] & 0 & 0 & 1 & 1  \\ \none [{\phantom -}1] & 0 & 0 & 2\\  \none [-1] & 1 & 2 \\ \none [-3] & 1\\ 
\end{ytableau}.
\end{center}

\item For the orbit ${\mathcal O}_2^2$ of dimension 2

\begin{center}
${\mathcal T}_2^2$ = 
\ytableausetup{centertableaux}
\begin{ytableau}
\none & \none [-3] & \none [-1] & \none [1] & \none [3] \\ \none [{\phantom -}3] & 0 & 0 & 0 & 1  \\ \none [{\phantom -}1] & 0 & 1 & 1\\  \none [-1] & 0 & 1 \\ \none [-3] & 1\\ \end{ytableau}
\quad
\text{and}
\quad
${\mathcal R}_2^2$ = 
\ytableausetup{centertableaux}
\begin{ytableau}
\none & \none [-3] & \none [-1] & \none [1] & \none [3] \\ \none [{\phantom -}3] & 0 & 0 & 0 & 1  \\ \none [{\phantom -}1] & 0 & 1 & 2\\  \none [-1] & 0 & 2 \\ \none [-3] & 1\\ 
\end{ytableau}.
\end{center}

\item For the orbit ${\mathcal O}_0^1$ of dimension 0

\begin{center}
${\mathcal T}_0^1$ = 
\ytableausetup{centertableaux}
\begin{ytableau}
\none & \none [-3] & \none [-1] & \none [1] & \none [3] \\ \none [{\phantom -}3] & 0 & 0 & 0 & 1  \\ \none [{\phantom -}1] & 0 & 0 & 2\\  \none [-1] & 0 & 2 \\ \none [-3] & 1\\ \end{ytableau}
\quad
\text{and}
\quad
${\mathcal R}_0^1$ = 
\ytableausetup{centertableaux}
\begin{ytableau}
\none & \none [-3] & \none [-1] & \none [1] & \none [3] \\ \none [{\phantom -}3] & 0 & 0 & 0 & 1  \\ \none [{\phantom -}1] & 0 & 0 & 2\\  \none [-1] & 0 & 2 \\ \none [-3] & 1\\ 
\end{ytableau}.
\end{center}

\end{itemize}

\begin{figure}[p]
\begin{center}
\begin{tikzpicture}
\draw (0, 0) -- (-2, 2);
\draw (0, 0) -- (2, 2);
\draw (-2, 2) -- (-2, 3);
\draw  (2, 2) -- (2, 3);
\draw (2, 2) -- (-2, 3);
\draw (-2, 3) -- (-2, 4);
\draw (-2, 3) -- (2, 4);
\draw (2, 3) -- (2, 4); 
\draw (-2, 4) -- (0, 5);
\draw (2, 4) -- (0, 5);
\draw (0, 0) node[below] {${\mathcal O}_0^1$};
\draw (-2, 2) node[left] {${\mathcal O}_2^1$};
\draw (2, 2) node[right] {${\mathcal O}_2^2$};
\draw (-2, 3) node[left] {${\mathcal O}_3^1$};
\draw (2, 3) node[right] {${\mathcal O}_3^2$};
\draw (-2, 4) node[left] {${\mathcal O}_4^2$};
\draw (2, 4) node[right] {${\mathcal O}_4^1$};
\draw (0, 5) node[above] {${\mathcal O}_5^1$};
\draw (0, 0) node {$\bullet$};
\draw (-2, 2) node {$\bullet$};
\draw (2, 2) node {$\bullet$};
\draw (-2, 3) node {$\bullet$};
\draw (2, 3) node {$\bullet$};
\draw (-2, 4) node {$\bullet$};
\draw (2, 4) node {$\bullet$};
\draw (0, 5) node {$\bullet$};
\end{tikzpicture}
\end{center}
\caption{Poset of orbits obtained by the rank tableaux comparaison}\label{FigurePosetExample4.23}
\end{figure}

We have illustrated in figure~\ref{FigurePosetExample4.23} the Hasse diagram of the poset of orbits obtained by comparing the rank tableaux ${\mathcal R}$ for the set of orbits.
\end{example}

\begin{example}\label{E:4.24}
(In the situation of \ref{S:SetUpAeven})  Let $m = 2$, $G$ be the connected component of the identity of the group of automorphisms of the finite dimensional vector space $V$ over ${\mathbf k}$ preserving  a fixed non-degenerate  symmetric bilinear form $\langle\ , \  \rangle:V \times V \rightarrow {\mathbf k}$,   ${\mathcal I} = \{3, 1, -1, -3\}$, $\dim(V_1) = \dim(V_{-1}) = 3$, $\dim(V_3) = \dim(V_{-3}) = 2$. In other words,  $\delta_1 = \delta_{-1} = 3$,  $\delta_3 = \delta_{-3} = 2$, $G = SO_{10}({\mathbf k})$ and $\iota(t)$ on $V_i$ for $i \in {\mathcal I}$ is given by $\iota(t) v = t^i v$ for all $v \in V_i$ and $t \in {\mathbf k}^{\times}$.  In this case,  $\dim({\mathfrak g}_2) = 9$ and there are 8 orbits. We will write for each of these orbits:  the orthogonal symmetric ${\mathcal I}$-tableau ${\mathcal T} \in {\mathfrak T}_{\delta}$, the ${\mathcal I}$-tableau ${\mathcal R}= \varTheta({\mathcal T}) \in  {\mathfrak R}_{\delta}$ and its dimension. 

\begin{table}[h]
	\begin{center}\renewcommand{\arraystretch}{1.25}
		\begin{tabular} {| l || r  | r  | r  | r  | r |  r | c |}
			\hline
			Orbit &  $10$  & $01$ & $11$ & $02$ & $12$ & $22$ & Jordan Decomposition \\ \hline
			${\mathcal O}_9^1$ & $0$ & $1$ & $0$ & $0$ & $0$ & $1$ & $1^2 4^2$\\ \hline
			${\mathcal O}_8^1$ & $0$ & $0$ & $1$ & $0$ & $1$ & $0$ & $2^2 3^2$\\ \hline
			${\mathcal O}_7^1$ & $1$ & $1$ & $0$ & $0$ & $1$ & $0$ & $1^4 3^2$\\ \hline
			${\mathcal O}_6^1$ & $0$ & $1$ & $2$ & $0$ & $0$ & $0$ & $1^2 2^4$\\ \hline
			${\mathcal O}_5^1$ & $1$ & $0$ & $1$ & $1$ & $0$ & $0$ & $1^2 2^4$\\ \hline
			${\mathcal O}_4^1$ & $1$ & $2$ & $1$ & $0$ & $0$ & $0$ & $1^6 2^2$\\ \hline
			${\mathcal O}_3^1$ & $2$ & $1$ & $0$ & $1$ & $0$ & $0$ & $1^6 2^2$\\ \hline
			${\mathcal O}_0^1$ & $2$ & $3$ & $0$ & $0$ & $0$ & $0$  & $1^{10}$\\ \hline
		\end{tabular}
	\end{center}
	\caption{List of the values of the coefficient function ${\mathbf c}$.}\label{T:Table2} 	
\end{table}
The columns of the  table~\ref{T:Table2}  are indexed by the  dimension of the indecomposable orthogonal representation of $A_2^{even}$ and Jordan type for each orbit.	
The subscript in the notation for each orbit is its dimension and the upperscript is there to distinguish between the different orbits of the same dimension.
Now  the tableau ${\mathcal T} \in {\mathfrak T}_{\delta}$, the tableau ${\mathcal R} \in  {\mathfrak R}_{\delta}$ for each orbit are

\begin{itemize}
\item For the orbit ${\mathcal O}_9^1$ of dimension 9

\begin{center}
${\mathcal T}_9^1$ = 
\ytableausetup{centertableaux}
\begin{ytableau}
\none & \none [-3] & \none [-1] & \none [1] & \none [3] \\ \none [{\phantom -}3] & 2 & 0 & 0 & 0  \\ \none [{\phantom -}1] & 0 & 0 & 1\\  \none [-1] & 0 & 1 \\ \none [-3] & 0\\ \end{ytableau}
\quad
\text{and}
\quad
${\mathcal R}_9^1$ = 
\ytableausetup{centertableaux}
\begin{ytableau}
\none & \none [-3] & \none [-1] & \none [1] & \none [3] \\ \none [{\phantom -}3] & 2 & 2 & 2 & 2  \\ \none [{\phantom -}1] & 2 & 2 & 3\\  \none [-1] & 2 & 3 \\ \none [-3] & 2\\ 
\end{ytableau}.
\end{center}

\item For the orbit ${\mathcal O}_8^1$ of dimension 8

\begin{center}
${\mathcal T}_8^1$ = 
\ytableausetup{centertableaux}
\begin{ytableau}
\none & \none [-3] & \none [-1] & \none [1] & \none [3] \\ \none [{\phantom -}3] & 0 & 1 & 1 & 0  \\ \none [{\phantom -}1] & 1 & 0 & 0\\  \none [-1] & 1 & 0 \\ \none [-3] & 0\\ \end{ytableau}
\quad
\text{and}
\quad
${\mathcal R}_8^1$ = 
\ytableausetup{centertableaux}
\begin{ytableau}
\none & \none [-3] & \none [-1] & \none [1] & \none [3] \\ \none [{\phantom -}3] & 0 & 1 & 2 & 2  \\ \none [{\phantom -}1] & 1 & 2 & 3\\  \none [-1] & 2 & 3 \\ \none [-3] & 2\\ 
\end{ytableau}.
\end{center}

\item For the orbit ${\mathcal O}_7^1$ of dimension 7

\begin{center}
${\mathcal T}_7^1$ = 
\ytableausetup{centertableaux}
\begin{ytableau}
\none & \none [-3] & \none [-1] & \none [1] & \none [3] \\ \none [{\phantom -}3] & 0 & 1 & 0 & 1  \\ \none [{\phantom -}1] & 1 & 0 & 1\\  \none [-1] & 0 & 1 \\ \none [-3] & 1\\ \end{ytableau}
\quad
\text{and}
\quad
${\mathcal R}_7^1$ = 
\ytableausetup{centertableaux}
\begin{ytableau}
\none & \none [-3] & \none [-1] & \none [1] & \none [3] \\ \none [{\phantom -}3] & 0 & 1 & 1 & 2  \\ \none [{\phantom -}1] & 1 & 2 & 3\\  \none [-1] & 1 & 3 \\ \none [-3] & 2\\ 
\end{ytableau}.
\end{center}

\item For the orbit ${\mathcal O}_6^1$ of dimension 6

\begin{center}
${\mathcal T}_6^1$ = 
\ytableausetup{centertableaux}
\begin{ytableau}
\none & \none [-3] & \none [-1] & \none [1] & \none [3] \\ \none [{\phantom -}3] & 0 & 0 & 2 & 0  \\ \none [{\phantom -}1] & 0 & 0 & 1\\  \none [-1] & 2 & 1 \\ \none [-3] & 0\\ \end{ytableau}
\quad
\text{and}
\quad
${\mathcal R}_6^1$ = 
\ytableausetup{centertableaux}
\begin{ytableau}
\none & \none [-3] & \none [-1] & \none [1] & \none [3] \\ \none [{\phantom -}3] & 0 & 0 & 2 & 2  \\ \none [{\phantom -}1] & 0 & 0 & 3\\  \none [-1] & 2 & 3 \\ \none [-3] & 2\\ 
\end{ytableau}.
\end{center}

\item For the orbit ${\mathcal O}_5^1$ of dimension 5

\begin{center}
${\mathcal T}_5^1$ = 
\ytableausetup{centertableaux}
\begin{ytableau}
\none & \none [-3] & \none [-1] & \none [1] & \none [3] \\ \none [{\phantom -}3] & 0 & 0 & 1 & 1  \\ \none [{\phantom -}1] & 0 & 2 & 0\\  \none [-1] & 1 & 0 \\ \none [-3] & 1\\ \end{ytableau}
\quad
\text{and}
\quad
${\mathcal R}_5^1$ = 
\ytableausetup{centertableaux}
\begin{ytableau}
\none & \none [-3] & \none [-1] & \none [1] & \none [3] \\ \none [{\phantom -}3] & 0 & 0 & 1 & 2  \\ \none [{\phantom -}1] & 0 & 2 & 3\\  \none [-1] & 1 & 3 \\ \none [-3] & 2\\ 
\end{ytableau}.
\end{center}

\item For the orbit ${\mathcal O}_4^1$ of dimension 4

\begin{center}
${\mathcal T}_4^1$ = 
\ytableausetup{centertableaux}
\begin{ytableau}
\none & \none [-3] & \none [-1] & \none [1] & \none [3] \\ \none [{\phantom -}3] & 0 & 0 & 1 & 1  \\ \none [{\phantom -}1] & 0 & 0 & 2\\  \none [-1] & 1 & 2 \\ \none [-3] & 1\\ \end{ytableau}
\quad
\text{and}
\quad
${\mathcal R}_4^1$ = 
\ytableausetup{centertableaux}
\begin{ytableau}
\none & \none [-3] & \none [-1] & \none [1] & \none [3] \\ \none [{\phantom -}3] & 0 & 0 & 1 & 2  \\ \none [{\phantom -}1] & 0 & 0 & 3\\  \none [-1] & 1 & 3 \\ \none [-3] & 2\\ 
\end{ytableau}.
\end{center}

\item For the orbit ${\mathcal O}_3^1$ of dimension 3

\begin{center}
${\mathcal T}_3^1$ = 
\ytableausetup{centertableaux}
\begin{ytableau}
\none & \none [-3] & \none [-1] & \none [1] & \none [3] \\ \none [{\phantom -}3] & 0 & 0 & 0 & 2  \\ \none [{\phantom -}1] & 0 & 2 & 1\\  \none [-1] & 0 & 1 \\ \none [-3] & 2\\ \end{ytableau}
\quad
\text{and}
\quad
${\mathcal R}_3^1$ = 
\ytableausetup{centertableaux}
\begin{ytableau}
\none & \none [-3] & \none [-1] & \none [1] & \none [3] \\ \none [{\phantom -}3] & 0 & 0 & 0 & 2  \\ \none [{\phantom -}1] & 0 & 2 & 3\\  \none [-1] & 0 & 3 \\ \none [-3] & 2\\ 
\end{ytableau}.
\end{center}

\item For the orbit ${\mathcal O}_0^1$ of dimension 0

\begin{center}
${\mathcal T}_0^1$ = 
\ytableausetup{centertableaux}
\begin{ytableau}
\none & \none [-3] & \none [-1] & \none [1] & \none [3] \\ \none [{\phantom -}3] & 0 & 0 & 0 & 2  \\ \none [{\phantom -}1] & 0 & 0 & 3\\  \none [-1] & 0 & 3 \\ \none [-3] & 2\\ \end{ytableau}
\quad
\text{and}
\quad
${\mathcal R}_0^1$ = 
\ytableausetup{centertableaux}
\begin{ytableau}
\none & \none [-3] & \none [-1] & \none [1] & \none [3] \\ \none [{\phantom -}3] & 0 & 0 & 0 & 2  \\ \none [{\phantom -}1] & 0 & 0 & 3\\  \none [-1] & 0 & 3 \\ \none [-3] & 2\\ 
\end{ytableau}.
\end{center}

\end{itemize}

\begin{figure}[p]
\begin{center}
\begin{tikzpicture}
\draw (0, 0) -- (-2, 4);
\draw (0, 0) -- (2, 3);
\draw (2, 3) -- (2, 5);
\draw  (-2, 4) -- (2, 5);
\draw (-2, 4) -- (-2, 6);
\draw (2, 5) -- (2, 7);
\draw (2, 7) -- (0, 8);
\draw (-2, 6) -- (0, 8);
\draw (0, 8) -- (0, 9); 
\draw (0, 0) node[below] {${\mathcal O}_0^1$};
\draw (-2, 4) node[left] {${\mathcal O}_4^1$};
\draw (2, 5) node[right] {${\mathcal O}_5^1$};
\draw (-2, 6) node[left] {${\mathcal O}_6^1$};
\draw (2, 7) node[right] {${\mathcal O}_7^1$};
\draw (0, 8) node[right] {\hskip 0.1in ${\mathcal O}_8^1$ };
\draw (2, 3) node[right] {${\mathcal O}_3^1$};
\draw (0, 9) node[above] {${\mathcal O}_9^1$};
\draw (0, 0) node {$\bullet$};
\draw (-2, 4) node {$\bullet$};
\draw (2, 3) node {$\bullet$};
\draw (-2, 6) node {$\bullet$};
\draw (2, 7) node {$\bullet$};
\draw (0, 8) node {$\bullet$};
\draw (2, 5) node {$\bullet$};
\draw (0, 9) node {$\bullet$};
\end{tikzpicture}
\end{center}
\caption{Poset of orbits obtained by the rank tableaux comparaison}\label{FigurePosetExample4.24}
\end{figure}

We have illustrated in figure~\ref{FigurePosetExample4.24} the Hasse diagram of the poset of orbits obtained by comparing the rank tableaux ${\mathcal R}$ for the set of orbits  
\end{example}

\section{$G^{\iota}$-orbits in ${\mathfrak g}_2$ when $\vert {\mathcal I} \vert$ is odd.}

\subsection{}
In this section, we will restrict ourself to the symplectic and orthogonal cases when $\vert {\mathcal I} \vert$ is odd. It is simpler and more similar to what we did in the previous section. The special orthogonal case will be treated in the next section. The reason for this is that special orthogonal isomorphism is more restricted that orthogonal isomorphism as we saw in \ref{SS:ExampleIsomorphismSpecialOrtho}.

\subsection{}\label{S:SetUpOdd}
In this section,  
\begin{itemize}
\item $m$ is an integer $> 0$;
\item $G$ is the group of automorphisms of a finite dimensional vector space $V$ over ${\mathbf k}$ preserving a fixed non-degenerate  symmetric (respectively skew-symmetric) form $\langle\ , \  \rangle: V \times V \rightarrow {\mathbf k}$, in other words $G = O(V)$ (respectively $G = Sp(V)$); 
\item $\mathfrak g$ is the Lie algebra of $G$, more precisely $X \in {\mathfrak g}$ if and only if $X:V \rightarrow V$ is an endomorphism of $V$ such that $\langle X(v_1), v_2\rangle + \langle v_1 , X(v_2) \rangle = 0$ for all $v_1, v_2 \in V$;
\item $V$ is identified to its dual $V^*$ by the isomorphism $\phi: V \rightarrow V^*$ defined by $v \mapsto \phi_v$, where $\phi_v: V \rightarrow {\mathbf k}$ is $\phi_v(x) = \langle v, x\rangle$ for all $x \in V$;  
\item ${\mathcal I} = \{n \in {\mathbb Z} \mid n \equiv 0 \pmod 2, -2m <  n  < 2m \}$;
\item $\oplus_{i \in {\mathcal I}} V_i$ is a direct sum decomposition of $V$ by  vector subspaces $V_i \ne 0$ such that 
\begin{itemize} \item $\langle v, v'\rangle = 0$ whenever $v \in V_i$, $v' \in V_j$ and $i + j \ne 0$;
			\item $V_{-i} = V_i^*$ for all $i \in {\mathcal I}$ under the identification given by $\phi$ above.
\end{itemize} 
Note that the restriction of the non-degenerate symmetric (respectively skew-symmetric) form bilinear $\langle\ , \ \rangle$ to $V_0 \times V_0$ is a non-degenerate symmetric (respectively skew-symmetric)  form.
\item ${\mathcal B} = \coprod_{i\in {\mathcal I}} {\mathcal B}_i$ is a basis of $V$ such that each ${\mathcal B}_i = \{u_{i, j}  \mid 1 \leq j \leq \delta_i\}$  is a basis of $V_i$ for each $i \in {\mathcal I}$ and $i \ne 0$, where $\delta_i$ is the dimension of $V_i$, and ${\mathcal B}_0 = \{u_{0, j}\mid 1 \leq j \leq \delta_0\}$ (respectively ${\mathcal B}_0 = \{u_{0, j}, u_{0, -j} \mid 1 \leq j \leq \delta'_0\}$) is a basis of $V_0$, where $\delta_0$ (respectively $\delta_0 = 2\delta'_0$) is the dimension of $V_0$  and these bases are such that, for $i, j \in {\mathcal I}$, $i \ne 0$, $j \ne 0$, $1 \leq r \leq \delta_i$ and $1 \leq s \leq \delta_j$, we have
\[
\langle u_{i, r}, u_{j, s}\rangle = \begin{cases}  1, &\text{if $i + j = 0$, $r = s$ and $i > 0$;}\\ \epsilon, &\text{if $i + j = 0$, $r = s$ and $i< 0$;} \\ 0, &\text{otherwise;} \end{cases}
\]
while for $1 \leq r, s \leq \delta_0$, we have 
\[
\langle u_{0, r}, u_{0, s}\rangle = \begin{cases} 1, &\text{if $r + s = \delta_0 + 1$;}\\ 0, &\text{otherwise;}\end{cases}
\]
(respectively for $1 \leq \vert r \vert \leq \delta'_0$ and $1 \leq \vert s \vert \leq \delta'_0$, we have
\[
\langle u_{0, r}, u_{0, s}\rangle = \begin{cases} 1, &\text{if $r + s = 0$ and $r> 0$;}\\ \epsilon, &\text{if $r + s = 0$ and $r < 0$;} \\ 0, &\text{otherwise;)}\end{cases}
\]
and, if $i \ne 0$, $1 \leq s \leq \delta_i$, $1 \leq r \leq \delta_0$(respectively $1 \leq \vert r \vert \leq \delta'_0$), we have 
\[
\langle u_{0, r}, u_{i, s}\rangle = \langle  u_{i, s},  u_{0, r}\rangle = 0.
\]
\item $\iota: {\mathbf k}^{\times} \rightarrow G$ is the homomorphism defined by $\iota(t) v = t^i v$ for all $i \in {\mathcal I}$, $v \in V_i$ and $t \in {\mathbf k}^{\times}$;
\item ${\mathbf B}$ is the set of ${\mathcal I}$-boxes;
\item ${\mathbf B}/\tau$ denote the set of $\langle \tau \rangle$-orbits ${\mathcal O}$ in ${\mathbf B}$; 
\item $\nu = \vert {\mathbf B}/\tau \vert = m^2$;
\item In the symplectic case, $V_0^+$ denotes the subspace of $V_0$ generated  by  ${\mathcal B}_0^+ = \{u_{0, j} \mid 1 \leq j \leq \delta'_0 \}$ and $V_0^-$ denotes the subspace of $V_0$ generated by  ${\mathcal B}_0^- = \{u_{0, -j} \mid 1 \leq j \leq \delta'_0 \}$. Obviously ${\mathcal B}_0^+ $ and ${\mathcal B}_0^- $ are respectively bases of $V_0^+$ and  $V_0^-$
\end{itemize}

\begin{lemma}\label{G2AsRepresOdd}
\begin{enumerate}[\upshape (a)]

\item $g \in G^{\iota}$ if and only if all the following conditions are verified:
\begin{itemize}
\item $g(V_i) \subseteq V_i$   for all $i  \in {\mathcal I}$;
\item the restriction $g_{\vert_{V_i}} = g_i$ of $g$ to $V_i$  belongs to $GL(V_i)$ for all $i \in {\mathcal I}$;
\item $(g_{-i})^* g_i = Id_{V_i}$ for all $i \in {\mathcal I}$, $i \ne 0$;
\item $g_0 \in O(V_0)$ (respectively $g_0 \in Sp(V_0)$) relative to the non-degenerate bilinear form $\langle\  , \  \rangle$ restricted to $V_0 \times V_0$.
\end{itemize}
Here $g_{-i}^*: V_i \rightarrow V_i$ is defined using the identification of $V_i$ with $V_{-i}^*$ given by $\phi$. So $\langle g_{-i}^*(u), v \rangle = \langle u, g_{-i}(v)\rangle$ for all $u \in V_i$ and $v \in V_{-i}$.

\item $X \in {\mathfrak g}_2$  if and only if all the following conditions are verified:
\begin{itemize}
\item $X(V_i) \subseteq V_{i + 2}$ for all $i \in {\mathcal I}$, $i \ne (2m - 2)$;
\item $X_{-i} = -X_{i - 2}^T$ for all $i \in {\mathcal I}$  and $i \geq 4$;
\item $X_{-2} = - JX^T_0$ where $J$ is the $\delta_0 \times \delta_0$-matrix equal to 
\[
J = \begin{pmatrix} 0 & 0 & \dots & 0 & 1\\ 0 & 0 & \dots & 1 & 0 \\ \vdots & \vdots & \udots & \vdots & \vdots\\ 0 & 1 &\dots & 0 & 0\\ 1 & 0 & \dots & 0 & 0 \end{pmatrix};
\] 
\[
\left(\text{respectively  }
X_0 = (X_0^+ ,  X_0^-) \text{  and   } X_{-2} = \begin{pmatrix} X_{-2}^+ \\ X_{-2}^- \end{pmatrix} = \begin{pmatrix}{\phantom -} (X_0^-)^T \\  -(X_0^+)^T \end{pmatrix} \right),
\]
\end{itemize}
where $X_i$ is the matrix of the linear transformation $X\vert_{V_i}: V_i \rightarrow V_{i + 2}$ of the restriction of $X$ to $V_i$ relative to the bases ${\mathcal B}_i$ and ${\mathcal B}_{i + 2}$ for all $i \in {\mathcal I}$ and $i \ne (2m - 2)$,  $X_0^+$ is the matrix of the restriction of $X_0$ to $V_0^+$ relative to the bases ${\mathcal B}_0^+$ and ${\mathcal B}_{2}$ , $X_0^-$ is the matrix of the restriction of $X_0$ to $V_0^-$  relative to the bases ${\mathcal B}_0^-$ and ${\mathcal B}_{2}$, $X_{-2}^+$ is the matrix of the linear transformation obtained from $X_{-2}$ followed by the projection of $V_0^+ \oplus V_0^-$ to $V_0^+$  relative to the bases ${\mathcal B}_{-2}$ and ${\mathcal B}_{0}^+$ and $X_{-2}^-$ is the matrix of the linear transformation obtained from $X_{-2}$ followed by the projection of $V_0^+ \oplus V_0^-$ to $V_0^-$  relative to the bases ${\mathcal B}_{-2}$ and ${\mathcal B}_{0}^-$. Above $M^T$ denotes the transpose of the matrix $M$. 

Moreover the subset of elements 
\[
(X_i)_{\begin{subarray}{l} i \in {\mathcal I}\\ i \ne (2m - 2)\end{subarray}} \in \left(\bigoplus_{\begin{subarray}{c} i \in {\mathcal I}\\ i \ne (2m - 2)\end{subarray}} Hom(V_i, V_{i + 2})\right),
\]
  such that the above conditions for the $X_i$ are satisfied,  is a subspace of 
  \[
  \bigoplus_{\begin{subarray}{c} i \in {\mathcal I}\\ i \ne (2m - 2)\end{subarray}} Hom(V_i, V_{i + 2})
  \]
   isomorphic to ${\mathfrak g}_2$.

\item $Z \in {\mathfrak g}_0$ if and only if all the following conditions are verified
\begin{itemize}
\item $Z(V_i) \subseteq V_{i}$ for all $i \in {\mathcal I}$;
\item $Z_{-i} = -Z_{i}^T$ for all $i \in {\mathcal I}$ and $i \geq 2$;
\item  $Z_0$ is such that $\langle Z_0(v), v'\rangle + \langle v, Z_0(v') \rangle = 0$ for all $v, v' \in V_0$; in other words,  $Z_0$ is in the Lie algebra corresponding to $V_0$ and the bilinear form $\langle\ , \  \rangle$;
\end{itemize}
where we denote  $Z_i$ is the matrix of the linear transformation $Z\vert_{V_i}: V_{i} \rightarrow V_{i}$ the restriction of $Z$ to $V_i$ relative to the bases ${\mathcal B}_i$ and ${\mathcal B}_{i}$ for all $i \in {\mathcal I}$. 

Moreover the subset of elements 
\[
(Z_i)_{i \in {\mathcal I}} \in \left(\bigoplus_{ i \in {\mathcal I}} Hom(V_i, V_{i})\right),
\]
 such that the above conditions for the $Z_i$ are satisfied,  is a subspace of 
 \[
 \bigoplus_{ i \in {\mathcal I}} Hom(V_i, V_{i})
 \]
  isomorphic to ${\mathfrak g}_{0}$.
 

\item $\{(g_i)_{i \in {\mathcal I}} \in \prod_{i \in {\mathcal I}} GL(V_i) \mid g_{-i} = (g_i^*)^{-1} \text{ for all $i > 0$ and } g_0 \in G_0\}$ is a closed subgroup of  $\prod_{i \in {\mathcal I}} GL(V_i)$ and is isomorphic to $\prod_{i \in {\mathcal I}, i > 0} GL(V_i) \times G_0$,  where $G_0$  is the orthogonal group $O(V_0)$ (respectively the symplectic group $Sp(V_0)$). 

Moreover the function 
\[
\Phi: G^{\iota} \rightarrow \left. \left\{(g_i)_{i \in {\mathcal I}} \in \prod_{i \in {\mathcal I}} GL(V_i)\   \right\vert g_{-i} = (g_i^*)^{-1} \text{ for all $i > 0$ and } g_0 \in G_0 \right\}
\]
defined by $\Phi(g) = (g_{i})_{i \in {\mathcal I}}$  is a well-defined  isomorphism of groups.


\item The restriction of the adjoint action of $G$ to $G^{\iota}$ acts on ${\mathfrak g}_2$ and,  under the description of ${\mathfrak g}_2$ as a subspace of $\oplus_{i \in {\mathcal I}, i \ne (2m - 2)} Hom(V_i, V_{i + 2})$ in (b), this action  is given by 
\[
Ad(g)(X) \mapsto (g_{i + 2}X_{i}g_{i}^{-1})_{i \in {\mathcal I}, i \ne (2m - 2)}
\]
for all $g \in G^{\iota}$ and where $\Phi(g) = (g_{i})_{i \in {\mathcal I}}$. 
\end{enumerate}
\end{lemma}

\begin{proof}
(a) $\Rightarrow$) For  $i \in {\mathcal I}$, $V_i$ is the eigenspace of $\iota(t)$ with eigenvalue $t^i$ for all $t \in {\mathbf k}^{\times}$. Moreover  when $i, j \in {\mathcal I}$ and $i \ne j$,  these eigenvalues $t^i$ and $t^j$ are distinct.  Because $g \iota(t) = \iota(t) g$ for all $t \in {\mathbf k}^{\times}$ when $g \in G^{\iota}$, then $g$ preserves these eigenspaces and we get $g(V_i) \subseteq V_i$ for all $i \in I$. $g_i \in GL(V_i)$ follows easily  because $g \in GL(V)$ and $g(V_i) \subseteq V_i$.

Because we have $\langle g(v), g(v') \rangle = \langle v, v' \rangle$ for all $v, v' \in V$ and our conditions on the direct sum decomposition, we get for all $v \in V_i$ and $v' \in V_{-i}$ that 
\[
\langle g_i(v), g_{-i}(v')\rangle = \langle g_{-i}^* g_i(v), v'\rangle = \langle v, v'\rangle  \Rightarrow  (g_{-i}^*) g_i = Id_{V_i} 
\]
 for all $i \in I$, $ i \ne 0$.  For $i = 0$, we get that $\langle g_{0}(v), g_{0}(v')\rangle = \langle v, v'\rangle$ for all $v, v' \in V_0$.
 Consequently  $g_0 \in O(V_0)$ in the orthogonal case (respectively $g_0 \in Sp(V_0)$ in the symplectic case) relative to the non-degenerate bilinear form $\langle\  , \  \rangle$ restricted to $V_0 \times V_0$.
  
 $\Leftarrow$) If $g(V_i) \subseteq V_i$, $g_i \in GL(V_i)$, $(g_{-i})^* g_i = Id_{V_i}$ for all $i \in {\mathcal I}$, $i \ne 0$ and $g_0 \in O(V_0)$ in the orthogonal case  (respectively $g_0 \in Sp(V_0)$ in the symplectic case), then $g$ is an automorphism. Because of the condition on our direct sum decomposition of $V$,  we get easily that $g$ preserves the non-degenerate bilinear form $\langle\ , \ \rangle$, in other words $g \in G$. Because $g(V_i) \subseteq V_i$, we get that $g$ commutes with $\iota(t)$ for all $t \in {\mathbf k}^{\times}$ and $g \in G^{\iota}$.  (a) is proved.

(b) $\Rightarrow$)  Because $Ad(\iota(t)) X = t^2 X$ for all $t \in {\mathbf k}^{\times}$ when $X \in {\mathfrak g}_2$, we have $\iota(t) X = t^2 X \iota(t)$ for all $t \in {\mathbf k}^{\times}$.  So if $v \in V_i$, then $\iota(t) X(v) =  t^2 X \iota(t) (v) = t^{i + 2} X(v)$ and consequently $X(V_i) \in V_{i + 2}$ for all $i \in {\mathcal I}$, $i \ne (2m - 2)$. Note that this last equation means also that $X(V_{2m - 2}) = 0$

Note that if $i \in {\mathcal I}$, then $-i \ne i - 2$, because $-i = i -2$ we would have $i = 1$ and the elements of ${\mathcal I}$ are even numbers.  Because $\langle X(v), v' \rangle + \langle v, X(v')\rangle = 0$ for all $v, v' \in V$ and our conditions on the direct sum decomposition, we get for $i \in {\mathcal I}$, $i \geq 2$ and for all $v \in V_{-i}$,  $v' \in V_{i - 2}$ that  
\[
\langle X_{-i}(v), v'\rangle + \langle v, X_{i - 2} (v')\rangle = 0.
\]
In this last equation, if we substitute 
\[
X_{-i}(u_{-i, s}) = \sum_{q' = 1}^{\delta_{i - 2}} \xi_{q', s}^{(-i)} u_{(-i + 2), q'} \quad \text{ and } \quad X_{(i - 2)}(u_{(i - 2), r}) = \sum_{q = 1}^{\delta_{i}} \xi_{q, r}^{(i - 2)} u_{i, q},
\]
we get for $i \geq 4$ that $\epsilon \xi_{rs}^{(-i)} + \epsilon \xi_{sr}^{(i - 2)} = 0$ for all $1 \leq s \leq \delta_i$ and $1 \leq r \leq \delta_{i - 2}$ and thus $X_{-i} = -X_{i - 2}^T $. 

If $i = 2$ and we are in the orthogonal case, from $\langle X_{-2}(v), v'\rangle + \langle v, X_0(v')\rangle = 0$ for all $v \in V_{-2}$, $v' \in V_0$ and writing 
\[
X_{-2}(u_{-2, s}) = \sum_{q' = 1}^{\delta_{0}} \xi_{q',s}^{(-2)} u_{0, q'}  \quad \text{ and } \quad X_{0}(u_{0, r}) = \sum_{q = 1}^{\delta_{2}} \xi_{q, r}^{(0)} u_{2, q},
\]
we get that $\xi_{(\delta_0 + 1 - r), s}^{(-2)} + \xi_{s, r}^{(0)} = 0$ for all $1 \leq r \leq \delta_0$, $1 \leq s \leq \delta_2$ and $X_{-2} = -J X_0^T$.

If $i = 2$ and we are in the symplectic case,  from $\langle X_{-2}(v), v'\rangle + \langle v, X_{0} (v')\rangle = 0$ for all $v \in V_{-2}$,  $v' \in V_{0}$ and writing
\[
X_{-2}(u_{-2, s}) = \sum_{q' = 1}^{\delta'_{0}} \xi_{-q',s}^{(-2)} u_{0, -q'}  +  \sum_{q'' = 1}^{\delta'_{0}} \xi_{q'', s}^{(-2)} u_{0, q''} \quad \text{ and } \quad X_{0}(u_{0, r}) = \sum_{q = 1}^{\delta_{2}} \xi_{q, r}^{(0)} u_{2, q},
\]
for $s = 1, 2, \dots, \delta_2$ and $r = -\delta'_0, \dots, -2, -1, 1, 2, \dots, \delta'_0$, we get that  
\[
\begin{cases} \xi_{-r, s}^{(-2)} -  \xi_{s, r}^{(0)} = 0 &\text{if  $r < 0$;}\\ \\ -\xi_{-r, s}^{(-2)} -  \xi_{s, r}^{(0)} = 0 &\text{if  $r > 0$.}  \end{cases}
\] 

Consequently
\[
X_{-2} =  \begin{pmatrix} X_{-2}^+\\ X_{-2}^- \end{pmatrix} = \begin{pmatrix} \phantom -(X_0^-)^T\\ -(X_0^+)^T \end{pmatrix} \text{ where }  X_0 = (X_0^+ ,  X_0^-).
\]

$\Leftarrow$)  If $X_{-i} = -X_{i  - 2}^T$ for all $i \in {\mathcal I}$, $i \geq 4$,  and in the orthogonal case, we have $X_{-2} = -X_0^T$; while in the symplectic case, we have 
\[
X_0 = (X_0^+ ,  X_0^-) \quad \text{  and } X_{-2} =  \begin{pmatrix} X_{-2}^+\\ X_{-2}^- \end{pmatrix} = \begin{pmatrix} \phantom -(X_0^-)^T\\ -(X_0^+)^T \end{pmatrix},
\]
with the above notation,  then we get that $\langle X(v), v' \rangle + \langle v, X(v')\rangle = 0$ for all $v, v' \in V$ by only checking it on the basis ${\mathcal B}$  and this means that $X \in {\mathfrak g}$. If we add the condition  $X(V_i) \subseteq V_{i + 2}$ for all $i \in {\mathcal I}$, $i \ne (2m - 2)$, we have  that $X \in {\mathfrak g}_2$. The rest of the statements in (b) are easily proved. 

(c) $\Rightarrow$)  Because $Ad(\iota(t)) Z = Z$ for all $t \in {\mathbf k}^{\times}$ when $Z \in {\mathfrak g}_0$, we have $\iota(t) Z =  Z \iota(t)$ for all $t \in {\mathbf k}^{\times}$.  So if $v \in V_i$, then $\iota(t) Z(v) = Z \iota(t) (v) = t^{i} Z(v)$ and consequently $Z(V_i) \in V_{i}$ for all $i \in {\mathcal I}$. 

Because $\langle Z(v), v' \rangle + \langle v, Z(v')\rangle = 0$ for all $v, v' \in V$ and our conditions on the direct sum decomposition, we get for $i \in {\mathcal I}$ and for all $v \in V_{i}$,  $v' \in V_{-i}$ that  
\[
\langle Z_{i}(v), v'\rangle + \langle v, Z_{-i} (v')\rangle = 0.
\]
In this last equation, if we substitute 
\[
Z_{i}(u_{i, s}) = \sum_{q' = 1}^{\delta_{i}} \zeta_{q', s}^{(i)} u_{i, q'} \quad \text{ and } \quad Z_{-i}(u_{-i, r}) = \sum_{q = 1}^{\delta_{i}} \zeta_{q, r}^{(-i)} u_{-i, q},
\]
we get for $i \geq 2$ that $ \zeta_{r,s}^{(i)} +  \zeta_{s,r}^{(-i)} = 0$ for all $1 \leq s, r \leq \delta_i$. Thus $Z_{-i} = -Z_{i}^T $ for all $i \in {\mathcal I}$ and $i \geq 2$.

For $i = 0$, we get that 
\[
\langle Z_{0}(v), v'\rangle + \langle v, Z_{0} (v')\rangle = 0 \quad \text{ for all $v, v' \in V_0$,}
\]
in other words, $Z_0$  is in the Lie algebra corresponding to $V_0$ and the bilinear form $\langle \  , \  \rangle$.

$\Leftarrow$)  If $Z_{-i} = -Z_{i}^T$ for all $i \in {\mathcal I}$,  $i \geq 2$ and if $Z_0$  is in the Lie algebra corresponding to $V_0$ and the bilinear form $\langle \  , \  \rangle$, then we get that $\langle Z(v), v' \rangle + \langle v, Z(v')\rangle = 0$ for all $v, v' \in V$ by only checking it on the basis ${\mathcal B}$  and this means that $Z \in {\mathfrak g}$. If we add the condition  $Z(V_i) \subseteq V_{i}$ for all $i \in {\mathcal I}$, we have  that $Z \in {\mathfrak g}_0$. The rest of the statements in (c) are easily proved. 

(d) follows easily from (a).

(e) follows easily from (b) and the fact that $Ad(g)(X) = gXg^{-1}$ for all $g \in G^{\iota}$. 
\end{proof}

\begin{corollary} The dimension $\dim({\mathfrak g}_2)$ of ${\mathfrak g}_2$  is equal to
\[
 \sum_{i \in {\mathcal I}, 0 \leq i < (2m - 2)} \delta_i \delta_{i + 2}
\]
\end{corollary}
\begin{proof}
This follows easily from proposition~\ref{G2AsRepresOdd} (b).
\end{proof}

\subsection{}
In (b) and (e) above, we have shown that ${\mathfrak g}_2$ with the action of the restriction of the adjoint action of $G$ to $G^{\iota}$ is isomorphic to the space of orthogonal (respectively symplectic) representations of the symmetric quiver of type $A_{m}^{odd}$. 

\subsection{}
The notions of symmetric function $c : {\mathbf B} \rightarrow {\mathbb N}$, of associated coefficient function ${\mathbf c}:{\mathbf B}/\tau \rightarrow {\mathbb N}$ and dimension vector $\delta({\mathbf c}) = \delta(c)$ for ${\mathbf c}$ and $c$ are identical to the ones when ${\mathcal I}$ is even as in definition~\ref{D:CoefficientFunction}.

${\mathfrak C}$ and ${\mathfrak C}_{\delta}$ for $\delta = (\delta_i)_{i \in {\mathcal I}} \in {\mathbb N}^{\mathcal I}$ such that $\delta_i = \delta_{-i}$ for all $i \in {\mathcal I}$ denote respectively sets of coefficient functions as in notation~\ref{N:CoefficientFunctionSet}. 

\begin{notation}\label{N:CCorrespondanceOdd}
Let $\delta = (\delta_i)_{i \in {\mathcal I}} \in {\mathbb N}^{\mathcal I}$ be a vector such that $\delta_{-i} = \delta_i$ for all $i \in {\mathcal I}$ and let  ${\mathbf c} \in  {\mathfrak C}_{\delta}$ be a coefficient function.  Denote by $V({\mathbf c})$: the direct sum as ${\mathcal I}$-graded vector space and  as  $\epsilon$-representations of ${\mathbf c}({\mathcal O})$ copies of $V({\mathcal O})$ with a ${\mathcal I}$-basis 
\[
{\mathcal B}_{\mathbf c} = \coprod_{{\mathcal O} \in {\mathbf B}/\tau} \coprod_{b \in {\mathcal O}} \{v_i^j(b) \mid i \in Supp(b), 1 \leq j \leq {\mathbf c}({\mathcal O})\}
\]
such that, for  $b = b(i_1, j_1, k_1)$, $b' = b(i_2, j_2, k_2)$, $i, i' \in {\mathcal I}$, $1 \leq j \leq c(b)$ and $1 \leq j' \leq c(b')$, we have 
\[
\langle v_i^j(b), v_{i'}^{j'}(b')\rangle_{\mathbf c} 
= \begin{cases} 1, &\text{if $j' = j$, $b' = \tau(b)$, $i > 0$ and $i + i' = 0$; } \\
\epsilon, &\text{if $j' = j$,  $b' = \tau(b)$, $i < 0$ and $i + i' = 0$; }\\ 
1, &\text{if $j' = j$,  $b' = \tau(b)$, $i = i' = 0$ and $b$ is above}\\ &\text{the principal diagonal; }\\ 
\epsilon, &\text{if $j' = j$,  $b' = \tau(b)$, $i = i' = 0$ and $b$ is below}\\ &\text{the principal diagonal; }\\ 
\epsilon^{k_1}, &\text{if $j' = j$,  $b' = \tau(b)$,  $\tau(b) \ne b$,  $i = i' = 0$ and $b$ is on}\\ &\text{the principal diagonal; }\\ 
1, &\text{if $j' = j$,  $b' = \tau(b) = b$,  $i = i' = 0$ and $b$ is on}\\ &\text{the principal diagonal; }\\
0, &\text{otherwise.}
\end{cases}
\]
This is just from the formulae of \ref{N:BasisVOrbit} adapted to our situation.

We will denote by ${\mathfrak g}(V({\mathbf c}))$:  the Lie algebra corresponding to $(V({\mathbf c}),  \langle\  , \   \rangle_{\mathbf c})$. In other words, ${\mathfrak g}(V({\mathbf c}))$ is the vector space of linear transformations $X:V({\mathbf c}) \rightarrow V({\mathbf c})$ such that $\langle X(u), v\rangle_{\mathbf c} + \langle u, X(v)\rangle_{\mathbf c} = 0$ for all  $u, v \in V({\mathbf c})$.

From our construction and proved similarly as in notation~\ref{N:CCorrespondance}, we have that $V({\mathbf c}) = \oplus_{i \in {\mathcal I}} V_i({\mathbf c})$ and $\dim(V_i({\mathbf c})) = \delta_i$ for all $i \in {\mathcal I}$.  

Let the three ${\mathcal I}$-graded  linear transformations  
\[
E_{\mathbf c}:V({\mathbf c}) \rightarrow V({\mathbf c}), \quad   H_{\mathbf c}:V({\mathbf c}) \rightarrow V({\mathbf c}) \quad \text{ and } \quad F_{\mathbf c}:V({\mathbf c}) \rightarrow V({\mathbf c})
\]
denote the direct sum as ${\mathcal I}$-graded linear transformation of ${\mathbf c}({\mathcal O})$ copies of the linear transformations 
\[
E_{\mathcal O}:V({\mathcal O}) \rightarrow V({\mathcal O}), \quad  H_{\mathcal O}:V({\mathcal O}) \rightarrow V({\mathcal O}) \quad \text{ and } \quad F_{\mathcal O}:V({\mathcal O}) \rightarrow V({\mathcal O}).
\]
By applying proposition~\ref{P:JMTripleProof} on each summand, we get that these three linear transformations are such that 
\begin{itemize}
\item $E_{\mathbf c}$ is of degree $2$, $H_{\mathbf c}$ is of degree $0$ and $F_{\mathbf c}$ is of degree $-2$;
\item $E_{\mathbf c}, H_{\mathbf c}, F_{\mathbf c}$ belong to the Lie algebra  ${\mathfrak g}(V({\mathbf c}))$;
\item $[H_{\mathbf c}, E_{\mathbf c}] = 2E_{\mathbf c}$, $[H_{\mathbf c}, F_{\mathbf c}] = -2F_{\mathbf c}$ and $[E_{\mathbf c}, F_{\mathbf c}] = H_{\mathbf c}$, in other words,  the triple $(E_{\mathbf c}, H_{\mathbf c}, F_{\mathbf c})$ is an ${\mathcal I}$-graded Jacobson-Morozov triple for the Lie algebra of $(V({\mathbf c}),  \langle\  , \   \rangle_{\mathbf c})$;
\item  the ${\mathcal I}$-graded linear transformation $\phi_{\mathbf c}: V({\mathbf c}) \rightarrow V({\mathbf c})$ given by 
\[
\phi_{{\mathbf c}, i} = {E_{\mathbf c}}\vert_{V_i({\mathbf c})}: V_i({\mathbf c}) \rightarrow V_{i + 2}({\mathbf c}) \quad \text{for all $i \in {\mathcal I}$  such that $(i + 2) \in {\mathcal I}$},
\]
is  such that $(V({\mathbf c}), \phi_{\mathbf c}, \langle\ , \ \rangle_{\mathbf c})$ is an $\epsilon$-representation.
\end{itemize}

As such we can express these linear transformations on the elements of the basis ${\mathcal B}_{\mathbf c}$ using the expressions in \ref{SS:EHFDefinition}. For example, let ${\mathcal O}$ denote the $\langle \tau \rangle$-orbit of $b$, then,  if the orbit ${\mathcal O}$ has no overlapping ${\mathcal I}$-box supports,  $1 \leq j \leq {\mathbf c}({\mathcal O})$ and $i \in Supp(b)$, we have
\[
E_{\mathbf c}(v_i^j(b)) = \begin{cases} {\phantom -} 0, &\text{ if $i = \max(Supp(b))$;}\\ {\phantom -} v_{i + 2}^j(b), &\text{ if $i \ne \max(Supp(b))$ and $i > 0$;}\\  -v_{i + 2}^j(b), &\text{ if $i \ne \max(Supp(b))$ and $i < 0$;} \end{cases}
\]
while if the orbit ${\mathcal O}$ has overlapping ${\mathcal I}$-box supports,  $b = b(i_1, j_1, k_1) \in {\mathcal O}$, $1 \leq j \leq {\mathbf c}({\mathcal O})$  and $i \in Supp(b)$, then we have 
\[
E_{\mathbf c}(v_i^j(b)) = \begin{cases} {\phantom -} 0, &\text{if $i = \max(Supp(b))$;}\\ {\phantom -} v_{i + 2}^j(b), &\text{if $i \ne \max(Supp(b))$ and $i > 0$;}\\  -v_{i + 2}^j(b), &\text{if  $i < 0$;}\\   {\phantom -} \epsilon v_2^j(b), &\text{if $i = 0$, $\max(Supp(b)) > 0$ and $b$ is below the principal}\\ &\text{diagonal;}\\ {\phantom -}  v_2^j(b), &\text{if $i = 0$, $\max(Supp(b)) > 0$ and $b$ is above the principal}\\ &\text{diagonal;} \\ {\phantom -} \epsilon^{k_1}v_{2}^j(b), &\text{if $i = 0$, $\max(Supp(b)) > 0$ and $b$ is on the principal}\\ &\text{diagonal;}  \end{cases}
\]
\end{notation}

\subsection{}\label{SS:coefficientOrbitOdd}
Given an element $X \in {\mathfrak g}_2$, we will now associate to it a coefficient function ${\mathbf c}_X:{\mathbf B}/\tau \rightarrow {\mathbb N}$ as follows. By lemma~\ref{G2AsRepresOdd}  (b) and with the notation of this lemma, $X \in {\mathfrak g}_2$ is equivalent to giving the $\epsilon$-representation  $(V, \phi_V, \langle\  ,\  \rangle)$, where 
\[
\phi_{V, i} = X\vert_{V_i} : V_i \rightarrow  V_{i + 2} \quad v \mapsto X(v) \quad  \text{ for all $v \in V_i$}
\]
is the restriction of $X$ on $V_i$ for all $i \in {\mathcal I}$, $i \ne (2m - 1)$. By the theorem of Krull-Remak-Schmidt, this $\epsilon$-representation can be written in essentially unique way as a direct sum of indecomposable  $\epsilon$-representations (up to isomorphism). We saw in proposition~\ref{P:JMTripleProof} that the set of indecomposable $\epsilon$-representations is in bijection with the set ${\mathbf B}/\tau$. More precisely, for each orbit ${\mathcal O} \in {\mathbf B}/\tau$, we have constructed a representative $(V({\mathcal O}), \phi_{\mathcal O}, \langle \  , \  \rangle_{\mathcal O})$ of the isomorphism class of the indecomposable $\epsilon$-representation. Denote by ${\mathbf c}_X({\mathcal O})$: the multiplicity (up to isomorphism) of the indecomposable  $\epsilon$-representation $(V({\mathcal O}), \phi_{\mathcal O}, \langle \  , \  \rangle_{\mathcal O})$. So the coefficient function ${\mathbf c}_X:{\mathbf B}/\tau \rightarrow {\mathbb N}$ is defined by ${\mathcal O} \mapsto {\mathbf c}_X({\mathcal O})$.

\begin{remark}
Above in the case of  $\epsilon = 1$, i.e. the orthogonal case, we are in the  situation where $G = O(V)$  as in the article \cite{DW2002} of Derksen and Weyman. We will treat the case of $SO(V)$ in the next section. Because we are in the orthogonal case, the coefficient function ${\mathbf c}_X:{\mathbf B}/\tau \rightarrow {\mathbb N}$  is well-defined.
\end{remark}

\subsection{}\label{SS:IsomorphismV(c)andVOddCase}
We will now constructed an ${\mathcal I}$-graded isomorphism $T_{\mathbf c}:V({\mathbf c}) \rightarrow V$ such that $\langle T_{\mathbf c}(u), T_{\mathbf c}(v)\rangle = \langle u, v \rangle_{\mathbf c}$ for all $u, v \in V({\mathbf c})$. We will define this isomorphism on the bases ${\mathcal B}_{\mathbf c}$ of $V({\mathbf c})$ and ${\mathcal B}$ of $V$. 

Let $i \in {\mathcal I}$. We will first define the function if $i > 0$ and secondly if $i = 0$.  In the following, let $c:{\mathbf B} \rightarrow {\mathbb N}$ be the unique symmetric function corresponding to the coefficient function ${\mathbf c}: {\mathbf B}/\tau \rightarrow {\mathbb N}$.  

Start with $i > 0$. The proof is similar to the one when $\vert {\mathcal I} \vert$ is even in proposition~\ref{Prop_Indexation_Orbits_Even}.   We get easily that 
\[
\{b \in {\mathbf B} \mid i \in Supp(b)\} \rightarrow \{b' \in {\mathbf B} \mid -i = \sigma(i) \in Supp(b')\} \text{ defined by } b \mapsto \tau(b)
\]
is a bijection.  Fix a total order $b_1 > b_2 > \dots > b_n$ on the set $\{b \in {\mathbf B} \mid i \in Supp(b)\}$ such that $c(b_r) \geq c(b_s)$ whenever $r < s$ for $1 \leq r, s \leq n$ and let $b'_1 > b'_2 > \dots > b'_n$ be the total order on the set $\{b' \in {\mathbf B} \mid -i = \sigma(i) \in Supp(b')\}$ such that     $b'_r = \tau(b_r)$ for $1 \leq r \leq n$. 
Consider the Young diagram corresponding to the partition $c(b_1) \geq c(b_2) \geq \dots \geq c(b_n)$. So the row indexed by $b_r$ will have $c(b_r)$ columns.  We would get the same Young diagram $c(b'_1) \geq c(b'_2) \geq \dots \geq c(b'_n)$ if we had started with the total order  $b'_1 > b'_2 > \dots > b'_n$  on the set $\{b' \in {\mathbf B} \mid -i = \sigma(i) \in Supp(b')\}$. This follows from the fact that $c$ is symmetric.  This partition is a partition of $\delta_i$ from~\ref{N:CCorrespondanceOdd}.  In this Young diagram, we fill the boxes with the integers from 1 to $\delta_i$ in strictly increasing order from left to right and top to bottom.   We define $T_{{\mathbf c}, i}: V_i({\mathbf c}) \rightarrow V_i$ (respectively $T_{{\mathbf c}, -i}: V_{-i}({\mathbf c}) \rightarrow V_{-i}$)  to be the unique linear transformation such that if the entry in the box at row $r$ and column $s$ is $j$, then 
$T_{{\mathbf c}, i}(v_i^s(b_r)) = u_{i, j}$ (respectively $T_{{\mathbf c}, -i}(v_{-i}^s(b'_r)) = T_{{\mathbf c}, -i}(v_{-i}^s(\tau(b_r))) = u_{-i, j}$).

Secondly we consider the case $i = 0$.  In \ref{SS:EHFDefinition} (second case), we have noted that  the orbit ${\mathcal O} \in {\mathbf B}/\tau$ has overlapping ${\mathcal I}$-box supports if and only if  $0 \in Supp(b)$ for any ${\mathcal I}$-box $b \in {\mathcal O}$. 

In the symplectic case and when $\vert {\mathcal I} \vert$ is odd, all the orbits ${\mathcal O} \in {\mathbf B}/\tau$ are such that ${\mathcal O} = \{b, \tau(b)\}$ with $\tau(b) \ne b$ for some $b \in {\mathbf B}$, in other words, $\vert {\mathcal O} \vert = 2$. From this observation, we get that $c(b) = c(\tau(b)) = {\mathbf c}({\mathcal O})$. we have that
\[
\sum_{\begin{subarray}{c} b \in {\mathbf B}\\ 0 \in Supp(b)\end{subarray}} c(b) = 2 \sum_{\mathcal O} {\mathbf c}({\mathcal O}) = \delta_0 = 2\delta'_0 \quad \Rightarrow \quad \sum_{\mathcal O} {\mathbf c}({\mathcal O}) = \delta'_0,
\]
where the second sum over ${\mathcal O}$ is over the set of orbits ${\mathcal O}$ having overlapping ${\mathcal I}$-box supports. Fix a total order ${\mathcal O}_1 > {\mathcal O}_2 > \dots > {\mathcal O}_{n'}$ on the set 
\[
\{{\mathcal O} \in {\mathbf B}/\tau \mid {\mathcal O} \text{ has overlapping ${\mathcal I}$-box supports}\} 
\]
such that  ${\mathbf c}({\mathcal O}_r) \geq {\mathbf c}({\mathcal O}_s)$ whenever $r < s$ for $1 \leq r, s \leq n'$. Consider the Young diagram corresponding to the partition ${\mathbf c}({\mathcal O}_1) \geq {\mathbf c}({\mathcal O}_2) \geq \dots \geq {\mathbf c}({\mathcal O}_{n'})$. So the row indexed by ${\mathcal O}_r$ will have ${\mathbf c}({\mathcal O}_r)$ columns. This partition is a partition of $\delta'_0$ from our observation above.  In this Young diagram, we fill the boxes with the integers from 1 to $\delta'_0$ in strictly increasing order from left to right and top to bottom.   We define $T_{{\mathbf c}, 0}: V_0({\mathbf c}) \rightarrow V_0$  to be the unique linear transformation such that if the entry in the box at row $r$ corresponding to the orbit ${\mathcal O}_r$ and column $s$, where $1 \leq s \leq {\mathbf c}({\mathcal O}_r)$ is $j$,   then  we define
\[
T_{{\mathbf c}, 0}(v_0^s(b)) = \begin{cases} u_{0, j}, &\text{if $b$ is above the principal diagonal;}\\ u_{0, -j}, &\text{if $b$ is below the principal diagonal;}\\ u_{0, \epsilon^{k_1} j}, &\text{if $b$ is on the principal diagonal and $b = b(i_1, j_1, k_1)$;} \end{cases} 
\]
for $b \in {\mathcal O}_r$. 

In the orthogonal case and when $\vert {\mathcal I}\vert$ is odd,  the $\langle \tau \rangle$-orbits are of two types: either the $\langle \tau \rangle$-orbit ${\mathcal O} = \{ b \}$, where $b \in {\mathbf B}$, $\tau(b) = b$ and $b$ is on the principal diagonal, or the $\langle \tau \rangle$-orbit ${\mathcal O} = \{ b, \tau(b) \}$, where $b \in {\mathbf B}$, $\tau(b) \ne b$ and $b$ is not on the principal diagonal. Denote by ${\mathbf B}^+$: the set of ${\mathcal I}$-boxes strictly above the principal diagonal and by ${\mathbf B}^-$: the set of ${\mathcal I}$-boxes strictly below the principal diagonal.  Write 
\[
\sum_{\begin{subarray}{c} b \in {\mathbf B}\\ 0 \in Supp(b)\\ \tau(b) = b \end{subarray}} c(b) = \delta''_0 \quad \text{ et } \quad \sum_{\begin{subarray}{c} b \in {\mathbf B}^+\\ 0 \in Supp(b) \end{subarray}} c(b) = \delta'''_0 
\]
Note that $\delta''_0$ has the same parity than $\delta_0$ because
\[
\delta_0 = \sum_{\begin{subarray}{c} b \in {\mathbf B}\\ 0 \in Supp(b) \end{subarray}} c(b) = \sum_{\begin{subarray}{c} b \in {\mathbf B}\\ 0 \in Supp(b)\\ \tau(b) = b \end{subarray}} c(b) + 2 \sum_{\begin{subarray}{c} b \in {\mathbf B}^+\\ 0 \in Supp(b) \end{subarray}} c(b) = \delta''_0 + 2 \delta'''_0.
\]
Here we have used the fact that the function $c: {\mathbf B} \rightarrow {\mathbb N}$ is symmetric. 

We get easily that $\{b \in {\mathbf B}^+ \mid 0 \in Supp(b) \} \rightarrow \{b \in {\mathbf B}^- \mid 0 \in Supp(b) \}$ defined by $b \mapsto \tau(b)$ is a bijection. Fix a total order $b_1 > b_2 > \dots > b_n$ on the set $\{ b \in {\mathbf B}^+ \mid  0 \in Supp(b)\}$ such that $c(b_r) \geq c(b_s)$ whenever $r < s$ for $1 \leq r, s \leq n$ and let $b'_1 > b'_2 > \dots > b'_n$ be the total order on the set $\{ b' \in {\mathbf B}^- \mid  0 \in Supp(b')\}$ such that $b'_r = \tau(b_r)$ for $1 \leq r \leq n$. Consider the Young diagram ${\mathcal Y}_1$ corresponding to the partition $c(b_1) \geq c(b_2) \geq \dots \geq c(b_n)$. So the row indexed by $b_r$ will have $c(b_r)$ columns. We would get the same Young diagram $c(b'_1) \geq c(b'_2) \geq \dots \geq c(b'_n)$ if we had started with the total order $b'_1 > b'_2 > \dots > b'_n$, because $c$ is symmetric. This partition is a partition of $\delta'''_0$. In the Young diagram ${\mathcal Y}_1$, fill the boxes with the integers from $1$ to $\delta'''_0$ with strictly increasing order from left to right and top to bottom.

Consider also $\{ b'' \in {\mathbf B} \mid 0 \in Supp(b''), \tau(b'') = b''\}$. Fix a total order $b''_1 > b''_2 > \dots > b''_{n'}$ on this last set such that $c(b''_r) \geq c(b''_s)$ whenever $1 \leq r < s \leq n'$. Consider the Young diagram ${\mathcal Y}_2$ corresponding to the partition $c(b''_1) \geq c(b''_2) \geq \dots \geq c(b''_{n'})$. So the row indexed by $b''_r$ will have $c(b''_r)$ columns. This partition is a partition of $\delta''_0$.  In the Young diagram ${\mathcal Y}_2$, fill the boxes with the integers from $(\delta'''_0 + 1)$ to $(\delta'''_0 + \delta''_0)$ with strictly increasing order from left to right and top to bottom.

To define $T_{{\mathbf c}, 0}:V_0({\mathbf c}) \rightarrow V_0$, we need to consider two situations:  $\delta_0$ is either even or odd. Fix  a square root $\sqrt 2$  of $2$ and a square root $\sqrt{-2}$ of $-2$ in the field ${\mathbf k}$.  Recall our hypothesis about the characteristic $p$ of the field ${\mathbf k}$ in \ref{SS:CharacteristicK}. 

 First we start with $\delta_0$ being even and consequently $\delta''_0$ is even. Define with the same notation as above and for $b \in {\mathbf B}$, $0 \in Supp(b)$
\[
T_{{\mathbf c}, 0}(v_0^s(b)) = \begin{cases} 
u_{0, j}, &\text{if $b \in {\mathbf B}^+$ and $b = b_r$;}\\ \\
u_{0, (\delta_0 + 1 - j)}, &\text{if $b \in {\mathbf B}^-$ and $b = \tau(b_r) = b'_r$;}\\ \\
\frac{1}{\sqrt 2}\left(u_{0, j'} + u_{0, (\delta_0 + 1 - j')}\right), &\text{if $b \in {\mathbf B}$, $\tau(b) = b$, $b = b''_r$ and }\\ &\text{$(\delta'''_0 + 1) \leq j' \leq (\delta'''_0 + (\delta''_0 /2))$; }\\ \\
\frac{1}{\sqrt {-2}}\left(u_{0, j'} - u_{0, (\delta_0 + 1 - j')}\right), &\text{if $b \in {\mathbf B}$, $\tau(b) = b$, $b = b''_r$ and }\\ &\text{$(\delta'''_0 + (\delta''_0/2) + 1) \leq j' \leq (\delta'''_0 + \delta''_0)$; }
\end{cases}
\]
where the entry in the filled Young diagram ${\mathcal Y}_1$ in the box at row $r$ and column $s$ is $j$ and the entry in the filled Young diagram  ${\mathcal Y}_2$ in the box at row $r$ and column $s$ is $j'$.

Secondly consider the case where $\delta_0$ being odd and consequently $\delta''_0$ is also odd. Define with the same notation as above and for $b \in {\mathbf B}$, $0 \in Supp(b)$
\[
T_{{\mathbf c}, 0}(v_0^s(b)) = \begin{cases} 
u_{0, j}, &\text{if $b \in {\mathbf B}^+$ and $b = b_r$;}\\ \\
u_{0, (\delta_0 + 1 - j)}, &\text{if $b' \in {\mathbf B}^-$ and $b = \tau(b_r) = b'_r$;}\\ \\
\frac{1}{\sqrt 2}\left(u_{0, j'} + u_{0, (\delta_0 + 1 - j')}\right), &\text{if $b \in {\mathbf B}$, $\tau(b) = b$, $b = b''_r$ and }\\ &\text{$(\delta'''_0 + 1) \leq j' < (\delta'''_0 + ((\delta''_0 + 1)/2))$; }\\ \\
u_{0, j'} &\text{if $b \in {\mathbf B}$, $\tau(b) = b$, $b = b''_r$ and }\\ &\text{$j' = (\delta'''_0 + (\delta''_0 + 1)/2)$; }\\ \\
\frac{1}{\sqrt {-2}}\left(u_{0, j'} - u_{0, (\delta_0 + 1 - j')}\right), &\text{if $b \in {\mathbf B}$, $\tau(b) = b$, $b = b''_r$ and }\\ &\text{$(\delta'''_0 + (\delta''_0 + 1)/2) < j' \leq (\delta'''_0 + \delta''_0)$; }
\end{cases}
\]
where the entry in the filled Young diagram ${\mathcal Y}_1$ in the box at row $r$ and column $s$ is $j$ and the entry in the filled Young diagram  ${\mathcal Y}_2$ in the box at row $r$ and column $s$ is $j'$.

Finally $T_{\mathbf c}:V({\mathbf c}) \rightarrow V$ is the direct sum of the linear transformations $T_{{\mathbf c}, i}: V_i({\mathbf c}) \rightarrow V_i$ and $T_{{\mathbf c}, -i}: V_{-i}({\mathbf c}) \rightarrow V_{-i}$ for all $i \in {\mathcal I}$, $i > 0$ and of the linear transformation $T_{{\mathbf c}, 0}: V_0({\mathbf c}) \rightarrow V_0$.   Note that  the dimension vector  $\delta({\mathbf c}) $ of ${\mathbf c}$ is  $\delta({\mathbf c}) = (\delta_i)_{i \in {\mathcal I}}$, where $\delta_i = \dim(V_i)$ for all $i \in {\mathcal I}$.

\begin{lemma}\label{L:TcIsomorphismIodd}
With the above definition, we have that $T_{\mathbf c}:V({\mathbf c}) \rightarrow V$ is an ${\mathcal I}$-graded isomorphism such that $\langle T_{\mathbf c}(u), T_{\mathbf c}(v)\rangle = \langle u, v \rangle_{\mathbf c}$ for all $u, v \in V({\mathbf c})$.
\end{lemma}
\begin{proof}
This proof is similar to the one in lemma~\ref{L:TcIsomorphism}. There are just  more cases to consider. To be complete , we have included the proof.

From our construction, it is clear that $T_{\mathbf c}$ is an ${\mathcal I}$-graded isomorphism.
To check that $\langle T_{\mathbf c}(u), T_{\mathbf c}(v)\rangle = \langle u, v \rangle_{\mathbf c}$ for all $u, v \in V({\mathbf c})$, it is enough to check it on the basis ${\mathcal B}_{\mathbf c}$. We want to consider $\langle T_{\mathbf c}(v_i^s(b)), T_{\mathbf c}(v_{i'}^{s'}(b'))\rangle$ where $i, i' \in {\mathcal I}$, $i \in Supp(b)$, $1 \leq s \leq c(b)$, $i' \in Supp(b')$ and $1 \leq s' \leq c(b')$. Because $T_{\mathbf c}$ is ${\mathcal I}$-graded, our hypothesis in~\ref{S:SetUpOdd} on the basis ${\mathcal B}$ of $V$ and our construction in~\ref{N:CCorrespondanceOdd}, we get if $i + i' \ne 0$, that 
\[
\langle T_{\mathbf c}(v_i^s(b)), T_{\mathbf c}(v_{i'}^{s'}(b'))\rangle = 0 = \langle v_i^s(b), v_{i'}^{s'}(b')\rangle_{\mathbf c}.
\]

Now let $i' = -i$ and  assume first that $i > 0$.  Consider the same total orders as above on the sets $\{b \in {\mathbf B} \mid i \in Supp(b)\}$  and $\{b' \in {\mathbf B} \mid -i = \sigma(i) \in Supp(b)\}$. Proceeding as above with the Young diagram, assume that   $b = b_r$ in the total order $b_1 > b_2 > \dots > b_n$ on the set $\{b \in {\mathbf B} \mid i \in Supp(b)\}$ and   $b' = b'_{r'}$ in the total order $b'_1 > b'_2 > \dots > b'_n$ on the set $\{b' \in {\mathbf B} \mid -i = \sigma(b) \in Supp(b')\}$, the entry in the box at row $r$ and column $s$ is $j$ and the entry in the box at row $r'$ and column $s'$ is $j'$, then 
\[
\begin{aligned}
\langle T_{\mathbf c}(v_i^s(b)), T_{\mathbf c}(v_{-i}^{s'}(b'))\rangle &= \langle T_{\mathbf c}(v_i^s(b_r)), T_{\mathbf c}(v_{-i}^{s'}(b'_{r'}))\rangle = \langle u_{i, j}, u_{-i, j'}\rangle\\ &= \begin{cases} 1, &\text{if $j = j'$;} \\ 0, &\text{otherwise;}\end{cases} = \begin{cases} 1, &\text{if $r= r'$ and $s = s'$;} \\ 0, &\text{otherwise.}\end{cases} 
\end{aligned}
\]
If we now consider the bilinear form on $V({\mathbf c})$, we have
\[
\begin{aligned}
\langle v_i^s(b), v_{-i}^{s'}(b')\rangle_{\mathbf c} &= \langle v_i^s(b_r), v_{-i}^{s'}(b'_{r'})\rangle_{\mathbf c} = \langle v_i^s(b_r), v_{-i}^{s'}(\tau(b_{r'}))\rangle_{\mathbf c}\\ &= \begin{cases} 1, &\text{if $r= r'$ and $s = s'$;} \\ 0, &\text{otherwise.}\end{cases} 
\end{aligned}
\]
So if $i > 0$, then $\langle T_{\mathbf c}(v_i^s(b)), T_{\mathbf c}(v_{-i}^{s'}(b'))\rangle = \langle v_i^s(b), v_{-i}^{s'}(b')\rangle_{\mathbf c}$.

If $i \in {\mathcal I}$, $i < 0$, then 
\[
\begin{aligned}
\langle T_{\mathbf c}(v_i^s(b)), T_{\mathbf c}(v_{-i}^{s'}(b'))\rangle &= \epsilon \langle T_{\mathbf c}(v_{-i}^{s'}(b')), T_{\mathbf c}(v_{i}^{s}(b))\rangle\\ &=  \epsilon  \langle v_{-i}^{s'}(b'), v_{i}^{s}(b)\rangle_{\mathbf c} =  \langle v_{-i}^{s}(b), v_{i}^{s'}(b')\rangle_{\mathbf c}.
\end{aligned}
\]

Assume now that $i = 0$ and we are in the symplectic case.  Consider the same total order as above on the set 
\[
\{{\mathcal O} \in {\mathbf B}/\tau \mid {\mathcal O} \text{ has overlapping ${\mathcal I}$-box supports}\}. 
\]
Proceeding as above with the Young diagram, assume that $0 \in Supp(b) \cap Supp(b')$, $b \in {\mathcal O}_r$, $b' \in {\mathcal O}_{r'}$, the entry in the box at row $r$ and column $s$ is $j$ and the entry in the box at row $r'$ and column $s'$ is $j'$. From our definition, we have that
\[
T_{\mathbf c}(v_0^s(b)) = \begin{cases} u_{0, j}, &\text{if $b$ is above the principal diagonal;}\\ u_{0, -j}, &\text{if $b$ is below the principal diagonal;}\\ u_{0, \epsilon^{k_1} j}, &\text{if $b$ is on the principal diagonal and $b = b(i_1, j_1, k_1)$;} \end{cases} 
\]
and 
\[
T_{\mathbf c}(v_0^{s'}(b')) = \begin{cases} u_{0, j'}, &\text{if $b'$ is above the principal diagonal;}\\ u_{0, -j'}, &\text{if $b'$ is below the principal diagonal;}\\ u_{0, \epsilon^{k'_1} j'}, &\text{if $b'$ is on the principal diagonal and $b' = b(i'_1, j'_1, k'_1)$.} \end{cases} 
\]

If $s' \ne s$, then $\langle T_{\mathbf c}(v_0^s(b)), T_{\mathbf c}(v_0^{s'}(b'))\rangle = 0$, because  $j \ne \pm j'$ from our construction above. 

If $s' = s$ and $b' \not \in {\mathcal O}_r$, then  $\langle T_{\mathbf c}(v_0^s(b)), T_{\mathbf c}(v_0^{s}(b'))\rangle = 0$,  because  $r' \ne r$ and again $j \ne \pm j'$. 

If $s' = s$ and $b' \in {\mathcal O}_r$, then either $b' = b$ or $b' = \tau(b)$. Recall that in the symplectic case, we never have $\tau(b) = b$, in other words the orbit ${\mathcal O}_r = \{b, \tau(b)\}$ has cardinality 2.  Now in the first case:  $b' = b$, we get that 
$\langle T_{\mathbf c}(v_0^s(b)), T_{\mathbf c}(v_0^{s}(b))\rangle = 0$ because $\langle\  , \  \rangle$ is a symplectic bilinear form.
In the second case, $b' = \tau(b)$, we get that  $\langle T_{\mathbf c}(v_0^s(b)), T_{\mathbf c}(v_0^{s}(\tau(b)))\rangle$ is equal to
\[
\begin{cases} \langle u_{0, j}, u_{0, -j}\rangle = 1&\text{if $b$ is above the principal diagonal;}\\  \langle u_{0, -j}, u_{0, j}\rangle = \epsilon &\text{if $b$ is below the principal diagonal;}\\  \langle u_{0, \epsilon^{k_1} j}, u_{0, \epsilon^{(1 - k_1)}j}\rangle = \epsilon^{k_1}, &\text{if $b = b(i_1, j_1, k_1)$ is on the principal diagonal.}
\end{cases}
\]
In this last case when the ${\mathcal I}$-box $b$ is on the principal diagonal, there are two situations: either $b = b(i_1, j_1, 0)$ or $b = b(i_1, j_1, 1)$. We have respectively $\tau(b) = b(i_1, j_1, 1)$ and $\tau(b) = b(i_1, j_1, 0)$. So we get 
\[
\begin{cases} \langle u_{0, j}, u_{0, -j}\rangle = 1,  &\text{if $b = b(i_1, j_1, 0)$ is on the principal diagonal.}\\  \langle u_{0, -j}, u_{0, j}\rangle = \epsilon,  &\text{if $b = b(i_1, j_1, 1)$ is on the principal diagonal.}
\end{cases}
\]

Note that we have 
\[
\langle v_0^s(b), v_{0}^{s'}(b')\rangle_{\mathbf c} 
= \begin{cases} 1, &\text{if $s' = s$,  $b' = \tau(b)$ and $b$ is above the principal diagonal;}\\ 
\epsilon, &\text{if $s' = s$,  $b' = \tau(b)$ and $b$ is below the principal diagonal;}\\ 
\epsilon^{k_1}, &\text{if $s' = s$,  $b' = \tau(b)$ and $b = b(i_1, j_1, k_1)$ is on the}\\ &\text{principal diagonal; }\\ 
0, &\text{otherwise.}
\end{cases}
\]

So we have $\langle T_{\mathbf c}(v_i^s(b)), T_{\mathbf c}(v_{i'}^{s'}(b'))\rangle = \langle v_i^s(b), v_{i'}^{s'}(b')\rangle_{\mathbf c}$ in the symplectic case,   where $i, i' \in {\mathcal I}$, $i \in Supp(b)$, $1 \leq s \leq c(b)$, $i' \in Supp(b')$ and $1 \leq s' \leq c(b')$. 

Assume now that $i = 0$ and we are in the orthogonal case. Consider the same total order: $b_1 > b_2 > \dots > b_n$ on the set $\{b \in {\mathbf B}^+ \mid 0 \in Supp(b)\}$ as above, the corresponding total order: $b'_1 > b'_2 > \dots > b'_n$ on the set $\{b' \in {\mathbf B}^- \mid 0 \in Supp(b')\}$ where $b'_r = \tau(b_r)$ for $1 \leq r \leq n$ as above and the same total order: $b''_1 > b''_2 > \dots > b''_{n'}$ on the set $\{b'' \in {\mathbf B} \mid 0 \in Supp(b''), \tau(b'') = b''\}$ as above.  By our definition, we have 
\[
\begin{cases}
T_{{\mathbf c}, 0}(v_0^s(b)) = u_{0, j}, &\text{if $b \in {\mathbf B}^+$ and $b = b_r$;}\\ \\
T_{{\mathbf c}, 0}(v_0^s(b')) = u_{0, (\delta_0 + 1 -  j)}, &\text{if $b \in {\mathbf B}^-$ and $b' = b'_r = \tau(b_r)$;}
\end{cases}
\]
where $j$ is the entry in the Young diagram ${\mathcal Y}_1$ in the box at row $r$ and column $s$.  In this case, we have that $1 \leq j \leq \delta'''_0$ and $(\delta''_0 + \delta'''_0 + 1) \leq (\delta_0 + 1 - j) \leq \delta_0$.

We also have that 
\[
T_{{\mathbf c}, 0}(v_0^s(b)) = \begin{cases} 
\frac{1}{\sqrt 2}\left(u_{0, j'} + u_{0, (\delta_0 + 1 - j')}\right), &\text{if $b \in {\mathbf B}$, $\tau(b) = b$, $b = b''_r$ and }\\ &\text{$(\delta'''_0 + 1) \leq j' \leq (\delta'''_0 + (\delta''_0 /2))$; }\\ \\
\frac{1}{\sqrt {-2}}\left(u_{0, j'} - u_{0, (\delta_0 + 1 - j')}\right), &\text{if $b \in {\mathbf B}$, $\tau(b) = b$, $b = b''_r$ and }\\ &\text{$(\delta'''_0 + (\delta''_0/2) + 1) \leq j' \leq (\delta'''_0 + \delta''_0)$; }
\end{cases}
\]
when $\delta_0$ is even and 
\[
T_{{\mathbf c}, 0}(v_0^s(b)) = \begin{cases} 
\frac{1}{\sqrt 2}\left(u_{0, j'} + u_{0, (\delta_0 + 1 - j')}\right), &\text{if $b \in {\mathbf B}$, $\tau(b) = b$, $b = b''_r$ and }\\ &\text{$(\delta'''_0 + 1) \leq j' < (\delta'''_0 + ((\delta''_0 + 1)/2))$; }\\ \\
u_{0, j'} &\text{if $b \in {\mathbf B}$, $\tau(b) = b$, $b = b''_r$ and }\\ &\text{$j' = (\delta'''_0 + (\delta''_0 + 1)/2)$; }\\ \\
\frac{1}{\sqrt {-2}}\left(u_{0, j'} - u_{0, (\delta_0 + 1 - j')}\right), &\text{if $b \in {\mathbf B}$, $\tau(b) = b$, $b = b''_r$ and }\\ &\text{$(\delta'''_0 + (\delta''_0 + 1)/2) < j' \leq (\delta'''_0 + \delta''_0)$; }
\end{cases}
\]
when $\delta_0$ is odd, where  the entry in the filled Young diagram  ${\mathcal Y}_2$ in the box at row $r$ and column $s$ is $j'$.
In this case we have that $(\delta'''_0 + 1) \leq j' \leq (\delta'''_0 + \delta''_0)$ and  $(\delta'''_0 + 1) \leq (\delta_0 + 1 - j') \leq (\delta'''_0 + \delta''_0)$.  Because of the above inequalities for $j$ and $j'$ and our hypothesis in  \ref{S:SetUpOdd}, we have for $b, {\overline b} \in {\mathbf B}$, $0 \in Supp(b) \cap Supp({\overline b})$, $1 \leq s \leq c(b)$ and $1 \leq {\overline s} \leq c({\overline b})$ that 
$\langle T_{{\mathbf c}, 0}(v_0^s(b)), T_{{\mathbf c}, 0}(v_0^{\overline s}({\overline b}))\rangle \ne 0$ only in the following cases:
\begin{itemize}
\item $b \in {\mathbf B}^+$, ${\overline b} \in {\mathbf B}^-$;
\item $b \in {\mathbf B}^-$, ${\overline b} \in {\mathbf B}^+$;
\item $b, {\overline b} \in {\mathbf B}$, $\tau(b) = b$, $\tau({\overline b}) = ({\overline b})$;   
\end{itemize}
and similarly we have that
$\langle v_0^s(b), v_0^{\overline s}({\overline b})\rangle_{\mathbf c} \ne 0$ only in the same cases as above:
\begin{itemize}
\item $b \in {\mathbf B}^+$, ${\overline b} \in {\mathbf B}^-$;
\item $b \in {\mathbf B}^-$, ${\overline b} \in {\mathbf B}^+$;
\item $b, {\overline b} \in {\mathbf B}$, $\tau(b) = b$, $\tau({\overline b}) = ({\overline b})$.   
\end{itemize}
We will now consider each of these cases.

If $b \in {\mathbf B}^+$, $0 \in Supp(b)$, $b = b_r$, $1 \leq s \leq c(b_r)$, ${\overline b} \in {\mathbf B}^-$, $0 \in Supp({\overline b})$, ${\overline b} = b'_{\overline r}  = \tau(b_{\overline r})$, $1 \leq {\overline s} \leq c({\overline b})$ and $j$ (respectively ${\overline j}$) is the entry in the Young diagram ${\mathcal Y}_1$ in the box at row $r$ (respectively ${\overline r}$)  and column $s$ (respectively ${\overline s}$), then
\[
\begin{aligned}
\langle T_{{\mathbf c}, 0}(v_0^s(b)), T_{{\mathbf c}, 0}(v_0^{\overline s}({\overline b}))\rangle &= \langle u_{0, j}, u_{0, (\delta_0 + 1 - {\overline j})}\rangle = \begin{cases} 1, &\text{if $j = {\overline j}$;}\\ 0, &\text{otherwise} \end{cases}\\ &= \begin{cases} 1, &\text{if $r = {\overline r}$ and $s = {\overline s}$;}\\ 0, &\text{otherwise} \end{cases}
\end{aligned}
\] 
but we also have
\[
\langle v_0^s(b_r), v_0^{\overline s}(b'_{\overline r})\rangle_{\mathbf c} = \langle v_0^s(b_r), v_0^{\overline s}(\tau(b_{\overline r}))\rangle_{\mathbf c} = \begin{cases} 1, &\text{if $r = {\overline r}$ and $s = {\overline s}$;}\\ 0, &\text{otherwise} \end{cases}
\]
and consequently 
\[
\langle T_{{\mathbf c}, 0}(v_0^s(b)), T_{{\mathbf c}, 0}(v_0^{\overline s}({\overline b}))\rangle = \langle v_0^s(b), v_0^{\overline s}({\overline b})\rangle_{\mathbf c}. 
\]

If $b \in {\mathbf B}^-$, $0 \in Supp(b)$, $b = b'_r = \tau(b_r)$, $1 \leq s \leq c(b'_r)$, ${\overline b} \in {\mathbf B}^+$, $0 \in Supp({\overline b})$, ${\overline b} = b_{\overline r}$, $1 \leq {\overline s} \leq c({\overline b})$ , then we just use  the previous case, as well as the symmetry of the bilinear forms $\langle\ , \ \rangle$ and $\langle\ , \ \rangle_{\mathbf c}$  to get that 
\[
\langle T_{{\mathbf c}, 0}(v_0^s(b)), T_{{\mathbf c}, 0}(v_0^{\overline s}({\overline b}))\rangle = \langle v_0^s(b), v_0^{\overline s}({\overline b})\rangle_{\mathbf c}.
\] 

We have now to consider the last case where $b, {\overline b} \in {\mathbf B}$, $0 \in Supp(b) \cap Supp({\overline b})$, $\tau(b) = b$ and $\tau({\overline b}) = {\overline b}$.  First note that if $\delta_0$ is even, $(\delta'''_0 + 1) \leq j'_1, j'_2 \leq (\delta'''_0 + (\delta''_0/2))$ and $(\delta'''_0 + (\delta''_0/2) + 1) \leq j''_1, j''_2 \leq (\delta'''_0 + \delta''_0)$,we get easily that
\[
\left\langle\frac{1}{\sqrt 2}\left(u_{0, j'_1} + u_{0, (\delta_0 + 1 - j'_1)}\right), \frac{1}{\sqrt 2}\left(u_{0, j'_2} + u_{0, (\delta_0 + 1 - j'_2)}\right) \right\rangle = \begin{cases} 1, &\text{if $j'_1 = j'_2$;} \\ 0, &\text{ otherwise;}\end{cases}
\]

\[
\left\langle\frac{1}{\sqrt 2}\left(u_{0, j'_1} + u_{0, (\delta_0 + 1 - j'_1)}\right), \frac{1}{\sqrt {-2}}\left(u_{0, j''_1} - u_{0, (\delta_0 + 1 - j''_1)}\right) \right\rangle = 0 
\]
 for all $j'_1$ and $j''_1$ as above;
\[
\left\langle\frac{1}{\sqrt {-2}}\left(u_{0, j''_1} - u_{0, (\delta_0 + 1 - j''_1)}\right), \frac{1}{\sqrt {-2}}\left(u_{0, j''_2} - u_{0, (\delta_0 + 1 - j''_2)}\right) \right\rangle =  \begin{cases} 1, &\text{if $j''_1 = j''_2$;} \\ 0, &\text{ otherwise;}\end{cases}
\]
while if $\delta_0$ is odd, $(\delta'''_0 + 1) \leq j'_1, j'_2 < (\delta'''_0 + (\delta''_0 + 1)/2)$,  $(\delta'''_0 + (\delta''_0 + 1)/2)) < j''_1, j''_2 \leq (\delta'''_0 + \delta''_0)$ and $j''' = (\delta'''_0 + (\delta''_0 + 1)/2)$, we also get easily that
\[
\left\langle\frac{1}{\sqrt 2}\left(u_{0, j'_1} + u_{0, (\delta_0 + 1 - j'_1)}\right), \frac{1}{\sqrt 2}\left(u_{0, j'_2} + u_{0, (\delta_0 + 1 - j'_2)}\right) \right\rangle = \begin{cases} 1, &\text{if $j'_1 = j'_2$;} \\ 0, &\text{ otherwise;}\end{cases}
\]

\[
\left\langle\frac{1}{\sqrt 2}\left(u_{0, j'_1} + u_{0, (\delta_0 + 1 - j'_1)}\right), \frac{1}{\sqrt {-2}}\left(u_{0, j''_1} - u_{0, (\delta_0 + 1 - j''_1)}\right) \right\rangle = 0 
\]
 for all $j'_1$ and $j''_1$ as above;
\[
\left\langle\frac{1}{\sqrt {-2}}\left(u_{0, j''_1} - u_{0, (\delta_0 + 1 - j''_1)}\right), \frac{1}{\sqrt {-2}}\left(u_{0, j''_2} - u_{0, (\delta_0 + 1 - j''_2)}\right) \right\rangle =  \begin{cases} 1, &\text{if $j''_1 = j''_2$;} \\ 0, &\text{ otherwise;}\end{cases}
\]
\[
\left\langle\frac{1}{\sqrt 2}\left(u_{0, j'_1} + u_{0, (\delta_0 + 1 - j'_1)}\right), u_{0, j'''} \right\rangle = \left\langle\frac{1}{\sqrt {-2}}\left(u_{0, j''_1} - u_{0, (\delta_0 + 1 - j''_1)}\right), u_{0, j'''} \right\rangle = 0 
\]
and $\langle u_{0, j'''}, u_{0, j'''} \rangle = 1$.

Consider now $b, {\overline b} \in {\mathbf B}$, $0 \in Supp(b) \cap Supp({\overline b})$, $\tau(b) = b$, $\tau({\overline b}) = {\overline b}$, $b = b''_r$, ${\overline b} = b''_{\overline r}$ and let $j$ (respectively ${\overline j}$) be the entry in the Young diagram ${\mathcal Y}_2$ in the box at row $r$ (respectively  ${\overline r}$) and column $s$ (respectively ${\overline s}$). In this case, $(\delta'''_0 + 1) \leq j, {\overline j} \leq (\delta'''_0 + \delta''_0)$. By our above computations, we get that 
\[
\langle T_{{\mathbf c}, 0}(v_0^s(b)), T_{{\mathbf c}, 0}(v_0^{\overline s}({\overline b}))\rangle  = \begin{cases} 1, &\text{if $j = {\overline j}$;}\\ 0, &\text{otherwise} \end{cases} = \begin{cases} 1, &\text{if $r = {\overline r}$ and $s = {\overline s}$;}\\ 0, &\text{otherwise} \end{cases}
\] 
but we also have
\[
\langle v_0^s(b), v_0^{\overline s}(\overline b)\rangle_{\mathbf c} = \langle v_0^s(b''_r), v_0^{\overline s}(b''_{\overline r})\rangle_{\mathbf c} = \begin{cases} 1, &\text{if $r = {\overline r}$ and $s = {\overline s}$;}\\ 0, &\text{otherwise} \end{cases}
\]
and consequently 
\[
\langle T_{{\mathbf c}, 0}(v_0^s(b)), T_{{\mathbf c}, 0}(v_0^{\overline s}({\overline b}))\rangle = \langle v_0^s(b), v_0^{\overline s}({\overline b})\rangle_{\mathbf c}. 
\] 

Thus we have proved that $\langle T_{\mathbf c}(u), T_{\mathbf c}(v)\rangle = \langle u, v \rangle_{\mathbf c}$ for all $u, v \in V({\mathbf c})$.

\end{proof}

\begin{proposition} \label{Prop_Indexation_Orbits_Odd}
Let $\delta = (\delta_i)_{i \in {\mathcal I}}$ be as in \ref{S:SetUpOdd}  and ${\mathfrak g}_2/G^{\iota}$ denotes the set of  $G^{\iota}$-orbits  in ${\mathfrak g}_2$. Then the map 
\[
\Upsilon:{\mathfrak g}_2/G^{\iota} \rightarrow  {\mathfrak C}_{\delta} \quad \text{ defined by } \quad G^{\iota} \cdot X \mapsto {\mathbf c}_X,
\]
is a well-defined bijection.
\end{proposition}
\begin{proof}
By our last observation in \ref{SS:coefficientOrbitOdd}, if $G^{\iota} \cdot X = G^{\iota} \cdot X'$, then  the corresponding $\epsilon$-representations for $X$ and $X'$ are isomorphic, we get  that ${\mathbf c}_X = {\mathbf c}_{X'}$ and the  function $\Upsilon$ is well-defined. 

To show that $\Upsilon$ is an injective function,  assume that ${\mathbf c}_X = {\mathbf c}_{X'}$, then this means that, by the Krull-Remak-Schmidt theorem, the  $\epsilon$-representations for $X$ and $X'$ are isomorphic. So $X$ and $X'$ are in the same $G^{\iota}$-orbit and $G^{\iota} \cdot X = G^{\iota} \cdot X'$.

To prove that $\Upsilon$ is surjective, we need to show that there is an element $X \in {\mathfrak g}_2$ such that ${\mathbf c}_X = {\mathbf c} \in {\mathfrak C}_{\delta}$.  Consider $X = T_{\mathbf c} E_{\mathbf c} T_{\mathbf c}^{-1}:V \rightarrow V$, where $T_{\mathbf c}:V({\mathbf c}) \rightarrow V$ is the ${\mathcal I}$-graded isomorphism defined in \ref{SS:IsomorphismV(c)andVOddCase}.  Because $E_{\mathbf c} \in {\mathfrak g}_2(V({\mathbf c}))$ and lemma~\ref{L:TcIsomorphismIodd}, $X \in {\mathfrak g}_2$. From our construction of $E_{\mathbf c}: V({\mathbf c}) \rightarrow V({\mathbf c})$, it is isomorphic as ${\mathcal I}$-graded and $\epsilon$-representation  to the direct sum of ${\mathbf c}({\mathcal O})$ copies of the linear transformation $E_{\mathcal O}:V({\mathcal O}) \rightarrow V({\mathcal O})$. By carrying this direct sum using $T_{\mathbf c}$ to $V$, we do get that ${\mathbf c}_X = {\mathbf c}$ and the proof is complete. 

\end{proof}

\begin{notation}
Let $\delta = (\delta_i)_{i \in {\mathcal I}}$ be such that $\delta_i = \dim(V_i)$ for all $i \in {\mathcal I}$. If ${\mathbf c} \in {\mathfrak C}_{\delta}$, we will denote by ${\mathcal O}_{\mathbf c}$: the unique $G^{\iota}$-orbit in ${\mathfrak g}_2$ such that $\Upsilon({\mathcal O}_{\mathbf c}) = {\mathbf c}$.
\end{notation}

\begin{remark}\label{R:IsomoVVcSymplecticOdd}
Using the isomorphism $T_{\mathbf c}:V({\mathbf c}) \rightarrow V$, we can transfer results about the orbit ${\mathcal O}_{\mathbf c}$ in the Lie algebra ${\mathfrak g}$ to the orbit of $E_{\mathbf c}$ in the Lie algebra 
${\mathfrak g}(V({\mathbf c}))$.
\end{remark}

\subsection{}\label{SS:ObservationDimensionOrthogonal}
To compute the dimension of the orbit ${\mathcal O}_{\mathbf c}$  corresponding to ${\mathbf c} \in {\mathfrak C}_{\delta}$, we will proceed as in the previous section by considering a parabolic subalgebra ${\mathfrak p}_X$ associated by Lusztig to an element $X \in {\mathcal O}_{\mathbf c} \subseteq {\mathfrak g}_2$ and then using the formula
\[
\dim({\mathcal O}_{\mathbf c}) = \dim({\mathfrak g}_0) - \dim({\mathfrak p}_0) + \dim({\mathfrak p}_2),
\]
where ${\mathfrak p}$ is the parabolic subalgebra ${\mathfrak p}_X$.

We now want to use this formula to compute the dimension $\dim({\mathcal O}_{\mathbf c})$ of the orbit ${\mathcal O}_{\mathbf c}$ in term of the coefficient function ${\mathbf c}$. This is done in the next proposition.

Note that we have to be careful in the orthogonal case, because the above formula for the dimension $\dim({\mathcal O}_{\mathbf c})$ of the orbit ${\mathcal O}_{\mathbf c}$ proved using 5.4 (a) and 5.9 of \cite{L1995} is under the hypothesis that the group $G$ is connected. In the orthogonal case, $G$ is the orthogonal group $O(V)$ and $G^{\iota}$ is not connected as we saw in lemma~\ref{G2AsRepresOdd} (a). From our definition of $\iota$, it is easy to see that the connected component of the identity $(G^{\iota})^0$ of $G^{\iota}$ is equal to $(G^0)^{\iota}$: the centralizer of $\iota$ in the connected component $G^0$ of $G$. Here $G^0$ is the special orthogonal group and is connected. Because of this observation, we can also used the above formula for the dimension of the orbit ${\mathcal O}_{\mathbf c}$.

\begin{proposition}\label{DimensionFormulaSymplecticOdd}
Let ${\mathbf c}  \in  {\mathfrak C}_{\delta}$  and consider the symmetric function $c:{\mathbf B} \rightarrow {\mathbb N}$ corresponding to ${\mathbf c}$.  Then the dimension $\dim({\mathcal O}_{\mathbf c})$ of the orbit ${\mathcal O}_{\mathbf c}$ is equal to 
\[
\left[
\begin{aligned}
&\frac{\delta_0(\delta_0 - \epsilon)}{2} + \sum_{\begin{subarray}{c} i \in {\mathcal I}\\ 0 <  i\end{subarray}} \delta_i^2 + \sum_{\begin{subarray}{c} i \in {\mathcal I}\\ 0 \leq i < (2m - 2)\\ (b, b') \in {\mathbf B} \times {\mathbf  B}\\ i \in Supp(b)\\ (i + 2) \in Supp(b')\\ \lambda(b) \geq \lambda(b') \end{subarray}} c(b) c(b')   - \sum_{\begin{subarray}{c}i \in {\mathcal I}\\ i > 0\\ (b, b') \in {\mathbf B} \times {\mathbf B}\\ i \in Supp(b) \cap Supp(b')\\ \lambda(b) \geq \lambda(b')\end{subarray}} c(b) c(b') \\ \\ & - \frac {1}{2} \sum_{\begin{subarray}{c} (b, b') \in {\mathbf B} \times {\mathbf B}\\ 0 \in Supp(b) \cap Supp(b') \\ b \ne \tau(b')\\ \lambda(b)\geq \lambda(b') \end{subarray}} c(b) c(b')  - \frac{1}{2}
\sum_{\begin{subarray}{c} b \in {\mathbf B}\\ 0 \in Supp(b)\\\lambda(b)\geq \lambda(\tau(b)) \end{subarray}} c(b) (c(b) - \epsilon)
\end{aligned}
\right]
\]
\end{proposition}
\begin{proof}
Taking into account our observation in \ref{SS:ObservationDimensionOrthogonal} in the orthogonal case and our remark~\ref{R:IsomoVVcSymplecticOdd},  we will assume that $V = V({\mathbf c})$ and ${\mathfrak g}$ will be ${\mathfrak g}(V({\mathbf c}))$.  We have that $\dim(V_i) = \delta_i$ for all $i \in {\mathcal I}$ from the fact that $\delta({\mathbf c}) = \delta$.
We have constructed in~\ref{N:CCorrespondanceOdd}  a Jacobson-Morozov triple  $(E_{\mathbf c}, H_{\mathbf c}, F_{\mathbf c})$ with $E_{\mathbf c} \in {\mathcal O}_{\mathbf c}$; in other words we have 
\[
[H_{\mathbf c}, E_{\mathbf c}] = 2 E_{\mathbf c}, \qquad [H_{\mathbf c}, F_{\mathbf c}] = -2 F_{\mathbf c}, \qquad [E_{\mathbf c}, F_{\mathbf c}] = H_{\mathbf c},
\]
where $E_{\mathbf c} \in {\mathfrak g}_2$, $H_{\mathbf c} \in {\mathfrak g}_0$ and $F_{\mathbf c} \in {\mathfrak g}_{-2}$  and such that $E_{\mathbf c} \in {\mathcal O}_{\mathbf c}$.   Thus we have a Lie algebra homomorphism $\phi: {\mathfrak sl}_2({\mathbf k}) \rightarrow {\mathfrak g}$ such that 
\[
\phi\begin{pmatrix}0 & 1\\ 0 & 0\end{pmatrix} = E_{\mathbf c}, \qquad \phi\begin{pmatrix}1 & \phantom{-}0\\ 0 & -1\end{pmatrix} = H_{\mathbf c}, \qquad \phi\begin{pmatrix}0 & 0\\ 1 & 0\end{pmatrix} = F_{\mathbf c}. 
\]
As in definition~\ref{D:LieHomo}, let $\widetilde{\phi}: SL_2({\mathbf k}) \rightarrow G$ be the homomorphism of algebraic groups whose differential is $\phi$ and by $\iota':{\mathbf k}^{\times} \rightarrow G$: the homomorphism defined by 
\[
\iota'(t) = \widetilde{\phi}\begin{pmatrix}t & 0\phantom{-} \\ 0 & t^{-1}\end{pmatrix} \qquad  \text{ for all $t \in {\mathbf k}^{\times}$.}
\]
We have the grading for the Lie algebra ${\mathfrak g} = \oplus_r  ({ }_r {\mathfrak g})$ where 
\[
({ }_r {\mathfrak g}) = \{ X \in {\mathfrak g} \mid Ad(\iota'(t))(X) = t^r X \quad \text{ for all $t \in {\mathbf k}^{\times}$}\}
\]
By taking the derivative, we get that 
\[
({ }_r {\mathfrak g}) = \{ X \in {\mathfrak g} \mid ad(H_{\mathbf c})(X) =r X \}.
\]

By lemma~\ref{G2AsRepresOdd} (c), we have that 
\[
\dim({\mathfrak g}_0) = \frac{\delta_0(\delta_0 - \epsilon)}{2} + \sum_{\begin{subarray}{c} i \in {\mathcal I}\\ i > 0\end{subarray}} \delta_i^2.
\]

Let ${\mathfrak p}$ be the parabolic subalgebra associated to the nilpotent element $E_{\mathbf c}$. From the definition of ${\mathfrak p}$, we get that ${\mathfrak p}_0$ and  ${\mathfrak p}_2$ are such that
\[
{\mathfrak p}_0 = \bigoplus_{r \geq 0}  ({ }_r{\mathfrak g}_0) \qquad \text{ and } \qquad {\mathfrak p}_2 = \bigoplus_{r \geq 2}  ({ }_r{\mathfrak g}_2) 
\]
By our description of ${\mathfrak g}_0$ and ${\mathfrak g}_2$ in lemma~\ref{G2AsRepresOdd}, we get that ${\mathfrak p}_0 =  {\mathfrak p}'_0 \oplus {\mathfrak p}''_0$, where 
\[
 {\mathfrak p}'_0 = \bigoplus_{r \geq 0} \bigoplus_{\begin{subarray}{c} i \in {\mathcal I}\\ 0 < i  \end{subarray}} \{(Z_i : V_i \rightarrow V_i) \in Hom(V_i, V_i) \mid ad(H_{\mathbf c}\vert_{V_i}) Z_i = r Z_i \},
\]
\[
{\mathfrak p}''_0 = \bigoplus_{r \geq 0}\left\{ Z_0: V_0 \rightarrow V_0 \in Hom(V_0, V_0)  \left\vert\  \begin{aligned} &ad(H_{\mathbf c}\vert_{V_0}) Z_0 = r Z_0 \text{ and }\\ &\langle Z_0(v), v'\rangle_{\mathbf c} + \langle v, Z_0(v')\rangle_{\mathbf c}  = 0\\ &\text{ for all $v, v' \in V_0$.} \end{aligned} \right.\right\}, 
\]
and 
\[
{\mathfrak p}_2 = 
\bigoplus_{r \geq 2} \bigoplus_{\begin{subarray}{c} i \in {\mathcal I}\\ 0 \leq i < (2m - 2)\end{subarray}} \{(X_i : V_i \rightarrow V_{i + 2}) \in Hom(V_i, V_{i + 2}) \mid ad(H_{\mathbf c}\vert_{V_i}) X_i = r X_i  \}.
\]

We can give  bases for the spaces $Hom(V_i, V_i)$ and $Hom(V_i, V_{i + 2})$ that are eigenvectors for $ad(H_{\mathbf c}\vert_{V_i})$.  From our construction of the triple $\{E_{\mathbf c}, H_{\mathbf c}, F_{\mathbf c}\}$, we have for each $i \in {\mathcal I}$ a basis  
\[
{\mathcal B}_i = \{v_i^j(b)\in V_i  \mid b \in {\mathbf B}, i \in Supp(b), 1 \leq j \leq c(b)\}
\]
of $V_i$. 

Let now $i \in {\mathcal I}$ and $0 \leq i  < (2m - 2)$.  For each $b, b' \in {\mathbf B}$ with $i \in Supp(b)$, $(i + 2) \in Supp(b')$, $1 \leq j \leq c(b)$ and $1 \leq j' \leq c(b')$, let $X_{(b, j)}^{(b', j')} \in Hom(V_i, V_{i + 2})$ be defined on the basis ${\mathcal B}_i $ by 
\[
X_{(b, j)}^{(b', j')}(v_i^{j''}(b'')) = \begin{cases} v_{i + 2}^{j'}(b'), &\text{if $b'' = b$ and $j'' = j$;}\\ 0, &\text{otherwise.}\end{cases}
\]
Clearly these linear transformations $X_{(b, j)}^{(b', j')}$ with  $b,  b' \in {\mathbf B}$,  $i \in Supp(b)$,  $(i + 2) \in Supp(b')$, $1 \leq j \leq c(b)$ and $1 \leq j' \leq c(b')$ form a basis of $Hom(V_i, V_{i + 2})$. 

We have that  $H_{\mathbf c}(v_i^j(b)) = h_i(b) v_i^j(b)$ for all $b \in {\mathbf B}$, $i \in Supp(b)$ and $1 \leq j \leq c(b)$ and  $X_{(b, j)}^{(b', j')}$ is an eigenvector of $ad(H_{\mathbf c}\vert_{V_i})$ with
\[
ad(H_{\mathbf c}\vert_{V_i}) (X_{(b, j)}^{(b', j')}) = (h_{(i + 2)}(b') - h_i(b)) X_{(b, j)}^{(b', j')}
\]
by computing both sides  on the basis vector $v_i^{j''}(b'')$. Consequently  we get that 
\[
\dim({\mathfrak p}_2) = 
\sum_{\begin{subarray}{c}i \in {\mathcal I}\\0 \leq i < (2m - 2)\\ (b, b') \in {\mathbf B} \times {\mathbf B}\\ i \in Supp(b)\\ (i + 2) \in  Supp(b')\\ \lambda(b) \geq \lambda(b')\end{subarray}} c(b) c(b')
\]
because $(h_{i + 2}(b') - h_i(b)) \geq 2$ is equivalent to $\lambda(b) \geq \lambda(b')$ as we saw in the proof of proposition~\ref{DimensionFormulaEven}

We can now consider the vector space ${\mathfrak p}'_0$. Let $i > 0$. For each $b, b' \in {\mathbf B}$ with $i \in Supp(b) \cap Supp(b')$, $1 \leq j \leq c(b)$ and $1 \leq j' \leq c(b')$, let $Z_{(b, j)}^{(b', j')} \in Hom(V_i, V_i)$ be defined on the basis ${\mathcal B}_i $ by 
\[
Z_{(b, j)}^{(b', j')}(v_i^{j''}(b'')) = \begin{cases} v_i^{j'}(b'), &\text{if $b'' = b$ and $j'' = j$;}\\ 0, &\text{otherwise.}\end{cases}
\]
These linear transformations $Z_{(b, j)}^{(b', j')}$ with  $b, b' \in {\mathbf B}$,  $i \in Supp(b) \cap Supp(b')$, $1 \leq j \leq c(b)$ and $1 \leq j' \leq c(b')$ form a basis of $Hom(V_i, V_i)$. Moreover $Z_{(b, j)}^{(b', j')}$ is an eigenvector of $ad(H_{\mathbf c}\vert_{V_i})$ with
\[
ad(H_{\mathbf c}) (Z_{(b, j)}^{(b', j')}) = (h_i(b') - h_i(b)) Z_{(b, j)}^{(b', j')}
\]
by computing both sides  on the basis vector $v_i^{j''}(b'')$. Consequently
\[
\dim({\mathfrak p}'_0) = 
\sum_{\begin{subarray}{c}i \in {\mathcal I}\\ i > 0\\ (b, b') \in {\mathbf B} \times{\mathbf  B}\\ i \in Supp(b) \cap Supp(b')\\ \lambda(b)  \geq  \lambda(b') \end{subarray}} c(b)c(b')
\]
because $(h_i(b') - h_i(b)) \geq 0$ is equivalent to $\lambda(b) \geq \lambda(b')$ as we saw in the proof of proposition~\ref{DimensionFormulaEven}

We must now consider the dimension $\dim({\mathfrak p}''_0)$ and we will  prove that 
\[
\dim({\mathfrak p}''_0) =  \frac {1}{2} \left(\sum_{\begin{subarray}{c} (b, b') \in {\mathbf B} \times {\mathbf B}\\ 0 \in Supp(b) \cap Supp(b') \\ b \ne \tau(b')\\ \lambda(b)\geq \lambda(b') \end{subarray}} c(b) c(b')  +
\sum_{\begin{subarray}{c} b \in {\mathbf B}\\ 0 \in Supp(b)\\\lambda(b)\geq \lambda(\tau(b)) \end{subarray}} c(b) (c(b) - \epsilon)\right).
\]

Recall that $V_0$ has  a basis ${\mathcal B}_0 = \{v_0^j(b) \mid b \in {\mathbf B}, 0 \in Supp(b), 1 \leq j \leq c(b)\}$ such that 
\[
\langle v_0^s(b), v_{0}^{s'}(b')\rangle 
= \begin{cases} 1, &\text{if $s' = s$,  $b' = \tau(b)$ and $b$ is above the principal diagonal;}\\ 
\epsilon, &\text{if $s' = s$,  $b' = \tau(b)$ and $b$ is below the principal diagonal;}\\ 
\epsilon^{k_1}, &\text{if $s' = s$,  $b' = \tau(b)$ and $b = b(i_1, j_1, k_1)$ is on the}\\ &\text{principal diagonal; }\\ 
0, &\text{otherwise.}
\end{cases}
\]

Let  ${\mathcal E} = \{b \in {\mathbf B} \mid 0 \in Supp(b)\}$ and  $(Z_0:V_0 \rightarrow V_0) \in Hom(V_0, V_0)$ be such that 
\begin{equation}\label{E:LieCondition}
\langle Z_0(v), v' \rangle_{\mathbf c} + \langle v, Z_0(v')\rangle_{\mathbf c} = 0 \quad \text{for all $v, v' \in V_0$.}
\end{equation}

For $v_0^j(b), v_0^{j'}(b') \in {\mathcal B}_0$, then we can write
\[
Z_0(v_0^j(b)) = \sum_{(b'', j'')} \xi_{(b, j)}^{(b'', j'')} v_0^{j''}(\tau(b''))
\]
where the sum runs over all $(b'', j'')$ such that $b'' \in {\mathcal E}$ and $1 \leq j'' \leq c(b'')$ and 
\[
Z_0(v_0^{j'}(b')) = \sum_{(b''', j''')} \xi_{(b', j')}^{(b''', j''')} v_0^{j'''}(\tau(b'''))
\]
where the sum runs over all $(b''', j''')$ such that $b''' \in {\mathcal E}$ and $1 \leq j''' \leq c(b''')$, 

We will first consider the symplectic case. 
The set ${\mathcal E}$ can be partitionned as follows: ${\mathcal E} = {\mathcal E}^+ \coprod {\mathcal E}^-$ where
\[
\begin{aligned}
{\mathcal E}^+ &= \{b \in {\mathcal E} \mid  b(i, j, 0) \text{ where $i, j \in {\mathcal I}$, $i \geq j$ and $i + j \geq 0$} \} \quad \text{ and }\\
{\mathcal E}^- &= \{b \in {\mathcal E} \mid b(i, j, 0) \text{ where $i, j \in {\mathcal I}$, $i \geq j$ and $i + j < 0$} \}\\ &\quad \bigcup \{b \in {\mathcal E} \mid b(i, -i, 1) \text{ where $i \in {\mathcal I}$ and $i \geq 0$} \}.
\end{aligned}
\]

${\mathcal E}^+$ consists of the ${\mathcal I}$-boxes strictly above the principal diagonale and of half of the ${\mathcal I}$-boxes on the principal diagonale, while  ${\mathcal E}^-$ consists of the ${\mathcal I}$-boxes strictly below the principal diagonale and the other half of the ${\mathcal I}$-boxes on the principal diagonale. Note that if $b \in {\mathcal E}^+$, then $\langle v_0^s(b), v_0^s(\tau(b))\rangle_{\mathbf c} = 1$  for any $1 \leq s \leq c(b)$, while if $b \in {\mathcal E}^-$, then $\langle v_0^s(b), v_0^s(\tau(b))\rangle_{\mathbf c} = -1$  for any $1 \leq s \leq c(b)$. 

Setting $v = v_0^j(b)$ and $v' = v_0^{j'}(b')$ in the equation~(\ref{E:LieCondition}) and using the values of  $\langle \   ,  \  \rangle_{\mathbf c}$  in notation~\ref{N:CCorrespondanceOdd}, we get  the equations: 
\begin{equation}\label{CoordinateEquationSymplecticOdd}
\left\{
\begin{aligned}
-\xi_{(b, j)}^{(b', j')} + \xi_{(b', j')}^{(b, j)} &= 0, \quad \text{ if $b, b' \in {\mathcal E}^+$;}\\ \\
 \xi_{(b, j)}^{(b', j')} + \xi_{(b', j')}^{(b, j)} &= 0, \quad \text{ if $b \in {\mathcal E}^+$ and $b' \in {\mathcal E}^-$;}\\ \\
-\xi_{(b, j)}^{(b', j')} - \xi_{(b', j')}^{(b, j)} &= 0, \quad \text{ if $b \in {\mathcal E}^-$ and $b' \in {\mathcal E}^+$;}\\ \\
\xi_{(b, j)}^{(b', j')} - \xi_{(b', j')}^{(b, j)} &= 0, \quad \text{ if $b, b' \in {\mathcal E}^-$;}
\end{aligned}
\right.
\end{equation}
Reciprocally if the above equations are satisfied for all $b, b' \in {\mathcal  E}$,  $1 \leq j \leq c(b)$ and $1 \leq j' \leq c(b')$, then $Z_0$ defined above on the vectors $v_0^j(b)$ will be such that the equation~(\ref{E:LieCondition}) is satisfied.

We will now give a basis of the vector space 
\[
\{(Z_0:V_0 \rightarrow V_0) \in Hom(V_0, V_0) \mid \langle Z_0(v), v' \rangle_{\mathbf c} + \langle v, Z_0(v')\rangle_{\mathbf c} = 0 \text{ for all $v, v' \in V_0$}\}
\]
and each of its element will be an eigenvector of $ad(H_{\mathbf c}\vert_{V_0})$.

Order the basis ${\mathcal B}_0$ such that the first elements are the basis vectors of the form $v_0^j(b)$ with $b \in {\mathcal E}^+$ and $1 \leq j \leq c(b)$ and  the following elements are the basis vectors of the form $v_0^j(\tau(b))$  with $b \in {\mathcal E}^+$ and $1 \leq j \leq c(b)$ and if $b \in {\mathcal E}^+$, then  $v_0^j(b)$ and $v_0^j(\tau(b))$ are respectively in the $m_1^{th}$ and $m_2^{th}$ positions with $(m_2 - m_1) = \delta_0'$. Relative to this ordered basis, the linear transformation $Z_0$ as above has the form
\begin{equation}\label{MatrixZ0Symplect}
\begin{pmatrix}
A_{11} & A_{12}\\ A_{21} & A_{22}
\end{pmatrix} 
\quad \text{with} \quad A_{22} = -A^t_{11}, \quad A_{12} = A^t_{12}, \quad A_{21} = A^t_{21}
\end{equation}

We will define  linear transformations $Z_{(b, j)}^{(b', j')}:V_0 \rightarrow V_0$ when $b, b' \in {\mathcal E}$,  $1 \leq j \leq c(b)$ and $1 \leq j' \leq c(b')$ for three cases:
\begin{enumerate}[\upshape (i)]
\item either $b, b' \in {\mathcal E}^+$ or $b, b' \in {\mathcal E}^-$;

\item $(b, j) \ne (\tau(b'), j')$ and either  $b \in {\mathcal E}^+$, $b' \in {\mathcal E}^-$ or  $b \in {\mathcal E}^-$, $b' \in {\mathcal E}^+$;

\item $(b, j) = (\tau(b'), j')$ and either  $b \in {\mathcal E}^+$, $b' \in {\mathcal E}^-$ or  $b \in {\mathcal E}^-$, $b' \in {\mathcal E}^+$.
\end{enumerate}

Case (i): If either $b, b' \in {\mathcal E}^+$ or $b, b' \in {\mathcal E}^-$, we define
\[
Z_{(b, j)}^{(b', j')}(v_0^{j''}(b''))= \begin{cases} {\phantom -}v_0^{j'}(b'), &\text{ if $(b'', j'') = (b, j)$;}\\ -v_0^j(\tau(b)), &\text{ if $(b'', j'') = (\tau(b'), j')$;} \\ {\phantom -} 0, &\text{ otherwise.} \end{cases}
\]

This means
\[
\xi_{(b_1, j_1)}^{(b_2,  j_2)} = \begin{cases} {\phantom -}1, &\text{ if $(b_1, j_1) = ( b, j)$ and $(b_2, j_2) = (\tau(b'), j')$;}\\ -1, &\text{ if $(b_1, j_1) = ( \tau(b'), j')$ and $(b_2, j_2) = (b, j)$;}\\ {\phantom -}0, &\text{otherwise.} \end{cases}
\]

We get easily that  
\[
\langle Z_{(b, j)}^{(b', j')}(v), v' \rangle_{\mathbf c} + \langle v, Z_{(b, j)}^{(b', j')}(v')\rangle_{\mathbf c} = 0 
\]
 for all $v, v' \in V_0$ from the  equations~(\ref{CoordinateEquationSymplecticOdd}).  Note that  the pair $\tau(b'), \tau(b)$ falls also in the  case (i) and it is easy to see that $Z_{(\tau(b'), j')}^{(\tau(b), j)} = - Z_{(b, j)}^{(b', j')}$ by evaluating both sides on the basis vectors $v_0^{j''}(b'')$.
 
 By looking on the basis vectors $v_0^{j''}(b'')$, we have that 
\[
ad(H_{\mathbf c}\vert_{V_0}) Z_{(b, j)}^{(b', j')} = (h_0(b') - h_0(b)) Z_{(b, j)}^{(b', j')}.
\]
We have used lemma~\ref{L:RelationCoefficientsHandF} (a) in the proof of this formula.

Case (ii): If $(b, j) \ne (\tau(b'), j')$ and either  $b \in {\mathcal E}^+$, $b' \in {\mathcal E}^-$ or  $b \in {\mathcal E}^-$, $b' \in {\mathcal E}^+$, we define
\[
Z_{(b, j)}^{(b', j')}(v_0^{j''}(b''))= \begin{cases} v_0^{j'}(b'), &\text{ if $(b'', j'') = (b, j)$;}\\ v_0^j(\tau(b)), &\text{ if $(b'', j'') = (\tau(b'), j')$;} \\  0, &\text{ otherwise.} \end{cases}
\]

This means
\[
\xi_{(b_1, j_1)}^{(b_2,  j_2)} = \begin{cases} 1, &\text{ if $(b_1, j_1) = ( b, j)$ and $(b_2, j_2) = (\tau(b'), j')$;}\\ 1, &\text{ if $(b_1, j_1) = ( \tau(b'), j')$ and $(b_2, j_2) = (b, j)$;}\\ 0, &\text{otherwise.} \end{cases}
\]

We get easily that  
\[
\langle Z_{(b, j)}^{(b', j')}(v), v' \rangle_{\mathbf c} + \langle v, Z_{(b, j)}^{(b', j')}(v')\rangle_{\mathbf c} = 0 
\]
 for all $v, v' \in V_0$ from the  equations~(\ref{CoordinateEquationSymplecticOdd}).  Note that  the pair $\tau(b'), \tau(b)$ falls also in the  case (ii) and it is easy to see that $Z_{(\tau(b'), j')}^{(\tau(b), j)} =  Z_{(b, j)}^{(b', j')}$ by evaluating both sides on the basis vectors $v_0^{j''}(b'')$.
 
 By looking on the basis vectors $v_0^{j''}(b'')$, we have that 
\[
ad(H_{\mathbf c}\vert_{V_0}) Z_{(b, j)}^{(b', j')} = (h_0(b') - h_0(b)) Z_{(b, j)}^{(b', j')}.
\]
We have used lemma~\ref{L:RelationCoefficientsHandF} (a) in the proof of this formula.

Case (iii): If $(b, j) = (\tau(b'), j')$ and either  $b \in {\mathcal E}^+$, $b' \in {\mathcal E}^-$ or  $b \in {\mathcal E}^-$, $b' \in {\mathcal E}^+$, we define
\[
Z_{(b, j)}^{(b', j')}(v_0^{j''}(b''))= \begin{cases} v_0^{j'}(b'), &\text{ if $(b'', j'') = (b, j)$;} \\  0, &\text{ otherwise.} \end{cases}
\]

This means
\[
\xi_{(b_1, j_1)}^{(b_2,  j_2)} = \begin{cases} 1, &\text{ if $(b_1, j_1) = ( b, j)$ and $(b_2, j_2) = (\tau(b), j)$;}\\  0, &\text{otherwise.} \end{cases}
\]

We get easily that  
\[
\langle Z_{(b, j)}^{(b', j')}(v), v' \rangle_{\mathbf c} + \langle v, Z_{(b, j)}^{(b', j')}(v')\rangle_{\mathbf c} = 0 
\]
 for all $v, v' \in V_0$ from the  equations~(\ref{CoordinateEquationSymplecticOdd}).  
  
By looking on the basis vectors $v_0^{j''}(b'')$, we have that 
\[
ad(H_{\mathbf c}\vert_{V_0}) Z_{(b, j)}^{(b', j')} = (h_0(b') - h_0(b)) Z_{(b, j)}^{(b', j')}.
\]
We have used lemma~\ref{L:RelationCoefficientsHandF} (a) in the proof of this formula.

From this, we get in the symplectic case that 
\[
\dim({\mathfrak p}''_0) =  \frac {1}{2} \left(\sum_{\begin{subarray}{c} (b, b') \in {\mathbf B} \times {\mathbf B}\\ 0 \in Supp(b) \cap Supp(b') \\ b \ne \tau(b')\\ \lambda(b)\geq \lambda(b') \end{subarray}} c(b) c(b')  +
\sum_{\begin{subarray}{c} b \in {\mathbf B}\\ 0 \in Supp(b)\\\lambda(b)\geq \lambda(\tau(b)) \end{subarray}} c(b) (c(b)  + 1)\right).
\]
Note that we have also used the fact that $(h_0(b') - h_0(b)) \geq 0$ is equivalent to $(0 - h_0(b)) \geq (0 - h_0(b'))$; in other words,  $\lambda(b) \geq \lambda(b')$.

We must now consider the orthogonal case. The set ${\mathcal E}$ can be partitionned as ${\mathcal E} = {\mathcal E}^{> 0} \coprod {\mathcal E}^{=0} \coprod {\mathcal E}^{< 0}$, where
\[
\begin{aligned}
{\mathcal E}^{>0} &= \{b \in {\mathcal E} \mid  b(i, j, 0) \text{ where $i, j \in {\mathcal I}$, $i \geq j$ and $(i + j) > 0$} \}\\
{\mathcal E}^{=0} &= \{b \in {\mathcal E} \mid  b(i, j, 0) \text{ where $i, j \in {\mathcal I}$, $i \geq j$ and $(i + j) = 0$} \} \quad \text{ and }\\
{\mathcal E}^{<0} &= \{b \in {\mathcal E} \mid b(i, j, 0) \text{ where $i, j \in {\mathcal I}$, $i \geq j$ and $(i + j) < 0$} \}.
\end{aligned}
\]

${\mathcal E}^{>0}$ consists of the ${\mathcal I}$-boxes strictly above the principal diagonale, ${\mathcal E}^{=0}$ consists of the ${\mathcal I}$-boxes on the principal diagonale, while  ${\mathcal E}^{<0}$ consists of the ${\mathcal I}$-boxes strictly below the principal diagonale. Note that if $b \in {\mathcal E}$, then $\langle v_0^s(b), v_0^s(\tau(b))\rangle_{\mathbf c} = 1$  for any $1 \leq s \leq c(b)$.

The equation~(\ref{E:LieCondition}) is equivalent to the equations: 
\begin{equation}\label{E:CoordinateEquationOrthogonalOdd}
\xi_{(b, j)}^{(b', j')} + \xi_{(b', j')}^{(b, j)} = 0 \quad \text{ if $b, b' \in {\mathcal E}$, $1 \leq j \leq c(b)$ and $1 \leq j' \leq c(b')$.}
\end{equation}

Order the basis ${\mathcal B}_0$ such that the first elements are the basis vectors of the form $v_0^j(b)$ with $b \in {\mathcal E}^{>0}$ and $1 \leq j \leq c(b)$, the following elements are the basis vectors of the form $v_0^j(b)$ with $b \in {\mathcal E}^{=0}$ and $1 \leq j \leq c(b)$ and finally the last elements are the basis vectors of the form $v_0^j(\tau(b))$ with $b \in {\mathcal E}^{>0}$ by keeping the same order of the first basis elements. Relative to this ordered basis, the linear transformation $Z_0$ as above has the form
\[
\begin{pmatrix} A_{11} & A_{12} & A_{13}\\ A_{21} & A_{22} & -A_{12}^T\\ A_{31} & - A_{21}^T & -A_{11}^T \end{pmatrix} \quad \text{$A_{13} = - A_{13}^T$, $A_{22} = - A_{22}^T$ and $A_{31} = - A_{31}^T$, } 
\]  
where $A_{11}$, $A_{13}$, $A_{31}$ are matrices of order $\delta'''_0 \times \delta'''_0$, $A_{21}$ is a matrix of order $\delta''_0 \times \delta'''_0$, $A_{12}$ is a matrix of order $\delta'''_0 \times \delta''_0$ and $A_{22}$ is a matrix of order $\delta''_0 \times \delta''_0$. Here we are using the same notation as in proposition~\ref{Prop_Indexation_Orbits_Odd} for $\delta''_0$ and $\delta'''_0$. 

We will define  linear transformations $Z_{(b, j)}^{(b', j')}:V_0 \rightarrow V_0$ when $b, b' \in {\mathcal E}$,  $1 \leq j \leq c(b)$ and $1 \leq j' \leq c(b')$ for six cases:
\begin{enumerate}[\upshape (i)]
\item either $b, b' \in {\mathcal E}^{>0}$ or $b, b' \in {\mathcal E}^{<0}$;

\item either $b \in {\mathcal E}^{=0}, b' \in {\mathcal E}^{>0}$ or $b \in {\mathcal E}^{< 0}, b' \in {\mathcal E}^{= 0}$;

\item either $b \in {\mathcal E}^{>0}, b' \in {\mathcal E}^{=0}$ or $b \in {\mathcal E}^{= 0}, b' \in {\mathcal E}^{< 0}$;

\item $(b', j') \ne (\tau(b), j)$ with $b \in {\mathcal E}^{< 0}$ and $b' \in {\mathcal E}^{> 0}$;

\item $(b', j') \ne (\tau(b), j)$ with $b, b' \in {\mathcal E}^{= 0}$;

\item $(b', j') \ne (\tau(b), j)$ with $b \in {\mathcal E}^{> 0}$ and $b' \in {\mathcal E}^{< 0}$.

\end{enumerate}

Case (i): If either $b, b' \in {\mathcal E}^{> 0}$ or $b, b' \in {\mathcal E}^{<0}$, we define
\[
Z_{(b, j)}^{(b', j')}(v_0^{j''}(b''))= \begin{cases} {\phantom -}v_0^{j'}(b'), &\text{ if $(b'', j'') = (b, j)$;}\\ -v_0^j(\tau(b)), &\text{ if $(b'', j'') = (\tau(b'), j')$;} \\ {\phantom -} 0, &\text{ otherwise.} \end{cases}
\]

This means
\[
\xi_{(b_1, j_1)}^{(b_2,  j_2)} = \begin{cases} {\phantom -}1, &\text{ if $(b_1, j_1) = ( b, j)$ and $(b_2, j_2) = (\tau(b'), j')$;}\\ -1, &\text{ if $(b_1, j_1) = ( \tau(b'), j')$ and $(b_2, j_2) = (b, j)$;}\\ {\phantom -}0, &\text{otherwise.} \end{cases}
\]

We get easily that  
\[
\langle Z_{(b, j)}^{(b', j')}(v), v' \rangle_{\mathbf c} + \langle v, Z_{(b, j)}^{(b', j')}(v')\rangle_{\mathbf c} = 0 
\]
 for all $v, v' \in V_0$ from the  equations~(\ref{E:CoordinateEquationOrthogonalOdd}).  Note that  the pair $\tau(b'), \tau(b)$ falls also in the  case (i) and it is easy to see that $Z_{(\tau(b'), j')}^{(\tau(b), j)} = - Z_{(b, j)}^{(b', j')}$ by evaluating both sides on the basis vectors $v_0^{j''}(b'')$.
 
 By looking on the basis vectors $v_0^{j''}(b'')$, we have that 
\[
ad(H_{\mathbf c}\vert_{V_0}) Z_{(b, j)}^{(b', j')} = (h_0(b') - h_0(b)) Z_{(b, j)}^{(b', j')}.
\]
We have used lemma~\ref{L:RelationCoefficientsHandF} (a) in the proof of this formula.

Case (ii): If either $b \in {\mathcal E}^{=0}$, $b' \in {\mathcal E}^{> 0}$ or $b\in {\mathcal E}^{<0}$, $b' \in {\mathcal E}^{=0}$, we define
\[
Z_{(b, j)}^{(b', j')}(v_0^{j''}(b''))= \begin{cases} {\phantom -}v_0^{j'}(b'), &\text{ if $(b'', j'') = (b, j)$;}\\ -v_0^j(\tau(b)), &\text{ if $(b'', j'') = (\tau(b'), j')$;} \\ {\phantom -} 0, &\text{ otherwise.} \end{cases}
\]

This means
\[
\xi_{(b_1, j_1)}^{(b_2,  j_2)} = \begin{cases} {\phantom -}1, &\text{ if $(b_1, j_1) = ( b, j)$ and $(b_2, j_2) = (\tau(b'), j')$;}\\ -1, &\text{ if $(b_1, j_1) = ( \tau(b'), j')$ and $(b_2, j_2) = (b, j)$;}\\ {\phantom -}0, &\text{otherwise.} \end{cases}
\]

We get easily that  
\[
\langle Z_{(b, j)}^{(b', j')}(v), v' \rangle_{\mathbf c} + \langle v, Z_{(b, j)}^{(b', j')}(v')\rangle_{\mathbf c} = 0 
\]
 for all $v, v' \in V_0$ from the  equations~(\ref{E:CoordinateEquationOrthogonalOdd}).  Note that  the pair $\tau(b'), \tau(b)$ falls also in the  case (ii) and it is easy to see that $Z_{(\tau(b'), j')}^{(\tau(b), j)} = - Z_{(b, j)}^{(b', j')}$ by evaluating both sides on the basis vectors $v_0^{j''}(b'')$.
 
By looking on the basis vectors $v_0^{j''}(b'')$, we have that 
\[
ad(H_{\mathbf c}\vert_{V_0}) Z_{(b, j)}^{(b', j')} = (h_0(b') - h_0(b)) Z_{(b, j)}^{(b', j')}.
\]
We have used lemma~\ref{L:RelationCoefficientsHandF} (a) in the proof of this formula.

Case (iii): If either $b \in {\mathcal E}^{>0}$, $b' \in {\mathcal E}^{= 0}$ or $b\in {\mathcal E}^{= 0}$, $b' \in {\mathcal E}^{<0}$, we define
\[
Z_{(b, j)}^{(b', j')}(v_0^{j''}(b''))= \begin{cases} {\phantom -}v_0^{j'}(b'), &\text{ if $(b'', j'') = (b, j)$;}\\ -v_0^j(\tau(b)), &\text{ if $(b'', j'') = (\tau(b'), j')$;} \\ {\phantom -} 0, &\text{ otherwise.} \end{cases}
\]

This means
\[
\xi_{(b_1, j_1)}^{(b_2,  j_2)} = \begin{cases} {\phantom -}1, &\text{ if $(b_1, j_1) = ( b, j)$ and $(b_2, j_2) = (\tau(b'), j')$;}\\ -1, &\text{ if $(b_1, j_1) = ( \tau(b'), j')$ and $(b_2, j_2) = (b, j)$;}\\ {\phantom -}0, &\text{otherwise.} \end{cases}
\]

We get easily that  
\[
\langle Z_{(b, j)}^{(b', j')}(v), v' \rangle_{\mathbf c} + \langle v, Z_{(b, j)}^{(b', j')}(v')\rangle_{\mathbf c} = 0 
\]
 for all $v, v' \in V_0$ from the  equations~(\ref{E:CoordinateEquationOrthogonalOdd}).  Note that  the pair $\tau(b'), \tau(b)$ falls also in the  case (iii) and it is easy to see that $Z_{(\tau(b'), j')}^{(\tau(b), j)} = - Z_{(b, j)}^{(b', j')}$ by evaluating both sides on the basis vectors $v_0^{j''}(b'')$.
 
By looking on the basis vectors $v_0^{j''}(b'')$, we have that 
\[
ad(H_{\mathbf c}\vert_{V_0}) Z_{(b, j)}^{(b', j')} = (h_0(b') - h_0(b)) Z_{(b, j)}^{(b', j')}.
\]
We have used lemma~\ref{L:RelationCoefficientsHandF} (a) in the proof of this formula.

Case (iv): If $(b', j') \ne (\tau(b), j)$ with $b \in {\mathcal E}^{<0}$, $b' \in {\mathcal E}^{> 0}$, we define
\[
Z_{(b, j)}^{(b', j')}(v_0^{j''}(b''))= \begin{cases} {\phantom -}v_0^{j'}(b'), &\text{ if $(b'', j'') = (b, j)$;}\\ -v_0^j(\tau(b)), &\text{ if $(b'', j'') = (\tau(b'), j')$;} \\ {\phantom -} 0, &\text{ otherwise.} \end{cases}
\]

This means
\[
\xi_{(b_1, j_1)}^{(b_2,  j_2)} = \begin{cases} {\phantom -}1, &\text{ if $(b_1, j_1) = ( b, j)$ and $(b_2, j_2) = (\tau(b'), j')$;}\\ -1, &\text{ if $(b_1, j_1) = ( \tau(b'), j')$ and $(b_2, j_2) = (b, j)$;}\\ {\phantom -}0, &\text{otherwise.} \end{cases}
\]

We get easily that  
\[
\langle Z_{(b, j)}^{(b', j')}(v), v' \rangle_{\mathbf c} + \langle v, Z_{(b, j)}^{(b', j')}(v')\rangle_{\mathbf c} = 0 
\]
 for all $v, v' \in V_0$ from the  equations~(\ref{E:CoordinateEquationOrthogonalOdd}).  Note that  the pair $\tau(b'), \tau(b)$ falls also in the  case (iv) and it is easy to see that $Z_{(\tau(b'), j')}^{(\tau(b), j)} = - Z_{(b, j)}^{(b', j')}$ by evaluating both sides on the basis vectors $v_0^{j''}(b'')$.
 
By looking on the basis vectors $v_0^{j''}(b'')$, we have that 
\[
ad(H_{\mathbf c}\vert_{V_0}) Z_{(b, j)}^{(b', j')} = (h_0(b') - h_0(b)) Z_{(b, j)}^{(b', j')}.
\]
We have used lemma~\ref{L:RelationCoefficientsHandF} (a) in the proof of this formula.

Case (v): If $(b', j') \ne (\tau(b), j)$ with $b \in {\mathcal E}^{=0}$, $b' \in {\mathcal E}^{= 0}$, we define
\[
Z_{(b, j)}^{(b', j')}(v_0^{j''}(b''))= \begin{cases} {\phantom -}v_0^{j'}(b'), &\text{ if $(b'', j'') = (b, j)$;}\\ -v_0^j(\tau(b)), &\text{ if $(b'', j'') = (\tau(b'), j')$;} \\ {\phantom -} 0, &\text{ otherwise.} \end{cases}
\]

This means
\[
\xi_{(b_1, j_1)}^{(b_2,  j_2)} = \begin{cases} {\phantom -}1, &\text{ if $(b_1, j_1) = ( b, j)$ and $(b_2, j_2) = (\tau(b'), j')$;}\\ -1, &\text{ if $(b_1, j_1) = ( \tau(b'), j')$ and $(b_2, j_2) = (b, j)$;}\\ {\phantom -}0, &\text{otherwise.} \end{cases}
\]

We get easily that  
\[
\langle Z_{(b, j)}^{(b', j')}(v), v' \rangle_{\mathbf c} + \langle v, Z_{(b, j)}^{(b', j')}(v')\rangle_{\mathbf c} = 0 
\]
 for all $v, v' \in V_0$ from the  equations~(\ref{E:CoordinateEquationOrthogonalOdd}).  Note that  the pair $\tau(b'), \tau(b)$ falls also in the  case (v) and it is easy to see that $Z_{(\tau(b'), j')}^{(\tau(b), j)} = - Z_{(b, j)}^{(b', j')}$ by evaluating both sides on the basis vectors $v_0^{j''}(b'')$.
 
By looking on the basis vectors $v_0^{j''}(b'')$, we have that 
\[
ad(H_{\mathbf c}\vert_{V_0}) Z_{(b, j)}^{(b', j')} = (h_0(b') - h_0(b)) Z_{(b, j)}^{(b', j')}.
\]
We have used lemma~\ref{L:RelationCoefficientsHandF} (a) in the proof of this formula.

Case (vi): If $(b', j') \ne (\tau(b), j)$ with $b \in {\mathcal E}^{>0}$, $b' \in {\mathcal E}^{< 0}$, we define
\[
Z_{(b, j)}^{(b', j')}(v_0^{j''}(b''))= \begin{cases} {\phantom -}v_0^{j'}(b'), &\text{ if $(b'', j'') = (b, j)$;}\\ -v_0^j(\tau(b)), &\text{ if $(b'', j'') = (\tau(b'), j')$;} \\ {\phantom -} 0, &\text{ otherwise.} \end{cases}
\]

This means
\[
\xi_{(b_1, j_1)}^{(b_2,  j_2)} = \begin{cases} {\phantom -}1, &\text{ if $(b_1, j_1) = ( b, j)$ and $(b_2, j_2) = (\tau(b'), j')$;}\\ -1, &\text{ if $(b_1, j_1) = ( \tau(b'), j')$ and $(b_2, j_2) = (b, j)$;}\\ {\phantom -}0, &\text{otherwise.} \end{cases}
\]

We get easily that  
\[
\langle Z_{(b, j)}^{(b', j')}(v), v' \rangle_{\mathbf c} + \langle v, Z_{(b, j)}^{(b', j')}(v')\rangle_{\mathbf c} = 0 
\]
 for all $v, v' \in V_0$ from the  equations~(\ref{E:CoordinateEquationOrthogonalOdd}).  Note that  the pair $\tau(b'), \tau(b)$ falls also in the  case (vi) and it is easy to see that $Z_{(\tau(b'), j')}^{(\tau(b), j)} = - Z_{(b, j)}^{(b', j')}$ by evaluating both sides on the basis vectors $v_0^{j''}(b'')$.
 
By looking on the basis vectors $v_0^{j''}(b'')$, we have that 
\[
ad(H_{\mathbf c}\vert_{V_0}) Z_{(b, j)}^{(b', j')} = (h_0(b') - h_0(b)) Z_{(b, j)}^{(b', j')}.
\]
We have used lemma~\ref{L:RelationCoefficientsHandF} (a) in the proof of this formula.

From this, we get in the orthogonal case that 
\[
\dim({\mathfrak p}''_0) =  \frac {1}{2} \left(\sum_{\begin{subarray}{c} (b, b') \in {\mathbf B} \times {\mathbf B}\\ 0 \in Supp(b) \cap Supp(b') \\ b \ne \tau(b')\\ \lambda(b)\geq \lambda(b') \end{subarray}} c(b) c(b')  +
\sum_{\begin{subarray}{c} b \in {\mathbf B}\\ 0 \in Supp(b)\\\lambda(b)\geq \lambda(\tau(b)) \end{subarray}} c(b) (c(b)  - 1)\right).
\]
Note that we have also used the fact that $(h_0(b') - h_0(b)) \geq 0$ is equivalent to $(0 - h_0(b)) \geq (0 - h_0(b'))$; in other words,  $\lambda(b) \geq \lambda(b')$.

From all the formulas above, the proposition follows.
\end{proof}

\begin{proposition}\label{P:Height_Jordan_Type_Odd}
Let $X$ be an element of the $G^{\iota}$-orbit ${\mathcal O}_{\mathbf c}$ where ${\mathbf c}  \in  {\mathfrak C}_{\delta}$ and let $c:{\mathbf B} \rightarrow {\mathbb N}$ the symmetric function corresponding to ${\mathbf c}$.  Then the partition corresponding to the Jordan type of $X$ is 
\[
\prod_{b \in {\mathbf B}} \vert Supp(b) \vert^{c(b)}.
\]
Here we have described the partition in multiplicative form.
\end{proposition}
\begin{proof}
The proof is identical to the one of proposition~\ref{P:Height_Jordan_Type} and is left to the reader.
\end{proof}

\subsection{}
As when $\vert {\mathcal I} \vert$ is even, we want to give now a compact way to describe the symmetric function $c: {\mathbf B} \rightarrow {\mathbb N}$ corresponding to the coefficient function ${\mathbf c} \in  {\mathfrak C}_{\delta}$ but this time when $\vert {\mathcal I} \vert$ is odd. In fact,  we will give two approachs to describe the symmetric function $c: {\mathbf B} \rightarrow {\mathbb N}$.

\begin{definition}
A   {\it  orthogonal symmetric ${\mathcal I}$-tableau} is the ${\mathcal I}$-diagram with its boxes filled with elements of ${\mathbb N}$ such that if $c_{i,j}$ denoted the entry in the row indexed by $i$ and in the column indexed by $j$,   then we have $c_{-j,-i} = c_{i, j}$ for all $i, j \in {\mathcal I}$ and $i \geq j$.  A   {\it  symplectic symmetric ${\mathcal I}$-tableau} is the ${\mathcal I}$-diagram with its boxes filled with elements of ${\mathbb N}$ such that if $c_{i,j}$ denoted the entry in the row indexed by $i$ and in the column indexed by $j$,   then we have $c_{-j,-i} = c_{i, j}$ for all $i, j \in {\mathcal I}$ and $i \geq j$ with the added condition that $c_{i, -i}$ is an even integer for all $i \in {\mathcal I}$, $i \geq 0$ (i.e. the entries on the principal diagonal are even). 

Note this is different from how we define orthogonal and symplectic symmetric tableau when $\vert {\mathcal I}\vert$ is even. The two notions are transposed here.
\end{definition}

\begin{example} \label{ExampleYoungDiagramOddCase}
For the ${\mathcal I} = \{4, 2, 0, -2, -4\}$, then 
\begin{center}
${\mathcal T}$ = 
\ytableausetup{centertableaux}
\begin{ytableau}
\none & \none [-4] & \none [-2] & \none [0] & \none [2] & \none [4] \\ \none [{\phantom -}4] & 0 & 1 & 2 & 0 & 0\\ \none [{\phantom -}2] & 1 & 1 & 3 & 1\\  \none [{\phantom -}0] & 2 & 3 & 2 \\ \none [-2] & 0 & 1\\ \none [-4] & 0 \\
\end{ytableau}
\end{center}
is an orthogonal symmetric ${\mathcal I}$-tableau, but not a symplectic symmetric ${\mathcal I}$-tableau.  
\begin{center}
${\mathcal T}'$ = 
\ytableausetup{centertableaux}
\begin{ytableau}
\none & \none [-4] & \none [-2] & \none [0] & \none [2] & \none [4] \\ \none [{\phantom -}4] & 2 & 1 & 3 & 0 & 1\\ \none [{\phantom -}2] & 1 & 0 & 2 & 1\\  \none [{\phantom -}0] & 3 & 2 & 2 \\ \none [-2] & 0 & 1\\ \none [-4] & 1 \\
\end{ytableau}
\end{center}
is a symplectic symmetric ${\mathcal I}$-tableau. We have kept the indices for the rows and columns  inscribed in both cases.
\end{example}

\begin{definition}
Given the symplectic (respectively orthogonal) symmetric ${\mathcal I}$-tableau ${\mathcal T} = (c_{i, j})_{i, j \in {\mathcal I}, i \geq j}$, we define its {\it dimension vector} $\delta({\mathcal T}) = (\delta_i)_{i \in {\mathcal I}} \in {\mathbb N}^{\mathcal I}$, where
\[
\delta_i = \sum_{\begin{subarray}{c} u, v \in {\mathcal I}\\ u \geq i \geq v \end{subarray}} c_{u, v} \quad \text{for all $i \in {\mathcal I}$.}
\]
\end{definition}

\begin{example} 
If ${\mathcal T}$ and ${\mathcal T}'$  are the symmetric ${\mathcal I}$-tableaux of example~\ref{ExampleYoungDiagramOddCase}, then their dimension vectors are  
\[
\begin{aligned}
\delta({\mathcal T}) &= (\delta_{4}, \delta_2, \delta_0, \delta_{-2}, \delta_{-4}) =  (3, 9, 15, 9, 3) \quad \text{ and }\\
\delta({\mathcal T}') &= (\delta_{4}, \delta_2, \delta_0, \delta_{-2}, \delta_{-4}) =  (7, 10, 16, 10, 7)
\end{aligned}
\]
\end{example}

\begin{lemma}
Let $\delta = (\delta_i)_{i \in {\mathcal I}} \in {\mathbb N}^{\mathcal I}$ with $\delta_{-i} = \delta_i$ for all $i \in {\mathcal I}$. Denote by ${\mathfrak T}_{\delta}$: the set of symplectic (respectively orthogonal) symmetric ${\mathcal I}$-tableaux ${\mathcal T}$ of dimension $\delta({\mathcal T}) = \delta = (\delta_i)_{i \in {\mathcal I}}$ and by ${\mathfrak R}_{\delta}$: the set of symplectic (respectively orthogonal) symmetric ${\mathcal I}$-tableaux ${\mathcal R} = (r_{i, j})_{i, j \in I, i \geq j}$ such that 
\[
r_{i, i} = \delta_i \quad \text{ and } \quad (r_{i, j} - r_{i + 2, j} - r_{i , j - 2} + r_{i + 2, j - 2}) \geq 0 \quad \text{ for all $i, j \in {\mathcal I}$ and $i \geq j$}
\]
where we set $r_{i, j} = 0$ if either $i \not\in {\mathcal I}$ or $j \not\in {\mathcal I}$ or both. Then the map $\varTheta:{\mathfrak T}_{\delta} \longrightarrow {\mathfrak R}_{\delta}$
\[
\varTheta({\mathcal T}) = \varTheta((c_{i, j})_{i, j \in {\mathcal I}, i \geq j}) = (r_{i, j})_{i, j \in I, i \geq j} \quad \text{where $r_{i, j} = \sum_{\begin{subarray}{c} u, v \in {\mathcal I}\\ u \geq i \geq j \geq v\end{subarray}} c_{u, v}$}  
\]
for all $i , j \in {\mathcal I}$, $i \geq j$,  is a well-defined bijection.
\end{lemma}
\begin{proof}
The proof is similar to the proof of lemma~\ref{L:RankComponentTableau} and is left to the reader.
\end{proof}

\begin{example}
If ${\mathcal T}$ and ${\mathcal T}'$ are the symmetric ${\mathcal I}$-tableau  of  example~\ref{ExampleYoungDiagramOddCase}. Then
\begin{center}
${\varTheta(\mathcal T})$ = 
\ytableausetup{centertableaux}
\begin{ytableau}
\none & \none [-4] & \none [-2] & \none [0] & \none [2] & \none [4] \\ \none [{\phantom -}4] & 0 & 1 & 3 & 3 & 3\\ \none [{\phantom -}2] & 1 & 3 & 8 & 9\\  \none [{\phantom -}0] & 3 & 8 & 15 \\ \none [-2] & 3 & 9\\ \none [-4] & 3 \\
\end{ytableau}
\text{ and }
${\varTheta(\mathcal T}')$ = 
\ytableausetup{centertableaux}
\begin{ytableau}
\none & \none [-4] & \none [-2] & \none [0] & \none [2] & \none [4] \\ \none [{\phantom -}4] & 2 & 3 & 6 & 6 & 7\\ \none [{\phantom -}2] & 3 & 4 & 9 & 10\\  \none [{\phantom -}0] & 6 & 9 & 16 \\ \none [-2] & 6 & 10\\ \none [-4] & 7 \\
\end{ytableau}
\end{center}
 \end{example}

\begin{definition}\label{D:SymmetricTableauOddSymplectic}
Let $\delta = (\delta_i)_{i \in {\mathcal I}} \in {\mathbb N}^{\mathcal I}$ with $\delta_{-i} = \delta_i$ for all $i \in {\mathcal I}$ and  the coefficient function ${\mathbf c} \in  {\mathfrak C}_{\delta}$.  To ${\mathbf c}$, denote the corresponding symmetric function by $c: {\mathbf B} \rightarrow {\mathbb N}$. We define the symmetric ${\mathcal I}$-tableau ${\mathcal T}({\mathbf c})$ as follows: if $c_{i, j}$ is the entry at the row $i$ and column $j$ with $i, j \in {\mathcal I}$ and $i \geq j$, then 
\[
c_{i, j} = \begin{cases} c(b(i, j, 0)), &\text{ if $i + j \ne 0$;}\\  c(b(i, j, 0)), &\text{ if $i + j = 0$ in the orthogonal case;}\\c(b(i, j, 0)) + c(b(i, j, 1)), &\text{ if $i + j = 0$ in the symplectic case.} \end{cases}
\]

Because $c: {\mathbf B} \rightarrow {\mathbb N}$ is symmetric,  ${\mathcal T}({\mathbf c})$ is clearly a symmetric ${\mathcal I}$-tableau. Moreover in the symplectic case, because  we have that $c(b(i, j, 0)) = c(b(i, j, 1))$ when $i + j = 0$, we get that $c_{i, -i}$ is even. Thus ${\mathcal T}({\mathbf c})$ is an  symplectic symmetric ${\mathcal I}$-tableau.  

We have 
\[
\begin{aligned}
 \sum_{\begin{subarray}{c} u, v \in {\mathcal I}\\ u \geq i \geq v \end{subarray}} c_{u, v} &=  \sum_{\begin{subarray}{c} u, v \in {\mathcal I}\\ u \geq i \geq v \\ u + v \ne 0 \end{subarray}} c(b(u, v, 0))\\ & \quad\quad +  \begin{cases} \displaystyle\sum_{\begin{subarray}{c} u \in {\mathcal I}\\ u \geq i \geq -u \end{subarray}} (c(b(u, -u, 0)) + c(b(u, -u, 1))) \text{ in the symplectic case;} \\ \\ \displaystyle\sum_{\begin{subarray}{c} u \in {\mathcal I}\\ u \geq i \geq -u \end{subarray}} c(b(u, -u, 0))  \text{ in the orthogonal case;} \end{cases} \\ &= \sum_{\begin{subarray}{c} b \in {\mathbf B}\\ i \in Supp(b) \end{subarray}} c(b) = \delta_i.
 \end{aligned}
 \] 
Consequently $\delta({\mathcal T}({\mathbf c})) = \delta$ and ${\mathcal T}({\mathbf c}) \in {\mathfrak T}_{\delta}$. 
\end{definition}

\begin{lemma}\label{L:TableauOrbitOdd}
Let $\delta = (\delta_i)_{i \in {\mathcal I}} \in {\mathbb N}^{\mathcal I}$ with $\delta_{-i} = \delta_i$ for all $i \in {\mathcal I}$ and let ${\mathfrak g}_2/G^{\iota}$ denotes the set of  $G^{\iota}$-orbits  in ${\mathfrak g}_2$. Then
\begin{enumerate}[\upshape (a)]
\item the map ${\mathcal T}: {\mathfrak C}_{\delta} \rightarrow {\mathfrak T}_{\delta}$, where   ${\mathbf c} \mapsto {\mathcal T}({\mathbf c})$ is a bijection;
\item the map ${\mathfrak g}_2/G^{\iota} \rightarrow  {\mathfrak T}_{\delta} \quad \text{ defined by } \quad G^{\iota} \cdot X \mapsto {\mathcal T}({\mathbf c}_X)$ is a well-defined bijection. 
\end{enumerate}
\end{lemma}
\begin{proof}
(a) is obvious from the definition of ${\mathcal T}$ above and (b) is a corollary of proposition~\ref{Prop_Indexation_Orbits_Odd}
\end{proof}

\begin{notation}
We will denote the $G^{\iota}$-orbit  in ${\mathfrak g}_2$ corresponding to the ${\mathcal I}$-tableau ${\mathcal T}  \in {\mathfrak T}_{\delta}$ as defined in the bijection of lemma~\ref{L:TableauOrbitOdd} (b)  by ${\mathcal O}_{\mathcal T}$
\end{notation}

In \cite{BI2021},  M. Boos and G. Cerulli Irelli, were able to describe the partial order on the $G^{\iota}$-orbits in ${\frak g}_2$ for symmetric quivers in the orthogonal and symplectic cases for  symmetric quivers of type $A_m^{even}$ and $A_m^{odd}$. We will now describe their result in our case (i.e. $\vert {\mathcal I} \vert$ odd).

\begin{proposition}[Boos and Cerulli Irelli]\label{P:PartialOrderIOddSymplectic}
(In the situation of \ref{S:SetUpOdd}) Let  ${\mathcal I}$-tableaux ${\mathcal T}$, ${\mathcal T}' \in  {\mathfrak T}_{\delta}$ and  the corresponding $G^{\iota}$-orbits ${\mathcal O}_{\mathcal T}$ , ${\mathcal O}_{{\mathcal T}'}$ in ${\frak g}_2$. Then the  $G^{\iota}$-orbit  ${\mathcal O}_{\mathcal T}$ is contained in the Zariski closure of the   $G^{\iota}$-orbit ${\mathcal O}_{{\mathcal T}'}$ if and only if
\[
(r_{i, j})_{i, j \in {\mathcal I}, i \geq j} = {\mathcal R} = \varTheta({\mathcal T}) \leq \varTheta({\mathcal T}') = {\mathcal R}' = (r'_{i, j})_{i, j \in {\mathcal I}, i \geq j}
\]
where the inequality here between the tableaux ${\mathcal R}$ and ${\mathcal R}'$ means $r_{i, j} \leq r'_{i, j}$ for all $i, j \in {\mathcal I}$, $i \geq j$. 
\end{proposition}
\begin{proof}
The proof is similar to the proof of the proposition~\ref{P:PartialOrderIEven} and is left to the reader.
\end{proof}

\begin{example}\label{E:5.28}
With the notation of \ref{S:SetUpOdd}, let $m = 3$, $G$ be the group of automorphisms of the finite dimensional vector space $V$ over ${\mathbf k}$ preserving  a fixed non-degenerate skew-symmetric form $\langle\ , \  \rangle:V \times V \rightarrow {\mathbf k}$,   ${\mathcal I} = \{4, 2, 0, -2, -4\}$, $\dim(V_0) = 4 = \delta_0$, $\dim(V_2) = 2 = \delta_2 = \delta_{-2}$, $\dim(V_4) = 2 = \delta_4 = \delta_{-4}$. In other words,  $G = Sp_{12}({\mathbf k})$ and $\iota(t)$ on $V_i$ for $i \in {\mathcal I}$ is given by $\iota(t) v= t^i v$ for all $v \in V_i$ and $t\in {\mathbf k}^{\times}$. In this case, $\dim({\mathfrak g}_2) = 12$ and there are 13 orbits. We will write for each of these orbits:  the tableau ${\mathcal T} \in {\mathfrak T}_{\delta}$, the tableau ${\mathcal R} \in  {\mathfrak R}_{\delta}$ and its dimension. 

\begin{table}[h]
	\begin{center}\renewcommand{\arraystretch}{1.25}
		\begin{tabular} {| l || r  | r  | r  | r  | r |  r | r | r  | r | c |}
			\hline
			Orbit &  $100$  & $010$ & $002$ & $110$ & $022$ & $012$ & $112$ & $122$ & $222$ & Jordan Dec.  \\ \hline
			${\mathcal O}_{12}^1$ & $0$ & $0$ & $1$ & $0$ & $0$ & $0$ & $0$ &  $0$ & $1$ &$1^2 5^2$\\ \hline
			${\mathcal O}_{11}^1$ & $1$ & $0$ & $1$ & $0$ & $0$ & $0$ & $0$  & $1$ & $0$ &$1^4 4^2$\\ \hline
			${\mathcal O}_{11}^2$ & $0$ & $0$ & $0$ & $0$ & $0$ & $0$ & $2$  & $0$ & $0$ &$3^4$\\ \hline
			${\mathcal O}_{10}^1$ & $1$ & $0$ & $0$ & $0$ & $0$ & $1$ & $1$  & $0$ & $0$ &$1^2 2^2 3^2$\\ \hline
			${\mathcal O}_{9}^1$ & $0$ & $0$ & $1$ & $1$ & $0$ & $0$ & $1$  & $0$ & $0$ &$1^2 2^2 3^2$\\ \hline
			${\mathcal O}_{8}^1$ & $2$ & $0$ & $1$ & $0$ & $1$ & $0$ & $0$  & $0$ & $0$ &$1^6 3^2$\\ \hline
			${\mathcal O}_{8}^2$ & $1$ & $1$ & $1$ & $0$ & $0$ & $0$ & $1$  & $0$ & $0$ &$1^6 3^2$\\ \hline
			${\mathcal O}_{7}^1$ & $2$ & $0$ & $0$ & $0$ & $0$ & $2$ & $0$  & $0$ & $0$ &$1^4 2^4$\\ \hline
			${\mathcal O}_{7}^2$ & $1$ & $0$ & $1$ & $1$ & $0$ & $1$ & $0$  & $0$ & $0$ &$1^4 2^4$\\ \hline
			${\mathcal O}_{5}^1$ & $2$ & $1$ & $1$ & $0$ & $0$ & $1$ & $0$  & $0$ & $0$ &$1^8 2^2$\\ \hline
			${\mathcal O}_{4}^1$ & $0$ & $0$ & $2$ & $2$ & $0$ & $0$ & $0$  & $0$ & $0$ &$1^4 2^4$\\ \hline
			${\mathcal O}_{3}^1$ & $1$ & $1$ & $2$ & $1$ & $0$ & $0$ & $0$  & $0$ & $0$ &$1^8 2^2$\\ \hline
			${\mathcal O}_{0}^1$ & $2$ & $2$ & $2$ & $0$ & $0$ & $0$ & $0$  & $0$ & $0$ &$1^{12}$\\ \hline
		\end{tabular}
	\end{center}
	\caption{List of the values of the coefficient function ${\mathbf c}$.}\label{T:Table3} 	
\end{table}
The columns of the  table~\ref{T:Table3}  are indexed by the  dimension of the indecomposable symplectic representation of $A_3^{odd}$ and Jordan type for each orbit.	
The subscript in the notation for each orbit is its dimension and the upperscript is there to distinguish between the different orbits of the same dimension.
Now  the tableau ${\mathcal T} \in {\mathfrak T}_{\delta}$, the tableau ${\mathcal R} \in  {\mathfrak R}_{\delta}$ for each orbit are
\begin{itemize}
\item For the orbit ${\mathcal O}_{12}^1$ of dimension 12

\begin{center}
${\mathcal T}_{12}^1$ = 
\ytableausetup{centertableaux}
\begin{ytableau}
\none & \none [-4] & \none [-2] & \none [0] & \none [2] & \none [4]\\ \none [{\phantom -}4] & 2 & 0 & 0 & 0 & 0\\ \none [{\phantom -}2] & 0 & 0 & 0 & 0\\  \none [{\phantom -}0] & 0 & 0 & 2\\ \none [-2] & 0 & 0 \\ \none [-4] & 0 \\ \end{ytableau}
\quad
\text{and}
\quad
${\mathcal R}_{12}^1$ = 
\ytableausetup{centertableaux}
\begin{ytableau}
\none & \none [-4] & \none [-2] & \none [0] & \none [2] & \none [4]\\ \none [{\phantom -}4] & 2 & 2 & 2 & 2 & 2\\ \none [{\phantom -}2] & 2 & 2 & 2 & 2\\  \none [{\phantom -}0] & 2 & 2 & 4\\ \none [-2] & 2 & 2 \\ \none [-4] & 2 \\ \end{ytableau}.

\end{center}

\item For the orbit ${\mathcal O}_{11}^1$ of dimension 11

\begin{center}
${\mathcal T}_{11}^1$ = 
\ytableausetup{centertableaux}
\begin{ytableau}
\none & \none [-4] & \none [-2] & \none [0] & \none [2] & \none [4]\\ \none [{\phantom -}4] & 0 & 1 & 0 & 0 & 1\\ \none [{\phantom -}2] & 1 & 0 & 0 & 0\\  \none [{\phantom -}0] & 0 & 0 & 2\\ \none [-2] & 0 & 0 \\ \none [-4] & 1 \\ \end{ytableau}
\quad
\text{and}
\quad
${\mathcal R}_{11}^1$ = 
\ytableausetup{centertableaux}
\begin{ytableau}
\none & \none [-4] & \none [-2] & \none [0] & \none [2] & \none [4]\\ \none [{\phantom -}4] & 0 & 1 & 1 & 1 & 2\\ \none [{\phantom -}2] & 1 & 2 & 2 & 2\\  \none [{\phantom -}0] & 1 & 2 & 4\\ \none [-2] & 1 & 2 \\ \none [-4] & 2 \\ \end{ytableau}.

\end{center}

\item For the orbit ${\mathcal O}_{11}^2$ of dimension 11

\begin{center}
${\mathcal T}_{11}^2$ = 
\ytableausetup{centertableaux}
\begin{ytableau}
\none & \none [-4] & \none [-2] & \none [0] & \none [2] & \none [4]\\ \none [{\phantom -}4] & 0 & 0 & 2 & 0 & 0\\ \none [{\phantom -}2] & 0 & 0 & 0 & 0\\  \none [{\phantom -}0] & 2 & 0 & 0\\ \none [-2] & 0 & 0 \\ \none [-4] & 0 \\ \end{ytableau}
\quad
\text{and}
\quad
${\mathcal R}_{11}^2$ = 
\ytableausetup{centertableaux}
\begin{ytableau}
\none & \none [-4] & \none [-2] & \none [0] & \none [2] & \none [4]\\ \none [{\phantom -}4] & 0 & 0 & 2 & 2 & 2\\ \none [{\phantom -}2] & 0 & 0 & 2 & 2\\  \none [{\phantom -}0] & 2 & 2 & 4\\ \none [-2] & 2 & 2 \\ \none [-4] & 2 \\ \end{ytableau}.

\end{center}

\item For the orbit ${\mathcal O}_{10}^1$ of dimension 10

\begin{center}
${\mathcal T}_{10}^1$ = 
\ytableausetup{centertableaux}
\begin{ytableau}
\none & \none [-4] & \none [-2] & \none [0] & \none [2] & \none [4]\\ \none [{\phantom -}4] & 0 & 0 & 1 & 0 & 1\\ \none [{\phantom -}2] & 0 & 0 & 1 & 0\\  \none [{\phantom -}0] & 1 & 1 & 0\\ \none [-2] & 0 & 0 \\ \none [-4] & 1 \\ \end{ytableau}
\quad
\text{and}
\quad
${\mathcal R}_{10}^1$ = 
\ytableausetup{centertableaux}
\begin{ytableau}
\none & \none [-4] & \none [-2] & \none [0] & \none [2] & \none [4]\\ \none [{\phantom -}4] & 0 & 0 & 1 & 1 & 2\\ \none [{\phantom -}2] & 0 & 0 & 2 & 2\\  \none [{\phantom -}0] & 1 & 2 & 4\\ \none [-2] & 1 & 2 \\ \none [-4] & 2 \\ \end{ytableau}.

\end{center}

\item For the orbit ${\mathcal O}_{9}^1$ of dimension 9

\begin{center}
${\mathcal T}_{9}^1$ = 
\ytableausetup{centertableaux}
\begin{ytableau}
\none & \none [-4] & \none [-2] & \none [0] & \none [2] & \none [4]\\ \none [{\phantom -}4] & 0 & 0 & 1 & 1 & 0\\ \none [{\phantom -}2] & 0 & 0 & 0 & 0\\  \none [{\phantom -}0] & 1 & 0 & 2\\ \none [-2] & 1 & 0 \\ \none [-4] & 0 \\ \end{ytableau}
\quad
\text{and}
\quad
${\mathcal R}_{9}^1$ = 
\ytableausetup{centertableaux}
\begin{ytableau}
\none & \none [-4] & \none [-2] & \none [0] & \none [2] & \none [4]\\ \none [{\phantom -}4] & 0 & 0 & 1 & 2 & 2\\ \none [{\phantom -}2] & 0 & 0 & 1 & 2\\  \none [{\phantom -}0] & 1 & 1 & 4\\ \none [-2] & 2 & 2 \\ \none [-4] & 2 \\ \end{ytableau}.

\end{center}

\item For the orbit ${\mathcal O}_{8}^1$ of dimension 8

\begin{center}
${\mathcal T}_{8}^1$ = 
\ytableausetup{centertableaux}
\begin{ytableau}
\none & \none [-4] & \none [-2] & \none [0] & \none [2] & \none [4]\\ \none [{\phantom -}4] & 0 & 0 & 0 & 0 & 2\\ \none [{\phantom -}2] & 0 & 2 & 0 & 0\\  \none [{\phantom -}0] & 0 & 0 & 2\\ \none [-2] & 0 & 0 \\ \none [- 4] & 2 \\ \end{ytableau}
\quad
\text{and}
\quad
${\mathcal R}_{8}^1$ = 
\ytableausetup{centertableaux}
\begin{ytableau}
\none & \none [-4] & \none [-2] & \none [0] & \none [2] & \none [4]\\ \none [{\phantom -}4] & 0 & 0 & 0 & 0 & 2\\ \none [{\phantom -}2] & 0 & 2 & 2 & 2\\  \none [{\phantom -}0] & 0 & 2 & 4\\ \none [-2] & 0 & 2 \\ \none [-4] & 2 \\ \end{ytableau}.

\end{center}

\item For the orbit ${\mathcal O}_{8}^2$ of dimension 8

\begin{center}
${\mathcal T}_{8}^2$ = 
\ytableausetup{centertableaux}
\begin{ytableau}
\none & \none [-4] & \none [-2] & \none [0] & \none [2] & \none [4]\\ \none [{\phantom -}4] & 0 & 0 & 1 & 0 & 1\\ \none [{\phantom -}2] & 0 & 0 & 0 & 1\\  \none [{\phantom -}0] & 1 & 0 & 2\\ \none [-2] & 0 & 1 \\ \none [-4] & 1 \\ \end{ytableau}
\quad
\text{and}
\quad
${\mathcal R}_{8}^2$ = 
\ytableausetup{centertableaux}
\begin{ytableau}
\none & \none [-4] & \none [-2] & \none [0] & \none [2] & \none [4]\\ \none [{\phantom -}4] & 0 & 0 & 1 & 1 & 2\\ \none [{\phantom -}2] & 0 & 0 & 1 & 2\\  \none [{\phantom -}0] & 1 & 1 & 4\\ \none [-2] & 1 & 2 \\ \none [-4] & 2 \\ \end{ytableau}.

\end{center}

\item For the orbit ${\mathcal O}_{7}^1$ of dimension 7

\begin{center}
${\mathcal T}_{7}^1$ = 
\ytableausetup{centertableaux}
\begin{ytableau}
\none & \none [-4] & \none [-2] & \none [0] & \none [2] & \none [4]\\ \none [{\phantom -}4] & 0 & 0 & 0 & 0 & 2\\ \none [{\phantom -}2] & 0 & 0 & 2 & 0\\  \none [{\phantom -}0] & 0 & 2 & 0\\ \none [-2] & 0 & 0 \\ \none [-4] & 2 \\ \end{ytableau}
\quad
\text{and}
\quad
${\mathcal R}_{7}^1$ = 
\ytableausetup{centertableaux}
\begin{ytableau}
\none & \none [-4] & \none [-2] & \none [0] & \none [2] & \none [4]\\ \none [{\phantom -}4] & 0 & 0 & 0 & 0 & 2\\ \none [{\phantom -}2] & 0 & 0 & 2 & 2\\  \none [{\phantom -}0] & 0 & 2 & 4\\ \none [-2] & 0 & 2 \\ \none [-4] & 2 \\ \end{ytableau}.

\end{center}

\item For the orbit ${\mathcal O}_{7}^2$ of dimension 7

\begin{center}
${\mathcal T}_{7}^2$ = 
\ytableausetup{centertableaux}
\begin{ytableau}
\none & \none [-4] & \none [-2] & \none [0] & \none [2] & \none [4]\\ \none [{\phantom -}4] & 0 & 0 & 0 & 1& 1\\ \none [{\phantom -}2] & 0 & 0 & 1 & 0\\  \none [{\phantom -}0] & 0 & 1 & 2\\ \none [-2] & 1 & 0 \\ \none [-4] & 1 \\ \end{ytableau}
\quad
\text{and}
\quad
${\mathcal R}_{7}^2$ = 
\ytableausetup{centertableaux}
\begin{ytableau}
\none & \none [-4] & \none [-2] & \none [0] & \none [2] & \none [4]\\ \none [{\phantom -}4] & 0 & 0 & 0 & 1 & 2\\ \none [{\phantom -}2] & 0 & 0 & 1 & 2\\  \none [{\phantom -}0] & 0 & 1 & 4\\ \none [-2] & 1 & 2 \\ \none [-4] & 2 \\ \end{ytableau}.

\end{center}

\item For the orbit ${\mathcal O}_{5}^1$ of dimension 5

\begin{center}
${\mathcal T}_{5}^1$ = 
\ytableausetup{centertableaux}
\begin{ytableau}
\none & \none [-4] & \none [-2] & \none [0] & \none [2] & \none [4]\\ \none [{\phantom -}4] & 0 & 0 & 0 & 0 & 2\\ \none [{\phantom -}2] & 0 & 0 & 1 & 1\\  \none [{\phantom -}0] & 0 & 1 & 2\\ \none [-2] & 0 & 1 \\ \none [-4] & 2 \\ \end{ytableau}
\quad
\text{and}
\quad
${\mathcal R}_{5}^1$ = 
\ytableausetup{centertableaux}
\begin{ytableau}
\none & \none [-4] & \none [-2] & \none [0] & \none [2] & \none [4]\\ \none [{\phantom -}4] & 0 & 0 & 0 & 0 & 2\\ \none [{\phantom -}2] & 0 & 0 & 1 & 2\\  \none [{\phantom -}0] & 0 & 1 & 4\\ \none [-2] & 0 & 2 \\ \none [-4] & 2 \\ \end{ytableau}.

\end{center}

\item For the orbit ${\mathcal O}_{4}^1$ of dimension 4

\begin{center}
${\mathcal T}_{4}^1$ = 
\ytableausetup{centertableaux}
\begin{ytableau}
\none & \none [-4] & \none [-2] & \none [0] & \none [2] & \none [4]\\ \none [{\phantom -}4] & 0 & 0 & 0 & 2 & 0\\ \none [{\phantom -}2] & 0 & 0 & 0 & 0\\  \none [{\phantom -}0] & 0 & 0 & 4\\ \none [-2] & 2 & 0 \\ \none [-4] & 0 \\ \end{ytableau}
\quad
\text{and}
\quad
${\mathcal R}_{4}^1$ = 
\ytableausetup{centertableaux}
\begin{ytableau}
\none & \none [-4] & \none [-2] & \none [0] & \none [2] & \none [4]\\ \none [{\phantom -}4] & 0 & 0 & 0 & 2 & 2\\ \none [{\phantom -}2] & 0 & 0 & 0 & 2\\  \none [{\phantom -}0] & 0 & 0 & 4\\ \none [-2] & 2 & 2 \\ \none [-4] & 2 \\ \end{ytableau}.

\end{center}

\item For the orbit ${\mathcal O}_{3}^1$ of dimension 3

\begin{center}
${\mathcal T}_{3}^1$ = 
\ytableausetup{centertableaux}
\begin{ytableau}
\none & \none [-4] & \none [-2] & \none [0] & \none [2] & \none [4]\\ \none [{\phantom -}4] & 0 & 0 & 0 & 1 & 1\\ \none [{\phantom -}2] & 0 & 0 & 0 & 1\\  \none [{\phantom -}0] & 0 & 0 & 4\\ \none [-2] & 1 & 1 \\ \none [-4] & 1 \\ \end{ytableau}
\quad
\text{and}
\quad
${\mathcal R}_{3}^1$ = 
\ytableausetup{centertableaux}
\begin{ytableau}
\none & \none [-4] & \none [-2] & \none [0] & \none [2] & \none [4]\\ \none [{\phantom -}4] & 0 & 0 & 0 & 1 & 2\\ \none [{\phantom -}2] & 0 & 0 & 0 & 2\\  \none [{\phantom -}0] & 0 & 0 & 4\\ \none [-2] & 1 & 2 \\ \none [-4] & 2 \\ \end{ytableau}.

\end{center}

\item For the orbit ${\mathcal O}_{0}^1$ of dimension 0

\begin{center}
${\mathcal T}_{0}^1$ = 
\ytableausetup{centertableaux}
\begin{ytableau}
\none & \none [-4] & \none [-2] & \none [0] & \none [2] & \none [4]\\ \none [{\phantom -}4] & 0 & 0 & 0 & 0 & 2\\ \none [{\phantom -}2] & 0 & 0 & 0 & 2\\  \none [{\phantom -}0] & 0 & 0 & 4\\ \none [-2] & 0 & 2 \\ \none [-4] & 2 \\ \end{ytableau}
\quad
\text{and}
\quad
${\mathcal R}_{0}^1$ = 
\ytableausetup{centertableaux}
\begin{ytableau}
\none & \none [-4] & \none [-2] & \none [0] & \none [2] & \none [4]\\ \none [{\phantom -}4] & 0 & 0 & 0 & 0 & 2\\ \none [{\phantom -}2] & 0 & 0 & 0 & 2\\  \none [{\phantom -}0] & 0 & 0 & 4\\ \none [-2] & 0 & 2 \\ \none [-4] & 2 \\ \end{ytableau}.

\end{center}

\end{itemize}

\begin{figure}[h]
\begin{center}
\begin{tikzpicture}
\draw (0, 0) -- (2, 3);
\draw (0, 0) -- (-2, 5);
\draw (2, 3) -- (6, 4);
\draw  (2, 3) -- (2, 7);
\draw (6, 4) -- (2, 9);
\draw (-2, 5) -- (-2, 7);
\draw (-2, 5) -- (2, 7);
\draw (-2, 7) -- (-4.5, 8);
\draw (-2, 7) -- (-2, 10); 
\draw (2, 7) -- (2, 8);
\draw (-4.5, 8) -- (-2, 11);
\draw (2, 8) -- (-2, 10);
\draw (2, 8) -- (2, 9);
\draw (2, 9) -- (2, 11);
\draw (-2, 10) -- (-2, 11);
\draw (-2, 10) -- (2, 11);
\draw (-2, 11) -- (0, 12);
\draw (2, 11) -- (0, 12);
\draw (0, 0) node[below] {${\mathcal O}_0^1$};
\draw (2, 3) node[left] {${\mathcal O}_3^1$};
\draw (6, 4) node[right] {${\mathcal O}_4^1$};
\draw (-2, 5) node[left] {${\mathcal O}_5^1$};
\draw (-2, 7) node[right] {${\mathcal O}_7^1$};
\draw (-4.5, 8) node[left] {${\mathcal O}_8^1$};
\draw (2, 7) node[right] {${\mathcal O}_7^2$};
\draw (2, 8) node[right] {${\mathcal O}_8^2$};
\draw (2, 9) node[right] {${\mathcal O}_9^1$};
\draw (-2, 10) node[left] {${\mathcal O}_{10}^1$};
\draw (-2, 11) node[left] {${\mathcal O}_{11}^1$};
\draw (2, 11) node[right] {${\mathcal O}_{11}^2$};
\draw (0, 12) node[above] {${\mathcal O}_{12}^1$};
\draw (0, 0) node {$\bullet$};
\draw (2, 3) node {$\bullet$};
\draw (6, 4) node {$\bullet$};
\draw (-2, 5) node {$\bullet$};
\draw (-2, 7) node {$\bullet$};
\draw (2, 7) node {$\bullet$};
\draw (-4.5, 8) node {$\bullet$};
\draw (2, 8) node {$\bullet$};
\draw (2, 9) node {$\bullet$};
\draw (-2, 10) node {$\bullet$};
\draw (-2, 11) node {$\bullet$};
\draw (2, 11) node {$\bullet$};
\draw (0, 12) node {$\bullet$};
\end{tikzpicture}
\end{center}
\caption{Poset obtained by the rank tableaux comparaison}\label{FigurePosetExample5.21}
\end{figure}
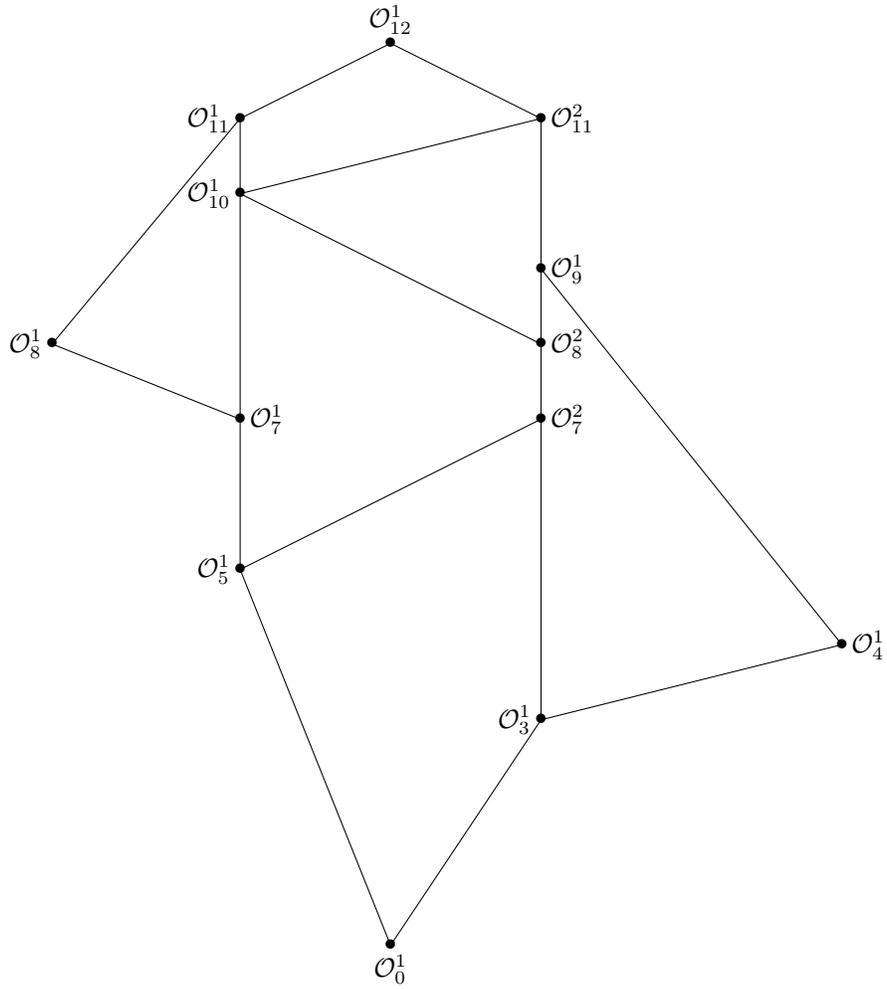

We have illustrated the Hasse diagram of the poset obtained when we compare the rank tableaux ${\mathcal R}$ for the set of orbits  in figure~\ref{FigurePosetExample5.21}.
\end{example}

\begin{example}\label{E:4.27}
With the notation of  \ref{S:SetUpOdd}, let $m = 3$, $G$ be the group of automorphisms of the finite dimensional vector space $V$ over ${\mathbf k}$ preserving  a fixed non-degenerate symmetric form $\langle\ , \  \rangle:V \times V \rightarrow {\mathbf k}$,   ${\mathcal I} = \{4, 2, 0, -2, -4\}$, $\dim(V_0) = 3 = \delta_0$, $\dim(V_2) = 2 = \delta_2 = \delta_{-2}$, $\dim(V_4) = 1 = \delta_4 = \delta_{-4}$. In other words,  $G = O_{9}({\mathbf k})$ and $\iota(t)$ on $V_i$ for $i \in {\mathcal I}$ is given by $\iota(t) v = t^i v$ for all $v \in V_i$ and all $t \in {\mathbf k}^{\times}$. In this case, $\dim({\mathfrak g}_2) = 8$ and there are 14 orbits. We will write for each of these orbits:  the tableau ${\mathcal T} \in {\mathfrak C}_{\delta}$, the tableau ${\mathcal R} \in  {\mathfrak R}_{\delta}$ and its dimension. 

\begin{table}[h]
	\begin{center}\renewcommand{\arraystretch}{1.25}
		\begin{tabular} {| l || r  | r  | r  | r  | r |  r | r | r  | r | c |}
			\hline
			Orbit &  $100$  & $010$ & $001$ & $110$ & $011$ & $012$ & $111$ & $112$ & $122$ & Jordan Dec.  \\ \hline
			${\mathcal O}_{8}^1$  & $0$ & $0$ & $1$ & $0$ & $1$ & $0$ & $1$ &  $0$ & $0$ &$1^1 3^1 5^1$\\ \hline
			${\mathcal O}_{7}^1$& $0$ & $0$ & $0$ & $0$ & $0$ & $1$ & $1$  & $0$ & $0$ &$2^2 5^1$\\ \hline
			${\mathcal O}_{7}^2$& $0$ & $0$ & $1$ & $0$ & $0$ & $0$ & $0$  & $0$ & $1$ &$1^1 4^2$\\ \hline
			${\mathcal O}_{6}^1$ & $0$ & $1$ & $2$ & $0$ & $0$ & $0$ & $1$  & $0$ & $0$ &$1^4 5^1$\\ \hline
			${\mathcal O}_{6}^2$& $0$ & $0$ & $0$ & $0$ & $1$ & $0$ & $0$  & $1$ & $0$ &$3^3$\\ \hline
			${\mathcal O}_{6}^3$ & $1$ & $0$ & $1$ & $0$ & $2$ & $0$ & $0$  & $0$ & $0$ &$1^3 3^2$\\ \hline
			${\mathcal O}_{5}^1$ & $0$ & $1$ & $1$ & $0$ & $0$ & $0$ & $0$  & $1$ & $0$ &$1^3 3^2$\\ \hline
			${\mathcal O}_{5}^2$ & $0$ & $0$ & $2$ & $1$ & $1$ & $0$ & $0$  & $0$ & $0$ &$1^2 2^2 3^1$\\ \hline
			${\mathcal O}_{5}^3$ & $1$ & $0$ & $0$ & $0$ & $1$ & $1$ & $0$  & $0$ & $0$ &$1^2 2^2 3^1$\\ \hline
			${\mathcal O}_{4}^1$ & $0$ & $0$ & $1$ & $1$ & $0$ & $1$ & $0$  & $0$ & $0$ &$1^1 2^4$\\ \hline
			${\mathcal O}_{4}^2$ & $1$ & $1$ & $2$ & $0$ & $1$ & $0$ & $0$  & $0$ & $0$ &$1^6 3^1$\\ \hline
			${\mathcal O}_{3}^1$ & $1$ & $1$ & $1$ & $0$ & $0$ & $1$ & $0$  & $0$ & $0$ &$1^5 2^2$\\ \hline
			${\mathcal O}_{2}^1$ & $0$ & $1$ & $3$ & $1$ & $0$ & $0$ & $0$  & $0$ & $0$ &$1^5 2^2$\\ \hline
			${\mathcal O}_{0}^1$ & $1$ & $2$ & $3$ & $0$ & $0$ & $0$ & $0$  & $0$ & $0$ &$1^{9}$\\ \hline
		\end{tabular}
	\end{center}
	\caption{List of the values on the coefficient function  ${\mathbf c}$.}\label{T:Table4}
\end{table}
The columns of the  table~\ref{T:Table4}  are indexed by the  dimension of the indecomposable orthogonal representation of $A_3^{odd}$ and Jordan type for each orbit.	
The subscript in the notation for each orbit is its dimension and the upperscript is there to distinguish between the different orbits of the same dimension.
Now  the tableau ${\mathcal T} \in {\mathfrak T}_{\delta}$, the tableau ${\mathcal R} \in  {\mathfrak R}_{\delta}$ for each orbit are
\begin{itemize}
\item For the orbit ${\mathcal O}_{8}^1$ of dimension 8

\begin{center}
${\mathcal T}_{8}^1$ = 
\ytableausetup{centertableaux}
\begin{ytableau}
\none & \none [-4] & \none [-2] & \none [0] & \none [2] & \none [4]\\ \none [{\phantom -}4] & 1 & 0 & 0 & 0 & 0\\ \none [{\phantom -}2] & 0 & 1 & 0 & 0\\  \none [{\phantom -}0] & 0 & 0 & 1\\ \none [-2] & 0 & 0 \\ \none [-4] & 0 \\ \end{ytableau}
\quad
\text{and}
\quad
${\mathcal R}_{8}^1$ = 
\ytableausetup{centertableaux}
\begin{ytableau}
\none & \none [-4] & \none [-2] & \none [0] & \none [2] & \none [4]\\ \none [{\phantom -}4] & 1 & 1 & 1 & 1 & 1\\ \none [{\phantom -}2] & 1 & 2 & 2 & 2\\  \none [{\phantom -}0] & 1 & 2 & 3\\ \none [-2] & 1 & 2 \\ \none [-4] & 1 \\ \end{ytableau}.

\end{center}

\item For the orbit ${\mathcal O}_{7}^1$  of dimension 7

\begin{center}
${\mathcal T}_{7}^1$ = 
\ytableausetup{centertableaux}
\begin{ytableau}
\none & \none [-4] & \none [-2] & \none [0] & \none [2] & \none [4]\\ \none [{\phantom -}4] & 1 & 0 & 0 & 0 & 0\\ \none [{\phantom -}2] & 0 & 0 & 1 & 0\\  \none [{\phantom -}0] & 0 & 1 & 0\\ \none [-2] & 0 & 0 \\ \none [-4] & 0 \\ \end{ytableau}
\quad
\text{and}
\quad
${\mathcal R}_{7}^1$ = 
\ytableausetup{centertableaux}
\begin{ytableau}
\none & \none [-4] & \none [-2] & \none [0] & \none [2] & \none [4]\\ \none [{\phantom -}4] & 1 & 1 & 1 & 1 & 1\\ \none [{\phantom -}2] & 1 & 1 & 2 & 2\\  \none [{\phantom -}0] & 1 & 2 & 3\\ \none [-2] & 1 & 2 \\ \none [-4] & 1 \\ \end{ytableau}.

\end{center}

\item For the orbit ${\mathcal O}_{7}^2$  of dimension 7

\begin{center}
${\mathcal T}_{7}^2$ = 
\ytableausetup{centertableaux}
\begin{ytableau}
\none & \none [-4] & \none [-2] & \none [0] & \none [2] & \none [4]\\ \none [{\phantom -}4] & 0 & 1 & 0 & 0 & 0\\ \none [{\phantom -}2] & 1 & 0 & 0 & 0\\  \none [{\phantom -}0] & 0 & 0 & 1\\ \none [-2] & 0 & 0 \\ \none [-4] & 0 \\ \end{ytableau}
\quad
\text{and}
\quad
${\mathcal R}_{7}^2$ = 
\ytableausetup{centertableaux}
\begin{ytableau}
\none & \none [-4] & \none [-2] & \none [0] & \none [2] & \none [4]\\ \none [{\phantom -}4] & 0 & 1 & 1 & 1 & 1\\ \none [{\phantom -}2] & 1 & 2 & 2 & 2\\  \none [{\phantom -}0] & 1 & 2 & 3\\ \none [-2] & 1 & 2 \\ \none [-4] & 1 \\ \end{ytableau}.

\end{center}

\item For the orbit ${\mathcal O}_{6}^1$ of dimension 6

\begin{center}
${\mathcal T}_{6}^1$ = 
\ytableausetup{centertableaux}
\begin{ytableau}
\none & \none [-4] & \none [-2] & \none [0] & \none [2] & \none [4]\\ \none [{\phantom -}4] & 1 & 0 & 0 & 0 & 0\\ \none [{\phantom -}2] & 0 & 0 & 0 & 1\\  \none [{\phantom -}0] & 0 & 0 & 2\\ \none [-2] & 0 & 1 \\ \none [-4] & 0 \\ \end{ytableau}
\quad
\text{and}
\quad
${\mathcal R}_{6}^1$ = 
\ytableausetup{centertableaux}
\begin{ytableau}
\none & \none [-4] & \none [-2] & \none [0] & \none [2] & \none [4]\\ \none [{\phantom -}4] & 1 & 1 & 1 & 1 & 1\\ \none [{\phantom -}2] & 1 & 1 & 1 & 2\\  \none [{\phantom -}0] & 1 & 1 & 3\\ \none [-2] & 1 & 2 \\ \none [-4] & 1 \\ \end{ytableau}.

\end{center}

\item For the orbit ${\mathcal O}_{6}^2$ of dimension 6

\begin{center}
${\mathcal T}_{6}^2$ = 
\ytableausetup{centertableaux}
\begin{ytableau}
\none & \none [-4] & \none [-2] & \none [0] & \none [2] & \none [4]\\ \none [{\phantom -}4] & 0 & 0 & 1 & 0 & 0\\ \none [{\phantom -}2] & 0 & 1 & 0 & 0\\  \none [{\phantom -}0] & 1 & 0 & 0\\ \none [-2] & 0 & 0 \\ \none [-4] & 0 \\ \end{ytableau}
\quad
\text{and}
\quad
${\mathcal R}_{6}^2$ = 
\ytableausetup{centertableaux}
\begin{ytableau}
\none & \none [-4] & \none [-2] & \none [0] & \none [2] & \none [4]\\ \none [{\phantom -}4] & 0 & 0 & 1 & 1 & 1\\ \none [{\phantom -}2] & 0 & 1 & 2 & 2\\  \none [{\phantom -}0] & 1 & 2 & 3\\ \none [-2] & 1 & 2 \\ \none [-4] & 1 \\ \end{ytableau}.

\end{center}

\item For the orbit ${\mathcal O}_{6}^3$ of dimension 6

\begin{center}
${\mathcal T}_{6}^3$ = 
\ytableausetup{centertableaux}
\begin{ytableau}
\none & \none [-4] & \none [-2] & \none [0] & \none [2] & \none [4]\\ \none [{\phantom -}4] & 0 & 0 & 0 & 0 & 1\\ \none [{\phantom -}2] & 0 & 2 & 0 & 0\\  \none [{\phantom -}0] & 0 & 0 & 1\\ \none [-2] & 0 & 0 \\ \none [-4] & 1 \\ \end{ytableau}
\quad
\text{and}
\quad
${\mathcal R}_{6}^3$ = 
\ytableausetup{centertableaux}
\begin{ytableau}
\none & \none [-4] & \none [-2] & \none [0] & \none [2] & \none [4]\\ \none [{\phantom -}4] & 0 & 0 & 0 & 0 & 1\\ \none [{\phantom -}2] & 0 & 2 & 2 & 2\\  \none [{\phantom -}0] & 0 & 2 & 3\\ \none [-2] & 0 & 2 \\ \none [-4] & 1 \\ \end{ytableau}.

\end{center}

\item For the orbit ${\mathcal O}_{5}^1$ of dimension 5

\begin{center}
${\mathcal T}_{5}^1$ = 
\ytableausetup{centertableaux}
\begin{ytableau}
\none & \none [-4] & \none [-2] & \none [0] & \none [2] & \none [4]\\ \none [{\phantom -}4] & 0 & 0 & 1 & 0 & 0\\ \none [{\phantom -}2] & 0 & 0 & 0 & 1\\  \none [{\phantom -}0] & 1 & 0 & 1\\ \none [-2] & 0 & 1 \\ \none [-4] & 0 \\ \end{ytableau}
\quad
\text{and}
\quad
${\mathcal R}_{5}^1$ = 
\ytableausetup{centertableaux}
\begin{ytableau}
\none & \none [-4] & \none [-2] & \none [0] & \none [2] & \none [4]\\ \none [{\phantom -}4] & 0 & 0 & 1 & 1 & 1\\ \none [{\phantom -}2] & 0 & 0 & 1 & 2\\  \none [{\phantom -}0] & 1 & 1 & 3\\ \none [-2] & 1 & 2 \\ \none [-4] & 1 \\ \end{ytableau}.

\end{center}

\item For the orbit ${\mathcal O}_{5}^2$ of dimension 5

\begin{center}
${\mathcal T}_{5}^2$ = 
\ytableausetup{centertableaux}
\begin{ytableau}
\none & \none [-4] & \none [-2] & \none [0] & \none [2] & \none [4]\\ \none [{\phantom -}4] & 0 & 0 & 0 & 1 & 0\\ \none [{\phantom -}2] & 0 & 1 & 0 & 0\\  \none [{\phantom -}0] & 0 & 0 & 2\\ \none [-2] & 1 & 0 \\ \none [-4] & 0 \\ \end{ytableau}
\quad
\text{and}
\quad
${\mathcal R}_{5}^2$ = 
\ytableausetup{centertableaux}
\begin{ytableau}
\none & \none [-4] & \none [-2] & \none [0] & \none [2] & \none [4]\\ \none [{\phantom -}4] & 0 & 0 & 0 & 1 & 1\\ \none [{\phantom -}2] & 0 & 1 & 1 & 2\\  \none [{\phantom -}0] & 0 & 1 & 3\\ \none [-2] & 1 & 2 \\ \none [-4] & 1 \\ \end{ytableau}.

\end{center}

\item For the orbit ${\mathcal O}_{5}^3$ of dimension 5

\begin{center}
${\mathcal T}_{5}^3$ = 
\ytableausetup{centertableaux}
\begin{ytableau}
\none & \none [-4] & \none [-2] & \none [0] & \none [2] & \none [4]\\ \none [{\phantom -}4] & 0 & 0 & 0 & 0& 1\\ \none [{\phantom -}2] & 0 & 1 & 1 & 0\\  \none [{\phantom -}0] & 0 & 1 & 0\\ \none [-2] & 0 & 0 \\ \none [-4] & 1 \\ \end{ytableau}
\quad
\text{and}
\quad
${\mathcal R}_{5}^3$ = 
\ytableausetup{centertableaux}
\begin{ytableau}
\none & \none [-4] & \none [-2] & \none [0] & \none [2] & \none [4]\\ \none [{\phantom -}4] & 0 & 0 & 0 & 0 & 1\\ \none [{\phantom -}2] & 0 & 1 & 2 & 2\\  \none [{\phantom -}0] & 0 & 2 & 3\\ \none [-2] & 0 & 2 \\ \none [-4] & 1 \\ \end{ytableau}.

\end{center}

\item For the orbit ${\mathcal O}_{4}^1$ of dimension 4

\begin{center}
${\mathcal T}_{4}^1$ = 
\ytableausetup{centertableaux}
\begin{ytableau}
\none & \none [-4] & \none [-2] & \none [0] & \none [2] & \none [4]\\ \none [{\phantom -}4] & 0 & 0 & 0 & 1 & 0\\ \none [{\phantom -}2] & 0 & 0 & 1 & 0\\  \none [{\phantom -}0] & 0 & 1 & 1\\ \none [-2] & 1 & 0 \\ \none [-4] & 0 \\ \end{ytableau}
\quad
\text{and}
\quad
${\mathcal R}_{4}^1$ = 
\ytableausetup{centertableaux}
\begin{ytableau}
\none & \none [-4] & \none [-2] & \none [0] & \none [2] & \none [4]\\ \none [{\phantom -}4] & 0 & 0 & 0 & 1 & 1\\ \none [{\phantom -}2] & 0 & 0 & 1 & 2\\  \none [{\phantom -}0] & 0 & 1 & 3\\ \none [-2] & 1 & 2 \\ \none [-4] & 1 \\ \end{ytableau}.

\end{center}

\item For the orbit ${\mathcal O}_{4}^2$ of dimension 4

\begin{center}
${\mathcal T}_{4}^2$ = 
\ytableausetup{centertableaux}
\begin{ytableau}
\none & \none [-4] & \none [-2] & \none [0] & \none [2] & \none [4]\\ \none [{\phantom -}4] & 0 & 0 & 0 & 0 & 1\\ \none [{\phantom -}2] & 0 & 1 & 0 & 1\\  \none [{\phantom -}0] & 0 & 0 & 2\\ \none [-2] & 0 & 1 \\ \none [-4] & 1 \\ \end{ytableau}
\quad
\text{and}
\quad
${\mathcal R}_{4}^2$ = 
\ytableausetup{centertableaux}
\begin{ytableau}
\none & \none [-4] & \none [-2] & \none [0] & \none [2] & \none [4]\\ \none [{\phantom -}4] & 0 & 0 & 0 & 0 & 1\\ \none [{\phantom -}2] & 0 & 1 & 1 & 2\\  \none [{\phantom -}0] & 0 & 1 & 3\\ \none [-2] & 0 & 2 \\ \none [-4] & 1 \\ \end{ytableau}.

\end{center}

\item For the orbit ${\mathcal O}_{3}^1$ of dimension 3

\begin{center}
${\mathcal T}_{3}^1$ = 
\ytableausetup{centertableaux}
\begin{ytableau}
\none & \none [-4] & \none [-2] & \none [0] & \none [2] & \none [4]\\ \none [{\phantom -}4] & 0 & 0 & 0 & 0 & 1\\ \none [{\phantom -}2] & 0 & 0 & 1 & 1\\  \none [{\phantom -}0] & 0 & 1 & 1\\ \none [-2] & 0 & 1 \\ \none [-4] & 1 \\ \end{ytableau}
\quad
\text{and}
\quad
${\mathcal R}_{3}^1$ = 
\ytableausetup{centertableaux}
\begin{ytableau}
\none & \none [-4] & \none [-2] & \none [0] & \none [2] & \none [4]\\ \none [{\phantom -}4] & 0 & 0 & 0 & 0 & 1\\ \none [{\phantom -}2] & 0 & 0 & 1 & 2\\  \none [{\phantom -}0] & 0 & 1 & 3\\ \none [-2] & 0 & 2 \\ \none [-4] & 1 \\ \end{ytableau}.

\end{center}

\item For the orbit ${\mathcal O}_{2}^1$ of dimension 2

\begin{center}
${\mathcal T}_{2}^1$ = 
\ytableausetup{centertableaux}
\begin{ytableau}
\none & \none [-4] & \none [-2] & \none [0] & \none [2] & \none [4]\\ \none [{\phantom -}4] & 0 & 0 & 0 & 1 & 0\\ \none [{\phantom -}2] & 0 & 0 & 0 & 1\\  \none [{\phantom -}0] & 0 & 0 & 3\\ \none [-2] & 1 & 1 \\ \none [-4] & 0 \\ \end{ytableau}
\quad
\text{and}
\quad
${\mathcal R}_{2}^1$ = 
\ytableausetup{centertableaux}
\begin{ytableau}
\none & \none [-4] & \none [-2] & \none [0] & \none [2] & \none [4]\\ \none [{\phantom -}4] & 0 & 0 & 0 & 1 & 1\\ \none [{\phantom -}2] & 0 & 0 & 0 & 2\\  \none [{\phantom -}0] & 0 & 0 & 3\\ \none [-2] & 1 & 2 \\ \none [-4] & 1 \\ \end{ytableau}.

\end{center}

\item For the orbit ${\mathcal O}_{0}^1$ of dimension 0

\begin{center}
${\mathcal T}_{0}^1$ = 
\ytableausetup{centertableaux}
\begin{ytableau}
\none & \none [-4] & \none [-2] & \none [0] & \none [2] & \none [4]\\ \none [{\phantom -}4] & 0 & 0 & 0 & 0 & 1\\ \none [{\phantom -}2] & 0 & 0 & 0 & 2\\  \none [{\phantom -}0] & 0 & 0 & 3\\ \none [-2] & 0 & 2 \\ \none [-4] & 1 \\ \end{ytableau}
\quad
\text{and}
\quad
${\mathcal R}_{0}^1$ = 
\ytableausetup{centertableaux}
\begin{ytableau}
\none & \none [-4] & \none [-2] & \none [0] & \none [2] & \none [4]\\ \none [{\phantom -}4] & 0 & 0 & 0 & 0 & 1\\ \none [{\phantom -}2] & 0 & 0 & 0 & 2\\  \none [{\phantom -}0] & 0 & 0 & 3\\ \none [-2] & 0 & 2 \\ \none [-4] & 1 \\ \end{ytableau}.

\end{center}

\end{itemize}

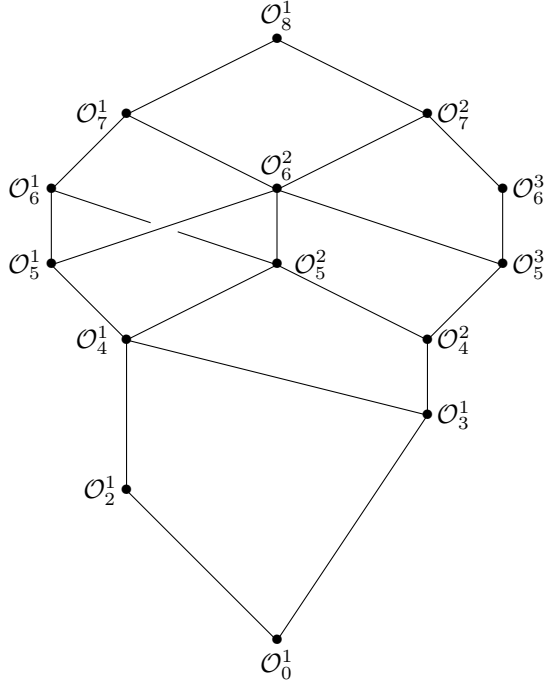
\begin{figure}[h]
\begin{center}
\begin{tikzpicture}
\draw (0, 0) -- (2, 3);
\draw (0, 0) -- (-2, 2);
\draw (2, 3) -- (-2, 4);
\draw  (2, 3) -- (2, 4);
\draw (-2, 2) -- (-2, 4);
\draw (-2, 4) -- (-3, 5);
\draw (-2, 4) -- (0, 5);
\draw (2, 4) -- (0, 5); 
\draw (2, 4) -- (3, 5);
\draw (-3, 5) -- (-3, 6);
\draw (-3, 5) -- (0, 6);
\draw (0, 5) -- (-1.32, 5.44);
\draw (-1.68, 5.56) -- (-3, 6);
\draw (0, 5) -- (0, 6);
\draw (3, 5) -- (0, 6);
\draw (3, 5) -- (3, 6);
\draw (-3, 6) -- (-2, 7);
\draw (0, 6) -- (-2, 7);
\draw (0, 6) -- (2, 7);
\draw (3, 6) -- (2, 7);
\draw (-2, 7) -- (0, 8);
\draw (2, 7) -- (0, 8);
\draw (0, 0) node[below] {${\mathcal O}_0^1$};
\draw (-2, 2) node[left] {${\mathcal O}_2^1$};
\draw (2, 3) node[right] {${\mathcal O}_3^1$};
\draw (-2.1, 4) node[left] {${\mathcal O}_4^1$};
\draw (2, 4) node[right] {${\mathcal O}_4^2$};
\draw (-3, 5) node[left] {${\mathcal O}_5^1$};
\draw (0.1, 5) node[right] {${\mathcal O}_5^2$};
\draw (3, 5) node[right] {${\mathcal O}_5^3$};
\draw (-3, 6) node[left] {${\mathcal O}_6^1$};
\draw (0, 6) node[above] {${\mathcal O}_{6}^2$};
\draw (3, 6) node[right] {${\mathcal O}_{6}^3$};
\draw (-2.1, 7) node[left] {${\mathcal O}_{7}^1$};
\draw (2, 7) node[right] {${\mathcal O}_{7}^2$};
\draw (0, 8) node[above] {${\mathcal O}_{8}^1$};
\draw (0, 0) node {$\bullet$};
\draw (-2, 2) node {$\bullet$};
\draw (2, 3) node {$\bullet$};
\draw (-2, 4) node {$\bullet$};
\draw (2, 4) node {$\bullet$};
\draw (-3, 5) node {$\bullet$};
\draw (0, 5) node {$\bullet$};
\draw (3, 5) node {$\bullet$};
\draw (-3, 6) node {$\bullet$};
\draw (0, 6) node {$\bullet$};
\draw (3, 6) node {$\bullet$};
\draw (-2, 7) node {$\bullet$};
\draw (2, 7) node {$\bullet$};
\draw (0, 8) node {$\bullet$};
\end{tikzpicture}
\end{center}
\caption{Poset obtained by the rank tableaux comparaison}\label{FigurePosetExampleA}
\end{figure}

We have illustrated the Hasse diagram of the poset obtained when we compare the rank tableaux ${\mathcal R}$ for the set of orbits  in figure~\ref{FigurePosetExampleA}.
\end{example}

\section{$G^{\iota}$-orbits in ${\mathfrak g}_2$ in the special orthogonal case when $\vert{\mathcal I}\vert$ is odd.}

\subsection{}
In this section, we will restrict ourself to the special orthogonal case when $\vert {\mathcal I} \vert$ is odd.  The special orthogonal case is trickier. The reason being  that special orthogonal isomorphism is more restricted that orthogonal isomorphism as we saw in \ref{SS:ExampleIsomorphismSpecialOrtho}.

\subsection{}\label{S:SetUpOddOrtho}
In this section,  
\begin{itemize}
\item $m$ is an integer $> 0$;
\item $\tilde G$ is  the group of automorphisms of a finite dimensional vector space $V$ over $k$ preserving a fixed non-degenerate symmetric form $\langle\ , \  \rangle: V \times V \rightarrow {\mathbf k}$;  in other words, $\tilde G = O(V)$; 
\item $G$ is the connected component of the identity element of the group $\tilde G$,  in other words, $G = SO(V)$;
\item $\epsilon = 1$;
\item $\mathfrak g$ is the Lie algebra of $G$, more precisely $X \in {\mathfrak g}$ if and only if $X:V \rightarrow V$ is an endomorphism of $V$ such that $\langle X(v_1), v_2\rangle + \langle v_1 , X(v_2)\rangle = 0$ for all $v_1, v_2 \in V$;
\item $V$ is identified to its dual $V^*$ by the isomorphism $\phi: V \rightarrow V^*$ defined by $v \mapsto \phi_v$, where $\phi_v: V \rightarrow {\mathbf k}$ is $\phi_v(x) = \langle v, x\rangle$ for all $x \in V$;  
\item ${\mathcal I} = \{n \in {\mathbb Z} \mid n \equiv 0 \pmod 2, -2m <  n  < 2m \}$;
\item $\oplus_{i \in {\mathcal I}} V_i$ is a direct sum decomposition of $V$ by  vector subspaces $V_i \ne 0$ such that 
\begin{itemize} \item $\langle v, v'\rangle = 0$ whenever $v \in V_i$, $v' \in V_j$ and $i + j \ne 0$;
			\item $V_{-i} = V_i^*$ for all $i \in {\mathcal I}$ under the identification given by $\phi$ above.
\end{itemize} 
Note that the restriction of the non-degenerate symmetric  form $\langle\ , \ \rangle$ to $V_0 \times V_0$ is a non-degenerate symmetric  form.
\item ${\mathcal B} = \coprod_{i \in {\mathcal I}} {\mathcal B}_i$ is a basis of $V$ such that each ${\mathcal B}_i = \{ u_{i, j} \mid  1 \leq j \leq \delta_i \}$  is a basis of $V_i$ for each $i \in {\mathcal I}$ and $i \ne 0$, where $\delta_i$ is the dimension of $V_i$, and ${\mathcal B}_0 = \{ u_{0, j} \mid 1 \leq j \leq \delta_0 \}$ is a basis of $V_0$, where $\delta_0$ is the dimension of $V_0$ and these bases are such that, for $i, j \in {\mathcal I}$, 
\[
\begin{aligned}
\langle u_{i, r}, u_{j, s}\rangle &= \begin{cases} 1, &\text{if $i , j \ne 0$, $i + j = 0$ and $r = s$;}\\ 0, &\text{otherwise;} \end{cases}
\\
\langle u_{0, r}, u_{0, s}\rangle &= \begin{cases} 1, &\text{if  $r + s = \delta_0 + 1$;}\\ 0, &\text{otherwise;}\\  \end{cases}
\\
\text{ and }
\\
\langle u_{0, r}, u_{i, s}\rangle &= \langle u_{i, s}, u_{0, r}\rangle = 0, \text{ if $i \ne 0$, $1 \leq s \leq \delta_i$, $1 \leq r \leq \delta_0$.}
\end{aligned}
\]

\item $\iota: {\mathbf k}^{\times} \rightarrow {\tilde G}$ is the homomorphism defined by $\iota(t) v = t^i v$ for all $i \in {\mathcal I}$, $v \in V_i$ and $t \in {\mathbf k}^{\times}$ and clearly $\iota(t) \in G$ for all $t \in {\mathbf k}^{\times}$;
\item ${\mathbf B}$ is the set of ${\mathcal I}$-boxes;
\item ${\mathbf B}/\tau$ denote the set of $\langle \tau\rangle$-orbits ${\mathcal O}$ in ${\mathbf B}$; 
\item $\nu = \vert {\mathbf B}/\tau \vert = m^2$.
\end{itemize}

\subsection{}
In proposition~ \ref{Prop_Indexation_Orbits_Odd}, we gave a parametrization of the set ${\mathfrak g}_2/{\tilde G}^{\iota}$ of ${\tilde G}^{\iota}$-orbits in ${\mathfrak g}_2$ by giving a bijection $\Upsilon: {\mathfrak g}_2/{\tilde G}^{\iota} \rightarrow {\mathfrak C}_{\delta}$. To ${\mathbf c} \in {\mathfrak C}_{\delta}$, we have denoted the corresponding ${\tilde G}^{\iota}$-orbit under $\Upsilon$ by ${\mathcal O}_{\mathbf c}$. From our observation in \ref{SS:ObservationDimensionOrthogonal} and lemma~\ref{G2AsRepresOdd} (a), we have that the indice $[ {\tilde G}^{\iota}: G^{\iota}] = [O(V_0): SO(V_0)] = 2$. Note here we used the fact that $\dim(V_0) \ne 0$.  Consequently the ${\tilde G}^{\iota}$-orbit  ${\mathcal O}_{\mathbf c}$ is either a $G^{\iota}$-orbit  or splits into two $G^{\iota}$-orbits. We need to give a criteria on  ${\mathbf c} \in  {\mathfrak C}_{\delta}$ for when the ${\tilde G}^{\iota}$-orbit ${\mathcal O}_{\mathbf c}$ splits or not. 

\begin{proposition}\label{P:OrbitSplittingSO}
Let ${\mathbf c} \in {\mathfrak C}_{\delta}$ and the  symmetric function  $c:{\mathbf B} \rightarrow {\mathbb N}$ associated to ${\mathbf  c}$. 
\begin{enumerate}[\upshape (a)]
\item There exists $i, j \in {\mathcal I}$, $i \geq 0 \geq j$ such that $c(b(i, j, 0)) > 0$.
\item If there is an element $i' \in {\mathcal I}$, $i' \geq 0$ such that $c(b(i', -i', 0)) > 0$, then the ${\tilde G}^{\iota}$-orbit ${\mathcal O}_{\mathbf c}$ does not split and it is a unique $G^{\iota}$-orbit ${\mathcal O}_{\mathbf c}$.
\item If $c(b(i, -i, 0)) = 0$ for all $i \in {\mathcal I}$, $i \geq 0$, then  the ${\tilde G}^{\iota}$-orbit ${\mathcal O}_{\mathbf c}$ splits into two $G^{\iota}$-orbits: ${\ }'{\mathcal O}_{\mathbf c}$ and ${\ }''{\mathcal O}_{\mathbf c}$. Moreover we have that  $\dim({\mathcal O}_{\mathbf c}) = \dim({\ }'{\mathcal O}_{\mathbf c}) = \dim({\ }''{\mathcal O}_{\mathbf c})$
\end{enumerate}
\end{proposition}
\begin{proof}
(a) We have 
\[
\sum_{\begin{subarray}{c} b \in {\mathbf B}\\ 0 \in Supp(b) \end{subarray}} c(b) = \sum_{\begin{subarray}{c} i, j \in {\mathcal I}\\ i \geq 0 \geq j \end{subarray}} c(b(i, j, 0)) = \delta_0 > 0
\]
and the result follows easily, because $c(b) \geq 0$ for all $b \in {\mathbf B}$. 

(b) Let $T_{\mathbf c}:V({\mathbf c}) \rightarrow V$ be the ${\mathcal I}$-graded isomorphism such that $\langle T_{\mathbf c}(u), T_{\mathbf c}(v)\rangle = \langle u, v\rangle_{\mathbf c}$ for all $u, v \in V({\mathbf c})$ constructed in proposition~\ref{Prop_Indexation_Orbits_Odd}. Using this isomorphism we saw that ${\mathcal O}_{\mathbf c}$ is the ${\tilde G}^{\iota}$-orbit of $T_{\mathbf c}E_{\mathbf c} T_{\mathbf c}^{-1}$. We will write this element $T_{\mathbf c}E_{\mathbf c}T_{\mathbf c}^{-1}$ as $X_{\mathbf c}: V \rightarrow V$. Thus $X_{\mathbf c}: V \rightarrow V$ is an ${\mathcal I}$-graded linear transformation of degree 2 such that $\langle X_{\mathbf c}(u), v\rangle + \langle u , X_{\mathbf c}(v)\rangle = 0$ for all $u, v \in V$. Let $X: V \rightarrow V$ be an ${\mathcal I}$-graded linear transformation of degree 2 such that $\langle X(v_1), v_2\rangle + \langle v_1 , X(v_2) \rangle = 0$ for all $v_1, v_2 \in V$ and $X \in {\mathcal O}_{\mathbf c}$.  Thus there exists ${\tilde g} \in {\tilde G}^{\iota}$ such that $Ad({\tilde g})(X_{\mathbf c}) = X$.

If we consider $X_{\mathbf c}$, $X$ and ${\tilde g}$ as in lemma~\ref{G2AsRepresOdd} (a) and (b);  in other words, if we write $X_{{\mathbf c}, i}: X_{\mathbf c}\vert_{V_i}: V_i \rightarrow V_{i + 2}$ and $X_{i}: X\vert_{V_i}: V_i \rightarrow V_{i + 2}$ for all $i \in {\mathcal I}$, $i \ne 2(m - 1)$ and ${\tilde g}_i = {\tilde g}\vert_{V_i}:V_i \rightarrow V_i$ for all $i \in {\mathcal I}$, we have $X_i \circ {\tilde g}_i = {\tilde g}_{i + 2} \circ X_{{\mathbf c}, i}$ for all $i \in {\mathcal I}$, $i \ne 2(m - 1)$. We have ${\tilde g}_i \in GL(V_i)$, $\det({\tilde g}_{-i}) = (\det({\tilde g}_i))^{-1}$ for all $i \in {\mathcal I}$, $i \ne 0$,  ${\tilde g}_0 \in O(V_0)$ and $\det({\tilde g}) = \det({\tilde g}_0)$. 

By hypothesis, there exists an ${\mathcal I}$-box $b' \in {\mathbf B}$ such that $b' = b(i', -i', 0)$ and $c(b') \ne 0$, where  $i' \in {\mathcal I}$, $i' \geq 0$.  The $\langle\tau\rangle$-orbit of $b'$ is ${\mathcal O}' = \{b'\}$. Consider the symmetric function $c':{\mathbf B} \rightarrow {\mathbb N}$ defined by
\[
c'(b) = \begin{cases} c(b), &\text{if $b \ne b'$;} \\ (c(b') - 1), &\text{if $b = b'$.}\end{cases}
\]
and the associated coefficient function ${\mathbf c}':{\mathbf B}/\tau \rightarrow {\mathbb N}$. From its definition in \ref{N:CCorrespondanceOdd}, we have that $V({\mathbf c})$ is the direct sum of $V({{\mathbf c}'})$ and one copy of $V({\mathcal O}')$. Moreover the orthogonal complement $V({\mathcal O}')^{\perp}$ of $V({\mathcal O}')$ is $V({{\mathbf c}'})$. Because $b'$ is on the principal diagonal and from our construction for the orthogonal representation of $V({\mathcal O}') = V(b')$, then we know that the ${\mathcal I}$-component $V_i(b')$ of $V(b')$ for $i \in {\mathcal I}$ are 1-dimensional for $i \in Supp(b')$ and 0-dimensional otherwise; in particular $\dim(V_0(b')) = 1$.   

Using the isomorphism $T_{\mathbf c}$, we get that $V$ is the direct sum of the ${\mathcal I}$-graded indecomposable orthogonal representation $U = T_{\mathbf c}(V(b'))$ and its orthogonal complement $U^{\perp} = T_{\mathbf c}(V({{\mathbf c}'}))$ such that $U_i$ is 1-dimensional exactly when $i \in Supp(b')$ and 0-dimensional otherwise.  We want to show that there exists an element $g$ of $G^{\iota}$ such that $Ad(g)(X_{\mathbf c}) = X$. If the element ${\tilde g}$ above is such that  ${\tilde g}$ of $G^{\iota}$, then we are done by taking ${\tilde g}$ for $g$. 

Assume now that  ${\tilde g} \not\in  G^{\iota}$, in other words, $\det({\tilde g}_0) = -1$. Consider the linear transformation $h:V \rightarrow V$ such that $h(v) = h(u + u^{\perp}) = -u + u^{\perp}$, where $v = u + u^{\perp}$, with $u \in U$ and $u^{\perp} \in U^{\perp}$. It is easy to see that $h$ preserve the bilinear form $\langle\ , \ \rangle$ and that $h \in {\tilde G}^{\iota}$. Also we have that $Ad(h)(X_{\mathbf c}) = X_{\mathbf c}$, because
\[
\begin{aligned}
Ad(h)X_{\mathbf c}(v) &= h X_{\mathbf c} h^{-1} (u + u^{\perp}) = h X_{\mathbf c}(-u + u^{\perp}) = h(-X_{\mathbf c}(u) + X_{\mathbf c}(u^{\perp}))\\ &= X_{\mathbf c}(u) + X_{\mathbf c}(u^{\perp}) = X_{\mathbf c}(v),
\end{aligned}
\]
where $v = u + u^{\perp}$, with $u \in U$ and $u^{\perp} \in U^{\perp}$. 

Consider now $g = {\tilde g}h \in {\tilde G}$. We have $Ad(g)(X_{\mathbf c}) = Ad({\tilde g} h)(X_{\mathbf c}) = Ad({\tilde g})(X_{\mathbf c}) = X$ and $\det(g) = \det({\tilde g}_0)\det(h_0) = 1$, because $\det({\tilde g}_0) = -1$ and $\det(h_0) = -1$ from the fact that  $\dim(U_0) = 1$ and the definition of $h$. So $g \in G^{\iota}$ and the ${\tilde G}^{\iota}$-orbit ${\mathcal O}_{\mathbf c}$ is a unique $G^{\iota}$-orbit ${\mathcal O}_{\mathbf c}$.

(c) We will prove this by contradiction.  Assume there exist a dimension vector $\delta = (\delta_i)_{i \in {\mathcal I}}$, a coefficient function  ${\mathbf c} \in {\mathfrak C}_{\delta}$ and the associated symmetric function $c:{\mathbf B} \rightarrow {\mathbb N}$ such that 
\begin{itemize}
\item $\delta_i > 0$ for all $i \in {\mathcal I}$;
\item $c(b(i, -i, 0)) = 0$ for all $i \in {\mathcal I}$, $i \geq 0$;
\item the ${\tilde G}^{\iota}$-orbit ${\mathcal O}_{\mathbf c}$ is a $G^{\iota}$-orbit, in other words the  ${\tilde G}^{\iota}$-orbit ${\mathcal O}_{\mathbf c}$ does not split;
\item the sum
\[
\sum_{{\mathcal O} \in {\mathbf B}/\tau} {\mathbf c}({\mathcal O})
\]
is minimal among all of the  ${\tilde G}^{\iota}$-orbit ${\mathcal O}_{\mathbf c}$ as above.
\end{itemize}

Let ${\mathbf B}' = \{b \in {\mathbf B} \mid b = b(i, j, 0) \text{ for some $i, j \in {\mathcal I}$, $i \geq 0 \geq j$ and $i + j \ne 0$}\}$. By (a) and,  because of our hypothesis,  ${\mathbf B}' \ne \emptyset$ and  there exist  $i', j' \in {\mathcal I}$ with $i' \geq 0 \geq j'$, $i' + j' > 0$ such that $b' = b(i', j', 0) \in {\mathbf B}'$ and $c(b') > 0$.  Note that $0 \in Supp(b')$ and $0 \in Supp(\tau(b'))$. Denote by ${\mathcal O}'$ the $\langle \tau\rangle$-orbit of $b'$. So ${\mathcal O}' = \{b', \tau(b')\}$.

Consider the symmetric function $c':{\mathbf B} \rightarrow {\mathbb N}$ defined by
\[
c'(b) = \begin{cases} c(b), &\text{if $b \not\in \{ b', \tau(b')\}$;} \\ (c(b') - 1), &\text{if $b \in \{ b', \tau(b')\}$.}\end{cases}
\]
and the associated coefficient function ${\mathbf c}':{\mathbf B}/\tau \rightarrow {\mathbb N}$. From its definition in \ref{N:CCorrespondanceOdd}, we have that $V({\mathbf c})$ is the direct sum of $V({{\mathbf c}'})$ and one copy of $V({\mathcal O}')$.  More precise,  
$V({\mathbf c})$: the direct sum as ${\mathcal I}$-graded vector space and  as  $\epsilon$-representations of ${\mathbf c}({\mathcal O})$ copies of $V({\mathcal O})$ with a ${\mathcal I}$-basis 
\[
{\mathcal B}_{\mathbf c} = \coprod_{{\mathcal O} \in {\mathbf B}/\tau} \coprod_{b \in {\mathcal O}} \{v_i^j(b) \mid i \in Supp(b), 1 \leq j \leq {\mathbf c}({\mathcal O})\}.
\]
and the bilinear form $\langle \  , \  \rangle_{\mathbf c}$ is described in \ref{N:CCorrespondanceOdd} with $\epsilon = 1$.  But $V({\mathbf c}')$: the direct sum as ${\mathcal I}$-graded vector space and  as   orthogonal representations of ${\mathbf c}'({\mathcal O})$ copies of $V({\mathcal O})$ with a ${\mathcal I}$-basis 
\[
{\mathcal B}_{{\mathbf c}'} = \coprod_{{\mathcal O} \in {\mathbf B}/\tau} \coprod_{b \in {\mathcal O}} \{v_i^j(b) \mid i \in Supp(b), 1 \leq j \leq {\mathbf c}'({\mathcal O})\}.
\]
and the bilinear form $\langle \  , \  \rangle_{{\mathbf c}'}$ is obtained by restriction of $\langle \  , \  \rangle_{\mathbf c}$ to $V({{\mathbf c}'})$ and $V({\mathcal O}')$ as ${\mathcal I}$-graded vector space and  as orthogonal representation  with a ${\mathcal I}$-basis $\{v_i^{c(b')}(b'),\    v_i^{c(b')}(\tau(b')) \mid i \in Supp(b')\}$ and the bilinear form $\langle \  , \  \rangle_{{\mathcal O}'}$ is obtained by restriction of $\langle \  , \  \rangle_{\mathbf c}$ to $V({\mathcal O}')$.   Note that  the orthogonal complement $V({\mathcal O}')^{\perp}$ of $V({\mathcal O}')$ is $V({{\mathbf c}'})$. 

Let $S_{{\mathcal O}'}:V({\mathbf c}) \rightarrow V({\mathbf c})$ be the ${\mathcal I}$-graded linear transformation of degree $0$ defined by
\[
S_{{\mathcal O}'}(u) = \begin{cases} u, &\text{if $u \in V({\mathbf c}')$;}\\ u, &\text{if $u \in V_i({\mathcal O}')$ and $i \ne 0$;} \\ v_0^{c(b')}(\tau(b')), &\text{if $u = v_0^{c(b')}(b')$;}\\  v_0^{c(b')}(b'), &\text{if $u = v_0^{c(b')}(\tau(b'))$.}
\end{cases}
\]

Let $T_{\mathbf c}:V({\mathbf c}) \rightarrow V$ be the ${\mathcal I}$-graded isomorphism such that $\langle T_{\mathbf c}(u), T_{\mathbf c}(v)\rangle = \langle u, v\rangle_{\mathbf c}$ for all $u, v \in V({\mathbf c})$ constructed in proposition~\ref{Prop_Indexation_Orbits_Odd}. 
We can now consider the two ${\mathcal I}$-graded linear transformations $E_{\mathbf c}:V({\mathbf c}) \rightarrow V({\mathbf c})$ and $E'_{\mathbf c} = S_{{\mathcal O}'}E_{\mathbf c}S_{{\mathcal O}'}^{-1}:V({\mathbf c}) \rightarrow V({\mathbf c})$ and their corresponding elements $X_{\mathbf c}= T_{\mathbf c}E_{\mathbf c}T^{-1}_{\mathbf c}:V \rightarrow V$ and $X'_{\mathbf c}= T_{\mathbf c}E'_{\mathbf c}T^{-1}_{\mathbf c}:V \rightarrow V$. Using this isomorphism,  we saw that ${\mathcal O}_{\mathbf c}$ is the ${\tilde G}^{\iota}$-orbit of $X_{\mathbf c}$, but also the ${\tilde G}^{\iota}$-orbit of $X'_{\mathbf c}$. From our construction, both of these are in the same $G^{\iota}$-orbit. Thus by our assumption,  there exists a special orthogonal element $g \in G^{\iota}$ such that $Ad(g) X_{\mathbf c} = X'_{\mathbf c}$, i.e. $det(g) = 1$. 

By considering the ${\mathcal I}$-graded vector subspace $U = T_{\mathbf c}(V({\mathcal O}'))$ and its orthogonal complement  $U^{\perp}  = T_{\mathbf c}(V({\mathbf c}'))$, we get that both are $Ad(g)$-invariant subspaces. In this case,  $\det(g) = \det(g\vert_U) \det(g\vert_{U^{\perp}}) = 1$ and by the minimality of ${\mathbf c}$, we have that  $\det(g\vert_{U^{\perp}}) = 1$.  Thus $\det(g\vert_U) = 1$, but this contradicts the fact we saw in \ref{SS:SpecialOrthoRepresentationsInSameOrthoClass} and \ref{SS:ExampleIsomorphismSpecialOrtho}  that  $E_{{\mathcal O}'}$ and $S_{{\mathcal O}'}\vert_{V({\mathcal O}')} E_{{\mathcal O}'} S^{-1}_{{\mathcal O}'}\vert_{V({\mathcal O}')}$ are not isomorphic relative to the special orthogonal group.  This completes the proof that the ${\tilde G}^{\iota}$-orbit ${\mathcal O}_{\mathbf c}$ splits into two $G^{\iota}$-orbits $'{\mathcal O}_{\mathbf c}$ and $''{\mathcal O}_{\mathbf c}$.  With the above notation, $'{\mathcal O}_{\mathbf c}$ is the $G^{\iota}$-orbit of $X_{\mathbf c}$ and $''{\mathcal O}_{\mathbf c}$ is the $G^{\iota}$-orbit of $X'_{\mathbf c}$ .
\end{proof}

\begin{remark}\label{R:DimensionForSplitting}
If there is at least one ${\tilde G}^{\iota}$-orbit ${\mathcal O}$ that splits into two $G^{\iota}$-orbits,  then the dimension $\dim(V)$ of $V$ (equivalently $\dim(V_0)$ of $V_0$) is even.  This follows easily from the fact that, if ${\mathcal O}_{\mathbf c}$ splits into two $G^{\iota}$-orbits, then $c(b(i, -i, 0)) = 0$ for all $i \in {\mathcal I}$, $i \geq 0$  for the symmetric function $c:{\mathbf B} \rightarrow {\mathbb N}$ corresponding  to ${\mathbf c}$ by proposition~\ref{P:OrbitSplittingSO} (c). Then
\[
\begin{aligned}
\delta_0 &= \dim(V_0) = \sum_{\begin{subarray}{c}  i, j \in {\mathcal I}\\ i \geq 0 \geq j \end{subarray}} c(b(i, j, 0)) = \sum_{\begin{subarray}{c} i, j \in {\mathcal I}\\ i \geq 0 \geq j\\ i + j \ne 0 \end{subarray}} c(b(i, j, 0)) + \sum_{\begin{subarray}{c} i \in {\mathcal I}\\ i \geq 0\end{subarray}} c(b(i, -i, 0)) \\ &= 2\sum_{\begin{subarray}{c}  i, j \in {\mathcal I}\\ i \geq 0 \geq j\\ i + j > 0 \end{subarray}} c(b(i, j, 0)) 
\end{aligned}
\]
because, if $i, j \in {\mathcal I}$, $i \geq 0 \geq j$ and $i + j > 0$, then  $c(b(i, j, 0)) = c(b(-j, -i, 0))$.  
\end{remark}

\begin{notation}\label{N:CCorrespondanceSpecialOrtho}
Using the same notation as in proposition~\ref{P:OrbitSplittingSO}, we can associated to each $G^{\iota}$-orbit: a Jacobson-Morozov triple. 
\begin{itemize}
\item If the ${\tilde G}^{\iota}$-orbit ${\mathcal O}_{\mathbf c}$ does not split (case (b) in the proposition~\ref{P:OrbitSplittingSO}), then the Jacobson-Morozov triple  associated to the $G^{\iota}$-orbit  ${\mathcal O}_{\mathbf c}$ is $(E_{\mathbf c}, H_{\mathbf c}, F_{\mathbf c})$ as in \ref{N:CCorrespondanceOdd}. 
\item If the ${\tilde G}^{\iota}$-orbit ${\mathcal O}_{\mathbf c}$ does split (case (c) in proposition ~\ref{P:OrbitSplittingSO}) and the $G^{\iota}$-orbit is ${\ }'{\mathcal O}_{\mathbf c}$, then the Jacobson-Morozov triple associated to the  $G^{\iota}$-orbit ${\ }'{\mathcal O}_{\mathbf c}$ is $(E_{\mathbf c}, H_{\mathbf c}, F_{\mathbf c})$ as in \ref{N:CCorrespondanceOdd}. We will denote this triple as $({\ }'E_{\mathbf c}, {\ }'H_{\mathbf c}, {\ }'F_{\mathbf c})$.
\item If the ${\tilde G}^{\iota}$-orbit ${\mathcal O}_{\mathbf c}$ does split (case (c) in proposition~\ref{P:OrbitSplittingSO}) and the $G^{\iota}$-orbit is ${\ }''{\mathcal O}_{\mathbf c}$, then the Jacobson-Morozov triple associated to the  $G^{\iota}$-orbit 
${\ }''{\mathcal O}_{\mathbf c}$ will be denoted as $({\ }''E_{\mathbf c}, {\ }''H_{\mathbf c}, {\ }''F_{\mathbf c})$.  If ${\mathcal O}' \in {\mathbf B}'/\tau$ is as in the proof of  proposition~\ref{P:OrbitSplittingSO} (c), where $''E_{\mathbf c} = S_{{\mathcal O}' } E_{\mathbf c} S^{-1}_{{\mathcal O}'}$, $''H_{\mathbf c} = S_{{\mathcal O}' } H_{\mathbf c} S^{-1}_{{\mathcal O}'}$ and $''F_{\mathbf c} = S_{{\mathcal O}' } F_{\mathbf c} S^{-1}_{{\mathcal O}'}$.  Note that if we had chosen another ${\mathcal O}' \in {\mathbf B}'/\tau$ is as in the proof of  proposition~\ref{P:OrbitSplittingSO} (c), then we get another Jacobson-Morozov triple, but it is in the same $G(V_{\mathbf c})^{\iota}$-orbit.
\end{itemize}

Using the isomorphism $T_{\mathbf c}:V_{\mathbf c} \rightarrow V$,  we can transport an associated Jacobson-Morozov triple to a $G^{\iota}$-orbit to the Lie algebra ${\mathfrak g}$.
\end{notation}

\begin{example}\label{E:5.7}
If we reconsider the example~\ref{E:4.27}, then none of the fourteen ${\tilde G}^{\iota}$-orbits splits because the dimension $\dim(V_0) = 3$ of $V_0$ is odd. Consequently all fourteen ${\tilde G}^{\iota}$-orbits are $G^{\iota}$-orbits.
\end{example}

\begin{example}\label{E:5.29}
With the notation of  \ref{S:SetUpOdd}, let $m = 2$, $G$ be the connected component of the identity element of  the group of automorphisms of the finite dimensional vector space $V$ over ${\mathbf k}$ preserving  a fixed non-degenerate symmetric form $\langle\ , \  \rangle:V \times V \rightarrow {\mathbf k}$,   ${\mathcal I} = \{2, 0, -2\}$, $\dim(V_0) = 2 = \delta_0$, $\dim(V_2) = 2 = \delta_2 = \delta_{-2}$. In other words,  $G = SO_{6}(k)$ and $\iota(t)$ on $V_i$ for $i \in {\mathcal I}$ is given by $\iota(t) v = t^i v$ for all $v \in V_i$ and all $t \in {\mathbf k}^{\times}$. In this case, $\dim({\mathfrak g}_2) = 4$ and there are five $G^{\iota}$-orbits. We will write for each of these orbits:  the tableau ${\mathcal T} \in {\mathfrak T}_{\delta}$, the tableau ${\mathcal R} \in  {\mathfrak R}_{\delta}$ and its dimension. 

\begin{table}[h]
	\begin{center}\renewcommand{\arraystretch}{1.25}
		\begin{tabular} {| l || r  | r  | r  | r  | c |}
			\hline
			Orbit(s) &  $10$  & $01$ & $11$ & $12$ & Jordan Dec.  \\ \hline
			${\mathcal O}_{4}$  & $0$ & $0$ & $2$ & $0$ & $3^2$\\ \hline
			${\mathcal O}_{3}$& $1$ & $1$ & $1$ & $0$  &$1^3 3^1$\\ \hline
			${\ }'{\mathcal O}_{2}$, ${\ }''{\mathcal O}_{2}$& $1$ & $0$ & $0$ & $1$ &$1^2 2^2$\\ \hline
			${\mathcal O}_{0}$ & $2$ & $2$ & $0$ & $0$ &$1^6$\\ \hline
		\end{tabular}
	\end{center}
	\caption{List of the values on the coefficient function  ${\mathbf c}$.}\label{T:Table5}
\end{table}
The columns of the  table~\ref{T:Table5}  are indexed by the  dimension of the indecomposable orthogonal representation of $A_2^{odd}$ and Jordan type for each orbit.	
The subscript in the notation for each orbit is its dimension. In this case, there are two $G^{\iota}$-orbits: $'{\mathcal O}_{2}$, ${\ }''{\mathcal O}_{2}$ of dimension 2. If we had been working with the group of automorphisms of the finite dimensional vector space $V$ over ${\mathbf k}$ preserving  a fixed non-degenerate symmetric form $\langle\ , \  \rangle:V \times V \rightarrow {\mathbf k}$ rather than the connected component of the identity element, these two orbits form for just one orthogonal orbit: $'{\mathcal O}_{2} \cup {\ }''{\mathcal O}_{2}$ of dimension 2.

Now  the tableau ${\mathcal T} \in {\mathfrak T}_{\delta}$, the tableau ${\mathcal R} \in  {\mathfrak R}_{\delta}$ for each orbit are
\begin{itemize}
\item For the orbit ${\mathcal O}_{4}$ of dimension 4

\begin{center}
${\mathcal T}_{4}$ = 
\ytableausetup{centertableaux}
\begin{ytableau}
\none & \none [-2] & \none [0] & \none [2] \\ \none [{\phantom -}2] & 2 & 0 & 0 \\  \none [{\phantom -}0] & 0 & 0 \\ \none [-2] & 0 \\  \end{ytableau}
\quad
\text{and}
\quad
${\mathcal R}_{4}$ = 
\ytableausetup{centertableaux}
\begin{ytableau}
\none & \none [-2] & \none [0] & \none [2] \\ \none [{\phantom -}2] & 2 & 2 & 2\\  \none [{\phantom -}0] & 2 & 2\\ \none [-2] & 2 \\ \end{ytableau}.

\end{center}

\item For the orbit ${\mathcal O}_{3}$ of dimension 3

\begin{center}
${\mathcal T}_{3}$ = 
\ytableausetup{centertableaux}
\begin{ytableau}
\none & \none [-2] & \none [0] & \none [2]\\ \none [{\phantom -}2] & 1 & 0 & 1\\  \none [{\phantom -}0] & 0 & 1\\ \none [-2] & 1 \\ \end{ytableau}
\quad
\text{and}
\quad
${\mathcal R}_{3}$ = 
\ytableausetup{centertableaux}
\begin{ytableau}
\none & \none [-2] & \none [0] & \none [2] \\ \none [{\phantom -}2] & 1 & 1 & 2 \\  \none [{\phantom -}0] & 1 & 2\\ \none [-2] & 2  \\ \end{ytableau}.

\end{center}

\item For the orbits $'{\mathcal O}_{2}$ and ${\ }''{\mathcal O}_{2}$ of dimension 2

\begin{center}
${\mathcal T}_{2}$ = 
\ytableausetup{centertableaux}
\begin{ytableau}
\none & \none [-2] & \none [0] & \none [2] \\ \none [{\phantom -}2] & 0 & 1 & 1\\  \none [{\phantom -}0] & 1 & 0\\ \none [-2] & 1 \\\end{ytableau}
\quad
\text{and}
\quad
${\mathcal R}_{2}$ = 
\ytableausetup{centertableaux}
\begin{ytableau}
\none & \none [-2] & \none [0] & \none [2] ]\\ \none [{\phantom -}2] & 0 & 1 & 2\\  \none [{\phantom -}0] & 1 & 2\\ \none [-2] & 2 \\ \end{ytableau}.

\end{center}

\item For the orbit ${\mathcal O}_{0}$ of dimension 0

\begin{center}
${\mathcal T}_{0}$ = 
\ytableausetup{centertableaux}
\begin{ytableau}
\none & \none [-2] & \none [0] & \none [2] \\ \none [{\phantom -}2] & 0 & 0 & 2\\  \none [{\phantom -}0] & 0 & 2\\ \none [-2] & 2 \\ \end{ytableau}
\quad
\text{and}
\quad
${\mathcal R}_{0}$ = 
\ytableausetup{centertableaux}
\begin{ytableau}
\none & \none [-2] & \none [0] & \none [2] \\ \none [{\phantom -}2] & 0 & 0 & 2\\  \none [{\phantom -}0] & 0  & 2 \\ \none [-2] &  2 \\  \end{ytableau}.

\end{center}

\end{itemize}

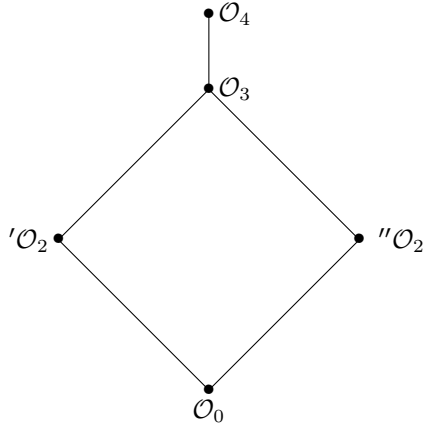
\begin{figure}[h]
\begin{center}
\begin{tikzpicture}
\draw (0, 0) -- (2, 2);
\draw (0, 0) -- (-2, 2);
\draw (2, 2) -- (0, 4);
\draw  (-2, 2) -- (0, 4);
\draw (0, 4) -- (0, 5);
\draw (0, 0) node[below] {${\mathcal O}_0$};
\draw (-2, 2) node[left] {${\ }'{\mathcal O}_2$};
\draw (2, 2) node[right] {${\  }''{\mathcal O}_2$};
\draw (0, 4) node[right] {${\mathcal O}_3$};
\draw (0, 5) node[right] {${\mathcal O}_4$};
\draw (0, 0) node {$\bullet$};
\draw (-2, 2) node {$\bullet$};
\draw (2, 2) node {$\bullet$};
\draw (0, 4) node {$\bullet$};
\draw (0, 5) node {$\bullet$};
\end{tikzpicture}
\end{center}
\caption{Poset obtained by the rank tableaux comparaison}\label{FigurePosetExample5.22}
\end{figure}

We have illustrated the Hasse diagram of the poset obtained when we compare the rank tableaux ${\mathcal R}$ for the set of orbits  in figure~\ref{FigurePosetExample5.22}. We could described each of these $G^{\iota}$-orbits as follows: an element of ${\mathfrak g}_2$ is given by the quiver representation
\[
\begin{CD}
	{\mathbf k}^2    @<X_0<<  {\mathbf k}^2   @< X_{-2} <<  {\mathbf k}^2
\end{CD}
\]
where $X_0$ and $X_{-2}$ are the matrices
\[
X_0 = \begin{bmatrix} a & b\\ c & d\end{bmatrix} \quad \text{ and } \quad X_{-2} = - \begin{bmatrix} d & b \\ c & a\end{bmatrix}.
\]
Then 
\[
\begin{aligned}
{\mathcal O}_4 &= \{(a, b, c, d) \in {\mathbf k}^4 \mid  (ad - bc) \ne 0\};\\ {\mathcal O}_3 &= \{(a, b, c, d) \in {\mathbf k}^4 \mid  (ad - bc) = 0 \text{ and either  $ad \ne 0$ or $bc \ne 0$ or $ab \ne 0$ or $cd \ne 0$}\};\\
{\mathcal O}_2 &= \{(a, b, c, d) \in {\mathbf k}^4 \mid \text{$ad = bc = ab = cd = 0$ and ($a, b, c, d) \ne (0, 0, 0, 0)$}\}\\
{\ }'{\mathcal O}_2 &=\{(a, b, c, d) \in {\mathbf k}^4 \mid \text{$a = c = 0$ and $(b, d) \ne (0, 0)$}\}\\
{\ }''{\mathcal O}_2 &=\{(a, b, c, d) \in {\mathbf k}^4 \mid \text{$b = d = 0$ and $(a, c) \ne (0, 0)$}\}\\
{\mathcal O}_0 &= \{(a, b, c, d) \in {\mathbf k}^4 \mid a = b= c = d = 0\}
\end{aligned}
\] 
We have included in this list the ${\tilde G}^{\iota}$-orbit ${\mathcal O}_2$ and how it splits into two $G^{\iota}$-orbits:  ${\ }'{\mathcal O}_2$ and ${\ }''{\mathcal O}_2$.
\end{example}

\begin{remark}
It would be interesting to know what is the Zariski closure of the $G^{\iota}$-orbits in this case. The article of Boos and Cerulli Irelli \cite{BI2021}  does not address this question.
\end{remark}

\section{Conditions for two $G^{\iota}$-orbits in ${\mathfrak g}_2$ to be in the same $G$-orbit.}

\subsection{}
In this section, we want to answer the question of when two $G^{\iota}$-orbits ${\mathcal O}$, ${\mathcal O}'$ in ${\mathfrak g}_2$ are in the same $G$-orbit where $G$ is either the reductive group $Sp(V)$ (respectively $SO(V)$) as in \ref{S:SetUpAeven} when $\vert {\mathcal I} \vert$ is even or the reductive group $Sp(V)$ (respectively $O(V)$) as in \ref{S:SetUpOdd} when $\vert {\mathcal I} \vert$ is odd or finally the reductive group $SO(V)$ as in \ref{S:SetUpOddOrtho} when $\vert {\mathcal I} \vert$ is odd.

Let two $G^{\iota}$-orbits ${\mathcal O}$, ${\mathcal O}'$ in ${\mathfrak g}_2$ whose  partitions corresponding to their Jordan  type are respectively $1^{a(1)} 2^{a(2)} 3^{a(3)} \dots$  and $1^{a'(1)} 2^{a'(2)} 3^{a'(3)} \dots$. If these two orbits are in the same $G$-orbit, then these two partitions will be equal (i.e. $a(k) = a'(k)$ for all $k \in {\mathbb N}$). This follows easily from the fact that they are $GL(V)$-conjugate. 

Thus in this section,   we will consider  two $G^{\iota}$-orbits ${\mathcal O}$, ${\mathcal O}'$ in ${\mathfrak g}_2$ that have the same partition corresponding to their Jordan type $1^{a(1)} 2^{a(2)} 3^{a(3)} \dots$ and  we will answer the question of when they are $G$-conjugate.

\subsection{}\label{SS:ConjClassJordanPart}
By propositions \ref{P:Height_Jordan_Type} and \ref{P:Height_Jordan_Type_Odd}, we know the  partition corresponding to the Jordan type for the elements in a $G^{\iota}$-orbit ${\mathcal O}$ and consequently in their $G$-orbit.  It is known  when the characteristic $p$ is 0 or  is sufficiently large that the nilpotent class in the Lie algebra of $Sp(V)$(respectively $O(V)$)  is the intersection of the nilpotent class of the Lie algebra of $GL(V)$ with the Lie algebra of $Sp(V)$ (respectively $O(V)$). Moreover a nilpotent class of $O(V)$ in its Lie algebra is also a nilpotent class of $SO(V)$ in its Lie algebra except when the   partition $1^{a(1)} 2^{a(2)} 3^{a(3)} \dots$ (in multiplicative form) corresponding to its Jordan  type is such that $a(k) = 0$ when $k \equiv 1 \pmod 2$ and  $a(k) \equiv 0 \pmod 2$ when $k \equiv 0 \pmod 2$.  We will say that these latter partitions  are {\it totally even}. See for example section 2.5 in \cite{S1982}. In this reference, it is quoted for unipotent classes, but it is also valid for the nilpotent classes when $p$ is $0$ or sufficiently large.

\begin{proposition}
Let $G$ be the reductive group $Sp(V)$ (respectively $O(V)$) in the situation either as in \ref{S:SetUpAeven} when $\vert {\mathcal I} \vert$  is even or  as in \ref{S:SetUpOdd} when $\vert {\mathcal I} \vert$ is odd and let ${\mathcal O}$ and ${\mathcal O}'$ be two $G^{\iota}$-orbits in ${\mathfrak g}_2$ who both have the same partition $1^{a(1)} 2^{a(2)} 3^{a(3)} \dots$ corresponding  to their Jordan type. Then these two $G^{\iota}$-orbits are in the same $G$-orbit. 
\end{proposition}
\begin{proof}
If their corresponding partitions are equal, then they are $GL(V)$-conjugate and by our observation in ~\ref{SS:ConjClassJordanPart} they are  $Sp(V)$-conjugate (resp.  $O(V)$-conjugate). 
\end{proof}

\begin{proposition}\label{P:OrbitNotTotallyEven}
Let $G$ be the reductive group $SO(V)$ in the situation either as in \ref{S:SetUpAeven} when $\vert {\mathcal I} \vert$ is even or  as in \ref{S:SetUpOddOrtho} when $\vert {\mathcal I} \vert$ is odd and let ${\mathcal O}$ and ${\mathcal O}'$ be two $G^{\iota}$-orbits in ${\mathfrak g}_2$ who both have the same  partition $1^{a(1)} 2^{a(2)} 3^{a(3)} \dots$ corresponding  to their Jordan type and this partition is not totally even. Then these two $G^{\iota}$-orbits are in the same $G$-orbit. 
\end{proposition}
\begin{proof}
If their corresponding partitions are equal, then they are $GL(V)$-conjugate and by our observation in ~\ref{SS:ConjClassJordanPart} and because the partition is not totally even, then they are  $SO(V)$-conjugate. 
\end{proof}

\subsection{}
So we are left to study  the case of the reductive group $G = SO(V)$ and  two $G^{\iota}$-orbits ${\mathcal O}$ and ${\mathcal O}'$ in ${\mathfrak g}_2$,  that  have the same partition $1^{a(1)} 2^{a(2)} 3^{a(3)} \dots$ corresponding  to their Jordan type  and this partition is totally even, i.e. the partition is of the form $2^{a(2)} 4^{a(4)} 6^{a(6)} \dots$ with $a(k) \equiv 0 \pmod 2$ when $k \equiv 0 \pmod 2$. 

In the case where $\vert {\mathcal I}\vert$ is odd and  ${\tilde G}$ denotes  the group $O(V)$,  if  ${\mathcal O}_{\mathbf c}$ denotes a ${\tilde G}^{\iota}$-orbit in ${\mathfrak g}_2$,  where  ${\mathbf c} \in {\mathfrak C}_{\delta}$ is a coefficient  function ${\mathbf c}:{\mathbf B}/\tau \rightarrow {\mathbb N}$ and $c:{\mathbf B} \rightarrow {\mathbb N}$ is the associated symmetric function to ${\mathbf c}$,  such that  $c(b(i, -i, 0)) > 0$ for at least one element $i \in {\mathcal I}$,  $i \geq 0$, then,  by  proposition~\ref{P:OrbitSplittingSO} (b), we know that this ${\tilde G}^{\iota}$-orbit  ${\mathcal O}_{\mathbf c}$  is a $G^{\iota}$-orbit, i.e. does not splits into two $G^{\iota}$-orbits:  ${}'{\mathcal O}_{\mathbf c}$ and ${}''{\mathcal O}_{\mathbf c}$. Moreover the partition  $1^{a(1)} 2^{a(2)} 3^{a(3)} \dots$ corresponding  to the Jordan type  of ${\mathcal O}_{\mathbf c}$ is not totally even.  This last fact follows because $\vert Supp(b(i,-i, 0)) \vert$ is odd for the ${\mathcal I}$-box $b(i, -i, 0)$ on the principal diagonal such that $c(b(i, -i, 0)) > 0$, where $i \in {\mathcal I}$, $i \geq  0$. 

In the case where $\vert {\mathcal I}\vert$ is odd, if  ${\tilde G}$ denotes  the group $O(V)$ and  ${\mathcal O}_{\mathbf c}$ denotes a ${\tilde G}^{\iota}$-orbit in ${\mathfrak g}_2$,  where  ${\mathbf c} \in {\mathfrak C}_{\delta}$ is a coefficient  function ${\mathbf c}:{\mathbf B}/\tau \rightarrow {\mathbb N}$ and $c:{\mathbf B} \rightarrow {\mathbb N}$ is the associated symmetric function to ${\mathbf c}$,    such that  $c(b(i, -i, 0)) = 0$ for all $i \in {\mathcal I}$,  $i \geq 0$, then,  by  proposition~\ref{P:OrbitSplittingSO} (c), we know that this ${\tilde G}^{\iota}$-orbit  ${\mathcal O}_{\mathbf c}$  splits into two $G^{\iota}$-orbits:  ${}'{\mathcal O}_{\mathbf c}$ and ${}''{\mathcal O}_{\mathbf c}$.  By the preceding paragraph, it is among these $G^{\iota}$-orbits that we must look for $G^{\iota}$-orbits whose partition  $1^{a(1)} 2^{a(2)} 3^{a(3)} \dots$ corresponding  to the Jordan type could be totally even. 

\begin{lemma}\label{L:TotallyEvenCaseCoefficient}
Let $G$ be the reductive group $SO(V)$   in the situation  as in \ref{S:SetUpAeven} when $\vert {\mathcal I} \vert$ is even (respectively as in \ref{S:SetUpOddOrtho} when $\vert {\mathcal I} \vert$ is odd) and let ${\mathcal O} = {\mathcal O}_{\mathbf c}$  (respectively ${}'{\mathcal O}_{\mathbf c}$ or ${}''{\mathcal O}_{\mathbf c}$) be a $G^{\iota}$-orbit in ${\mathfrak g}_2$ whose partition $2^{a(2)} 4^{a(4)} 6^{a(6)} \dots$ corresponding  to its Jordan type  is totally even, where ${\mathbf c} \in {\mathfrak C}_{\delta}$. Denote by $c:{\mathbf B} \rightarrow {\mathbb N}$: the symmetric function associated to ${\mathbf c}$.  If $b \in {\mathbf B}$ is such that  $\lambda(b) \not\equiv \vert {\mathcal I} \vert \pmod 2$, then $c(b) = 0$.  Recall  that $\lambda(b) = (i + j)/2$ for $b = b(i, j, k)$, with $i, j \in {\mathcal I}$,  $i \geq j$ and $0 \leq k \leq \mu(i, j)$ as defined in \ref{N:LambdaValue}. 
\end{lemma}
\begin{proof}
In both cases: $\vert {\mathcal I} \vert$ even or odd,  the result follows  easily from the fact that  if $b \in {\mathbf B}$, then $\vert Supp(b) \vert \equiv (\lambda(b) + \vert {\mathcal I} \vert) \pmod 2$ and  the partition corresponding to the  Jordan type is totally even. 
\end{proof}

\begin{notation}\label{N:SuppCardinality}
With the same notation and hypothesis as the above lemma, we get that if $c(b) \ne 0$ for a symmetric function $c:{\mathbf B} \rightarrow  {\mathbb N}$, then $\vert Supp(b) \vert$ is even.  For $b = b(i, j, k)$ with $i, j \in {\mathcal I}$, $i \ge j$ and $0 \leq k \leq \mu(i, j)$, we denote by 
\[
Supp^{top}(b) = \left\{ i' \in {\mathcal I}\  \left\vert\   i \geq i' \geq \left(\frac{i + j + 2}{2}\right)\right. \right\}
\]
This is the top-half of the elements in the support $Supp(b)$ of the ${\mathcal I}$-box $b \in {\mathbf B}$.
\end{notation}

We will consider separately the situation in \ref{S:SetUpAeven} when $\vert {\mathcal I} \vert$ is even  and the one in \ref{S:SetUpOddOrtho} when $\vert {\mathcal I} \vert$ is odd.

\begin{lemma}\label{L:MaximalIsotropicV(c)}
Let $G$ be the reductive group $SO(V)$  as in \ref{S:SetUpAeven} where $\vert {\mathcal I} \vert$ is even and let ${\mathcal O} = {\mathcal O}_{\mathbf c}$  be a $G^{\iota}$-orbit in ${\mathfrak g}_2$,  whose partition $2^{a(2)} 4^{a(4)} 6^{a(6)} \dots$ corresponding  to its Jordan type  is totally even, where ${\mathbf c}:{\mathbf B}/\tau \rightarrow {\mathbb N}$ is a coefficient function ${\mathbf c} \in {\mathfrak C}_{\delta}$. Denote by $c:{\mathbf B} \rightarrow {\mathbb N}$: the symmetric function associated to ${\mathbf c}$.  We will now recall the notation of \ref{N:CCorrespondance}.  $V({\mathbf c})$ is a vector space with its ${\mathcal I}$-basis 
\[
{\mathcal B}_{\mathbf c} = \coprod_{{\mathcal O} \in {\mathbf B}/\tau} \coprod_{b \in {\mathcal O}} \{v_i^j(b) \mid i \in Supp(b), 1 \leq j \leq {\mathbf c}(b)\}
\]
and  the non-degenerate bilinear form $\langle \  ,\   \rangle_{\mathbf c}: V({\mathbf c}) \times V({\mathbf c}) \rightarrow {\mathbf k}$,   $E_{\mathbf c}: V({\mathbf c}) \rightarrow V({\mathbf c})$ is also defined in \ref{N:CCorrespondance}. It is a nilpotent element whose partition corresponding to the Jordan type is  the even partition $2^{a(2)} 4^{a(4)} 6^{a(6)} \dots$. Then  the vector subspace 
\[
L({\mathbf c}) = E_{\mathbf c} [\ker(E^2_{\mathbf c})] + E^2_{\mathbf c} [\ker(E^4_{\mathbf c})]  + \dots + E^k_{\mathbf c} [\ker(E^{2k}_{\mathbf c})]  + \dots
\]
is the vector subspace of $V({\mathbf c})$ spanned by the basis
\[
\coprod_{\substack { b \in {\mathbf B} \\ c(b) \ne 0}} \{v_i^j(b) \mid i \in Supp^{top}(b), 1 \leq j \leq c(b)\}
\]
and this is a maximal isotropic subspace of $V({\mathbf c})$. 
\end{lemma}
\begin{proof}
Recall that the linear transformation $E_{\mathbf c}$ on the elements of the basis ${\mathcal B}_{\mathbf c}$  was described in notation~\ref{N:CCorrespondance}. Let ${\mathcal O}$ denote the $\langle \tau \rangle$-orbit of $b$, then,  if the orbit ${\mathcal O}$ has no overlapping ${\mathcal I}$-box supports,  $1 \leq j \leq {\mathbf c}({\mathcal O})$ and $i \in Supp(b)$, we have
\[
E_{\mathbf c}(v_i^j(b)) = \begin{cases} {\phantom -} 0, &\text{ if $i = \max(Supp(b))$;}\\ {\phantom -} v_{i + 2}^j(b), &\text{ if $i \ne \max(Supp(b))$ and $i > 0$;}\\  -v_{i + 2}^j(b), &\text{ if $i \ne \max(Supp(b))$ and $i < 0$;} \end{cases}
\]
while if the orbit ${\mathcal O}$ has overlapping ${\mathcal I}$-box supports,  $b = b(i_1, j_1, k_1) \in {\mathcal O}$, $1 \leq j \leq {\mathbf c}({\mathcal O})$  and $i \in Supp(b)$, then we have 
\[
E_{\mathbf c}(v_i^j(b)) = \begin{cases} {\phantom -} 0, &\text{ if $i = \max(Supp(b))$;}\\ {\phantom -} v_{i + 2}^j(b), &\text{ if $i \ne \max(Supp(b))$ and $i \geq 1$;}\\  -v_{i + 2}^j(b), &\text{ if  $i < -1$;}\\ {\phantom -} v_1^j(b), &\text{ if  $i = -1$ and $b$ is below  the principal diagonal;}\\ - v_1^j(b), &\text{ if $i = -1$ and $b$ is above  the principal diagonal;}\\  (- 1)^{k_1} v_1^j(b), &\text{ if $i = -1$ and $b$ is on  the principal diagonal.}
\end{cases}
\]
Let $b \in {\mathbf B}$ be an ${\mathcal I}$-box such that $c(b) \ne 0$. As noted in \ref{N:SuppCardinality}, then $\vert Supp(b) \vert$ is even. Write this cardinality as $2s$. 

From our definition of $E_{\mathbf c}$, we get easily that 
\begin{enumerate}[\upshape (i)]
\item for $1 \leq k \leq s$ and $1 \leq j \leq c(b)$, then 
\[
E^{2k}_{\mathbf c}(v_i^j(b)) = 0 \iff i > (\max(Supp(b)) - 4k);
\]
\item for $k > s$ and $1 \leq j \leq c(b)$, then $E^{2k}_{\mathbf c}(v_i^j(b)) = 0$ for all $i  \in Supp(b)$.
\end{enumerate}
If $1 \leq k \leq s$, $1 \leq j \leq c(b)$ and $i > (\max(Supp(b)) - 4k)$, then $v_i^j(b) \in \ker(E^{2k}_{\mathbf c})$ and 
\[
E^k_{\mathbf c}(v_i^j(b)) \text{ is } \begin{cases} 0, &\text{if $(i + 2k) > \max(Supp(b))$;}\\  \gamma v_{i + 2k}^j(b),  &\text{if $(i + 2k) \leq \max(Supp(b))$;} \end{cases}
\]
where $\gamma \in \{1, -1\}$.  

In this second case, i.e. $(i + 2k) \leq \max(Supp(b))$, then $(i + 2k)  \in Supp^{top}(b)$, because $
Supp(b) = \{i' \in {\mathcal I} \mid \max(Supp(b)) \geq i' \geq \min(Supp(b)) \}$,  with $\min(Supp(b)) =  (\max(Supp(b)) - 4s + 2)$ and
\[
Supp^{top}(b) = \{i' \in {\mathcal I} \mid \max(Supp(b)) \geq i'  \geq (\max(Supp(b)) - 2s + 2) \}.
\]
So in this second case, we have $(i + 2k) \leq \max(Supp(b))$ and also the lower inequality $(i + 2k) > (\max(Supp(b)) - 2k) \geq (\max(Supp(b)) - 2s)$, because $1 \leq k \leq s$.  This show that $(i + 2k) \in Supp^{top}(b)$. 

If $k > s$, $1 \leq j \leq c(b)$ and $i \in Supp(b)$, then $v_i^j(b) \in \ker(E^{2k}_{\mathbf c})$ and 
\[
E^k_{\mathbf c}(v_i^j(b)) \text{ is } \begin{cases} 0, &\text{if $(i + 2k) > \max(Supp(b))$;}\\  \gamma v_{i + 2k}^j(b),  &\text{if $(i + 2k) \leq \max(Supp(b))$;} \end{cases}
\]
where $\gamma \in \{1, -1\}$.  

In this second case, i.e. $(i + 2k) \leq \max(Supp(b))$, then $(i + 2k)  \in Supp^{top}(b)$.  because we have $(i + 2k) \leq \max(Supp(b))$ and also the lower inequality $(i + 2k) \geq  (\max(Supp(b)) - 4s + 2 + 2k) > (\max(Supp(b)) - 2s + 2)$, because $k > s$.  This show that $(i + 2k) \in Supp^{top}(b)$. 

Thus $L({\mathbf c})$ is included in the subspace of $V({\mathbf c})$ spanned by 
\[
\coprod_{\substack { b \in {\mathbf B} \\ c(b) \ne 0}} \{v_i^j(b) \mid i \in Supp^{top}(b), 1 \leq j \leq c(b)\}.
\]

Reciprocally if $b \in {\mathbf B}$ with $c(b) > 0$, $1 \leq j \leq c(b)$ and $i \in Supp^{top}(b)$, then we want to prove that $v_i^j(b) \in L({\mathbf c})$. Write again  $\vert Supp(b) \vert = 2s$ and we have that $Supp(b) = \{i' \in {\mathcal I} \mid \max(Supp(b)) \geq i' \geq \min(Supp(b)) \}$,  with $\min(Supp(b)) =  (\max(Supp(b)) - 4s + 2)$ . Because $i \in Supp^{top}(b)$, this means that $\max(Supp(b)) \geq i \geq (\max(Supp(b)) - 2s + 2)$. Write $i = \max(Supp(b)) - 2r$ with $0 \leq r \leq (s - 1)$.  We have  that 
\[
\begin{aligned}
\max(Supp(b)) \geq (i - 2r - 2) &\geq (\max(Supp(b)) - 4r - 2)\\  &\geq (\max(Supp(b)) - 4(s - 1) - 2 = \min(Supp(b)).
\end{aligned}
\]
Thus $(i - 2r - 2)  \in Supp(b)$ and we have  $E^{(r + 1)}_{\mathbf c}(v_{(i - 2r - 2)}^j(b)) = \gamma v_i^j(b)$, where $\gamma \in \{1, -1\}$.  
Note that $v_{(i - 2r - 2)}^j(b) \in \ker E_{\mathbf c}^{2(r + 1)}$, because  $(i - 2r - 2) + 4(r + 1) = (\max(Supp(b)) - 4r - 2) + 4(r + 1) > \max(Supp(b))$.
From all of the above, $L({\mathbf c})$ is the vector subspace of $V({\mathbf c})$ spanned by the basis
\[
\coprod_{\substack { b \in {\mathbf B} \\ c(b) \ne 0}} \{v_i^j(b) \mid i \in Supp^{top}(b), 1 \leq j \leq c(b)\}.
\]

The dimension $\dim(L({\mathbf c}))$ of $L({\mathbf c})$ is 
\[
\dim(L({\mathbf c})) = \sum_{b \in {\mathbf B}} c(b) \vert Supp^{top}(b) \vert = \frac{1}{2} \sum_{b \in {\mathbf B}} c(b) \vert Supp(b) \vert = \frac{1}{2} \dim(V({\mathbf c})).
\]

We will now show that $L({\mathbf c})$ is an isotropic subspace. From our construction of the bilinear form $\langle\  , \  \rangle_{\mathbf c}$ in notation~\ref{N:CCorrespondance}, we have $\langle v_i^j(b), v_{i'}^{j'}(b')\rangle \ne 0$ only if $j' = j$, $b' = \tau(b)$ and $i + i' = 0$.   Let $b \in {\mathbf B}$ with $c(b) > 0$, $1 \leq j \leq c(b)$,  $i \in Supp^{top}(b)$, $b' = \tau(b)$, $j' = j$ and  $i' \in Supp^{top}(b')$.  Write $\vert Supp(b) \vert = 2s$.  As above 
\[
Supp(b) = \{i'' \in {\mathcal I} \mid (\max(Supp(b)) - 4s + 2) \leq i'' \leq \max(Supp(b))\}.
\]
Thus $\vert Supp(\tau(b)) \vert = 2s$ and 
\[
Supp(\tau(b)) = \{i''' \in {\mathcal I} \mid  -\max(Supp(b)) \leq i''' \leq -(\max(Supp(b)) - 4s + 2)\}.
\]
From this we get that 
\[
Supp^{top}(b) = \{ i'' \in {\mathcal I} \mid \max(Supp(b)) \geq i'' \geq (\max(Supp(b)) - 2s + 2) \}
\]
and 
\[
Supp^{top}(\tau(b)) = \{ i''' \in {\mathcal I} \mid -(\max(Supp(b)) - 4s + 2) \geq i''' \geq -(\max(Supp(b)) - 2s) \}
\]
Now if $i \in Supp^{top}(b)$ and $i' \in Supp^{top}(\tau(b))$, then 
\[
(\max(Supp(b)) - 2s + 2) \leq i \leq \max(Supp(b)) 
\]
and 
\[
(-\max(Supp(b)) + 2s) \leq i' \leq (-\max(Supp(b)) + 4s - 2).
\]
From this, we get that $(i+ i') \geq 2$ and consequently $(i + i') \ne 0$ and $L({\mathbf c})$ is an isotropic subspace. It is a maximal isotropic subspace,  because $2 \dim(L({\mathbf c})) = \dim(V({\mathbf c}))$. 
\end{proof}

The following lemma is probably well-known. Its proof is included for completeness. 

\begin{lemma}\label{L:ConditionSO(V)orbit}
 Let ${\mathcal I}$ be even and using the notation of  \ref{S:SetUpAeven}, consider the set ${\mathcal J}$ of subsets $J$ of  $\{(i, j) \in {\mathcal I} \times {\mathbb N} \mid 1 \leq j \leq \delta_i\}$ such that  if $(i, j), (i', j) \in J$, then $i + i' \ne 0$ and  $J$ is maximal among such subsets.
\begin{enumerate}[\upshape (a)]
\item If $J \in {\mathcal J}$, then $2 \vert J \vert = \dim(V)$.
\item If $J \in {\mathcal J}$, then the subspace $L_J$ of $V$ spanned by $\{u_{i, j} \mid  (i, j) \in J\}$ is a maximal isotropic subspace of $V$.
\item Let  $J, J' \in {\mathcal J}$. Then the two maximal isotropic subspaces $L_J$ and $L_{J'}$ of $V$ are in the same $SO(V)$-orbit if and only if 
\[
 \vert \{(i, j) \in J \mid i < 0\}\vert  \equiv  \vert \{(i, j) \in J' \mid i < 0\}\vert  \pmod 2.
\]
\end{enumerate}
\end{lemma}
\begin{proof}
(a) From our notation in  \ref{S:SetUpAeven}, we have that 
\[
\dim(V) = \sum_{i \in {\mathcal I}} \delta_i =  \sum_{i \in {\mathcal I}}\vert \{ (i, j) \in {\mathcal I} \times {\mathbb N} \mid 1 \leq j \leq \delta_i\}\vert
\]
If $(i , j) \in {\mathcal I} \times {\mathbb N}$ is such that $1 \leq j \leq \delta_i$, then,  by the maximality of $J$, exactly one and only one of $(i, j)$, $(-i, j)$ belongs to $J$. Consequently  $2 \vert J \vert = \dim(V)$. 

(b) First we will show that $L_J$ is isotropic. From our hypothesis on the basis as defined in \ref{S:SetUpAeven}, we have that $\langle u_{i, j}, u_{i', j'}\rangle \ne 0$ if and only if $j = j'$ and $i + i' = 0$ when $(i, j), (i', j') \in {\mathcal I} \times {\mathbb N}$ are such that $1 \leq j \leq \delta_i$ and $1 \leq j' \leq \delta_{i'}$. But by our hypothesis on $J$, when $(i, j), (i', j') \in J$, then  $\langle u_{i, j}, u_{i', j'} \rangle = 0$. Otherwise we would have $j = j'$ and $i + i' = 0$. But this is impossible for two elements of $J$.  So $L_J$ is isotropic. Because $2\dim(L_J) = \dim(V)$ by (a), we get that $L_J$ is maximal. 

(c) We first need to remind some facts. Let $J_+ = \{(i, j) \in {\mathcal I} \times {\mathbb N} \mid i > 0, 1 \leq j \leq \delta_i\}$. Clearly $J_+ \in {\mathcal J}$.  Note that if $g \in O(V)$ is such that $g(L_{J_+}) = L_{J_+}$, then $g \in SO(V)$.  In fact by considering the matrix of $g$ in the basis $\{u_{i, j} \mid i \in {\mathcal I},  1 \leq j \leq \delta_i\}$ ordered properly, we get that this matrix is of the form 
\[
\begin{bmatrix} A & B\\ 0 & D\end{bmatrix} \quad \text{ such that } \quad  \begin{bmatrix} A & B\\ 0 & D\end{bmatrix}^T \begin{bmatrix} 0 & I\\ I & 0\end{bmatrix} \begin{bmatrix} A & B\\ 0 & D\end{bmatrix} = \begin{bmatrix} 0 & I\\ I & 0\end{bmatrix}.
\]
Consequently $A^TD = I$ ,  $\det(D) = (\det(A))^{-1}$ and $g \in SO(V)$. 

Let $g_J: V \rightarrow V$ denote the unique linear transformation such that
\[
g_J(u_{i, j}) = \begin{cases} u_{i, j}, &\text{if $i > 0$ and $(i, j) \in J$:}\\  u_{-i, j}, &\text{if $i >0$ and $(i, j) \not \in J$;}\\ u_{-i, j}, &\text{if $i < 0$ and $(i, j) \in J$;} \\ u_{i, j}, &\text{if $i < 0$ and $(i, j) \not\in J$.}
\end{cases}
\]
We can now show that $g_J \in O(V)$.  For $i_1, i_2 \in {\mathcal I}$, $1 \leq j_1 \leq \delta_{i_1}$ and $1 \leq j_2 \leq \delta_{i_2}$,   we must consider $\langle u_{i_1, j_1}, u_{i_2,j_2}\rangle$ and $\langle g_J(u_{i_1, j_1}), g_J(u_{i_2,j_2)}\rangle$.  We have 
\[
\langle u_{i_1, j_1}, u_{i_2, j_2}\rangle = \begin{cases} 1, &\text{ if $j_1 = j_2$ and $i_1 + i_2 = 0$;}\\ 0, &\text{otherwise.} \end{cases}
\]

If $j_1 \ne j_2$, then 
\[
\langle g_J(u_{i_1, j_1}), g_J(u_{i_2, j_2}) \rangle = \langle u_{i'_1, j_1}, u_{i'_2, j_2}\rangle = 0, 
\]
where $i'_1 \in \{i_1, -i_1\}$ and $i'_2  \in \{i_2, -i_2\}$, because $j_1 \ne j_2$.  

If $j_1 = j_2$ and $i_1 + i_2 = 0$, then we can assume by symmetry  that $i_1 > 0$ and $i_2 = -i_1 < 0$. Moreover $(i_1, j_1) \in J$ if and only if $(i_2, j_2) \not\in J$. Thus 
\[
\langle g_J(u_{i_1, j_1}), g_J(u_{i_2, j_2})\rangle = \begin{cases} \langle u_{i_1, j_1}, u_{i_2, j_2} \rangle, &\text{if $(i_1, j_1) \in J$;}\\  \langle u_{-i_1, j_1}, u_{-i_2, j_2} \rangle, &\text{if $(i_1, j_1) \not\in J$;} \end{cases} \quad = 1.
\]

If $j_1 = j_2$, $i_1 + i_2 \ne 0$ and  $i _1 - i_2 \ne 0$, then 
\[
\langle g_J(u_{i_1, j_1}), g_J(u_{i_2, j_2})\rangle = \langle u_{i'_1, j_1}, u_{i'_2, j_2}\rangle = 0
\]
because $i'_1 \in \{i_1, -i_1\}$, $i'_2 \in \{i_2, -i_2\}$ and $i'_1 + i'_2 \ne 0$ from our hypothesis.

Finally if $j_1 = j_2$, $i_1 + i_2 \ne 0$ and $i_1 = i_2$, then 
\[
\langle g_J(u_{i_1, j_1}), g_J(u_{i_2, j_2})\rangle  =  \langle g_J(u_{i_1, j_1}), g_J(u_{i_1, j_1})\rangle = \langle u_{i'_1, j_1}, u_{i'_1, j_1}\rangle = 0
\]
where  $i'_1 \in \{i_1, -i_1\}$. 

This completes the proof that $g_J \in O(V)$.   Note also that we have $g_J(L_{J_+}) =L_J$, because we only have to consider the basis element $u_{i, j}$ for $i > 0$ and in that case we get if $(i, j) \not\in J$, that $(-i, j) \in J$ and 
\[
g_J(u_{i, j}) =  \begin{cases} u_{i, j} \in L_J, &\text{if $i > 0$ and $(i, j) \in J$:}\\  u_{-i, j} \in L_J, &\text{if $i >0$ and $(i, j) \not \in J$.}
\end{cases}
\]

We can now compute the determinant of the matrix of $g_J$ relative to the basis $\{u_{i, j} \mid i \in {\mathcal I}, 1 \leq j \leq \delta_i\}$ where we consider first the basis elements $u_{i, j}$ for $(i, j) \in J_+$  and then the basis element $u_{-i, j}$  for $(i, j) \in J_+$ using the same total order for $J_+$. Using the Laplace expansion of $\det(g_J)$ and using recursion, we get easily that 
\[
\det(g_J) = (-1)^{\vert \{(i, j) \in J \mid i < 0\}\vert}
\]
If $g \in O(V)$ and $g(L_{J_+}) = L_J$, then $\det(g) = (-1)^{\vert \{(i, j) \in J \mid i < 0\}\vert}$, because 
\[
g(L_{J_+}) = L_J = g_J(L_{J_+})  \quad \Rightarrow \quad g^{-1}_Jg(L_{J_+}) =L_{J_+} \quad \text{ and } \quad g^{-1}_Jg \in SO(V)
\]
and $\det(g) = \det(g_J)$. 

We  complete the proof of (c).  If $L_J$ and $L_{J'}$ are in the same $SO(V)$-orbit, then they are both in the $SO(V)$-orbit of $L_{J_+}$ or both are not in the $SO(V)$-orbit of $L_{J_+}$. Thus 
\[
(-1)^{\vert \{(i, j) \in J \mid i < 0\}\vert} = (-1)^{\vert \{(i, j) \in J' \mid i < 0\}\vert},  
\]
 i.e.  $\vert \{(i, j) \in J \mid i < 0\}\vert  \equiv  \vert \{(i, j) \in J' \mid i < 0\}\vert  \pmod 2$.  
 
 Conversely if $\vert \{(i, j) \in J \mid i < 0\}\vert  \equiv  \vert \{(i, j) \in J' \mid i < 0\}\vert  \pmod 2$, then we get that $L_J$ and $L_{J'}$ are both in the $SO(V)$-orbit of $L_{J_+}$ or both are not in the $SO(V)$-orbit of $L_{J_+}$.  Because there are exactly two $SO(V)$-orbits on the set of maximal isotropic subspaces of $V$, we get that $L_J$ and $L_{J'}$ are in the same $SO(V)$-orbit. 
\end{proof}

\begin{notation}
Let $G$ be the reductive group $SO(V)$  as in \ref{S:SetUpAeven} where $\vert {\mathcal I} \vert$ is even and let ${\mathcal O} = {\mathcal O}_{\mathbf c}$ be a $G^{\iota}$-orbit  in ${\mathfrak g}_2$,  whose Jordan type is a  totally even partition $2^{a(2)} 4^{a(4)} 6^{a(6)} \dots$, where ${\mathbf c}:{\mathbf B}/\tau \rightarrow {\mathbb N}$ is a coefficient function ${\mathbf c} \in {\mathfrak C}_{\delta}$. Denote by $c:{\mathbf B} \rightarrow {\mathbb N}$: the symmetric function associated respectively to ${\mathbf c}$.  We will denote  
\[
\varXi({\mathcal O}) = \varXi({\mathcal O}_{\mathbf c}) =  \sum_{b \in {\mathbf B}} c(b) \vert \{i \in Supp^{top}(b) \mid i < 0\}\vert 
\]
and call $\varXi({\mathcal O})$: the {\it parity of the orbit} ${\mathcal O}$.  We will say that the $G^{\iota}$-orbit ${\mathcal O}$ is {\it even}. (respectively {\it odd}) if $\varXi({\mathcal O})$ is even  (respectively odd). 
\end{notation}

\begin{proposition}\label{P:TotallyEvenOrbitSOIEvenCase}
Let $G$ be the reductive group $SO(V)$  as in \ref{S:SetUpAeven} where $\vert {\mathcal I} \vert$ is even and let ${\mathcal O} = {\mathcal O}_{\mathbf c}$ and  ${\mathcal O}' = {\mathcal O}_{{\mathbf c}'}$  be two $G^{\iota}$-orbits in ${\mathfrak g}_2$,  who both have the same totally even partition $2^{a(2)} 4^{a(4)} 6^{a(6)} \dots$ corresponding  to their Jordan type, where ${\mathbf c}:{\mathbf B}/\tau \rightarrow {\mathbb N}$ and ${\mathbf c}':{\mathbf B}/\tau \rightarrow {\mathbb N}$ are both coefficient functions ${\mathbf c}, {\mathbf c}' \in {\mathfrak C}_{\delta}$. Denote by $c:{\mathbf B} \rightarrow {\mathbb N}$ and $c':{\mathbf B} \rightarrow {\mathbb N}$: the symmetric functions associated respectively to ${\mathbf c}$ and ${\mathbf c}'$.  Then these two $G^{\iota}$-orbits  ${\mathcal O}$ and  ${\mathcal O}' $ are in the same $G$-orbit if and only if   ${\mathcal O}$ and  ${\mathcal O}' $ have the same parity, i.e. they are both even or both odd.
\end{proposition}
\begin{proof}
Recall that ${\mathcal O}_{\mathbf c}$ is the $G^{\iota}$-orbit of the element $X_{\mathbf c} = T_{\mathbf c} E_{\mathbf c} T^{-1}_{\mathbf c}: V \rightarrow V$ where $E_{\mathbf c}:V({\mathbf c}) \rightarrow V({\mathbf c})$ is defined in \ref{N:CCorrespondance}  and $T_{\mathbf c}:V({\mathbf c}) \rightarrow V$ is the ${\mathcal I}$-graded isomorphism defined in \ref{SS:IsomorphismV(c)andV}. We have shown in lemma~\ref{L:TcIsomorphism} that $\langle T_{\mathbf c}(u), T_{\mathbf c}(v)\rangle = \langle u, v \rangle_{\mathbf c}$ for all $u, v \in V({\mathbf c})$.  Using the notation of the lemma~\ref{L:MaximalIsotropicV(c)}, we get easily that $T_{\mathbf c}(L({\mathbf c}))$ is the maximal isotropic subspace of $V$ equal to 
\[
X_{\mathbf c}[\ker(X^2_{\mathbf c})] + X^2_{\mathbf c}[\ker(X^4_{\mathbf c})] + \dots + X^k_{\mathbf c}[\ker(X^{2k}_{\mathbf c})] + \dots
\]
and it is spanned by  
\[
\coprod_{\substack { b \in {\mathbf B} \\ c(b) \ne 0}} \{ T_{\mathbf c}(v_i^s(b)) \mid  i \in Supp^{top}(b) \text{ and  } 1 \leq s \leq c(b)\}.
\]
From our definition of $T_{\mathbf c}$, we get that this set is 
\[
\coprod_{\substack { b \in {\mathbf B} \\ c(b) \ne 0}} \{u_{i, j} \mid i \in Supp^{top}(b) \text{ and }  j \in J_i(b)\}
\]
where $J_i(b) = \{1 \leq j \leq \delta_i \mid u_{i, j} = T_{\mathbf c}(v_i^s(b)) \text{ for some $1 \leq s \leq c(b)$} \}$.   For all $i \in {\mathcal I}$ and $b, b' \in {\mathbf B}$,  note that $\vert J_i(b) \vert = c(b)$ and if $b \ne b'$  are two distinct ${\mathcal I}$-boxes such that $i \in Supp(b) \cap Supp(b')$, then $J_i(b) \cap J_i(b') = \emptyset $. Because $L({\mathbf c})$ is a maximal isotropic subspace of $V({\mathbf c})$, we get that 
\[
J({\mathbf c}) = \coprod_{\substack { b \in {\mathbf B} \\ c(b) \ne 0}} \{(i, j) \in {\mathcal I} \times {\mathbb N} \mid   i \in Supp^{top}(b) \text{ and }  j \in J_i(b)\} \in {\mathcal J}
\]
and $T_{\mathbf c}(L({\mathbf c})) = L_{J({\mathbf c})}$ with the notation of lemma~\ref{L:ConditionSO(V)orbit}.

We can do the same thing for the second orbit ${\mathcal O}_{{\mathbf c}'}$.  Now if ${\mathcal O}_{\mathbf c}$  and ${\mathcal O}_{{\mathbf c}'}$ are in the same $SO(V)$-orbit, then the two maximal isotropic subspaces $T_{\mathbf c}(L({\mathbf c}))$ and $T_{{\mathbf c}'}(L({\mathbf c}'))$ are in the same $SO(V)$-orbit and by lemma~\ref{L:ConditionSO(V)orbit}, we get
\[
\sum_{b \in {\mathbf B}} c(b) \vert \{i \in Supp^{top}(b) \mid i < 0\}\vert \equiv \sum_{b \in {\mathbf B}} c'(b) \vert \{i \in Supp^{top}(b) \mid i < 0\}\vert  \pmod 2.
\]

Reciprocally assume that 
\[
\sum_{b \in {\mathbf B}} c(b) \vert \{i \in Supp^{top}(b) \mid i < 0\}\vert \equiv \sum_{b \in {\mathbf B}} c'(b) \vert \{i \in Supp^{top}(b) \mid i < 0\}\vert  \pmod 2.
\]
 and the two orbits ${\mathcal O}_{\mathbf c}$  and ${\mathcal O}_{{\mathbf c}'}$ are not in the same $SO(V)$-orbit, but they are in the same $O(V)$-orbit because they both have the same partition $2^{a(2)} 4^{a(4)} 6^{a(6)} \dots$ corresponding  to their Jordan type.  This means that there exists $g_3X_{\mathbf c} g_3^{-1} = X_{{\mathbf c}'}$ with $g_3 \in O(V)$ and $\det(g_3) = -1$.  Consequently we have that $g_3 T_{\mathbf c}(L({\mathbf c})) g_3^{-1} = T_{{\mathbf c}'}(L({\mathbf c}')$ with $g_3 \in O(V)$ and $\det(g_3) = -1$.  One of these maximal isotropic subspaces, say $T_{\mathbf c}(L({\mathbf c}))$ is in the $SO(V)$-orbit of $L_{J_+}$ and the other $T_{{\mathbf c}'}(L({\mathbf c}'))$  is not in the $SO(V)$-orbit of $L_{J_+}$. By the proof of lemma~\ref{L:ConditionSO(V)orbit}, we must have
 \[
\sum_{b \in {\mathbf B}} c(b) \vert \{i \in Supp^{top}(b) \mid i < 0\}\vert \equiv 0 \pmod 2
\]
and 
\[
\sum_{b \in {\mathbf B}} c'(b) \vert \{i \in Supp^{top}(b) \mid i < 0\}\vert \equiv 1 \pmod 2.
\]
But this contradicts our hypothesis. So the two $G^{\iota}$-orbits ${\mathcal O}_{\mathbf c}$  and ${\mathcal O}_{{\mathbf c}'}$ are in the same $G$-orbit.

\end{proof}

\subsection{ }\label{SS:IoddGGiotaConjugaisonTotallyEven}
To conclude this section, we will consider the case where $\vert {\mathcal I} \vert$ is odd and the group $G$ is the reductive group $SO(V)$ as in \ref{S:SetUpOddOrtho}. Because of proposition~\ref{P:OrbitNotTotallyEven} and lemma~\ref{L:TotallyEvenCaseCoefficient}, we are left with the study of the $G^{\iota}$-orbit  ${\mathcal O}$ in ${\mathfrak g}_2$ being either ${}'{\mathcal O}_{\mathbf c}$ or ${}''{\mathcal O}_{\mathbf c}$ as defined in proposition~\ref{P:OrbitSplittingSO} (c) and its proof,  whose partition corresponding to its Jordan type is the totally even partition $2^{a(2)} 4^{a(4)} 6^{a(6)} \dots$, where  ${\mathbf c}:{\mathbf B}/\tau \rightarrow {\mathbb N}$ is a coefficient function ${\mathbf c} \in {\mathfrak C}_{\delta}$. Denote by $c:{\mathbf B} \rightarrow {\mathbb N}$: the symmetric function associated respectively to ${\mathbf c}$.  By lemma~\ref{L:TotallyEvenCaseCoefficient}, if $b \in {\mathbf B}$, then we have that $c(b) = 0$ whenever $\lambda(b)$ is even. 

Because $c(b(i, -i, 0)) = 0$ for all $i \in {\mathcal I}$, $i \geq 0$, we get as proved in  \ref{SS:IsomorphismV(c)andVOddCase} that $\dim(V_0) = \delta_0$ is even and we can write $\delta_0 =2 \delta'''_0$. Here we are using the same notation as in \ref{SS:IsomorphismV(c)andVOddCase}. 

The $G^{\iota}$-orbits ${}'{\mathcal O}_{\mathbf c}$ and ${}''{\mathcal O}_{\mathbf c}$ in ${\mathfrak g}_2$ were defined in  \ref{N:CCorrespondanceOdd} and proposition~\ref{P:OrbitSplittingSO} (c). Let $T_{\mathbf c}:V({\mathbf c}) \rightarrow V$ be the ${\mathcal I}$-graded isomorphism defined in \ref{SS:IsomorphismV(c)andVOddCase} such that $\langle T_{\mathbf c}(u), T_{\mathbf c}(v)\rangle = \langle u, v \rangle_{\mathbf c}$ for all $u, v \in V({\mathbf c})$.  If $E_{\mathbf c}:V({\mathbf c}) \rightarrow V({\mathbf c})$ is the linear transformation defined in \ref{N:CCorrespondanceOdd}, then  ${}'{\mathcal O}_{\mathbf c}$ is the $G^{\iota}$-orbit of  $X_{\mathbf c} = T_{\mathbf c} E_{\mathbf c} T^{-1}_{\mathbf c}: V \rightarrow V$.  

Let $b' \in {\mathbf B}$ be an ${\mathcal I}$-box   $b' = b(i', j', 0)$ such that $i', j' \in {\mathcal I}$ with $i' \geq 0 \geq j'$, $i' + j' > 0$ and $c(b') > 0$.   As noted in proposition~\ref{P:OrbitSplittingSO} (c), such an ${\mathcal I}$-box $b'$ exists. We get easily that $0 \in Supp(b') \cap Supp(\tau(b'))$. Because $i' + j' > 0$, then $b'$ is above the principal diagonal.    Denote by ${\mathcal O}' \in {\mathbf B}/\tau$: the $\langle \tau \rangle$-orbit of $b'$, i.e.  ${\mathcal O}' = \{b', \tau(b')\}$.  Consider $S_{{\mathcal O}'}:V({\mathbf c}) \rightarrow V({\mathbf c})$ defined in the proof of proposition~\ref{P:OrbitSplittingSO} (c) and denote by $E'_{\mathbf c}$: the linear transformation $E'_{\mathbf c}= S_{{\mathcal O}'} E_{\mathbf c}S^{-1}_{{\mathcal O}'}: V({\mathbf c}) \rightarrow V({\mathbf c})$, then  ${}''{\mathcal O}_{\mathbf c}$ is the $G^{\iota}$-orbit of  $X'_{\mathbf c} = T_{\mathbf c} E'_{\mathbf c} T^{-1}_{\mathbf c}: V \rightarrow V$. 

We will now recall the notation of \ref{N:CCorrespondance}.  $V({\mathbf c})$ is a vector space with its ${\mathcal I}$-basis 
\[
{\mathcal B}_{\mathbf c} = \coprod_{{\mathcal O} \in {\mathbf B}/\tau} \coprod_{b \in {\mathcal O}} \{v_i^j(b) \mid i \in Supp(b), 1 \leq j \leq {\mathbf c}({\mathcal O})\}
\]
and  the non-degenerate bilinear form $\langle \  ,\   \rangle_{\mathbf c}: V({\mathbf c}) \times V({\mathbf c}) \rightarrow {\mathbf k}$. 

\begin{lemma}\label{L:MaxIsotropicSubspaceOdd}
Using the notation of \ref{SS:IoddGGiotaConjugaisonTotallyEven},    we have the following.
\begin{enumerate}[\upshape (a)]
\item $0 \not \in Supp^{top}(b')$ and $0 \in Supp^{top}(\tau(b'))$. 
\item  The vector subspace 
\[
{}'L({\mathbf c}) = E_{\mathbf c} [\ker(E^2_{\mathbf c})] + E^2_{\mathbf c} [\ker(E^4_{\mathbf c})]  + \dots + E^k_{\mathbf c} [\ker(E^{2k}_{\mathbf c})]  + \dots
\]
is the vector subspace of $V({\mathbf c})$ spanned by the basis
\[
\coprod_{\substack { b \in {\mathbf B} \\ c(b) \ne 0}} \{v_i^j(b) \mid i \in Supp^{top}(b), 1 \leq j \leq c(b)\}
\]
and this is a maximal isotropic subspace of $V({\mathbf c})$. 
\item The vector subspace 
\[
{}''L({\mathbf c}) = E'_{\mathbf c} [\ker((E')^2_{\mathbf c})] + (E')^2_{\mathbf c} [\ker((E')^4_{\mathbf c})]  + \dots + (E')^k_{\mathbf c} [\ker((E')^{2k}_{\mathbf c})]  + \dots
\]
is the vector subspace of $V({\mathbf c})$ spanned by the basis
\[
\coprod_{\substack { b \in {\mathbf B} \\ c(b) \ne 0}} \{S_{{\mathcal O}'}(v_i^j(b)) \mid i \in Supp^{top}(b), 1 \leq j \leq c(b)\}
\]
and this is a maximal isotropic subspace of $V({\mathbf c})$. Moreover this last basis is equal to
\[
\left[\left[\coprod_{\substack { b \in {\mathbf B} \\ c(b) \ne 0}} \{v_i^j(b) \mid i \in Supp^{top}(b), 1 \leq j \leq c(b)\}\right] \setminus \{v_0^{c(b')}(\tau(b'))\}  \right] \cup \{v_0^{c(b')}(b')\}.
\]

\end{enumerate}
\end{lemma}
\begin{proof}

We can express the linear transformation $E_{\mathbf c}$ on the elements of the basis ${\mathcal B}_{\mathbf c}$. Let ${\mathcal O}$ denote the $\langle \tau \rangle$-orbit of $b$, then,  if the orbit ${\mathcal O}$ has no overlapping ${\mathcal I}$-box supports,  $1 \leq j \leq {\mathbf c}({\mathcal O})$ and $i \in Supp(b)$, we have
\[
E_{\mathbf c}(v_i^j(b)) = \begin{cases} {\phantom -} 0, &\text{if $i = \max(Supp(b))$;}\\ {\phantom -} v_{i + 2}^j(b), &\text{if $i \ne \max(Supp(b))$ and $i > 0$;}\\  -v_{i + 2}^j(b), &\text{if $i \ne \max(Supp(b))$ and $i < 0$;} \end{cases}
\]
while if the orbit ${\mathcal O}$ has overlapping ${\mathcal I}$-box supports,  $b = b(i_1, j_1, k_1) \in {\mathcal O}$, $1 \leq j \leq {\mathbf c}({\mathcal O})$  and $i \in Supp(b)$, then we have 
\[
E_{\mathbf c}(v_i^j(b)) = \begin{cases} {\phantom -} 0, &\text{if $i = \max(Supp(b))$;}\\ {\phantom -} v_{i + 2}^j(b), &\text{if $i \ne \max(Supp(b))$ and $i > 0$;}\\  -v_{i + 2}^j(b), &\text{if  $i < 0$;}\\   {\phantom -}  v_2^j(b), &\text{if $i = 0$, $\max(Supp(b)) > 0$ and $b$ is below the principal}\\ &\text{diagonal;}\\ {\phantom -}  v_2^j(b), &\text{if $i = 0$, $\max(Supp(b)) > 0$ and $b$ is above the principal}\\ &\text{diagonal;}  \end{cases}
\]
Note that if $b \in {\mathbf B}$ is such that $c(b) > 0$, then $b$ is not on the principal diagonal as noted in \ref{SS:IoddGGiotaConjugaisonTotallyEven}

(a) We have that $b' = b(i', j', 0)$ with $i', j' \in {\mathcal I}$, $i' \geq 0 \geq j'$ and $(i' + j') > 0$. Because the elements of  ${\mathcal I}$ are even integers, we have that $(i' + j') \geq 2$.  We have that $i'' \in Supp^{top}(b')$ if and only if 
\[
i' \geq i'' \geq \frac{(i' + j' + 2)}{2} \geq 2 \quad \text{ and } \quad 0 \not \in Supp^{top}(b').
\]
We have that $\tau(b') = b(-j', -i', 0)$. We have that $i'' \in Supp(\tau(b'))$ if and only if 
\[
-j' \geq i'' \geq \frac{(-j' - i' + 2)}{2}. 
\]
But $0 \geq j'$ implies that $-j' \geq 0$. $(i' + j') \geq 2$ implies that  $0 \geq (-j' - i' + 2)$ and $0 \in Supp^{top}(\tau(b'))$.

(b) The proof is similar to the proof in lemma~\ref{L:MaximalIsotropicV(c)} and is left to the reader.

(c) We have that $(E'_{\mathbf c})^{r} = S_{{\mathcal O}'} E_{\mathbf c}^{r} S_{{\mathcal O}'}^{-1}$ for all integers $r\in {\mathbb N}$.  From this,  we get easily that  
\[
\ker((E'_{\mathbf c})^{2k}) = \ker(S_{{\mathcal O}'}E_{\mathbf c}^{2k} S_{{\mathcal O}'}^{-1}) =  S_{{\mathcal O}'}(\ker(E_{\mathbf c}^{2k}))
\]
 and 
 \[
 (E'_{\mathbf c})^k [\ker((E'_{\mathbf c})^{2k})] = S_{{\mathcal O}'}E_{\mathbf c}^{k} S_{{\mathcal O}'}^{-1}[S_{{\mathcal O}'}(\ker(E_{\mathbf c}^{2k}))] = S_{{\mathcal O}'} (E_{\mathbf c}^k[\ker(E_{\mathbf c}^{2k})])
 \]
 Thus ${}''L({\mathbf c}) = S_{{\mathcal O}'}({}'L({\mathbf c}))$.  Because $S_{{\mathcal O}'}:V({\mathbf c}) \rightarrow V({\mathbf c})$ is an ${\mathcal I}$-graded isomorphism such that 
 $\langle S_{{\mathcal O}'}(u),  S_{{\mathcal O}'}(v)\rangle_{\mathbf c} = \langle u,  v\rangle_{\mathbf c}$ for all $u, v \in V({\mathbf c})$, then ${}''L({\mathbf c}) $ is a maximal isotropic subspace of $V({\mathbf c})$ and 
 \[
\coprod_{\substack { b \in {\mathbf B} \\ c(b) \ne 0}} \{S_{{\mathcal O}'}(v_i^j(b)) \mid i \in Supp^{top}(b), 1 \leq j \leq c(b)\}
\]
is a basis of  ${}''L({\mathbf c})$.  

For the latter description of this basis, we need to note that $S_{{\mathcal O}'}$ is the identity on all the vectors except for $v_0^{c(b')}(b')$ and $v_0^{c(b')}(\tau(b'))$. By (a), $v_0^{c(b')}(b') \not\in {}'L({\mathbf c})$ and $v_0^{c(b')}(\tau(b')) \in {}'L({\mathbf c})$. Thus we get that $S_{{\mathcal O}'}(v_0^{c(b')}(b')) = v_0^{c(b')}(\tau(b')) \not\in {}''L({\mathbf c})$ and  $S_{{\mathcal O}'}(v_0^{c(b')}(\tau(b')) = v_0^{c(b')}(b') \in {}''L({\mathbf c})$. This concludes the proof of (c). 
\end{proof}

The following lemma is probably well-known. Its proof is included for completeness. 

\begin{lemma}\label{L:ConditionSO(V)orbitOdd}
Let ${\mathcal I}$ be odd and using the notations of \ref{S:SetUpOddOrtho} and \ref{SS:IoddGGiotaConjugaisonTotallyEven}, consider the set ${\mathcal J}$ of subsets $J$ of $\{(i, j) \in {\mathcal I} \times {\mathbb N} \mid 1 \leq j \leq \delta_i\}$ such that 
\begin{itemize}
\item if $(i, j),  (i', j) \in J$ are such that $i, i' \ne 0$, then $i + i' \ne 0$ and
\item if $(0, j), (0, j') \in J$, then $j + j' \ne (\delta_0 + 1)$,
\end{itemize}
and moreover $J$ is maximal among such subsets.
\begin{enumerate}[\upshape (a)]
\item If $J \in {\mathcal J}$, then $2 \vert J \vert = \dim(V)$.
\item If $J \in {\mathcal J}$, then the subspace $L_J$ of $V$ spanned by $\{u_{i, j} \mid  (i, j) \in J\}$ is a maximal isotropic subspace of $V$.
\item Let  $J, J' \in {\mathcal J}$. Then the two maximal isotropic subspaces $L_J$ and $L_{J'}$ of $V$ are in the same $SO(V)$-orbit if and only if 
\begin{multline*}
 \vert \{(i, j) \in J \mid i < 0\} \cup \{(0, j) \in J \mid \delta'''_0 < j\} \vert  \equiv \\ \vert \{(i, j) \in J' \mid i < 0\}\cup \{(0, j) \in J' \mid \delta'''_0 < j\} \vert  \pmod 2.
\end{multline*}
Recall that $2\delta'''_0 = \delta_0$.
\end{enumerate}
\end{lemma}
\begin{proof}
(a) From our notation in \ref{S:SetUpOddOrtho}, we have that 
\[
\dim(V) = \sum_{i \in {\mathcal I}} \delta_i =  \sum_{i \in {\mathcal I}}\vert \{ (i, j) \in {\mathcal I} \times {\mathbb N} \mid 1 \leq j \leq \delta_i\}\vert
\]
If $(i , j) \in {\mathcal I} \times {\mathbb N}$ is such that $1 \leq j \leq \delta_i$, then,  by the maximality of $J$, we have if $i \ne 0$, that exactly one and only one of $(i, j)$, $(-i, j)$ belongs to $J$, while if $i = 0$,  exactly one and only one of $(0, j)$, $(0, (\delta_0 + 1 - j))$ belongs to $J$. Consequently  $2 \vert J \vert = \dim(V)$. 

(b) First we will show that $L_J$ is isotropic. From our hypothesis on the basis as defined in \ref{S:SetUpOddOrtho}, we have that $\langle u_{i, j}, u_{i', j'}\rangle \ne 0$ if and  only if either $i, i' \ne 0$, $j = j'$ and $i + i' = 0$ or $i = i' = 0$, $j + j' = (\delta_0 + 1)$ when $(i, j), (i', j') \in {\mathcal I} \times {\mathbb N}$ are such that $1 \leq j \leq \delta_i$ and $1 \leq j' \leq \delta_{i'}$. But by our hypothesis on $J$, when $(i, j), (i', j') \in J$, then  $\langle u_{i, j}, u_{i', j'} \rangle = 0$. Otherwise we would have either $i, i' \ne 0$, $j = j'$ and $i + i' = 0$ or $i = i' = 0$ and $j + j' = \delta_0$. But this is impossible for two elements of $J$.  So $L_J$ is isotropic. Because $2\dim(L_J) = \dim(V)$ by (a), we get that $L_J$ is maximal. 

(c) We first need to remind some facts. Let 
\[
J_+ = \{(i, j) \in {\mathcal I} \times {\mathbb N} \mid i > 0, 1 \leq j \leq \delta_i\} \cup \{(0, j) \in {\mathcal I} \times {\mathbb N} \mid  1 \leq j \leq \delta'''_0\}.
\]
Clearly $J_+ \in {\mathcal J}$.  Note that if $g \in O(V)$ is such that $g(L_{J_+}) = L_{J_+}$, then $g \in SO(V)$.  In fact by considering the matrix of $g$ in the basis $\{u_{i, j} \mid i \in {\mathcal I},  1 \leq j \leq \delta_i\}$ ordered properly, i.e. the basis elements corresponding to the indexes in $J_+$ and then the one  corresponding to the indexes in the complement of $J_+$ in   $\{(i, j) \in {\mathcal I} \times {\mathbb N} \mid 1 \leq j \leq \delta_i\} $ in the right  order, we get that this matrix is of the form 
\[
\begin{bmatrix} A & B\\ 0 & D\end{bmatrix} \quad \text{ such that } \quad  \begin{bmatrix} A & B\\ 0 & D\end{bmatrix}^T \begin{bmatrix} 0 & I\\ I & 0\end{bmatrix} \begin{bmatrix} A & B\\ 0 & D\end{bmatrix} = \begin{bmatrix} 0 & I\\ I & 0\end{bmatrix}.
\]
Consequently $A^TD = I$ ,  $\det(D) = (\det(A))^{-1}$ and $g \in SO(V)$. 

Let $g_J: V \rightarrow V$ denote the unique linear transformation such that
\[
g_J(u_{i, j}) = \begin{cases} u_{i, j}, &\text{if $i > 0$ and $(i, j) \in J$:}\\  u_{-i, j}, &\text{if $i >0$ and $(i, j) \not \in J$;}\\  u_{i, j}, &\text{if $i = 0$, $1 \leq j \leq \delta'''_0$ and $(i, j) \in J$;}\\  u_{i, (\delta_0 + 1 - j)}, &\text{if $i = 0$, $1 \leq j \leq \delta'''_0$ and $(i, j) \not \in J$;}\\  u_{-i, j}, &\text{if $i < 0$ and $(i, j) \in J$;} \\ u_{i, j}, &\text{if $i < 0$ and $(i, j) \not\in J$;}\\ u_{i, (\delta_0 + 1 - j)}, &\text{if $i = 0$, $(\delta'''_0 + 1) \leq j \leq \delta_0$ and $(i, j) \in J$;}\\  u_{i, j}, &\text{if $i = 0$, $(\delta'''_0 + 1) \leq j \leq \delta_0$ and $(i, j) \not\in J$.}
\end{cases}
\]
We can now show that $g_J \in O(V)$.  For this, we have to consider the following two situations:
\begin{itemize}
\item $\langle u_{i_1, j_1}, u_{i_2, j_2}\rangle$ and $\langle g_J(u_{i_1, j_1}), g_J(u_{i_2, j_2}) \rangle$ whenever $i_1, i_2 \in {\mathcal I}$, with $i_1, i_2 \ne 0$,  $1 \leq j_1 \leq \delta_{i_1}$ and $1 \leq j_2 \leq \delta_{i_2}$;
\item $\langle u_{0, j_1}, u_{0, j_2}\rangle$ and $\langle g_J(u_{0, j_1}), g_J(u_{0, j_2}) \rangle$ whenever  $1 \leq j_1 \leq \delta_0$ and $1 \leq j_2 \leq \delta_0$.
\end{itemize}

We will now consider the first case:  $i_1, i_2 \in {\mathcal I}$ with $i_1, i_2 \ne 0$,  $1 \leq j_1 \leq \delta_{i_1}$ and $1 \leq j_2 \leq \delta_{i_2}$,   we must consider $\langle u_{i_1, j_1}, u_{i_2,j_2}\rangle$ and $\langle g_J(u_{i_1, j_1}), g_J(u_{i_2,j_2)}\rangle$.  We have 
\[
\langle u_{i_1, j_1}, u_{i_2, j_2}\rangle = \begin{cases} 1, &\text{ if $j_1 = j_2$ and $i_1 + i_2 = 0$;}\\ 0, &\text{otherwise.} \end{cases}
\]

If $j_1 \ne j_2$, then 
\[
\langle g_J(u_{i_1, j_1}), g_J(u_{i_2, j_2}) \rangle = \langle u_{i'_1, j_1}, u_{i'_2, j_2}\rangle = 0, 
\]
where $i'_1 \in \{i_1, -i_1\}$ and $i'_2  \in \{i_2, -i_2\}$, because $j_1 \ne j_2$.  

If $j_1 = j_2$ and $i_1 + i_2 = 0$, then we can assume by symmetry  that $i_1 > 0$ and $i_2 = -i_1 < 0$. Moreover $(i_1, j_1) \in J$ if and only if $(i_2, j_2) \not\in J$. Thus 
\[
\langle g_J(u_{i_1, j_1}), g_J(u_{i_2, j_2})\rangle = \begin{cases} \langle u_{i_1, j_1}, u_{i_2, j_2} \rangle, &\text{if $(i_1, j_1) \in J$;}\\  \langle u_{-i_1, j_1}, u_{-i_2, j_2} \rangle, &\text{if $(i_1, j_1) \not\in J$;} \end{cases} \quad = 1.
\]

If $j_1 = j_2$, $i_1 + i_2 \ne 0$ and  $i _1 - i_2 \ne 0$, then 
\[
\langle g_J(u_{i_1, j_1}), g_J(u_{i_2, j_2})\rangle = \langle u_{i'_1, j_1}, u_{i'_2, j_2}\rangle = 0
\]
because $i'_1 \in \{i_1, -i_1\}$, $i'_2 \in \{i_2, -i_2\}$ and $i'_1 + i'_2 \ne 0$ from our hypothesis.

Finally if $j_1 = j_2$, $i_1 + i_2 \ne 0$ and $i_1 = i_2$, then 
\[
\langle g_J(u_{i_1, j_1}), g_J(u_{i_2, j_2})\rangle  =  \langle g_J(u_{i_1, j_1}), g_J(u_{i_1, j_1})\rangle = \langle u_{i'_1, j_1}, u_{i'_1, j_1}\rangle = 0
\]
where  $i'_1 \in \{i_1, -i_1\}$. 

We will now consider the second case. If  $1 \leq j_1, j_2 \leq \delta_0$, then we must consider $\langle u_{0, j_1}, u_{0, j_2}\rangle$ and $\langle g_J(u_{0, j_1}), g_J(u_{0, j_2}) \rangle$. We have 
\[
\langle u_{0, j_1}, u_{0, j_2}\rangle = \begin{cases} 1, &\text{ if $j_1 + j_2 = (\delta_0 + 1)$;}\\ 0, &\text{otherwise;} \end{cases}
\]
If $\langle u_{0, j_1}, u_{0, j_2}\rangle = 1$,  we cannot have both $j_1 \geq (\delta'''_0 + 1)$ and $j_2 \geq  (\delta'''_0 + 1)$, because otherwise $(j_1 + j_2) \geq (\delta_0 + 2) > (\delta_0 + 1)$. Similarly we cannot have both $1 \leq j_1 \leq \delta'''_0$ and  $1 \leq j_2 \leq \delta'''_0$, because otherwise $(j_1 + j_2) \leq  2\delta'''_0 < (\delta_0 + 1)$.  Consequently we can assume, using symmetry,  that $1 \leq j_1 \leq \delta'''_0$ and $(\delta'''_0 + 1) \leq j_2 \leq \delta_0$ when  $\langle u_{0, j_1}, u_{0, j_2}\rangle = 1$. 

So when $\langle u_{0, j_1}, u_{0, j_2}\rangle = 1$, i.e $(j_1 +j_2) = (\delta_0 + 1)$, we have that $(0, j_1) \in J$ if and only if $(0, j_2) = (0, (\delta_0 + 1 - j_1)) \not \in J$
\[
g_J(u_{0, j_1}) = \begin{cases} u_{0, j_1}, &\text{if $(0, j_1) \in J$;}\\ u_{0, (\delta_0 + 1 - j_1)}, &\text{if $(0, j_1) \not\in J$;}\end{cases} 
\]
and
\[
g_J(u_{0, j_2}) = \begin{cases} u_{0, (\delta_0 + 1 - j_2)}, &\text{if $(0, j_2) \in J$;}\\ u_{0, j_2}, &\text{if $(0, j_2) \not\in J$.}\end{cases}
\]
Thus
\[
\langle g_J(u_{0, j_1}),  g_J(u_{0, j_2}) \rangle = \begin{cases} \langle u_{0, j_1}, u_{0, j_2} \rangle &\text{if $(0, j_1) \in J$;}\\   \langle u_{0, \delta_0 + 1 -  j_1}, u_{0,\delta_0 + 1 -  j_2} \rangle &\text{if $(0, j_1) \not\in J$;} \end{cases}\quad  = 1
\]
because $j_1 + j_2 = \delta_0 + 1$. 

If $\langle g_J(u_{0, j_1}),  g_J(u_{0, j_2}) \rangle  = \langle u_{0, j'_1}, u_{0, j'_2} \rangle = 1$, where $j'_1 \in \{j_1, (\delta_0 + 1 - j_1)\}$ and  $j'_2 \in \{j_2, (\delta_0 + 1 - j_2)\}$, then $j'_1 + j'_2 = (\delta_0 + 1)$. We thus have four cases to consider: either $j'_1 = j_1$, $j'_2 = j_2$ with $j_1 + j_2 = (\delta_0 + 1)$ or $j'_1 = (\delta_0 + 1 - j_1)$, $j'_2 = (\delta_0 + 1 - j_2)$ with $j_1 + j_2 = (\delta_0 + 1)$ or  $j'_1 =  j_1$, $j'_2 = (\delta_0 + 1 - j_2)$ with $j_1 =  j_2$ or $j'_1 = (\delta_0 + 1 - j_1)$, $j'_2 = j_2$ with $j_1 = j_2$.  

If $j_1 = j_2$, then we get that 
\[
\langle g_J(u_{0, j_1}), g_J(u_{0, j_2})\rangle = \begin{cases} \langle u_{0, j_1}, u_{0, j_1}\rangle, &\text{if $1 \leq j_1 \leq \delta'''_0$ and $(0, j_1) \in J$;} \\ \langle u_{0, (\delta_0 + 1 -  j_1)}, u_{0, (\delta_0 + 1 - j_1)}\rangle, &\text{if $1 \leq j_1 \leq \delta'''_0$ and $(0, j_1) \not \in J$;} \\
 \langle u_{0, (\delta_0 + 1 - j_1)}, u_{0, (\delta_0 + 1 - j_1)}\rangle, &\text{if $(\delta'''_0 + 1) \leq j_1 \leq \delta_0$}\\ &\text{and $(0, j_1) \in J$;} \\
 \langle u_{0,  j_1}, u_{0, j_1}\rangle, &\text{if $(\delta'''_0 + 1) \leq j_1 \leq \delta_0$}\\ &\text{and $(0, j_1) \not\in J$.} 
\end{cases}
\]
So if $j_1 = j_2$, we get that $\langle g_J(u_{0, j_1}), g_J(u_{0, j_2})\rangle = 0$ contrary to our hypothesis that $\langle g_J(u_{0, j_1}),  g_J(u_{0, j_2}) \rangle  = 1$.  

We are left with the two cases:  either $j'_1 = j_1$, $j'_2 = j_2$ with $j_1 + j_2 = (\delta_0 + 1)$ or $j'_1 = (\delta_0 + 1 - j_1)$, $j'_2 = (\delta_0 + 1 - j_2)$ with $j_1 + j_2 = (\delta_0 + 1)$. Thus we have that $\langle g_J(u_{0, j_1}), g_J(u_{0, j_2})\rangle = \langle u_{0, j_1}, u_{0, j_2}\rangle = 1$.

This completes the proof that $g_J \in O(V)$.   Note also that we have $g_J(L_{J_+}) =L_J$, because we only have to consider either the basis elements $u_{i, j}$ for $i > 0$ or the basis elements $u_{0, j}$ for $1 \leq j \leq \delta'''_0$. In the first  case,  if $(i, j) \not\in J$,  we get that $(-i, j) \in J$ and 
\[
g_J(u_{i, j}) =  \begin{cases} u_{i, j} \in L_J, &\text{if $i > 0$ and $(i, j) \in J$:}\\  u_{-i, j} \in L_J, &\text{if $i >0$ and $(i, j) \not \in J$;}
\end{cases}
\]
and, in the second case,  if  $(0, j) \not \in J$,  we get that $(0, (\delta_0 + 1 - j)) \in J$ and
\[
g_J(u_{0, j}) =  \begin{cases} u_{0, j} \in L_J, &\text{if $1 \leq j \leq \delta'''_0$ and $(0, j) \in J$;}\\  u_{0, (\delta_0 + 1 - j)} \in L_J, &\text{if $1 \leq j \leq \delta'''_0$ and $(i, j) \not \in J$.}
\end{cases}
\]

We can now compute the determinant of the matrix of $g_J$ relative to the basis $\{u_{i, j} \mid i \in {\mathcal I}, 1 \leq j \leq \delta_i\}$ where we consider first the basis elements $u_{i, j}$ for $i \in {\mathcal I}$, $i > 0$, $1 \leq j \leq \delta_i$,  secondly the basis elements $u_{0, j}$ for $1 \leq j \leq \delta'''_0$   and then thirdly the basis element $u_{-i, j}$   for $i \in {\mathcal I}$, $i > 0$, $1 \leq j \leq \delta_i$ and finally the basis elements $u_{0, (\delta_0 + 1 - j)}$ for $1 \leq j \leq \delta'''_0$.  Using the Laplace expansion of $\det(g_J)$ and using recursion, we get easily that 
\[
\det(g_J) = (-1)^{\vert \{(i, j) \in J \mid i < 0\}\vert + \vert\{(0, j) \in J \mid \delta'''_0 < j\}\vert}
\]
If $g \in O(V)$ and $g(L_{J_+}) = L_J$, then $\det(g) = (-1)^{\vert \{(i, j) \in J \mid i < 0\}\vert + \vert\{(0, j) \in J \mid \delta'''_0 < j\}\vert}$, because 
\[
g(L_{J_+}) = L_J = g_J(L_{J_+})  \quad \Rightarrow \quad g^{-1}_Jg(L_{J_+}) =L_{J_+} \quad \text{ and } \quad g^{-1}_Jg \in SO(V)
\]
and $\det(g) = \det(g_J)$. 

We  complete the proof of (c).  If $L_J$ and $L_{J'}$ are in the same $SO(V)$-orbit, then they are both in the $SO(V)$-orbit of $L_{J_+}$ or both are not in the $SO(V)$-orbit of $L_{J_+}$. Thus 
\[
(-1)^{\vert \{(i, j) \in J \mid i < 0\}\vert + \vert\{(0, j) \in J \mid \delta'''_0 < j\}\vert} = (-1)^{\vert \{(i, j) \in J' \mid i < 0\}\vert + \vert\{(0, j) \in J' \mid \delta'''_0 < j\}\vert},  
\]
 \begin{multline*}
 \vert \{(i, j) \in J \mid i < 0\}\vert +  \vert\{(0, j) \in J \mid \delta'''_0 < j\}\vert  \equiv\\  \vert \{(i, j) \in J' \mid i < 0\}\vert +  \vert\{(0, j) \in J' \mid \delta'''_0 < j\}\vert \pmod 2.
 \end{multline*}  
 
 Conversely if 
 \begin{multline*}
 \vert \{(i, j) \in J \mid i < 0\}\vert +  \vert\{(0, j) \in J \mid \delta'''_0 < j\}\vert  \equiv\\  \vert \{(i, j) \in J' \mid i < 0\}\vert +  \vert\{(0, j) \in J' \mid \delta'''_0 < j\}\vert \pmod 2,
 \end{multline*}  
then we get that $L_J$ and $L_{J'}$ are both in the $SO(V)$-orbit of $L_{J_+}$ or both are not in the $SO(V)$-orbit of $L_{J_+}$.  Because there are exactly two $SO(V)$-orbits on the set of maximal isotropic subspaces of $V$, we get that $L_J$ and $L_{J'}$ are in the same $SO(V)$-orbit. 

\end{proof}

\begin{lemma}\label{L:SO(V)OrbitSplittingEvenConjugacyClass}
Using the notation of \ref{SS:IoddGGiotaConjugaisonTotallyEven},  of lemma~\ref{L:MaxIsotropicSubspaceOdd} and consider the ${\mathcal I}$-graded isomorphism $T_{\mathbf c}: V({\mathbf c}) \rightarrow V$ defined in \ref{SS:IsomorphismV(c)andVOddCase} such that $\langle T_{\mathbf c}(u), T_{\mathbf c}(v)\rangle = \langle u, v \rangle_{\mathbf c}$ for all $u, v \in V({\mathbf c})$.  Then both subspaces $T_{\mathbf c}({}'L({\mathbf c}))$ and $T_{\mathbf c}({}''L({\mathbf c}))$ of $V$ are maximal isotropic subspaces of $V$ and these subspaces are not in the same $SO(V)$-orbit. Consequently the $G^{\iota}$-orbits ${}'{\mathcal O}_{\mathbf c}$  and ${}''{\mathcal O}_{\mathbf c}$ are not in the same $G$-orbit. Here $G = SO(V)$. 
\end{lemma}
\begin{proof}
As we saw in lemma~\ref{L:MaxIsotropicSubspaceOdd}, both ${}'L({\mathbf c})$ and  ${}''L({\mathbf c})$ are maximal isotropic subspaces of $V({\mathbf c})$, we get easily that $T_{\mathbf c}({}'L({\mathbf c}))$ and $T_{\mathbf c}({}''L({\mathbf c}))$ of $V$ are maximal isotropic subspaces of $V$. 

Let $X_{\mathbf c} = T_{\mathbf c} E_{\mathbf c} T_{\mathbf c}^{-1}: V \rightarrow V$ and $X'_{\mathbf c} = T_{\mathbf c} S_{{\mathcal O}'} E_{\mathbf c} S^{-1}_{{\mathcal O}'}T_{\mathbf c}^{-1}: V \rightarrow V$ as in the proof of proposition~\ref{P:OrbitSplittingSO}  (c).  As we have shown ${}'{\mathcal O}_{\mathbf c}$ is the $G^{\iota}$-orbit of $X_{\mathbf c}$ and ${}''{\mathcal O}_{\mathbf c}$ is the $G^{\iota}$-orbit of $X'_{\mathbf c}$.  We get easily that $T_{\mathbf c}({}'L({\mathbf c}))$ is equal to the subspace 
\[
X_{\mathbf c}[\ker(X^2_{\mathbf c})] + X^2_{\mathbf c}[\ker(X^4_{\mathbf c})] + \dots + X^k_{\mathbf c}[\ker(X^{2k}_{\mathbf c})] + \dots
\]
and it is spanned by  
\[
\coprod_{\substack { b \in {\mathbf B} \\ c(b) \ne 0}} \{ T_{\mathbf c}(v_i^s(b)) \mid   i \in Supp^{top}(b) \text{ and  } 1 \leq s \leq c(b)\}.
\]

From our definition of $T_{\mathbf c}$, we get that this set is 
\[
\coprod_{\substack { b \in {\mathbf B} \\ c(b) \ne 0}} \{u_{i, j} \mid i \in Supp^{top}(b) \text{ and } j \in J_i(b) \}
\]
where $J_i(b)$ is equal to
\[
 \begin{cases} \{1 \leq j \leq \delta_i \mid u_{i, j} = T_{\mathbf c}(v_i^s(b)) \text{ for some $1 \leq s \leq c(b)$} \}, &\text{if $i \ne 0$;}\\ \\ \{1 \leq j \leq \delta'''_0 \mid  u_{0, j} = T_{\mathbf c}(v_0^s(b)) \text{ for some $1 \leq s \leq c(b)$} \}, &\text{if $i = 0$ and $b \in {\mathbf B}^+$;}\\ \\ \{(\delta'''_0 + 1 )\leq j \leq \delta_0 \mid  u_{0, j} = T_{\mathbf c}(v_0^s(b)) \text{ for some $1 \leq s \leq c(b)$} \}, &\text{if $i = 0$ and $b \in {\mathbf B}^-$.} \end{cases}
\]
For all $i \in {\mathcal I}$ and $b, b' \in {\mathbf B}$,   note that $\vert J_i(b) \vert = c(b)$ and if $b \ne b'$  are two distinct ${\mathcal I}$-boxes such that $i \in Supp(b) \cap Supp(b')$, then $J_i(b) \cap J_i(b') = \emptyset $. Because ${}'L({\mathbf c})$ is a maximal isotropic subspace of $V({\mathbf c})$, we get that 
\[
J({\mathbf c}) = \coprod_{\substack { b \in {\mathbf B} \\ c(b) \ne 0}} \{(i, j) \in {\mathcal I} \times {\mathbb N} \mid   i \in Supp^{top}(b) \text{ and }  j \in J_i(b)\} \in {\mathcal J}
\]
and $T_{\mathbf c}({}'L({\mathbf c})) = L_{J({\mathbf c})}$ with the notation of lemma~\ref{L:ConditionSO(V)orbitOdd}.

Similarly $T_{\mathbf c}({}''L({\mathbf c}))$ is equal to the subspace 
\[
X'_{\mathbf c}[\ker((X')^2_{\mathbf c})] + (X')^2_{\mathbf c}[\ker((X')^4_{\mathbf c})] + \dots + (X')^k_{\mathbf c}[\ker((X')^{2k}_{\mathbf c})] + \dots
\]
and it is spanned by  
\[
\coprod_{\substack { b \in {\mathbf B} \\ c(b) \ne 0}} \{ T_{\mathbf c}(S_{{\mathcal O}'}(v_i^s(b))) \mid   i \in Supp^{top}(b) \text{ and  } 1 \leq s \leq c(b)\}.
\]

Because of  lemma~\ref{L:MaxIsotropicSubspaceOdd} (c), we can be more precise about this last set.  It is equal to 
\begin{multline*}
\left[\left[\coprod_{\substack { b \in {\mathbf B} \\ c(b) \ne 0}} \{ T_{\mathbf c}(v_i^s(b)) \mid  i \in Supp^{top}(b) \text{ and  } 1 \leq s \leq c(b)\}\right]  \setminus  \{T_{\mathbf c}(v_0^{c(b')}(\tau(b'))) \} \right]\\ \cup \{T_{\mathbf c}(v_0^{c(b')}(b'))\}
\end{multline*}

There is a unique ${\widehat {j}} \in [1, \delta'''_0]$ such that  $T_{\mathbf c}(v_0^{c(b')}(b')) = u_{0, {\widehat j}}$, because of our definition of $T_{\mathbf c}$ and the fact that $b' \in {\mathbf B}^+$.  Consequently $T_{\mathbf c}(v_0^{c(b')}(\tau(b'))) = u_{0, (\delta_0 + 1 - {\widehat  j}\, )}$ again from our definition of $T_{\mathbf c}$.  

Because of lemma~\ref{L:MaxIsotropicSubspaceOdd} (a), we have that $(0, (\delta_0 + 1 - {\widehat j}\,  ) ) \in J({\mathbf c})$,  $(0,  {\widehat j}\,  ) \not \in J({\mathbf c})$ and  
\[
{}'J({\mathbf c}) = \left[\left[J({\mathbf c}) \setminus \{(0, (\delta_0 + 1 - {\widehat j}\,   ) )\}\right] \cup \{(0, {\widehat j}\,  )\}\right]\in {\mathcal J}.
\]
We also have that  $T_{\mathbf c}({}''L({\mathbf c})) = L_{{}'J({\mathbf c})}$ with the notation of lemma~\ref{L:ConditionSO(V)orbitOdd}. 

Clearly we have that 
\begin{multline*}
\vert \{(i, j) \in J({\mathbf c}) \mid i < 0\} \cup \{(0, j) \in J({\mathbf c}) \mid \delta'''_0 < j\} \vert \\  \not \equiv \vert \{(i, j) \in {}'J({\mathbf c}) \mid i < 0\} \cup \{(0, j) \in {}'J({\mathbf c}) \mid \delta'''_0 < j\} \vert  \pmod 2
\end{multline*}
and  we get that the maximal isotropic subspaces  $T_{\mathbf c}({}'L({\mathbf c}))$ and $T_{\mathbf c}({}''L({\mathbf c}))$ are not in the same $SO(V)$-orbit because of lemma~\ref{L:ConditionSO(V)orbitOdd} (c).  One consequence of this is that the $G^{\iota}$-orbits  ${}'{\mathcal O}_{\mathbf c}$  and ${}''{\mathcal O}_{\mathbf c}$ are not in the same $G$-orbit. 

\end{proof}

\begin{notation}
Let $\vert {\mathcal I} \vert$ be odd and  $G$ be the reductive group $SO(V)$ as in \ref{S:SetUpOddOrtho}. Consider the $G^{\iota}$-orbits   ${}'{\mathcal O}_{\mathbf c}$ and ${}''{\mathcal O}_{\mathbf c}$ as defined in proposition~\ref{P:OrbitSplittingSO} (c) and its proof,  whose partition corresponding to their Jordan type is the totally even partition $2^{a(2)} 4^{a(4)} 6^{a(6)} \dots$, where  ${\mathbf c}:{\mathbf B}/\tau \rightarrow {\mathbb N}$ is a coefficient function ${\mathbf c} \in {\mathfrak C}_{\delta}$. Denote by $c:{\mathbf B} \rightarrow {\mathbb N}$: the symmetric function associated respectively to ${\mathbf c}$.  We will denote 
\[
 {}'\varXi({}'{\mathcal O}_{\mathbf c}) =  \sum_{b \in {\mathbf B}} c(b) \vert \{i \in Supp^{top}(b) \mid i < 0\}\vert  + \sum_{\substack{ b \in {\mathbf B}^-\\ 0 \in Supp^{top}(b)}} c(b)
\]
and is called the {\it parity of the orbit} ${}'{\mathcal O}_{\mathbf c}$.  Recall that ${\mathbf B}^-$ is the set of ${\mathcal I}$-boxes strictly below the principal diagonal.  We will say that the $G^{\iota}$-orbit ${}'{\mathcal O}_{\mathbf c}$ is even (respectively odd) if ${}'\varXi({}'{\mathcal O}_{\mathbf c})$ is even (respectively odd).

We will denote  ${}''\varXi({}''{\mathcal O}_{\mathbf c}) = {}'\varXi({}'{\mathcal O}_{\mathbf c})  - 1$ and it is called the {\it parity of the orbit} ${}''{\mathcal O}_{\mathbf c}$.   Similarly we will say that the $G^{\iota}$-orbit ${}''{\mathcal O}_{\mathbf c}$ is even (respectively odd) if ${}''\varXi({}''{\mathcal O}_{\mathbf c})$ is even (respectively odd).

By lemma~\ref{L:SO(V)OrbitSplittingEvenConjugacyClass}, we get that the $G^{\iota}$-orbit ${}'{\mathcal O}_{\mathbf c}$ is even (respectively odd) if and only if ${}''{\mathcal O}_{\mathbf c}$ is odd (respectively even).  We will denote by ${\mathcal O}^{even}_{\mathbf c}$ (respectively ${\mathcal O}^{odd}_{\mathbf c}$): the only $G^{\iota}$-orbit between ${}'{\mathcal O}_{\mathbf c}$ and ${}''{\mathcal O}_{\mathbf c}$ that is even (respectively odd). 
\end{notation}

\begin{proposition}
Let $G$ be the reductive group $SO(V)$  as in \ref{S:SetUpOddOrtho} where $\vert {\mathcal I} \vert$ is odd and let ${\mathcal O}$ and ${\mathcal O}'$  be two $G^{\iota}$-orbits such that ${\mathcal O} \in \{ {\mathcal O}^{even}_{\mathbf c}, {\mathcal O}^{odd}_{\mathbf c} \}$ and   ${\mathcal O}' \in \{ {\mathcal O}^{even}_{{\mathbf c}'}, {\mathcal O}^{odd}_{{\mathbf c}'} \}$,  who both have the same totally even partition $2^{a(2)} 4^{a(4)} 6^{a(6)} \dots$ corresponding  to their Jordan type, where ${\mathbf c}:{\mathbf B}/\tau \rightarrow {\mathbb N}$ and ${\mathbf c}':{\mathbf B}/\tau \rightarrow {\mathbb N}$ are both coefficient functions ${\mathbf c}, {\mathbf c}' \in {\mathfrak C}_{\delta}$. Denote by $c:{\mathbf B} \rightarrow {\mathbb N}$ and $c':{\mathbf B} \rightarrow {\mathbb N}$: the symmetric functions associated respectively to ${\mathbf c}$ and ${\mathbf c}'$.  Then these two $G^{\iota}$-orbits  ${\mathcal O}$ and  ${\mathcal O}' $ are in the same $G$-orbit if and only if they have the same parity.  In other words,  ${\mathcal O}$ and  ${\mathcal O}' $ are in the same $G$-orbit if and only if  either (${\mathcal O} = {\mathcal O}^{even}_{\mathbf c}$ and ${\mathcal O}' = {\mathcal O}^{even}_{{\mathbf c}'}$) or (${\mathcal O} = {\mathcal O}^{odd}_{\mathbf c}$ and ${\mathcal O}' = {\mathcal O}^{odd}_{{\mathbf c}'}$).
\end{proposition}
\begin{proof}
The proof is very similar to the one for proposition~\ref{P:TotallyEvenOrbitSOIEvenCase}. We will use the same notation as in the lemma~\ref{L:SO(V)OrbitSplittingEvenConjugacyClass} and its proof.  Let $X_{\mathbf c} = T_{\mathbf c} E_{\mathbf c} T_{\mathbf c}^{-1}: V \rightarrow V$ and $X'_{\mathbf c} = T_{\mathbf c} S_{{\mathcal O}'} E_{\mathbf c} S^{-1}_{{\mathcal O}'}T_{\mathbf c}^{-1}: V \rightarrow V$ as in the proof of proposition~\ref{P:OrbitSplittingSO}  (c).  As we have shown ${}'{\mathcal O}_{\mathbf c}$ is the $G^{\iota}$-orbit of $X_{\mathbf c}$ and ${}''{\mathcal O}_{\mathbf c}$ is the $G^{\iota}$-orbit of $X'_{\mathbf c}$. 

As we saw in the proof of lemma~\ref{L:SO(V)OrbitSplittingEvenConjugacyClass}, 
\[
X_{\mathbf c}[\ker(X^2_{\mathbf c})] + X^2_{\mathbf c}[\ker(X^4_{\mathbf c})] + \dots + X^k_{\mathbf c}[\ker(X^{2k}_{\mathbf c})] + \dots
\]
is a maximal isotropic subspace of $V$ and it is equal to $L_{J({\mathbf c})}$ with the notation of lemma~\ref{L:ConditionSO(V)orbitOdd}, where 
\[
J({\mathbf c}) = \coprod_{\substack { b \in {\mathbf B} \\ c(b) \ne 0}} \{(i, j) \in {\mathcal I} \times {\mathbb N} \mid   i \in Supp^{top}(b) \text{ and }  j \in J_i(b)\} \in {\mathcal J}
\]
with $J_i(b)$ being equal to
\[
 \begin{cases} \{1 \leq j \leq \delta_i \mid u_{i, j} = T_{\mathbf c}(v_i^s(b)) \text{ for some $1 \leq s \leq c(b)$} \}, &\text{if $i \ne 0$;}\\ \\ \{1 \leq j \leq \delta'''_0 \mid  u_{0, j} = T_{\mathbf c}(v_0^s(b)) \text{ for some $1 \leq s \leq c(b)$} \}, &\text{if $i = 0$ and $b \in {\mathbf B}^+$;}\\ \\ \{(\delta'''_0 + 1 )\leq j \leq \delta_0 \mid  u_{0, j} = T_{\mathbf c}(v_0^s(b)) \text{ for some $1 \leq s \leq c(b)$} \}, &\text{if $i = 0$ and $b \in {\mathbf B}^-$.} \end{cases}
\]

Similarly 
\[
X'_{\mathbf c}[\ker((X')^2_{\mathbf c})] + (X')^2_{\mathbf c}[\ker((X')^4_{\mathbf c})] + \dots + (X')^k_{\mathbf c}[\ker((X')^{2k}_{\mathbf c})] + \dots
\]
is a maximal isotropic subspace of $V$ and it is equal to $L_{{}'J({\mathbf c})}$ with the notation of lemma~\ref{L:ConditionSO(V)orbitOdd}, where 
\[
{}'J({\mathbf c}) = \left[\left[J({\mathbf c}) \setminus \{(0, (\delta_0 + 1 - {\widehat j}\,   ) )\}\right] \cup \{(0, {\widehat j}\,  )\}\right]\in {\mathcal J}.
\]
such that 	${\widehat j}$ being the  unique integer ${\widehat {j}} \in [1, \delta'''_0]$ such that  $T_{\mathbf c}(v_0^{c(b')}(b')) = u_{0, {\widehat j}}$,

The same thing can done for ${\mathbf c}'$. Let $X_{{\mathbf c}'} = T_{{\mathbf c}'} E_{{\mathbf c}'} T_{{\mathbf c}'}^{-1}: V \rightarrow V$ and $X'_{{\mathbf c}'} = T_{{\mathbf c}'} S_{{\mathcal O}''} E_{{\mathbf c}'} S^{-1}_{{\mathcal O}''}T_{{\mathbf c}'}^{-1}: V \rightarrow V$ as in the proof of proposition~\ref{P:OrbitSplittingSO}  (c).  Note here ${\mathcal O}'' \in {\mathbf B}/\tau$ is defined in an  analoguous way to how  ${\mathcal  O}' \in {\mathbf B}/\tau$ was defined.  ${\mathcal O}''$  is not necessarily equal to ${\mathcal O}'$, because ${\mathbf c}'$ is not necessarily equal to  ${\mathbf c}$.   As we have shown ${}'{\mathcal O}_{{\mathbf c}'}$ is the $G^{\iota}$-orbit of $X_{{\mathbf c}'}$ and ${}''{\mathcal O}_{{\mathbf c}'}$ is the $G^{\iota}$-orbit of $X'_{{\mathbf c}'}$.  So we have $J({\mathbf c}')$ and  ${}'J({\mathbf c}')$. 

If $X_{\mathbf c}$ is in the same $G$-orbit with $X_{{\mathbf c}'}$ (respectively ${}'X_{{\mathbf c}'}$) , then  we have that  the parity of 
\[
\vert \{(i, j) \in J({\mathbf c}) \mid i < 0\} \cup \{(0, j) \in J({\mathbf c}) \mid \delta'''_0 < j\} \vert
\]
is equal to the parity of 
\[
\vert \{(i, j) \in J({{\mathbf c}'}) \mid i < 0\} \cup \{(0, j) \in J({{\mathbf c}'}) \mid \delta'''_0 < j\} \vert 
\]
(respectively 
\[
\vert \{(i, j) \in {}'J({{\mathbf c}'}) \mid i < 0\} \cup \{(0, j) \in J({{\mathbf c}'}) \mid \delta'''_0 < j\} \vert )
\] 
by lemma~\ref{L:ConditionSO(V)orbitOdd}.  Similarly ${}'X_{\mathbf c}$ is in the same $G$-orbit with $X_{{\mathbf c}'}$ (respectively ${}'X_{{\mathbf c}'}$) , then  we have that  the parity of 
\[
\vert \{(i, j) \in {}'J({\mathbf c}) \mid i < 0\} \cup \{(0, j) \in {}'J({\mathbf c}) \mid \delta'''_0 < j\} \vert
\]
is equal to the parity of 
\[
\vert \{(i, j) \in J({{\mathbf c}'}) \mid i < 0\} \cup \{(0, j) \in J({{\mathbf c}'}) \mid \delta'''_0 < j\} \vert 
\]
(respectively 
\[
\vert \{(i, j) \in {}'J({{\mathbf c}'}) \mid i < 0\} \cup \{(0, j) \in J({{\mathbf c}'}) \mid \delta'''_0 < j\} \vert )
\] 
Note that if  $X_{\mathbf c}$ is in the same $G$-orbit with $X_{{\mathbf c}'}$ , then ${}'X_{\mathbf c}$ is in the same $G$-orbit with ${}'X_{{\mathbf c}'}$, while  $X_{\mathbf c}$ is in the same $G$-orbit with ${}'X_{{\mathbf c}'}$ , then ${}'X_{\mathbf c}$ is in the same $G$-orbit with $X_{{\mathbf c}'}$ as a consequence of lemma~\ref{L:SO(V)OrbitSplittingEvenConjugacyClass}. Thus ${\mathcal O}^{even}_{\mathbf c}$  is in the same $G$-orbit as  ${\mathcal O}^{even}_{{\mathbf c}'}$ and ${\mathcal O}^{odd}_{\mathbf c}$  is in the same $G$-orbit as  ${\mathcal O}^{odd}_{{\mathbf c}'}$.

Reciprocally assume that ${\mathcal O} = {\mathcal O}^{even}_{\mathbf c}$,  ${\mathcal O}' = {\mathcal O}^{even}_{{\mathbf c}'}$ and  ${\mathcal O}$ and ${\mathcal O}'$ are not in the same $G$-orbit. Note that they are in the same $O(V)$-orbit, because they both have the same partition $2^{a(2)} 4^{a(4)} 6^{a(6)} \dots $ corresponding to their Jordan type.  ${\mathcal O}$ is either the $G^{\iota}$-orbit of $X_{\mathbf c}$ or the $G^{\iota}$-orbit of ${}'X_{\mathbf c}$, but only one of these two possibilities prevails. We will denote this unique element by $Y_{\mathbf c}$.  Similarly  ${\mathcal O}'$ is either the $G^{\iota}$-orbit of $X_{{\mathbf c}'}$ or the $G^{\iota}$-orbit of ${}'X_{{\mathbf c}'}$, but only one of these two possibilities prevails. We will denote this unique element by $Y_{{\mathbf c}'}$.  Because the two orbits ${\mathcal O}$ and ${\mathcal O}'$ are not in the same $SO(V)$-orbit, but are in the same $O(V)$-orbit,   then there exists an element $g \in O(V)$ such that $\det(g) = -1$ and  $Y_{{\mathbf c}'} = g(Y_{\mathbf c})g^{-1}$.

The maximal isotropic subspace 
\[
Y_{\mathbf c}[\ker(Y^2_{\mathbf c})] + Y^2_{\mathbf c}[\ker(Y^4_{\mathbf c})] + \dots + Y^k_{\mathbf c}[\ker(Y^{2k}_{\mathbf c})] + \dots
\]
is equal to either  $L_{J({\mathbf c})}$  if $Y_{\mathbf c} = X_{\mathbf c}$ or $L_{{}'J({\mathbf c})}$  if $Y_{\mathbf c} = X'_{\mathbf c}$.  We will denote this unique maximal isotropic subspace by $L(Y_{\mathbf c})$. 

Similarly  the maximal isotropic subspace 
\[
Y_{{\mathbf c}'}[\ker(Y^2_{{\mathbf c}'})] + Y^2_{{\mathbf c}'}[\ker(Y^4_{{\mathbf c}'})] + \dots + Y^k_{{\mathbf c}'}[\ker(Y^{2k}_{{\mathbf c}'})] + \dots
\]
is equal to either  $L_{J({{\mathbf c}'})}$  if $Y_{{\mathbf c}'} = X_{{\mathbf c}'}$ or $L_{{}'J({{\mathbf c}'})}$  if $Y_{{\mathbf c}'} = X'_{{\mathbf c}'}$. We will denote this unique maximal isotropic subspace by $L(Y_{{\mathbf c}'})$. 

Consequently we have that $L(Y_{{\mathbf c}'}) = g(L(Y_{\mathbf c}))$,  where $g \in O(V)$ and $\det(g) = -1$ and these maximal isotropic subspaces are not in the same $SO(V)$-orbit. Thus only one of these maximal isotropic subspaces is in the $SO(V)$-orbit of $L_{J^+}$.  If $L(Y_{\mathbf c})$ is in the same $SO(V)$-orbit as $L_{J^+}$ and $L(Y_{{\mathbf c}'})$ is not in the same $SO(V)$-orbit as $L_{J^+}$, then  the parity of 
\[
\vert \{(i, j) \in J({\mathbf c}) \mid i < 0\} \cup \{(0, j) \in J({\mathbf c}) \mid \delta'''_0 < j\} \vert
\]
is equal to the parity of 
\[
\vert \{(i, j) \in J^+ \mid i < 0\} \cup \{(0, j) \in J^+ \mid \delta'''_0 < j\} \vert  = 0
\]
if $Y_{\mathbf c} = X_{\mathbf c}$, i.e. it is even, while either 
\[
\vert \{(i, j) \in J({{\mathbf c}'}) \mid i < 0\} \cup \{(0, j) \in J({{\mathbf c}'}) \mid \delta'''_0 < j\} \vert
\]
is odd if $Y_{{\mathbf c}'} = X_{{\mathbf c}'}$ or  
\[
\vert \{(i, j) \in {}'J({{\mathbf c}'}) \mid i < 0\} \cup \{(0, j) \in {}'J({{\mathbf c}'}) \mid \delta'''_0 < j\} \vert
\]
is odd if $Y_{{\mathbf c}'} = X'_{{\mathbf c}'}$
and
\[
\vert \{(i, j) \in {}'J({\mathbf c}) \mid i < 0\} \cup \{(0, j) \in {}'J({\mathbf c}) \mid \delta'''_0 < j\} \vert
\]
is equal to the parity of 
\[
\vert \{(i, j) \in J^+ \mid i < 0\} \cup \{(0, j) \in J^+ \mid \delta'''_0 < j\} \vert  = 0
\]
if $Y_{\mathbf c} = X'_{\mathbf c}$, i.e. it is even, while either 
\[
\vert \{(i, j) \in J({{\mathbf c}'}) \mid i < 0\} \cup \{(0, j) \in J({{\mathbf c}'}) \mid \delta'''_0 < j\} \vert
\]
is odd if $Y_{{\mathbf c}'} = X_{{\mathbf c}'}$ or  
\[
\vert \{(i, j) \in {}'J({{\mathbf c}'}) \mid i < 0\} \cup \{(0, j) \in {}'J({{\mathbf c}'}) \mid \delta'''_0 < j\} \vert
\]
is odd if $Y_{{\mathbf c}'} = X'_{{\mathbf c}'}$. Thus ${\mathcal O}$ has even parity and ${\mathcal O}'$ has odd parity contradicting our hypothesis. 

We can proceed analogously if $L(Y_{\mathbf c})$ is not in the same $SO(V)$-orbit as $L_{J^+}$ and $L(Y_{{\mathbf c}'})$ is in the same $SO(V)$-orbit as $L_{J^+}$.  In this case, ${\mathcal O}$ has odd parity and ${\mathcal O}'$ has even parity contradicting our hypothesis. 
\end{proof}

\section{Parabolic and Levi subgroups associated to $G^{\iota}$-orbits  in ${\mathfrak g}_2$.}

\subsection{}
Let $m$, $G$, ${\mathfrak g}$, $V$, ${\mathcal I}$, $\oplus_{i \in {\mathcal I}} V_i$, $\delta = (\delta_i)_{i \in {\mathcal I}}$,  ${\mathcal B}$, $\iota$, ${\mathbf B}$ and ${\mathbf B}/\tau$ as in either \ref{S:SetUpAeven} or  \ref{S:SetUpOdd} or \ref{S:SetUpOddOrtho}. More precisely $G$ should be a connected  reductive group. Thus in the situation of  \ref{S:SetUpOdd}, we restrict ourself only to the group $Sp(V)$, while in the situation of  \ref{S:SetUpOddOrtho}, $G$ is  the group $SO(V)$. In the situation of \ref{S:SetUpAeven} , $G$ is already connected.

Using the isomorphism $T_{\mathbf c}:V({\mathbf c}) \rightarrow V$, we can transfer results about the orbit ${\mathcal O}_{\mathbf c}$ in the Lie algebra ${\mathfrak g}_2$ to the orbit of $E_{\mathbf c}$ in the Lie algebra 
\[
{\mathfrak g}(V({\mathbf c})) =  \{Y \in End(V({\mathbf c})) \mid \langle Y(u), v\rangle + \langle u, Y(v)\rangle = 0 \quad \text{ for all $u, v \in V({\mathbf c})$}\}.
\]
In this section and the next, we will describe everything in $V({\mathbf c})$ and in ${\mathfrak g}(V({\mathbf c}))$ rather than in $V$ and ${\mathfrak g}$. 

\subsection{} 
In definition~\ref{D:LieHomo}, we have recalled the definition of Lusztig in 5.2 of \cite{L1995} for a parabolic subalgebra ${\mathfrak p}_X$ of ${\mathfrak g}$ and a Levi subalgebra ${\mathfrak l}_X$ of ${\mathfrak p}_X$ associated to an element $X \in {\mathfrak g}_2$. Using the same notation as in definition~\ref{D:LieHomo}, we have that 
\[
{\mathfrak p}_X = \bigoplus_{\begin{subarray}{c} r, r' \in {\mathbb Z} \\ r' \leq r \end{subarray}} {}_r{\mathfrak g}_{r'} \quad \text{ and }  \quad {\mathfrak l}_X = \bigoplus_{\begin{subarray}{c} r, r' \in {\mathbb Z} \\ r' = r \end{subarray}} {}_r{\mathfrak g}_{r'}.
\]

Denote by $P_X$: the parabolic subgroup of $G$ which has ${\mathfrak p}_X$  as Lie algebra, and by $L_X$: the Levi subgroup of $P_X$ whose Lie algebra is ${\mathfrak l}_X$.

\subsection{}
For the coefficient function ${\mathbf c} \in {\mathfrak C}_{\delta}$, we have constructed for an $G^{\iota}$-orbit ${\mathcal O}_{\mathbf c}$ (respectively ${\  }'{\mathcal O}_{\mathbf c}$, ${\  }''{\mathcal O}_{\mathbf c}$)  in \ref{N:CCorrespondance}, \ref{N:CCorrespondanceOdd} and \ref{N:CCorrespondanceSpecialOrtho},  a Jacobson-Morozov  triple $(E_{\mathbf c}, H_{\mathbf c}, F_{\mathbf c})$ (respectively $({\ }'E_{\mathbf c}, {\ }'H_{\mathbf c}, {\ }'F_{\mathbf c})$, $({\ }''E_{\mathbf c}, {\ }''H_{\mathbf c}, {\ }''F_{\mathbf c})$).  In this section, we will describe the parabolic subgroup $P_E$  and its Levi subgroup $L_E$ for the element $E = E_{\mathbf c} \in {\mathcal O}_{\mathbf c}$ (respectively $E = {\ }'E_{\mathbf c} \in {\ }'{\mathcal O}_{\mathbf c}$, $E = {\ }''E_{\mathbf c} \in {\ }''{\mathcal O}_{\mathbf c}$).

\begin{notation}\label{N:Basis_Adapted_To_C}
Let ${\mathbf c} \in {\mathfrak C}_{\delta}$ and the associated symmetric function $c: {\mathbf B} \rightarrow {\mathbb N}$.
For each $i \in {\mathcal I}$, recall that $V_i({\mathbf c})$ has the basis 
\[
{\mathcal B}_i = \{ v_{i}^j(b) \in V_i({\mathbf c}) \mid b \in {\mathbf B}, i \in Supp(b), 1 \leq j \leq c(b) \}
\]
is the basis of $V_i({\mathbf c})$. We know how the elements of the  Jacobson-Morozov triple act on the elements of ${\mathcal B}_i$.  

Write 
\[
\Lambda({\mathbf c}) = \{\lambda(b) \in {\mathbb Z} \mid b \in {\mathbf B}, c(b) \ne 0\} = \{\lambda_1 < \lambda_2 < \dots < \lambda_{\ell}\}.
\]
Recall that 
\[
\lambda(b) = \frac{(i + j)}{2} \quad \text{ if $b = b(i, j, k)$ with $i, j \in {\mathcal I}$, $i \geq j$ and $0 \leq k \leq \mu(i, j)$.}
\]

By the symmetry of the function $c$, it is easy to see that $\lambda_s + \lambda_{\ell + 1 - s} = 0$ for all $1 \leq s \leq \ell$.

For $1 \leq r \leq \ell$, denote by $U_r$: the subspace of $V({\mathbf c})$ spanned by the set of basis vectors 
\[
\widehat{\mathcal B}_r = \{ v_i^j(b) \mid b \in {\mathbf B}, 1 \leq j \leq c(b), i \in Supp(b) \text{ and } \lambda(b) = \lambda_r\}
\]
This gives us a direct sum decomposition of $V({\mathbf c})$: $V({\mathbf c}) = U_1 \oplus U_2 \oplus \dots \oplus U_{\ell}$.
\end{notation}

\begin{proposition}\label{P:ParabolicLeviAssociated}
Let $X \in {\mathfrak g}$ and $E = E_{\bf c}$ or ${\  }'E_{\mathbf c}$.  Then 
\begin{enumerate}[\upshape (a)]
\item $X \in {\mathfrak l}_{E}$ if and only if $X(U_r) \subseteq U_r$ for all $r = 1, 2, \dots, \ell$.
\item $X \in {\mathfrak p}_{E}$ if and only if $X(U_1 \oplus U_2 \dots \oplus U_r) \subseteq (U_1 \oplus U_2 \oplus \dots \oplus U_r)$ for all  $r = 1, 2, \dots, \ell$.
\end{enumerate}
\end{proposition}
\begin{proof}
(a) $\Rightarrow)$ Let $X \in {\mathfrak l}_E$ and consider the basis vector $v_i^j(b)$ with $b \in {\mathbf B}$, $1 \leq j \leq c(b)$, $i \in Supp(b)$ and $\lambda(b) = \lambda_r$. We want to prove that $X(v_i^j(b)) \in U_r$. We have that $X = \sum_s X_s$ with $X_s \in {}_s{\mathfrak g}_s$. It is enough to prove the result for $X = X_s$. 

Write 
\[
X(v_i^j(b)) = \sum_{(b', j', i')}  \xi_{(b', j', i')} v_{i'}^{j'}(b')
\]
where  $\xi_{(b', j', i')} \in {\mathbf k}$ for all $(b', j', i')$ such that  $b' \in {\mathbf B}$, $1 \leq j' \leq c(b')$ and $i' \in Supp(b')$ and  the sum runs  over these $(b', j', i')$.  If $\xi_{(b', j', i')} \ne 0$, we want to prove that $\lambda(b') = \lambda_r$.

Because $X = X_s$, we have $Ad(\iota(t))X = t^s X$ and  $Ad(\iota'(t))X = t^s X$ for all $t \in {\mathbf k}^{\times}$. Recall that $\iota':{\mathbf k}^{\times} \rightarrow G$ is defined  in definition~\ref{D:LieHomo} and in the proof of proposition~\ref{DimensionFormulaSymplecticOdd}.  We have

\[
\begin{aligned}
Ad(\iota(t))X(v_i^j(b)) &= \iota(t) X( \iota(t^{-1})(v_i^j(b))) = \iota(t) X(t^{-i} v_i^j(b))\\
&= t^{-i} \iota(t) \left( \sum_{(b', j', i')}  \xi_{(b', j', i')} v_{i'}^{j'}(b')\right) \\
&=  \sum_{(b', j', i')}  t^{(i' - i)} \xi_{(b', j', i')} v_{i'}^{j'}(b')
\end{aligned}
\]

and 

\[
\begin{aligned}
Ad(\iota'(t))X(v_i^j(b)) &= \iota'(t) X( \iota'(t^{-1})(v_i^j(b))) = \iota'(t) X(t^{-h_i(b)} v_i^j(b)) \\
&= t^{-h_i(b)} \iota'(t) \left( \sum_{(b', j', i')}  \xi_{(b', j', i')} v_{i'}^{j'}(b')\right) \\
&=  \sum_{(b', j', i')}  t^{(h_{i'}(b') - h_i(b))} \xi_{(b', j', i')} v_{i'}^{j'}(b')
\end{aligned}
\]
where both sums  run over $(b', j', i')$ as above.

Because $Ad(\iota'(t))X = Ad(\iota(t)) X = t^sX$ for all $t \in {\mathbf k}^{\times}$, we get that 
\[
t^{(i' - i)} \xi_{(b', j', i')} = t^{(h_{i'}(b') - h_i(b))} \xi_{(b', j', i')} 
\]
for all $t \in {\mathbf k}^{\times}$ and all $(b', j', i')$ where $b' \in {\mathbf B}$, $1 \leq j' \leq c(b')$ and $i' \in Supp(b')$ . If $\xi_{(b', j', i')} \ne 0$, then $t^{(i' - i)}  = t^{(h_{i'}(b') - h_i(b))}$ for all $t \in {\mathbf k}^{\times}$ and consequently $i' - h_{i'}(b') = i - h_i(b)$. Thus $\lambda(b') = \lambda(b) = \lambda_r$ and $X(U_r) \subseteq U_r$.

$\Leftarrow)$ If $X \in {\mathfrak g}$ is such that $X(U_r) \subseteq U_r$ for all $r = 1, 2, \dots, \ell$, we must show that $X \in {\mathfrak l}_E$. We will first prove that $Ad(\iota(t)) X = Ad(\iota'(t))X$ for all $t \in {\mathbf k}^{\times}$ by computing both sides on the basis vector  $v_i^j(b)$ with $b \in {\mathbf B}$, $1 \leq j \leq c(b)$, $i \in Supp(b)$ and $\lambda(b) = \lambda_r$.

Write 
\[
X(v_i^j(b)) = \sum_{(b', j', i')}  \xi_{(b', j', i')} v_{i'}^{j'}(b')
\]
where $\xi_{(b', j', i')} \in {\mathbf k}$ for all $(b', j', i')$ such that  $b' \in {\mathbf B}$, $1 \leq j' \leq c(b')$ and $i' \in Supp(b')$  and the sum runs over these $(b', j', i')$ and $\lambda(b') = \lambda(b) = \lambda_r$. 

\[
\begin{aligned}
Ad(\iota(t))X(v_i^j(b)) &= \iota(t) X \iota(t^{-1})(v_i^j(b)) = \iota(t) X(t^{-i} v_i^j(b)) \\
&= t^{-i} \iota(t) \left( \sum_{(b', j', i')}  \xi_{(b', j', i')} v_{i'}^{j'}(b')\right) \\
&=  \sum_{(b', j', i')}  t^{(i' - i)} \xi_{(b', j', i')} v_{i'}^{j'}(b')
\end{aligned}
\]

and 

\[
\begin{aligned}
Ad(\iota'(t))X(v_i^j(b)) &= \iota'(t) X \iota'(t^{-1})(v_i^j(b)) = \iota'(t) X(t^{-h_i(b)} v_i^j(b)) \\
&= t^{-h_i(b)} \iota'(t) \left( \sum_{(b', j', i')}  \xi_{(b', j', i')} v_{i'}^{j'}(b')\right) \\
&=  \sum_{(b', j', i')}  t^{(h_{i'}(b') - h_i(b))} \xi_{(b', j', i')} v_{i'}^{j'}(b')
\end{aligned}
\]
where both sums run over $b' \in {\mathbf B}$, $1 \leq j' \leq c(b')$ ,  $i' \in Supp(b')$ and $\lambda(b') = \lambda(b) = \lambda_r$. 
Thus if $\xi_{(b', j', i')} \ne 0$ in the above sums, we get that $(i' - h_{i'}(b')) = (i - h_i(b))= \lambda_r$ and $(i' - i) = (h_{i'}(b') - h_i(b))$. Thus we have $Ad(\iota(t))X(v_i^j(b)) = Ad(\iota'(t))X(v_i^j(b)) $ for all basis vectors $v_i^j(b)$.  So $Ad(\iota(t)) X = Ad(\iota'(t))X$ for all $t \in {\mathbf k}^{\times}$.

We will now prove that $X \in {\mathfrak l}_E$. Write 
\[
X = \sum_{s, s'} {}_sX_{s'} \   \in \  \bigoplus_{s, s'} {}_s{\mathfrak g}_{s'}
\]
where ${}_sX_{s'} \in {}_s{\mathfrak g}_{s'}$ and the sum is over $s, s' \in {\mathbb Z}$ and there is a finite number of summands.  Then 
\[
Ad(\iota(t))X = \sum_{s, s'} Ad(\iota(t)) {}_sX_{s'} = \sum_{s, s'} t^{s'} {}_sX_{s'}
\]
and 
\[
Ad(\iota'(t))X = \sum_{s, s'} Ad(\iota'(t)) {}_sX_{s'} = \sum_{s, s'} t^{s} {}_sX_{s'}.
\]
By the above equality, we get that $t^{s'} = t^s$ for all $t \in {\mathbf k}^{\times}$ and $s' = s$ when ${}_sX_{s'} \ne 0$. Thus $X \in {\mathfrak l}_E$.

(b) $\Rightarrow)$ Let $X \in {\mathfrak p}_E$  and  the basis vector $v_i^j(b)$ with $b \in {\mathbf B}$, $1 \leq j \leq c(b)$, $i \in Supp(b)$ and $\lambda(b) = \lambda_r$, we want to prove that $X(v_i^j(b)) \in U_1 \oplus U_2 \oplus \dots  \oplus U_r$ for all $r = 1, 2, \dots, \ell$. We have 
\[
X = \sum_{\begin{subarray}{c} s, s'\\ s' \leq s\end{subarray}} {}_sX_{s'}  \   \in \  \bigoplus_{\begin{subarray}{c} s, s'\\ s' \leq s\end{subarray} } {}_s{\mathfrak g}_{s'}
\]
 with  ${}_sX_{s'} \in {}_s{\mathfrak g}_{s'}$. It is enough to prove the result for $X = {}_sX_{s'}$.
 
Write 
\[
X(v_i^j(b)) = \sum_{(b', j', i')}  \xi_{(b', j', i')} v_{i'}^{j'}(b')
\]
where $\xi_{(b', j', i')} \in {\mathbf k}$ for all $(b', j', i')$ such that  $b' \in {\mathbf B}$, $1 \leq j' \leq c(b')$ and $i' \in Supp(b')$  and the sum runs over these $(b', j', i')$.   If $\xi_{(b', j', i')} \ne 0$, we want to show that   $\lambda(b')  \leq \lambda_r$. 

Because $X = {}_sX_{s'}$ with $s' \leq s$, then $Ad(\iota(t))X = t^{s'} X$ and $Ad(\iota'(t))X = t^s X$ for all $t \in {\mathbf k}^{\times}$, with $s' \leq s$.  Now 
\[
\begin{aligned}
Ad(\iota(t))X(v_i^j(b)) &= \iota(t) X \iota(t^{-1})(v_i^j(b)) = \iota(t) X(t^{-i} v_i^j(b)) \\
&= t^{-i} \iota(t) \left( \sum_{(b', j', i')}  \xi_{(b', j', i')} v_{i'}^{j'}(b')\right) \\
&=  \sum_{(b', j', i')}  t^{(i' - i)} \xi_{(b', j', i')} v_{i'}^{j'}(b') =  \sum_{(b', j', i')}  t^{s'} \xi_{(b', j', i')} v_{i'}^{j'}(b')
\end{aligned}
\]

and 

\[
\begin{aligned}
Ad(\iota'(t))X(v_i^j(b)) &= \iota'(t) X \iota'(t^{-1})(v_i^j(b)) = \iota'(t) X(t^{-h_i(b)} v_i^j(b)) \\
&= t^{-h_i(b)} \iota'(t) \left( \sum_{(b', j', i')}  \xi_{(b', j', i')} v_{i'}^{j'}(b')\right) \\
&=  \sum_{(b', j', i')}  t^{(h_{i'}(b') - h_i(b))} \xi_{(b', j', i')} v_{i'}^{j'}(b') =  \sum_{(b', j', i')}  t^{s} \xi_{(b', j', i')} v_{i'}^{j'}(b')
\end{aligned}
\]
where both sums run over $b' \in {\mathbf B}$, $1 \leq j' \leq c(b')$ and  $i' \in Supp(b')$.

If $\xi_{(b', j', i')} \ne 0$ in the above equations, then $(i' - i) = s'$ in the first sum and $(h_{i'}(b') - h_{i}(b)) = s$ in the second sum. Because $s' \leq s$, we get that $(i' - i)\leq (h_{i'}(b') - h_i(b))$. Thus  $\lambda(b') = (i' - h_{i'}(b')) \leq (i - h_i(b)) = \lambda(b) = \lambda_r$ and we get easily that $X(U_1 \oplus U_2 \oplus \dots U_r) \subseteq (U_1 \oplus U_2 \oplus \dots \oplus U_r)$ for $r = 1, 2, \dots, \ell$.

$\Leftarrow)$  If $X \in {\mathfrak g}$ is such that $X(U_1 \oplus U_2 \oplus \dots \oplus U_r) \subseteq (U_1 \oplus U_2 \oplus \dots \oplus U_r)$ for all $r = 1, 2, \dots, \ell$, we must show that $X \in {\mathfrak p}_E$.  We can assume that $X \ne 0$.
If we write 
\[
X = \sum_{s, s'} {}_sX_{s'} \   \in \  \bigoplus_{s, s'} {}_s{\mathfrak g}_{s'}
\]
where the sum is finite and the summands are such that  ${}_sX_{s'} \in {}_s{\mathfrak g}_{s'}$ and ${}_sX_{s'} \ne 0$, then we must prove that $s' \leq s$.  

Note we have that 
\[
Ad(\iota(t))X = \sum_{s, s'} t^{s'} {}_sX_{s'} \quad \text{ and } \quad Ad(\iota'(t))X = \sum_{s, s'} t^{s} {}_sX_{s'}  
\]
for all $t \in {\mathbf k}^{\times}$.

Let $b \in {\mathbf B}$, $1 \leq j \leq c(b)$, $i \in Supp(b)$,  $\lambda(b) = \lambda_r$ for some $r = 1, 2, \dots, \ell$ such that $X(v_i^j(b)) \ne 0$. Such a basis vector exists because $X \ne 0$.  Write
\[
X(v_i^j(b)) = \sum_{(b', j', i')}  \xi_{(b', j', i')} v_{i'}^{j'}(b')
\]
where $\xi_{(b', j', i')} \in {\mathbf k}$ for all $(b', j', i')$ such that  $b' \in {\mathbf B}$, $1 \leq j' \leq c(b')$,  $i' \in Supp(b')$, $\lambda(b') \leq \lambda_r$  and the sum runs over these $(b', j', i')$, by our hypothesis.  

Write 
\[
{}_sX_{s'}(v_i^j(b)) = \sum_{(b', j', i')} \chi_{(b', j', i')}^{(s, s')} v_{i'}^{j'}(b')
\]
where  $\chi_{(b', j', i')}^{(s, s')} \in {\mathbf k}$ for all $(b', j', i')$ such that  $b' \in {\mathbf B}$, $1 \leq j' \leq c(b')$,  $i' \in Supp(b')$ and  the sum runs over these $(b', j', i')$.

Note that 
\[
\begin{aligned}
Ad(\iota(t))X(v_i^j(b)) &= \iota(t) X(t^{-i} v_i^j(b)) = \sum_{(b', j', i')} t^{(i' - i)}  \xi_{(b', j', i')} v_{i'}^{j'}(b') \\
&= \sum_{s, s'} t^{s'} {}_sX_{s'}(v_i^j(b)) = \sum_{(b', j', i')}\left[\sum_{s'} t^{s'} \left(\sum_s \chi_{(b', j', i')}^{(s, s')}\right) \right] v_{i'}^{j'}(b')
\end{aligned}
\]
and
\[
\begin{aligned}
Ad(\iota'(t))X(v_i^j(b)) &= \iota'(t) X(t^{-h_i(b)} v_i^j(b)) = \sum_{(b', j', i')} t^{(h_{i'}(b') - h_i(b))}  \xi_{(b', j', i')} v_{i'}^{j'}(b') \\
&= \sum_{s, s'} t^{s} {}_sX_{s'}(v_i^j(b)) = \sum_{(b', j', i')}\left[\sum_{s} t^{s} \left(\sum_{s'} \chi_{(b', j', i')}^{(s, s')}\right) \right] v_{i'}^{j'}(b')
\end{aligned}
\]
where  both sums run over $b' \in {\mathbf B}$, $1 \leq j' \leq c(b')$,   $i' \in Supp(b')$ and $\lambda(b') \leq \lambda(b) = \lambda_r$. 

Because $X(v_i^j(b)) \ne 0$, there exists $b'' \in {\mathbf B}$, $1 \leq j'' \leq c(b'')$,   $i'' \in Supp(b'')$ and $\lambda(b'') \leq \lambda(b) = \lambda_r$ such that $\xi_{(b'', j'', i'')} \ne 0$. If we  consider the coefficient of $v_{i''}^{j''}(b'')$ in the above expressions, we get that 
\[
t^{(i'' - i)}  \xi_{(b'', j'', i'')} = t^{s'} \left(\sum_s \chi_{(b'', j'', i'')}^{(s, s')}\right)  \text{ and }  t^{(h_{i''}(b'') - h_i(b))}  \xi_{(b'', j'', i'')} =  t^{s} \left(\sum_{s'} \chi_{(b'', j'', i'')}^{(s, s')}\right).
\]
Because $\xi_{(b'', j'', i'')} \ne 0$,  then we have $s' = (i'' - i)$ and $s = (h_{i''}(b'') - h_i(b))$. Now we have that $\lambda(b'') = (i'' - h_{i''}(b'')) \leq (i - h_i(b)) = \lambda(b)$. So $s' = (i'' - i) \leq (h_{i''}(b'') - h_i(b)) = s$.  This concludes the proof of (b).
 
\end{proof}

\subsection{}
We need to consider also the nilpotent element $E = {}''E_{\mathbf c}$ for the $G^{\iota}$-orbit ${}''{\mathcal O}_{\mathbf c}$ as described in notation~\ref{N:CCorrespondanceSpecialOrtho} for the  case (c) of proposition~\ref{P:OrbitSplittingSO}. From our hypothesis in case (c), we get that $0 \not\in \Lambda({\mathbf c})$. Let ${\mathcal O}'$ be the $\langle \tau \rangle$-orbit  defined in the proof of the proposition~\ref{P:OrbitSplittingSO}. We can write ${\mathcal O}' = \{b_1, b_2\}$ with $\tau(b_1) = b_2$ and $\lambda(b_1) < \lambda(b_2)$. This last inequality follows because both $\lambda(b_1)$ and $\lambda(b_2)$ are not zero. Write $\lambda(b_1) = \lambda_{r'}$ and thus $\lambda(b_2) = \lambda_{\ell + 1 - r'}$.

For $1 \leq r \leq \ell$, denote by ${}''U_r$: the subspace of $V$ spanned by the subset ${}''{\widehat {\mathcal B}}_r$, where 
\begin{itemize}
\item if $r \ne r', (\ell + 1 - r')$, then 
\[
{}''{\widehat {\mathcal B}}_r = \{v_i^j(b) \mid b \in {\mathbf B}, \lambda(b) = \lambda_r, i \in Supp(b), 1 \leq j \leq c(b)\};
\]
\item if $r = r'$, then
\[
\begin{aligned}
{}''{\widehat {\mathcal B}}_r &= \{v_i^j(b) \mid b \in {\mathbf B}, \lambda(b) = \lambda_r, i \in Supp(b), i \ne 0,  1 \leq j \leq c(b)\}\\
& \quad \bigcup \{v_0^j(b) \mid b \in {\mathbf B}, b \ne b_1, \lambda(b) = \lambda_r, 0 \in Supp(b), 1 \leq j \leq c(b) \}\\
& \quad \bigcup \{v_0^j(b_1) \mid 1 \leq j \leq (c(b_1) - 1)\} \bigcup \{v_0^{c(b_1)}(b_2)\}
\end{aligned}
\]
\item if $r = (\ell + 1 - r')$, then
\[
\begin{aligned}
{}''{\widehat {\mathcal B}}_r &= \{v_i^j(b) \mid b \in {\mathbf B}, \lambda(b) = \lambda_r, i \in Supp(b), i \ne 0,  1 \leq j \leq c(b)\}\\
& \quad \bigcup \{v_0^j(b) \mid b \in {\mathbf B}, b \ne b_2, \lambda(b) = \lambda_r, 0 \in Supp(b), 1 \leq j \leq c(b) \}\\
& \quad \bigcup \{v_0^j(b_2) \mid 1 \leq j \leq (c(b_2) - 1)\} \bigcup \{v_0^{c(b_2)}(b_1)\}
\end{aligned}
\]
\end{itemize}

\begin{proposition}
Let $X \in {\mathfrak g}$ and $E = {}''E_{\bf c}$.  Then 
\begin{enumerate}[\upshape (a)]
\item $X \in {\mathfrak l}_{E}$ if and only if $X({}''U_r) \subseteq {}''U_r$ for all $r = 1, 2, \dots, \ell$.
\item $X \in {\mathfrak p}_{E}$ if and only if $X({}''U_1 \oplus {}''U_2 \dots \oplus {}''U_r) \subseteq ({}''U_1 \oplus {}''U_2 \oplus \dots \oplus {}''U_r)$ for all $r = 1, 2, \dots, \ell$.
\end{enumerate}
\end{proposition}
\begin{proof}
The proof is similar to the proof of proposition~\ref{P:ParabolicLeviAssociated} and is left to the reader. We just have to adapt the action of $Ad(\iota'(t))$ to this situation and this means transposing $v_0^{c(b_1)}(b_1)$ and $v_0^{c(b_2)}(b_2)$. Note that $c(b_1) = c(b_2)$.
\end{proof}

\begin{remark}\label{R:Remarks6.9}
We can observed the following:
\begin{enumerate}[\upshape (a)]
\item $\#\Lambda({\mathbf c}) = \ell$ is odd if and only if $0 \in \Lambda({\mathbf c})$. In that case, $\lambda_{\ell'} = 0$, where $\ell = 2\ell' - 1$. 
\item If $b \in {\mathbf B}$ is such that  $\lambda(b) = 0$, then $b$ is on the principal diagonal. Thus $0 \in \Lambda({\mathbf c})$ if and only if there exists at least an ${\mathcal I}$-box $b$ on the principal diagonal such that $c(b) \ne 0$. This follows easily from our interpretation of the function $\lambda$ via the ${\mathcal I}$-diagram. See notation~\ref{N:LambdaValue}
\item If $0 \in \Lambda({\mathbf c})$ and with the notation of (a) above, we have that $U_{\ell'} \ne 0$ and the bilinear form $\langle\ , \ \rangle$ of $V$ restricted to $U_{\ell'}$ is non-degenerate. Moreover if $X \in {\mathfrak g}$ is such that $X(U_{\ell'}) \subseteq U_{\ell'}$, then $X\vert_{U_{\ell'}}$ belongs to the Lie subalgebra of ${\mathfrak g}$ corresponding to the subspace $U_{\ell'}$. More precisely, we have 
\[
\langle X\vert_{U_{\ell'}}(u), v \rangle + \langle u, X\vert_{U_{\ell'}}(v) \rangle = 0 \quad \text{ for all } u, v \in U_{\ell'}
\]
\item For $1 \leq r \leq \ell$, we have $v_i^j(b) \in \widehat{\mathcal B}_r$ (respectively ${}''  \widehat{\mathcal B}_r$)  if and only if $v_{-i}^j(\tau(b)) \in \widehat{\mathcal B}_{\ell + 1 - r}$ (respectively ${}''\widehat{\mathcal B}_{\ell + 1 - r}$). For $\widehat{\mathcal B}_r$ and $\widehat{\mathcal B}_{\ell + 1 - r}$, we just need to recall that $c(b) = c(\tau(b))$, also $i \in Supp(b)$ if and only if $-i \in Supp(\tau(b))$ and finally $\lambda(b) = \lambda_r$ if and only if $\lambda(\tau(b)) = \lambda_{\ell + 1 - r}$.   For ${}''\widehat{\mathcal B}_r$ and ${}''\widehat{\mathcal B}_{\ell + 1 - r}$, the proof is similar and we just need to consider $r = r'$ or $\ell + 1 - r'$. 
 \item Let $X \in {\mathfrak l}_{E}$ where $E$ is either $E_{\mathbf c}$ or ${}'E_{\mathbf c}$ (respectively ${}''E_{\mathbf c}$)  and the basis vector $v_{i_1}^{j_1}(b_1) \in \widehat {\mathcal B}_r$ (respectively  ${}''\widehat{\mathcal B}_r$) of the subspace $U_r$ (respectively ${}''U_r$). Write 
 \[
 X(v_{i_1}^{j_1}(b_1)) = \sum_{(b_2, j_2, i_2)} \xi_{(b_1, j_1, i_1)}^{(b_2, j_2, i_2)} v_{i_2}^{j_2}(b_2)
 \]
 where $\xi_{(b_1, j_1, i_1)}^{(b_2, j_2, i_2)} \in {\mathbf k}$ and the sum runs over the triples $(b_2, j_2, i_2)$ where $v_{i_2}^{j_2}(b_2) \in \widehat {\mathcal B}_r$ (respectively ${}''\widehat{\mathcal B}_r$) and write for the basis vector $v_{-i_3}^{j_1}(\tau(b_3)) \in \widehat {\mathcal B}_{\ell + 1 - r}$ (respectively  ${}''\widehat{\mathcal B}_{\ell + 1 - r}$) of the subspace $U_{\ell + 1 - r}$ (respectively ${}''U_{\ell + 1 - r}$) 
 \[
 X(v_{-i_3}^{j_3}(\tau(b_3))) = \sum_{(b_4, j_4, i_4)} \psi_{(b_3, j_3, i_3)}^{(b_4, j_4, i_4)} v_{-i_4}^{j_4}(\tau(b_4))
 \]
where $ \psi_{(b_3, j_3, i_3)}^{(b_4, j_4, i_4)} \in {\mathbf k}$ and  the sum runs over the triples $(b_4, j_4, i_4)$ where $v_{-i_4}^{j_4}(\tau(b_4)) \in \widehat {\mathcal B}_{\ell + 1 - r}$ (respectively ${}''\widehat{\mathcal B}_r$). From 
\[
\langle X(v_{i_1}^{j_1}(b_1)) , v_{-i_3}^{j_3}(\tau(b_3))\rangle + \langle v_{i_1}^{j_1}(b_1) , X(v_{-i_3}^{j_3}(\tau(b_3)))\rangle = 0,
\]
we get that 
\[
\xi_{(b_1, j_1, i_1)}^{(b_3, j_3, i_3)} \langle v_{i_3}^{j_3}(b_3) , v_{-i_3}^{j_3}(\tau(b_3)\rangle  + 
\psi_{(b_3, j_3, i_3)}^{(b_1, j_1, i_1)} \langle v_{i_1}^{j_1}(b_1) , v_{-i_1}^{j_1}(\tau(b_3))\rangle
= 0
\]
Both $\langle v_{i_3}^{j_3}(b_3), v_{-i_3}^{j_3}(\tau(b_3))\rangle $ and $ \langle v_{i_1}^{j_1}(b_1), v_{-i_1}^{j_1}(\tau(b_1))\rangle$ are $\ne 0$. In fact they are equal to either $1$ or $\epsilon$. Thus 
\[
\psi_{(b_3, j_3, i_3)}^{(b_1, j_1, i_1)}  = - \frac{\langle v_{i_3}^{j_3}(b_3) , v_{-i_3}^{j_3}(\tau(b_3))\rangle} { \langle v_{i_1}^{j_1}(b_1), v_{-i_1}^{j_1}(\tau(b_1))\rangle} \xi_{(b_1, j_1, i_1)}^{b_3, j_3, i_3)}.
\]
In other words, $X\vert_{U_{\ell + 1 - r}}$ is obtained from $X\vert_{U_r}$.
\item Let $E$ be  either the nilpotent element $E_{\mathbf c}$ or ${}'E_{\mathbf c}$ (respectively  ${}''E_{\mathbf c}$).   It is known that if $E \in {\mathfrak l}_E$ and consequently $E(U_r) \subseteq U_r$ (respectively $E({}''U_r) \subseteq {}''U_r$) for all $r = 1, 2, \dots, \ell$, then
\begin{itemize}
\item the dimension of $U_r$ (respectively  ${}''U_r$) is 
\[
\sum_{\begin{subarray}{c} b \in {\mathbf B}\\ \lambda(b) = \lambda_r \end{subarray}} c(b) \vert Supp(b) \vert; 
\]
\item  the partition corresponding to the Jordan decomposition of the restriction $E\vert_{U_r}$ to $U_r$ (respectively $E\vert_{{}''U_r}$ to ${}''U_r$) is 
\[
\prod_{\begin{subarray}{c} b \in {\mathbf B}\\ \lambda(b) = \lambda_r \end{subarray}}  \vert Supp(b) \vert^{c(b)}.
\]
\end{itemize}
For the dimension of $U_r$, this follows by computing $\vert \widehat{\mathcal B}_r \vert$ or  $\vert {}''\widehat{\mathcal B}_r \vert$. Note that in the case of the dimension of ${}''U_r$ for $r = r'$ or $\ell + i - r'$, it follows from the fact that the function $c$ is symmetric and also $\vert Supp(\tau(b)) \vert = \vert Supp(b) \vert$.

As for the partition corresponding to the Jordan decomposition of the restriction $E\vert_{U_r}$ to $U_r$ (respectively $E\vert_{{}''U_r}$ to ${}''U_r$), it follows by the known action of the nilpotent element $E$ on the basis elements in $\widehat{\mathcal B}_r$ or  ${}''\widehat{\mathcal B}_r$. 

\end{enumerate}  

\end{remark}

We will now describe the isomorphism class of the Levi subalgebra ${\mathfrak l}_E$ and also the tuple of  partitions corresponding to the Jordan decomposition of $E\vert_{U_r}$.

\begin{proposition}\label{P:LieSubgroupDescription}
Let $E$ be  either the nilpotent element $E_{\mathbf c}$ or ${}'E_{\mathbf c}$ (respectively ${}''E_{\mathbf c}$). Denote by $n_r = \dim(U_r)$ (respectively $n_r = \dim({}''U_r)$) for $1 \leq r \leq \ell$. As we saw in remark~\ref{R:Remarks6.9} (f), 
\[
n_r = \sum_{\begin{subarray}{c} b \in {\mathbf B}\\ \lambda(b) = \lambda_r \end{subarray}} c(b) \vert Supp(b) \vert; 
\]
for $1 \leq r \leq \ell$.

\begin{itemize}
\item If $\# \Lambda({\mathbf c}) = \ell$ is even and write $\ell = 2\ell'$, then the map 
\[
{\mathfrak l}_{E} \rightarrow {\mathfrak {gl}}({n_1}) \oplus {\mathfrak {gl}}({n_2}) \oplus \dots \oplus {\mathfrak {gl}}({n_{\ell'}})
\] 
defined by  
\[
X \mapsto \begin{cases} ([X\vert_{U_1}]_{\widehat{\mathcal B}_1 }, [X\vert_{U_2}]_{\widehat{\mathcal B}_2 }, \dots, [X\vert_{U_{\ell'}}]_{\widehat{\mathcal B}_{\ell'} }), &\text{ if $E = E_{\mathbf c}$ or ${}'E_{\mathbf c}$;}\\  \\([X\vert_{{}''U_1}]_{{}''\widehat{\mathcal B}_1 }, [X\vert_{{}''U_2}]_{{}''\widehat{\mathcal B}_2 }, \dots, [X\vert_{{}''U_{\ell'}}]_{{}''\widehat{\mathcal B}_{\ell'} }), &\text{ if $E = {}''E_{\mathbf c}$;}
\end{cases}
\]
 is an isomorphism of Lie algebras and the tuple of  partitions corresponding to the Jordan decomposition of $E\vert_{U_r}$ as given in (f) of remark~\ref{R:Remarks6.9} is  
\[
\left(\prod_{\begin{subarray}{c} b \in {\mathbf B}\\ \lambda(b) = \lambda_r \end{subarray}}  \vert Supp(b) \vert^{c(b)}\right)_{1 \leq r \leq \ell'}.
\]
 \item If $\#  \Lambda({\mathbf c}) = \ell$ is odd and write $\ell = 2\ell' - 1$, then the map 
\[
{\mathfrak l}_{E} \rightarrow {\mathfrak {gl}}({n_1}) \oplus {\mathfrak {gl}}({n_2}) \oplus \dots \oplus {\mathfrak {gl}}({n_{\ell' - 1}}) \oplus {\mathfrak g'}
\]
defined by
\[
X \mapsto \begin{cases} ([X\vert_{U_1}]_{\widehat{\mathcal B}_1 }, [X\vert_{U_2}]_{\widehat{\mathcal B}_2 }, \dots, [X\vert_{U_{\ell'}}]_{\widehat{\mathcal B}_{\ell'} }), &\text{ if $E = E_{\mathbf c}$ or ${}'E_{\mathbf c}$;}\\  \\([X\vert_{{}''U_1}]_{{}''\widehat{\mathcal B}_1 }, [X\vert_{{}''U_2}]_{{}''\widehat{\mathcal B}_2 }, \dots, [X\vert_{{}''U_{\ell'}}]_{{}''\widehat{\mathcal B}_{\ell'} }), &\text{ if $E = {}''E_{\mathbf c}$;}
\end{cases}
\]
is an isomorphism of Lie algebras and the tuple of partitions corresponding to the Jordan decomposition of $E\vert_{U_r}$ as given in (f) in remark~\ref{R:Remarks6.9} is 
\[
\left(\prod_{\begin{subarray}{c} b \in {\mathbf B}\\ \lambda(b) = \lambda_r \end{subarray}}  \vert Supp(b) \vert^{c(b)}\right)_{1 \leq r \leq \ell'}.
\]
\end{itemize}
Here $[X\vert_{U_r}]_{\widehat{\mathcal B}_r }$ (respectively $[X\vert_{{}''U_r}]_{{}''\widehat{\mathcal B}_r }$)  is the matrix of $X\vert_{U_r}$ (respectively $X\vert_{{}''U_r}$) relative to the basis  $\widehat{\mathcal B}_r$ (respectively  ${}''\widehat{\mathcal B}_r$) of $U_r$ (respectively ${}''U_r$), ${\mathfrak {gl}}({n_r})$ is the general linear Lie algebra of $n_r \times n_r$-matrices  for $r = 1, 2, \dots, \ell'$ and ${\mathfrak g'}$ is the classical Lie algebra
\[
{\mathfrak g'} = 
\begin{cases} {\mathfrak {so}}({n_{\ell'}}), &\text{ if the bilinear form $\langle\  ,\   \rangle$ is symmetric;}\\  \\ {\mathfrak {sp}}({n_{\ell'}}), &\text{ if the bilinear form $\langle\  ,\   \rangle$ is skew-symmetric.}
\end{cases}
\]
\end{proposition}
\begin{proof}
We will only sketch this proof. Let $E$ be either the nilpotent element $E_{\mathbf c}$ or ${}'E_{\mathbf c}$.  It is clear that the map 
\[
X \mapsto ([X\vert_{U_1}]_{\widehat{\mathcal B}_1 }, [X\vert_{U_2}]_{\widehat{\mathcal B}_2 }, \dots, [X\vert_{U_{\ell'}}]_{\widehat{\mathcal B}_{\ell'} })
\]
is a linear Lie algebra homomorphism. If we consider its kernel, we get that $X$ is $0$ on $U_r$ for $r = 1, 2, \dots, \ell'$. By the equations in  remark~\ref{R:Remarks6.9} (e), we get also that $X$ is $0$ on $U_r$ for $r = \ell' + 1, \dots, \ell$. Because $V = U_1 \oplus U_2 \oplus \dots \oplus U_{\ell}$, we get that $X$ is $0$ and the kernel is trivial. So the map is injective. 

If we now consider $\ell'$-tuple $(A_1, A_2, \dots, A_{\ell'})$ of matrices such that $A_i \in {\mathfrak {gl}}_{n_i}$  for $i = 1, 2, \dots, (\ell' - 1)$ and $A_{\ell'}$ belongs to ${\mathfrak {gl}}_{n_{\ell'}}$
 in the case where $\# \Lambda({\mathbf c})$ is even and belongs to ${\mathfrak {g'}}$ in the case where $\# \Lambda({\mathbf c})$ is odd. We can define  the linear transformation $X$ by asking $[X\vert_{U_r}]_{\widehat{\mathcal B}_r }= A_r$ for $r = 1, 2, \dots, \ell'$. Then we can use the equations of remark~\ref{R:Remarks6.9} (e)  to define  $X$ on ${U_r}$ for $r = \ell' + 1, \dots, \ell$.  The equations of remark~\ref{R:Remarks6.9} (e) and the definition of ${\mathfrak g'}$ imply that $X \in {\mathfrak g}$. Because $X(U_r) \subseteq U_r$ for all $r = 0, 1, \dots, \ell$, then $X \in {\mathfrak l}_E$ and its image is the $\ell'$-tuple $(A_1, A_2, \dots, A_{\ell'})$.
 The proof when $E$ is  the nilpotent element ${}''E_{\mathbf c}$ is similar.
\end{proof}

\begin{corollary}
With the same notation as in the above proposition, we have that 
\[
\dim({\mathfrak l}_E)= \sum_{i = 1}^{(\ell' - 1)} n_i^2 + \begin{cases} n_{\ell'}^2, &\text{if $\# \Lambda({\mathbf c}) = \ell$ is even and $\ell = 2\ell'$;}\\ \\n_{\ell'}(n_{\ell'} - 1)/2, &\text{if $\#  \Lambda({\mathbf c})  = \ell$ is odd, $\ell = (2\ell' - 1)$ and}\\ &\text{the bilinear form $\langle\  ,\   \rangle$ is symmetric;} \\ \\n_{\ell'}(n_{\ell'} + 1)/2, &\text{if $\#  \Lambda({\mathbf c}) = \ell$ is odd, $\ell = (2\ell' - 1)$ and}\\ &\text{the bilinear form $\langle\  ,\   \rangle$ is skew-symmetric.}\end{cases}
\]
\end{corollary}
\begin{proof}
This follows easily from the above isomorphism and the dimension of each Lie algebras.
\end{proof}

\begin{corollary}\label{C:LeviSubgpAssociatedNilpotent}
With the same notation as in the above proposition, we have that  the Levi subgroup $L_E$ is isomorphic to
\[
 \begin{cases}GL_{n_1}({\mathbf k}) \times GL_{n_2}({\mathbf k}) \times \dots \times GL_{n_{\ell'}}({\mathbf k})&\text{if $\#  \Lambda({\mathbf c}) = \ell$ is even and $\ell = 2\ell'$;}\\ \\GL_{n_1}({\mathbf k})  \times \dots \times GL_{n_{\ell' - 1}}({\mathbf k}) \times SO_{n_{\ell'}}({\mathbf k})&\text{if $\#  \Lambda({\mathbf c})  = \ell$ is odd, $\ell = (2\ell' - 1)$ and}\\ &\text{the bilinear form $\langle\  ,\   \rangle$ is symmetric;} \\ \\GL_{n_1}({\mathbf k})  \times \dots \times GL_{n_{\ell' - 1}}({\mathbf k}) \times Sp_{n_{\ell'}}({\mathbf k}) &\text{if $\#  \Lambda({\mathbf c})  = \ell$ is odd, $\ell = (2\ell' - 1)$ and}\\ &\text{the bilinear form $\langle\  ,\   \rangle$ is skew-symmetric.}\end{cases}
\]
Note that in this last case: $\#  \Lambda({\mathbf c})  = \ell$ is odd, $\ell = (2\ell' - 1)$ and the bilinear form $\langle\  ,\   \rangle$ is skew-symmetric, it is easy to check that $n_{\ell'}$ is even.
\end{corollary}
\begin{proof}
This follows easily from the above isomorphism of Lie algebras.
\end{proof}

\section{Irreducible   $G^{\iota}$-equivariant local systems on $G^{\iota}$-orbits in ${\mathfrak g}_2$.}

\subsection{}
Let $m$, $G$, ${\mathfrak g}$, $V$, ${\mathcal I}$, $\oplus_{i \in {\mathcal I}} V_i$, $\delta = (\delta_i)_{i \in {\mathcal I}}$,  ${\mathcal B}$, $\iota$, ${\mathbf B}$ and ${\mathbf B}/\tau$ as in either \ref{S:SetUpAeven} or  \ref{S:SetUpOdd} or \ref{S:SetUpOddOrtho}. More precisely $G$ should be a connected  reductive group. Thus in the situation of  \ref{S:SetUpOdd}, we restrict ourself only to the group $Sp(V)$, while in the situation of  \ref{S:SetUpOddOrtho}, $G$ is to the group $SO(V)$. In the situation of \ref{S:SetUpAeven} , $G$ is already connected.

\subsection{}
Let ${\mathbf c}:{\mathbf B}/\tau \rightarrow {\mathbb N}$ be a coefficient function in ${\mathfrak C}_{\delta}$ and the associated symmetric function $c:{\mathbf B} \rightarrow {\mathbb N}$. Given a $G^{\iota}$-orbit ${\mathcal O}_{\mathbf c}$ (respectively ${}'{\mathcal O}_{\mathbf c}$ or ${}''{\mathcal O}_{\mathbf c}$) in ${\mathfrak g}_2$, we want to enumerate the irreducible   $G^{\iota}$-equivariant local systems on this $G^{\iota}$-orbit. For this,  we will use the notion of symbols as defined and noted in sections 11 to 13 of Lusztig's article \cite{L1984}. We won't recall these notions, but we will use them as in Lusztig's article.  

We have constructed in \ref{N:CCorrespondance}, \ref{N:CCorrespondanceOdd} and \ref{N:CCorrespondanceSpecialOrtho} for such an $G^{\iota}$-orbit:  a Jacobson-Morozov  triple $(E_{\mathbf c}, H_{\mathbf c}, F_{\mathbf c})$ (respectively $({\ }'E_{\mathbf c}, {\ }'H_{\mathbf c}, {\ }'F_{\mathbf c})$, $({\ }''E_{\mathbf c}, {\ }''H_{\mathbf c}, {\ }''F_{\mathbf c})$). In proposition~\ref{P:LieSubgroupDescription} and corollary~\ref{C:LeviSubgpAssociatedNilpotent},  we have described the Levi subalgebra ${\mathfrak l}_E$  and its corresponding Levi subgroup $L_E$ for the element $E = E_{\mathbf c}$ (respectively $E = {\ }'E_{\mathbf c}$, $E = {\ }''E_{\mathbf c}$).

\begin{remark}\label{R:CentralizerOverConnectedCentralizer}
Denote by $L_{E}$: the Levi subgroup associated to the nilpotent element $E$ as denoted in the previous section. From its definition, we have that $(L_{E}, \iota)$ is $2$-rigid and by proposition 4.2 in \cite{L1995}, we have that 
\[
Stab_{L^{\iota}_{E}}(E) / Stab_{L^{\iota}_{E}}^0(E) \rightarrow Stab_{L_{E}}(E) / Stab_{L_{E}}^0(E)
\]
is an isomorphism. By proposition 5.8 in \cite{L1995}, we also have that
\[
Stab_{L^{\iota}_{E}}(E) / Stab_{L^{\iota}_{E}}^0(E) \rightarrow Stab_{G^{\iota}}(E) / Stab_{G^{\iota}}^0(E)
\]
is an isomorphism.  So to enumerate the irreducible  $G^{\iota}$-equivariant local systems on $G^{\iota}$-orbit ${\mathcal O}_{\mathbf c}$ (respectively ${}'{\mathcal O}_{\mathbf c}$ or ${}''{\mathcal O}_{\mathbf c}$) in ${\mathfrak g}_2$, it is enough to only consider $L_{E}$ rather that $G^{\iota}$. Because of our description of $L_{E}$ in corollary~\ref{C:LeviSubgpAssociatedNilpotent} above and the results in section 11 to 13 of \cite{L1984}, then the combinatorics is known. 

\end{remark}

\subsection{}\label{SS:LocalSystemNumber}
For $GL_N({\mathbf k})$,  ($N \geq 1$),  it is known  that the unipotent classes  are in one-to-one correspondance with the partition of  $N = 1 \cdot i_1 + 2 \cdot i_2 + 3 \cdot i_3 + \dots$, where $i_1, i_2, i_3, \dots \geq 0$. For a unipotent element $u \in GL_N({\mathbf k})$ in the unipotent class ${\mathcal Cl}(u)$, there is one and only one irreducible $GL_N({\mathbf k})$-equivariant local system (up to isomorphism) on ${\mathcal Cl}(u)$, because it is easy to verify that the centralizer $Z_{GL_N({\mathbf k})}(u)$ of $u$ is connected. 

For $Sp_{2n}({\mathbf k})$ with $char({\mathbf k}) \ne 2$, it is known that the unipotent classes are in one-to-one correspondance with the partition of $2n = 1 \cdot i_1 + 2 \cdot i_2 + 3 \cdot i_3 + \dots$, where $i_1, i_2, i_3, \dots \geq 0$ and $i_1, i_3, i_5, \dots $ are even. ($i_a$ is the number of Jordan blocks of size $a$ of a unipotent element and, in the previous section, we have written this partition multiplicatively as $1^{i_1} 2^{i_2} 3^{i_3} \dots$.) For a unipotent element $u \in Sp_{2n}({\mathbf k})$ in the unipotent class ${\mathcal Cl}(u)$ corresponding to the partition $2n = 1 \cdot i_1 + 2 \cdot i_2 + 3 \cdot i_3 + \dots$, then the number of irreducible $Sp_{2n}({\mathbf k})$-equivariant local systems (up to isomorphism) on ${\mathcal Cl}(u)$ is $2^{\vert\{\text{ $a$  even } \mid\  i_a \geq 1\}\vert}$. See 10.4 in \cite{L1984}. 

 For $SO_N({\mathbf k})$, ($N \geq 1$) with $char({\mathbf k}) \ne 2$, it is known that the unipotent classes correspond to partitions of $N = 1 \cdot i_1 + 2 \cdot i_2 + 3 \cdot i_3 + \dots$, where $i_1, i_2, i_3, \dots \geq 0$ and $i_2, i_4, i_6, \dots $ are even. ($i_a$ is the number of Jordan blocks of size $a$ of a unipotent element and, in the previous section, we have written this partition multiplicatively as $1^{i_1} 2^{i_2} 3^{i_3} \dots$.) This correspondance is one-to-one except that there are two unipotent classes (said to be degenerate) corresponding to partitions such that $i_1 = i_3 = i_5 = \dots = 0$, i.e. the corresponding partition is totally even. For a non-degenerate unipotent element $u \in SO_N({\mathbf k})$ in the unipotent class ${\mathcal Cl}(u)$ corresponding to the partition $N = 1 \cdot i_1 + 2 \cdot i_2 + 3 \cdot i_3 + \dots$, then the number of irreducible $SO_N({\mathbf k})$-equivariant local systems (up to isomorphism) on ${\mathcal Cl}(u)$ is $2^{(\vert \{\text{ $a$  odd } \mid\  i_a \geq 1\}\vert - 1)}$. For a degenerate unipotent element $u \in SO_N({\mathbf k})$ in the unipotent class ${\mathcal Cl}(u)$ corresponding to the partition $N = 1 \cdot i_1 + 2 \cdot i_2 + 3 \cdot i_3 + \dots$, then there is one and only one irreducible $SO_N({\mathbf k})$-equivariant local system (up to isomorphism) on ${\mathcal Cl}(u)$. See 10.6 in \cite{L1984}. 
 
 \begin{proposition}\label{P:LocalSystemCardinality}
 Let ${\mathbf c}:{\mathbf B}/\tau \rightarrow {\mathbb N}$ be a coefficient function in ${\mathfrak C}_{\delta}$ and the associated symmetric function $c:{\mathbf B} \rightarrow {\mathbb N}$. 
 \begin{enumerate}[\upshape (a)]
 \item If $0 \not\in \Lambda({\mathbf c})$, then there exits only one and only one  irreducible $G^{\iota}$-equivariant local system (up to isomorphism) on the $G^{\iota}$-orbit ${\mathcal O}_{\mathbf c}$  (respectively ${}'{\mathcal O}_{\mathbf c}$ or ${}''{\mathcal O}_{\mathbf c}$) in ${\mathfrak g}_2$. This local system is the constant  local system.
 \item If $0 \in \Lambda({\mathbf c})$ and either ($\vert {\mathcal I} \vert$ is even and $\epsilon = 1$) or ($\vert {\mathcal I} \vert$ is odd and $\epsilon = -1$), then there exits only one and only one  irreducible $G^{\iota}$-equivariant local system (up to isomorphism) on the $G^{\iota}$-orbit ${\mathcal O}_{\mathbf c}$  in ${\mathfrak g}_2$. This local system is the constant  local system.
 \item If $0 \in \Lambda({\mathbf c})$,  $\vert {\mathcal I} \vert$ is even and $\epsilon = -1$, then there are   
 \[
 2^{\vert \{b\  \in\   {\mathbf B}\  \mid \  \lambda(b) = 0,\   c(b)   > 0 \}\vert}
 \]
 irreducible $G^{\iota}$-equivariant local systems (up to isomorphism) on the $G^{\iota}$-orbit ${\mathcal O}_{\mathbf c}$ in ${\mathfrak g}_2$.  
  \item If $0 \in \Lambda({\mathbf c})$,  $\vert {\mathcal I} \vert$ is odd and $\epsilon = 1$, then there are   
 \[
 2^{(\vert \{b\  \in \  {\mathbf B}\  \mid \  \lambda(b) = 0,\   c(b) > 0 \}\vert - 1)}
 \]
 irreducible $G^{\iota}$-equivariant local systems (up to isomorphism) on $G^{\iota}$-orbit ${\mathcal O}_{\mathbf c}$ in ${\mathfrak g}_2$.  
 \end{enumerate}
 \end{proposition}
 \begin{proof}
 By remark~\ref{R:CentralizerOverConnectedCentralizer}, we get that the number of irreducible $G^{\iota}$-equivariant local systems (up to isomorphism) on the $G^{\iota}$-orbit ${\mathcal O}_{\mathbf c}$ (respectively ${}'{\mathcal O}_{\mathbf c}$ or ${}''{\mathcal O}_{\mathbf c}$) in ${\mathfrak g}_2$ is 
 \[
 \vert Stab_{G^{\iota}}(E)/Stab^0_{G^{\iota}}(E)\vert =  \vert Stab_{L_E}(E)/Stab^0_{L_E}(E)\vert 
 \]
 We have described $L_E$ in corollary~\ref{C:LeviSubgpAssociatedNilpotent}.  In each case of this corollary, $L_E$ is a product of connected reductive groups and using external tensor product, we get the local systems from the ones on each factor. The first $(\ell' - 1)$ factors in this product  are isomorphic to a general linear group and from~\ref{SS:LocalSystemNumber} there is one and only one irreducible $L_E\vert_{U_r}$-equivariant local system (up to isomorphism) on the class of the restriction of $E$ to $U_r$, where $1 \leq r \leq (\ell' - 1)$. So to compute the number of irreducible $G^{\iota}$-equivariant local systems (up to isomorphism), we just have to consider the last factor in the product of $L_E$.
 
 (a) If $0 \not\in \Lambda({\mathbf c})$, then $\vert \Lambda({\mathbf c})\vert = \ell$ is even as note remark~\ref{R:Remarks6.9} (a) and the last factor in the product of $L_E$ is $GL_{n_{\ell'}}({\mathbf k})$ as we saw in corollary~\ref{C:LeviSubgpAssociatedNilpotent}. Here $\ell = 2\ell'$. In this case we have to consider a unipotent element $u$ in the  unipotent class corresponding to the partition 
 \[
\prod_{\begin{subarray}{c} b \in {\mathbf B}\\ \lambda(b) = \lambda_{\ell'} \end{subarray}}  \vert Supp(b) \vert^{c(b)}.
\]
 From~\ref{SS:LocalSystemNumber}, we get that there is one and only  irreducible $GL_N({\mathbf k})$-equivariant local system (up to isomorphism) on ${\mathcal Cl}(u)$. Thus there is one and only one irreducible $G^{\iota}$-equivariant local system (up to isomorphism) on the orbit ${\mathcal O}_{\mathbf c}$  (respectively ${}'{\mathcal O}_{\mathbf c}$ or ${}''{\mathcal O}_{\mathbf c}$) in ${\mathfrak g}_2$. This local system is the constant  local system.
 
 (b) If $0 \in \Lambda({\mathbf c})$,  $\vert {\mathcal I} \vert$ is even and $\epsilon = 1$, then $\vert \Lambda({\mathbf c}) \vert= \ell = (2\ell' - 1)$ is odd and and the last factor in the product of $L_E$ is $SO_{n_{\ell'}}({\mathbf k})$ as we saw in corollary~\ref{C:LeviSubgpAssociatedNilpotent}. In this case we have to consider a unipotent element $u$ in the  unipotent class corresponding to the partition 
 \[
\prod_{\begin{subarray}{c} b \in {\mathbf B}\\ \lambda(b) = \lambda_{\ell'} \end{subarray}}  \vert Supp(b) \vert^{c(b)}.
\]
Because $\vert {\mathcal I} \vert$ is even, then $\vert Supp(b) \vert$ is even for all ${\mathcal I}$-boxes $b$ on the principal diagonal. Note also if $b \in {\mathbf B}$ is on the principal diagonal, then $\tau(b)$ is also on the principal diagonal and the $\langle \tau \rangle$-orbit of $b$ is $\{b, \tau(b)\}$, where $\tau(b) \ne b$. Also $c(b) = c(\tau(b))$. So $u$ is a degenerate unipotent element of $L_E$ and then there is one and only one irreducible $SO_{n_{\ell'}}(k)$-equivariant local system (up to isomorphism) on ${\mathcal Cl}(u)$. Thus there is one and only one irreducible $G^{\iota}$-equivariant local system (up to isomorphism) on the orbit ${\mathcal O}_{\mathbf c}$  in ${\mathfrak g}_2$. 

If $0 \in \Lambda({\mathbf c})$,  $\vert {\mathcal I} \vert$ is odd and $\epsilon = -1$, then  $\vert \Lambda({\mathbf c}) \vert = \ell = (2\ell' - 1)$ is odd and and the last factor in the product of $L_E$ is $Sp_{n_{\ell'}}(k)$ as we saw in corollary~\ref{C:LeviSubgpAssociatedNilpotent}. In this case we have to consider a unipotent element $u$ in the  unipotent class corresponding to the partition 
 \[
\prod_{\begin{subarray}{c} b \in {\mathbf B}\\ \lambda(b) = \lambda_{\ell'} \end{subarray}}  \vert Supp(b) \vert^{c(b)}.
\]
Because $\vert {\mathcal I} \vert$ is odd, then $\vert Supp(b) \vert$ is odd for all ${\mathcal I}$-boxes $b$ on the principal diagonal.  Note also if $b \in {\mathbf B}$ is on the principal diagonal, then $\tau(b)$ is also on the principal diagonal and the $\langle \tau \rangle$-orbit of $b$ is $\{b, \tau(b)\}$, where $\tau(b) \ne b$. Also $c(b) = c(\tau(b))$. So in this case, there are no even part in the partition for the class ${\mathcal Cl}(u)$. Consequently from~\ref{SS:LocalSystemNumber}, there  is one and only one irreducible $Sp_{n_{\ell'}}(k)$-equivariant local system (up to isomorphism) on ${\mathcal Cl}(u)$. Thus there is one and only one irreducible $G^{\iota}$-equivariant local system (up to isomorphism) on the orbit ${\mathcal O}_{\mathbf c}$  in ${\mathfrak g}_2$. 

In case either ($\vert {\mathcal I} \vert$ is even and $\epsilon = 1$) or ($\vert {\mathcal I} \vert$ is odd and $\epsilon = -1$), then the unique irreducible $G^{\iota}$-equivariant local system (up to isomorphism) on the orbit ${\mathcal O}_{\mathbf c}$  is the constant  local system.

(c) If $0 \in \Lambda(c)$,  $\vert {\mathcal I} \vert$ is even and $\epsilon = -1$, then $\vert \Lambda({\mathbf c}) \vert = \ell = (2\ell' - 1)$ is odd and and the last factor in the product of $L_E$ is $Sp_{n_{\ell'}}(k)$ as we saw in corollary~\ref{C:LeviSubgpAssociatedNilpotent}. In this case we have to consider a unipotent element $u$ in the  unipotent class corresponding to the partition 
 \[
\prod_{\begin{subarray}{c} b \in {\mathbf B}\\ \lambda(b) = \lambda_{\ell'} \end{subarray}}  \vert Supp(b) \vert^{c(b)}.
\]
Because $\vert {\mathcal I} \vert$ is even, then $\vert Supp(b) \vert$ is even for all ${\mathcal I}$-boxes $b$ on the principal diagonal. Note also if $b \in {\mathbf B}$ is on the principal diagonal, then $\tau(b) = b$ and the $\langle \tau \rangle$-orbit of $b$ is $\{b\}$.  So all the parts of the partition corresponding to the class ${\mathcal Cl}(u)$ are even. From~\ref{SS:LocalSystemNumber}, there are 
 \[
 2^{\vert \{b\  \in\   {\mathbf B}\  \mid \  \lambda(b) = 0,\   c(b) > 0 \} \vert}
 \]
 irreducible $L_E$-equivariant local systems (up to isomorphism) on ${\mathcal Cl}(u)$. Consequently there are 
 \[
 2^{\vert \{b\  \in\   {\mathbf B}\  \mid \  \lambda(b) = 0,\   c(b) > 0 \} \vert}
 \]
 irreducible $G^{\iota}$-equivariant local systems (up to isomorphism) on the orbit ${\mathcal O}_{\mathbf c}$  in ${\mathfrak g}_2$. 

(d) If  $0 \in \Lambda(c)$,  $\vert {\mathcal I} \vert$ is odd and $\epsilon = 1$,  then $\vert \Lambda({\mathbf c}) \vert = \ell = (2\ell' - 1)$ is odd and and the last factor in the product of $L_E$ is $SO_{n_{\ell'}}(k)$ as we saw in corollary~\ref{C:LeviSubgpAssociatedNilpotent}. In this case we have to consider a unipotent element $u$ in the  unipotent class corresponding to the partition 
 \[
\prod_{\begin{subarray}{c} b \in {\mathbf B}\\ \lambda(b) = \lambda_{\ell'} \end{subarray}}  \vert Supp(b) \vert^{c(b)}.
\]
Because $\vert {\mathcal I} \vert$ is odd, then $\vert Supp(b) \vert$ is odd for all ${\mathcal I}$-boxes $b$ on the principal diagonal. Note also if $b \in {\mathbf B}$ is on the principal diagonal, then $\tau(b) = b$ and the $\langle \tau \rangle$-orbit of $b$ is $\{b\}$.  So all the parts of the partition corresponding to the class ${\mathcal Cl}(u)$ are odd. $u$ is a non-degenerate unipotent element. From~\ref{SS:LocalSystemNumber}, there are 
 \[
 2^{\vert \{b\  \in\   {\mathbf B}\  \mid \  \lambda(b) = 0,\   c(b) > 0 \}\vert  - 1}
 \]
irreducible $L_E$-equivariant local systems (up to isomorphism) on ${\mathcal Cl}(u)$. Consequently there are 
 \[
 2^{\vert \{b\  \in\   {\mathbf B}\  \mid \  \lambda(b) = 0,\   c(b) > 0 \}\vert - 1}
 \]
 irreducible $G^{\iota}$-equivariant local systems (up to isomorphism) on the orbit ${\mathcal O}_{\mathbf c}$  in ${\mathfrak g}_2$. 
 
 \end{proof}

\begin{example}
We will consider the orbits in example~\ref{E:4.23}. Recall that $G = Sp_6({\mathbf k})$ and $\iota(t) = \text{diag}(t^3, t, t, t^{-3}, t^{-1}, t^{-1})$. There are eight orbits. We will enumerate for each of these orbits ${\mathcal O}_{\mathbf c}$: the coefficient function  ${\bf c}$,   the tableau ${\mathcal T} \in {\mathfrak T}_{\delta}$, the set $\Lambda({\mathbf c})$, the integers $\ell$, $\ell'$, the dimensions: $n_1$, $n_2$, ... and $n_{\ell'}$, $L_{E_{\mathbf c}}$ (up to isomorphism) as in corollary~\ref{C:LeviSubgpAssociatedNilpotent}, the partition corresponding to the restriction of $E_{\mathbf c}$ on each of the factors of ${\mathfrak l}_{E_{\mathbf c}}$ given in corollary~\ref{P:LieSubgroupDescription}, the number of  irreducible   $G^{\iota}$-equivariant local systems (up to isomorphism) on the $G^{\iota}$-orbit  ${\mathcal O}_{\mathbf c}$ using proposition~\ref{P:LocalSystemCardinality} and  a list of irreducible   $G^{\iota}$-equivariant local systems on the $G^{\iota}$-orbit  ${\mathcal O}_{\mathbf c}$ using Lusztig's symbols as described in  \cite{L1984}, when there is more than one such irreducible   $G^{\iota}$-equivariant local system on the orbit. In this last case  when we list the  irreducible   $G^{\iota}$-equivariant local systems on the $G^{\iota}$-orbit, we will enumerate the symbols corresponding to the last factor of $L_{E_{\mathbf c}}$ that is a symplectic group. Note that,  for all the other factors,  they  are general linear groups and the local systems are trivial.

The table we gave in example~\ref{E:4.23} is below. 
\begin{table}[h]
	\begin{center}\renewcommand{\arraystretch}{1.25}
		\begin{tabular} {| l || r  | r  | r  | r  | r |  r | c |}
			\hline
			Orbit &  $10$  & $01_0$ & $01_1$ & $11_0$ & $11_1$ & $12$ & Jordan Decomposition \\ \hline
			${\mathcal O}_5^1$ & $0$ & $0$ & $1$ & $0$ & $1$ & $0$ & $2^1 4^1$\\ \hline
			${\mathcal O}_4^1$ & $0$ & $0$ & $0$ & $0$ & $0$ & $1$ & $3^2$\\ \hline
			${\mathcal O}_4^2$ & $0$ & $1$ & $0$ & $0$ & $1$ & $0$ & $1^2 4$\\ \hline
			${\mathcal O}_3^1$ & $0$ & $0$ & $1$ & $1$ & $0$ & $0$ & $2^3$\\ \hline
			${\mathcal O}_3^2$ & $1$ & $0$ & $2$ & $0$ & $0$ & $0$ & $1^2 2^2$\\ \hline
			${\mathcal O}_2^1$ & $0$ & $1$ & $0$ & $1$ & $0$ & $0$ & $1^2 2^2$\\ \hline
			${\mathcal O}_2^2$ & $1$ & $1$ & $1$ & $0$ & $0$ & $0$ & $1^4 2^1$\\ \hline
			${\mathcal O}_0^1$ & $1$ & $2$ & $0$ & $0$ & $0$ & $0$  & $1^6$\\ \hline
		\end{tabular}
	\end{center}
	\caption{List of the values of coefficient function ${\mathbf c}$.}\label{T:Table6}
\end{table}

\begin{itemize}
\item For the orbit ${\mathcal O}_5^1$ of dimension 5, we have that ${\bf c} = (0, 0, 1, 0 , 1, 0)$

\begin{center}
${\mathcal T}_5^1$ = 
\ytableausetup{centertableaux}
\begin{ytableau}
\none & \none [-3] & \none [-1] & \none [1] & \none [3] \\ \none [{\phantom -}3] & 1 & 0 & 0 & 0  \\ \none [{\phantom -}1] & 0 & 1 & 0\\  \none [-1] & 0 & 0 \\ \none [-3] & 0\\ \end{ytableau}
\end{center}
$\Lambda({\mathbf c}) = \{0\}$, $\ell = 1$, $\ell' = 1$, $n_1 = 6$, $L_{E_{\mathbf c}} = Sp_{6}({\mathbf k})$ and the partition corresponding to the nilpotent element $E_{\mathbf c}$  on the unique factor is $2^1 4^1$.  There are four  irreducible   $G^{\iota}$-equivariant local systems (up to isomorphism)  on the $G^{\iota}$-orbit  ${\mathcal O}_5^1$ and the  corresponding symbols are
\[
(\{0, 4\}, \{2\}), \quad (\{0, 2\}, \{4\}), \quad (\{0\}, \{2, 4\}), \quad (\{0, 2, 4\}, \emptyset)
\]

\item For the orbit ${\mathcal O}_4^1$ of dimension 4, we have that ${\bf c} = (0, 0, 0, 0 , 0, 1)$

\begin{center}
${\mathcal T}_4^1$ = 
\ytableausetup{centertableaux}
\begin{ytableau}
\none & \none [-3] & \none [-1] & \none [1] & \none [3] \\ \none [{\phantom -}3] & 0 & 1 & 0 & 0  \\ \none [{\phantom -}1] & 1 & 0 & 0\\  \none [-1] & 0 & 0 \\ \none [-3] & 0\\ \end{ytableau}
\end{center}
$\Lambda({\mathbf c}) = \{-1, 1\}$, $\ell = 2$, $\ell' = 1$, $n_1 = 3$, $L_{E_{\mathbf c}} \equiv GL_{3}({\mathbf k})$ and the partition corresponding to the nilpotent element $E_{\bf c}$  on the unique factor is $3^1$.  There is a unique  $G^{\iota}$-equivariant local system (up to isomorphism)  on the $G^{\iota}$-orbit  ${\mathcal O}_4^1$: the trivial local system.

\item For the orbit ${\mathcal O}_4^2$ of dimension 4, we have that ${\bf c} = (0, 1, 0, 0 , 1, 0)$

\begin{center}
${\mathcal T}_4^2$ = 
\ytableausetup{centertableaux}
\begin{ytableau}
\none & \none [-3] & \none [-1] & \none [1] & \none [3] \\ \none [{\phantom -}3] & 1 & 0 & 0 & 0  \\ \none [{\phantom -}1] & 0 & 0 & 1\\  \none [-1] & 0 & 1 \\ \none [-3] & 0\\ \end{ytableau}
\end{center}
$\Lambda({\mathbf c}) = \{-1, 0,  1\}$, $\ell = 3$, $\ell' = 2$, $n_1 = 1$, $n_2 = 4$, $L_{E_{\mathbf c}} \equiv GL_{1}({\mathbf k}) \times Sp_4({\mathbf k})$,  the partition corresponding to the nilpotent element $E_{\mathbf c}$  on the first factor is $1^1$ and the partition corresponding to the nilpotent element $E_{\mathbf c}$  on the second factor is $4^1$.  There are two  irreducible   $G^{\iota}$-equivariant local systems (up to isomorphism)  on the $G^{\iota}$-orbit  ${\mathcal O}_4^2$: on the first factor, it is always the trivial local system  and on the second factor, the  corresponding symbols for the local systems are
\[
(\emptyset, \{2\}), \quad (\{2\}, \emptyset)
\]

\item For the orbit ${\mathcal O}_3^1$ of dimension 3, we have that ${\bf c} = (0, 0, 1, 1 , 0, 0)$

\begin{center}
${\mathcal T}_3^1$ = 
\ytableausetup{centertableaux}
\begin{ytableau}
\none & \none [-3] & \none [-1] & \none [1] & \none [3] \\ \none [{\phantom -}3] & 0 & 0 & 1 & 0  \\ \none [{\phantom -}1] & 0 & 1 & 0\\  \none [-1] & 1 & 0 \\ \none [-3] & 0\\ \end{ytableau}
\end{center}
$\Lambda({\mathbf c})  = \{-2, 0,  2\}$, $\ell = 3$, $\ell' = 2$, $n_1 = 2$, $n_2 = 2$, $L_{E_{\mathbf c}} \equiv GL_{2}({\mathbf k}) \times Sp_2({\mathbf k})$,  the partition corresponding to the nilpotent element $E_{\bf c}$  on the first factor is $2^1$ and the partition corresponding to the nilpotent element $E_{\bf c}$  on the second factor is $2^1$.  There are two  irreducible   $G^{\iota}$-equivariant local systems (up to isomorphism)  on the $G^{\iota}$-orbit  ${\mathcal O}_3^1$: on the first factor, it is always the trivial local system  and on the second factor, the  corresponding symbols for the local systems are
\[
(\emptyset, \{1\}), \quad (\{1\}, \emptyset)
\]

\item For the orbit ${\mathcal O}_3^2$ of dimension 3, we have that ${\bf c} = (1, 0, 2, 0, 0, 0)$

\begin{center}
${\mathcal T}_3^2$ = 
\ytableausetup{centertableaux}
\begin{ytableau}
\none & \none [-3] & \none [-1] & \none [1] & \none [3] \\ \none [{\phantom -}3] & 0 & 0 & 0 & 1  \\ \none [{\phantom -}1] & 0 & 2 & 0\\  \none [-1] & 0 & 0 \\ \none [-3] & 1\\ \end{ytableau}
\end{center}
$\Lambda({\mathbf c}) = \{-3, 0,  3\}$, $\ell = 3$, $\ell' = 2$, $n_1 = 1$, $n_2 = 2$, $L_{E_{\mathbf c}} \equiv GL_{1}({\mathbf k}) \times Sp_4({\mathbf k})$,  the partition corresponding to the nilpotent element $E_{\bf c}$  on the first factor is $1^1$ and the partition corresponding to the nilpotent element $E_{\bf c}$  on the second factor is $2^2$.  There are two  irreducible   $G^{\iota}$-equivariant local systems (up to isomorphism)  on the $G^{\iota}$-orbit  ${\mathcal O}_3^2$: on the first factor, it is always the trivial local system  and on the second factor, the  corresponding symbols for the local systems are
\[
(\{0, 2\}, \{3\}), \quad (\{0, 3\}, \{2\})
\]

\item For the orbit ${\mathcal O}_2^1$ of dimension 2, we have that ${\bf c} = (0, 1, 0, 1, 0, 0)$

\begin{center}
${\mathcal T}_2^1$ = 
\ytableausetup{centertableaux}
\begin{ytableau}
\none & \none [-3] & \none [-1] & \none [1] & \none [3] \\ \none [{\phantom -}3] & 0 & 0 & 1 & 0  \\ \none [{\phantom -}1] & 0 & 0 & 1\\  \none [-1] & 1 & 1 \\ \none [-3] & 0\\ \end{ytableau}
\end{center}
$\Lambda({\mathbf c}) = \{-2, -1, 1, 2\}$, $\ell = 4$, $\ell' = 2$, $n_1 = 2$, $n_2 = 1$, $L_{E_{\mathbf c}} \equiv GL_{2}({\mathbf k}) \times GL_1({\mathbf k})$, the partition corresponding to the nilpotent element $E_{\bf c}$  on the unique factor is $2^1$ and the partition corresponding to the nilpotent element $E_{\bf c}$  on the second factor is $1^1$.  There is a unique  $G^{\iota}$-equivariant local system (up to isomorphism)  on the $G^{\iota}$-orbit  ${\mathcal O}_2^1$: the trivial local system.

\item For the orbit ${\mathcal O}_2^2$ of dimension 2, we have that ${\bf c} = (1, 1, 1, 0, 0, 0)$

\begin{center}
${\mathcal T}_2^2$ = 
\ytableausetup{centertableaux}
\begin{ytableau}
\none & \none [-3] & \none [-1] & \none [1] & \none [3] \\ \none [{\phantom -}3] & 0 & 0 & 0 & 1  \\ \none [{\phantom -}1] & 0 & 1 & 1\\  \none [-1] & 0 & 1 \\ \none [-3] & 1\\ \end{ytableau}
\end{center}
$\Lambda({\mathbf c})  = \{-3, -1, 0, 1,  3\}$, $\ell = 5$, $\ell' = 3$, $n_1 = 1$, $n_2 = 1$, $n_3 = 2$, $L_{E_{\mathbf c}} \equiv GL_{1}({\mathbf k}) \times GL_{1}({\mathbf k}) \times Sp_2({\mathbf k})$,  the partition corresponding to the nilpotent element $E_{\bf c}$  on the first two factors is $1^1$ and the partition corresponding to the nilpotent element $E_{\bf c}$  on the third factor is $2^1$.  There are two  irreducible   $G^{\iota}$-equivariant local systems (up to isomorphism)  on the $G^{\iota}$-orbit  ${\mathcal O}_2^2$: on the first two factor, it is always the trivial local system  and on the third factor, the  corresponding symbols for the local systems are
\[
(\emptyset, \{1\}), \quad (\{1\}, \emptyset)
\]

\item For the orbit ${\mathcal O}_0^1$ of dimension 0, we have that ${\bf c} = (1, 2, 0, 0, 0, 0)$

\begin{center}
${\mathcal T}_0^1$ = 
\ytableausetup{centertableaux}
\begin{ytableau}
\none & \none [-3] & \none [-1] & \none [1] & \none [3] \\ \none [{\phantom -}3] & 0 & 0 & 0 & 1  \\ \none [{\phantom -}1] & 0 & 0 & 2\\  \none [-1] & 0 & 2 \\ \none [-3] & 1\\ \end{ytableau}
\end{center}
$\Lambda({\mathbf c}) = \{-3, -1, 1, 3\}$, $\ell = 4$, $\ell' = 2$, $n_1 = 1$, $n_2 = 2$, $L_{E_{\mathbf c}} \equiv GL_{1}({\mathbf k}) \times GL_2({\mathbf k})$, the partition corresponding to the nilpotent element $E_{\bf c}$  on the first factor is $1^1$ and the partition corresponding to the nilpotent element $E_{\bf c}$  on the second factor is $1^2$.  There is a unique  $G^{\iota}$-equivariant local system (up to isomorphism)  on the $G^{\iota}$-orbit  ${\mathcal O}_0^1$: the trivial local system.

\end{itemize}

\end{example}

\begin{example}
We will consider the orbits in example~\ref{E:4.24}. Recall that $G = SO_{10}({\mathbf k})$ and $\iota(t) = \text{diag}(t^3, t^3,  t, t, t, t^{-3}, t^{-3}, t^{-1},  t^{-1}, t^{-1})$. There are eight orbits. We will enumerate for each of these orbits ${\mathcal O}_{\mathbf c}$: the coefficient function  ${\bf c}$,   the tableau ${\mathcal T} \in {\mathfrak T}_{\delta}$, the set $\Lambda({\mathbf c})$, the integers $\ell$, $\ell'$, the dimensions: $n_1$, $n_2$, ... and $n_{\ell'}$, $L_{E_{\mathbf c}}$ (up to isomorphism) as in corollary~\ref{C:LeviSubgpAssociatedNilpotent}, the partition corresponding to the restriction of $E_{\mathbf c}$ on each of the factors of ${\mathfrak l}_{E_{\mathbf c}}$ given in corollary~\ref{P:LieSubgroupDescription}, the number of  irreducible   $G^{\iota}$-equivariant local systems (up to isomorphism) on the $G^{\iota}$-orbit  ${\mathcal O}_{\mathbf c}$ using proposition~\ref{P:LocalSystemCardinality} and  a list of irreducible   $G^{\iota}$-equivariant local systems on the $G^{\iota}$-orbit  ${\mathcal O}_{\mathbf c}$ using Lusztig's symbols as described in  \cite{L1984}, when there is more than one such irreducible   $G^{\iota}$-equivariant local system on the orbit. In this last case  when we list the  irreducible   $G^{\iota}$-equivariant local systems on the $G^{\iota}$-orbit, we will enumerate the symbols corresponding to the last factor of $L_{E_{\mathbf c}}$ that is a special orthogonal group. Note that,  for all the other factors,  they  are general linear groups and the local systems are trivial. In the case of $\vert {\mathcal I}\vert$ even and $\epsilon = 1$, we have to be careful about the values in the boxes of the principal diagonal for the symmetric tableau ${\mathcal T} \in {\mathfrak T}_{\delta}$, the definition of the symmetric tableau is given in definition~\ref{D:SymmetricTableauIevenOrtho}.

The table we gave in example~\ref{E:4.24} is below.

\begin{table}[h]
	\begin{center}\renewcommand{\arraystretch}{1.25}
		\begin{tabular} {| l || r  | r  | r  | r  | r |  r | c |}
			\hline
			Orbit &  $10$  & $01$ & $11$ & $02$ & $12$ & $22$ & Jordan Decomposition \\ \hline
			${\mathcal O}_9^1$ & $0$ & $1$ & $0$ & $0$ & $0$ & $1$ & $1^2 4^2$\\ \hline
			${\mathcal O}_8^1$ & $0$ & $0$ & $1$ & $0$ & $1$ & $0$ & $2^2 3^2$\\ \hline
			${\mathcal O}_7^1$ & $1$ & $1$ & $0$ & $0$ & $1$ & $0$ & $1^4 3^2$\\ \hline
			${\mathcal O}_6^1$ & $0$ & $1$ & $2$ & $0$ & $0$ & $0$ & $1^2 2^4$\\ \hline
			${\mathcal O}_5^1$ & $1$ & $0$ & $1$ & $1$ & $0$ & $0$ & $1^2 2^4$\\ \hline
			${\mathcal O}_4^1$ & $1$ & $2$ & $1$ & $0$ & $0$ & $0$ & $1^6 2^2$\\ \hline
			${\mathcal O}_3^1$ & $2$ & $1$ & $0$ & $1$ & $0$ & $0$ & $1^6 2^2$\\ \hline
			${\mathcal O}_0^1$ & $2$ & $3$ & $0$ & $0$ & $0$ & $0$  & $1^{10}$\\ \hline
		\end{tabular}
	\end{center}
	\caption{List of the values of coefficient function ${\mathbf c}$.}\label{T:Table7}\end{table}

\begin{itemize}
\item For the orbit ${\mathcal O}_9^1$ of dimension 9,  we have that ${\bf c} = (0 , 1, 0, 0, 0, 1)$

\begin{center}
${\mathcal T}_9^1$ = 
\ytableausetup{centertableaux}
\begin{ytableau}
\none & \none [-3] & \none [-1] & \none [1] & \none [3] \\ \none [{\phantom -}3] & 2 & 0 & 0 & 0  \\ \none [{\phantom -}1] & 0 & 0 & 1\\  \none [-1] & 0 & 1 \\ \none [-3] & 0\\ \end{ytableau}
\end{center}
$\Lambda({\mathbf c}) = \{-1, 0, 1\}$, $\ell = 3$, $\ell' = 2$, $n_1 = 1$, $n_2 = 8$, $L_{\bf c} \equiv GL_{1}({\mathbf k}) \times SO_{8}({\mathbf k})$,  the partition corresponding to the nilpotent element $E_{\bf c}$  on the first  factor is $1^1$ and the partition corresponding to the nilpotent element $E_{\bf c}$  on the second factor is $4^2$.  There is a unique  $G^{\iota}$-equivariant local system (up to isomorphism) on the $G^{\iota}$-orbit  ${\mathcal O}_9^1$: the trivial local system.

\item For the orbit ${\mathcal O}_8^1$ of dimension 8, we have that ${\bf c} = (0 , 0, 1, 0, 1, 0)$

\begin{center}
${\mathcal T}_8^1$ = 
\ytableausetup{centertableaux}
\begin{ytableau}
\none & \none [-3] & \none [-1] & \none [1] & \none [3] \\ \none [{\phantom -}3] & 0 & 1 & 1 & 0  \\ \none [{\phantom -}1] & 1 & 0 & 0\\  \none [-1] & 1 & 0 \\ \none [-3] & 0\\ \end{ytableau}
\end{center}
$\Lambda({\mathbf c}) = \{-2, -1, 1, 2\}$, $\ell = 4$, $\ell' = 2$, $n_1 = 2$, $n_2 = 3$, $L_{\bf c} \equiv GL_{2}({\mathbf k}) \times GL_{3}({\mathbf k})$,  the partition corresponding to the nilpotent element $E_{\bf c}$  on the first  factor is $2^1$ and the partition corresponding to the nilpotent element $E_{\bf c}$  on the second factor is $3^1$.  There is a unique  $G^{\iota}$-equivariant local system (up to isomorphism)  on the $G^{\iota}$-orbit  ${\mathcal O}_8^1$: the trivial local system.

\item For the orbit ${\mathcal O}_7^1$ of dimension 7, we have that ${\bf c} = (1 , 1, 0, 0, 1, 0)$

\begin{center}
${\mathcal T}_7^1$ = 
\ytableausetup{centertableaux}
\begin{ytableau}
\none & \none [-3] & \none [-1] & \none [1] & \none [3] \\ \none [{\phantom -}3] & 0 & 1 & 0 & 1  \\ \none [{\phantom -}1] & 1 & 0 & 1\\  \none [-1] & 0 & 1 \\ \none [-3] & 1\\ \end{ytableau}
\end{center}
$\Lambda({\mathbf c}) = \{-3, -1, 1, 3\}$, $\ell = 4$, $\ell' = 2$, $n_1 = 1$, $n_2 = 4$, $L_{\bf c} \equiv GL_{1}({\mathbf k}) \times GL_{4}({\mathbf k})$,  the partition corresponding to the nilpotent element $E_{\bf c}$  on the first  factor is $1^1$ and the partition corresponding to the nilpotent element $E_{\bf c}$  on the second factor is $1^1 3^1$.  There is a unique  $G^{\iota}$-equivariant local system (up to isomorphism)  on the $G^{\iota}$-orbit  ${\mathcal O}_7^1$: the trivial local system.

\item For the orbit ${\mathcal O}_6^1$ of dimension 6, we have that ${\bf c} = (0 , 1, 2, 0, 0, 0)$

\begin{center}
${\mathcal T}_6^1$ = 
\ytableausetup{centertableaux}
\begin{ytableau}
\none & \none [-3] & \none [-1] & \none [1] & \none [3] \\ \none [{\phantom -}3] & 0 & 0 & 2 & 0  \\ \none [{\phantom -}1] & 0 & 0 & 1\\  \none [-1] & 2 & 1 \\ \none [-3] & 0\\ \end{ytableau}
\end{center}
$\Lambda({\mathbf c}) = \{-2, -1, 1, 2\}$, $\ell = 4$, $\ell' = 2$, $n_1 = 4$, $n_2 = 1$, $L_{\bf c} \equiv GL_{4}({\mathbf k}) \times GL_{1}({\mathbf k})$,  the partition corresponding to the nilpotent element $E_{\bf c}$  on the first  factor is $2^2$ and the partition corresponding to the nilpotent element $E_{\bf c}$  on the second factor is $1^1$.  There is a unique  $G^{\iota}$-equivariant local system (up to isomorphism)  on the $G^{\iota}$-orbit  ${\mathcal O}_6^1$: the trivial local system.

\item For the orbit ${\mathcal O}_5^1$ of dimension 5, we have that ${\bf c} = (0 , 1, 2, 0, 0, 0)$

\begin{center}
${\mathcal T}_5^1$ = 
\ytableausetup{centertableaux}
\begin{ytableau}
\none & \none [-3] & \none [-1] & \none [1] & \none [3] \\ \none [{\phantom -}3] & 0 & 0 & 1 & 1  \\ \none [{\phantom -}1] & 0 & 2 & 0\\  \none [-1] & 1 & 0 \\ \none [-3] & 1\\ \end{ytableau}
\end{center}
$\Lambda({\mathbf c}) = \{-3, -2, 0,  2, 3\}$, $\ell = 5$, $\ell' = 3$, $n_1 = 1$, $n_2 = 2$, $n_3 = 4$, $L_{\bf c} \equiv GL_{1}({\mathbf k}) \times GL_{2}({\mathbf k}) \times SO_4({\mathbf k})$,  the partition corresponding to the nilpotent element $E_{\bf c}$  on the first  factor is $1^1$, the partition corresponding to the nilpotent element $E_{\bf c}$  on the second factor is $2^1$ and the partition corresponding to the nilpotent element $E_{\bf c}$  on the third factor is $2^2$ .  There is a unique  $G^{\iota}$-equivariant local system (up to isomorphism)  on the $G^{\iota}$-orbit  ${\mathcal O}_5^1$: the trivial local system.

\item For the orbit ${\mathcal O}_4^1$ of dimension 4, we have that ${\bf c} = (1 , 2, 1, 0, 0, 0)$

\begin{center}
${\mathcal T}_4^1$ = 
\ytableausetup{centertableaux}
\begin{ytableau}
\none & \none [-3] & \none [-1] & \none [1] & \none [3] \\ \none [{\phantom -}3] & 0 & 0 & 1 & 1  \\ \none [{\phantom -}1] & 0 & 0 & 2\\  \none [-1] & 1 & 2 \\ \none [-3] & 1\\ \end{ytableau}
\end{center}
$\Lambda({\mathbf c}) = \{-3, -2, -1, 1,  2, 3\}$, $\ell = 6$, $\ell' = 3$, $n_1 = 1$, $n_2 = 2$, $n_3 = 2$, $L_{\bf c} = GL_{1}({\mathbf k}) \times GL_{2}({\mathbf k}) \times GL_2({\mathbf k})$,  the partition corresponding to the nilpotent element $E_{\bf c}$  on the first  factor is $1^1$, the partition corresponding to the nilpotent element $E_{\bf c}$  on the second factor is $2^1$ and the partition corresponding to the nilpotent element $E_{\bf c}$  on the third factor is $1^2$ .  There is a unique  $G^{\iota}$-equivariant local system  (up to isomorphism) on the $G^{\iota}$-orbit  ${\mathcal O}_4^1$: the trivial local system.

\item For the orbit ${\mathcal O}_3^1$ of dimension 3, we have that ${\bf c} = (2 , 1, 0, 1, 0, 0)$

\begin{center}
${\mathcal T}_3^1$ = 
\ytableausetup{centertableaux}
\begin{ytableau}
\none & \none [-3] & \none [-1] & \none [1] & \none [3] \\ \none [{\phantom -}3] & 0 & 0 & 0 & 2  \\ \none [{\phantom -}1] & 0 & 2 & 1\\  \none [-1] & 0 & 1 \\ \none [-3] & 2\\ \end{ytableau}

\end{center}
$\Lambda({\mathbf c}) = \{-3, -1, 0,  1,  3\}$, $\ell = 5$, $\ell' = 3$, $n_1 = 2$, $n_2 = 1$, $n_3 = 4$, $L_{\bf c} \equiv GL_{2}({\mathbf k}) \times GL_{1}({\mathbf k}) \times SO_4({\mathbf k})$,  the partition corresponding to the nilpotent element $E_{\bf c}$  on the first  factor is $1^2$, the partition corresponding to the nilpotent element $E_{\bf c}$  on the second factor is $1^1$ and the partition corresponding to the nilpotent element $E_{\bf c}$  on the third factor is $2^2$ .  There is a unique  $G^{\iota}$-equivariant local system (up to isomorphism)  on the $G^{\iota}$-orbit  ${\mathcal O}_3^1$: the trivial local system.

\item For the orbit ${\mathcal O}_0^1$ of dimension 0, we have that ${\bf c} = (2 , 3, 0, 0, 0, 0)$

\begin{center}
${\mathcal T}_0^1$ = 
\ytableausetup{centertableaux}
\begin{ytableau}
\none & \none [-3] & \none [-1] & \none [1] & \none [3] \\ \none [{\phantom -}3] & 0 & 0 & 0 & 2  \\ \none [{\phantom -}1] & 0 & 0 & 3\\  \none [-1] & 0 & 3 \\ \none [-3] & 2\\ \end{ytableau}
\end{center}
$\Lambda({\mathbf c}) = \{-3, -1,  1,  3\}$, $\ell = 4$, $\ell' = 2$, $n_1 = 2$, $n_2 = 3$, $L_{\bf c} \equiv GL_{2}({\mathbf k}) \times GL_{3}({\mathbf k})$,  the partition corresponding to the nilpotent element $E_{\bf c}$  on the first  factor is $1^2$ and the partition corresponding to the nilpotent element $E_{\bf c}$  on the second factor is $1^3$. There is a unique  $G^{\iota}$-equivariant local system (up to isomorphism)  on the $G^{\iota}$-orbit  ${\mathcal O}_0^1$: the trivial local system.
\end{itemize}
\end{example}

\begin{example}
We will consider the orbits in example~\ref{E:5.28}. Recall that $G = Sp_{12}({\mathbf k})$ and $\iota(t) = \text{diag}(t^4, t^4,  t^2, t^2, 1, 1, t^{-4}, t^{-4}, t^{-2},  t^{-2}, 1, 1)$. There are thirteen orbits. We will enumerate for each of these orbits ${\mathcal O}_{\mathbf c}$: the coefficient function  ${\bf c}$,   the tableau ${\mathcal T} \in {\mathfrak T}_{\delta}$, the set $\Lambda({\mathbf c})$, the integers $\ell$, $\ell'$, the dimensions: $n_1$, $n_2$, ... and $n_{\ell'}$, $L_{E_{\mathbf c}}$ (up to isomorphism) as in corollary~\ref{C:LeviSubgpAssociatedNilpotent}, the partition corresponding to the restriction of $E_{\mathbf c}$ on each of the factors of ${\mathfrak l}_{E_{\mathbf c}}$ given in corollary~\ref{P:LieSubgroupDescription}, the number of  irreducible   $G^{\iota}$-equivariant local systems (up to isomorphism) on the $G^{\iota}$-orbit  ${\mathcal O}_{\mathbf c}$ using proposition~\ref{P:LocalSystemCardinality} and  a list of irreducible   $G^{\iota}$-equivariant local systems on the $G^{\iota}$-orbit  ${\mathcal O}_{\mathbf c}$ using Lusztig's symbols as described in  \cite{L1984}, when there is more than one such irreducible   $G^{\iota}$-equivariant local system on the orbit. In this last case  when we list the  irreducible   $G^{\iota}$-equivariant local systems on the $G^{\iota}$-orbit, we will enumerate the symbols corresponding to the last factor of $L_{E_{\mathbf c}}$ that is a symplectic group. Note that,  for all the other factors,  they  are general linear groups and the local systems are trivial. In the case of $\vert {\mathcal I}\vert$ odd and $\epsilon = -1$, we have to be careful about the values in the boxes of the principal diagonal for the symmetric tableau ${\mathcal T} \in {\mathfrak T}_{\delta}$, the definition of the symmetric tableau is given in definition~\ref{D:SymmetricTableauOddSymplectic}

The table we gave in example~\ref{E:5.28} is below.

\begin{table}[h]
	\begin{center}\renewcommand{\arraystretch}{1.25}
		\begin{tabular} {| l || r  | r  | r  | r  | r |  r | r | r  | r | c |}
			\hline
			Orbit &  $100$  & $010$ & $002$ & $110$ & $022$ & $012$ & $112$ & $122$ & $222$ & Jordan Dec.  \\ \hline
			${\mathcal O}_{12}^1$ & $0$ & $0$ & $1$ & $0$ & $0$ & $0$ & $0$ &  $0$ & $1$ &$1^2 5^2$\\ \hline
			${\mathcal O}_{11}^1$ & $1$ & $0$ & $1$ & $0$ & $0$ & $0$ & $0$  & $1$ & $0$ &$1^4 4^2$\\ \hline
			${\mathcal O}_{11}^2$ & $0$ & $0$ & $0$ & $0$ & $0$ & $0$ & $2$  & $0$ & $0$ &$3^4$\\ \hline
			${\mathcal O}_{10}^1$ & $1$ & $0$ & $0$ & $0$ & $0$ & $1$ & $1$  & $0$ & $0$ &$1^2 2^2 3^2$\\ \hline
			${\mathcal O}_{9}^1$ & $0$ & $0$ & $1$ & $1$ & $0$ & $0$ & $1$  & $0$ & $0$ &$1^2 2^2 3^2$\\ \hline
			${\mathcal O}_{8}^1$ & $2$ & $0$ & $1$ & $0$ & $1$ & $0$ & $0$  & $0$ & $0$ &$1^6 3^2$\\ \hline
			${\mathcal O}_{8}^2$ & $1$ & $1$ & $1$ & $0$ & $0$ & $0$ & $1$  & $0$ & $0$ &$1^6 3^2$\\ \hline
			${\mathcal O}_{7}^1$ & $2$ & $0$ & $0$ & $0$ & $0$ & $2$ & $0$  & $0$ & $0$ &$1^4 2^4$\\ \hline
			${\mathcal O}_{7}^2$ & $1$ & $0$ & $1$ & $1$ & $0$ & $1$ & $0$  & $0$ & $0$ &$1^4 2^4$\\ \hline
			${\mathcal O}_{5}^1$ & $2$ & $1$ & $1$ & $0$ & $0$ & $1$ & $0$  & $0$ & $0$ &$1^8 2^2$\\ \hline
			${\mathcal O}_{4}^1$ & $0$ & $0$ & $2$ & $2$ & $0$ & $0$ & $0$  & $0$ & $0$ &$1^4 2^4$\\ \hline
			${\mathcal O}_{3}^1$ & $1$ & $1$ & $2$ & $1$ & $0$ & $0$ & $0$  & $0$ & $0$ &$1^8 2^2$\\ \hline
			${\mathcal O}_{0}^1$ & $2$ & $2$ & $2$ & $0$ & $0$ & $0$ & $0$  & $0$ & $0$ &$1^{12}$\\ \hline
		\end{tabular}
	\end{center}
	\caption{List of the values of coefficient function ${\mathbf c}$.}\label{T:Table8}\end{table}

\begin{itemize}
\item For the orbit ${\mathcal O}_{12}^1$ of dimension 12,  we have that ${\bf c} = (0, 0, 1, 0 , 0, 0, 0, 0, 1)$

\begin{center}
${\mathcal T}_{12}^1$ = 
\ytableausetup{centertableaux}
\begin{ytableau}
\none & \none [-4] & \none [-2] & \none [0] & \none [2] & \none [4]\\ \none [{\phantom -}4] & 2 & 0 & 0 & 0 & 0\\ \none [{\phantom -}2] & 0 & 0 & 0 & 0\\  \none [{\phantom -}0] & 0 & 0 & 2\\ \none [-2] & 0 & 0 \\ \none [-4] & 0 \\ \end{ytableau}
\end{center}
$\Lambda({\mathbf c}) = \{0\}$, $\ell = 1$, $\ell' = 1$, $n_1 = 12$, $L_{E_{\mathbf c}} = Sp_{12}({\mathbf k})$,  the partition corresponding to the nilpotent element $E_{\bf c}$  on the unique factor is $1^2 5^2$. There is a unique  $G^{\iota}$-equivariant local system (up to isomorphism)  on the $G^{\iota}$-orbit  ${\mathcal O}_{12}^1$: the trivial local system. 

\item For the orbit ${\mathcal O}_{11}^1$ of dimension 11,  we have that ${\bf c} = (1, 0, 1, 0 , 0, 0, 0, 1, 0)$

\begin{center}
${\mathcal T}_{11}^1$ = 
\ytableausetup{centertableaux}
\begin{ytableau}
\none & \none [-4] & \none [-2] & \none [0] & \none [2] & \none [4]\\ \none [{\phantom -}4] & 0 & 1 & 0 & 0 & 1\\ \none [{\phantom -}2] & 1 & 0 & 0 & 0\\  \none [{\phantom -}0] & 0 & 0 & 2\\ \none [-2] & 0 & 0 \\ \none [-4] & 1 \\ \end{ytableau}
\end{center}
$\Lambda({\mathbf c}) = \{-4, -1, 0, 1, 4\}$, $\ell = 5$, $\ell' = 3$, $n_1 = 1$, $n_2 = 4$, $n_3 =2$, $L_{E_{\mathbf c}} \equiv GL_{1}({\mathbf k}) \times GL_4({\mathbf k}) \times Sp_2({\mathbf k})$,  the partition corresponding to the nilpotent element $E_{\bf c}$  on the first factor is $1^1$, the partition corresponding to the nilpotent element $E_{\bf c}$  on the second factor is $4^1$ and the partition corresponding to the nilpotent element $E_{\bf c}$  on the third factor is $1^2$. There is a unique  $G^{\iota}$-equivariant local system (up to isomorphism)  on the $G^{\iota}$-orbit  ${\mathcal O}_{11}^1$: the trivial local system.

\item For the orbit ${\mathcal O}_{11}^2$ of dimension 11,  we have that ${\bf c} = (0, 0, 0, 0 , 0, 0, 2, 0, 0)$

\begin{center}
${\mathcal T}_{11}^2$ = 
\ytableausetup{centertableaux}
\begin{ytableau}
\none & \none [-4] & \none [-2] & \none [0] & \none [2] & \none [4]\\ \none [{\phantom -}4] & 0 & 0 & 2 & 0 & 0\\ \none [{\phantom -}2] & 0 & 0 & 0 & 0\\  \none [{\phantom -}0] & 2 & 0 & 0\\ \none [-2] & 0 & 0 \\ \none [-4] & 0 \\ \end{ytableau}
\end{center}
$\Lambda({\mathbf c}) = \{-2, 2\}$, $\ell = 2$, $\ell' = 1$, $n_1 = 6$, $L_{E_{\mathbf c}} \equiv GL_{6}({\mathbf k})$,  the partition corresponding to the nilpotent element $E_{\bf c}$  on the unique factor is $3^2$. There is a unique  $G^{\iota}$-equivariant local system (up to isomorphism)  on the $G^{\iota}$-orbit  ${\mathcal O}_{11}^2$: the trivial local system.

\item For the orbit ${\mathcal O}_{10}^1$ of dimension 10,  we have that ${\bf c} = (1, 0, 0, 0 , 0, 1, 1, 0, 0)$

\begin{center}
${\mathcal T}_{10}^1$ = 
\ytableausetup{centertableaux}
\begin{ytableau}
\none & \none [-4] & \none [-2] & \none [0] & \none [2] & \none [4]\\ \none [{\phantom -}4] & 0 & 0 & 1 & 0 & 1\\ \none [{\phantom -}2] & 0 & 0 & 1 & 0\\  \none [{\phantom -}0] & 1 & 1 & 0\\ \none [-2] & 0 & 0 \\ \none [-4] & 1 \\ \end{ytableau}
\end{center}
$\Lambda({\mathbf c}) = \{-4, -2, -1, 1, 2, 4\}$, $\ell = 6$, $\ell' = 3$, $n_1 = 1$, $n_2 = 3$, $n_3 = 2$,  $L_{E_{\mathbf c}} \equiv GL_{1}({\mathbf k}) \times GL_3({\mathbf k}) \times GL_2({\mathbf k})$,  the partition corresponding to the nilpotent element $E_{\bf c}$  on the first factor is $1^1$, the partition corresponding to the nilpotent element $E_{\bf c}$  on the second factor is $3^1$ and the partition corresponding to the nilpotent element $E_{\bf c}$  on the third factor is $2^1$. There is a unique  $G^{\iota}$-equivariant local system (up to isomorphism)  on the $G^{\iota}$-orbit  ${\mathcal O}_{10}^1$: the trivial local system.

\item For the orbit ${\mathcal O}_{9}^1$ of dimension 9,   we have that ${\bf c} = (0, 0, 1, 1 , 0, 0, 1, 0, 0)$

\begin{center}
${\mathcal T}_{9}^1$ = 
\ytableausetup{centertableaux}
\begin{ytableau}
\none & \none [-4] & \none [-2] & \none [0] & \none [2] & \none [4]\\ \none [{\phantom -}4] & 0 & 0 & 1 & 1 & 0\\ \none [{\phantom -}2] & 0 & 0 & 0 & 0\\  \none [{\phantom -}0] & 1 & 0 & 2\\ \none [-2] & 1 & 0 \\ \none [-4] & 0 \\ \end{ytableau}

\end{center}
$\Lambda({\mathbf c}) = \{-3, -2,  0, 2, 3\}$, $\ell = 5$, $\ell' = 3$, $n_1 = 2$, $n_2 = 3$, $n_3 = 2$  $L_{E_{\mathbf c}} \equiv GL_{2}({\mathbf k}) \times GL_3({\mathbf k}) \times Sp_2({\mathbf k})$,  the partition corresponding to the nilpotent element $E_{\bf c}$  on the first factor is $2^1$, the partition corresponding to the nilpotent element $E_{\bf c}$  on the second factor is $3^1$ and the partition corresponding to the nilpotent element $E_{\bf c}$  on the third factor is $1^2$. There is a unique  $G^{\iota}$-equivariant local system (up to isomorphism)  on the $G^{\iota}$-orbit  ${\mathcal O}_{9}^1$: the trivial local system.

\item For the orbit ${\mathcal O}_{8}^1$ of dimension 8,   we have that ${\bf c} = (2, 0, 1, 0 , 1, 0, 0, 0, 0)$

\begin{center}
${\mathcal T}_{8}^1$ = 
\ytableausetup{centertableaux}
\begin{ytableau}
\none & \none [-4] & \none [-2] & \none [0] & \none [2] & \none [4]\\ \none [{\phantom -}4] & 0 & 0 & 0 & 0 & 2\\ \none [{\phantom -}2] & 0 & 2 & 0 & 0\\  \none [{\phantom -}0] & 0 & 0 & 2\\ \none [-2] & 0 & 0 \\ \none [-4] & 2 \\ \end{ytableau}

\end{center}
$\Lambda({\mathbf c}) = \{-4, 0, 4\}$, $\ell = 3$, $\ell' = 2$, $n_1 = 2$, $n_2 = 8$,  $L_{E_{\mathbf c}} \equiv GL_{2}({\mathbf k}) \times Sp_8({\mathbf k})$,  the partition corresponding to the nilpotent element $E_{\bf c}$  on the first factor is $1^2$ and the partition corresponding to the nilpotent element $E_{\bf c}$  on the second factor is $1^2 3^2$. There is a unique  $G^{\iota}$-equivariant local system (up to isomorphism)  on the $G^{\iota}$-orbit  ${\mathcal O}_{8}^1$: the trivial local system.

\item For the orbit ${\mathcal O}_{8}^2$ of dimension 8,   we have that ${\bf c} = (1, 1, 1, 0 , 0, 0, 1, 0, 0)$

\begin{center}
${\mathcal T}_{8}^2$ = 
\ytableausetup{centertableaux}
\begin{ytableau}
\none & \none [-4] & \none [-2] & \none [0] & \none [2] & \none [4]\\ \none [{\phantom -}4] & 0 & 0 & 1 & 0 & 1\\ \none [{\phantom -}2] & 0 & 0 & 0 & 1\\  \none [{\phantom -}0] & 1 & 0 & 2\\ \none [-2] & 0 & 1 \\ \none [-4] & 1 \\ \end{ytableau}
\end{center}
$\Lambda({\mathbf c}) = \{-4, -2, 0, 2, 4\}$, $\ell = 5$, $\ell' = 3$, $n_1 = 1$, $n_2 = 4$, $n_3 = 2$,  $L_{E_{\mathbf c}} \equiv GL_{1}({\mathbf k}) \times GL_4({\mathbf k}) \times Sp_2({\mathbf k})$,  the partition corresponding to the nilpotent element $E_{\bf c}$  on the first factor is $1$, the partition corresponding to the nilpotent element $E_{\bf c}$  on the second factor is $1^1 3^1$ and the partition corresponding to the nilpotent element $E_{\bf c}$  on the third factor is $1^2$. There is a unique  $G^{\iota}$-equivariant local system (up to isomorphism)  on the $G^{\iota}$-orbit  ${\mathcal O}_{8}^2$: the trivial local system.

\item For the orbit ${\mathcal O}_{7}^1$ of dimension 7,   we have that ${\bf c} = (2, 0, 0, 0 , 0, 2, 0, 0, 0)$

\begin{center}
${\mathcal T}_{7}^1$ = 
\ytableausetup{centertableaux}
\begin{ytableau}
\none & \none [-4] & \none [-2] & \none [0] & \none [2] & \none [4]\\ \none [{\phantom -}4] & 0 & 0 & 0 & 0 & 2\\ \none [{\phantom -}2] & 0 & 0 & 2 & 0\\  \none [{\phantom -}0] & 0 & 2 & 0\\ \none [-2] & 0 & 0 \\ \none [-4] & 2 \\ \end{ytableau}
\end{center}
$\Lambda({\mathbf c}) = \{-4, -1, 1, 4\}$, $\ell = 4$, $\ell' = 2$, $n_1 = 2$, $n_2 = 4$,  $L_{E_{\mathbf c}} \equiv GL_{2}({\mathbf k}) \times GL_4({\mathbf k})$,  the partition corresponding to the nilpotent element $E_{\bf c}$  on the first factor is $1^2$ and the partition corresponding to the nilpotent element $E_{\bf c}$  on the second factor is $2^2$. There is a unique  $G^{\iota}$-equivariant local system (up to isomorphism) on the $G^{\iota}$-orbit  ${\mathcal O}_{7}^1$: the trivial local system.

\item For the orbit ${\mathcal O}_{7}^2$ of dimension 7,   we have that ${\bf c} = (1, 0, 1, 1 , 0, 1, 0, 0, 0)$

\begin{center}
${\mathcal T}_{7}^2$ = 
\ytableausetup{centertableaux}
\begin{ytableau}
\none & \none [-4] & \none [-2] & \none [0] & \none [2] & \none [4]\\ \none [{\phantom -}4] & 0 & 0 & 0 & 1& 1\\ \none [{\phantom -}2] & 0 & 0 & 1 & 0\\  \none [{\phantom -}0] & 0 & 1 & 2\\ \none [-2] & 1 & 0 \\ \none [-4] & 1 \\ \end{ytableau}
\end{center}
$\Lambda({\mathbf c})  = \{-4, -3, -1, 0, 1, 3, 4\}$, $\ell = 7$, $\ell' = 4$, $n_1 = 1$, $n_2 = 2$, $n_3 = 2$, $n_4 = 2$,   $L_{E_{\mathbf c}} \equiv GL_{1}({\mathbf k}) \times GL_2({\mathbf k}) \times GL_2({\mathbf k}) \times Sp_2({\mathbf k})$,  the partition corresponding to the nilpotent element $E_{\bf c}$  on the first factor is $1^1$, the partition corresponding to the nilpotent element $E_{\bf c}$  on the second factor is $2^1$, the partition corresponding to the nilpotent element $E_{\bf c}$  on the third factor is $2^1$ and the partition corresponding to the nilpotent element $E_{\bf c}$  on the fourth factor is $1^2$. There is a unique  $G^{\iota}$-equivariant local system (up to isomorphism) on the $G^{\iota}$-orbit  ${\mathcal O}_{7}^2$: the trivial local system.

\item For the orbit ${\mathcal O}_{5}^1$ of dimension 5,   we have that ${\bf c} = (2, 1, 1, 0 , 0, 1, 0, 0, 0)$

\begin{center}
${\mathcal T}_{5}^1$ = 
\ytableausetup{centertableaux}
\begin{ytableau}
\none & \none [-4] & \none [-2] & \none [0] & \none [2] & \none [4]\\ \none [{\phantom -}4] & 0 & 0 & 0 & 0 & 2\\ \none [{\phantom -}2] & 0 & 0 & 1 & 1\\  \none [{\phantom -}0] & 0 & 1 & 2\\ \none [-2] & 0 & 1 \\ \none [-4] & 2 \\ \end{ytableau}
\end{center}
$\Lambda({\mathbf c})  = \{-4, -2, -1, 0, 1, 2, 4\}$, $\ell = 7$, $\ell' = 4$, $n_1 = 2$, $n_2 = 1$, $n_3 = 2$, $n_4 = 2$,   $L_{E_{\mathbf c}} \equiv GL_{2}({\mathbf k}) \times GL_1({\mathbf k}) \times GL_2({\mathbf k}) \times Sp_2({\mathbf k})$,  the partition corresponding to the nilpotent element $E_{\bf c}$  on the first factor is $1^2$, the partition corresponding to the nilpotent element $E_{\bf c}$  on the second factor is $1^1$, the partition corresponding to the nilpotent element $E_{\bf c}$  on the third factor is $2^1$ and the partition corresponding to the nilpotent element $E_{\bf c}$  on the fourth factor is $1^2$. There is a unique  $G^{\iota}$-equivariant local system (up to isomorphism) on the $G^{\iota}$-orbit  ${\mathcal O}_{5}^1$: the trivial local system.

\item For the orbit ${\mathcal O}_{4}^1$ of dimension 4,   we have that ${\bf c} = (0, 0, 2, 2 , 0, 0, 0, 0, 0)$

\begin{center}
${\mathcal T}_{4}^1$ = 
\ytableausetup{centertableaux}
\begin{ytableau}
\none & \none [-4] & \none [-2] & \none [0] & \none [2] & \none [4]\\ \none [{\phantom -}4] & 0 & 0 & 0 & 2 & 0\\ \none [{\phantom -}2] & 0 & 0 & 0 & 0\\  \none [{\phantom -}0] & 0 & 0 & 4\\ \none [-2] & 2 & 0 \\ \none [-4] & 0 \\ \end{ytableau}
\end{center}
$\Lambda({\mathbf c}) = \{-3,  0, 3\}$, $\ell = 3$, $\ell' = 2$, $n_1 = 4$, $n_2 = 4$,   $L_{E_{\mathbf c}} \equiv GL_{4}({\mathbf k}) \times  Sp_4({\mathbf k})$,  the partition corresponding to the nilpotent element $E_{\bf c}$  on the first factor is $2^2$ and the partition corresponding to the nilpotent element $E_{\bf c}$  on the second factor is $1^4$. There is a unique  $G^{\iota}$-equivariant local system (up to isomorphism)  on the $G^{\iota}$-orbit  ${\mathcal O}_{4}^1$: the trivial local system.

\item For the orbit ${\mathcal O}_{3}^1$ of dimension 3,   we have that ${\bf c} = (1, 1, 2, 1, 0, 0, 0, 0, 0)$

\begin{center}
${\mathcal T}_{3}^1$ = 
\ytableausetup{centertableaux}
\begin{ytableau}
\none & \none [-4] & \none [-2] & \none [0] & \none [2] & \none [4]\\ \none [{\phantom -}4] & 0 & 0 & 0 & 1 & 1\\ \none [{\phantom -}2] & 0 & 0 & 0 & 1\\  \none [{\phantom -}0] & 0 & 0 & 4\\ \none [-2] & 1 & 1 \\ \none [-4] & 1 \\ \end{ytableau}
\end{center}
$\Lambda({\mathbf c}) = \{-4, -3, -2,  0, 2, 3, 4\}$, $\ell = 7$, $\ell' = 4$, $n_1 = 1$, $n_2 = 2$, $n_3 = 1$, $n_4 = 4$,   $L_{E_{\mathbf c}} \equiv GL_{1}({\mathbf k}) \times GL_2({\mathbf k}) \times GL_1({\mathbf k}) \times  Sp_4({\mathbf k})$,  the partition corresponding to the nilpotent element $E_{\bf c}$  on the first factor is $1^1$,  the partition corresponding to the nilpotent element $E_{\bf c}$  on the second factor is $2^1$,  the partition corresponding to the nilpotent element $E_{\bf c}$  on the third factor is $1^1$ and the partition corresponding to the nilpotent element $E_{\bf c}$  on the fourth factor is $1^4$. There is a unique  $G^{\iota}$-equivariant local system (up to isomorphism)  on the $G^{\iota}$-orbit  ${\mathcal O}_{3}^1$: the trivial local system.

\item For the orbit ${\mathcal O}_{0}^1$ of dimension 0,   we have that ${\bf c} = (2, 2, 2, 0, 0, 0, 0, 0, 0)$

\begin{center}
${\mathcal T}_{0}^1$ = 
\ytableausetup{centertableaux}
\begin{ytableau}
\none & \none [-4] & \none [-2] & \none [0] & \none [2] & \none [4]\\ \none [{\phantom -}4] & 0 & 0 & 0 & 0 & 2\\ \none [{\phantom -}2] & 0 & 0 & 0 & 2\\  \none [{\phantom -}0] & 0 & 0 & 4\\ \none [-2] & 0 & 2 \\ \none [-4] & 2 \\ \end{ytableau}
\end{center}
$\Lambda({\mathbf c}) = \{-4, -2,  0, 2, 4\}$, $\ell = 5$, $\ell' = 3$, $n_1 = 2$, $n_2 = 2$, $n_3 = 4$,   $L_{E_{\mathbf c}} \equiv GL_{2}({\mathbf k}) \times GL_2({\mathbf k}) \times  Sp_4({\mathbf k})$,  the partition corresponding to the nilpotent element $E_{\bf c}$  on the first factor is $1^2$,  the partition corresponding to the nilpotent element $E_{\bf c}$  on the second factor is $1^2$ and the partition corresponding to the nilpotent element $E_{\bf c}$  on the third factor is $1^4$. There is a unique  $G^{\iota}$-equivariant local system (up to isomorphism)  on the $G^{\iota}$-orbit  ${\mathcal O}_{0}^1$: the trivial local system.

\end{itemize}
\end{example}

\begin{example}
We will consider the orbits in examples~\ref{E:4.27} and~\ref{E:5.7}. Recall that $G = SO_{9}({\mathbf k})$ and $\iota(t) = \text{diag}(t^4, t^2,  t^2,  1, 1,  1,  t^{-2},  t^{-2},  t^{-4})$.  Here 
\[
SO_{9}({\mathbf k}) = \{ A \in GL_{9}({\mathbf k}) \mid A^T J A = J  \text{  and  }  \det(A) = 1\}
\]
where 
\[
J = \begin{bmatrix} 0 & 0 & 0 & 0 & 0 & 0 & 0 & 0 & 1\\ 0 & 0 & 0 & 0 & 0 & 0 & 0 & 1 & 0 \\  0 & 0 & 0 & 0 & 0 & 0 & 1 & 0 & 0\\  0 & 0 & 0 & 0 & 0 & 1 & 0 & 0 & 0\\  0 & 0 & 0 & 0 & 1 & 0 & 0 & 0 & 0\\  0 & 0 & 0 & 1 & 0 & 0 & 0 & 0 & 0\\  0 & 0 & 1 & 0 & 0 & 0 & 0 & 0 & 0\\  0 & 1 & 0 & 0 & 0 & 0 & 0 & 0 & 0\\  1 & 0 & 0 & 0 & 0 & 0 & 0 & 0 & 0 \end{bmatrix}
\]
There are fourteen orbits. Note that because $\dim(V_0) = 3$ is odd, none of these orbits splits as we saw in example~\ref{E:5.7}. Thus they are all $G^{\iota}$-orbits as we have stated in example~\ref{E:5.7}.  We will enumerate for each of these orbits ${\mathcal O}_{\mathbf c}$: the coefficient function  ${\bf c}$,   the tableau ${\mathcal T} \in {\mathfrak T}_{\delta}$, the set $\Lambda({\mathbf c})$, the integers $\ell$, $\ell'$, the dimensions: $n_1$, $n_2$, ... and $n_{\ell'}$, $L_{E_{\mathbf c}}$ (up to isomorphism) as in corollary~\ref{C:LeviSubgpAssociatedNilpotent}, the partition corresponding to the restriction of $E_{\mathbf c}$ on each of the factors of ${\mathfrak l}_{E_{\mathbf c}}$ given in corollary~\ref{P:LieSubgroupDescription}, the number of  irreducible   $G^{\iota}$-equivariant local systems (up to isomorphism) on the $G^{\iota}$-orbit  ${\mathcal O}_{\mathbf c}$ using proposition~\ref{P:LocalSystemCardinality} and  a list of irreducible   $G^{\iota}$-equivariant local systems on the $G^{\iota}$-orbit  ${\mathcal O}_{\mathbf c}$ using Lusztig's symbols as described in  \cite{L1984}, when there is more than one such irreducible   $G^{\iota}$-equivariant local system on the orbit. In this last case  when we list the  irreducible   $G^{\iota}$-equivariant local systems on the $G^{\iota}$-orbit, we will enumerate the symbols corresponding to the last factor of $L_{E_{\mathbf c}}$ that is a special orthogonal group. Note that,  for all the other factors,  they  are general linear groups and the local systems are trivial. 

The table we gave in  examples~\ref{E:4.27} and~\ref{E:5.7} is below.

\begin{table}[h]
	\begin{center}\renewcommand{\arraystretch}{1.25}
		\begin{tabular} {| l || r  | r  | r  | r  | r |  r | r | r  | r | c |}
			\hline
			Orbit(s) &  $100$  & $010$ & $001$ & $110$ & $011$ & $012$ & $111$ & $112$ & $122$ & Jordan Dec.  \\ \hline
			${\mathcal O}_{8}^1$ & $0$ & $0$ & $1$ & $0$ & $1$ & $0$ & $1$ &  $0$ & $0$ &$1^1 3^1 5^1$\\ \hline
			${\mathcal O}_{7}^1$ & $0$ & $0$ & $0$ & $0$ & $0$ & $1$ & $1$  & $0$ & $0$ &$2^2 5^1$\\ \hline
			${\mathcal O}_{7}^2$ & $0$ & $0$ & $1$ & $0$ & $0$ & $0$ & $0$  & $0$ & $1$ &$1^1 4^2$\\ \hline
			${\mathcal O}_{6}^1$ & $0$ & $1$ & $2$ & $0$ & $0$ & $0$ & $1$  & $0$ & $0$ &$1^4 5^1$\\ \hline
			${\mathcal O}_{6}^2$ & $0$ & $0$ & $0$ & $0$ & $1$ & $0$ & $0$  & $1$ & $0$ &$3^3$\\ \hline
			${\mathcal O}_{6}^3$ & $1$ & $0$ & $1$ & $0$ & $2$ & $0$ & $0$  & $0$ & $0$ &$1^3 3^2$\\ \hline
			${\mathcal O}_{5}^1$ & $0$ & $1$ & $1$ & $0$ & $0$ & $0$ & $0$  & $1$ & $0$ &$1^3 3^2$\\ \hline
			${\mathcal O}_{5}^2$ & $0$ & $0$ & $2$ & $1$ & $1$ & $0$ & $0$  & $0$ & $0$ &$1^2 2^2 3^1$\\ \hline
			${\mathcal O}_{5}^3$ & $1$ & $0$ & $0$ & $0$ & $1$ & $1$ & $0$  & $0$ & $0$ &$1^2 2^2 3^1$\\ \hline
			${\mathcal O}_{4}^1$ & $0$ & $0$ & $1$ & $1$ & $0$ & $1$ & $0$  & $0$ & $0$ &$1^1 2^4$\\ \hline
			${\mathcal O}_{4}^2$ & $1$ & $1$ & $2$ & $0$ & $1$ & $0$ & $0$  & $0$ & $0$ &$1^6 3^1$\\ \hline
			${\mathcal O}_{3}^1$ & $1$ & $1$ & $1$ & $0$ & $0$ & $1$ & $0$  & $0$ & $0$ &$1^5 2^2$\\ \hline
			${\mathcal O}_{2}^1$ & $0$ & $1$ & $3$ & $1$ & $0$ & $0$ & $0$  & $0$ & $0$ &$1^5 2^2$\\ \hline
			${\mathcal O}_{0}^1$ & $1$ & $2$ & $3$ & $0$ & $0$ & $0$ & $0$  & $0$ & $0$ &$1^{9}$\\ \hline
		\end{tabular}
	\end{center}
\caption{List of the values of coefficient function ${\mathbf c}$.}\label{T:Table9}\end{table}

\begin{itemize}
\item For the orbit ${\mathcal O}_{8}^1$ of dimension 8,   we have that ${\bf c} = (0, 0, 1, 0, 1, 0, 1, 0, 0)$

\begin{center}
${\mathcal T}_{8}^1$ = 
\ytableausetup{centertableaux}
\begin{ytableau}
\none & \none [-4] & \none [-2] & \none [0] & \none [2] & \none [4]\\ \none [{\phantom -}4] & 1 & 0 & 0 & 0 & 0\\ \none [{\phantom -}2] & 0 & 1 & 0 & 0\\  \none [{\phantom -}0] & 0 & 0 & 1\\ \none [-2] & 0 & 0 \\ \none [-4] & 0 \\ \end{ytableau}
\end{center}
$\Lambda({\mathbf c}) = \{0\}$, $\ell = 1$, $\ell' = 1$, $n_1 = 9$,   $L_{E_{\mathbf c}}=   SO_9({\mathbf k})$,  the partition corresponding to the nilpotent element $E_{\bf c}$  on the unique factor is $1^1 3^1 5^1$.  There is are  four  $G^{\iota}$-equivariant local systems (up to isomorphism)  on the $G^{\iota}$-orbit  ${\mathcal O}_{8}^1$ and the  corresponding symbols are
\[
(\{0, 4\}, \{2\}), \quad (\{0, 2\}, \{4\}), \quad (\{2, 4\}, \{0\}), \quad (\{0, 2, 4\}, \emptyset).
\]

\item For the orbits ${\mathcal O}_{7}^1$  of dimension 7,   we have that ${\bf c} = (0, 0, 0, 0, 0, 1, 1, 0, 0)$

\begin{center}
${\mathcal T}_{7}^1$ = 
\ytableausetup{centertableaux}
\begin{ytableau}
\none & \none [-4] & \none [-2] & \none [0] & \none [2] & \none [4]\\ \none [{\phantom -}4] & 1 & 0 & 0 & 0 & 0\\ \none [{\phantom -}2] & 0 & 0 & 1 & 0\\  \none [{\phantom -}0] & 0 & 1 & 0\\ \none [-2] & 0 & 0 \\ \none [-4] & 0 \\ \end{ytableau}
\end{center}
$\Lambda({\mathbf c}) = \{-1, 0, 1\}$, $\ell = 3$, $\ell' = 2$, $n_1 = 2$, $n_2 = 5$,   $L_{E_{\mathbf c}} \equiv   GL_2({\mathbf k}) \times SO_5({\mathbf k})$,  the partition corresponding to the nilpotent element $E_{\bf c}$  on the first  factor is $2^1$ and  the partition corresponding to the nilpotent element $E_{\bf c}$  on the second factor is $5^1$.    There is a unique  $G^{\iota}$-equivariant local system (up to isomorphism) on the $G^{\iota}$-orbit  ${\mathcal O}_{7}^1$: the trivial local system.

\item For the orbits ${\mathcal O}_{7}^2$ of dimension 7,   we have that ${\bf c} = (0, 0, 1, 0, 0, 0, 0, 0, 1)$

\begin{center}
${\mathcal T}_{7}^2$ = 
\ytableausetup{centertableaux}
\begin{ytableau}
\none & \none [-4] & \none [-2] & \none [0] & \none [2] & \none [4]\\ \none [{\phantom -}4] & 0 & 1 & 0 & 0 & 0\\ \none [{\phantom -}2] & 1 & 0 & 0 & 0\\  \none [{\phantom -}0] & 0 & 0 & 1\\ \none [-2] & 0 & 0 \\ \none [-4] & 0 \\ \end{ytableau}
\end{center}
$\Lambda({\mathbf c})  = \{-1, 0, 1\}$, $\ell = 3$, $\ell' = 2$, $n_1 = 4$, $n_2 = 1$,   $L_{E_{\mathbf c}} \equiv   GL_4({\mathbf k}) \times SO_1({\mathbf k})$,  the partition corresponding to the nilpotent element $E_{\bf c}$  on the first  factor is $4^1$ and  the partition corresponding to the nilpotent element $E_{\bf c}$  on the second factor is $1^1$.    There is a unique  $G^{\iota}$-equivariant local system (up to isomorphism) on the $G^{\iota}$-orbit  ${\mathcal O}_{7}^2$: the trivial local system.

\item For the orbit ${\mathcal O}_{6}^1$ of dimension 6,   we have that ${\bf c} = (0, 1, 2, 0, 0, 0, 1, 0, 0)$

\begin{center}
${\mathcal T}_{6}^1$ = 
\ytableausetup{centertableaux}
\begin{ytableau}
\none & \none [-4] & \none [-2] & \none [0] & \none [2] & \none [4]\\ \none [{\phantom -}4] & 1 & 0 & 0 & 0 & 0\\ \none [{\phantom -}2] & 0 & 0 & 0 & 1\\  \none [{\phantom -}0] & 0 & 0 & 2\\ \none [-2] & 0 & 1 \\ \none [-4] & 0 \\ \end{ytableau}
\end{center}
$\Lambda({\mathbf c})  = \{-2, 0, 2\}$, $\ell = 3$, $\ell' = 2$, $n_1 = 1$, $n_2 = 7$,   $L_{E_{\mathbf c}}  \equiv   GL_1({\mathbf k}) \times SO_7({\mathbf k})$,  the partition corresponding to the nilpotent element $E_{\bf c}$  on the first  factor is $1^1$ and  the partition corresponding to the nilpotent element $E_{\bf c}$  on the second factor is $1^2 5^1$.    There are two  $G^{\iota}$-equivariant local systems (up to isomorphism)  on the $G^{\iota}$-orbit  ${\mathcal O}_{6}^1$ and the  corresponding symbols are
\[
(\{0, 4\}, \{1\}), \quad (\{1, 4\}, \{0\}).
\]

\item For the orbits ${\mathcal O}_{6}^2$  of dimension 6,    we have that ${\bf c} = (0, 0, 0, 0, 1, 0, 0, 1, 0)$

\begin{center}
${\mathcal T}_{6}^2$ = 
\ytableausetup{centertableaux}
\begin{ytableau}
\none & \none [-4] & \none [-2] & \none [0] & \none [2] & \none [4]\\ \none [{\phantom -}4] & 0 & 0 & 1 & 0 & 0\\ \none [{\phantom -}2] & 0 & 1 & 0 & 0\\  \none [{\phantom -}0] & 1 & 0 & 0\\ \none [-2] & 0 & 0 \\ \none [-4] & 0 \\ \end{ytableau}
\end{center}
$\Lambda({\mathbf c})  = \{-2, 0, 2\}$, $\ell = 3$, $\ell' = 2$, $n_1 = 3$, $n_2 = 3$,   $L_{E_{\mathbf c}} \equiv   GL_3({\mathbf k}) \times SO_3({\mathbf k})$,  the partition corresponding to the nilpotent element $E_{\bf c}$  on the first  factor is $3^1$ and  the partition corresponding to the nilpotent element $E_{\bf c}$  on the second factor is $3^1$.    There is a unique  $G^{\iota}$-equivariant local system (up to isomorphism) on the $G^{\iota}$-orbit  ${\mathcal O}_{6}^2$: the trivial local system.

\item For the orbit ${\mathcal O}_{6}^3$ of dimension 6,    we have that ${\bf c} = (1, 0, 1, 0, 2, 0, 0, 0, 0)$

\begin{center}
${\mathcal T}_{6}^3$ = 
\ytableausetup{centertableaux}
\begin{ytableau}
\none & \none [-4] & \none [-2] & \none [0] & \none [2] & \none [4]\\ \none [{\phantom -}4] & 0 & 0 & 0 & 0 & 1\\ \none [{\phantom -}2] & 0 & 2 & 0 & 0\\  \none [{\phantom -}0] & 0 & 0 & 1\\ \none [-2] & 0 & 0 \\ \none [-4] & 1 \\ \end{ytableau}
\end{center}
$\Lambda({\mathbf c})  = \{-4, 0, 4\}$, $\ell = 3$, $\ell' = 2$, $n_1 = 1$, $n_2 = 7$,   $L_{E_{\mathbf c}} \equiv   GL_1({\mathbf k}) \times SO_7({\mathbf k})$,  the partition corresponding to the nilpotent element $E_{\bf c}$  on the first  factor is $1^1$ and  the partition corresponding to the nilpotent element $E_{\bf c}$  on the second factor is $1^1 3^2$.    There are two  $G^{\iota}$-equivariant local systems (up to isomorphism)  on the $G^{\iota}$-orbit  ${\mathcal O}_{6}^3$ and the  corresponding symbols are
\[
(\{0, 2\}, \{3\}), \quad (\{0, 3\}, \{2\}).
\]

\item For the orbits ${\mathcal O}_{5}^1$ of dimension 5,    we have that ${\bf c} = (0, 1, 1, 0, 0, 0, 0, 1, 0)$

\begin{center}
${\mathcal T}_{5}^1$ = 
\ytableausetup{centertableaux}
\begin{ytableau}
\none & \none [-4] & \none [-2] & \none [0] & \none [2] & \none [4]\\ \none [{\phantom -}4] & 0 & 0 & 1 & 0 & 0\\ \none [{\phantom -}2] & 0 & 0 & 0 & 1\\  \none [{\phantom -}0] & 1 & 0 & 1\\ \none [-2] & 0 & 1 \\ \none [-4] & 0 \\ \end{ytableau}
\end{center}
$\Lambda({\mathbf c})  = \{-2, 0, 2\}$, $\ell = 3$, $\ell' = 2$, $n_1 = 4$, $n_2 = 1$,   $L_{E_{\mathbf c}}  \equiv   GL_4({\mathbf k}) \times SO_1({\mathbf k})$,  the partition corresponding to the nilpotent element $E_{\bf c}$  on the first  factor is $1^1 3^1$ and  the partition corresponding to the nilpotent element $E_{\bf c}$  on the second factor is $1^1$.    There is a unique  $G^{\iota}$-equivariant local system (up to isomorphism)  on the $G^{\iota}$-orbit  ${\mathcal O}_{5}^1$: the trivial local system.

\item For the orbit ${\mathcal O}_{5}^2$ of dimension 5,   we have that ${\bf c} = (0, 0, 2, 1, 1, 0, 0, 0, 0)$

\begin{center}
${\mathcal T}_{5}^2$ = 
\ytableausetup{centertableaux}
\begin{ytableau}
\none & \none [-4] & \none [-2] & \none [0] & \none [2] & \none [4]\\ \none [{\phantom -}4] & 0 & 0 & 0 & 1 & 0\\ \none [{\phantom -}2] & 0 & 1 & 0 & 0\\  \none [{\phantom -}0] & 0 & 0 & 2\\ \none [-2] & 1 & 0 \\ \none [-4] & 0 \\ \end{ytableau}
\end{center}
$\Lambda({\mathbf c})  = \{-3, 0, 3\}$, $\ell = 3$, $\ell' = 2$, $n_1 = 2$, $n_2 = 5$,   $L_{E_{\mathbf c}}  \equiv   GL_2({\mathbf k}) \times SO_5({\mathbf k})$,  the partition corresponding to the nilpotent element $E_{\mathbf c}$  on the first  factor is $2^1$ and  the partition corresponding to the nilpotent element $E_{\bf c}$  on the second factor is $1^2 3^1$.    There are two  $G^{\iota}$-equivariant local systems (up to isomorphism)  on the $G^{\iota}$-orbit  ${\mathcal O}_{5}^2$ and the  corresponding symbols are
\[
(\{0, 3\}, \{1\}), \quad (\{1, 3\}, \{0\}).
\]

\item For the orbits ${\mathcal O}_{5}^3$ of dimension 5,   we have that ${\bf c} = (1, 0, 0, 0, 1, 1, 0, 0, 0)$

\begin{center}
${\mathcal T}_{5}^3$ = 
\ytableausetup{centertableaux}
\begin{ytableau}
\none & \none [-4] & \none [-2] & \none [0] & \none [2] & \none [4]\\ \none [{\phantom -}4] & 0 & 0 & 0 & 0& 1\\ \none [{\phantom -}2] & 0 & 1 & 1 & 0\\  \none [{\phantom -}0] & 0 & 1 & 0\\ \none [-2] & 0 & 0 \\ \none [-4] & 1 \\ \end{ytableau}
\end{center}
$\Lambda({\mathbf c})  = \{-4, -1, 0, 1,  4\}$, $\ell = 5$, $\ell' = 3$, $n_1 = 1$, $n_2 = 2$,  $n_3 = 3$,   $L_{E_{\mathbf c}} \equiv   GL_1({\mathbf k}) \times GL_2({\mathbf k}) \times SO_3({\mathbf k})$,  the partition corresponding to the nilpotent element $E_{\bf c}$  on the first  factor is $1^1$, the partition corresponding to the nilpotent element $E_{\bf c}$  on the second factor is $2^1$ and  the partition corresponding to the nilpotent element $E_{\bf c}$  on the third factor is $3^1$.    There is a unique  $G^{\iota}$-equivariant local system (up to isomorphism) on the $G^{\iota}$-orbit  ${\mathcal O}_{5}^3$: the trivial local system.

\item For the orbits ${\mathcal O}_{4}^1$ of dimension 4,   we have that ${\bf c} = (0, 0, 1, 1, 0, 1, 0, 0, 0)$

\begin{center}
${\mathcal T}_{4}^1$ = 
\ytableausetup{centertableaux}
\begin{ytableau}
\none & \none [-4] & \none [-2] & \none [0] & \none [2] & \none [4]\\ \none [{\phantom -}4] & 0 & 0 & 0 & 1 & 0\\ \none [{\phantom -}2] & 0 & 0 & 1 & 0\\  \none [{\phantom -}0] & 0 & 1 & 1\\ \none [-2] & 1 & 0 \\ \none [-4] & 0 \\ \end{ytableau}
\end{center}
$\Lambda({\mathbf c})  = \{-3, -1, 0, 1,  3\}$, $\ell = 5$, $\ell' = 3$, $n_1 = 2$, $n_2 = 2$,  $n_3 = 1$,   $L_{E_{\mathbf c}} \equiv   GL_2({\mathbf k}) \times GL_2({\mathbf k}) \times SO_1({\mathbf k})$,  the partition corresponding to the nilpotent element $E_{\mathbf c}$  on the first  factor is $2^1$, the partition corresponding to the nilpotent element $E_{\bf c}$  on the second factor is $2^1$ and  the partition corresponding to the nilpotent element $E_{\bf c}$  on the third factor is $1^1$.    There is a unique  $G^{\iota}$-equivariant local system (up to isomorphism)  on the $G^{\iota}$-orbit  ${\mathcal O}_{4}^1$: the trivial local system.

\item For the orbit ${\mathcal O}_{4}^2$ of dimension 4,   we have that ${\bf c} = (1, 1, 2, 0, 1, 0, 0, 0, 0)$

\begin{center}
${\mathcal T}_{4}^2$ = 
\ytableausetup{centertableaux}
\begin{ytableau}
\none & \none [-4] & \none [-2] & \none [0] & \none [2] & \none [4]\\ \none [{\phantom -}4] & 0 & 0 & 0 & 0 & 1\\ \none [{\phantom -}2] & 0 & 1 & 0 & 1\\  \none [{\phantom -}0] & 0 & 0 & 2\\ \none [-2] & 0 & 1 \\ \none [-4] & 1 \\ \end{ytableau}
\end{center}
$\Lambda({\mathbf c})  = \{-4, -2, 0, 2,  4\}$, $\ell = 5$, $\ell' = 3$, $n_1 = 1$, $n_2 = 1$,  $n_3 = 5$,   $L_{E_{\mathbf c}} \equiv   GL_1({\mathbf k}) \times GL_1({\mathbf k}) \times SO_5({\mathbf k})$,  the partition corresponding to the nilpotent element $E_{\bf c}$  on the first  factor is $1^1$, the partition corresponding to the nilpotent element $E_{\bf c}$  on the second factor is $1^1$ and  the partition corresponding to the nilpotent element $E_{\bf c}$  on the third factor is $1^2 3^1$.    There are two  $G^{\iota}$-equivariant local systems (up to isomorphism)  on the $G^{\iota}$-orbit  ${\mathcal O}_{4}^2$ and the  corresponding symbols are
\[
(\{0, 3\}, \{1\}), \quad (\{1, 3\}, \{0\}).
\]

\item For the orbits ${\mathcal O}_{3}^1$ of dimension 3, we have that ${\bf c} = (1, 1, 1, 0, 0, 1, 0, 0, 0)$

\begin{center}
${\mathcal T}_{3}^1$ = 
\ytableausetup{centertableaux}
\begin{ytableau}
\none & \none [-4] & \none [-2] & \none [0] & \none [2] & \none [4]\\ \none [{\phantom -}4] & 0 & 0 & 0 & 0 & 1\\ \none [{\phantom -}2] & 0 & 0 & 1 & 1\\  \none [{\phantom -}0] & 0 & 1 & 1\\ \none [-2] & 0 & 1 \\ \none [-4] & 1 \\ \end{ytableau}
\end{center}
$\Lambda({\mathbf c})  = \{-4, -2, -1, 0, 1,  2,  4\}$, $\ell = 7$, $\ell' = 4$, $n_1 = 1$, $n_2 = 1$,  $n_3 = 2$, $n_4 = 1$,   $L_{E_{\mathbf c}} \equiv   GL_1({\mathbf k}) \times GL_1({\mathbf k}) \times GL_2({\mathbf k}) \times SO_1({\mathbf k})$,  the partition corresponding to the nilpotent element $E_{\bf c}$  on the first  factor is $1^1$, the partition corresponding to the nilpotent element $E_{\mathbf c}$  on the second factor is $1^1$, the partition corresponding to the nilpotent element $E_{\bf c}$  on the third factor is $2^1$, and  the partition corresponding to the nilpotent element $E_{\bf c}$  on the fourth factor is $1^1$.    There is a unique  $G^{\iota}$-equivariant local system (up to isomorphism)  on the $G^{\iota}$-orbit  ${\mathcal O}_{3}^1$: the trivial local system.

\item For the orbit ${\mathcal O}_{2}^1$ of dimension 2,  we have that ${\bf c} = (0, 1, 3, 1, 0, 0, 0, 0, 0)$

\begin{center}
${\mathcal T}_{2}^1$ = 
\ytableausetup{centertableaux}
\begin{ytableau}
\none & \none [-4] & \none [-2] & \none [0] & \none [2] & \none [4]\\ \none [{\phantom -}4] & 0 & 0 & 0 & 1 & 0\\ \none [{\phantom -}2] & 0 & 0 & 0 & 1\\  \none [{\phantom -}0] & 0 & 0 & 3\\ \none [-2] & 1 & 1 \\ \none [-4] & 0 \\ \end{ytableau}
\end{center}
$\Lambda({\mathbf c})  = \{-3, -2,  0,   2,  3\}$, $\ell = 5$, $\ell' = 3$, $n_1 = 2$, $n_2 = 1$,  $n_3 = 3$,   $L_{E_{\mathbf c}} \equiv   GL_2({\mathbf k}) \times GL_1({\mathbf k})  \times SO_3({\mathbf k})$,  the partition corresponding to the nilpotent element $E_{\bf c}$  on the first  factor is $2^1$, the partition corresponding to the nilpotent element $E_{\bf c}$  on the second factor is $1^1$ and the partition corresponding to the nilpotent element $E_{\bf c}$  on the third factor is $1^3$.    There is a unique  $G^{\iota}$-equivariant local system (up to isomorphism) on the $G^{\iota}$-orbit  ${\mathcal O}_{2}^1$: the trivial local system.

\item For the orbit ${\mathcal O}_{0}^1$ of dimension 0,  we have that ${\bf c} = (1, 2, 3, 0, 0, 0, 0, 0, 0)$

\begin{center}
${\mathcal T}_{0}^1$ = 
\ytableausetup{centertableaux}
\begin{ytableau}
\none & \none [-4] & \none [-2] & \none [0] & \none [2] & \none [4]\\ \none [{\phantom -}4] & 0 & 0 & 0 & 0 & 1\\ \none [{\phantom -}2] & 0 & 0 & 0 & 2\\  \none [{\phantom -}0] & 0 & 0 & 3\\ \none [-2] & 0 & 2 \\ \none [-4] & 1 \\ \end{ytableau}
\end{center}
$\Lambda({\mathbf c}) = \{-4, -2,  0,   2,  4\}$, $\ell = 5$, $\ell' = 3$, $n_1 = 1$, $n_2 = 2$,  $n_3 = 3$,   $L_{E_{\mathbf c}} \equiv   GL_1({\mathbf k}) \times GL_2({\mathbf k})  \times SO_3({\mathbf k})$,  the partition corresponding to the nilpotent element $E_{\bf c}$  on the first  factor is $1^1$, the partition corresponding to the nilpotent element $E_{\bf c}$  on the second factor is $1^2$ and the partition corresponding to the nilpotent element $E_{\bf c}$  on the third factor is $1^3$.    There is a unique  $G^{\iota}$-equivariant local system (up to isomorphism) on the $G^{\iota}$-orbit  ${\mathcal O}_{0}^1$: the trivial local system.

\end{itemize}

\end{example}

\begin{example}
We will consider the orbits in example~\ref{E:5.29}. Recall that $G = SO_6({\mathbf k})$ and  $\iota(t) = diag(t^2, t^2, 1, t^{-2}, t^{-2}, 1)$. Here 
\[
SO_6({\mathbf k}) = \{A \in GL_6({\mathbf k}) \mid A^T J A = J \text{ and } \det(A) = 1\}
\]
where
\[
J = \begin{bmatrix} 0 & 0 & 0 & 1 & 0 & 0\\ 0 & 0 & 0 & 0 & 1 & 0\\ 0 & 0 & 0 & 0 & 0 & 1\\ 1 & 0 & 0 & 0 & 0 & 0\\ 0 & 1 & 0 & 0 & 0 & 0\\ 0 & 0 & 1 & 0 & 0 & 0\end{bmatrix} .
\]
In this case, $\dim({\mathfrak g}_2) = 4$ and there are five $G^{\iota}$-orbits.  We will enumerate for each of these orbits ${\mathcal O}_{\mathbf c}$: the coefficient function  ${\bf c}$,   the tableau ${\mathcal T} \in {\mathfrak T}_{\delta}$, the set $\Lambda({\mathbf c})$, the integers $\ell$, $\ell'$, the dimensions: $n_1$, $n_2$, ... and $n_{\ell'}$, $L_{E_{\mathbf c}}$ (up to isomorphism) as in corollary~\ref{C:LeviSubgpAssociatedNilpotent}, the partition corresponding to the restriction of $E_{\mathbf c}$ on each of the factors of ${\mathfrak l}_{E_{\mathbf c}}$ given in corollary~\ref{P:LieSubgroupDescription}, the number of  irreducible   $G^{\iota}$-equivariant local systems (up to isomorphism) on the $G^{\iota}$-orbit  ${\mathcal O}_{\mathbf c}$ using proposition~\ref{P:LocalSystemCardinality} and  a list of irreducible   $G^{\iota}$-equivariant local systems on the $G^{\iota}$-orbit  ${\mathcal O}_{\mathbf c}$ using Lusztig's symbols as described in  \cite{L1984}, when there is more than one such irreducible   $G^{\iota}$-equivariant local system on the orbit. In this last case  when we list the  irreducible   $G^{\iota}$-equivariant local systems on the $G^{\iota}$-orbit, we will enumerate the symbols corresponding to the last factor of $L_{E_{\mathbf c}}$ that is a special orthogonal group. Note that,  for all the other factors,  they  are general linear groups and the local systems are trivial. 

The table we gave in example~\ref{E:5.29} is below.
\begin{table}[h]
	\begin{center}\renewcommand{\arraystretch}{1.25}
		\begin{tabular} {| l || r  | r  | r  | r  | c |}
			\hline
			Orbit(s) &  $10$  & $01$ & $11$ & $12$ & Jordan Dec.  \\ \hline
			${\mathcal O}_{4}$  & $0$ & $0$ & $2$ & $0$ & $3^2$\\ \hline
			${\mathcal O}_{3}$& $1$ & $1$ & $1$ & $0$  &$1^3 3^1$\\ \hline
			${\ }'{\mathcal O}_{2}$, ${\ }''{\mathcal O}_{2}$& $1$ & $0$ & $0$ & $1$ &$1^2 2^2$\\ \hline
			${\mathcal O}_{0}$ & $2$ & $2$ & $0$ & $0$ &$1^6$\\ \hline
		\end{tabular}
	\end{center}
	\caption{List of the values on the coefficient function  ${\mathbf c}$.}\label{T:Table10}
\end{table}

There are two $G^{\iota}$-orbits: ${\mathcal O}_{2}$, ${\ }'{\mathcal O}_{2}$ of dimension 2. 

\begin{itemize}
\item For the orbit ${\mathcal O}_{4}$ of dimension 4

\begin{center}
${\mathcal T}_{4}$ = 
\ytableausetup{centertableaux}
\begin{ytableau}
\none & \none [-2] & \none [0] & \none [2] \\ \none [{\phantom -}2] & 2 & 0 & 0 \\  \none [{\phantom -}0] & 0 & 0 \\ \none [-2] & 0 \\  \end{ytableau}
\end{center}
$\Lambda({\mathbf c}) = \{0\}$, $\ell = 1$, $\ell' = 1$, $n_1 = 6$,   $L_{E_{\mathbf c}} \equiv   SO_6({\mathbf k})$,  the partition corresponding to the nilpotent element $E_{\bf c}$  on the only  factor is $3^2$.    There is a unique  $G^{\iota}$-equivariant local system (up to isomorphism) on the $G^{\iota}$-orbit  ${\mathcal O}_{4}$: the trivial local system.

\item For the orbit ${\mathcal O}_{3}$ of dimension 3

\begin{center}
${\mathcal T}_{3}$ = 
\ytableausetup{centertableaux}
\begin{ytableau}
\none & \none [-2] & \none [0] & \none [2]\\ \none [{\phantom -}2] & 1 & 0 & 1\\  \none [{\phantom -}0] & 0 & 1\\ \none [-2] & 1 \\ \end{ytableau}
\end{center}
$\Lambda({\mathbf c}) = \{-2, 0, 2\}$, $\ell = 3$, $\ell' = 2$, $n_1 = 1$, $n_2 = 4$,    $L_{E_{\mathbf c}} \equiv   GL_1({\mathbf k}) \times SO_4({\mathbf k})$,  the partition corresponding to the nilpotent element $E_{\bf c}$  on the first   factor is $1^1$, the partition corresponding to the nilpotent element $E_{\bf c}$  on the second   factor is $1^1 3^1$.    There are two  $G^{\iota}$-equivariant local systems (up to isomorphism) on the $G^{\iota}$-orbit  ${\mathcal O}_{3}$ and the  corresponding symbols are
\[
(\{0\}, \{2\}), \quad (\{0, 2\}, \emptyset).
\]

\item For the orbits ${\ }'{\mathcal O}_{2}$ and ${\ }''{\mathcal O}_{2}$ of dimension 2

\begin{center}
${\mathcal T}_{2}$ = 
\ytableausetup{centertableaux}
\begin{ytableau}
\none & \none [-2] & \none [0] & \none [2] \\ \none [{\phantom -}2] & 0 & 1 & 1\\  \none [{\phantom -}0] & 1 & 0\\ \none [-2] & 1 \\\end{ytableau}
\end{center}
$\Lambda({\mathbf c}) = \{-2, -1, 1, 2\}$, $\ell = 4$, $\ell' = 2$, $n_1 = 1$, $n_2 = 2$,   $L_{E_{\mathbf c}} \equiv   GL_1({\mathbf k}) \times GL_2({\mathbf k})$,  the partition corresponding to the nilpotent element $E_{\bf c}$  on the first  factor is $1^1$, the partition corresponding to the nilpotent element $E_{\bf c}$  on the second  factor is $2^1$.    There is a unique  $G^{\iota}$-equivariant local system (up to isomorphism) on each of the $G^{\iota}$-orbits  ${\  }'{\mathcal O}_{2}$ and ${\ }''{\mathcal O}_{2}$: the trivial local system.

\item For the orbit ${\mathcal O}_{0}$ of dimension 0

\begin{center}
${\mathcal T}_{0}$ = 
\ytableausetup{centertableaux}
\begin{ytableau}
\none & \none [-2] & \none [0] & \none [2] \\ \none [{\phantom -}2] & 0 & 0 & 2\\  \none [{\phantom -}0] & 0 & 2\\ \none [-2] & 2 \\ \end{ytableau}
\end{center}
$\Lambda({\mathbf c}) = \{-2, 0, 2\}$, $\ell = 3$, $\ell' = 2$, $n_1 = 2$, $n_2 = 2$,   $L_{E_{\mathbf c}} \equiv   GL_2({\mathbf k}) \times GL_2({\mathbf k})$,  the partition corresponding to the nilpotent element $E_{\bf c}$  on the first  factor is $1^2$, the partition corresponding to the nilpotent element $E_{\bf c}$  on the second  factor is $1^2$.    There is a unique  $G^{\iota}$-equivariant local system (up to isomorphism) on the $G^{\iota}$-orbit  ${\mathcal O}_{0}$: the trivial local system.

\end{itemize}
\end{example}

\begin{example}
Consider the even case with  $m = 1$ for a symplectic group. So ${\mathcal I} = \{1, -1\}$ and we have a vector space $V$ over ${\mathbf k}$ with a non-degenerate skew-symmetric bilinear form $\langle \  , \  \rangle: V \times V \rightarrow {\mathbf k}$. $G$ is the group of automorphisms of $V$ preserving the bilinear form $\langle \  , \  \rangle$. So $G = Sp(V)$. Suppose $V$ has a direct sum decomposition $\oplus_{i \in {\mathcal I}} V_i$ such that $\dim(V_1) = \dim(V_{-1}) = \delta > 0$  and $\langle u, v \rangle = 0$ whenever $u \in V_i$,  $v \in V_{j}$ with $i, j \in {\mathcal I}$ and $i + j \ne 0$.  Let $\iota: {\mathbf k}^{\times} \rightarrow G$ be the homomorphism  of groups defined by $\iota(t) v = t^i v$ for all $t \in {\mathbf k}^{\times}$ and $v \in V_i$ for any $i \in {\mathcal I}$.  

There is a basis ${\mathcal B} = {\mathcal B}_1 \coprod  {\mathcal B}_{-1}$ of $V$ such that ${\mathcal B}_i = \{u_{i, j} \mid 1 \leq j \leq \delta\}$ is a basis of $V_i$ for all $i \in {\mathcal I}$ such that, for $i, i' \in {\mathcal I}$ and $1 \leq j, j' \leq \delta$, we have
\[
\langle u_{i, j}, u_{i', j'}\rangle = \begin{cases} {\phantom -}1, &\text{if $i + i' = 0$, $i > 0$ and $j = j'$;}\\ -1, &\text{if $i + i' = 0$, $i < 0$ and $j = j'$;}\\  {\phantom -}0, &\text{otherwise.} \end{cases}
\]

As we saw in lemma~\ref{G2AsRepresEven} and corollary~\ref{C:DimensionG2EvenCase}, ${\mathfrak g}_2$ can be identified to the vector space $\{ M \in Mat_{\delta \times \delta}({\mathbf k}) \mid M^T = M\}$ and the  dimension 
\[
\dim({\mathfrak g}_2) =  \frac{\delta(\delta + 1)}{2}.
\]

In this case,  the symplectic symmetric tableaux corresponding to the orbits are of the form
\begin{center}
${\mathcal T}_{\alpha}$ = 
\ytableausetup{centertableaux}
\begin{ytableau}
\none & \none [-1] & \none [1]  \\ \none [{\phantom -}1] & \alpha & \beta \\   \none [-1] & \beta \\ \end{ytableau}
\end{center}
where $\alpha, \beta \in {\mathbb N}$ and $\alpha + \beta = \delta$. Thus there are $(1 + \delta)$ such tableaux, one for each integer $\alpha = 0, 1, 2, \dots, \delta$ and consequently there are $(\delta + 1)$  $G^{\iota}$-orbits: ${\mathcal O}_{\alpha}$  for ${\alpha} = 0, 1, 2, \dots, \delta$.  The tableau ${\mathcal T}_{\alpha}$ corresponds to the orbit 
\[
{\mathcal O}_{\alpha} = \{ M \in Mat_{\delta \times \delta}({\mathbf k}) \mid M^T = M  \text{ and rank(M) = ${\alpha}$}\}.
\]
From the proposition~\ref{DimensionFormulaEven}, we get that 
\[
\dim({\mathcal O}_{\alpha}) = {\alpha}(\delta - {\alpha}) + \frac{{\alpha}({\alpha} + 1)}{2}.
 \]
 
 The case $\alpha = 0$ correspond to the zero element and is easy to analyse. The dimension $\dim({\mathcal O}_0) = 0$ and there is a unique $G^{\iota}$-equivariant irreducible local system (up to isomorphism): the trivial local system.  From now on, we will assume that $\alpha > 0$. 
 
A Jacobson-Morozov triple $(E_{\alpha}, H_{\alpha}, F_{\alpha})$ corresponding to the orbit ${\mathcal O}_{\alpha}$ is 
\[
E_{\alpha} = \begin{bmatrix}  E_{11} & E_{12} & E_{13} & E_{14}\\ E_{21} & E_{22} & E_{23} & E_{24} \\ E_{31} & E_{32} & E_{33} & E_{34}\\ E_{41} & E_{42} & E_{43} & E_{44} \end{bmatrix} \quad \text{with $E_{13} = I_{\alpha \times \alpha}$ and all others $E_{ij} = 0$;}
\]
\[
H_{\alpha} = \begin{bmatrix}  H_{11} & H_{12} & H_{13} & H_{14}\\ H_{21} & H_{22} & H_{23} & H_{24} \\ H_{31} & H_{32} & H_{33} & H_{34}\\ H_{41} & H_{42} & H_{43} & H_{44} \end{bmatrix} \quad \text{with $H_{11} = I_{\alpha \times \alpha}$,  $H_{33} = - I_{\alpha \times \alpha}$  and all others $H_{ij} = 0$;} 
\]
\[
F_{\alpha} = \begin{bmatrix}  F_{11} & F_{12} & F_{13} & F_{14}\\ F_{21} & F_{22} & F_{23} & F_{24} \\ F_{31} & F_{32} & F_{33} & F_{34}\\ F_{41} & F_{42} & F_{43} & F_{44} \end{bmatrix}  \quad \text{with $F_{31} = I_{\alpha \times \alpha}$ and all others $F_{ij} = 0$.}
\]

The Levi subgroup $L$ associated to element $E_{\alpha}$ is the subgroup of matrices
\[
\begin{bmatrix}  A_{11} & 0 & B_{11} & 0\\ 0 & A_{22} & 0 &0 \\ C_{11} & 0 & D_{11} & 0\\ 0 & 0 & 0 & D_{22} \end{bmatrix} 
\]
such that 
\[
\begin{bmatrix}  A_{11} & B_{11} \\ C_{11} & D_{11} \end{bmatrix} \in Sp_{2\alpha}({\mathbf k}), \quad  A_{22}, D_{22} \in GL_{(\delta - \alpha)}({\mathbf k})   \text{ and } A_{22} D_{22} = I_{(\delta - \alpha) \times (\delta - \alpha)}.
\]

So $L$ is isomorphic to $GL_{(\delta - \alpha)}({\mathbf k}) \times Sp_{2\alpha}({\mathbf k})$ and we can consider $E_{\alpha}$ restricted each of the corresponding component of the Lie subalgebra ${\mathfrak l}$ of $L$. On the first component ${\mathfrak gl}_{(\delta - \alpha)}({\mathbf k})$, the nilpotent element  has for the Jordan decomposition: the corresponding partition: $1^{(\delta - \alpha)}$.  On the second component ${\mathfrak sp}_{2\alpha}({\mathbf k})$,  the nilpotent element  has for the Jordan decomposition: the corresponding partition: $2^{\alpha}$. Consequently on the $G^{\iota}$-orbit ${\mathcal O}_{\alpha}$, there are two $G^{\iota}$-equivariant irreducible local systems (up to isomorphism). We can compute the corresponding symbols  as described in  \cite{L1984}. We use the same notation as in this reference for the interval.  

If $\alpha \equiv 0 \pmod 2$, then the corresponding symbols for the $G^{\iota}$-equivariant irreducible local systems are
\begin{itemize}
\item $(\{0, 2, 4, \dots, \alpha\}, \{3, 5, \dots , (\alpha + 1)\}) \in \Psi_{2\alpha, 1}$  corresponding to the interval  $\emptyset$  in $C$ and  with defect $= 1$;\\
\item $(\{0, 3, 5, \dots , (\alpha + 1)\}, \{2, 4, \dots, \alpha\}) \in \Psi_{2\alpha, 1}$  corresponding to the interval  $\{2, 3, \dots , (\alpha + 1)\}$ in $C$ and  with defect $ = 1$.
\end{itemize}

If $\alpha \equiv 1 \pmod 2$, then the corresponding symbols for the $G^{\iota}$-equivariant irreducible local systems are
\begin{itemize}
\item $(\{2, 4, \dots, (\alpha - 1)\}, \{1, 3, 5, \dots , \alpha\}) \in \Psi_{2\alpha, -1}$  corresponding to the interval $\emptyset$  in $C$ and  with defect $= -1$; \\
\item $(\{1, 3, 5, \dots , \alpha\}, \{2, 4, \dots, (\alpha - 1)\}) \in \Psi_{2\alpha, 1}$  corresponding to the interval $\{1, 2, \dots , \alpha\}$ in $C$ and  with defect $ = 1$.
\end{itemize}

It is known that these are related with representations of Weyl groups.  In fact as stated in \cite{L1984}, there is a bijection between $\Psi_{2\alpha, 1}$ and the set $W^{\vee}(C_{\alpha})$ of irreducible representations of the Weyl group $W(C_{\alpha})$ of type $C_{\alpha}$. See 12.2 in \cite{L1984}.  $W(C_{\alpha})$  is the group of permutations $w$ of the set ${\mathcal E} = \{\alpha, \dots , 2, 1, -1, -2, \dots, - \alpha\}$ such that $w(-k) = -w(k)$ for all $k \in {\mathcal E}$.  The elements  of the set $W^{\vee}(C_{\alpha})$ are parametrized by ordered pairs of partitions $(0 \leq a_1 \leq a_2 \leq \dots \leq a_{m'}, 0 \leq b_1 \leq b_2 \leq \dots \leq b_{m''})$ with $\sum_{k'} a_{k'} + \sum_{k''} b_{k''} = \alpha$. See 2.7 (i) in \cite{L1977} for the precise description of the corresponding irreducible representation.

If $\alpha \equiv 0 \pmod 2$, we will write $\alpha = 2m$, then we have the following correspondance:
\begin{itemize}
\item $(\{0, 2, 4, \dots, \alpha\}, \{3, 5, \dots , (\alpha + 1)\}) \in \Psi_{2\alpha, 1}$ correspond to the ordered pair of partitions $(\emptyset, 2^{m})$ with total sum equal to  $\alpha$;\\
\item $(\{0, 3, 5, \dots , (\alpha + 1)\}, \{2, 4, \dots, \alpha\}) \in \Psi_{2\alpha, 1}$ correspond to the ordered pair of partitions $(1^m, 1^m)$ with total sum equal of $\alpha$.
\end{itemize}

If $\alpha \equiv 1 \pmod 2$, we will write $\alpha = (2m - 1)$, then we have the following correspondance:
\begin{itemize}
\item $(\{2, 4, \dots, (\alpha - 1)\}, \{1, 3, 5, \dots , \alpha\}) \in \Psi_{2\alpha, -1}$  correspond to the ordered pair of partitions $(2^{(m - 1)},  \emptyset)$ with total sum equal to  $(\alpha - 1)$; \\
\item $(\{1, 3, 5, \dots , \alpha\}, \{2, 4, \dots, (\alpha - 1)\}) \in \Psi_{2\alpha, 1}$  correspond to the ordered pair of partitions $(1^m,  1^{(m - 1)})$ with total sum equal to  $\alpha$.\\
\end{itemize}

Note that we used the bijection  $\Psi_{2(\alpha - 1), 1} \rightarrow \Psi_{2\alpha, -1}$  that maps 
\[
(\{2, 4, \dots, (\alpha - 1)\}, \{1, 3, \dots, (\alpha - 4)\}) \text{  to  } (\{2, 4, \dots, (\alpha - 1)\}, \{1, 3, 5, \dots , \alpha\})
\]
and then consider the representation of the appropriate Weyl group corresponding to the symbol $(\{2, 4, \dots, (\alpha - 1)\}, \{1, 3, \dots, (\alpha - 4)\})$. 

\end{example}

\begin{example}
Consider the even case with  $m = 1$ for a special orthogonal group. So ${\mathcal I} = \{1, -1\}$ and we have a vector space $V$ over ${\mathbf k}$ with a non-degenerate symmetric bilinear form $\langle \  , \  \rangle: V \times V \rightarrow {\mathbf k}$. $G$ is the connected component of the identity element of the group of automorphisms of $V$ preserving the bilinear form $\langle \  , \  \rangle$. So $G = SO(V)$. Suppose $V$ has a direct sum decomposition $\oplus_{i \in {\mathcal I}} V_i$ such that $\dim(V_1) = \dim(V_{-1}) = \delta > 0$  and $\langle u, v \rangle = 0$ whenever $u \in V_i$,  $v \in V_{j}$ with $i, j \in {\mathcal I}$ and $i + j \ne 0$.  Let $\iota: {\mathbf k}^{\times} \rightarrow G$ be the homomorphism  of groups defined by $\iota(t) v = t^i v$ for all $t \in {\mathbf k}^{\times}$ and $v \in V_i$ for any $i \in {\mathcal I}$.  

There is a basis ${\mathcal B} = {\mathcal B}_1 \coprod  {\mathcal B}_{-1}$ of $V$ such that ${\mathcal B}_i = \{u_{i, j} \mid 1 \leq j \leq \delta\}$ is a basis of $V_i$ for all $i \in {\mathcal I}$ such that, for $i, i' \in {\mathcal I}$ and $1 \leq j, j' \leq \delta$, we have
\[
\langle u_{i, j}, u_{i', j'}\rangle = \begin{cases} 1, &\text{if $i + i' = 0$ and $j = j'$;} \\  0, &\text{otherwise.} \end{cases}
\]

As we saw in lemma~\ref{G2AsRepresEven} and corollary~\ref{C:DimensionG2EvenCase},  ${\mathfrak g}_2$ can be identified to the vector space $\{ M \in Mat_{\delta \times \delta}({\mathbf k}) \mid M^T = - M\}$ and the  dimension 
\[
\dim({\mathfrak g}_2) =  \frac{\delta(\delta - 1)}{2}.
\]

In this case,  the orthogonal symmetric tableaux corresponding to the orbits are of the form
\begin{center}
${\mathcal T}_{\alpha}$ = 
\ytableausetup{centertableaux}
\begin{ytableau}
\none & \none [-1] & \none [1]  \\ \none [{\phantom -}1] & 2\alpha & \beta \\   \none [-1] & \beta \\ \end{ytableau}
\end{center}
where $\alpha, \beta \in {\mathbb N}$ and $2\alpha + \beta = \delta$. Thus there are $(1 + \lfloor \frac{\delta}{2} \rfloor)$ such tableaux, one for each integer $\alpha = 0, 1, 2, \dots,  \lfloor \frac{\delta}{2} \rfloor$ , where $\lfloor x \rfloor = \max\{n \in {\mathbb Z} \mid n \leq x\}$.  Consequently there are $(1 + \lfloor \frac{\delta}{2} \rfloor)$  $G^{\iota}$-orbits: ${\mathcal O}_{\alpha}$  for ${\alpha} = 0, 1, 2, \dots, \lfloor \frac{\delta}{2} \rfloor$.  The tableau ${\mathcal T}_{\alpha}$ correspond to the orbit 
\[
{\mathcal O}_{\alpha} = \{ M \in Mat_{\delta \times \delta}({\mathbf k}) \mid M^T = -M  \text{ and rank(M) = $2\alpha$}\}.
\]
From the proposition~\ref{DimensionFormulaEven}, we get that 
\[
\dim({\mathcal O}_{\alpha}) = \alpha(2\alpha + 2\beta - 1). 
 \]

Note that if $\delta = 2m$ with $m \in {\mathbb N}$, then the top orbit is ${\mathcal O}_m$, where $\alpha = m$, $\beta = 0$,  and we get that $\dim({\mathcal O}_m) = 2m^2  - m = (2m - 1)m = \delta(\delta - 1)/2$; while if $\delta = (2m + 1)$, then the top orbit is ${\mathcal O}_m$, where $\alpha = m$, $\beta = 1$,  and we get that $\dim({\mathcal O}_m) = m(2m + 1) = \delta(\delta - 1)/2$. 

 The case $\alpha = 0$ correspond to the zero element and is easy to analyse. The dimension $\dim({\mathcal O}_0) = 0$ and there is a unique $G^{\iota}$-equivariant irreducible local system (up to isomorphism): the trivial local system.  From now on, we will assume that $\alpha > 0$. 
 
A Jacobson-Morozov triple $(E_{\alpha}, H_{\alpha}, F_{\alpha})$ corresponding to the orbit ${\mathcal O}_{\alpha}$ is 
\[
E_{\alpha} = \begin{bmatrix}  E_{11} & E_{12} & E_{13} & E_{14} & E_{15} & E_{16}\\ E_{21} & E_{22} & E_{23} & E_{24}& E_{25} & E_{26} \\ E_{31} & E_{32} & E_{33} & E_{34} & E_{35} & E_{36}\\ E_{41} & E_{42} & E_{43} & E_{44} & E_{45} & E_{46} \\ E_{51} & E_{52} & E_{53} & E_{54} & E_{55} & E_{56} \\ E_{61} & E_{62} & E_{63} & E_{64} & E_{65} & E_{66}  \end{bmatrix} 
\]
with $E_{15} = I_{\alpha \times \alpha}$, $E_{24} = -I_{\alpha \times \alpha}$ and all others $E_{ij} = 0$;
\[
H_{\alpha} = \begin{bmatrix}  H_{11} & H_{12} & H_{13} & H_{14} &  H_{15} &  H_{16}\\ H_{21} & H_{22} & H_{23} & H_{24} &  H_{25} &  H_{26}\\ H_{31} & H_{32} & H_{33} & H_{34} &  H_{35} &  H_{36}\\ H_{41} & H_{42} & H_{43} & H_{44} &  H_{45} &  H_{46} \\ H_{51} & H_{52} & H_{53} & H_{54} &  H_{55} &  H_{56} \\ H_{61} & H_{62} & H_{63} & H_{64} &  H_{65} &  H_{66}  \end{bmatrix}
\]
with $H_{11} = I_{\alpha \times \alpha}$, $H_{22} = I_{\alpha \times \alpha}$,   $H_{44} = - I_{\alpha \times \alpha}$,  $H_{55} = - I_{\alpha \times \alpha}$  and all others $H_{ij} = 0$; 
\[
F_{\alpha} = \begin{bmatrix}  F_{11} & F_{12} & F_{13} & F_{14} &F_{15} & F_{16}\\ F_{21} & F_{22} & F_{23} & F_{24} & F_{25} & F_{26} \\ F_{31} & F_{32} & F_{33} & F_{34} &F_{35} & F_{36}\\ F_{41} & F_{42} & F_{43} & F_{44} & F_{45} & F_{46} \\ F_{51} & F_{52} & F_{53} & F_{54} & F_{55} & F_{56} \\ F_{61} & F_{62} & F_{63} & F_{64} & F_{65} & F_{66}\end{bmatrix}
\]
with $F_{42} = -I_{\alpha \times \alpha}$, $F_{51} = I_{\alpha \times \alpha}$ and all others $F_{ij} = 0$.

The Levi subgroup $L$ associated to element $E_{\alpha}$ is the subgroup of matrices
\[
\begin{bmatrix}  A_{11} & 0 & B_{11} & 0\\ 0 & A_{22} & 0 &0 \\ C_{11} & 0 & D_{11} & 0\\ 0 & 0 & 0 & D_{22} \end{bmatrix} 
\]
such that 
\[
\begin{bmatrix}  A_{11} & B_{11} \\ C_{11} & D_{11} \end{bmatrix} \in SO_{4\alpha}({\mathbf k}), \quad  A_{22}, D_{22} \in GL_{\beta}({\mathbf k})   \text{ and } A_{22} D_{22} = I_{\beta \times \beta}.
\]

So $L$ is isomorphic to $GL_{\beta}({\mathbf k}) \times SO_{4\alpha}({\mathbf k})$ and we can consider $E_{\alpha}$ restricted each of the corresponding component of the Lie subalgebra ${\mathfrak l}$ of $L$. On the first component ${\mathfrak gl}_{\beta}({\mathbf k})$, the nilpotent element  has for the Jordan decomposition: the corresponding partition: $1^{\beta}$.  On the second component ${\mathfrak so}_{4\alpha}({\mathbf k})$,  the nilpotent element  has for the Jordan decomposition: the corresponding partition: $2^{2\alpha}$. Consequently on the $G^{\iota}$-orbit ${\mathcal O}_{\alpha}$, there is exactly one $G^{\iota}$-equivariant irreducible local system (up to isomorphism). 

When $\alpha > 0$, then the corresponding symbol for the $G^{\iota}$-equivariant irreducible local system on the orbit ${\mathcal O}_{\alpha}$ is the unordered pair 
\[
(\{1, 3, \dots, (2\alpha - 1)\}, \{1, 3, \dots, (2\alpha - 1)\}) \in \Psi'_{4\alpha, 0}.
\]
This symbol has defect 0. Here we  computed  the corresponding symbol  as  in  \cite{L1984} and  we use the same notation as in this reference.  It is known that this symbol for  the orbit ${\mathcal O}_{\alpha}$ is related to an irreducible representation of Weyl group $W(D_{2\alpha})$.  This irreducible representation correspond to the unordered pair of partitions $(1^{\alpha}, 1^{\alpha})$ whose total sum is $2\alpha$.
 
\end{example}

\section*{Acknowledgements}
The author thanks George Lusztig for helpful discussions and inspiration.  The author also want to thank Daniel Ciubotaru for  answering  questions  about his article \cite{C2008}.


\end{document}